\documentclass[12pt]{article}
\usepackage{a4wide}
\usepackage{amsmath, amsthm, amsfonts, amssymb, bbm}
\usepackage{graphics}
\usepackage{xypic}
\usepackage[all]{xy}
\usepackage[english]{babel}
\usepackage[font=small,format=plain,labelfont=bf,up]{caption}
\usepackage{hyperref}

\newcommand\scanpic[1]		{\raisebox{-0.5\height}{\scalebox{.54}{\includegraphics{pic#1.png}}}}
\newcommand\scanPIC[1]	{\raisebox{-0.5\height}{\scalebox{.66}{\includegraphics{pic#1.png}}}}
\newcommand\scanPICC[1]		{\raisebox{-0.5\height}{\scalebox{.80}{\includegraphics{pic#1.png}}}}

\allowdisplaybreaks		

\theoremstyle{plain}

\newtheorem{theorem}{Theorem}
\newtheorem{lemma}[theorem]{Lemma}
\newtheorem{proposition}[theorem]{Proposition}
\newtheorem{corollary}[theorem]{Corollary}

\theoremstyle{definition}

\newtheorem{remark}[theorem]{Remark}
\newtheorem{definition}[theorem]{Definition}
\newtheorem{notation}[theorem]{Notations}

 \numberwithin{equation}{section}
 \numberwithin{theorem}{section}

\DeclareMathOperator{\Aut}{Aut}
\DeclareMathOperator{\End}{End}

\DeclareMathOperator{\Rep}{Rep}
\DeclareMathOperator{\ev}{ev}
\DeclareMathOperator{\coev}{coev}
\DeclareMathOperator{\Ad}{Ad}

\newcommand\be            {\begin{equation}}
\newcommand\ee            {\end{equation}}

\newcommand{\ot}{\otimes}
\newcommand{\lb}{\label}

\newcommand\nxt{\noindent\raisebox{.08em}{\rule{.44em}{.44em}}\hspace{.4em}}

\newcommand\vect{\mathcal{V}\hspace{-.5pt}ect}
\newcommand\svect{\mathcal{S}\mathcal{V}\hspace{-.5pt}ect}

\newcommand\eps           {\varepsilon}
\newcommand\id            {id}
\newcommand\Id            {I\hspace{-1pt}d}
\newcommand\one           {{\bf1}}

\newcommand\Zc            {\mathcal{Z}}
\newcommand\Cb            {\mathbb{C}}

\newcommand\Zb            {\mathbb{Z}}

\newcommand\Cc            {\mathcal{C}}

\newcommand\Sc            {\mathcal{S}}

\newcommand\h            {\mathfrak{h}}

\newcommand\void[1]	{}

\begin{document}

\thispagestyle{empty}
\def\thefootnote{\fnsymbol{footnote}}
\begin{flushright}
ZMP-HH/12-8\\
Hamburger Beitr\"age zur Mathematik 432
\end{flushright}
\vskip 3em
\begin{center}\LARGE
$\Zb/2\Zb$-extensions of Hopf algebra module categories\\ by their base categories
\end{center}

\vskip 2em
\begin{center}
{\large 
Alexei Davydov$^{a}$,\, Ingo Runkel$^{b}$,\, ~\footnote{Emails: {\tt alexei1davydov@gmail.com}, {\tt ingo.runkel@uni-hamburg.de}}}
\\[1em]
\it$^a$ 
Department of Mathematics\\
Ohio University, Athens OH 45701, USA
\\[1em]
$^b$ Fachbereich Mathematik, Universit\"at Hamburg\\
Bundesstra\ss e 55, 20146 Hamburg, Germany
\end{center}

\vskip 2em
\begin{center}
  July 2012
\end{center}
\vskip 2em

\begin{abstract}
Starting with a self-dual Hopf algebra $H$ in a braided monoidal category $\Sc$ we construct a $\Zb/2\Zb$-graded monoidal category $\Cc = \Cc_0 + \Cc_1$. The degree zero component is the category $Rep_\Sc(H)$ of representations of $H$ and the degree one component is the category $\Sc$. The extra structure on $H$ needed to define the associativity isomorphisms is a choice of self-duality map and cointegral, subject to certain conditions. We also describe rigid, braided and ribbon structures on $\Cc$ in Hopf algebraic terms.

Our construction permits a uniform treatment of Tambara-Yamagami categories and categories related to symplectic fermions in conformal field theory.
\end{abstract}

\setcounter{footnote}{0}
\def\thefootnote{\arabic{footnote}}

\newpage

\tableofcontents

\section{Introduction}

Let $\Sc$ be a braided monoidal category. A monoidal category over $\Sc$ is a monoidal category $\Cc$ with a full monoidal embedding $\Sc\to\Cc$, together with a lift to a braided monoidal functor $\Sc\to\Zc(\Cc)$ \cite{DGNO,Davydov:2011a}.
In this paper we construct examples of $\Zb/2\Zb$-graded monoidal categories $\Cc = \Cc_0 + \Cc_1$ over $\Sc$ such that $\Cc_1=\Sc$. The component $\Cc_0$ is the category $\Rep_\Sc(H)$ of representations of a Hopf algebra $H$ in $\Sc$. 
Our aim when setting up this construction was to have a natural framework in which to place an example obtained from so-called symplectic fermions; we will get back to this in a moment. 

\medskip

The first main result of this paper is the construction of a tensor product functor and associativity isomorphisms for the $\Zb/2\Zb$-graded monoidal category $\Cc = \Rep_\Sc(H) + \Sc$. We describe solutions to the pentagon in terms of Hopf algebraic data: one needs to fix a self-duality structure on $H$ and a cointegral, subject to compatibility conditions (Theorem \ref{thm:main1}).

In a followup paper we plan to address the converse construction \cite{DR-prep}: Take a $\Zb/2\Zb$-graded monoidal category $\Cc = \Cc_0 + \Cc_1$ over $\Sc$ such that $\Cc_1=\Sc$ (subject to certain extra conditions). 
The tensor product $\Cc_0\times\Cc_1\to \Cc_1$ gives rise to a monoidal (fibre) functor $\Cc_0\to\Sc$. Under mild conditions the fibre functor produces via Tannaka-Krein reconstruction a Hopf algebra object $H$ in $\Sc$ such that $\Cc_0$ coincides with the category $\Rep_\Sc(H)$ of representations of $H$ in $\Sc$ \cite[Thm.\,3.2]{Majid:1995}. The rest of the tensor product in $\Cc$ is also essentially fixed: the product $\Cc_1\times\Cc_0\to \Cc_1$ is obtained from the fibre functor, and if the category $\Cc$ is to be rigid, the product $\Cc_1\times\Cc_1\to \Cc_0$ must be the tensor product $\Sc\times\Sc\to\Sc$ followed by tensoring with $H$. The condition that $\Cc$ is a monoidal category over $\Sc$ restricts the possible shape of the associativity isomorphisms to the ones in Section \ref{results} (up to an extra invertible object, see \cite{DR-prep}).

\medskip

Our second main result is a description of braiding isomorphisms for the $\Zb/2\Zb$-graded monoidal category $\Cc = \Rep_\Sc(H) + \Sc$. Our assumption here is that $\Sc$ is a symmetric category with its own monoidal $\Zb/2\Zb$-grading $\Sc = \Sc_0 + \Sc_1$ (such as for example super-vector spaces)
and that $\Cc$ is a braided monoidal category over $\Sc_0$, i.e.\ $\Sc_0$ is a full subcategory of the symmetric centre\footnote{
  The symmetric centre of a braided category is the full subcategory of all objects $T$ such that $c_{U,T} \circ c_{T,U} = \id_{T \otimes U}$ for all objects $U$ \cite{Muger:1998a}. Such objects $T$ are also called transparent \cite{Bruguieres:2000}.}
of $\Cc$. This essentially restricts the possible shape of the braiding to the one in \eqref{eq:braiding-ansatz} \cite{DR-prep}. Again we characterise possible solutions in terms of Hopf algebraic data (Theorem \ref{thm:main2}).

\medskip

Let us now discuss our motivating examples in more detail. The first one comes from the theory of symplectic fermions \cite{Kausch:1995py}, which give an important family of examples of two-dimensional logarithmic conformal field theories.

Let $\h$ be a finite dimensional complex vector space with a symplectic form $(-|-)$. Consider $\h$ as a purely odd abelian Lie super-algebra. The symplectic form on $\h$ becomes an invariant (super-)symmetric form. The vertex operator super-algebra of symplectic fermions is the vacuum module of the affinisation of the pair $\h, (-|-)$, see \cite[Sect.\,3.5]{Kac:1998} and \cite{Abe:2005}. Via a conformal field theory computation one can find the category of representations $\Cc$ and the conformal three- and four-point blocks for symplectic fermions \cite{Runkel:2012cf}. As for any vertex operator super-algebra, the category $\Cc$ is $\Zb/2\Zb$-graded with $\Cc_0$ consisting of untwisted (Neveu-Schwarz) representations and $\Cc_1$ of twisted (Ramond) representations. For symplectic fermions the untwisted sector $\Cc_0$ coincides with the category of $\h$-modules (in the category of super-vector spaces). The twisted sector $\Cc_1$ has two irreducible representations and is equivalent to the category of super-vector spaces. The conformal block calculation gives $\Cc$ the structure of a $\Zb/2\Zb$-graded braided monoidal category. The associativity and braiding isomorphisms found in \cite{Runkel:2012cf} fit into the framework described here with $\Sc$ being the category of super-vector spaces and $H$ being the universal enveloping algebra of $\h$ in $\Sc$, see Section \ref{sec:symp-ferm-mon}.

The other motivating example which we consider in some detail are the well-known Tambara-Yamagami categories \cite{Tambara:1998}. In our language, these categories correspond to group algebras $H = k[A]$ of finite abelian groups in the category $\Sc$ of $k$-vector spaces. We describe associativity isomorphisms and braidings for Tambara-Yamagami categories and recover the results of \cite{Tambara:1998} and \cite{Siehler:2001}, see Sections \ref{sec:TY}.

\bigskip

This paper is organised as follows. In Section \ref{prem-res} we review our conventions for Hopf algebras in braided categories and state our main results in a self-contained fashion. The proofs of our results are contained in Sections \ref{sec:theorem1} and \ref{sec:braiding}. Finally, in Section \ref{sec:mon-ex} we discuss four examples of our construction.

\bigskip
\noindent

{\bf Acknowledgements:} We would like to thank M.\ Izumi, D.\ Jordan, M.\ Mombelli and D.\ Nikshych for helpful conversations. We also thank J.\ Lederich for useful comments on the draft.
AD thanks the Department of Mathematics of Hamburg University for hospitality during two visits in 2011 and 2012. During these visits AD was partially supported by the Graduiertenkolleg 1670 of the Deutsche Forschungsgemeinschaft.

\section{Conventions and main results}\label{prem-res}

Before we formulate the main results of the paper in more detail, we need to give our conventions on Hopf algebras in braided categories.

\subsection{Hopf algebras in braided categories}\label{sec:Hopf}

\begin{notation} \label{not:sec2}
We fix a ribbon category $\Sc$, which we will assume to be strict for notational simplicity; $\otimes$ (without index) stands for the tensor product in $\Sc$; $c_{U,V} : U \otimes V \to V \otimes U$ is the braiding in $\Sc$. 
We denote the evaluation and coevaluation by $\ev : U^\vee \otimes U \to \one$ and $\coev : \one \to U \otimes U^\vee$, the ribbon twist by $\theta_U$, and the natural monoidal isomorphism to the double dual by $\delta_U : U \to U^{\vee\vee}$.
We will think of $\End_{\Sc}(\one)$ as scalars and write e.g.\ $s \cdot f$ with $s : \one \to \one$ and $f : U \to V$ instead of $U \xrightarrow{=} \one \otimes U \xrightarrow{s \otimes f}  \one \otimes V \xrightarrow{=} V$. 
\end{notation}

An algebra in $\Sc$ is an object $A \in \Sc$, together with a product morphism $\mu = \mu_A : A \otimes A \to A$ and a unit morphism $\eta = \eta_A : \one \to A$, subject to the associativity and unit conditions. Analogously, a coalgebra is an object $C$ together with a coproduct $\Delta = \Delta_C : C \to C \otimes C$ and a counit $\eps = \eps_C : C \to \one$ subject to coassociativity and counit conditions. 

Given algebras $A,B$ in a symmetric category, the tensor product $A \otimes B$ carries a canonical algebra structure. In the present case, where $\Sc$ is braided, the product on $A \otimes B$ involves a choice between the braiding and its inverse. We fix
\be \label{eq:mult-on-product}
  \mu_{A \otimes B} = (\mu_A \otimes \mu_B) \circ (\id_A \otimes c_{B,A} \otimes \id_B) 
\ee
The main player in this paper is a Hopf algebra in the braided category $\Sc$ (see e.g. \cite{Majid:1995}).

\begin{definition}\label{def:braided-Hopf}
A {\em Hopf algebra} in $\Sc$ is an object $H \in \Sc$ together with morphisms $\mu,\eta,\Delta,\eps,S$, such that
\begin{enumerate}
\item $H$ is an algebra with product $\mu$ and unit $\eta$,
\item $H$ is a coalgebra with coproduct $\Delta$ and counit $\eps$,
\item $\Delta$ is an algebra morphism from $H$ to $H \otimes H$, $\eps$ is an algebra morphism from $H$ to $\one$,
\item $S$, the {\em antipode}, is an endomorphism of $H$ and satisfies
$$
  \mu \circ (S \otimes \id) \circ \Delta = \eta \circ \eps = \mu \circ (\id \otimes S) \circ \Delta \ .
$$ 
\end{enumerate}
\end{definition}

\begin{figure}[tb]
$$
\mu = \scanPIC{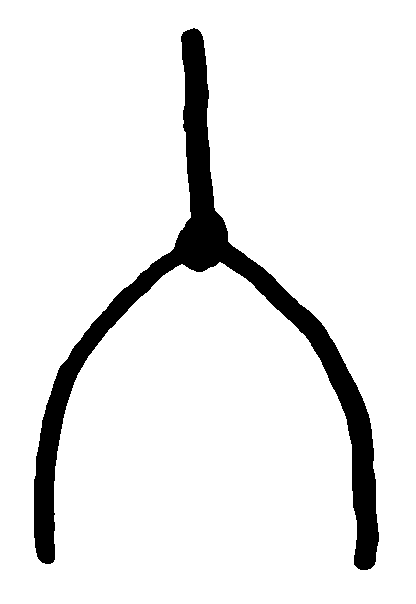}
\qquad
\eta = \scanPIC{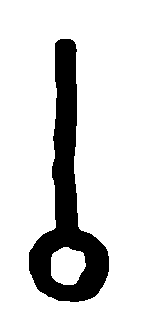}
\qquad
\Delta = \scanPIC{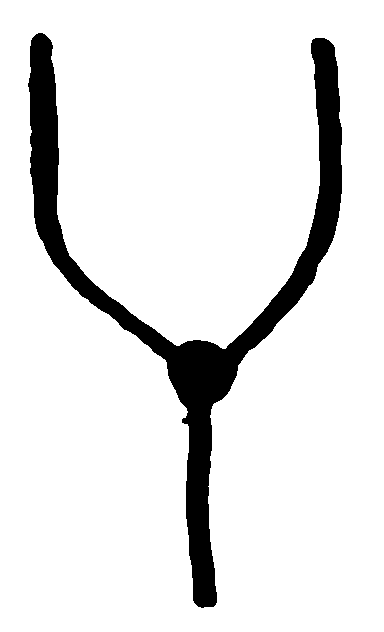}
\qquad
\eps = \scanPIC{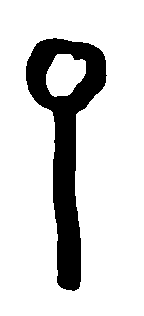}
\qquad
S = \scanPIC{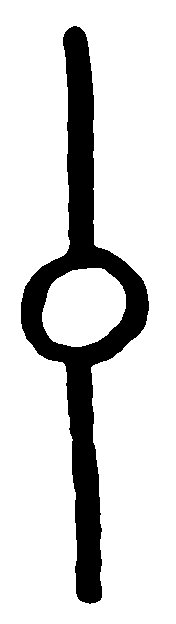}
\qquad
S^{-1} = \scanPIC{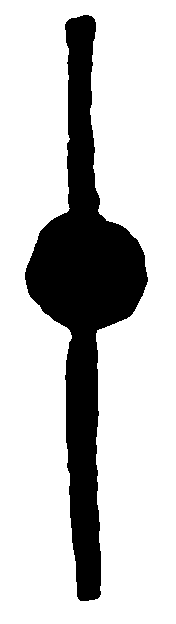}
$$
\caption{Our conventions for the string diagram representation of the structure morphisms of a Hopf algebra. In this paper, string diagrams are read from bottom to top. E.g.\ the product $\mu$ is a morphism $H \otimes H \to H$. If source and target object in a string diagram are not explicitly indicated, they are (tensor products of) the Hopf algebra $H$.}
\label{fig:Hopf-graph-convention}
\end{figure}

Figure \ref{fig:Hopf-graph-convention} shows our convention for string diagrams which will be used extensively below. The algebra-map property of the coproduct and counit, as well as the property of the antipode, which we will call `bubble-property', are shown in Figure \ref{fig:Hopf-property}\,a)--d) and g). One checks that $S$ is an algebra and a coalgebra anti-automorphism in the sense that (see e.g.\ \cite[Lem.\,2.3]{Majid:1995})
\be\lb{ap}
  S \circ \mu = \mu \circ c_{H,H} \circ (S \otimes S)
  ~~,~~~
  \Delta \circ S =  (S \otimes S) \circ c_{H,H} \circ \Delta
  ~~,~~~
  S \circ \eta = \eta
  ~~,~~~
  \eps \circ S = \eps \ .
\ee
The first two properties are shown in Figure \ref{fig:Hopf-property}\,e),\,f).
For the rest of this paper let us in addition agree on the convention that
\begin{quote}
  `Hopf algebra' stands for `Hopf algebra with invertible antipode'.
\end{quote}
This is not a strong restriction; for example, if $H$ is taken from the category of finite-dimensional vector spaces, $S$ is automatically invertible \cite[Cor.\,5.1.6]{Sweedler:1967}. More generally, if $\Sc$ is ribbon and abelian, the antipode of a Hopf-algebra in $\Sc$ is invertible \cite[Thm.\,4.1]{Takeuchi:1999} (where in fact $\Sc$ is just braided with equalisers and $H$ is rigid).

\begin{figure}[tb]
$$
\raisebox{2.5em}{\text{a)}}~~
\scanPIC{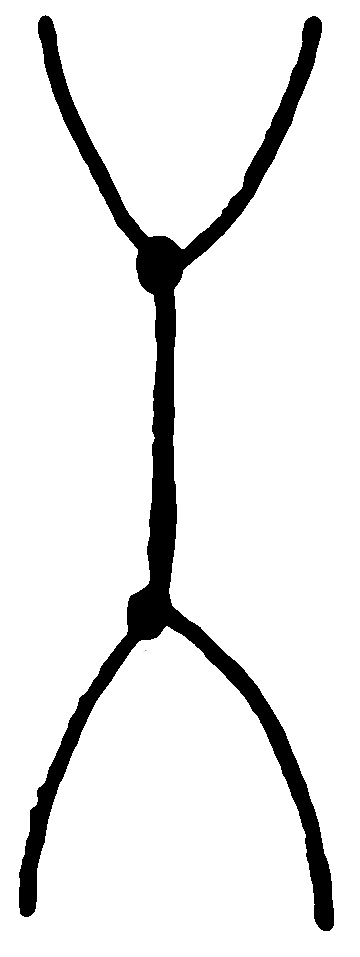}
=
\scanPIC{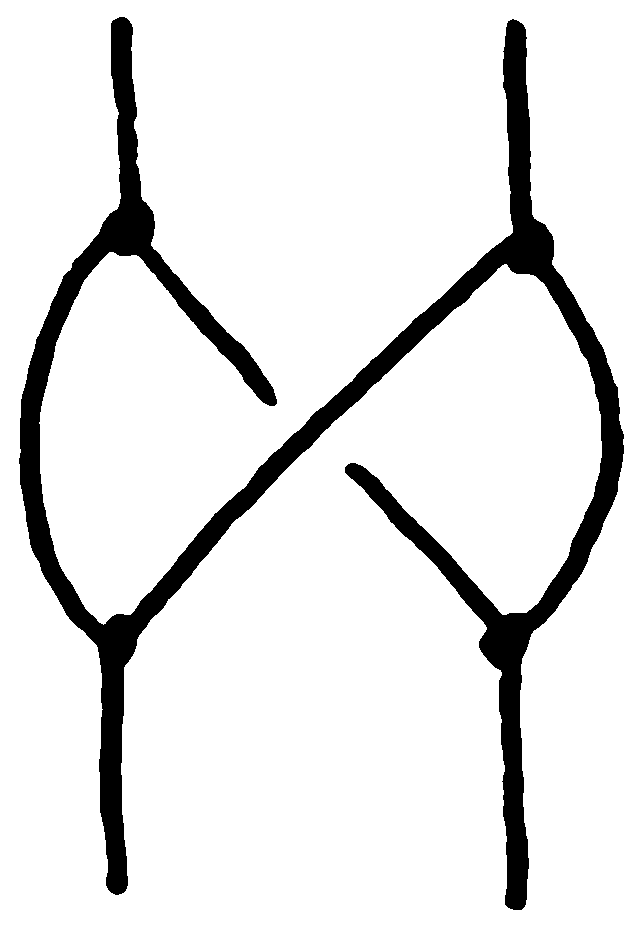}
\qquad
\raisebox{2.5em}{\text{b)}}~~
\scanPIC{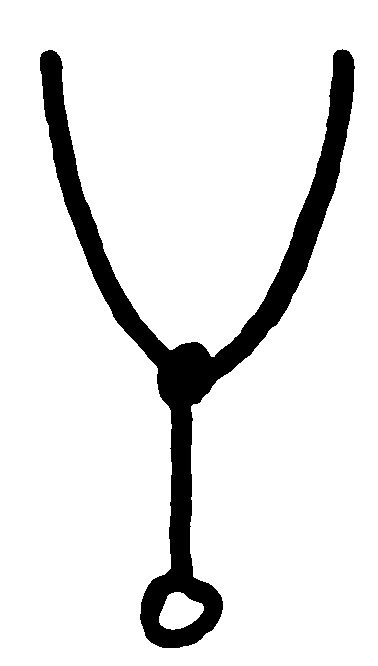}
=
\scanPIC{h02}\scanPIC{h02}
\qquad
\raisebox{2.5em}{\text{c)}}
\scanPIC{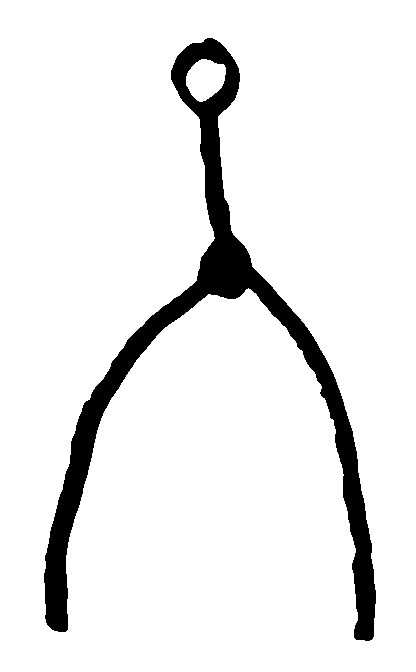}
=
\scanPIC{h03}\scanPIC{h03}
\qquad
\raisebox{2.5em}{\text{d)}}~~
\scanPIC{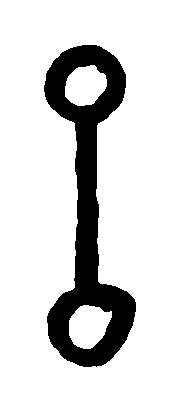} = \id_\one
$$
$$
\raisebox{2.5em}{\text{e)}}~~
\scanPIC{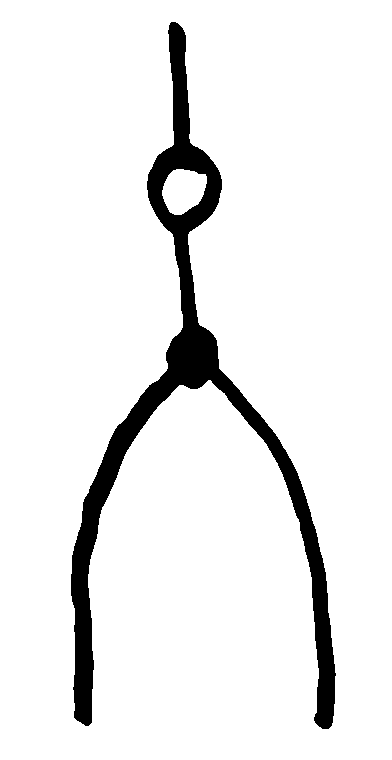}
=
\scanPIC{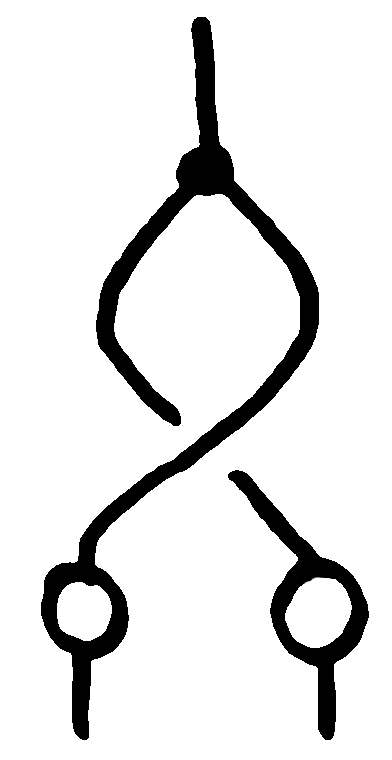}
\qquad
\qquad
\raisebox{2.5em}{\text{f)}}~~
\scanPIC{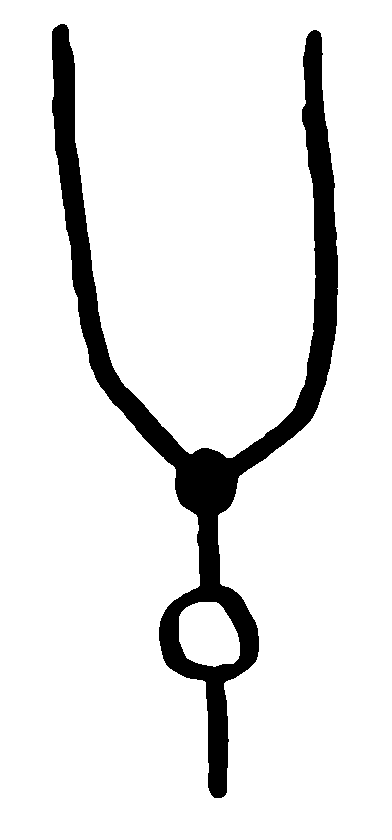}
=
\scanPIC{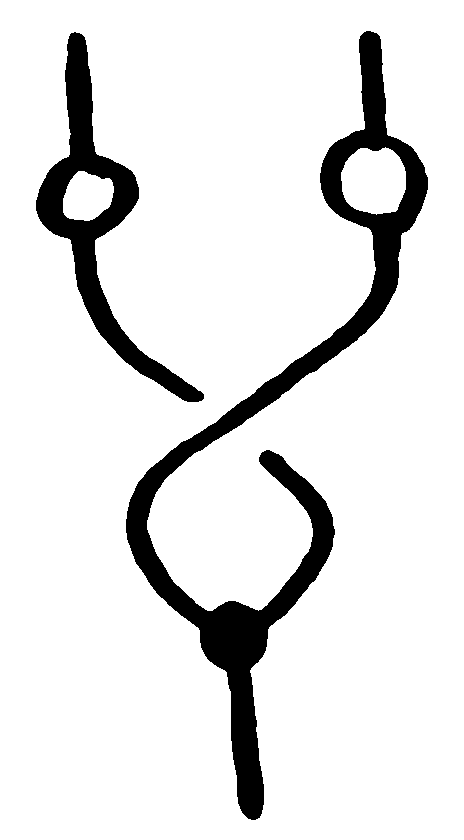}
$$
$$
\raisebox{2.5em}{\text{g)}}~~
\scanPIC{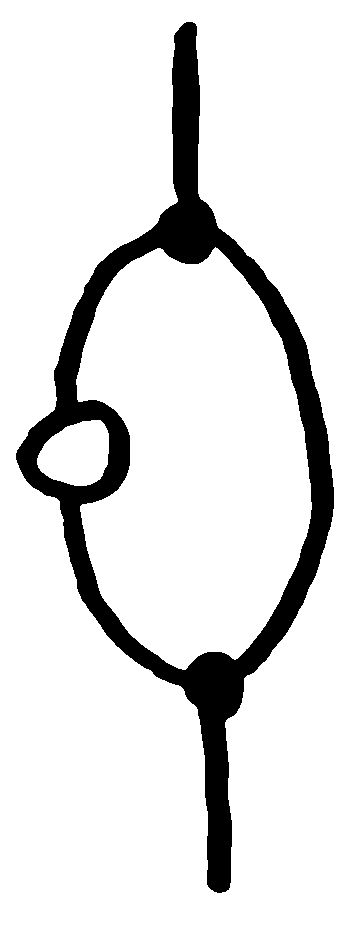}
=
\scanPIC{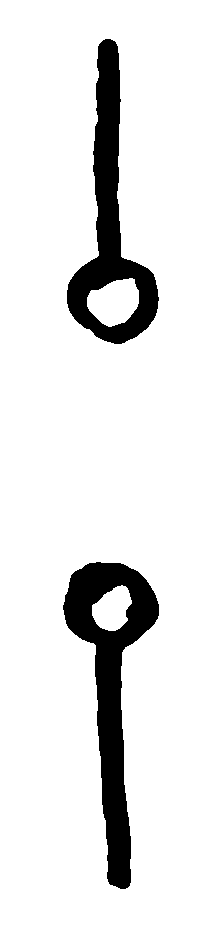}
=
\scanPIC{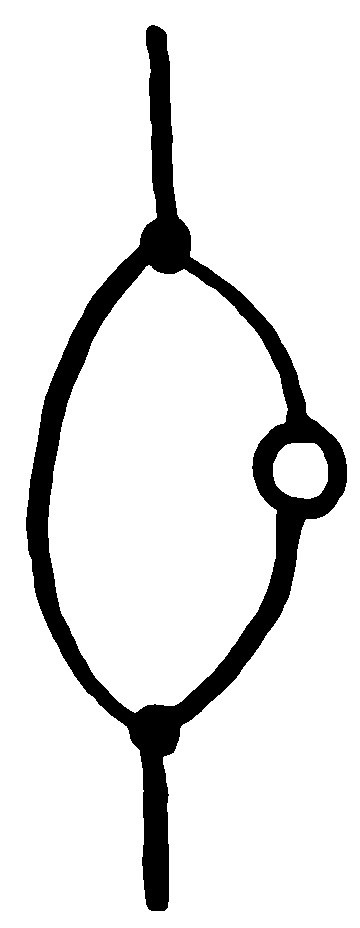}
\qquad
\qquad
\raisebox{2.5em}{\text{h)}}~~
\scanPIC{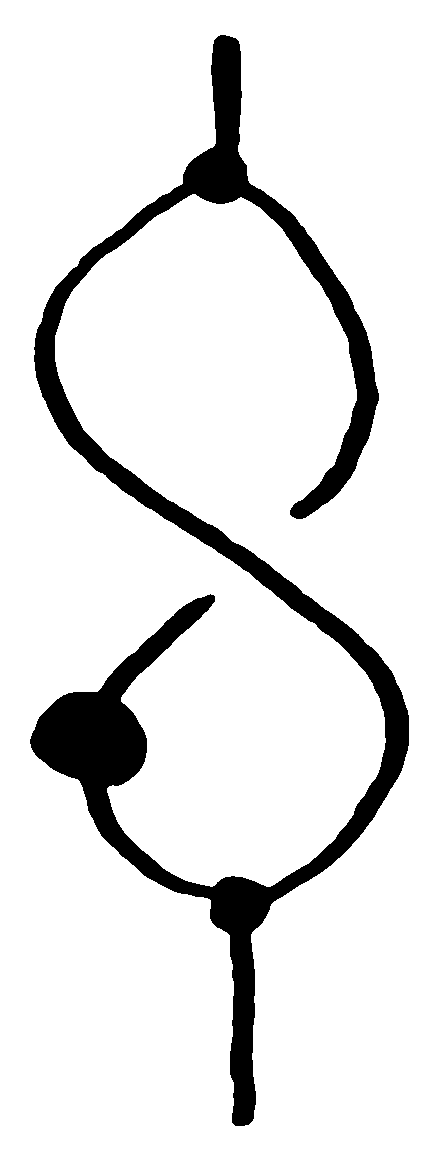}
=
\scanPIC{p12}
=
\scanPIC{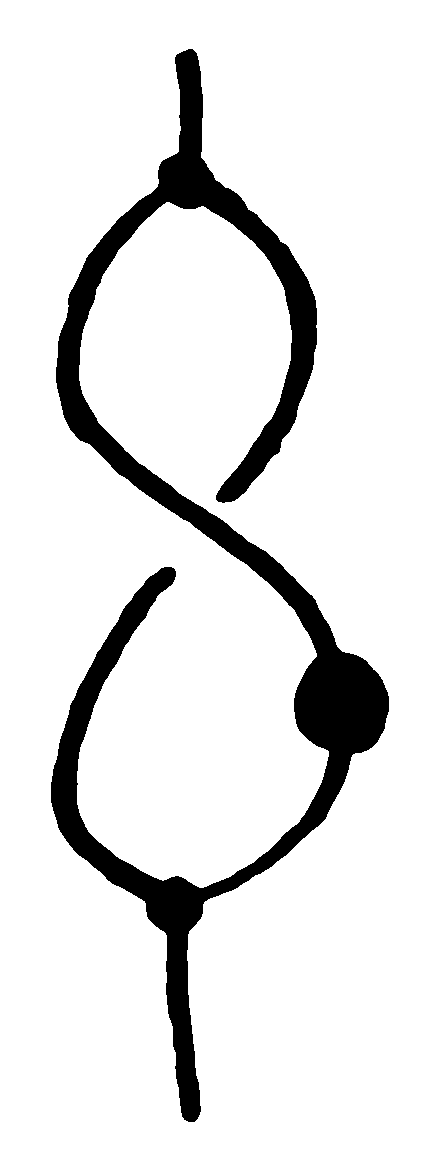}
$$
\caption{Properties of Hopf algebras: a)--d) state that $\Delta$ and $\eps$ are algebra maps. e) and f) give the sense in which $S$ is an algebra-anti-automorphism. g) is the `bubble-property' of the antipode $S$, h) is the corresponding property for $S^{-1}$.}
\label{fig:Hopf-property}
\end{figure}

Given that $S$ is invertible, one checks the counterpart of the bubble-property for $S^{-1}$ which is shown in Figure \ref{fig:Hopf-property}\,h).

\medskip

The {\em dual Hopf algebra} $H^\vee$ is defined to have structure morphisms\footnote{
  This convention agrees with \cite[Sect.\,2.2]{Majid:1995} and it is the natural one to use in the categorical formulation. However, if applied to elements it may look slightly unusual. Consider, for example, $\Sc = \vect(k)$ and let $H$ be a finite-dimensional Hopf algebra. Take $x,y \in H$ and $\alpha,\beta \in H^*$. In the present convention, one has $\big(\Delta_{H^\vee}(\alpha)\big)(x \otimes y) = \alpha\big(\, \mu(y \otimes x) \,\big)$ (note the switch).}  
\begin{align} 
  \mu_{H^\vee} ~&=~ \scanpic{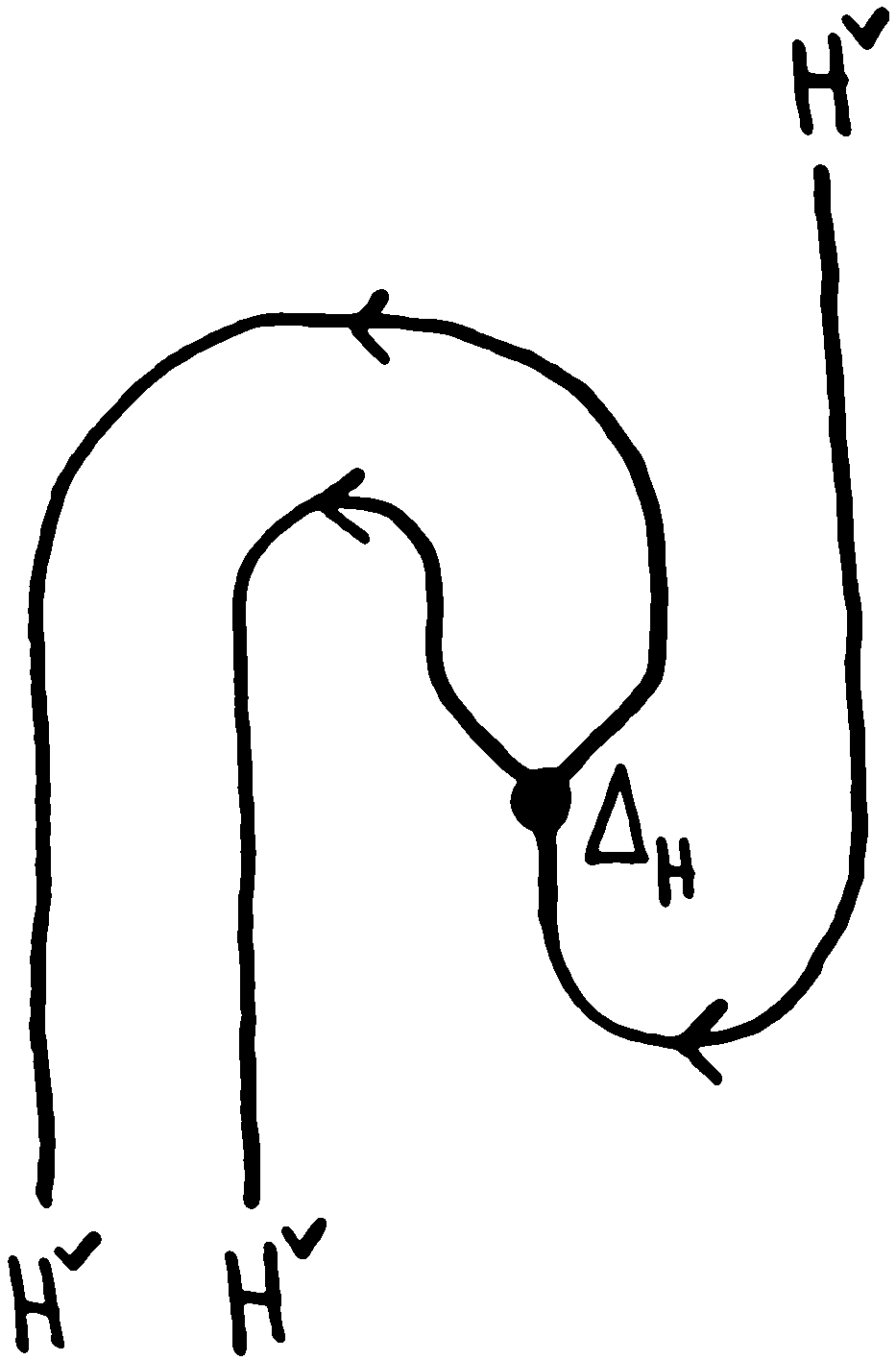} 
  \quad , & \quad
  \Delta_{H^\vee} ~&=~ \scanpic{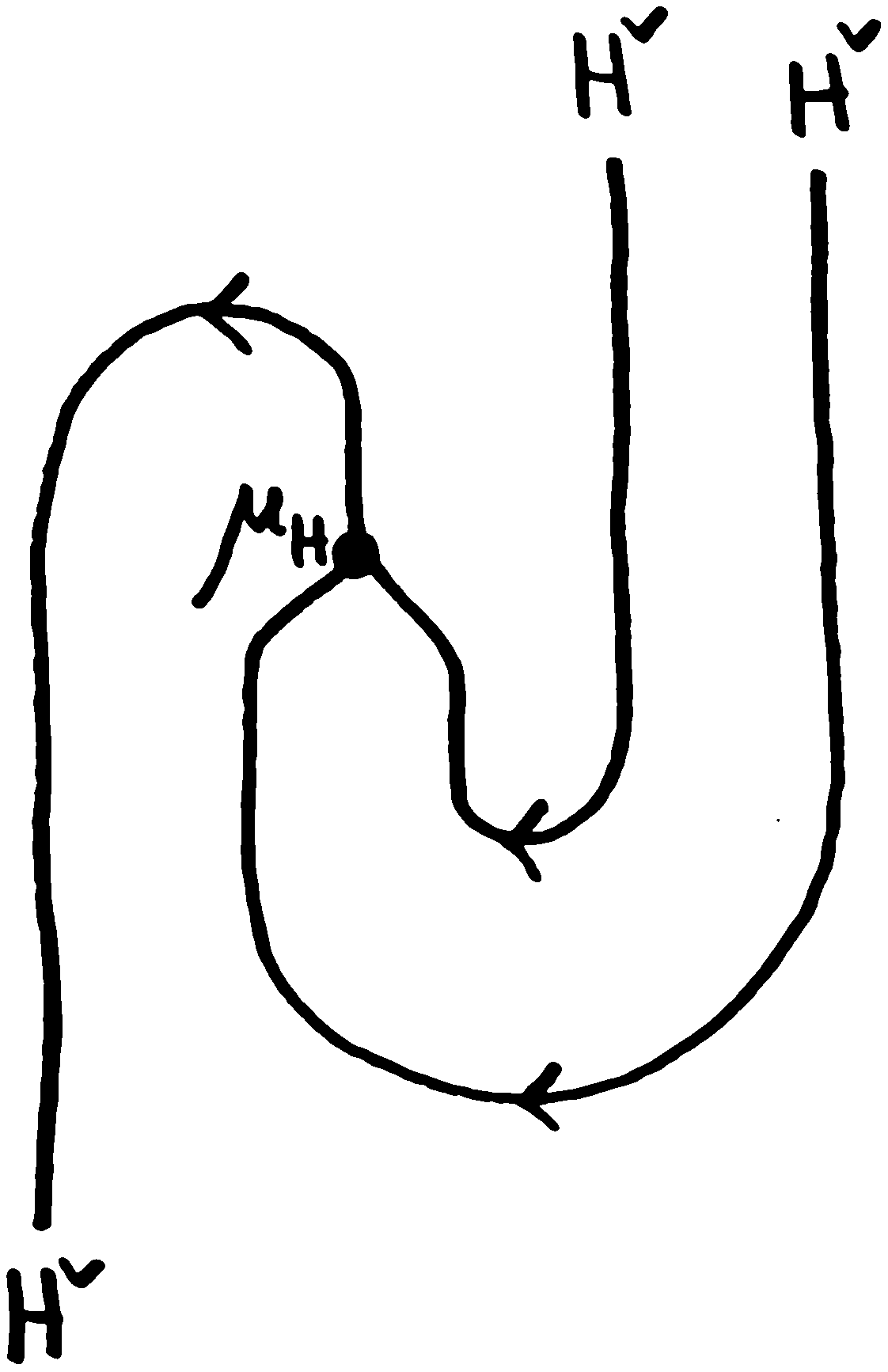} 
  \quad ,
  \nonumber\\
  \eta_{H^\vee} ~&=~ (\eps_H)^\vee
  \quad , & \quad
  \eps_{H^\vee} ~&=~ (\eta_H)^\vee
  \quad , \quad
  & S_{H^\vee} ~&=~ (S_H)^\vee \ .
\label{eq:H^-structuremaps}
\end{align}

Recall that for a braided category $\Sc$,  the reverse category $\overline{\Sc}$ is the category $\Sc$ equipped with the same tensor product and associativity isomorphisms, but with the new braiding
$\overline{c}_{X,Y} = c_{Y,X}^{-1}$. If $\Sc$ is ribbon, then $\overline{\Sc}$ also has inverse twist isomorphisms.

Given a Hopf-algebra $H$, we obtain two new Hopf algebras in $\overline{\Sc}$ (cf.\ \cite[Sect.\,2.3]{Bespalov:1995}):
\begin{itemize}
\item
$H^\mathrm{op}$: the Hopf algebra with the opposite product. The underlying object is $H$ and the structure morphisms are
\be \label{eq:Hop-structure-maps}
\mu_{H^\mathrm{op}} = \mu_H \circ \overline{c}_{H,H} 
~~,~~~
\eta_{H^\mathrm{op}} = \eta_H
~~,~~~
\Delta_{H^\mathrm{op}} = \Delta_H
~~,~~~
\eps_{H^\mathrm{op}} = \eps_H
~~,~~~
S_{H^\mathrm{op}} = S_H^{-1} \ .
\ee
The use of $\overline{c}_{H,H} = c_{H,H}^{-1}$ is dictated by Figure \ref{fig:Hopf-property}\,h), which guarantees that the morphisms \eqref{eq:Hop-structure-maps} also satisfy condition 4 in Definition \ref{def:braided-Hopf}.
\item
$H_\mathrm{cop}$: the Hopf algebra with the opposite coproduct. The underlying object is again $H$ and the structure morphisms are
\be \label{eq:Hcop-structure-maps}
\mu_{H_\mathrm{cop}} = \mu_H 
~~,~~~
\eta_{H_\mathrm{cop}} = \eta_H
~~,~~~
\Delta_{H_\mathrm{cop}} = \overline{c}_{H,H} \circ \Delta_H
~~,~~~
\eps_{H_\mathrm{cop}} = \eps_H
~~,~~~
S_{H_\mathrm{cop}} = S_H^{-1} \ .
\ee
\end{itemize}
The fact that $H^\mathrm{op}$ and $H_\mathrm{cop}$ are Hopf algebras in $\overline{\Sc}$ (rather than $\Sc$) becomes apparent when checking that the coproduct is an algebra map. Iterating the procedure gives Hopf algebras $(H^\mathrm{op})_\mathrm{cop}$ and $(H_\mathrm{cop})^\mathrm{op}$ in $\Sc$
, which are in general not equal (but isomorphic via the twist). The antipode provides an isomorphism of Hopf algebras in $\Sc$:  
\be
  S : (H^\mathrm{op})_\mathrm{cop} \xrightarrow{~\sim~} H \ .
\ee 
In any case, we will only use $H^\mathrm{op}$ and $H_\mathrm{cop}$ for symmetric $\Sc$, where $\overline\Sc = \Sc$ and $(H^\mathrm{op})_\mathrm{cop} = (H_\mathrm{cop})^\mathrm{op} =: H_\mathrm{cop}^\mathrm{op}$.

\medskip

The left and right multiplication by an `element' $x$ of $H$ will be denoted by ${}_xM$ and $M_x$, that is, for $x : \one \to H$,
\be\label{eq:left-right-mult-def}
  {}_xM = \mu \circ (x \otimes \id) ~:~ H \to H
  \qquad , \quad
  M_x = \mu \circ (\id \otimes x) ~:~ H \to H \ .
\ee
If there is an $x' : \one \to H$ such that $\mu \circ (x \otimes x') = \eta = \mu \circ (x' \otimes x)$, we say that $x$ has a {\em multiplicative inverse}, and we denote it by $x^{-1} = x'$. Since for $H \ncong \one$, a morphism $\one \to H$ never has an inverse morphism $H \to \one$, and since for the Hopf algebra $H=\one$ the two inverses agree, we hope that the notation $x^{-1}$ for the multiplicative inverse is not confusing. 

Given $x : \one \to H$ with multiplicative inverse, we write 
\be\label{eq:Ad_x-def}
  \Ad_x = {}_xM \circ M_{x^{-1}}  :  H \to H  
\ee
for the conjugation with $x$.

\medskip

We will also make use of left/right integrals and cointegrals. A morphism $\Lambda : \one \to H$ (respectively a morphism $\lambda : H \to \one$) is a left or right integral (respectively a left or right cointegral) if
\be
\begin{array}{l@{\hspace{3em}}l@{\hspace{3em}}l@{\hspace{3em}}l}
\text{left integral:}
&
\text{right integral:}
&
\text{left cointegral:}
&
\text{right cointegral:}
\\[.5em]
\scanpic{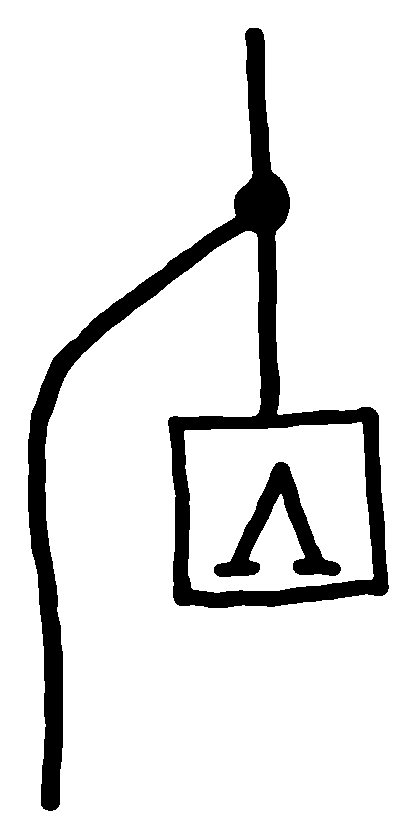}
=
\scanpic{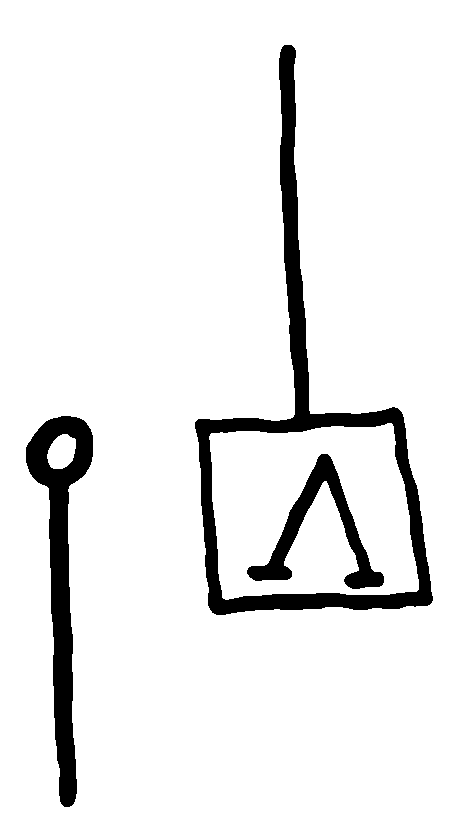}
&
\scanpic{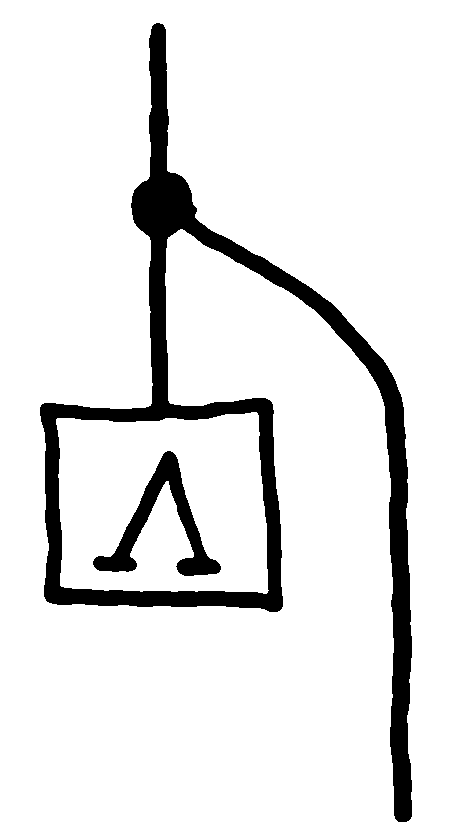}
=
\scanpic{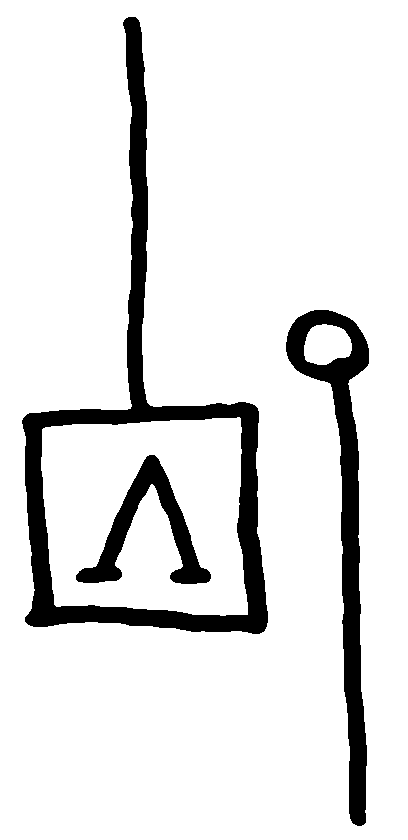}
&
\scanpic{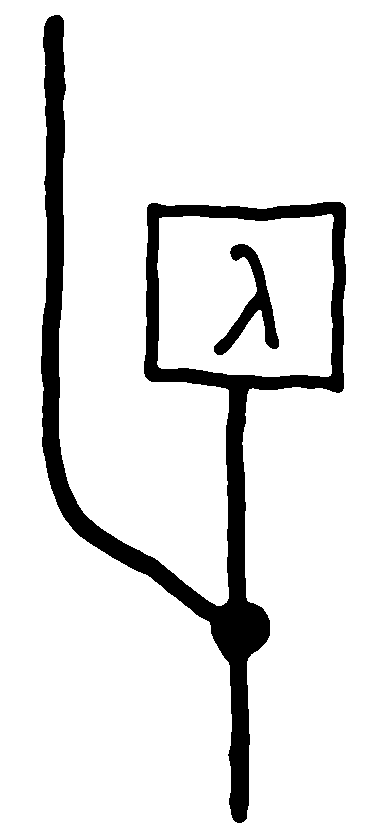}
=
\scanpic{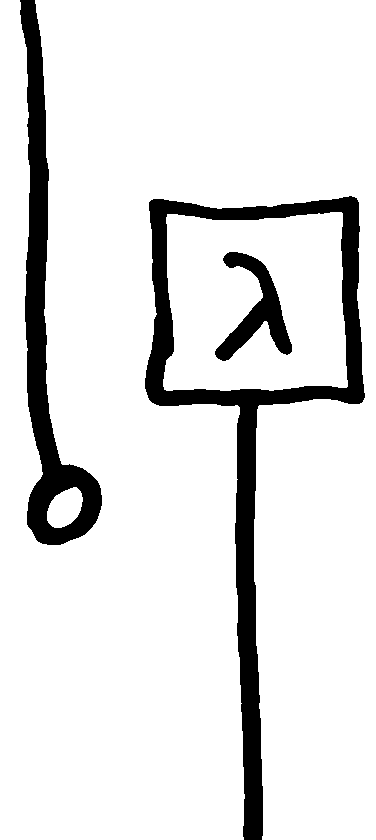}
&
\scanpic{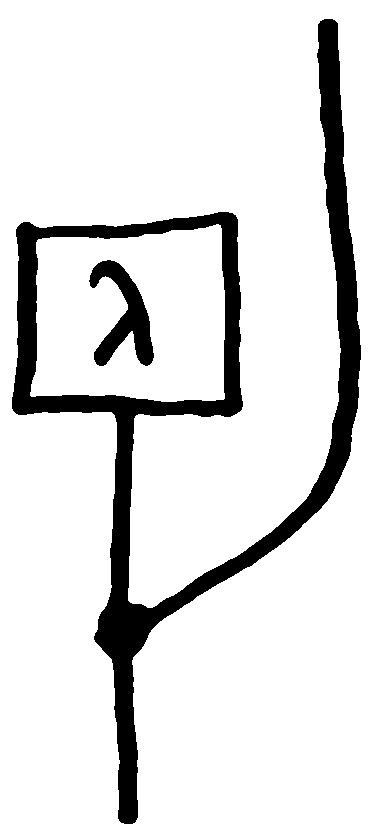}
=
\scanpic{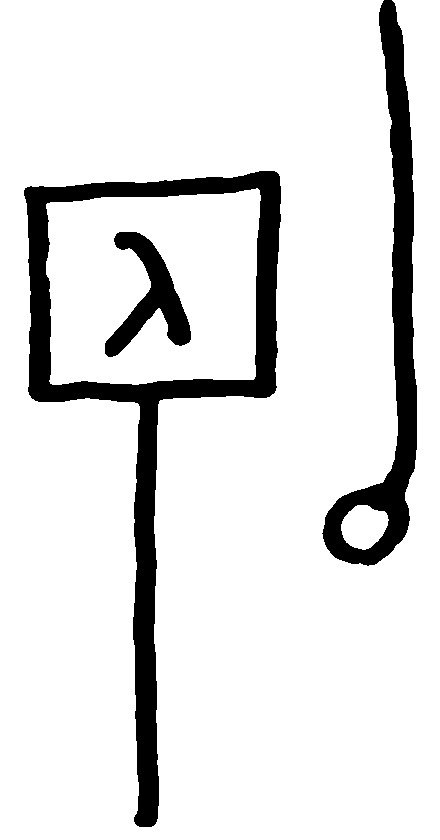}
\end{array}
\ee

A {\em module} of a Hopf algebra $H$ is an object $M \in \Sc$ together with a morphism $\rho : H \otimes M \to M$ such that the associativity condition $\rho \circ (\id_H \otimes \rho) = \rho \circ (\mu \otimes \id_M)$ and the unit condition $\rho \circ (\eta \otimes \id_M) = \id_M$ hold. The string diagram notation we use for the action morphism $\rho$ is
\be
  \rho ~=~ \scanpic{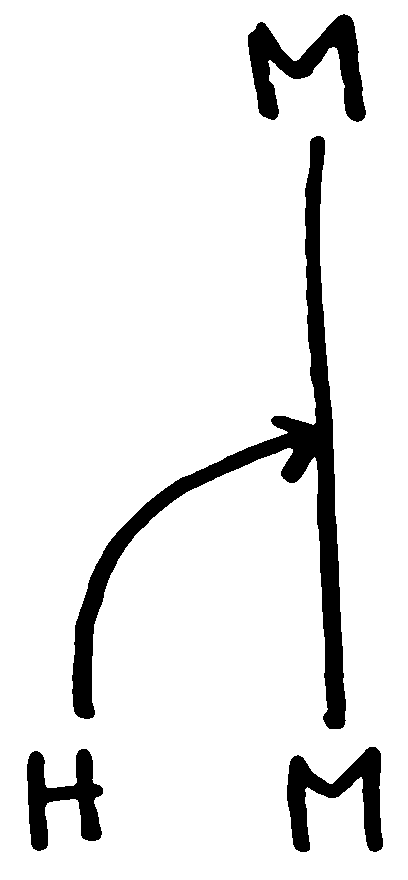} \quad .
\ee
If we want to stress the algebra and the module it is acting on, we write $\rho^H_M$ instead of $\rho$. Given two $H$-modules $M$, $N$, the action of $H \otimes H$ on $M \otimes N$ is, in accordance with \eqref{eq:mult-on-product},
\be
  \rho^{H \otimes H}_{M \otimes N} = (\rho^H_M \otimes \rho^H_N) \otimes (\id_H \otimes c_{H,M} \otimes \id_N) \ .
\ee

The category of $H$-modules is denoted by $\Rep_\Sc(H)$. It is a monoidal category; the tensor product of two $H$-modules $M$, $N$ has $M \otimes N$ as underlying object. The $H$-action is given by the coproduct, 
\be
  \rho^H_{M \otimes N} =  \rho^{H \otimes H}_{M \otimes N} \circ (\Delta \otimes \id_M \otimes \id_N) \ .
\ee
The forgetful functor $F : \Rep_\Sc(H) \to \Sc$ is strict monoidal. 

In the presentation of our results and later in the proof we will need the following four identities involving $H$-modules.

\begin{lemma} \label{lem:H-intertwiner}
Let $M$ be an $H$-module. We have:
$$
\begin{array}{ll}
\raisebox{3em}{\text{a)}}~  
\scanPICC{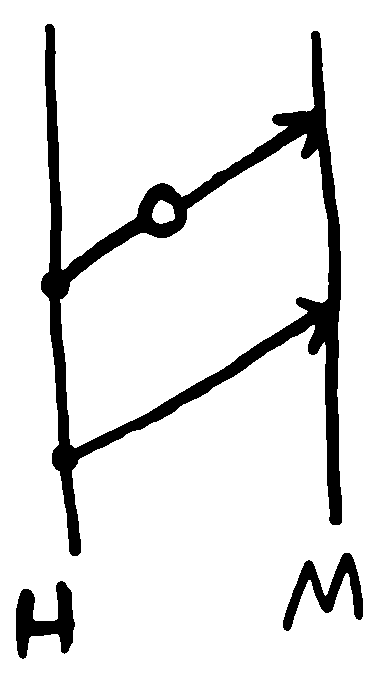}~=~\id_{H \otimes M}~=~\scanPICC{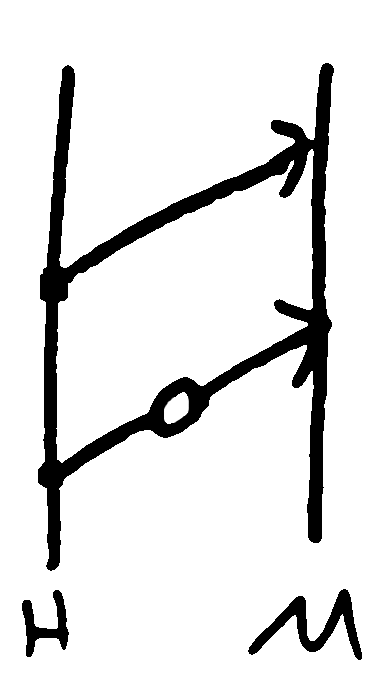} \qquad
&
\raisebox{3em}{\text{b)}}~  
\scanPIC{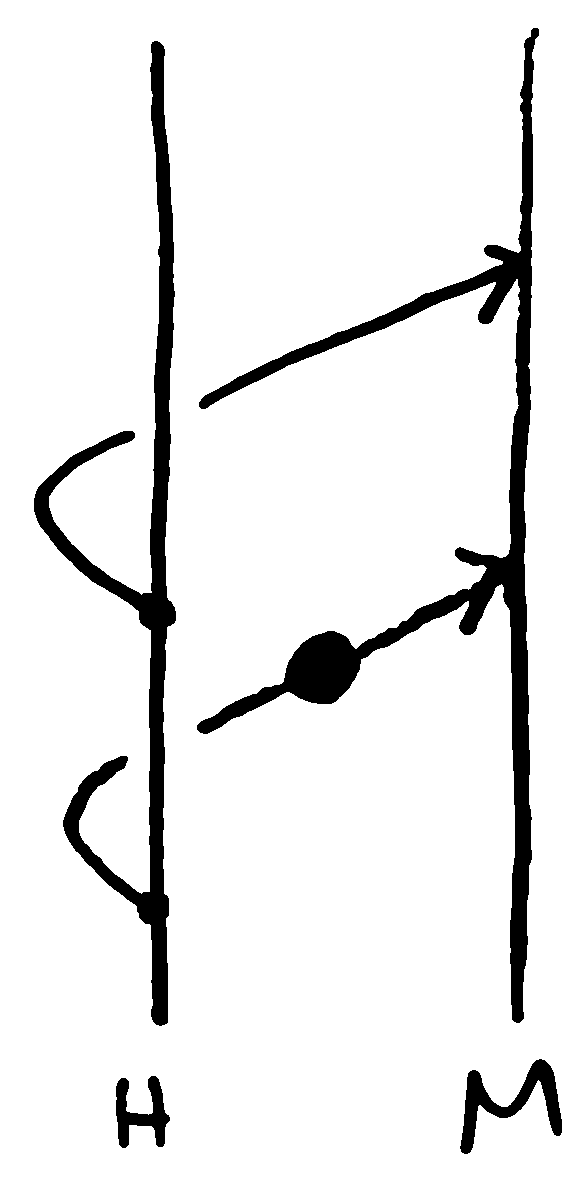}~=~\id_{H \otimes M}~=~\scanPIC{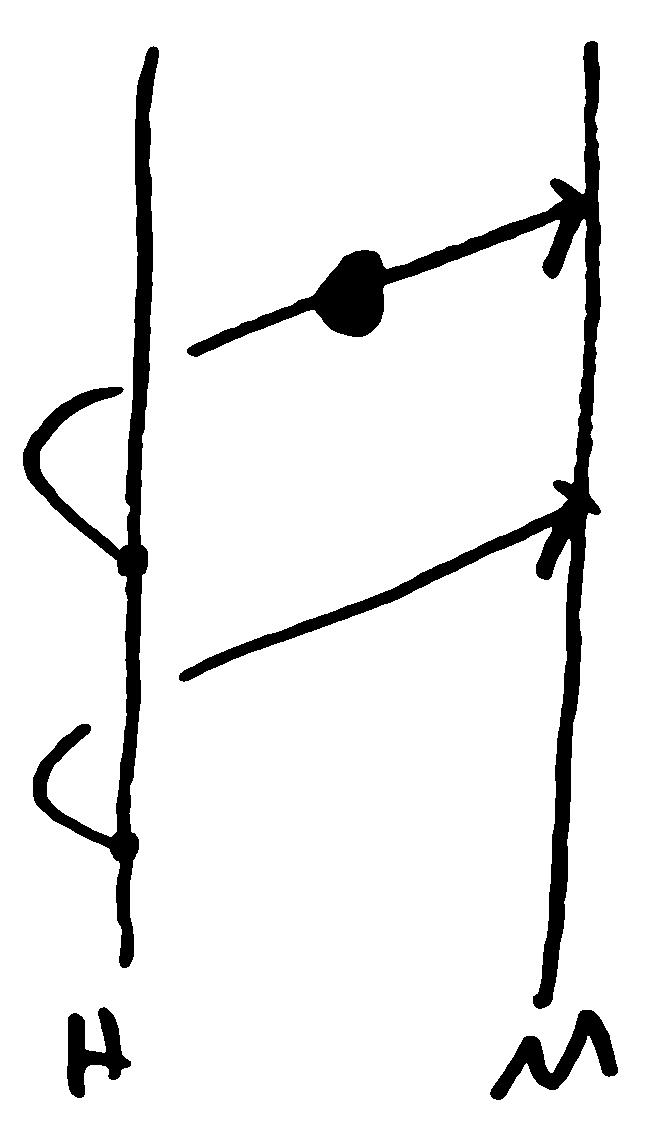}
\\[.5em]
\raisebox{3em}{\text{c)}}~  
\scanPIC{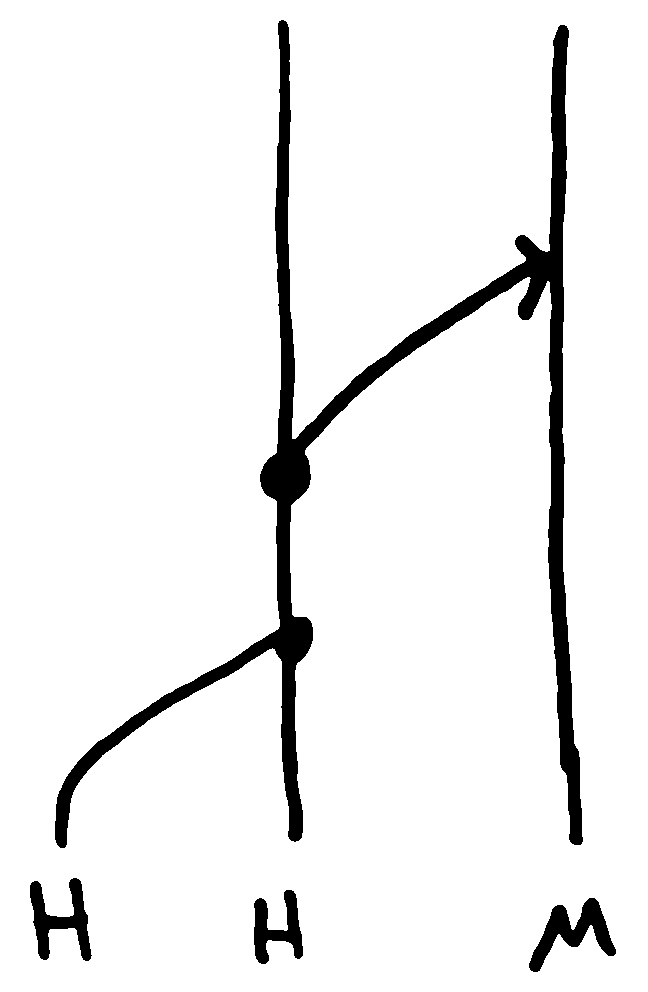}~=~\scanPIC{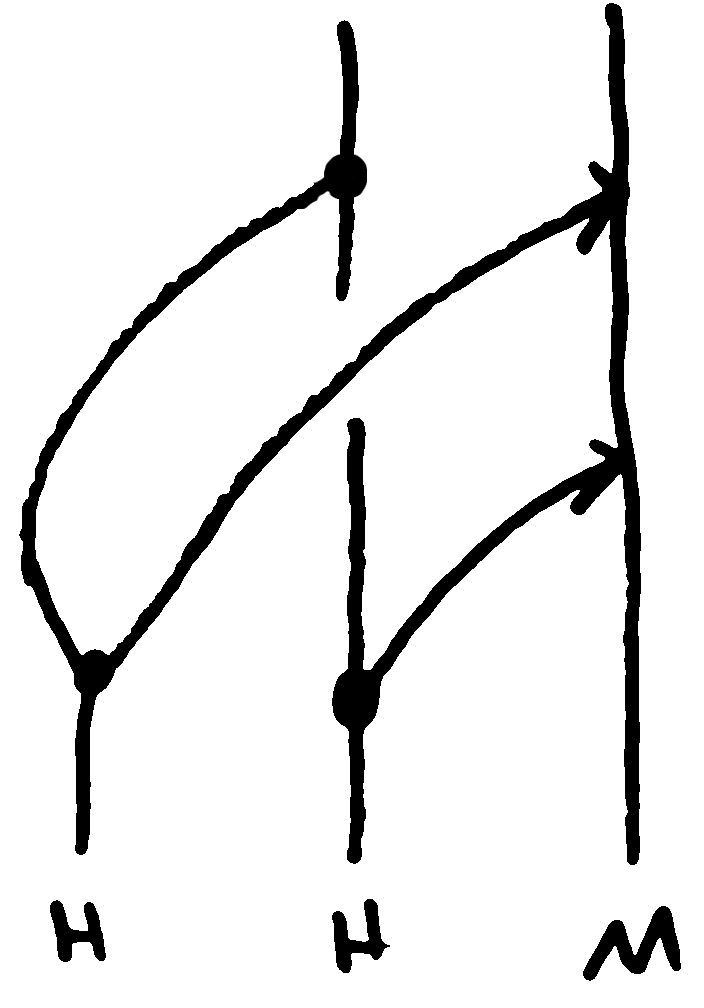} \qquad
&
\raisebox{3em}{\text{d)}}~  
\scanPIC{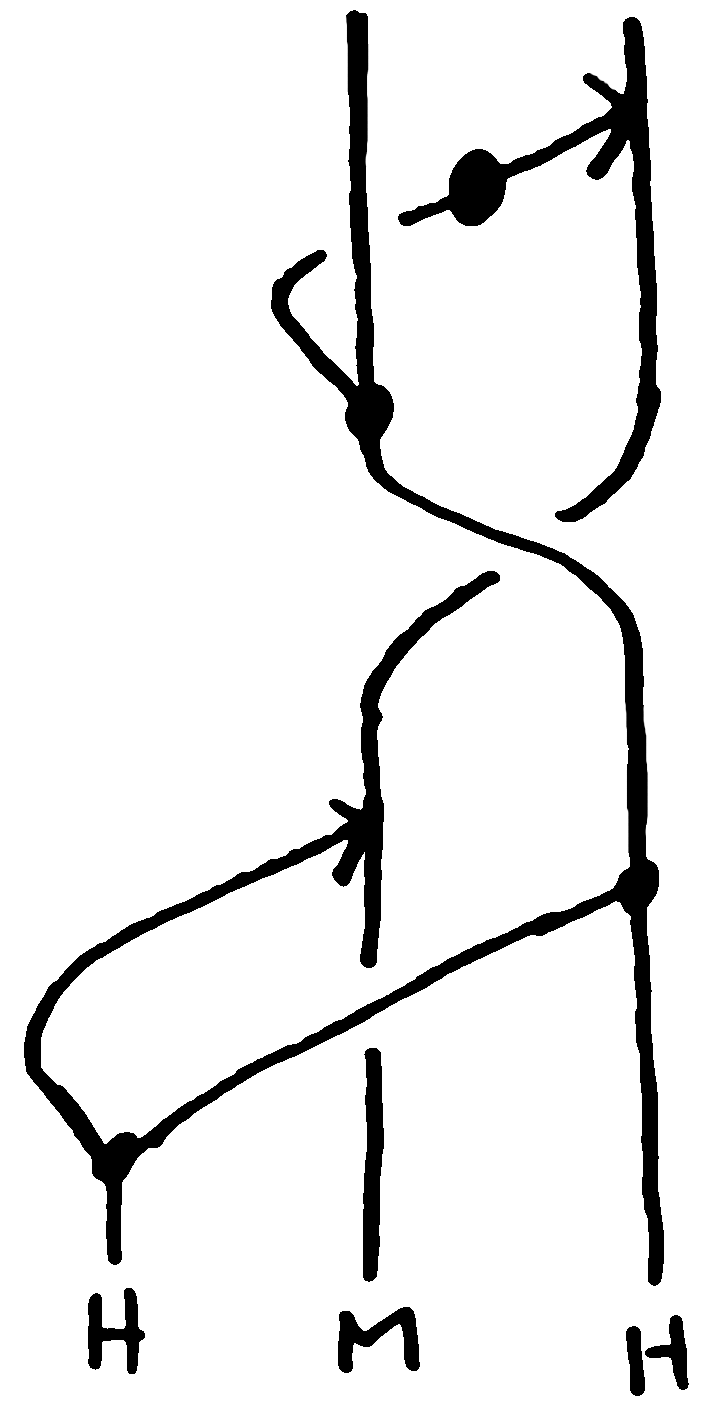}~=~\scanPIC{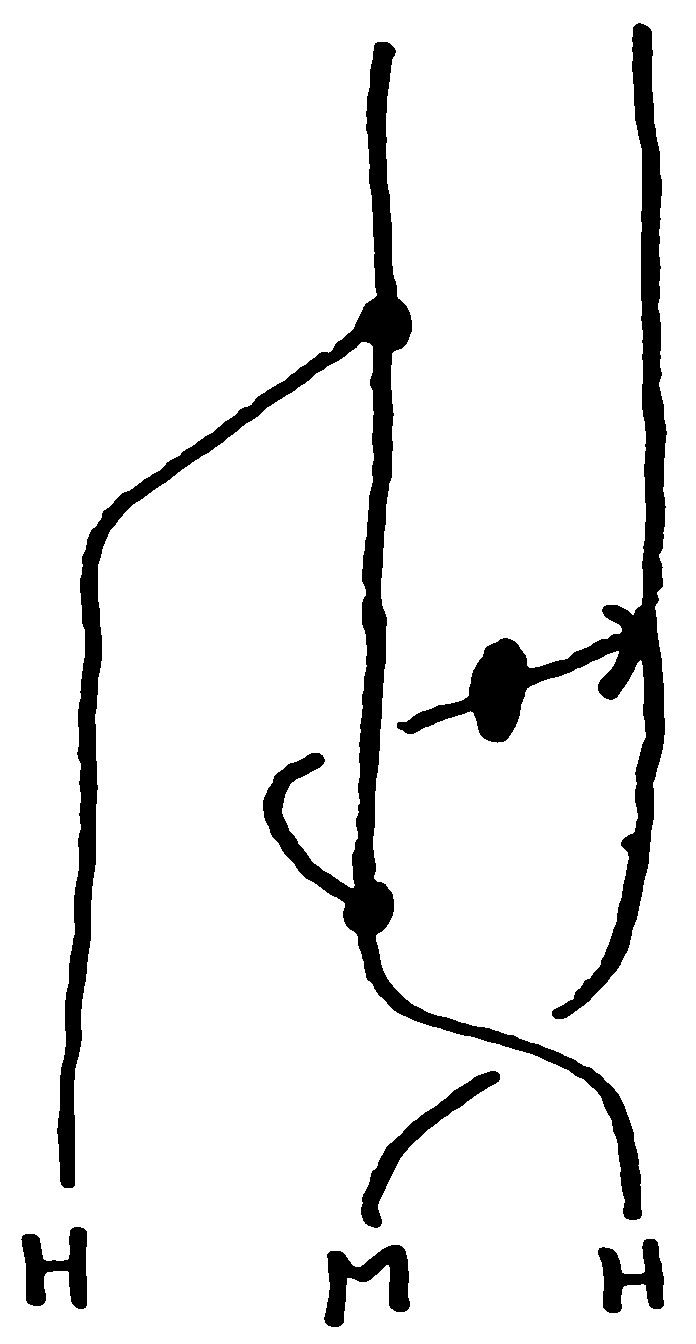}
\end{array}
$$
\end{lemma}

\begin{proof}
Part a) and b) are immediate from coassociativity of $H$ and the associativity of the action on $M$, together with Figure \ref{fig:Hopf-property}\,g),\,h). Part c) is the algebra-map property of $\Delta$, see Figure \ref{fig:Hopf-property}\,a), together with associativity of the action. For part d) we have
\be
  \text{lhs.\ of part d)} 
  ~=~
  \scanpic{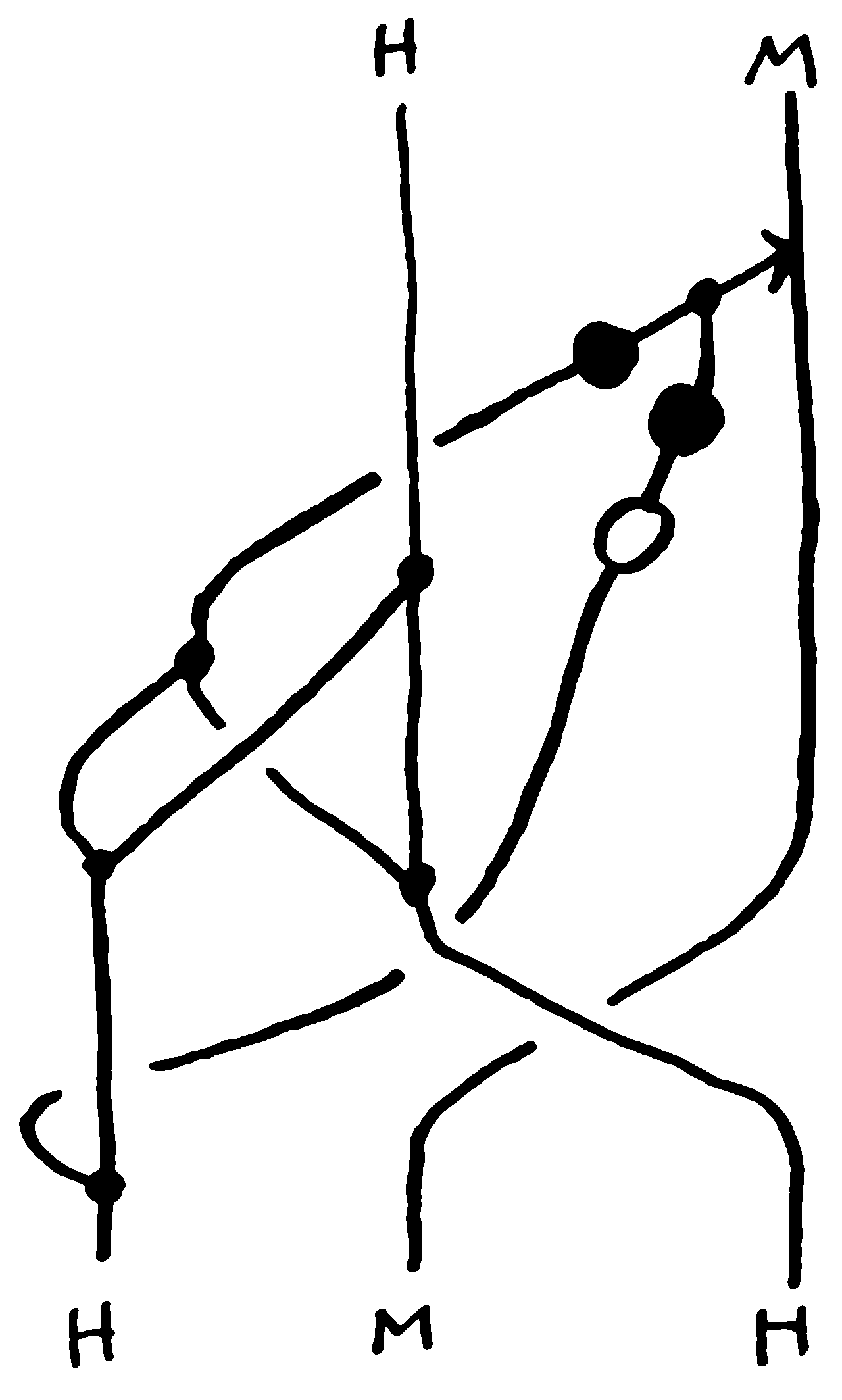}
  ~=~
  \scanpic{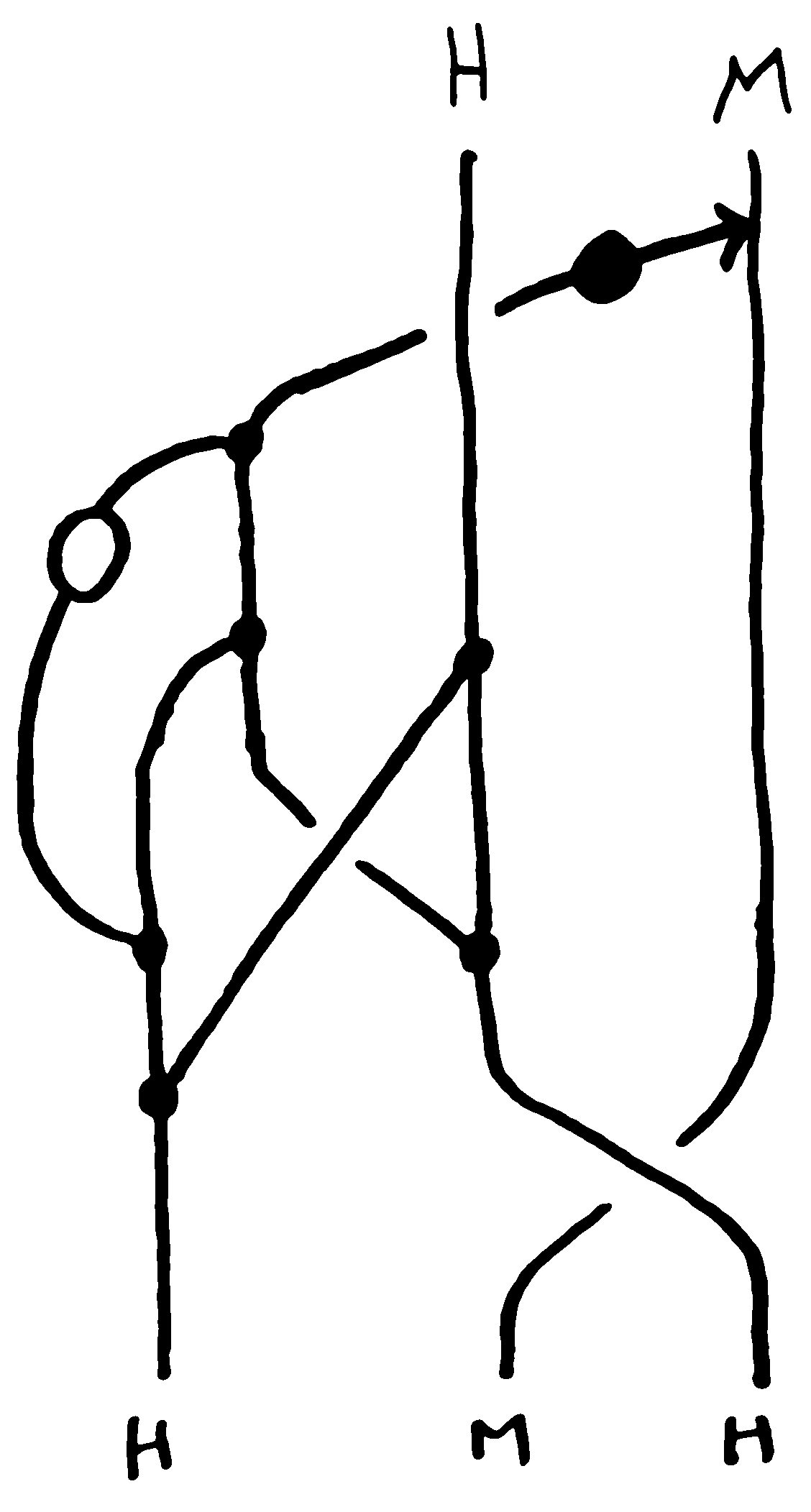}
  ~=~
  \text{rhs.\ of part d)}  \ .
\ee
In the second step the identity which follows from Figure \ref{fig:Hopf-property}\,e) for $S^{-1}$ has been used.
\end{proof}

\subsection{Results}\label{results}

Let $H$ be a Hopf algebra in $\Sc$ with invertible antipode $S_H$.
Consider the $\Zb/2\Zb$-graded category given by\footnote{
	In the general setting used in Sections \ref{prem-res}--\ref{sec:braiding}, where $\Sc$ is just ribbon, the notation `$+$' in $\Cc = \Cc_0 + \Cc_1$ refers to the union $\sqcup$ of the two categories $\Cc_0$ and $\Cc_1$ (where objects and morphisms belong either to $\Cc_0$ or else to $\Cc_1$). However, in examples where $\Sc$ is additive or $k$-linear we implicitly extend the union by mixed direct sums, so that then `$+$' refers to the direct sum $\oplus$ of categories.}
\be
  \Cc = \Cc_0 + \Cc_1
  \qquad , ~~ \text{where} \quad \Cc_0 = \Rep_\Sc(H) ~~,~~ \Cc_1 = \Sc \ .
\ee
We would like to extend the monoidal structure on $\Rep_\Sc(H)$ to all of $\Cc$. To this end we fix the specific form of the tensor product functor $\otimes_\Cc$ as follows:
\begin{equation}\label{tab:tensorproducts}
\begin{array}{c|c|ll}
 A  &  B  &  A \otimes_{\Cc} B  & \\
\hline
 \Cc_0  &  \Cc_0  &  A \otimes_{\Rep_\Sc(H)} B  &   \in \Cc_0  \\
 \Cc_0  &  \Cc_1  &  F(A) \otimes_{\Sc} B   &   \in \Cc_1 \\ 
 \Cc_1  &  \Cc_0  &  A \otimes_{\Sc} F(B)  &   \in \Cc_1  \\
 \Cc_1  &  \Cc_1  &  H \otimes_{\Sc} A \otimes_{\Sc} B  &   \in \Cc_0 
\end{array}
\qquad
\begin{array}{c|c|l}
 A\xrightarrow{f}A'  &  B\xrightarrow{g}B'  &  f \otimes_{\Cc} g  \\
\hline
 \Cc_0  &  \Cc_0  &  f \otimes_{\Rep_\Sc(H)} g  \\
 \Cc_0  &  \Cc_1  &  F(f) \otimes_{\Sc} g  \\ 
 \Cc_1  &  \Cc_0  &  f \otimes_{\Sc} F(g)  \\
 \Cc_1  &  \Cc_1  &  \id_H \otimes_{\Sc} f \otimes_{\Sc} g 
\end{array}
\end{equation}
Here, $F : \Rep_\Sc(H) \to \Sc$ denotes the forgetful functor and the left $H$-action on $H \otimes_{\Sc} A \otimes_{\Sc} B$ in the last line of the left table is by the multiplication of $H$.

\begin{notation} \label{not:sec3} 
In addition to Notations \ref{not:sec2} by $A,B,\dots$ we denote objects from $\Cc$, and by $A^i, B^i, \dots$ we refer to objects from $\Cc_i$. We will not make the forgetful functor $F:\Rep_{\Sc}(H) \to \Sc$  explicit  in string diagrams. All string diagrams below are string diagrams in $\Sc$ and our conventions for string diagrams are given in Figure \ref{fig:Hopf-graph-convention}.
\end{notation}

Next we
list our ansatz for the associativity natural isomorphisms $\alpha_{A,B,C} : A \otimes_{\Cc} (B \otimes_{\Cc} C) \to ( A \otimes_{\Cc} B ) \otimes_{\Cc} C$. They are defined in terms of three morphisms, namely
\be\label{eq:gdp-def}
  \gamma : \one \to H \otimes H \quad , ~~
  \delta : \one \to H \otimes H\quad , ~~
  \phi : H \to H \ , 
\ee
where
\begin{itemize}
\item $\gamma$ has a multiplicative inverse in $H \otimes H$, i.e.\ there exists a $\gamma^{-1} : \one \to H \otimes H$ such that $\mu_{H \otimes H} \circ (\gamma \otimes \gamma^{-1}) = \eta \otimes \eta = \mu_{H \otimes H} \circ (\gamma^{-1} \otimes \gamma)$,
\item $(S^{-1} \otimes \id) \circ \delta$ has a multiplicative inverse in $H \otimes H$,
\item $\phi$ is invertible.
\end{itemize}
These conditions guarantee that the $\alpha_{A,B,C}$ given below are indeed isomorphisms. We will order the discussion of the different associativity isomorphisms by the number of objects taken from $\Cc_1$.

\nxt $0$ objects from $\Cc_1$: 
The associativity isomorphisms are those of $\Rep_\Sc(H)$, i.e.\  $\alpha_{A^0,B^0,C^0} = \id_{A^0} \otimes \id_{B^0} \otimes \id_{C^0}$.

\medskip

\nxt $1$ object from $\Cc_1$: In this case $A \otimes_{\Cc} (B \otimes_{\Cc} C)$ and $( A \otimes_{\Cc} B ) \otimes_{\Cc} C$ are in $\Cc_1$, and the tensor product is that in $\Sc$ with $F$ applied to the two objects from $\Cc_0$. Our ansatz for the associativity isomorphisms allows the $\alpha$ with the middle entry from $\Cc_1$ to be non-trivial:
\be\begin{array}{rcl}
\alpha_{A^0,B^0,C^1} &=& \id_{F(A^0)} \otimes \id_{F(B^0)} \otimes \id_{C^1} \ , \\[.3em]
\alpha_{A^0,B^1,C^0} &=& 
(\id_{F(A^0)} \otimes c_{F(C^0),B^1}) 
\circ \Big[F\big(\rho_{A^0 \otimes C^0}^{H \otimes H} \circ (\gamma \otimes \id_{A^0}\otimes \id_{C^0})\big) \otimes \id_{B^1}\Big] \\[.3em]
 && \hspace{18em} \circ (\id_{F(A^0)} \otimes c^{-1}_{F(C^0),B^1}) \ ,
\\[.3em]
\alpha_{A^1,B^0,C^0} &=& \id_{A^1} \otimes \id_{F(B^0)} \otimes \id_{F(C^0)} \ .
\end{array}
\ee
The structure of $\alpha_{A^0,B^1,C^0}$ may be easier to understand in graphical notation. As mentioned in Notations \ref{not:sec3}, in string diagrams we do not spell out the forgetful functor.
\be\label{fig:assoc-b}
  \alpha_{A^0,B^1,C^0} ~=~ \scanPIC{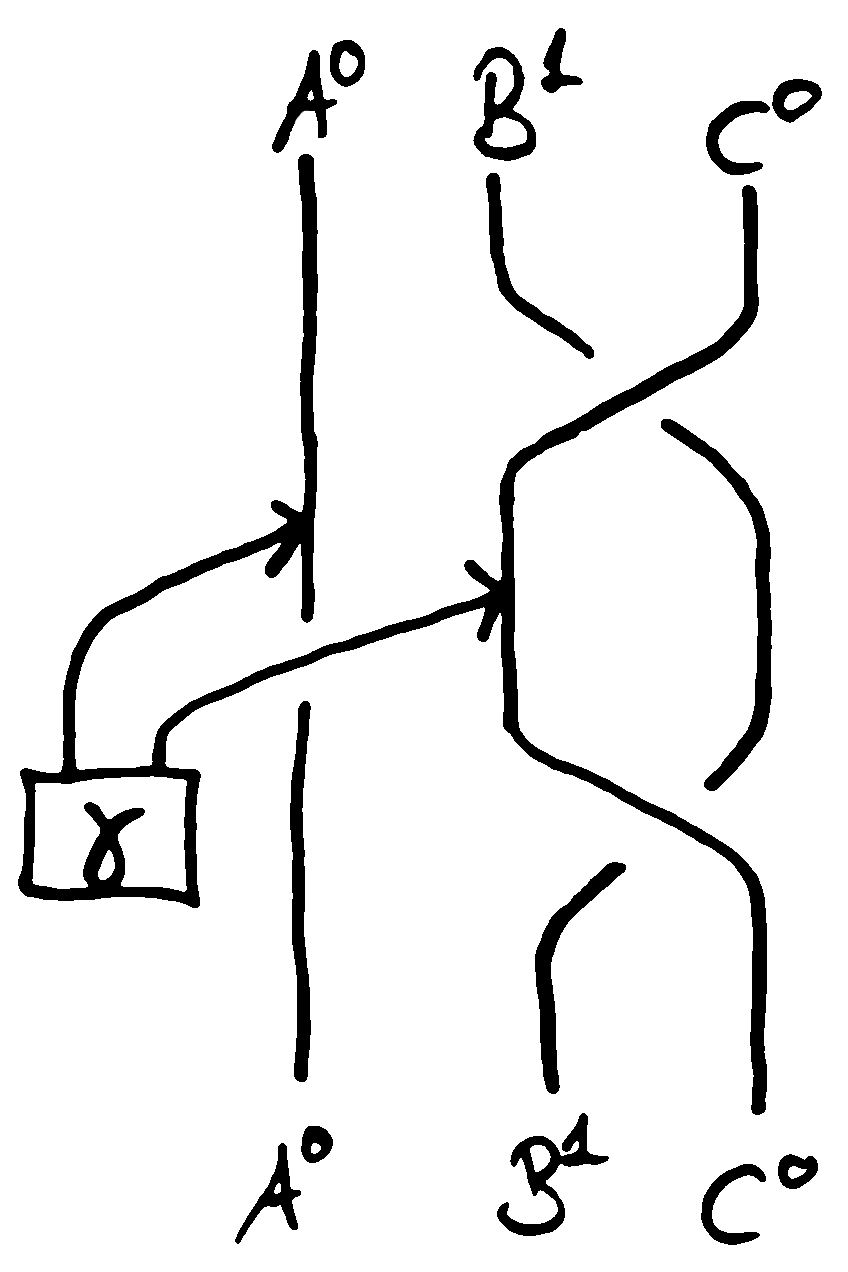}  ~=~   \scanPIC{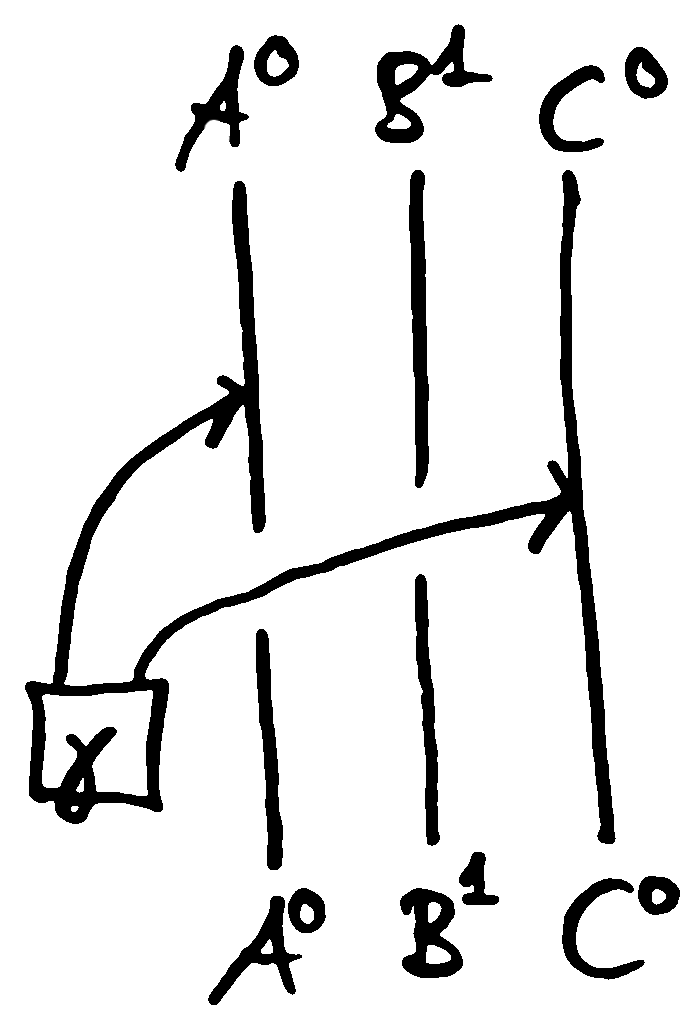}
\quad . 
\ee
Clearly, $\alpha_{A^0,B^0,C^1}$ and $\alpha_{A^1,B^0,C^0}$ are natural and isomorphisms. For $\alpha_{A^0,B^1,C^0}$, naturality in $A^0$ and $C^0$ follows since the action $\rho_{A^0 \otimes C^0}^{H \otimes H}$ commutes with $H$-module maps $f : A^0 \to {A^0}'$ and $g : C^0 \to {C^0}'$. The invertibility of $\alpha_{A^0,B^1,C^0}$ follows from our assumption that $\gamma$ has a multiplicative inverse in $H \otimes H$.

\medskip

\nxt $2$ objects from $\Cc_1$:
In this sector, the construction is more involved. Let us start with $\alpha_{A^0,B^1,C^1}$. Its source and target objects in $\Cc_0$ are:
\be
\begin{array}{ll}
 A^0 \otimes_{\Cc} (B^1 \otimes_{\Cc} C^1) 
 = A^0 \otimes_{\Cc} (H \otimes B^1 \otimes C^1) 
 = A^0 \otimes_{\Rep_{\Sc}(H)} (H \otimes B^1 \otimes C^1) \ ,
\\[.3em]
( A^0 \otimes_{\Cc} B^1 ) \otimes_{\Cc} C^1 = (F(A^0) \otimes B^1) \otimes_{\Cc} C^1 = H \otimes F(A^0) \otimes B^1 \otimes C^1 \ .
\end{array}
\ee 
In particular, in $A^0 \otimes_{\Cc} (B^1 \otimes_{\Cc} C^1)$ the left $H$-action is obtained from the modules $A^0$ and $H$, while in  $( A^0 \otimes_{\Cc} B^1 ) \otimes_{\Cc} C^1$ the left $H$-action is just on $H$. By Lemma \ref{lem:H-intertwiner}\,d), the string diagram 
\be
\label{fig:assoc-c1}
\alpha_{A^0,B^1,C^1} = \scanPIC{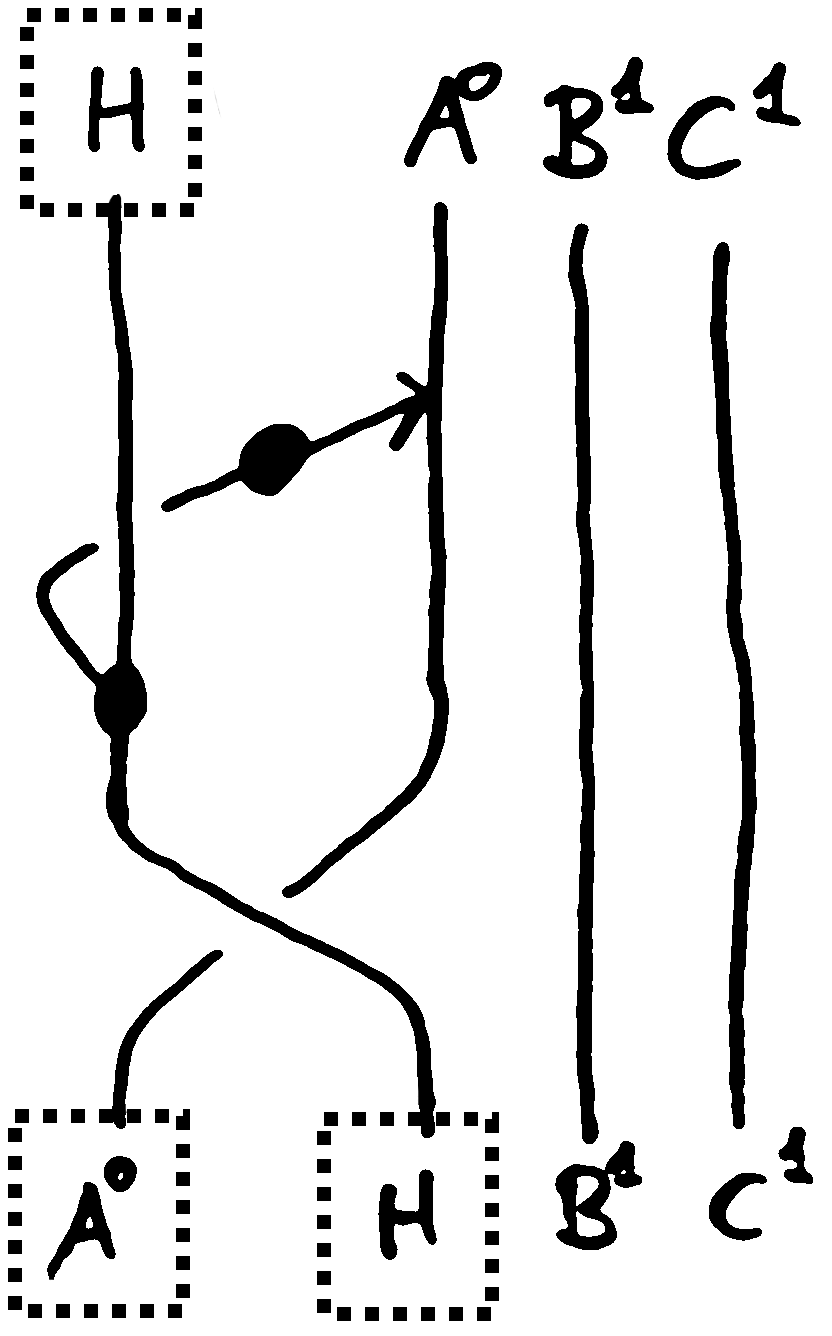}
\ee
provides a morphism in $\Rep_{\Sc}(H)$. The left $H$-module structure is given by the objects in dashed boxes.
It is invertible by Lemma \ref{lem:H-intertwiner}\,b). Naturality is straightforward in this case, as well as in the two cases to follow.

Next consider $\alpha_{A^1,B^0,C^1}$. The source and target objects in $\Cc_0$ are:
\be
\begin{array}{ll}
 A^1 \otimes_{\Cc} (B^0 \otimes_{\Cc} C^1) 
 = A^1 \otimes_{\Cc} (F(B^0) \otimes C^1) 
 = H \otimes A^1 \otimes F(B^0) \otimes C^1 \ ,
\\[.3em]
( A^1 \otimes_{\Cc} B^0 ) \otimes_{\Cc} C^1 
= (A^1 \otimes F(B^0)) \otimes_{\Cc} C^1 
= H \otimes A^1 \otimes F(B^0) \otimes C^1 \ .
\end{array}
\ee 
In both cases, the $H$-action is given by multiplication on $H$. The picture for the associativity isomorphism is \be\label{fig:assoc-c2}
\alpha_{A^1,B^0,C^1} = \scanPIC{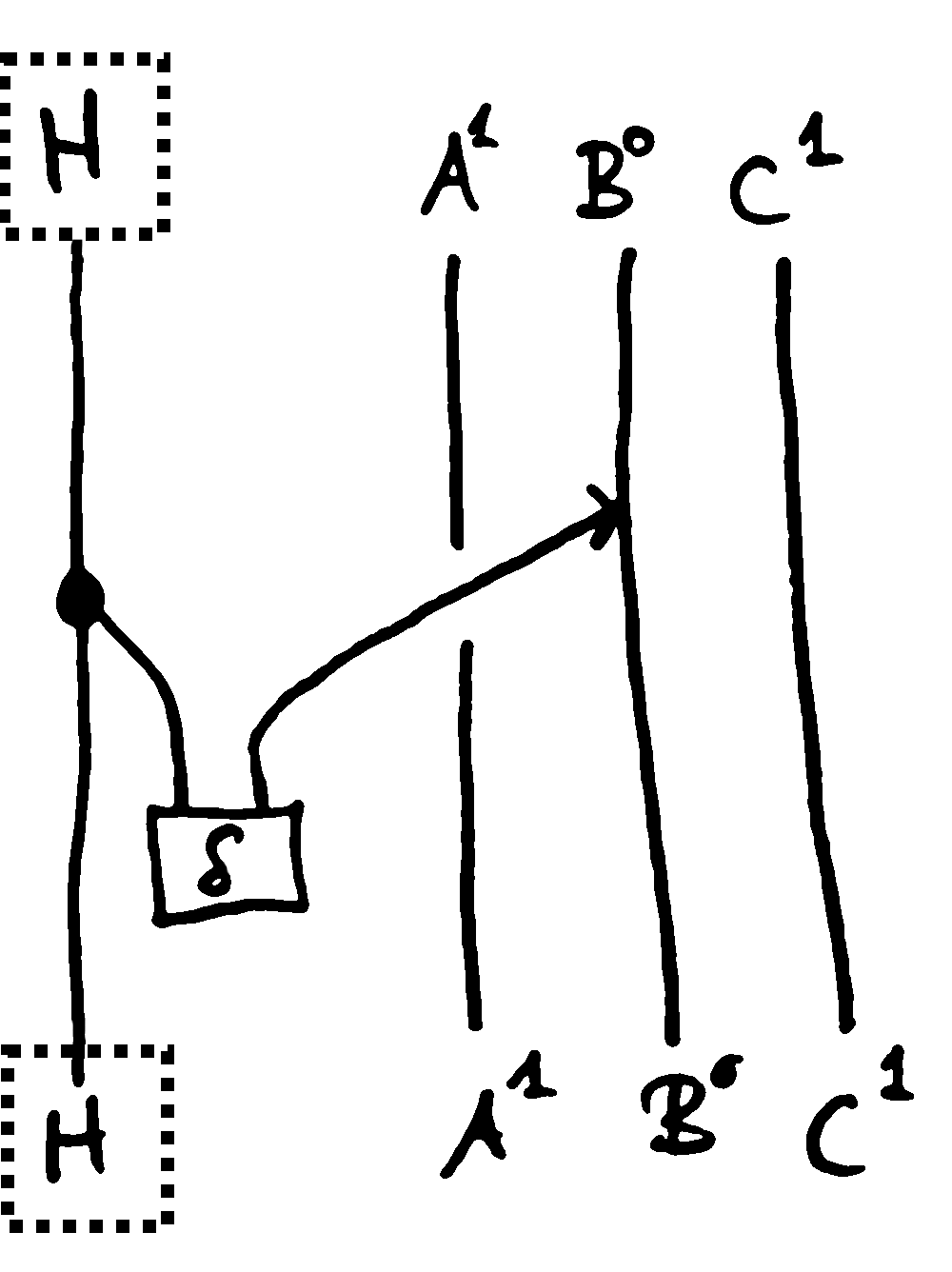}
\ee
By associativity of $H$, this is indeed a morphism in $\Rep_{\Sc}(H)$. To see invertibility, it is helpful to rewrite $\alpha_{A^1,B^0,C^1}$ as
\be
  \alpha_{A^1,B^0,C^1} = \scanPIC{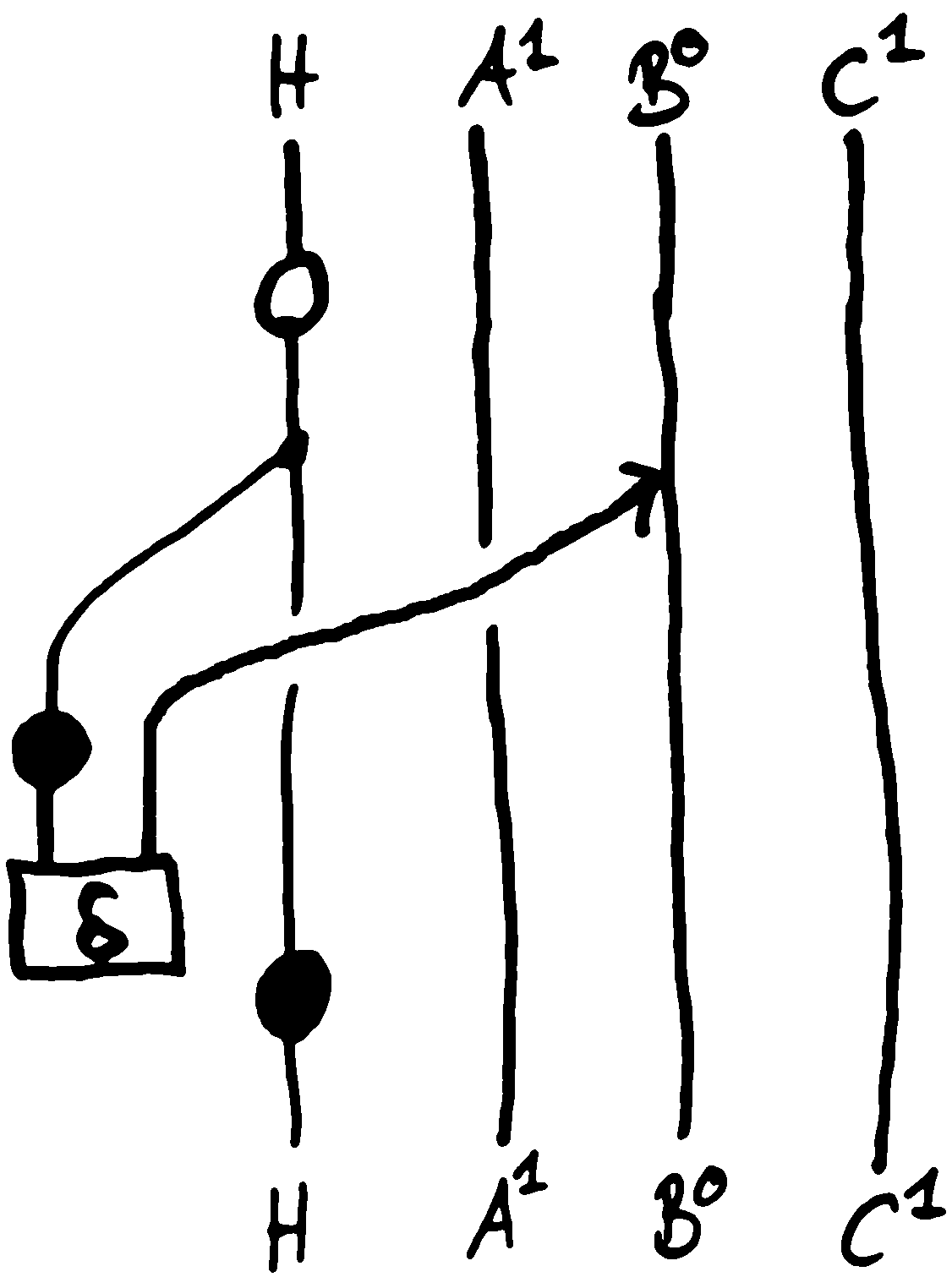} \ .
\ee
It is now evident that $\alpha_{A^1,B^0,C^1}$ is invertible because by assumption, $(S^{-1} \otimes \id) \circ \delta$ has a multiplicative inverse in $H \otimes H$.

Finally, for $\alpha_{A^1,B^1,C^0}$ we have 
\be
\begin{array}{ll}
 A^1 \otimes_{\Cc} (B^1 \otimes_{\Cc} C^0) 
 = A^1 \otimes_{\Cc} (B^1 \otimes F(C^0)) 
 = H \otimes A^1 \otimes B^1 \otimes F(C^0) \ ,
\\[.3em]
( A^1 \otimes_{\Cc} B^1 ) \otimes_{\Cc} C^0 
= (H \otimes A^1 \otimes B^1) \otimes_{\Cc} C^0 
= (H \otimes A^1 \otimes B^1) \otimes_{\Rep_{\Sc}(H)} C^0 \ .
\end{array}
\ee 
In the source object, the $H$-action is just by multiplication on the first $H$-factor. In the target object, the $H$-action is on the $H$- and on the $C^0$-factor in the tensor-product. The associativity isomorphism is 
\be\label{fig:assoc-c3}
\alpha_{A^1,B^1,C^0} = \scanPIC{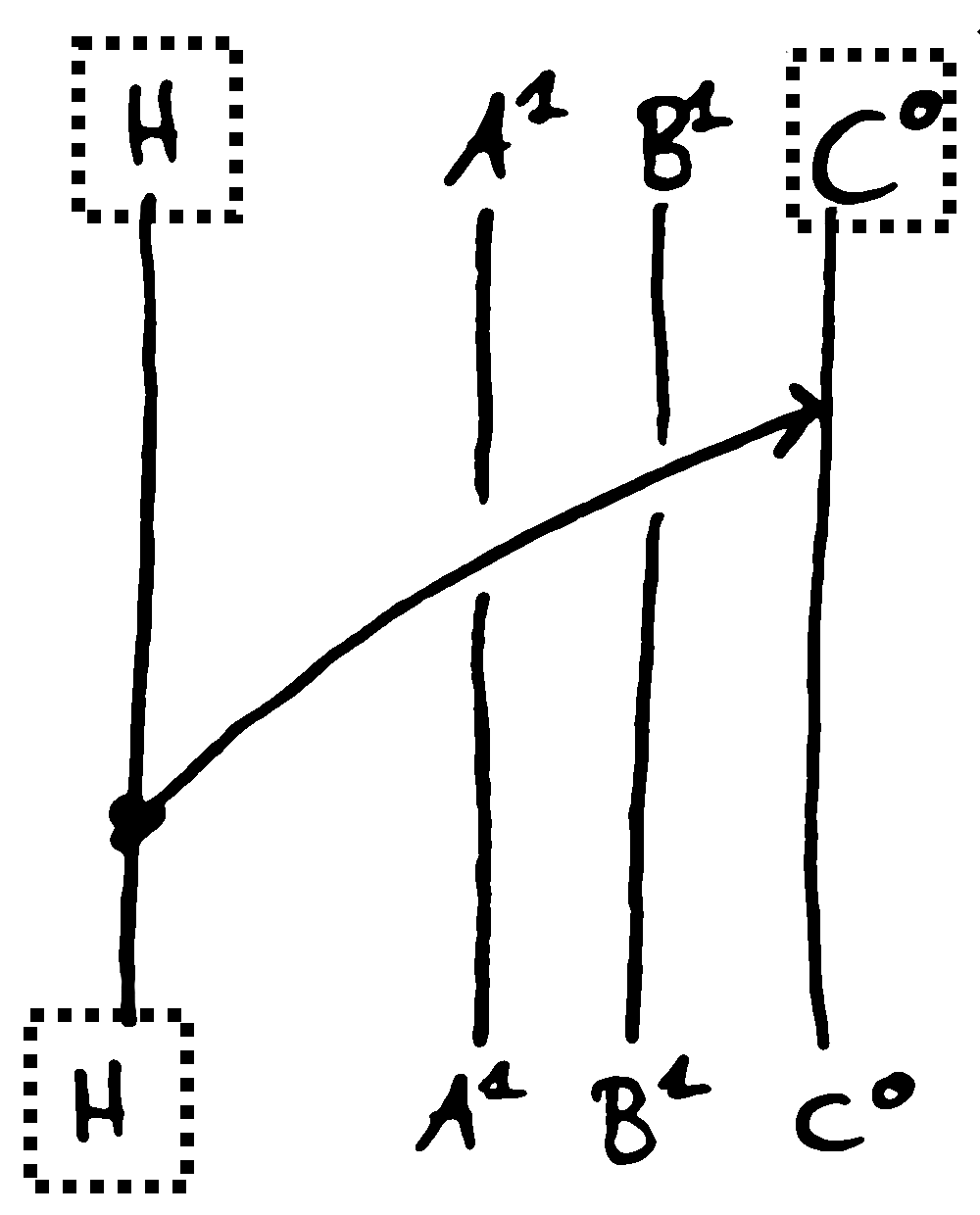}
\ee
Lemma \ref{lem:H-intertwiner}\,c) shows that this is indeed a morphism in $\Rep_{\Sc}(H)$. Invertibility follows from Lemma \ref{lem:H-intertwiner}\,a).
 
\medskip

\nxt $3$ objects from $\Cc_1$:
For $\alpha_{A^1,B^1,C^1}$ we have 
\be
\begin{array}{ll}
 A^1 \otimes_{\Cc} (B^1 \otimes_{\Cc} C^1) 
 = A^1 \otimes_{\Cc} (H \otimes B^1 \otimes C^1) 
 = A^1 \otimes H \otimes B^1 \otimes C^1 \ , 
\\[.3em]
( A^1 \otimes_{\Cc} B^1 ) \otimes_{\Cc} C^1 
= (H \otimes A^1 \otimes B^1) \otimes_{\Cc} C^1 
= H \otimes A^1 \otimes B^1 \otimes C^1 \ .
\end{array}
\ee 
Source and target object of $\alpha_{A^1,B^1,C^1}$ lie in $\Cc_1$, so there is no $H$-action. The graphical presentation for the isomorphism is 
\be\label{fig:assoc-d} 
\alpha_{A^1,B^1,C^1} = \scanPIC{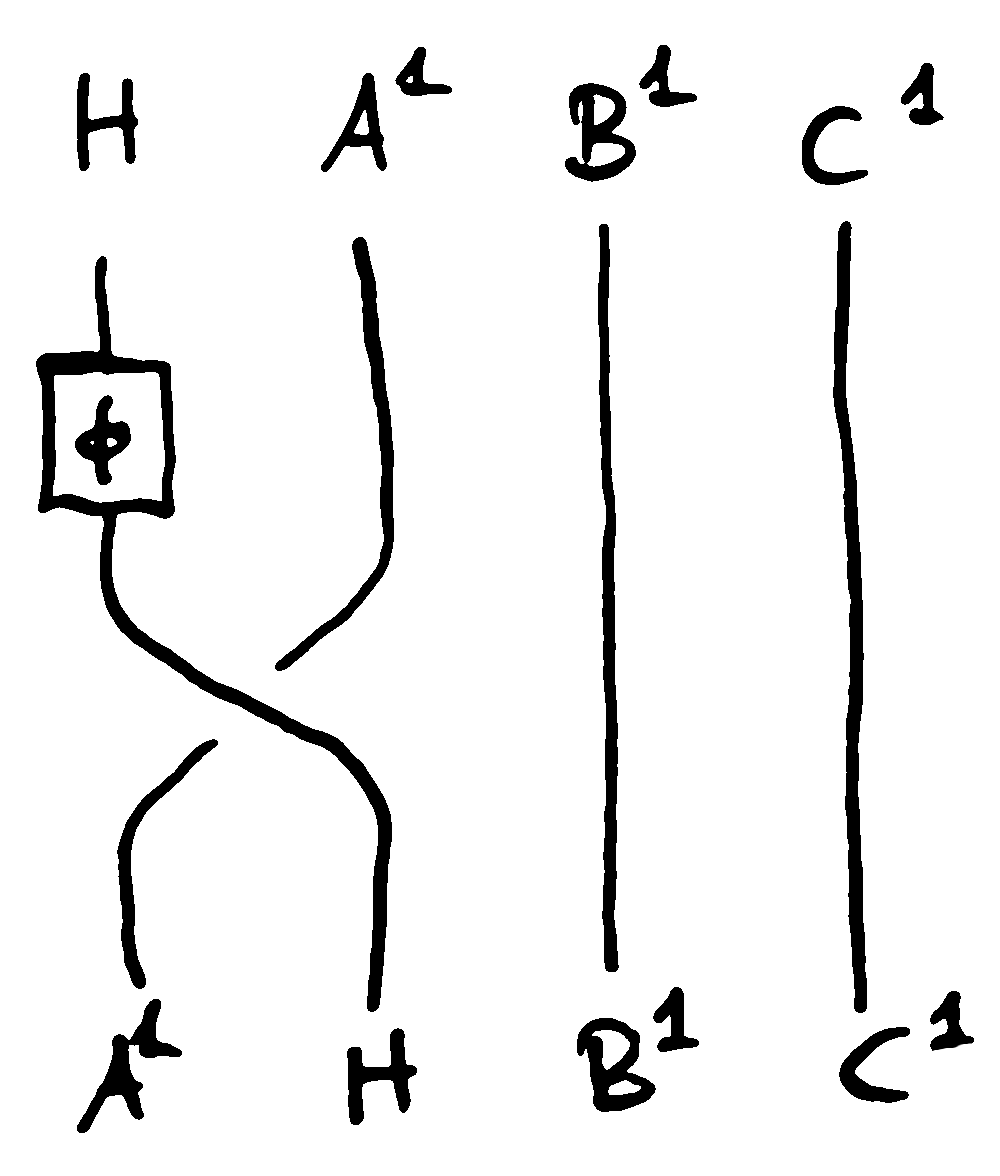}
\ee
Naturality in $A^1,B^1,C^1$ is again immediate.

\medskip

In formulating our main result (and parametrising monoidal structures on the category $\Cc$) it is useful to rewrite $\gamma,\delta,\phi$ in terms of some other Hopf-algebraic data $\Gamma: H^\vee \to H, \lambda: H \to \one$ and $g : \one \to H$ (which turns out to be uniquely determined):
\be\label{eq:gdp-via-Hopf}
\gamma ~=~ \scanPIC{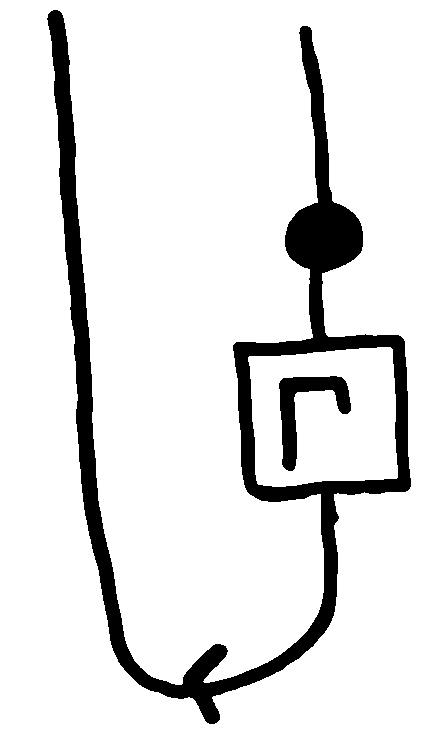}
\quad , \qquad
\delta ~=~ \scanPIC{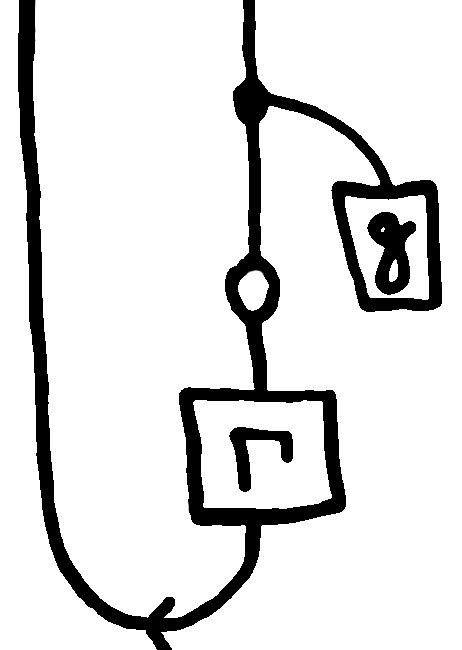}
\quad , \qquad
\phi ~=~ \scanPIC{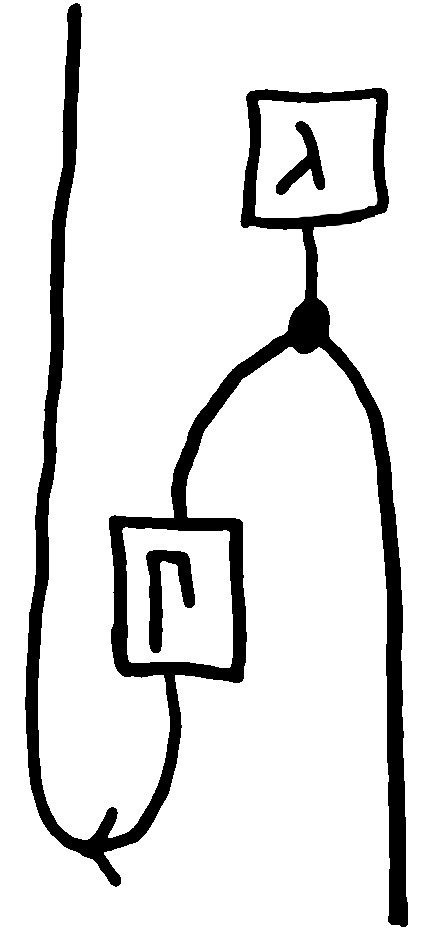}
\quad .
\ee
Conversely, the Hopf-algebraic data $\Gamma,\lambda$ is determined by $\gamma,\delta,\phi$ via 
\be\label{eq:Hopf-via-gdp}
  \Gamma = \scanPIC{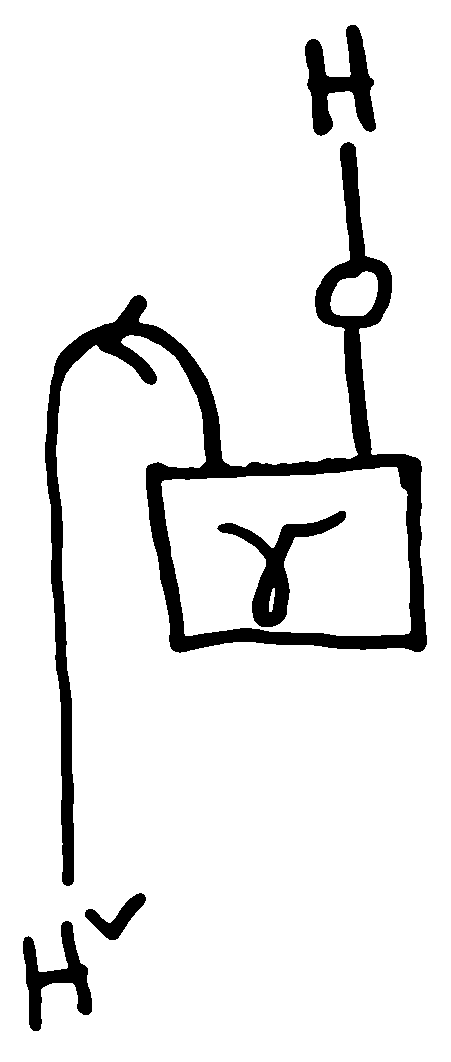} \qquad , \quad \lambda = \eps \circ \phi \ .
\ee

The next theorem is our first main result.

\begin{theorem} \label{thm:main1}
Let $\Sc$, $H$, $\Cc$, $\otimes_\Cc$ be as above. 
\\[.3em]
(i) Let $\Gamma : H^\vee \to H$ and $\lambda : H \to \one$ be two morphisms in $\Sc$ such that
\begin{itemize}
\item[a)] $\Gamma$ is a Hopf-algebra isomorphism. (In particular, $H$ is self-dual.)  
\item[b)] $\lambda$ is a right cointegral for $H$ such that there exists $g : \one \to H$ with $(\id_H \otimes \lambda) \circ \Delta_H = g \circ \lambda$, and such that $(\lambda \otimes \lambda) \circ (\id_H \otimes \Gamma) \circ \coev_H = \id_\one$. (It will be proved that $g$ is unique and has a multiplicative inverse.)
\item[c)] $\Gamma$ is related to its conjugate via $\Gamma^\vee = \delta_H \circ \theta_H^{-1} \circ S_H^2 \circ \Ad_{g^{-1}} \circ \Gamma$, where $\Ad_{g^{-1}}$ is conjugation with $g^{-1}$ (see \eqref{eq:Ad_x-def}).
\end{itemize}
Then the family of natural isomorphisms $\alpha_{A,B,C} : A \otimes_\Cc (B \otimes_\Cc C) \to (A \otimes_\Cc B) \otimes_\Cc C$ given above are associativity isomorphisms for $\otimes_\Cc$. 
\\[.3em]
(ii) Denote  by $\Cc(H,\Gamma,\lambda)$ the monoidal category resulting via (i) from $H,\Gamma,\lambda$. Given a Hopf algebra $H'$ in $\Sc$ and a Hopf algebra isomorphism $\varphi : H' \to H$ there is a monoidal equivalence $\Cc(H,\Gamma,\lambda) \cong \Cc(H',\Gamma',\lambda')$, where $\Gamma' = \varphi^{-1} \circ \Gamma \circ (\varphi^{-1})^\vee$ and $\lambda' = \lambda\circ \varphi$.
\end{theorem}

\begin{remark}\label{rem:main1}
(i) Theorem \ref{thm:main1} is proved in Propositions \ref{prop:Hopf-to-pent} and \ref{prop:mon-equiv}. We also prove the converse statement of part (i) (Proposition \ref{prop:pent-to-Hopf}): if one uses our ansatz  for the associativity isomorphisms, then all solutions to the pentagon are of the form stated in the theorem. We hasten to add that our ansatz for the associativity isomorphisms is not the most general natural transformation between the two functors $\Cc \times \Cc \times \Cc \to \Cc$. So an obvious question left open is whether there is a sense in which $\Cc(H,\Gamma,\lambda)$ provides {\em all} monoidal structures on $\Cc$. A framework in which this is the case is that of monoidal categories over $\Sc$; we will address this in a separate paper \cite{DR-prep}.
\\[.3em]
(ii) In the case where $\Sc = \vect(k)$ and $H = k[A]$ is the group algebra for a (necessarily abelian) finite group $A$, the above theorem reduces to the construction in \cite{Tambara:1998}. Tensor categories with these fusion rules are treated in the subfactor setting (for $k=\Cb$) by Izumi \cite{Izumi:1993}. Izumi also considers the case where $H$ is a semi-simple Hopf algebra \cite{Izumi:1998}.
\\[.3em]
(iii) Even in the case $\Sc=\vect(k)$, the construction in Theorem \ref{thm:main1} is more general than \cite{Izumi:1993,Tambara:1998,Izumi:1998}, because $H$ need not be commutative (since $H$ is self-dual, it is commutative if and only if it is cocommutative) or semi-simple (as would be required in the subfactor setting). The smallest example is Sweedler's 4-dimensional Hopf algebra (Section \ref{sec:Sweedler-mon}).
\\[.3em]
(iv) In \cite[Thm.\,1.3]{Etingof:2009}, Etinghof, Nikshych and Ostrik classify $G$-extensions of fusion categories in terms of cohomological data. The methods in \cite{Etingof:2009} should work in the tensor\footnote{\label{fn:tensor}
  By {\em tensor} we mean a $k$-linear monoidal category $\Cc$ with bi-linear tensor product and with $\Cc(I,I)=k$ for the unit object $I$.} 
(or probably just monoidal) setting as well. Theorem \ref{thm:main1} then gives such an extension in the case $G = \Zb/2\Zb$ and $\Cc_0=\Rep_\Sc(H)$. It would be interesting to compare the data $\Gamma$, $\lambda$ in Theorem \ref{thm:main1} to the cohomological data in \cite{Etingof:2009}; we hope to return to this point in \cite{DR-prep}.
\\[.3em]
(v) The category $\Cc(H,\Gamma,\lambda)$ is rigid. Left and right duals and the corresponding duality maps are given in Section \ref{sec:rigid}.
\end{remark}

Next we look at braidings on the category $\Cc(H,\Gamma,\lambda)$. For this, we restrict ourselves to symmetric $\Sc$ with trivial twist. We also need to fix a natural monoidal isomorphism 
\be
  \omega : \Id_{\Sc} \Rightarrow \Id_{\Sc}
\ee
which squares to the identity. The parity involution for super-vector spaces is an example to have in mind. This family of isomorphisms enters the ansatz for the braiding isomorphisms in \eqref{eq:braiding-ansatz} below. 

\begin{notation} \label{not:sec4} 
In addition to Notations \ref{not:sec2} and \ref{not:sec3} we assume the ribbon category $\Sc$ to be symmetric and to have trivial twist. On $\Sc$ we fix a natural monoidal automorphism of the identity functor, $\omega : \Id_\Sc \to \Id_\Sc$, such that $\omega$ squares to one. For the Hopf algebra $H$ in $\Sc$ we fix a choice of $\Gamma,\lambda$ satisfying the conditions of Theorem \ref{thm:main1}.
\end{notation}

Our ansatz for the braiding isomorphisms will be formulated in terms of the four morphisms
\be \label{eq:Rstn-def}
  R : \one \to H \otimes H
  \quad , \quad
  \sigma,\tau,\nu : \one \to H \ .
\ee
We require $R$ and $\sigma,\tau,\nu$ to have multiplicative inverses. In addition, $R$ has to make $H$ quasi-cocommutative:
\be\label{eq:R-intertwiner}
  \scanPIC{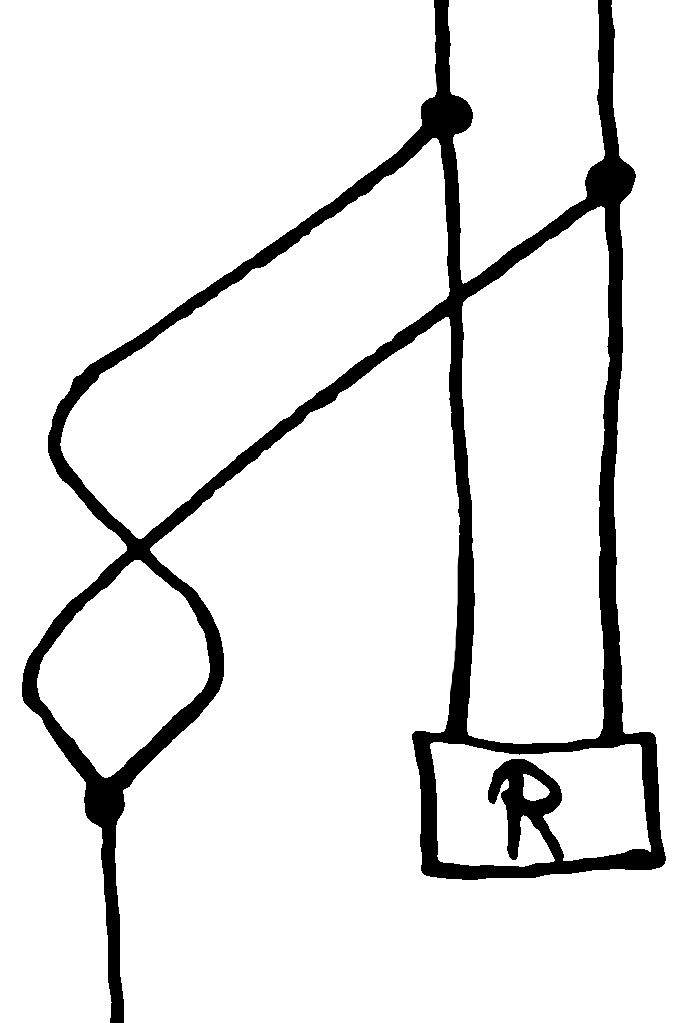}
  ~=~
  \scanPIC{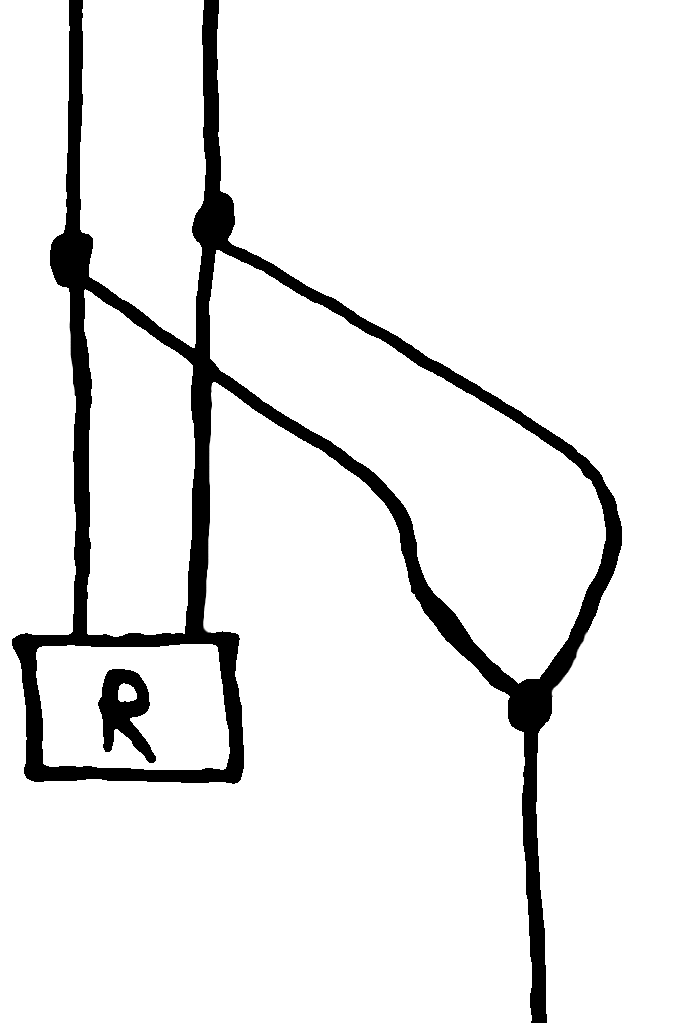} \quad .
\ee
There are four braiding isomorphisms depending on whether the two objects are from $\Cc_0$ or $\Cc_1$. Our ansatz is
\be\label{eq:braiding-ansatz} 
  c_{A^0,B^0} = \scanPIC{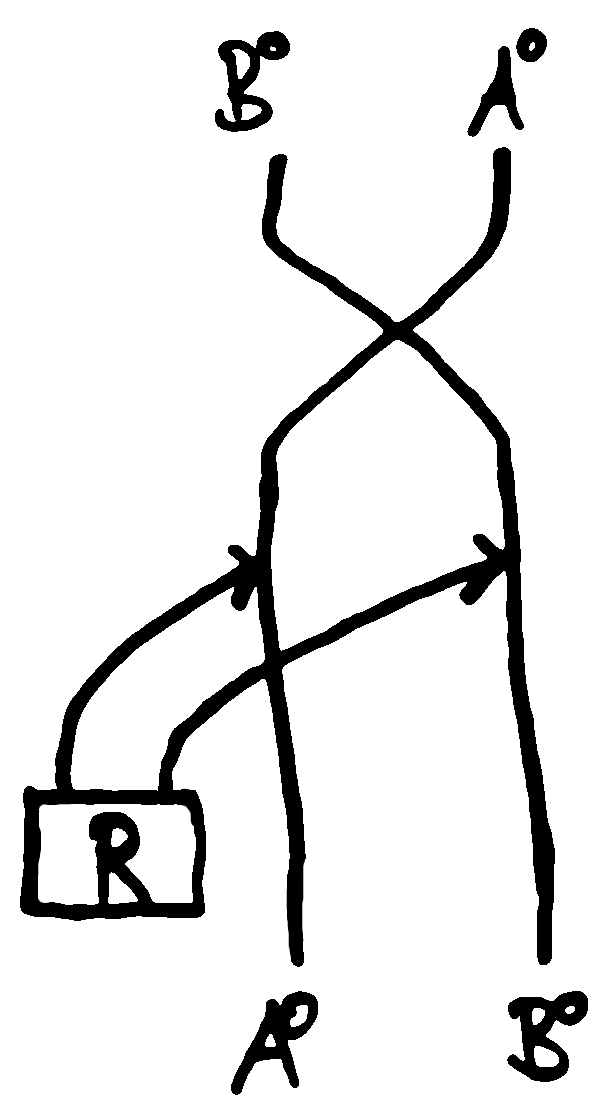}
  ~~, \quad
  c_{A^0,B^1} = \scanPIC{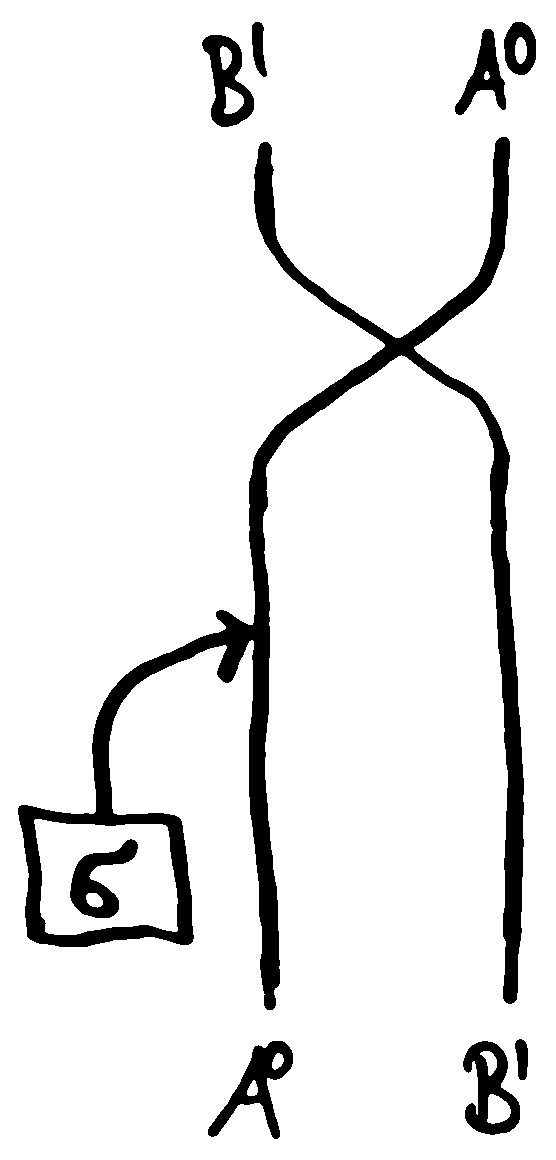} 
  ~~,\quad
  c_{A^1,B^0} = \scanPIC{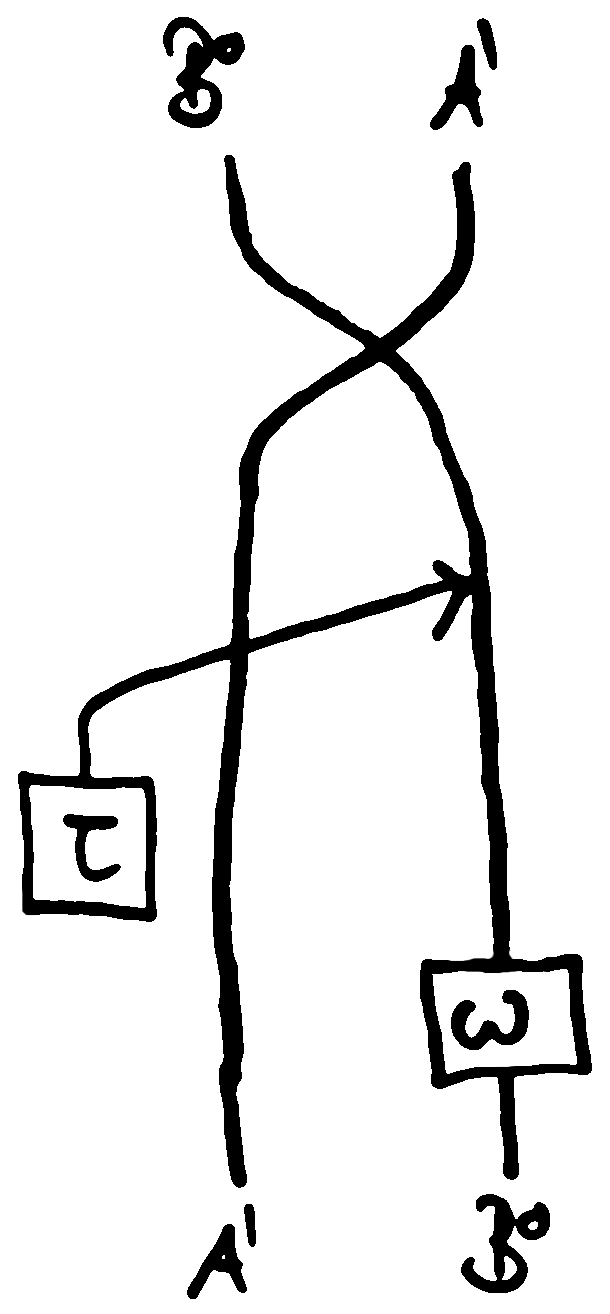}
  ~~,\quad
  c_{A^1,B^1} = \scanPIC{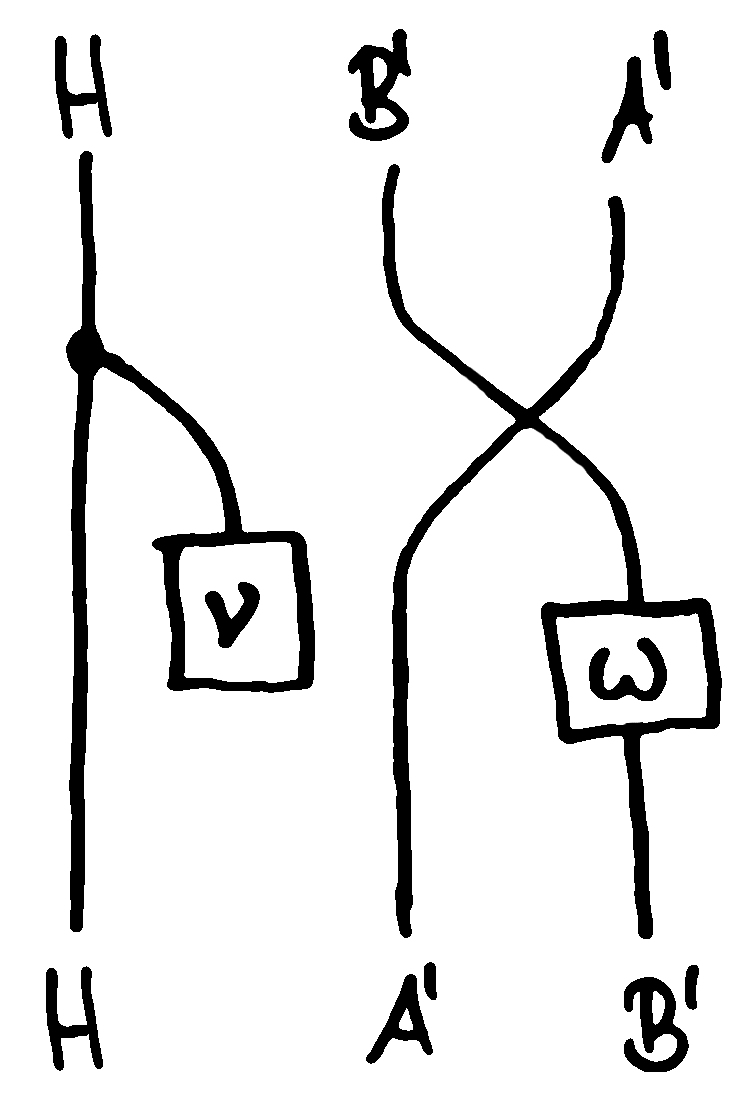}
  \quad .
\ee
Since $R,\sigma,\tau,\nu$ have multiplicative inverses, these maps are indeed isomorphisms. It is also immediate that they are natural in $A^a$ and $B^b$. Condition \eqref{eq:R-intertwiner} ensures that $c_{A^0,B^0}$ is a map of $H$-modules. 

\medskip

The expression for $R,\tau,\nu$ in terms of the Hopf-algebraic data $\sigma,\beta$ from Theorem \ref{thm:main2} below is as follows:
\be\label{eq:R-etc-via-Hopf}
  R = ({}_\sigma M \otimes {}_\sigma M) \circ \Delta \circ \sigma^{-1}
  \quad,\qquad
  \tau = \sigma
  \quad,\qquad
  \nu = \beta \cdot \sigma^{-1} \ .
\ee

Now we can formulate our second main result.

\begin{theorem} \label{thm:main2}
Let $\Sc$ 
      be as in Notation \ref{not:sec4}.
Let $\Cc(H,\Gamma,\lambda)$ and $g: \one \to H$ be as in Theorem \ref{thm:main1}. 
Let $\sigma : \one \to H$ and $\beta : \one \to \one$ be morphisms such that $\sigma$ has a multiplicative inverse and such that $\beta$ is invertible in $\End(\one)$.
\\[.3em]
(i)
Suppose that
\begin{itemize}
\item[a)] 
The isomorphism $\Gamma$ is determined through $\sigma$ by
$$
  (\id_H \otimes \Gamma) \circ \coev_H  
  =
  (\mu_H \otimes \mu_H)
  \circ
  (\id_H \otimes \sigma \otimes \sigma \otimes \id_H)
  \circ
  \Delta_H
  \circ
  \sigma^{-1} \ .
$$
\item[b)] The right cointegral $\lambda$ satisfies $\lambda \circ S_H = \lambda \circ \Ad_\sigma$
and $\lambda \circ \sigma = \beta \circ \beta$.
\item[c)] $\Ad_\sigma$ is a Hopf-algebra isomorphism $H \to H_\mathrm{cop}$ (the opposite coalgebra). Equivalently, 
$$
  \eps \circ \Ad_\sigma = \eps
  ~,~~
  \Delta_H \circ \Ad_\sigma = (\Ad_\sigma \otimes \Ad_\sigma) \circ c_{H,H} \circ \Delta_H
  ~,~~
  \Ad_\sigma \circ S_H = S_H^{-1} \circ \Ad_\sigma \ .
$$  
\item[d)] $S_H \circ \sigma = \mu_H \circ (g \otimes \sigma) = \mu_H \circ (\sigma \otimes g^{-1})$.
\item[e)] The natural isomorphism $\omega$ evaluated on $H$ is
$$
  \omega_H ~=~ \big[ \,
  H \xrightarrow{\Gamma^{-1}} H^\vee 
  \xrightarrow{(\Ad_\sigma)^\vee} (H^\vee)^\mathrm{op}
  \xrightarrow{~\Gamma~} H^\mathrm{op} 
  \xrightarrow{\Ad_\sigma^{-1}} H^\mathrm{op}_\mathrm{cop} 
  \xrightarrow{~S_H~} H \,\big] \ .
$$
(Since $\omega$ is natural and monoidal, both sides are Hopf algebra automorphisms.)
\end{itemize}
Then the family of natural isomorphisms $c_{A,B}^{\Cc} : A \otimes_\Cc B \to B \otimes_\Cc A$ in Eqn.\,\eqref{eq:braiding-ansatz} with $R,\tau,\nu$ given by Eqn.\,\eqref{eq:R-etc-via-Hopf} defines a braiding on $\Cc(H,\Gamma,\lambda)$.
\\[.3em]
(ii) Denote the braided category resulting via (i) from $\sigma,\beta$ by $\Cc(H,\lambda,\sigma,\beta)$. Given a Hopf algebra $H'$ in $\Sc$ and a Hopf algebra isomorphism $\varphi : H' \to H$, there is a braided monoidal equivalence $\Cc(H,\lambda,\sigma,\beta) \cong \Cc(H',\lambda',\sigma',\beta)$, where $\lambda' = \lambda\circ \varphi$ and $\sigma' = \varphi^{-1}\circ\sigma$.
\end{theorem}

\begin{remark} \label{rem:main2}
(i) This theorem is proved in Propositions \ref{prop:Hopf-to-hex} and \ref{prop:braid-mon-equiv}. To verify that the conditions a)--e) are not unnecessarily restrictive, in Proposition \ref{prop:hex-to-Hopf} we prove the converse statement to part (i): we start with the ansatz \eqref{eq:braiding-ansatz} and derive that all solutions to the two hexagons are of the form stated in Theorem \ref{thm:main2}. But also here our ansatz for the braiding is not the most general one and the question whether there is a sense in which Theorem \ref{thm:main2} provides all braidings on $\Cc(H,\Gamma,\lambda)$ will be postponed to \cite{DR-prep}.
\\[.3em]
(ii) Some consequences of conditions a)\,--\,e) are
\begin{itemize}
\item $\Ad_\sigma^{\,4} = \id_H$, $S_H^2 = \Ad_g \circ \Ad_\sigma^{\,2}$ and $\Ad_g \circ \Ad_\sigma = \Ad_\sigma \circ \Ad_g^{-1}$, see \eqref{eq:Adsig4=id}, \eqref{eq:Adsig-inv} and condition d).
\item $\Gamma$ is self-dual up to an inner automorphism, namely $\Gamma^\vee = \delta_H \circ \Ad_\sigma^{\,2} \circ \Gamma$ (combine condition c) in Theorem \ref{thm:main1} with $S_H^2 = \Ad_g \circ \Ad_\sigma^{\,2}$ and recall that in Theorem \ref{thm:main2} we assume $\Sc$ to have trivial twist). 
\item Denote by $R$ the R-matrix defining the braiding on $\Rep_\Sc(H)$ (see \eqref{eq:braiding-ansatz}). This R-matrix is non-degenerate as a copairing and satisfies $c_{H,H} \circ R =  (\Ad_\sigma \otimes \Ad_\sigma) \circ R$ (compare \eqref{eq:R-via-gamma_2nd} and \eqref{eq:R-via-gamma-op}).
\end{itemize}
(iii) If $H$ is commutative (hence cocommutative) then $g$ is the unit of $H$ and the conditions in Theorem \ref{thm:main2}\,(i) simplify to:
\begin{itemize}
\item[a)] (unchanged)
\item[b)] $\lambda$ satisfies $\lambda \circ S_H = \lambda$
and $\lambda \circ \sigma = \beta \circ \beta$
\item[c)] (trivially true)
\item[d)] $S_H \circ \sigma = \sigma$
\item[e)] $\omega_H =S_H$
\end{itemize}
We also have $S_H^2=\id_H$ (from e) together with the fact that $\omega$ is an involution).
We note that the conditions in Theorem \ref{thm:main2} do have non-commutative solutions. One such Hopf algebra is given in Section \ref{sec:H16-non-comm}, precisely to illustrate this point.
\\[.3em]
(iv) The R-matrix which gives the braiding on $\Rep_\Sc(H)$ is determined by $\sigma$ via (see \eqref{eq:R-etc-via-Hopf}):
\be
  R = ({}_\sigma M \otimes {}_\sigma M) \circ \Delta \circ \sigma^{-1} \ .
\ee
Here, ${}_\sigma M$ denotes the left-multiplication by $\sigma$ (see \eqref{eq:left-right-mult-def}). 
R-matrices of this form occur in the theory of quantum groups\footnote{
  However, $U_q(\mathfrak{g})$ does not provide an example for our construction because it is infinite-dimensional and so not a rigid (and in particular not a self-dual) object in $\Sc = \vect$. Furthermore, the $U_q(\mathfrak{g})$ R-matrices are elements of a completion of the tensor square.} 
\cite{Kirillov:1990,Levendorskii:1991} and are studied in \cite{Snyder:2008} as part of the definition of a `half-ribbon Hopf algebra'.
\end{remark}

\begin{remark} \label{rem:twist+reverse}
(i) From Remark \ref{rem:main2}\,(ii) we know that $\Ad_\sigma^{\,4}=\id_H$. If already $\Ad_\sigma^{\,2} = \id_H$, then the category $\Cc(H,\lambda,\sigma,\beta)$ is ribbon. A choice of twist isomorphisms is given by acting with $\sigma^{-2}$ on $H$-modules in $\Cc_0$ and by applying $\beta^{-1} \cdot \omega$ to objects in $\Cc_1$, see Proposition \ref{prop:twist}.
\\[.3em]
(ii) For a braided category $\Cc$  the {\em reverse category}\footnote{also called {\em mirror} category} $\overline{\Cc}$ is the category $\Cc$ with the same tensor product and associativity isomorphisms and with a new braiding
$\overline{c}_{X,Y} = c_{Y,X}^{-1}.$
We show in Proposition \ref{prop:Cbar-equiv-C} that
\be
  \overline{\Cc(H,\lambda,\sigma,\beta)} \cong \Cc(H,\lambda,\sigma^{-1},\beta^{-1})
\ee
as braided monoidal categories. 
\end{remark}

\begin{remark}
Let $\Cc$ be a $G$-graded monoidal category. 
Let $\Aut_{\Cc|\Cc}(Id_\Cc)$ denote the group of automorphisms of the identity functor on $\Cc$ as a $\Cc$-bimodule functor. This group is naturally isomorphic to the automorphism group $\Aut_{\Zc(\Cc)}(I)$ of the identity object in the monoidal centre of $\Cc$ (this follows from a monoidal equivalence between the monoidal centre $\Zc(\Cc)$ of $\Cc$ and the category $\End_{\Cc|\Cc}(\Cc)$ of $\Cc$-bimodule endofunctors of $\Cc$ \cite{Ostrik:2002}). Note that if $\Cc$ is tensor${}^{\text{\ref{fn:tensor}}}$ this group is $k^\times$, the group of invertible elements of the ground field. A 3-cocycle $a$ of $G$ with coefficients in $\Aut_{\Cc|\Cc}(Id_\Cc)$ can be used to twist the associativity isomorphisms $\alpha$ of $\Cc$,
\be
  \alpha^a_{X,Y,Z} = a(x,y,z)_{(X\ot Y)\ot Z} \circ \alpha_{X,Y,Z},\quad\quad X\in\Cc_x,\ Y\in\Cc_y,\ Z\in\Cc_z \ .
\ee  
It is easy to see that up to monoidal equivalence twisting depends only on the cohomology class of $a$.
(However, it is in general not true that different cohomology classes give non-equivalent monoidal structures.)
We obtain a (not necessarily faithful) action of third cohomology group $H^3\big(G\,,\,Aut_{\Cc|\Cc}(Id_\Cc)\big)$ 
on the set of equivalence classes of monoidal structures on $\Cc$.

\noindent
Similarly a braided structure on a $G$-graded category $\Cc$ can be twisted by an abelian 3-cocycle $(a,\gamma)$ \cite{Eilenberg:1954,Joyal:1993}. Namely the associativity is twisted as above via $\alpha^{(a,\gamma)} = \alpha^a$, while the deformed braiding $c^{(a,\gamma)} = c^\gamma$ is defined as
\be 
  c^\gamma_{X,Y} = \gamma(x,y)_{Y\ot X} \circ c_{X,Y},\quad\quad X\in\Cc_x,\ Y\in\Cc_y \ .
\ee
Let $G=\Zb/2\Zb$ be the grading group and let $\Sc$ be tensor with ground field $\Cb$ and consider the case $\Cc = \Cc(H,\Gamma,\lambda)$ as in Theorem \ref{thm:main1}. We have $H^3(\Zb/2\Zb,\Cb^\times) = \Zb/2\Zb$ and the non-trivial element of $\Zb/2\Zb$ acts as $(\Gamma,\lambda) \mapsto (\Gamma,-\lambda)$. 

\noindent
If $\Cc = \Cc(H,\lambda,\sigma,\beta)$ is braided via Theorem \ref{thm:main2}, the group of twists for associativity and braiding is $H_{ab}^3(\Zb/2\Zb,\Cb^\times) = \Zb/4\Zb$. An element $m \in \Zb/4\Zb$ acts as $(\lambda,\sigma,\beta) \mapsto ( (-1)^m \lambda , \sigma , i^{m} \beta)$, where $i = \sqrt{-1}$.
\end{remark}

\bigskip

This concludes the statement of our results. We now turn to the proofs of the assertions made above.

\section{Monoidal structure on $\Rep_\Sc(H)+\Sc$ for $\Sc$ ribbon}\label{sec:theorem1}

In this section we prove Theorem \ref{thm:main1} and its converse in the sense stated in Remark \ref{rem:main1}\,(i), as well as the existence of duals asserted in Remark \ref{rem:main1}\,(v). We assume the conventions in Notations \ref{not:sec2} and \ref{not:sec3}, but not those of Notations \ref{not:sec4}; the latter are specific to the discussion of the braiding in Section \ref{sec:braiding}.

\subsection{The pentagon equations}\label{sec:pentagon}

The associativity isomorphisms given in Section \ref{results} have to satisfy the pentagon axiom in $\Cc$. Namely, the diagram 
$$
\xygraph{ !{0;/r4.5pc/:;/u4.5pc/::}[]*+{A\otimes_\Cc(B\otimes_\Cc(C\otimes_\Cc D))} (
  :[u(1.1)r(1.7)]*+{(A\otimes_\Cc B)\otimes_\Cc(C\otimes_\Cc D)} ^{\alpha_{A,B,C D}}
  :[d(1.1)r(1.7)]*+{((A\otimes_\Cc B)\otimes_\Cc C)\otimes_\Cc D}="r" ^{\alpha_{A B, C,D}}
  ,
  :[r(.6)d(1.5)]*+!R(.3){A\otimes_\Cc((B\otimes_\Cc C)\otimes_\Cc D)} _{id_A\otimes_\Cc \alpha_{B,C,D}}
  :[r(2.2)]*+!L(.3){(A\otimes_\Cc(B\otimes_\Cc C))\otimes_\Cc D} ^{\alpha_{A,B C,D}}
  : "r" _{\alpha_{A,B,C}\otimes_\Cc id_D}
)
}
$$
has to commute. Here we have omitted the $\otimes_{\Cc}$ in the indices of $\alpha$. As an equation, this reads
\be\label{eq:pentagon}
  \alpha_{AB,C,D} \circ \alpha_{A,B,CD} = (\alpha_{A,B,C} \otimes_{\Cc} \id_D) \circ \alpha_{A,BC,D} \circ (\id_A \otimes_{\Cc} \alpha_{B,C,D}) \ .
\ee
Taking each of the objects $A,B,C,D$ either from $\Cc_0$ or from $\Cc_1$, we obtain 16 equations.
As before, we write $X^i$ for objects of $\Cc_i$. Let us number the 16 cases depending on $a,b,c,d \in \{0,1\}$ for $A = A^a$, $B = B^b$, $C = C^c$, $D = D^d$ as follows:

\begin{center}
\begin{tabular}{rcccc||rcccc||rcccc||rcccc}
case & $a$ & $b$ & $c$ & $d$ &\
case & $a$ & $b$ & $c$ & $d$ &
case & $a$ & $b$ & $c$ & $d$ &
case & $a$ & $b$ & $c$ & $d$ \\
\hline
1) & 0 & 0 & 0 & 0 & 5) & 1 & 0 & 0 & 0 &  9) & 0 & 1 & 1 & 0 & 13) & 1 & 0 & 1 & 1 \\
2) & 0 & 0 & 0 & 1 & 6) & 0 & 0 & 1 & 1 & 10) & 1 & 0 & 1 & 0 & 14) & 1 & 1 & 0 & 1 \\
3) & 0 & 0 & 1 & 0 & 7) & 0 & 1 & 0 & 1 & 11) & 1 & 1 & 0 & 0 & 15) & 1 & 1 & 1 & 0 \\
4) & 0 & 1 & 0 & 0 & 8) & 1 & 0 & 0 & 1 & 12) & 0 & 1 & 1 & 1 & 16) & 1 & 1 & 1 & 1 
\end{tabular}
\end{center}

Recall from Section \ref{results} that some of the associativity isomorphisms are identities. From this it is immediate that {\bf cases 1, 2 and 5} hold as they read $\id=\id$. Actually, also 
{\bf cases 6, 9 and 11} turn out to be true independent of the choice of $\gamma,\delta,\phi$, but the verification is slightly less immediate. Let us start with case 6. The pentagon reads
\be\label{eq:pentagon-6-aux1}
  \alpha_{A^0B^0,C^1,D^1} \circ \alpha_{A^0,B^0,C^1D^1} = (\alpha_{A^0,B^0,C^1} \otimes_{\Cc} \id_{D^1}) \circ \alpha_{A^0,B^0C^1,D^1} \circ (\id_{A^0} \otimes_{\Cc} \alpha_{B^0,C^1,D^1}) \ ,
\ee
and upon substituting the expressions from  Section \ref{results} we arrive at the string diagram
\be\label{eq:pentagon-6-aux2} 
  \scanPIC{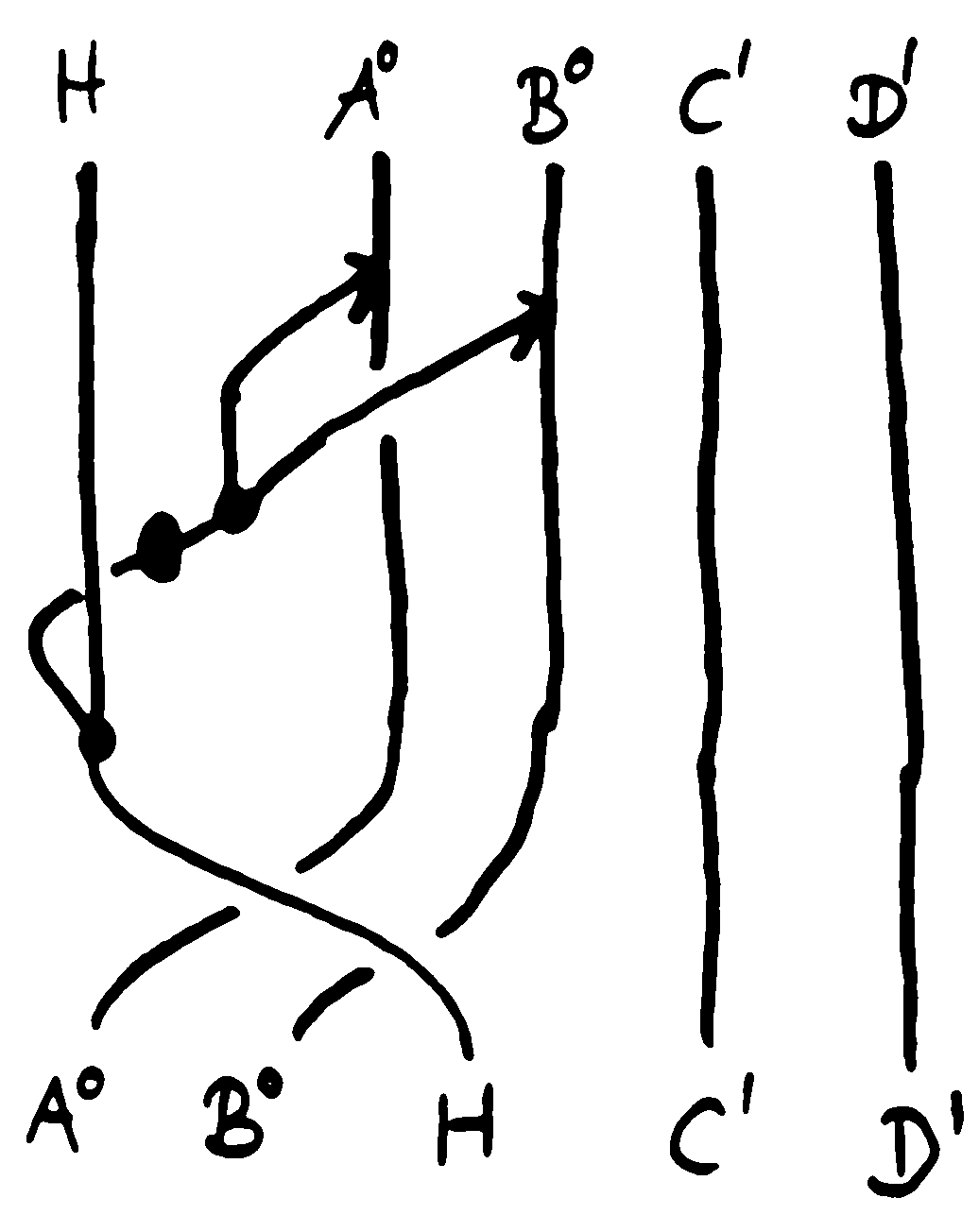}
  \quad = \quad
  \scanPIC{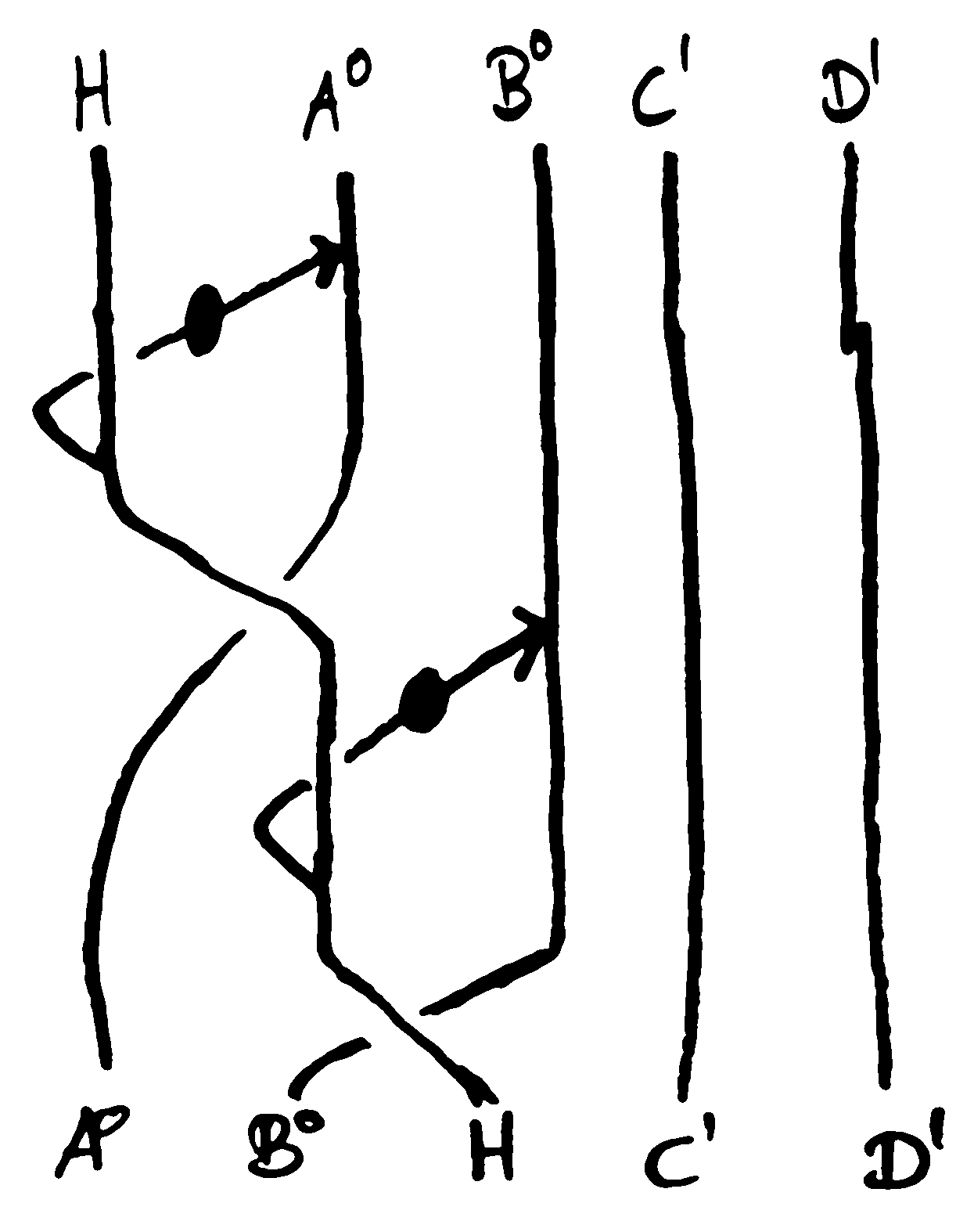} \quad . 
\ee
From coassociativity and Figure \ref{fig:Hopf-property}\,f) it follows that this equality indeed holds. Similar (but even easier) calculations show that cases 9 and 11 also hold. 

The remaining cases result in conditions on the morphisms $\gamma,\delta,\phi$. 
In {\bf case 3}  
the pentagon is as in \eqref{eq:pentagon-6-aux1} with $D^1$ replaced by $D^0$. By  Section \ref{results} the corresponding string diagram is
\be\label{eq:pentagon-3-aux2}
\scanPIC{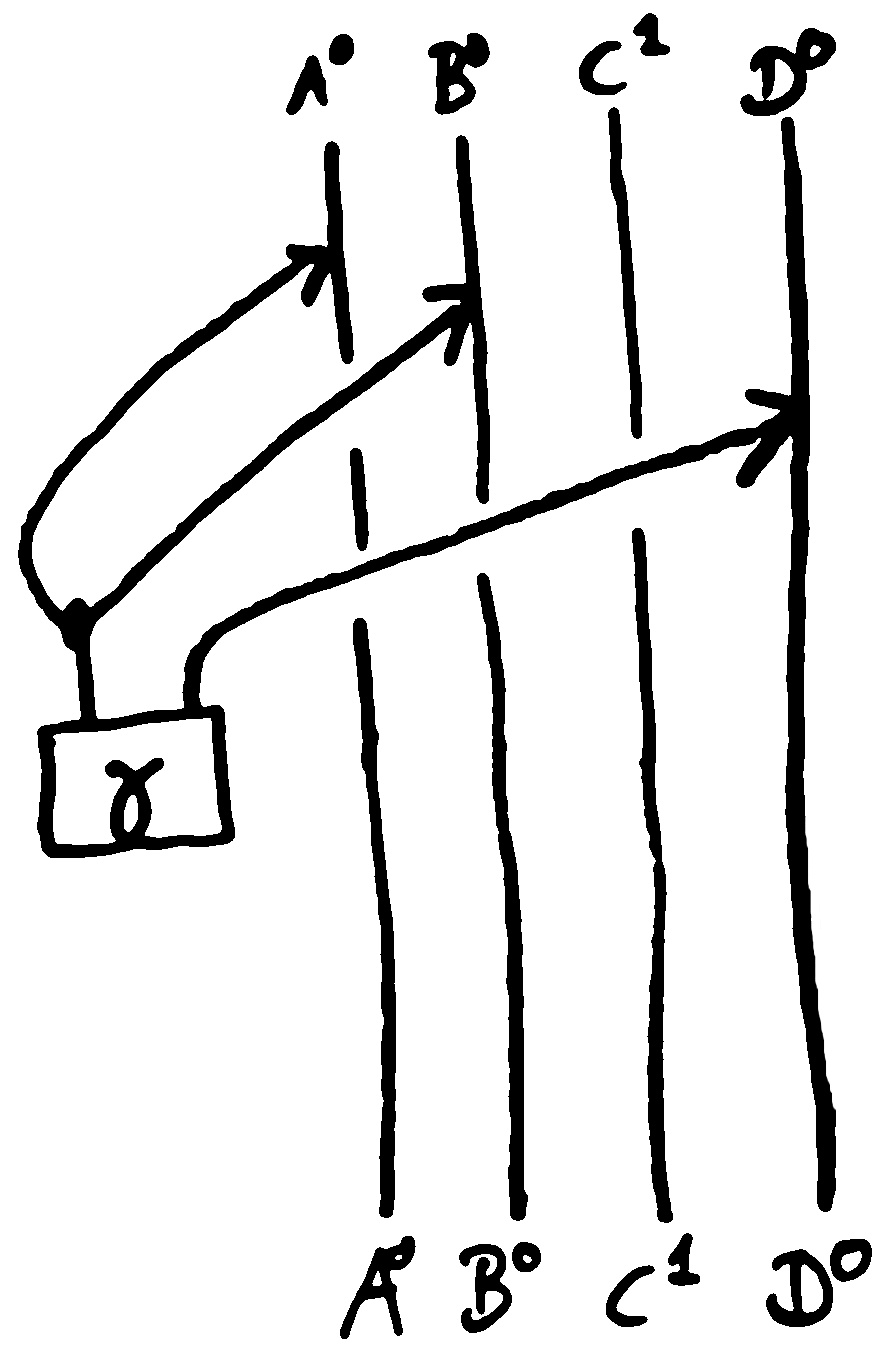} 
~=~
\scanPIC{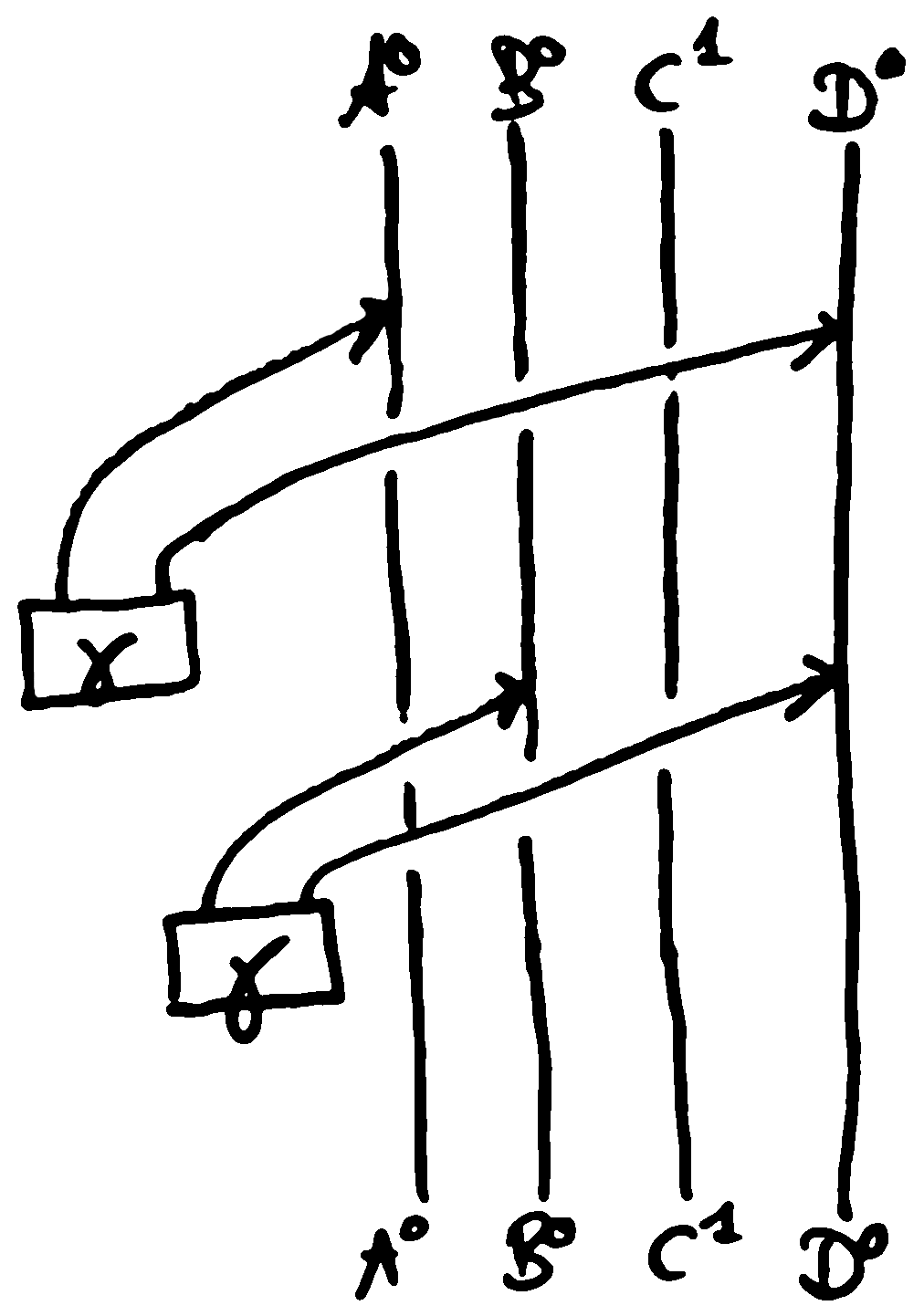} 
\ .
\ee
Equation \eqref{eq:pentagon-3-aux2} has to hold for all $A^0,B^0,C^1,D^0$. This is equivalent to 
\be\tag{P-3}\label{eq:pentagon-3} 
\scanPIC{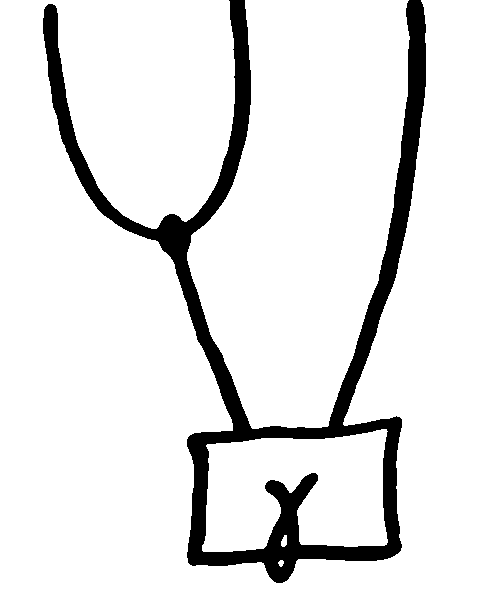} 
~=~
\scanPIC{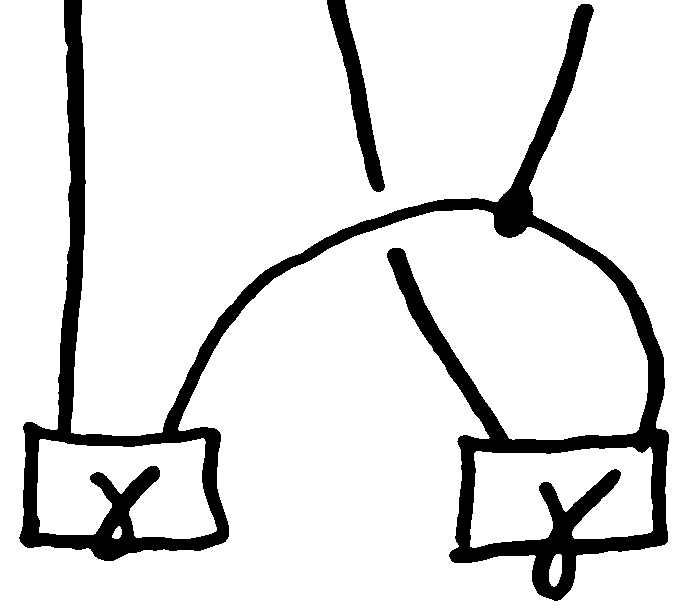} 
\ .
\ee
Indeed, to see that \eqref{eq:pentagon-3-aux2} implies \eqref{eq:pentagon-3} just specialise to $A^0=B^0=D^0=H$, $C^1=\one$ and compose with $\eta\otimes\eta\otimes\eta$. Conversely, it is immediate that \eqref{eq:pentagon-3} implies \eqref{eq:pentagon-3-aux2} by using associativity of the $H$-action on $D^0$. An analogous calculation in {\bf case 4} results in
\be\tag{P-4}\label{eq:pentagon-4} 
\scanPIC{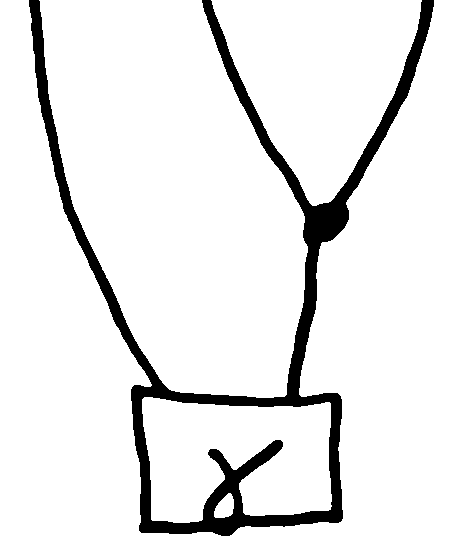} 
~=~
\scanPIC{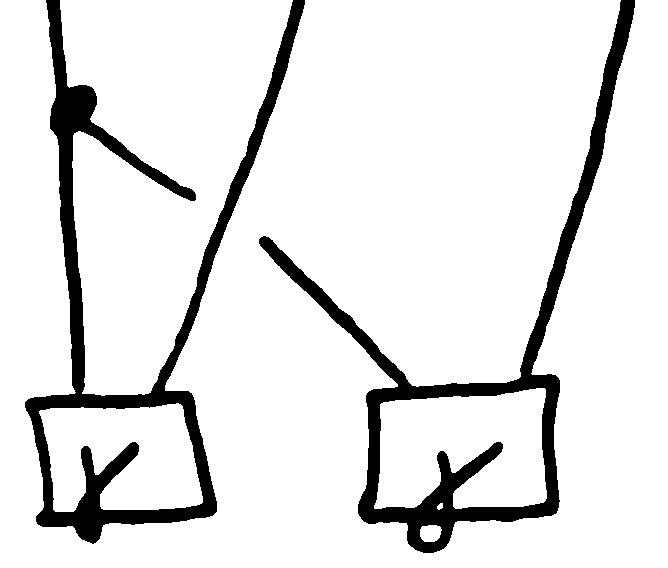} \quad .
\ee
The string diagram for {\bf case 7} is
\be\label{eq:pentagon-7-aux1} 
\scanPIC{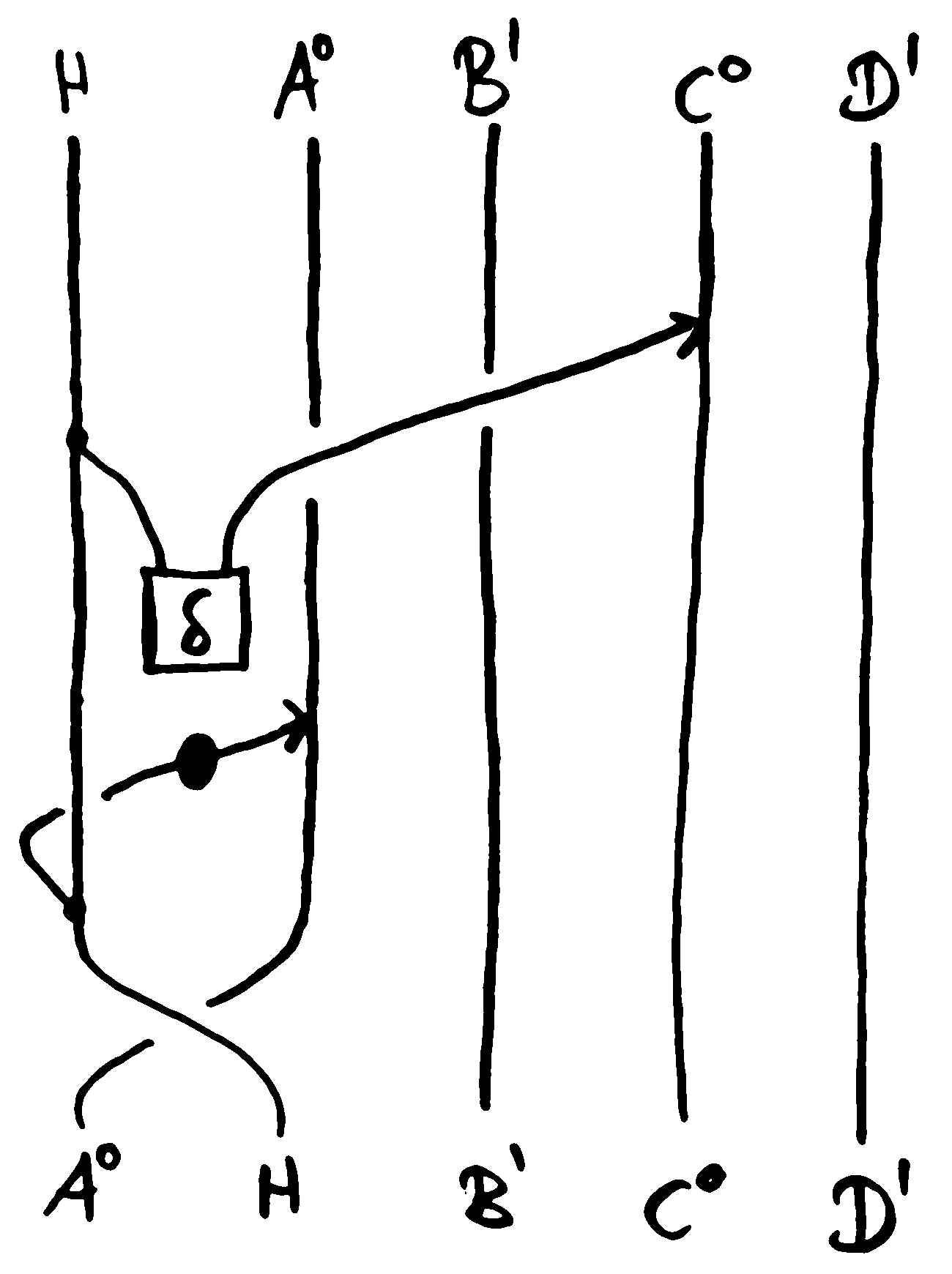} 
~~~=~~~
\scanPIC{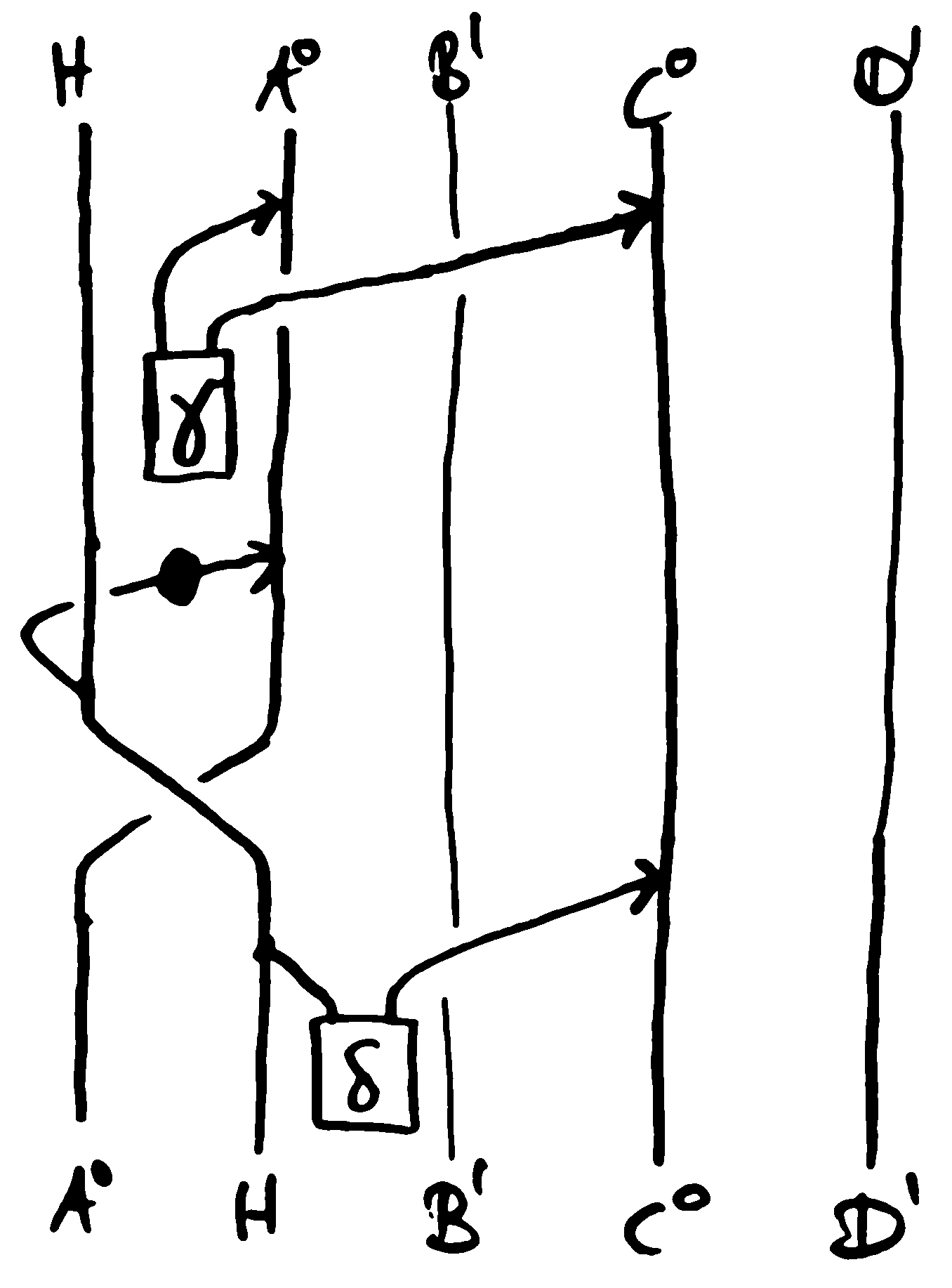} \quad .
\ee
By specialising to $A^0=C^0=H$ and $B^1=D^1=\one$ and composing with the unit of $H$, and conversely by using associativity of the $H$-action, one finds that \eqref{eq:pentagon-7-aux1} is equivalent to
\be\tag{P-7}\label{eq:pentagon-7} 
\scanPIC{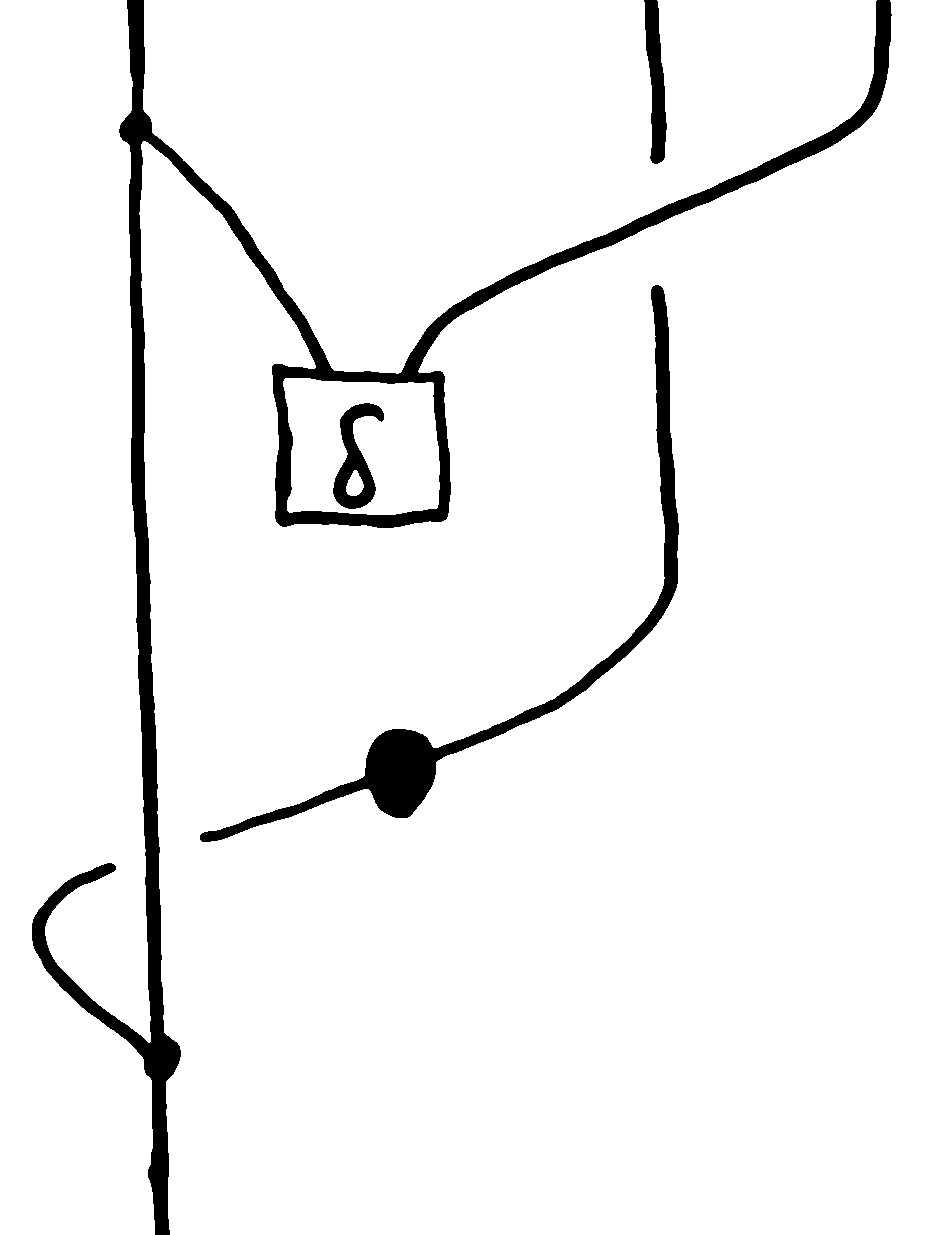} 
~=~~~
\scanPIC{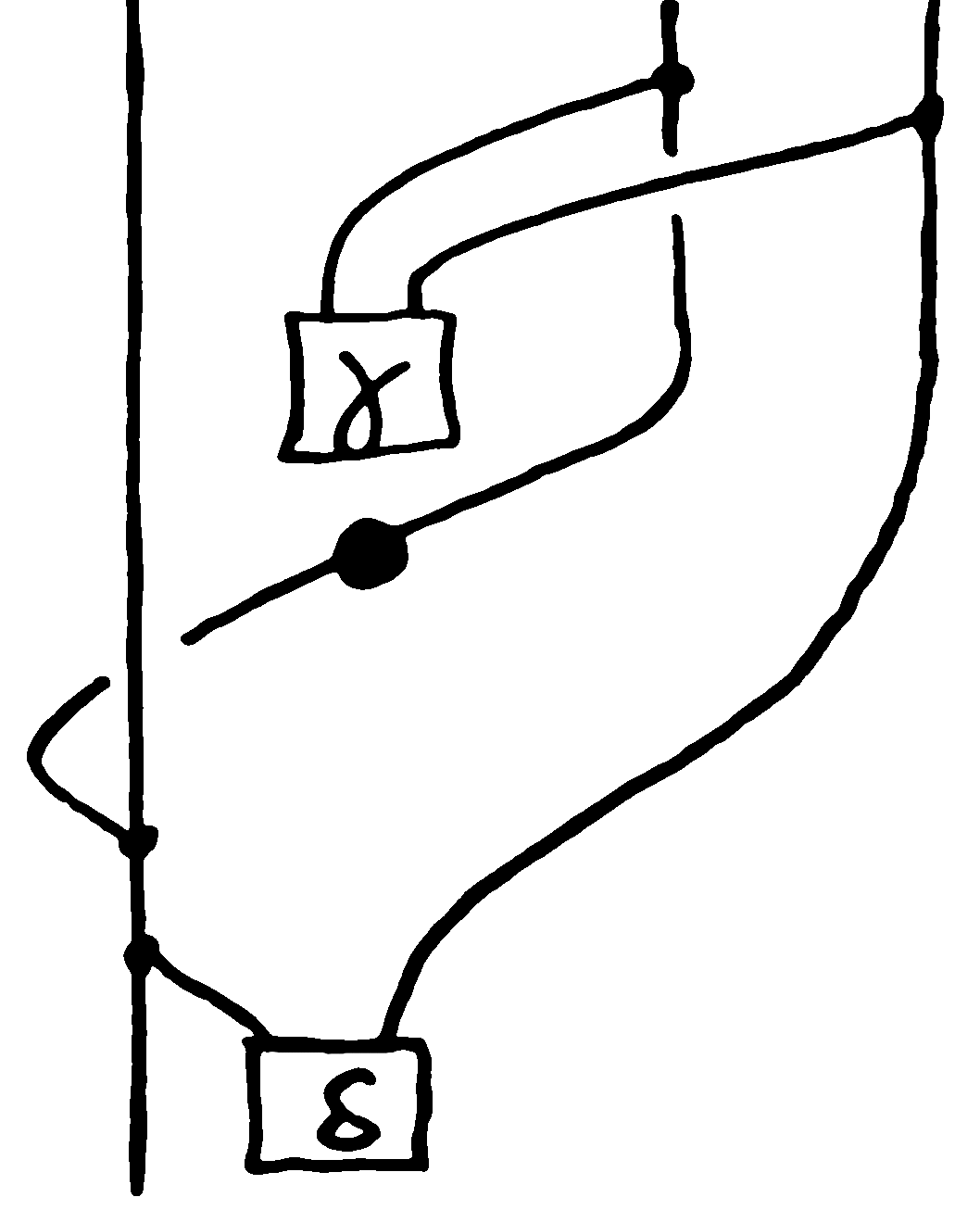} \quad .
\ee
In each of the remaining cases, the argument to pass from the pentagon to an identity involving only the object $H$ is the same and we will only list the resulting conditions on $\gamma$, $\delta$, $\phi$:

\medskip\noindent
{\bf Case 8:} 
\be\tag{P-8}\label{eq:pentagon-8}
\scanPIC{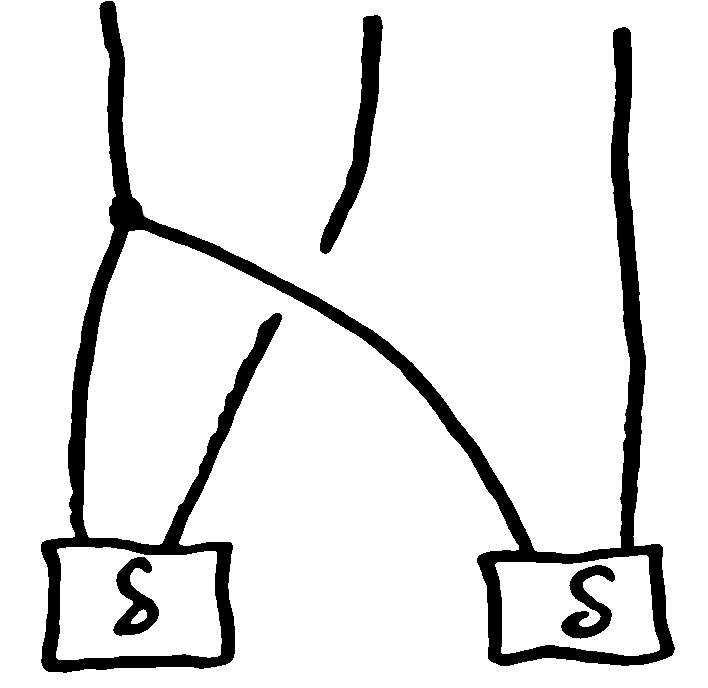}~~=~~\scanPIC{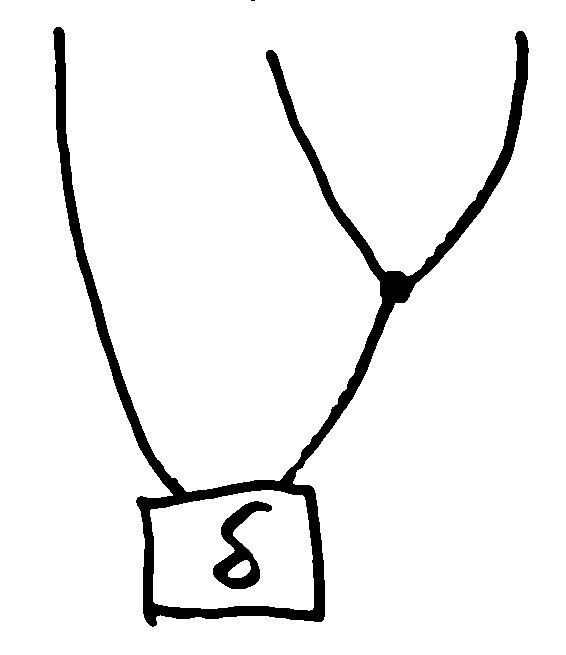}
\ee
{\bf Case 10:} 
\be\tag{P-10}\label{eq:pentagon-10}
\scanPIC{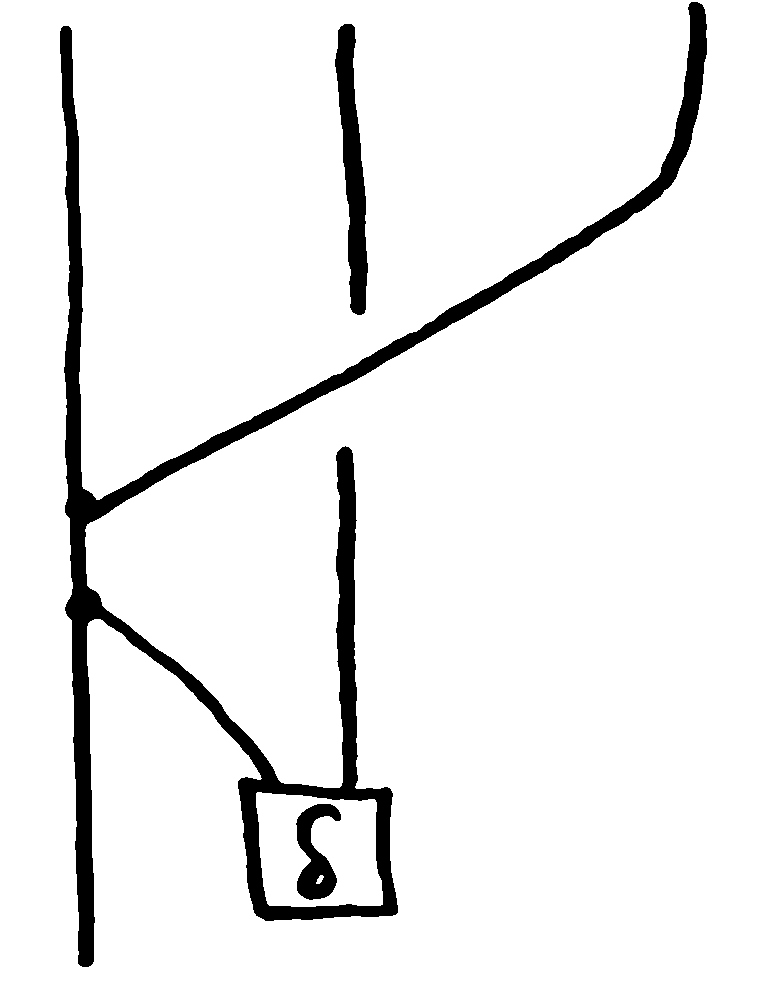}~~=~~\scanPIC{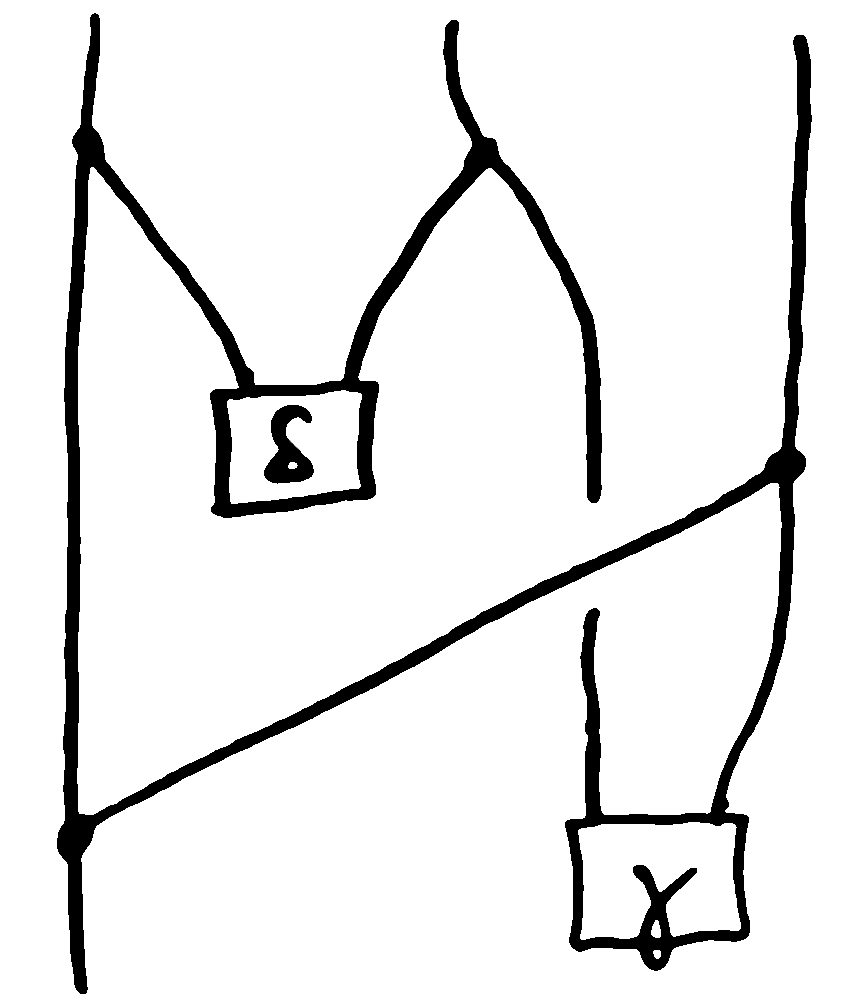}
\ee
{\bf Case 12:}
\be\tag{P-12}\label{eq:pentagon-12}
\scanPIC{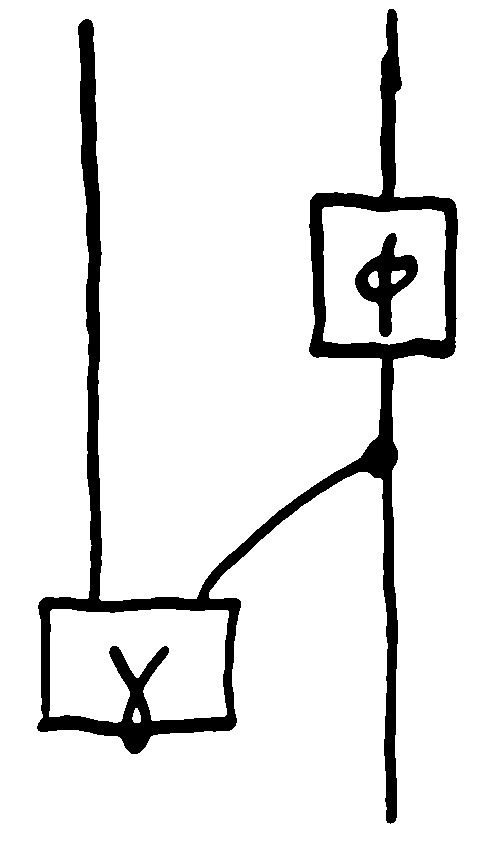}~~=~~\scanPIC{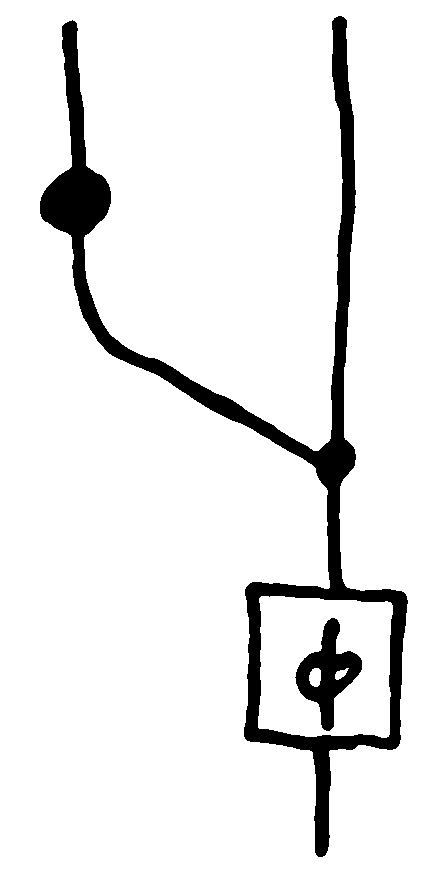}
\ee
{\bf Case 13:}
\be\tag{P-13}\label{eq:pentagon-13}
\scanPIC{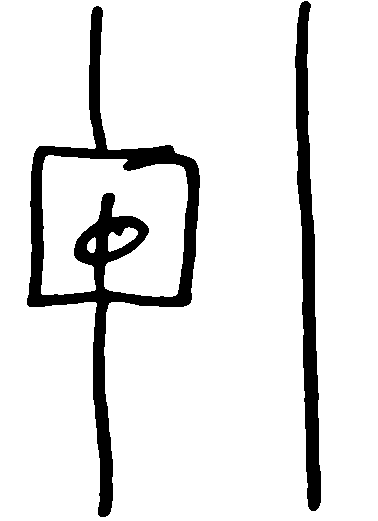}~~=~~\scanPIC{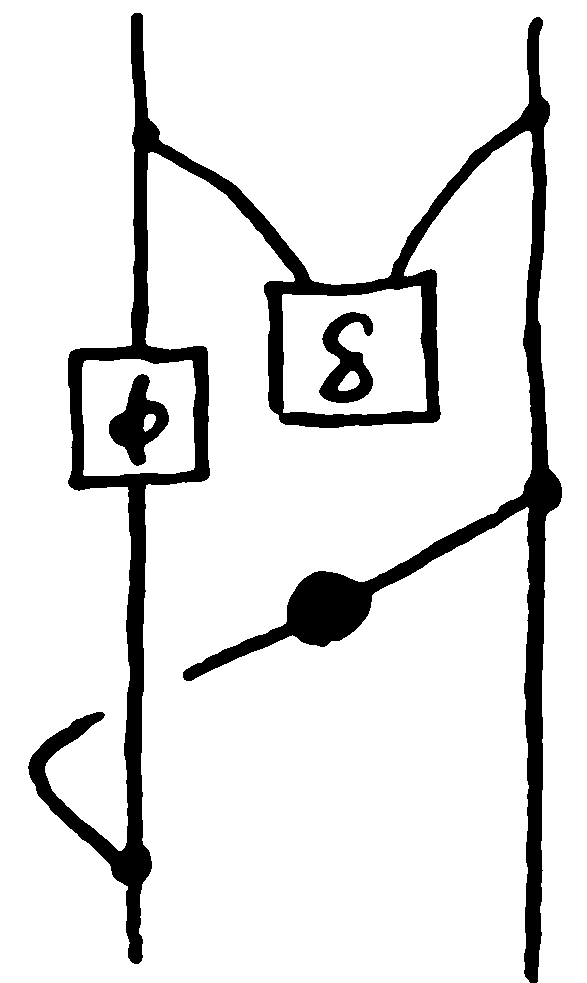}
\ee
{\bf Case 14:}
\be\tag{P-14}\label{eq:pentagon-14}
\scanPIC{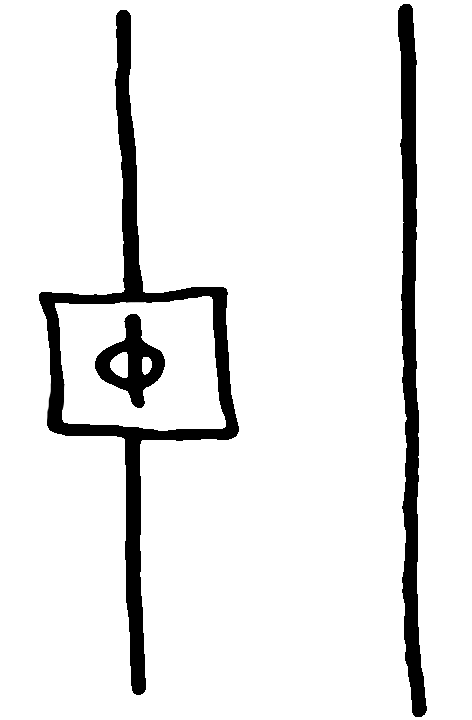}~~=~~\scanPIC{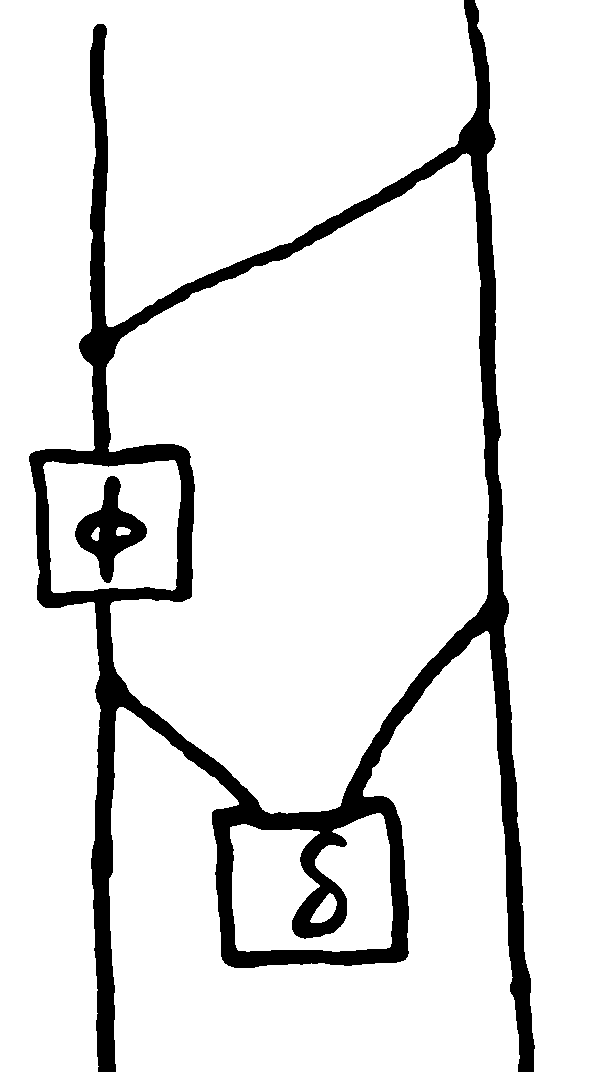}
\ee
{\bf Case 15:}
\be\tag{P-15}\label{eq:pentagon-15}
\scanPIC{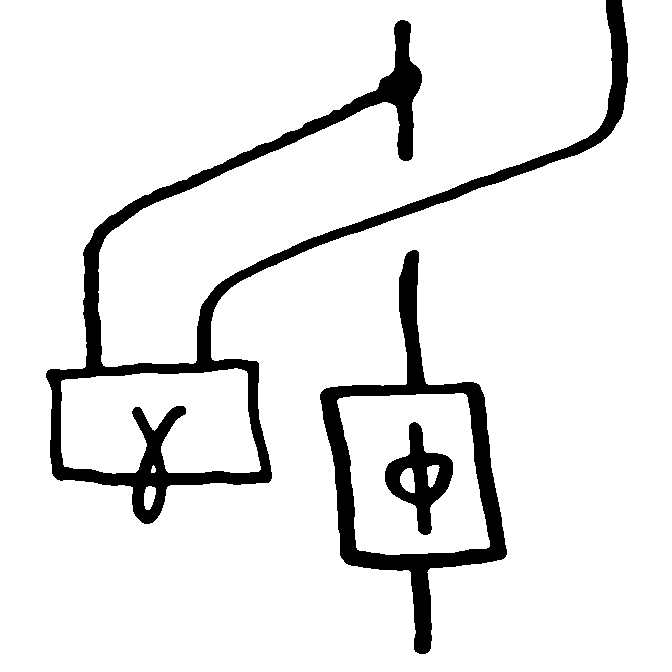}~~=~~\scanPIC{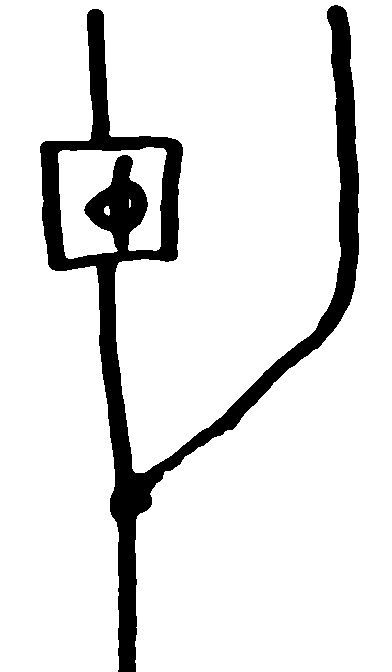}
\ee
{\bf Case 16:}
\be\tag{P-16}\label{eq:pentagon-16}
\scanPIC{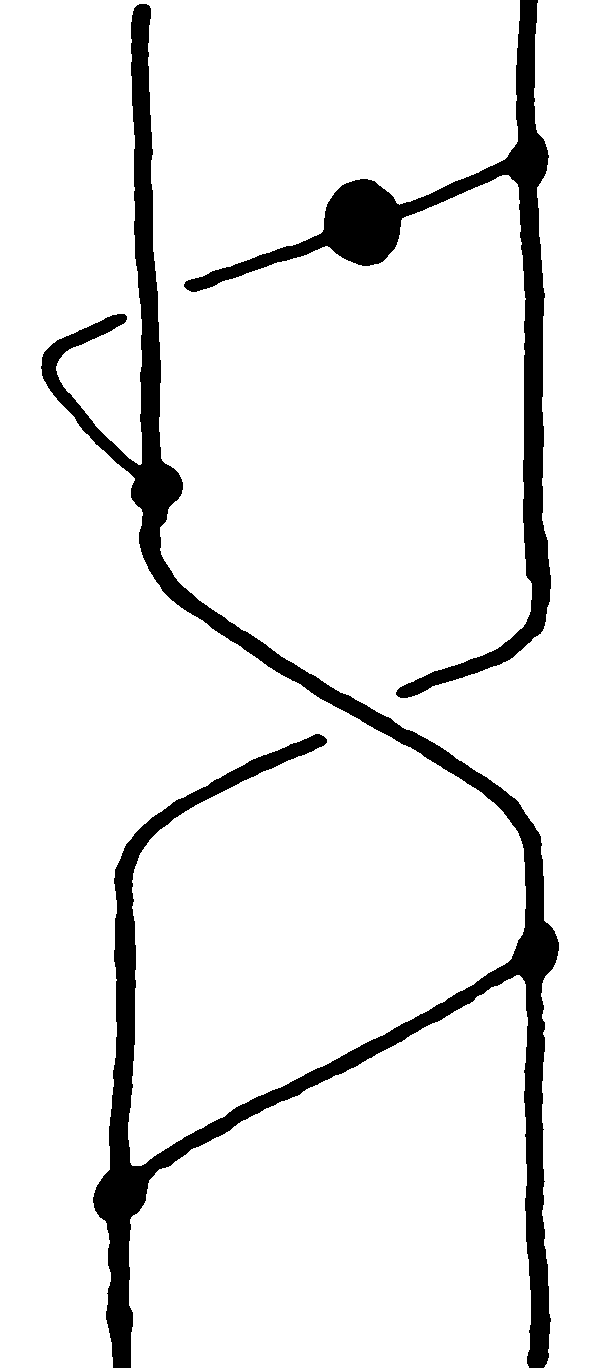}~~=~~\scanPIC{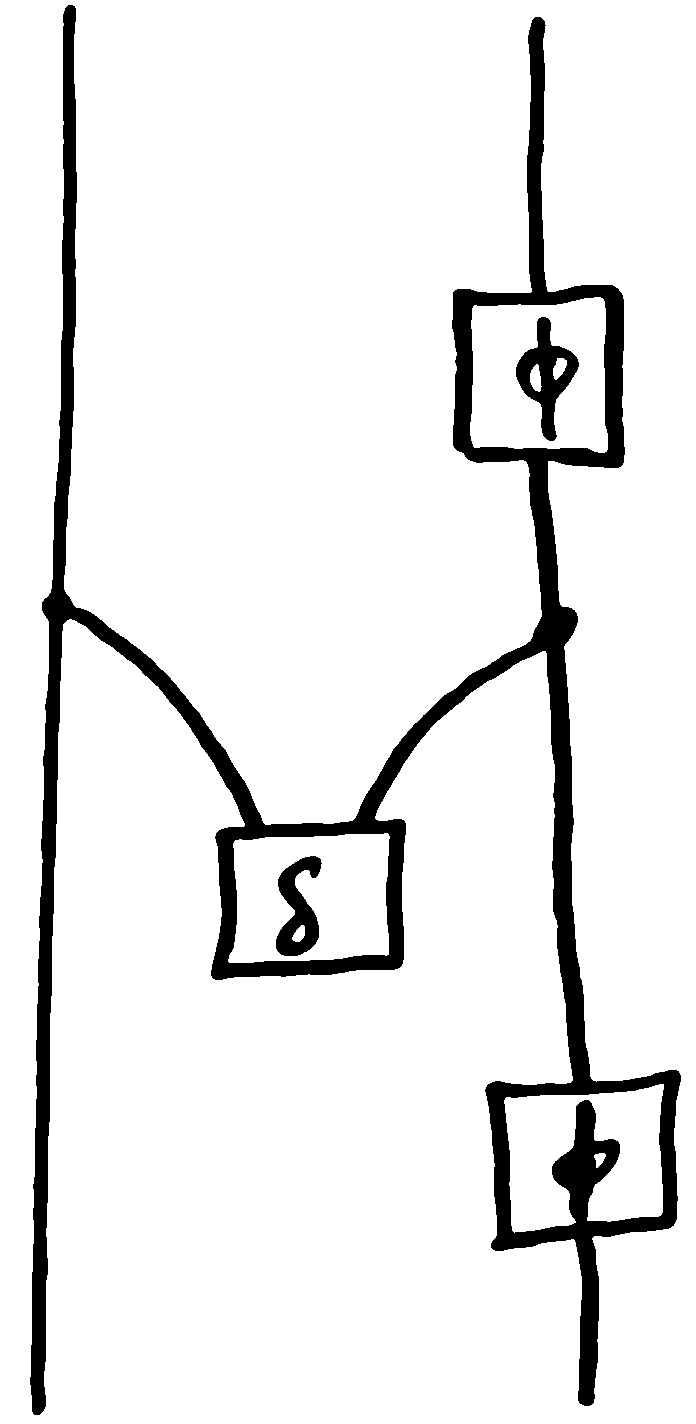}
\ee

\subsection{From pentagon to Hopf-algebraic data}\label{sec:pent-to-Hopf}

The following proposition shows that, given the ansatz for the associativity isomorphisms in  Section \ref{results}, the Hopf-algebraic data in Theorem \ref{thm:main1}\,(i) precisely parameterises the solutions to the 16 pentagon equations. Note that the proposition does not presuppose that $\gamma,\delta,\phi$ are of the form \eqref{eq:gdp-via-Hopf}.

\begin{proposition} \label{prop:pent-to-Hopf}
Suppose $\gamma,\delta,\phi$ as in \eqref{eq:gdp-def} satisfy the pentagon conditions \eqref{eq:pentagon-3}--\eqref{eq:pentagon-16}. Let $\Gamma,\lambda$ be given by \eqref{eq:Hopf-via-gdp}. Then
\\
(i) $\gamma,\delta,\phi$ are of the form \eqref{eq:gdp-via-Hopf}, 
\\
(ii) the pair $\Gamma,\lambda$ satisfies conditions a)--c) in Theorem \ref{thm:main1}\,(i).
\end{proposition}

The proof requires a series of lemmas. 

\medskip

Define the morphism $g : \one \to H$ as
\be\label{eq:g-def}
  g = (\eps \otimes \id) \circ \delta \ .
\ee
The first lemma shows that $g$ is group-like.

\begin{lemma} \label{lem:g-grouplike}
(i) $\Delta \circ g = g \otimes g$ and $\eps \circ g = \id_\one$. \\
(ii) The multiplicative inverse exists and is given by $g^{-1} = S \circ g = S^{-1} \circ g$.
\end{lemma}

\begin{proof}
Composing \eqref{eq:pentagon-8} with $\eps \otimes \id \otimes \id$ gives $g \otimes g = \Delta \circ g$. By assumption, $(S^{-1} \otimes \id) \circ \delta$ has a multiplicative inverse, say $\tilde\delta$. Composing the product of $(S^{-1} \otimes \id) \circ \delta$ and $\tilde\delta$ with $\eps \otimes \eps$ gives $(\eps \circ g) \cdot x = \id_\one$ with $x=(\eps \otimes \eps) \circ \tilde\delta$. Thus $\eps \circ g$ is an invertible element in $\End(\one)$. Composing $g \otimes g = \Delta \circ g$ with $\eps \otimes \eps$ gives $(\eps \circ g) \cdot (\eps \circ g) = (\eps \circ g)$, and by invertibility we must have $\eps \circ g = \id_\one$. This establishes part (i). For part (ii), note that
\be\begin{array}{l}
  \mu \circ (S^{-1} \otimes \id) \circ (g \otimes g)
  = \mu \circ (S^{-1} \otimes \id) \circ c_{H,H}^{-1} \circ (g \otimes g)
\\[.5em]
  \qquad = \mu \circ (S^{-1} \otimes \id) \circ c_{H,H}^{-1} \circ \Delta \circ g
  = \eta \cdot (\eps \circ g) = \eta \ ,
\end{array}
\ee
where Figure \ref{fig:Hopf-property}\,h) was used. The other cases are checked analogously.
\end{proof}

By assumption, the morphism $\gamma : \one \to H \otimes H$ has a multiplicative inverse $\gamma^{-1}$. 

\begin{lemma}\label{lem:gamma-inv-via-S}
(i) $(\eps \otimes \id) \circ \gamma = \eta = (\id \otimes \eps) \circ \gamma$. \\
(ii) The multiplicative inverse of $\gamma$ satisfies $\gamma^{-1} = (S \otimes \id) \circ \gamma = (\id \otimes S) \circ \gamma$.
\end{lemma}

\begin{proof}
For part (i) insert \eqref{eq:pentagon-12} into $(\eps \ot \phi^{-1}) \circ (-) \circ \eta$. This results in $(\eps \ot \id) \circ \gamma = \eta$. The second identity is obtained by inserting \eqref{eq:pentagon-15} into $(\id \ot \eps) \circ (-) \circ \phi^{-1} \circ \eta$.
For part (ii) let $\gamma' = (S \otimes \id) \circ \gamma$. Since $\gamma$ was already assumed to possess a multiplicative inverse, it is enough to check $\mu_{H \otimes H} \circ (\gamma' \otimes \gamma) = \eta \otimes \eta$. This follows from
\be
  \scanpic{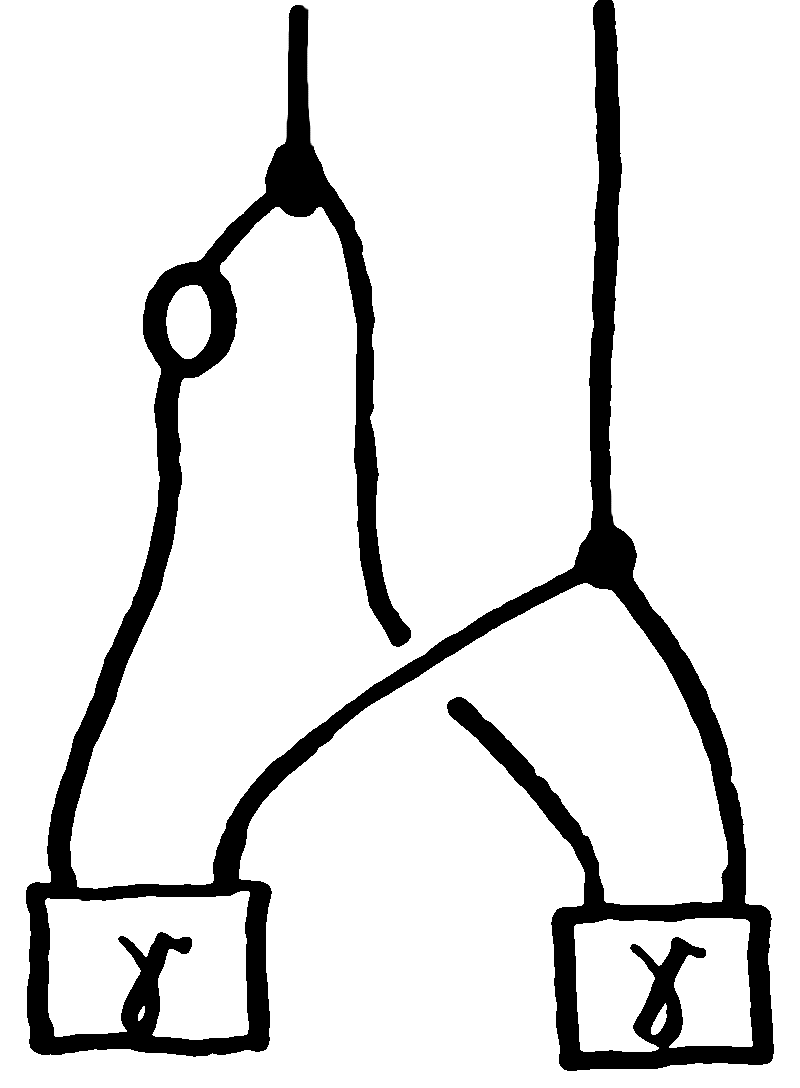}
  ~\overset{\text{\eqref{eq:pentagon-3}}}=~
  \scanpic{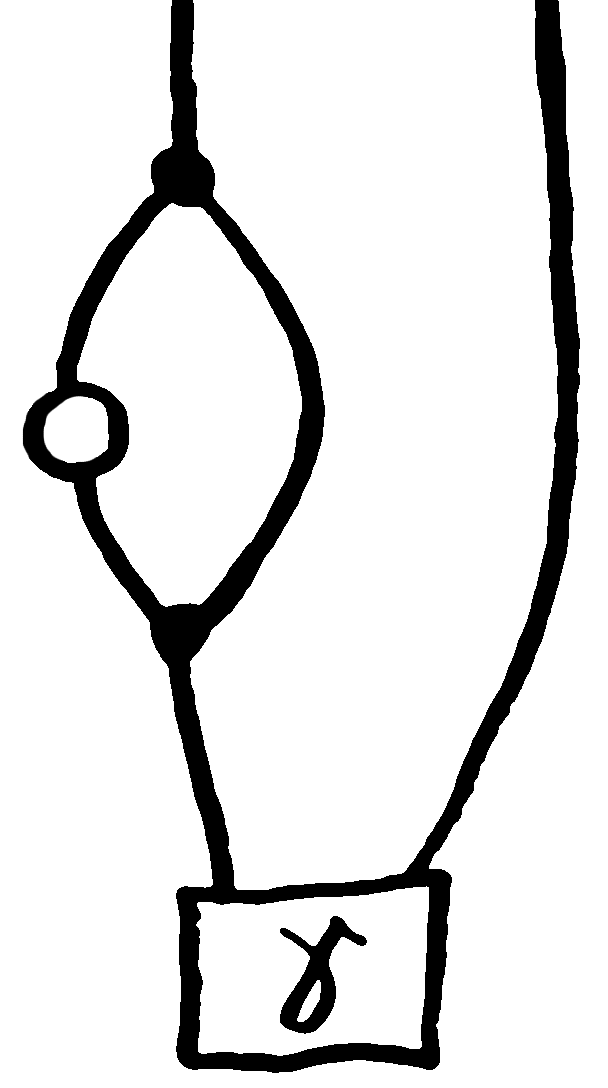}
  ~\overset{\text{Fig.\,\ref{fig:Hopf-property}\,g)}}=~
  ((\eta \circ \eps) \otimes \id) \circ \gamma
  ~\overset{\text{(i)}}=~
  \eta \otimes \eta \ .
\ee
The expression $\gamma^{-1} =  (\id \otimes S) \circ \gamma$ is verified analogously via \eqref{eq:pentagon-4}.
\end{proof}

\begin{lemma}\label{lem:delta+phi-via-gamma}
The morphisms $\delta$ and $\phi$ can be written as
$$
  \delta ~=~ \scanPIC{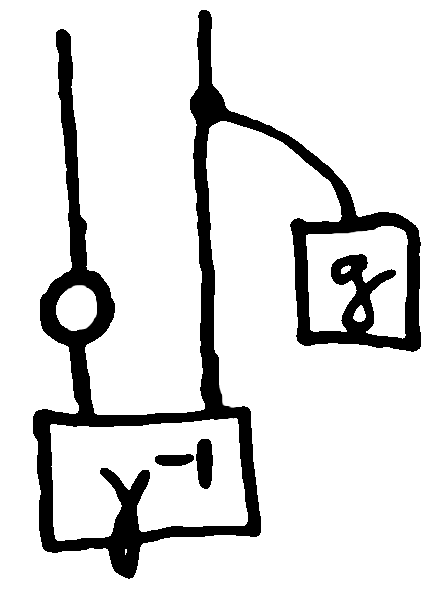}
  \qquad , \qquad
  \phi ~=~ \scanPIC{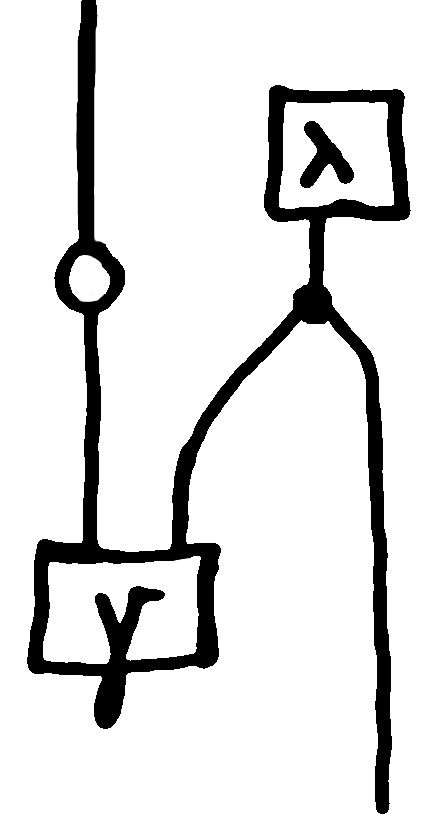} \quad . 
$$
\end{lemma}

\begin{proof}
Inserting \eqref{eq:pentagon-7} into $(\eps \otimes \id \otimes \id) \circ (-) \circ \eta$ results in the identity 
\be
  \eta \otimes g = \mu_{H \otimes H} \circ \big\{ \gamma \otimes ((S^{-1} \otimes \id) \circ \delta) \big\} \ .
\ee  
Multiplying this in $H\otimes H$ with $\gamma^{-1}$ and composing with $S \otimes \id$ results in the expression for $\delta$. The expression for $\phi$ follows by composing \eqref{eq:pentagon-12} with $S \otimes \eps$ from the left (i.e.\ `from the top').
\end{proof}

\begin{lemma}\label{lem:delta-braid}
The morphism $\delta$ can be expressed as
$$
  \delta ~=~ \scanPIC{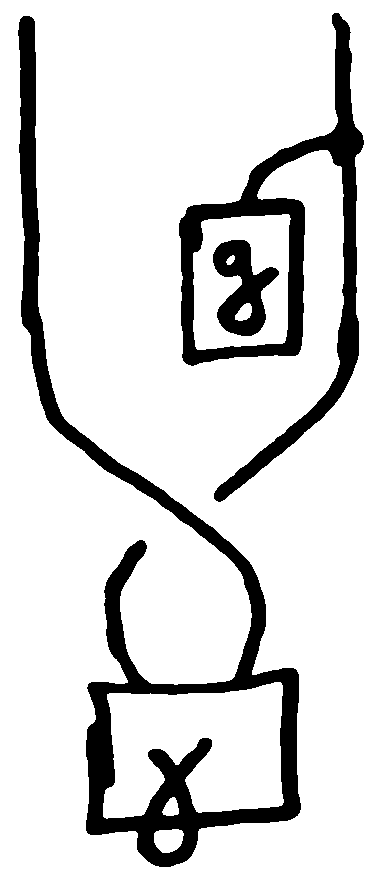} \quad .
$$
\end{lemma}

\begin{proof}
Apply $(\eps \otimes c_{H,H}^{-1}) \circ (-) \circ \eta$ to \eqref{eq:pentagon-10} and substitute the definition of $g$.
\end{proof}

For later use we note the following implication of the above lemma: Combining  the two expressions for $\delta$ given in Lemmas \ref{lem:delta+phi-via-gamma} and \ref{lem:delta-braid} with Lemma \ref{lem:gamma-inv-via-S}\,(ii), we obtain the equation
\be\label{eq:gamma-braid}
  c_{H,H}^{-1} \circ \gamma = (\id \otimes \{ \Ad_{g^{-1}} \circ S^2 \} ) \circ \gamma \ .
\ee

We say that a morphism $b : \one \to U \otimes V$ is {\em non-degenerate} if the induced morphism $U^\vee \to V$ is invertible (or, equivalently, the induced morphism $V^\vee \to U$ is invertible). In other words, $b$ is the coevaluation of a duality between $U$ and $V$. We will need the following general observation.

\begin{lemma}\label{lem:b-UU-secial-duality-property}
Let $b : \one \to U \ot U$ and $d: U \ot U \to \one$ be two morphisms in $\Sc$ such that
$(\id \ot d) \circ (b \ot \id) = \id$ and $c_{U,U} \circ b = (\id \ot \tau) \circ b$ 
(or $c_{U,U}^{-1} \circ b = (\id \ot \tau) \circ b$)
for some automorphism $\tau$ of $U$. Then also
$(d \ot \id) \circ (\id \ot b) = \id$, i.e.\ $b$ is non-degenerate.
\end{lemma}

\begin{proof}
All we need to show is that there exists $d' : U \ot U \to \one$ with $(d' \ot \id) \circ (\id \ot b) = \id$. By the properties of dualities, this implies $d=d'$.
The first condition in the lemma can be rewritten as
$( (d \circ c_{U,U}^{-1}) \ot \id ) \circ ( \id \ot ( c_{U,U} \circ b) ) = \id$, which, using the second condition, becomes
$( (d \circ c_{U,U}^{-1}) \ot \id ) \circ ( \id \ot ( (\id \ot \tau) \circ b) ) = \id$.
Conjugating with $\tau$ we get the desired result with $d' = d \circ c_{U,U}^{-1} \circ (\tau \ot \id)$.
The same argument works when replacing all $c_{U,U}$ by their inverses.
\end{proof}

\begin{lemma}\label{lem:gamma-inv}
$\gamma$ is non-degenerate.
\end{lemma}

\begin{proof}
The expression for $\phi$ in Lemma \ref{lem:delta+phi-via-gamma} gives us the first condition in Lemma \ref{lem:b-UU-secial-duality-property} (by composing with $\phi^{-1}$ and then conjugating with $S$). The second condition comes from \eqref{eq:gamma-braid}. Thus $\gamma$ is non-degenerate.
\end{proof}

\begin{lemma}\label{lem:Gamma-Hopf-iso}
The morphism $\Gamma : H^\vee \to H$ is a Hopf algebra isomorphism.
\end{lemma}

\begin{proof}
That $\Gamma$ in \eqref{eq:Hopf-via-gdp} is an isomorphism follows from Lemma \ref{lem:gamma-inv}. Recall the definition of the structure maps for the dual Hopf algebra $H^\vee$ given in \eqref{eq:H^-structuremaps}. That $\Gamma \circ S^\vee = S \circ \Gamma$ is immediate from Lemma \ref{lem:gamma-inv-via-S}\,(ii). That $\Gamma$ preserves unit and counit is the statement of Lemma \ref{lem:gamma-inv-via-S}\,(i). For the multiplication consider the equivalences
\begin{align}
  &\mu \circ (\Gamma \otimes \Gamma) = \Gamma \circ \mu_{H^\vee}
  \nonumber \\
  &\Leftrightarrow \quad
  \scanPIC{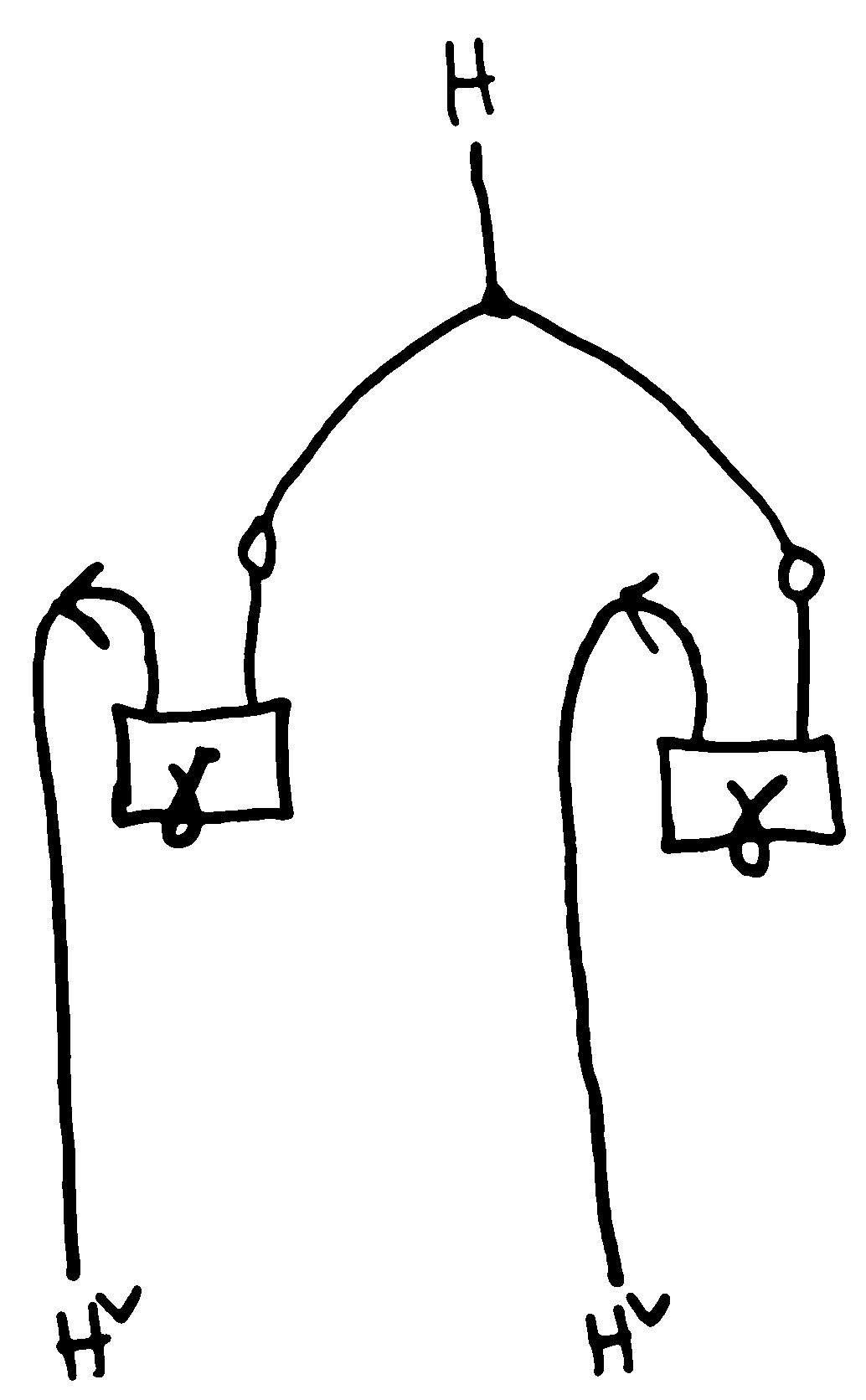}
  ~=~
  \scanPIC{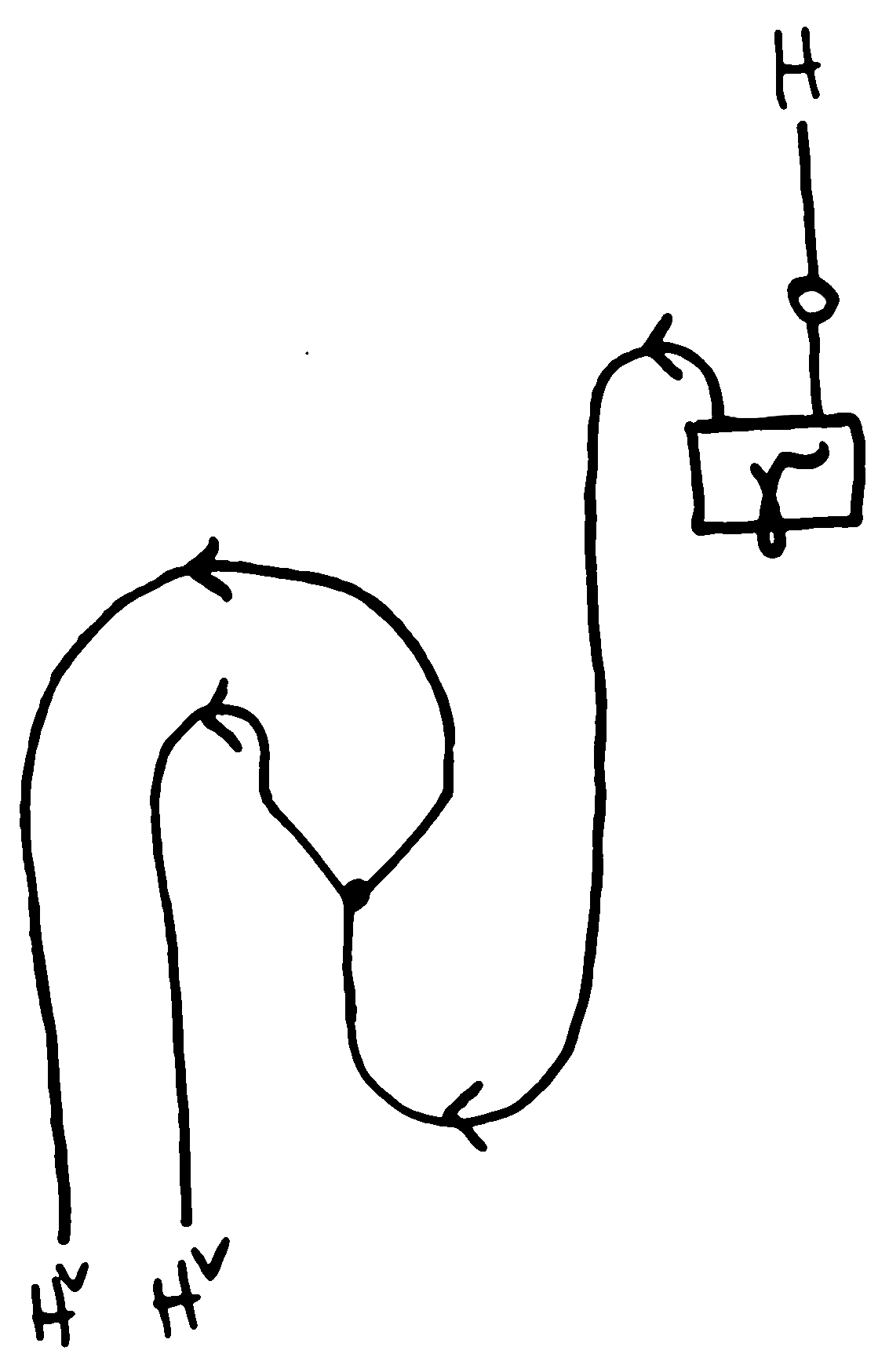}
  \quad \Leftrightarrow \quad
  \scanPIC{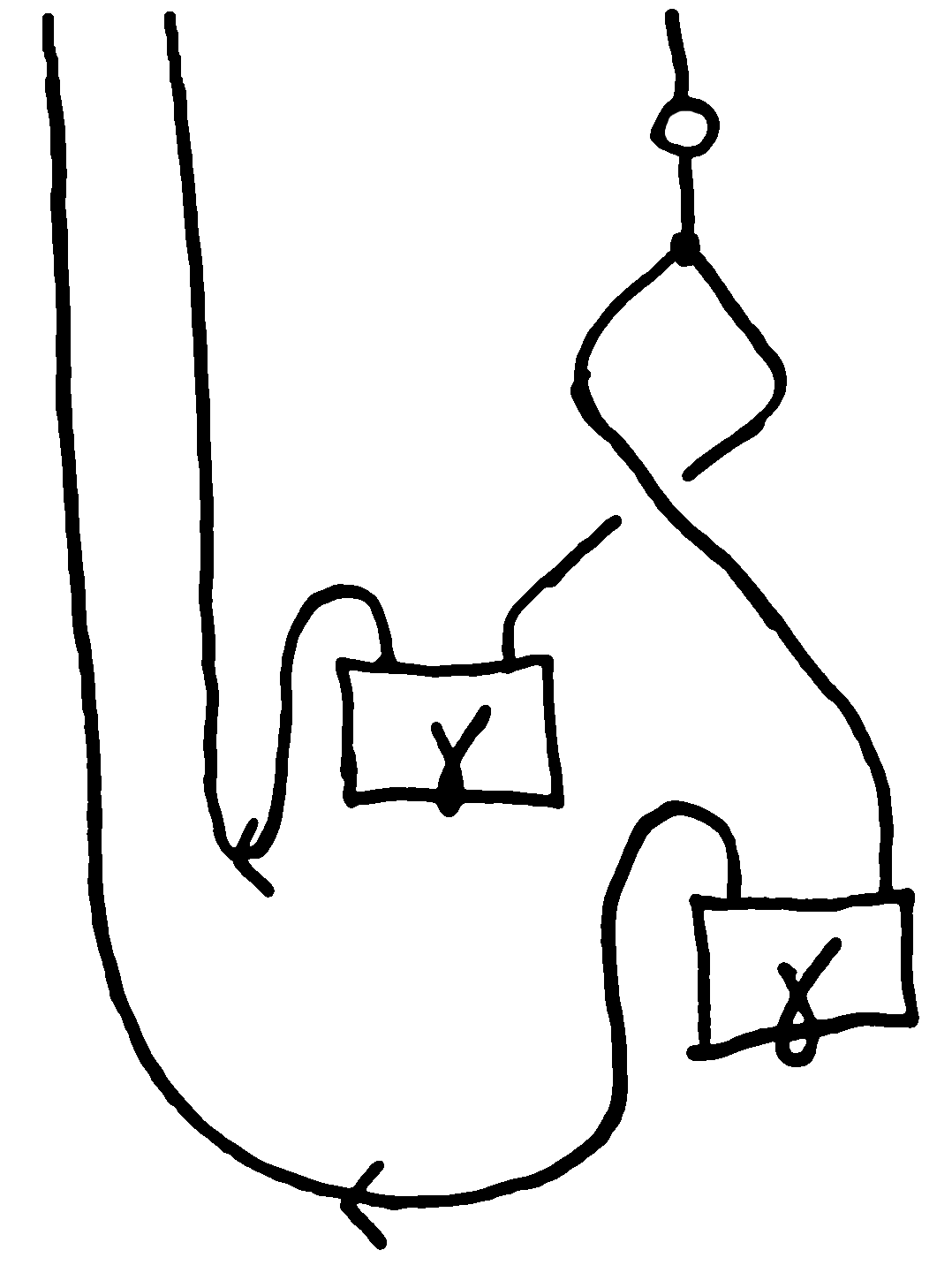}
  ~=~
  \scanPIC{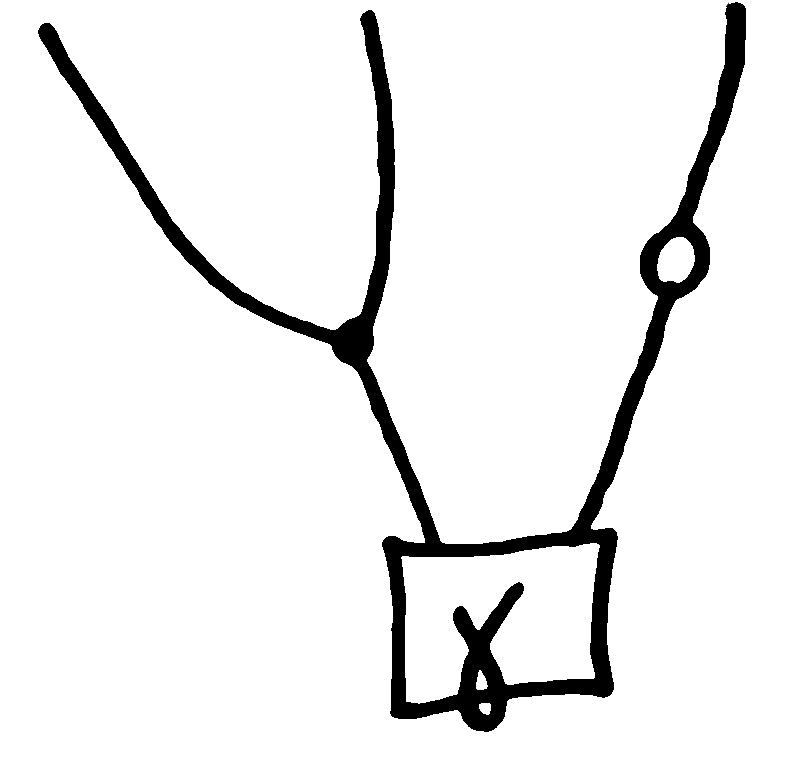}
  \nonumber \\
  &\Leftrightarrow \quad
  \text{\eqref{eq:pentagon-3}} \ .
\end{align}
In the first equivalence the definition of $\Gamma$ and $\mu_{H^\vee}$ has been substituted and in the second equivalence the antipode has been moved past the product and duality maps have been applied. An analogous calculation for the comultiplication yields
\be
  (\Gamma \otimes \Gamma) \circ \Delta_{H^\vee} = \Delta \circ \Gamma
  \quad \Leftrightarrow \quad
  \text{\eqref{eq:pentagon-4}} \ .
\ee
\end{proof}

\begin{lemma}\label{lem:lambda-cointegral-etc}
(i) $\lambda$ is a right cointegral: $(\lambda \otimes \id) \circ \Delta = \eta \circ \lambda$.\\
(ii) $g$ is the distinguished group-like element for $\lambda$: $(\id \otimes \lambda) \circ \Delta = g \circ \lambda$.\\  
\end{lemma}

\begin{proof}
Part (i) is obtained by applying $\eps \otimes \id$ to \eqref{eq:pentagon-15} from the left and using Lemma \ref{lem:gamma-inv-via-S}\,(i). 
For part (ii) apply $(\eps \otimes \id) \circ (-) \circ (\id \otimes \eta)$ to \eqref{eq:pentagon-13}. This gives
\be
  \eta \circ \lambda 
  = {}_gM \circ S^{-1} \circ (\id \otimes \lambda) \circ \Delta \ .
\ee
Composing with $S \circ {}_{g^{-1}}M$ results in the required identity (recall from Lemma \ref{lem:g-grouplike} that $g^{-1} = S \circ g$). 
\end{proof} 

\begin{lemma} \label{lem:integral-cointegral}
Set $\Lambda = (\lambda \otimes \Gamma) \circ \coev_H$ and $\Lambda' = (\id \otimes (\lambda \circ \Gamma)) \circ \coev_H$. Then both $\Lambda$ and $\Lambda'$ are left integrals for $H$.
\end{lemma}

\begin{proof}
That $\Lambda$ is a left integral follows from
\be
\scanpic{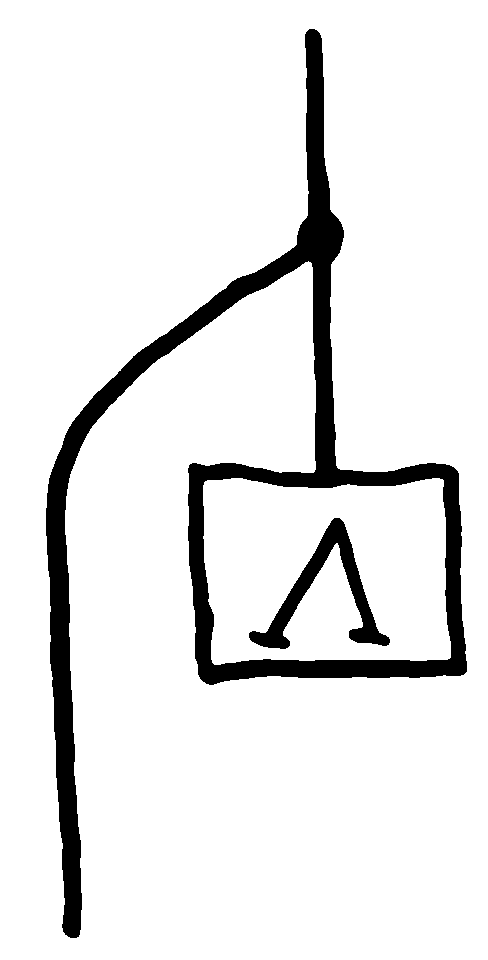} 
~~=~~
\scanpic{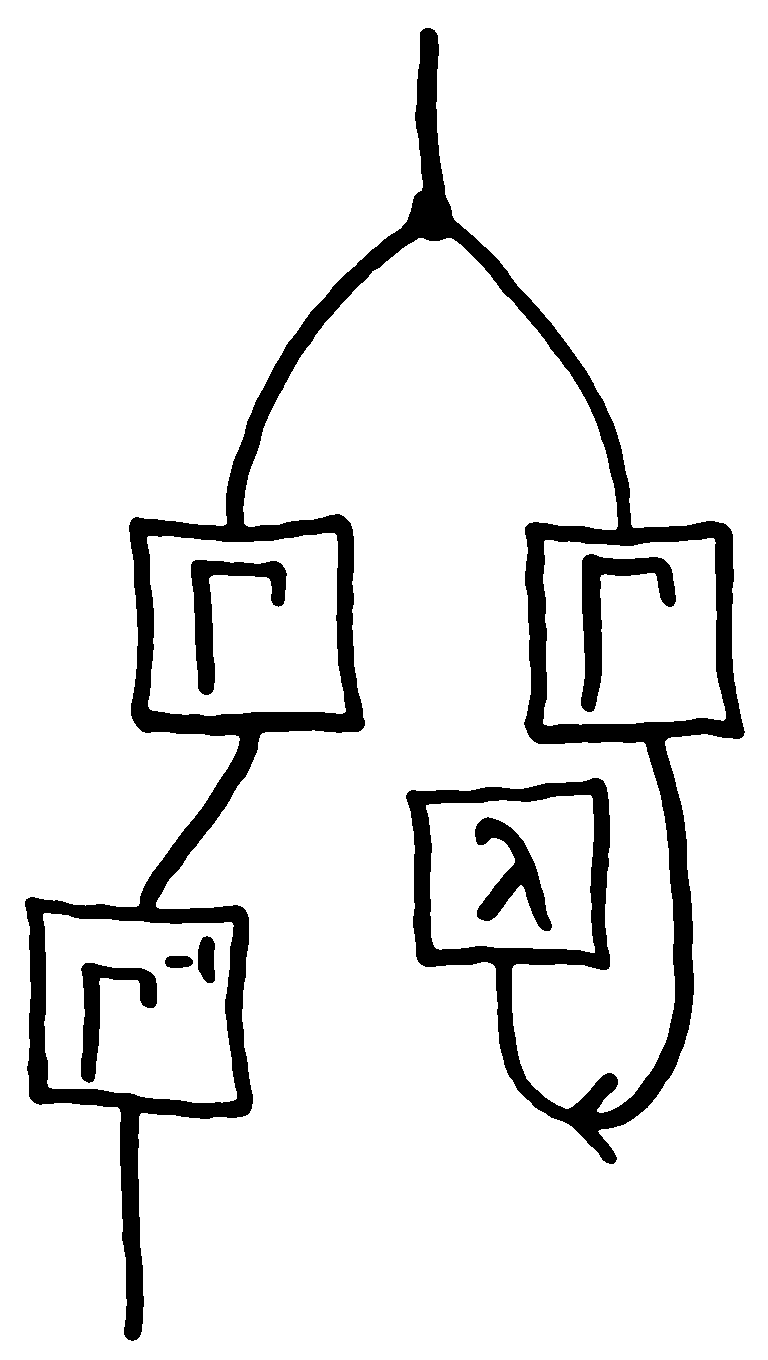}
~~=~~
\scanpic{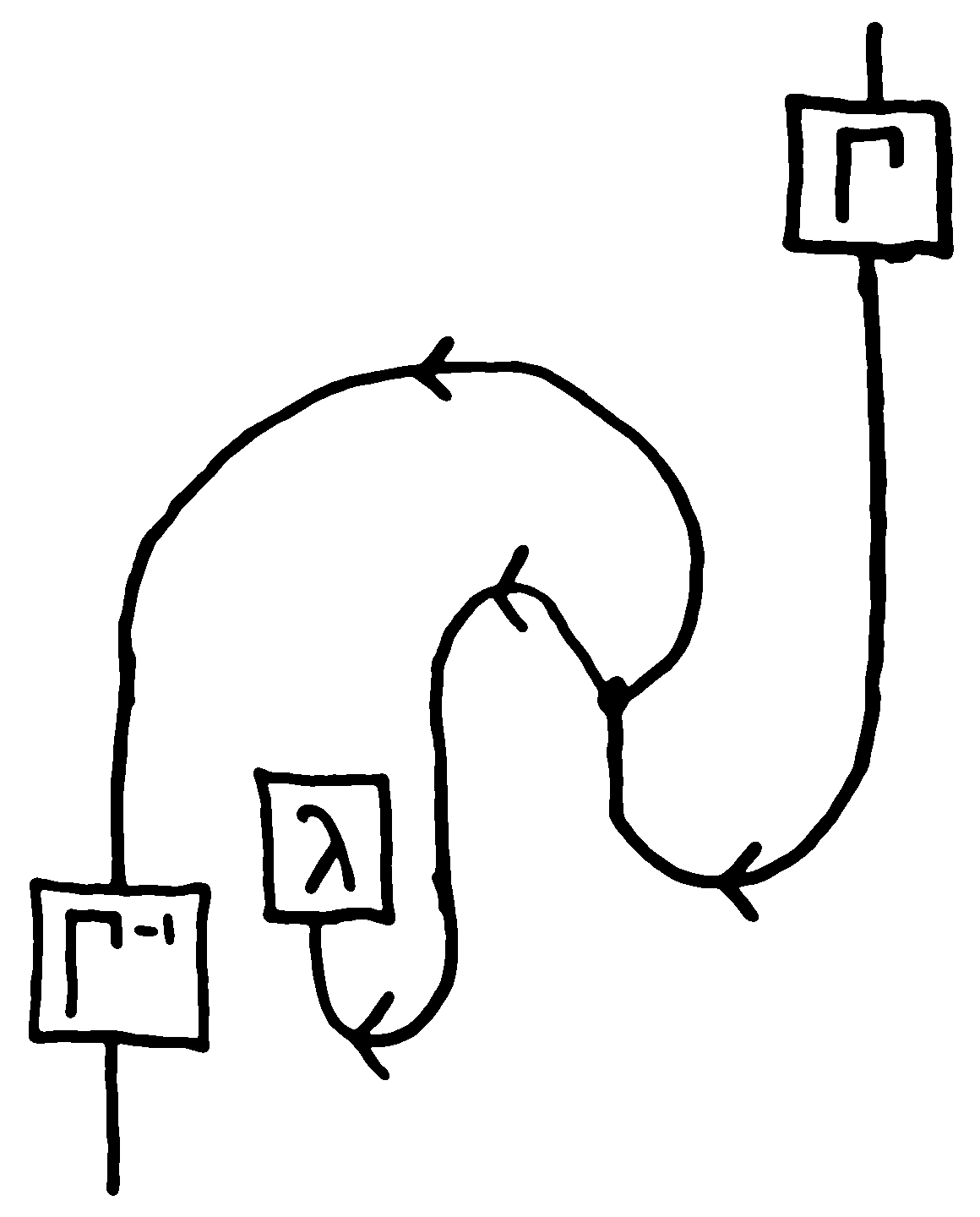}
\quad .
\ee
The second step uses that $\Gamma$ is an algebra map, i.e.\ that $\mu_H \circ (\Gamma \otimes \Gamma) = \Gamma \circ \mu_{H^\vee}$ with $\mu_{H^\vee}$ as given in \eqref{eq:H^-structuremaps}. In the last expression, apply the right cointegral property of $\lambda$ (Lemma \ref{lem:lambda-cointegral-etc}) and the fact that $\eps_{H^\vee} \circ \Gamma^{-1} = \eps_H$. The argument for $\Lambda'$ is similar: interpret $\mu_H$ as $\Delta_{H^\vee}$ by inserting an extra pair $\Gamma^{-1} \circ \Gamma$ and then use the coalgebra map property of $\Gamma$ to obtain $\Delta_H$.
\end{proof}

\begin{lemma}\label{lem:norm-cond}
The normalisation condition  $(\lambda \otimes \lambda) \circ \gamma^{-1} = \id_\one$ holds.
\end{lemma}

\begin{proof}
Applying $(\eps \otimes \eps) \circ (-) \circ (\eta \otimes \eta)$ to \eqref{eq:pentagon-16} and using Lemma \ref{lem:delta+phi-via-gamma} gives
\be
  \id_\one = \scanPIC{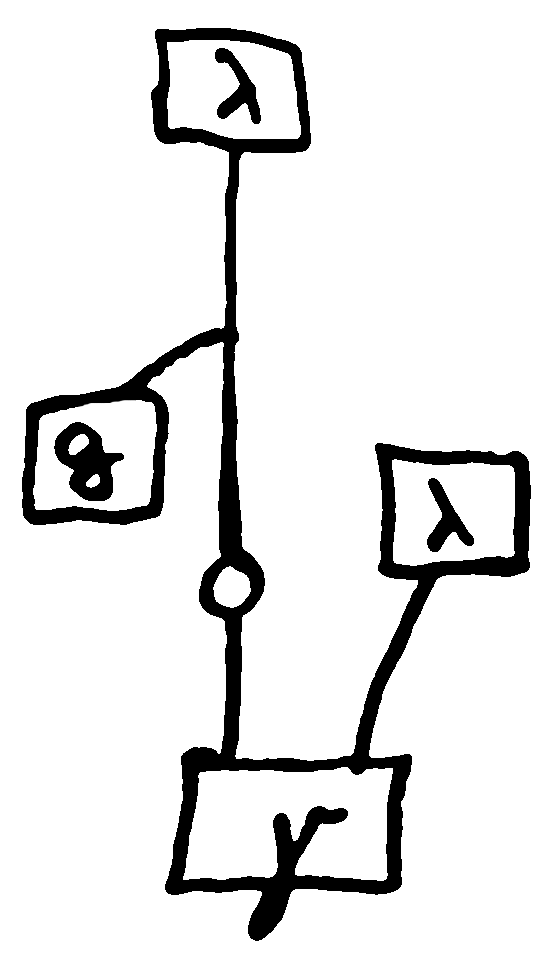} \quad .
\ee
Substituting the definition of $\Gamma$ to replace $\gamma$ (use Lemma \ref{lem:gamma-inv-via-S}\,(ii) first) and then replacing both $\Gamma$ and $\lambda$ by $\Lambda'$ from Lemma \ref{lem:integral-cointegral} one arrives at $\id_\one = \lambda \circ {}_gM \circ \Lambda'$. But by that lemma, $\Lambda'$ is a left integral, and so the product with $g$ splits off as $\eps \circ g = \id_\one$. 
In the resulting identity $\id_\one = \lambda \circ \Lambda'$, revert $\Lambda'$ back to $\Gamma$ and $\lambda$ and then express $\Gamma$ in terms of $\gamma$ via \eqref{eq:Hopf-via-gdp}. Together with Lemma \ref{lem:gamma-inv-via-S}\,(ii), this then gives the normalisation condition.
\end{proof} 

\begin{proof}[Proof of Proposition \ref{prop:pent-to-Hopf}]
Part (i) is implied by Lemmas \ref{lem:delta+phi-via-gamma} and \ref{lem:gamma-inv-via-S}\,(ii). That $\Gamma$ is a Hopf-algebra isomorphism was shown in Lemma \ref{lem:Gamma-Hopf-iso}. This establishes property a) in Theorem \ref{thm:main1}\,(i). Property b) amounts to Lemmas \ref{lem:g-grouplike}\,(ii), \ref{lem:lambda-cointegral-etc} and \ref{lem:norm-cond}. Uniqueness of $g$ can be seen as follows. Define $\Lambda'$ as in Lemma \ref{lem:integral-cointegral}. Then the normalisation condition reads $\lambda \circ \Lambda' = \id_\one$. Thus, applying $\Lambda'$ to both sides of the identity $(\id \otimes \lambda) \circ \Delta = g \circ \lambda$ gives an expression for $g$ in terms of $\Gamma$ and $\lambda$. It remains to check property c) in Theorem \ref{thm:main1}\,(i). We have the equalities
\be \label{eq:Gamma-dual-via-gamma}
  \scanPIC{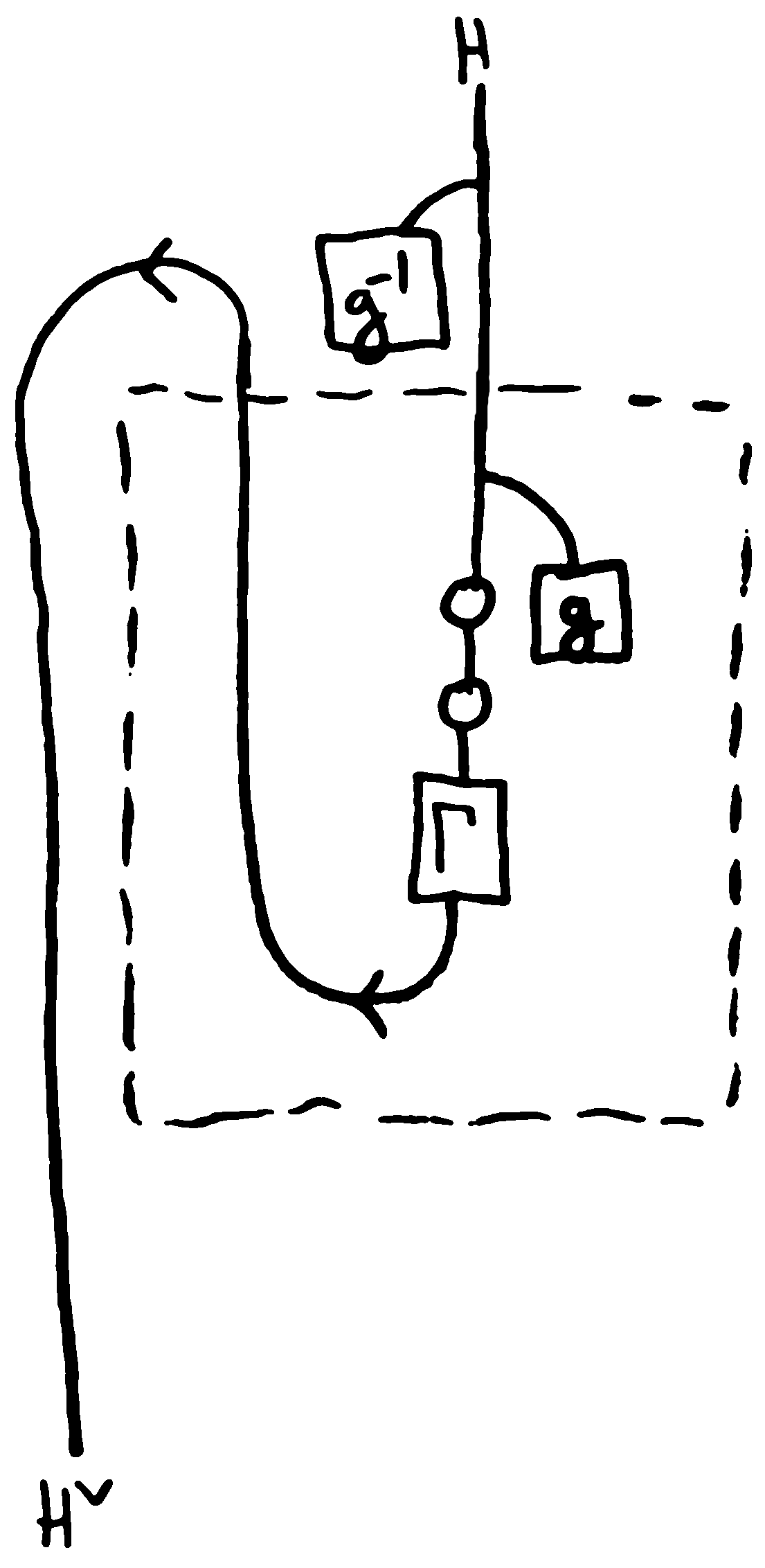}
  ~=~
  \scanPIC{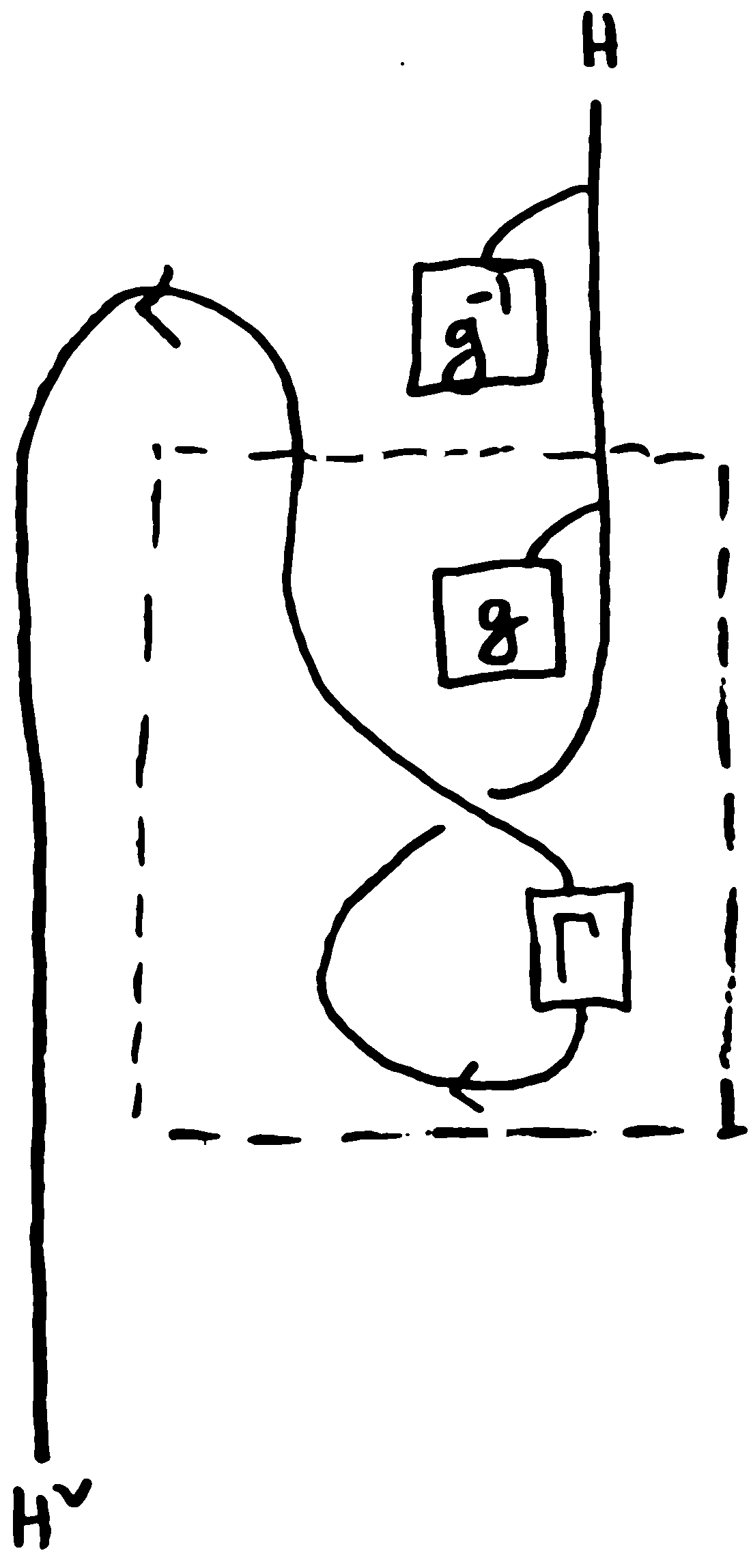}
  ~=~
  \scanPIC{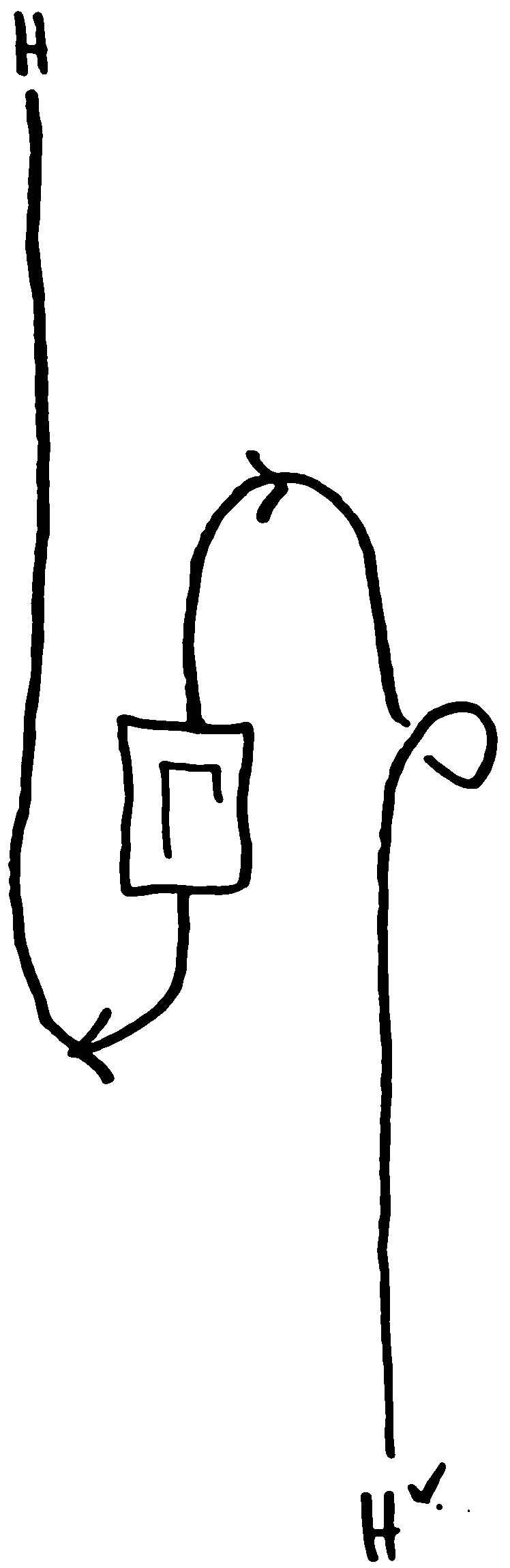}
  \quad .
\ee
In the first step, the equality of the two morphisms in the dashed box is obtained from substituting \eqref{eq:gdp-via-Hopf} into \eqref{eq:gamma-braid} and using $\Ad_g \circ S = S \circ \Ad_g$, which in turn follows from Lemma \ref{lem:g-grouplike}\,(ii). In the second step $g^{-1}$ is cancelled against $g$ and the string diagram is deformed. The little loop is equal to $\theta_{H^\vee}$. Note that the right hand side is not yet equal to $\Gamma^\vee \circ \theta_{H^\vee}$ as it uses two different duality maps. Instead, the right hand side is equal to $\delta_H^{-1} \circ \Gamma^\vee\circ \theta_{H^\vee}$, where $\delta_H : H \to H^{\vee\vee}$ is the natural isomorphism to the double dual. Altogether, we obtain property c) (the morphisms $\theta_H$, $S^2$ and $\Ad_{g^{-1}}$ all commute with each other).
\end{proof}

\subsection{From Hopf-algebraic data to pentagon}\label{sec:Hopf-to-pent}

The following proposition proves Theorem \ref{thm:main1}\,(i).

\begin{proposition} \label{prop:Hopf-to-pent}
Suppose the pair $\Gamma,\lambda$ satisfies properties a)--c) in Theorem \ref{thm:main1}\,(i). Let $\gamma,\delta,\phi$ be given in terms of $\Gamma,\lambda$ by \eqref{eq:gdp-via-Hopf}. Then the 16 cases of the pentagon identity stated in Section \ref{sec:pentagon} are satisfied.
\end{proposition}

The proof will again rely on a series of lemmas. The first is a general property of left/right integrals.

\begin{lemma}\label{lem:move-prod-past-integral}
Let $\Lambda : \one \to H$ be a left integral and let $\tilde\Lambda : \one \to H$ be a right integral. Then
$$
  \raisebox{2.5em}{\text{a)}}
  \quad
  \scanPIC{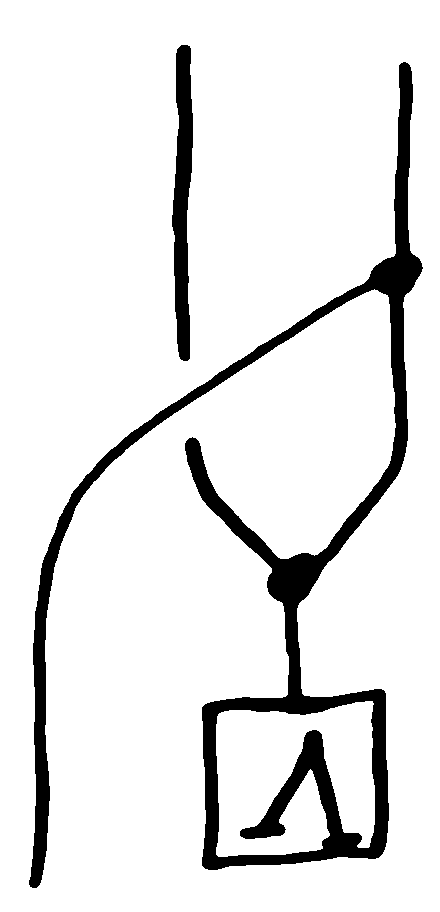}
  =
  \scanPIC{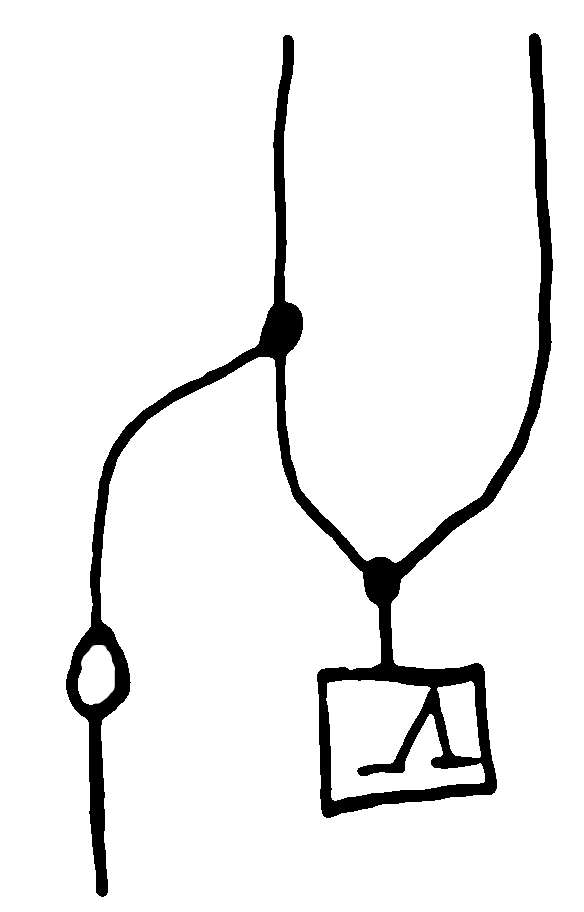}
  \qquad
  \raisebox{2.5em}{\text{b)}}
  \quad
  \scanPIC{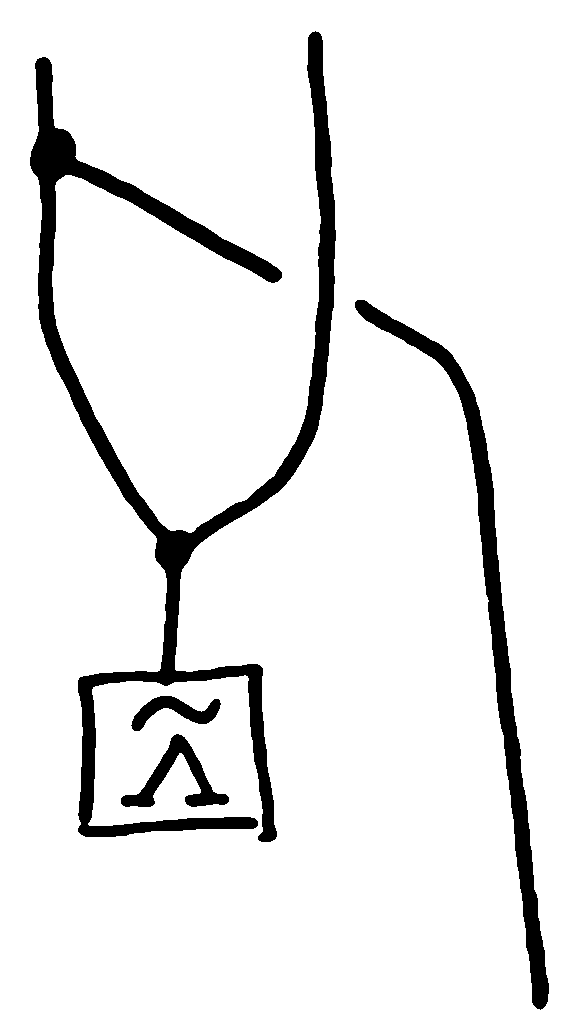}
  =
  \scanPIC{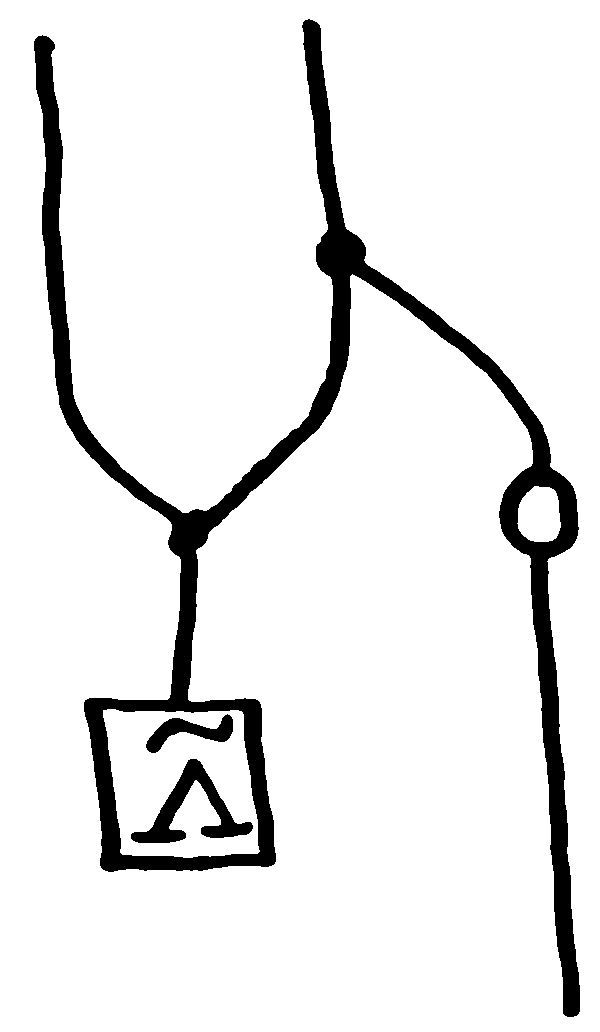}
$$
\end{lemma}

\begin{proof}
This is a braided version of \cite[Lem.\,1.2]{Larson:1988}. To start with, note the following two chains of equalities, which follow from Figure \ref{fig:Hopf-property}\,a),\,g):
\be
\raisebox{3em}{\text{$\alpha$)}}~
\scanpic{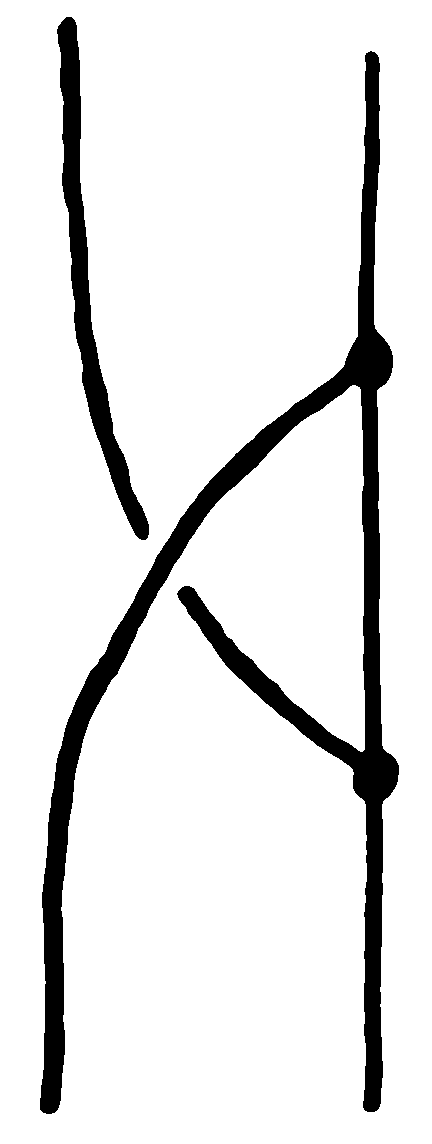}
~=~
\scanpic{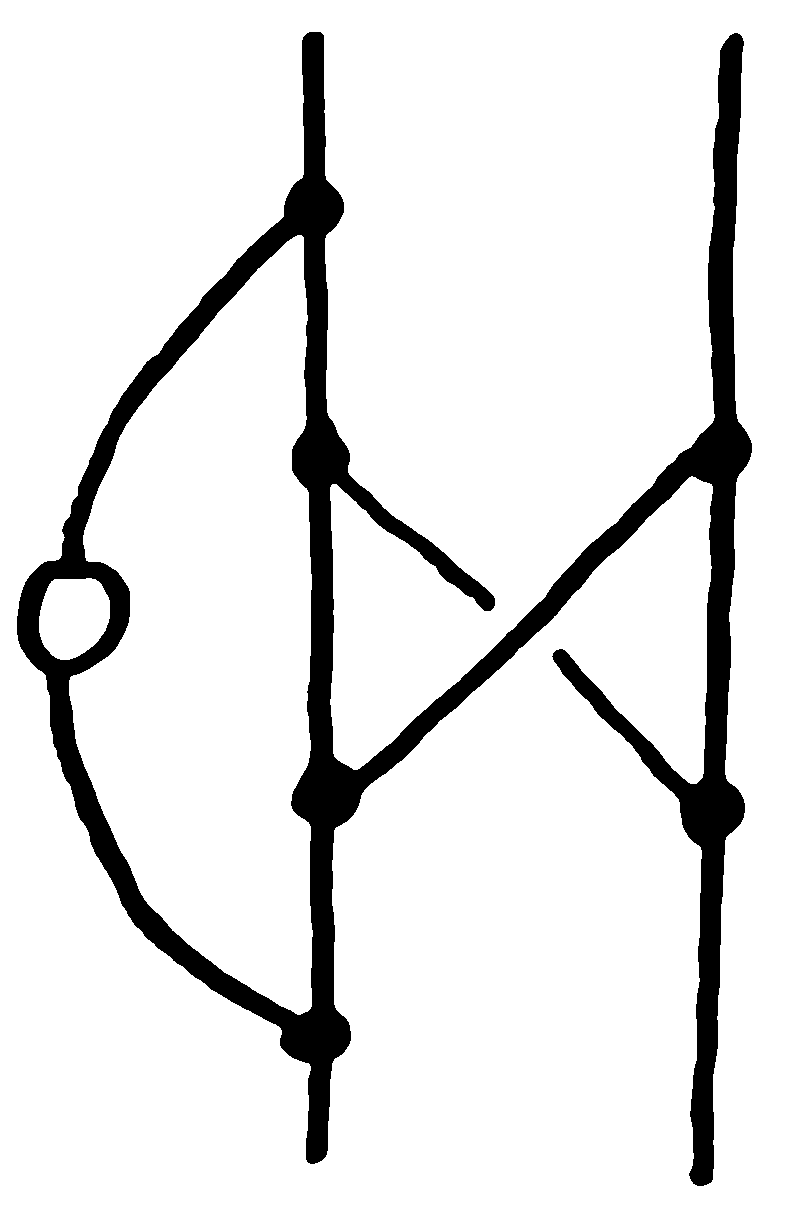}
~=~
\scanpic{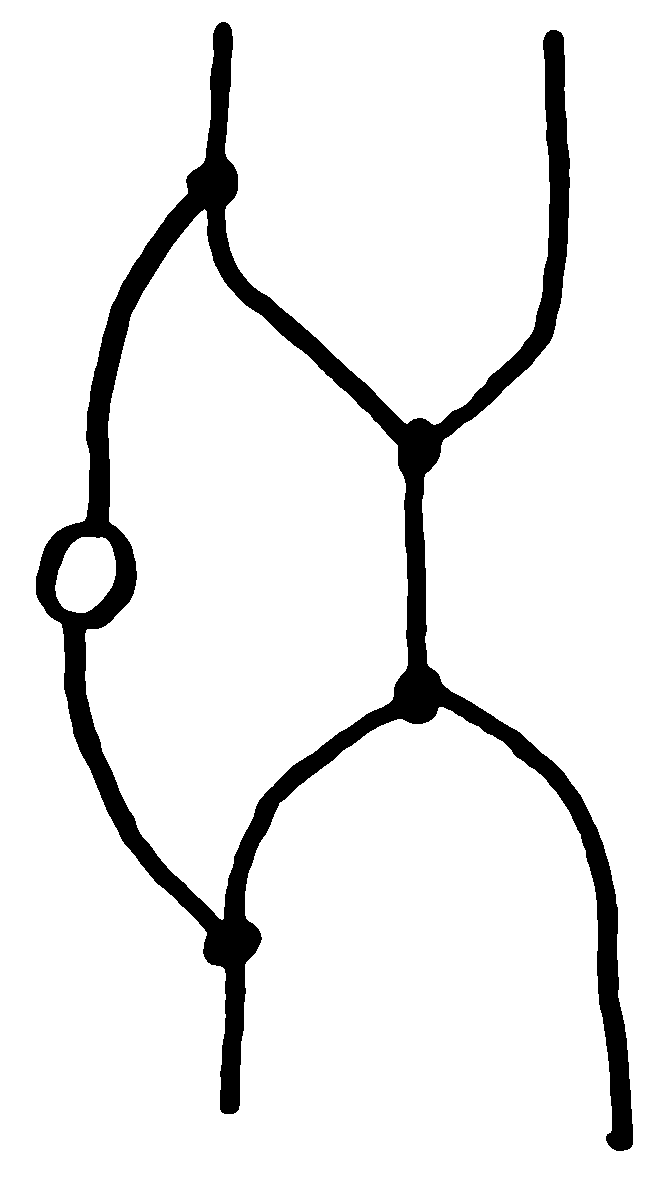}
\qquad
\raisebox{3em}{\text{$\beta$)}}~
\scanpic{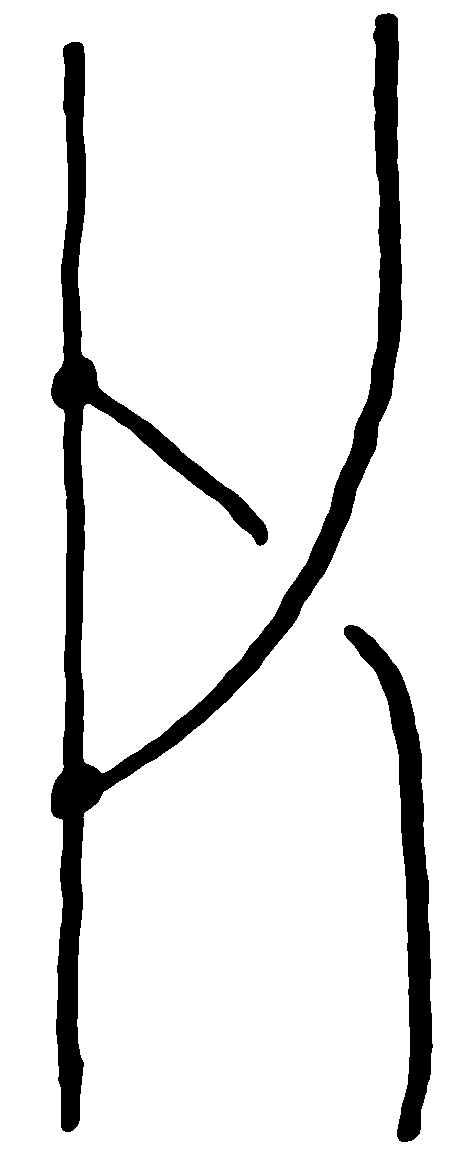}
~=~
\scanpic{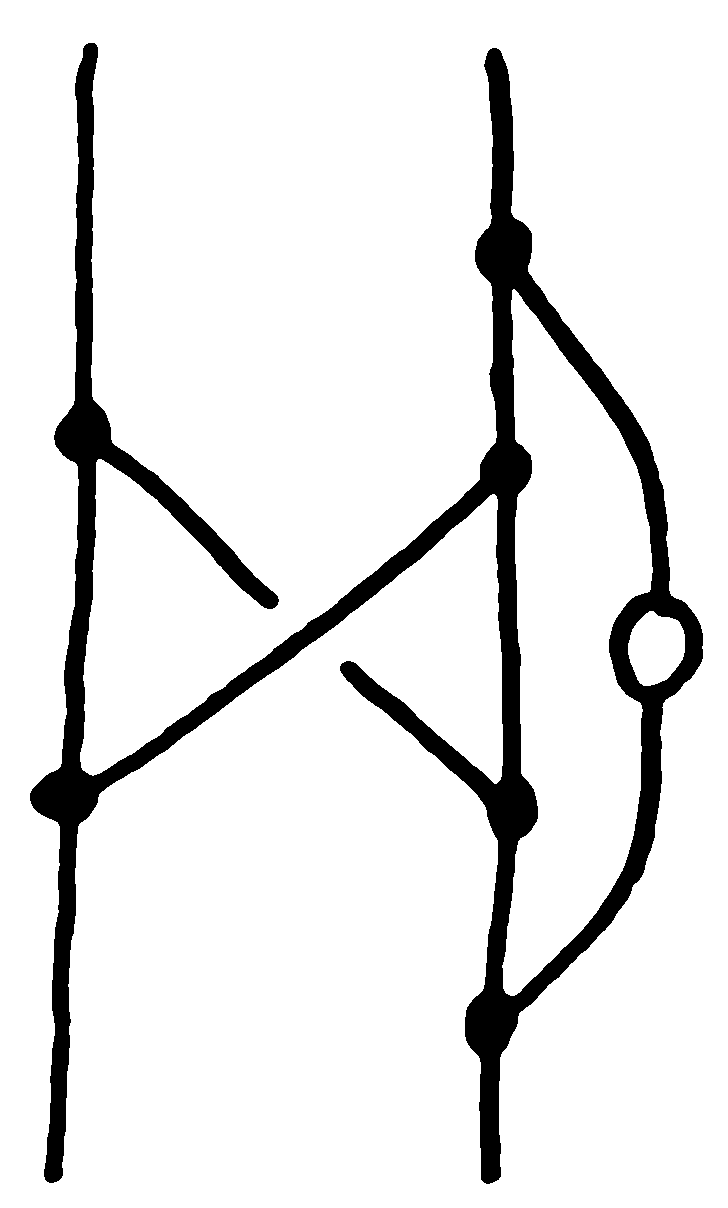}
~=
\scanpic{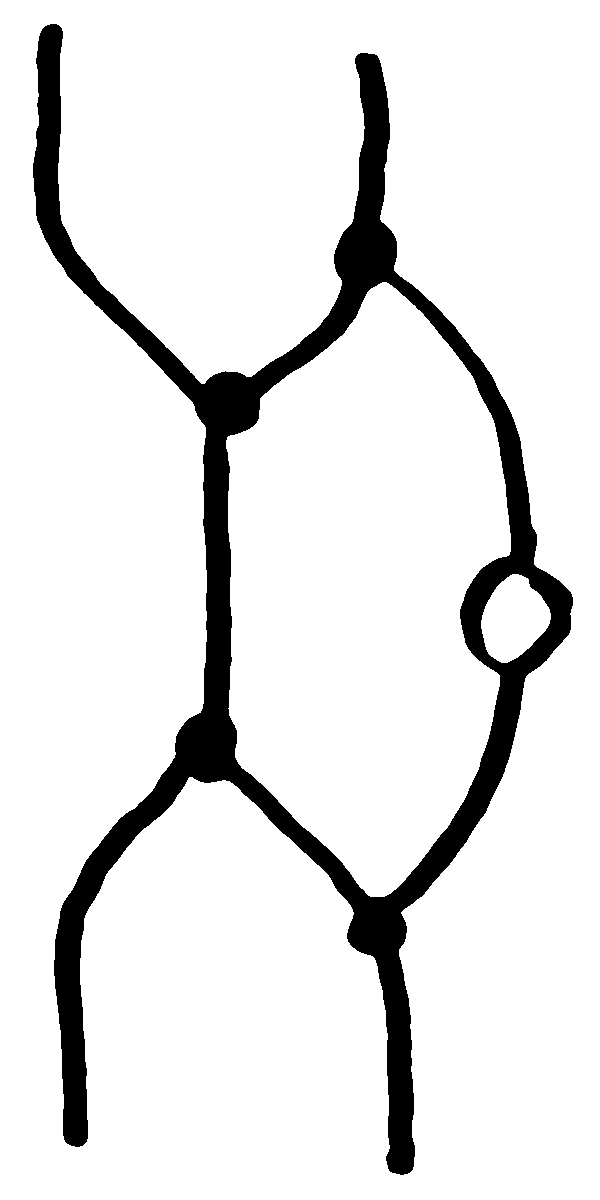}
\ee
Statement a) of the lemma follows from composing identity $\alpha$) with $\id \otimes \Lambda$ from the right and using the left integral property of $\Lambda$ to remove the extra product on the right hand side. Statement b) follows by composing identity $\beta$) with $\tilde\Lambda \otimes \id$.
\end{proof}

Recall that by condition b) in Theorem \ref{thm:main1}\,(i)
there exists $g : \one \to H$ with $(\id_H \otimes \lambda) \circ \Delta_H = g \circ \lambda$. Define $\Lambda'$ as in Lemma \ref{lem:integral-cointegral}. The same argument as in the proof of Proposition \ref{prop:pent-to-Hopf} shows that $g$ is unique. The `element' $g$ is again group-like, that is, we can repeat the statement of Lemma \ref{lem:g-grouplike} (but now starting from a different set of assumptions).

\begin{lemma} \label{lem:g-grouplike-2}
(i) $\Delta \circ g = g \otimes g$ and $\eps \circ g = \id_\one$. \\
(ii) The multiplicative inverse satisfies $g^{-1} = S \circ g = S^{-1} \circ g$.
\end{lemma}

\begin{proof}
Recall that the normalisation condition in property b) can be written as $\lambda \circ \Lambda' = \id_\one$.
To see the two equations in (i), apply $(\id \otimes \id \otimes \lambda) \circ (-) \circ \Lambda'$ to the coassociativity relation and evaluate $(\eps \otimes \lambda) \circ \Delta \circ \Lambda'$ in two ways. The proof of (ii) is identical to the proof in Lemma \ref{lem:g-grouplike}.
\end{proof}

From this lemma, it follows that the left and the right multiplication by $g$ are coalgebra maps. For example, in the case of left multiplication,
\be\label{eq:leftmult-g-coalg}
  \scanPIC{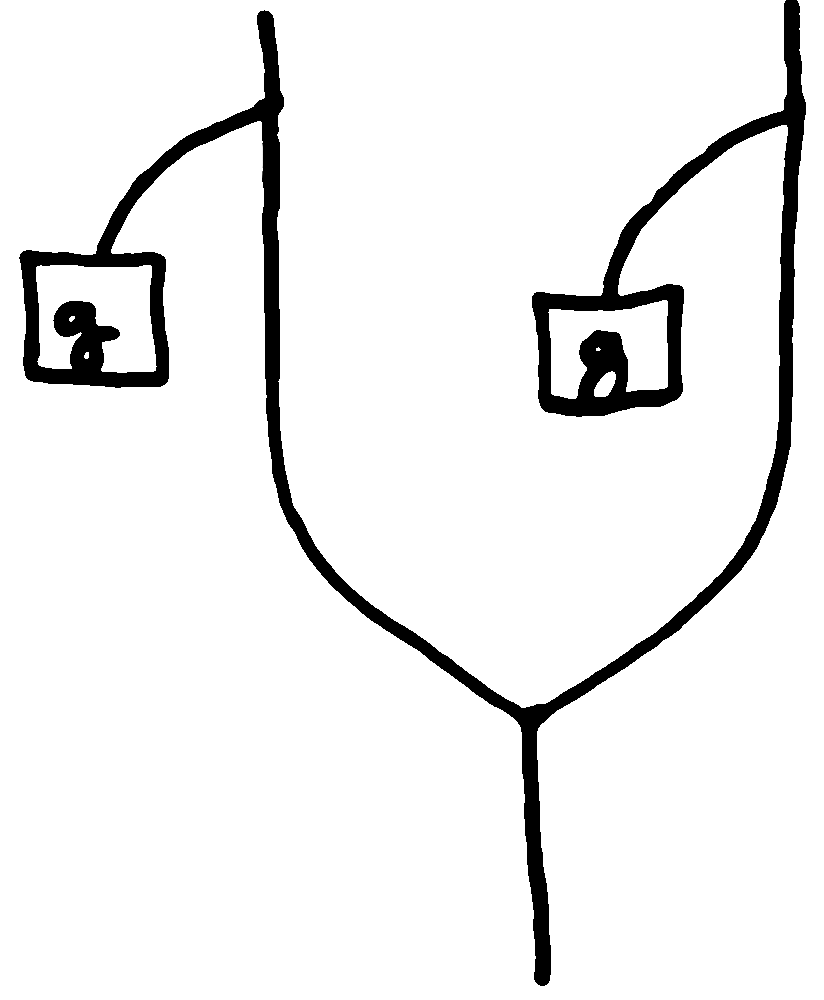}
  ~=~
  \scanPIC{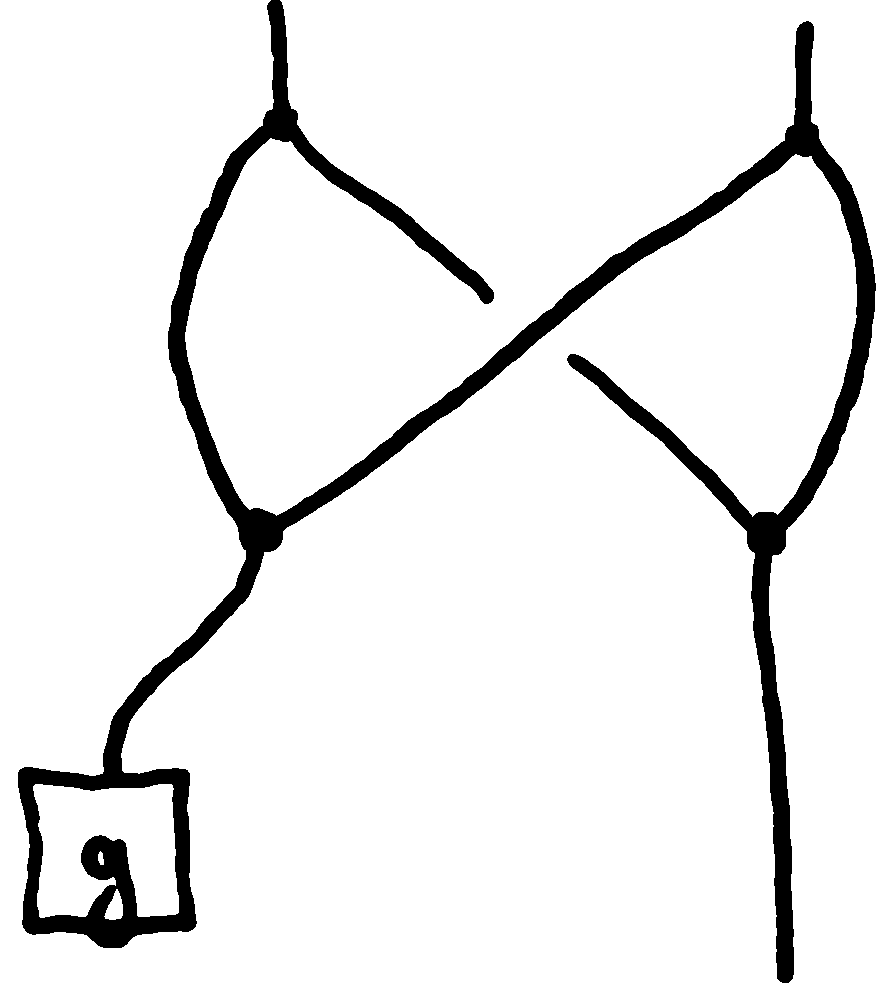}
  ~=~
  \scanPIC{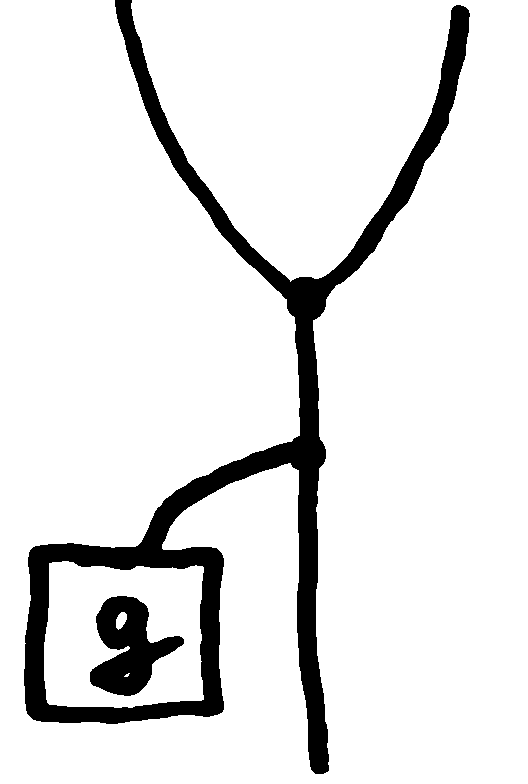}
  \quad .
\ee

We will need the following identities involving the cointegral $\lambda$.

\begin{lemma}\label{lem:move-coprod-past-lam}
We have
$$
\raisebox{3em}{\text{a)}}\quad
\scanPIC{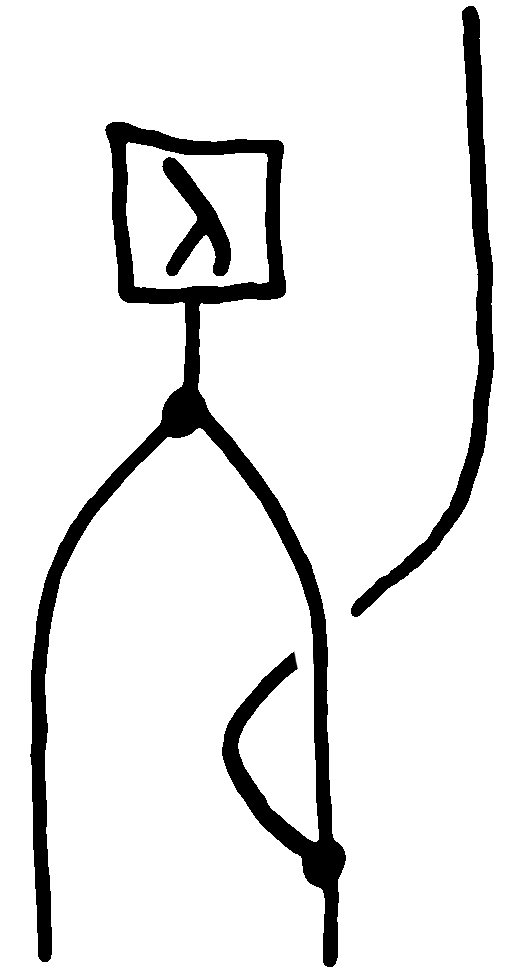}
~=~
\scanPIC{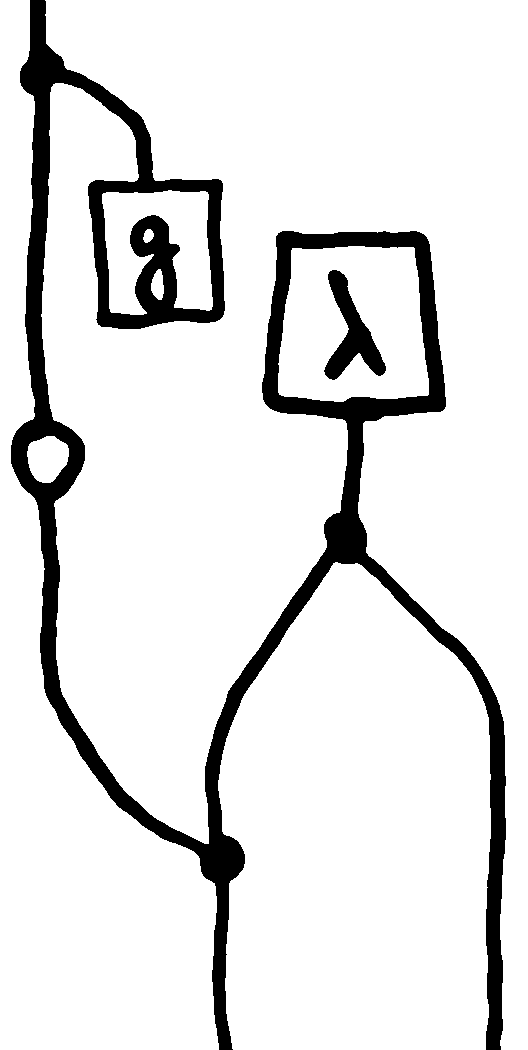}
\qquad , \qquad
\raisebox{3em}{\text{b)}}\quad
\scanPIC{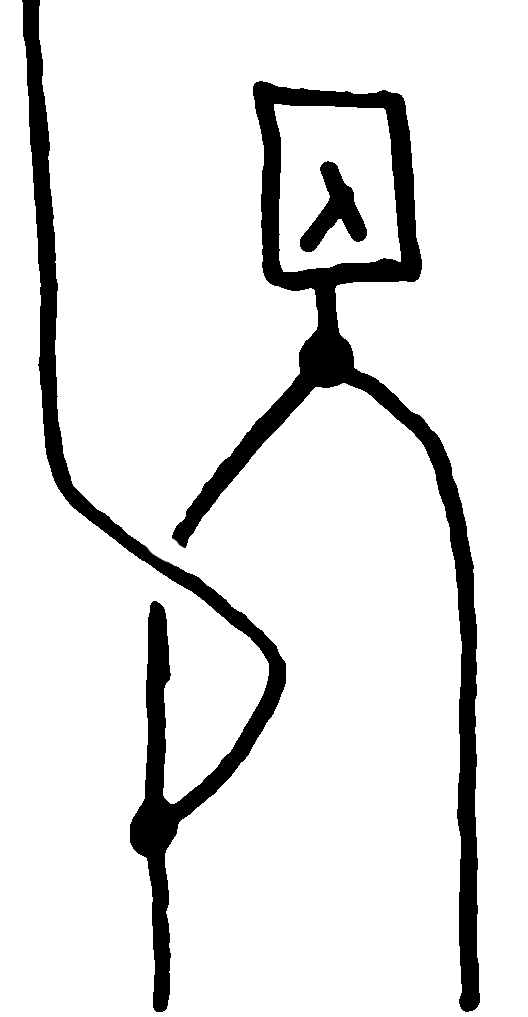}
~=~
\scanPIC{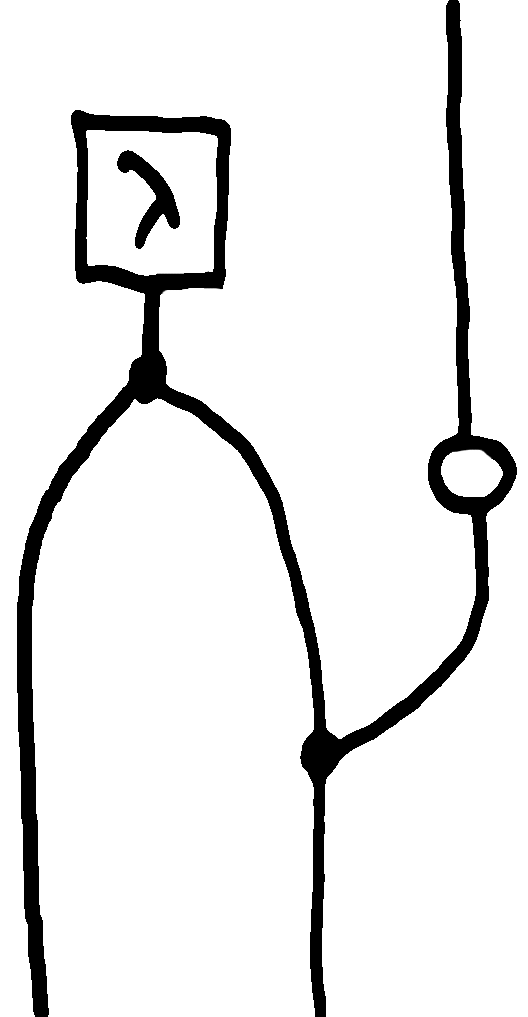}
$$
\end{lemma}

\begin{proof}
For a) note that
\be
\scanPIC{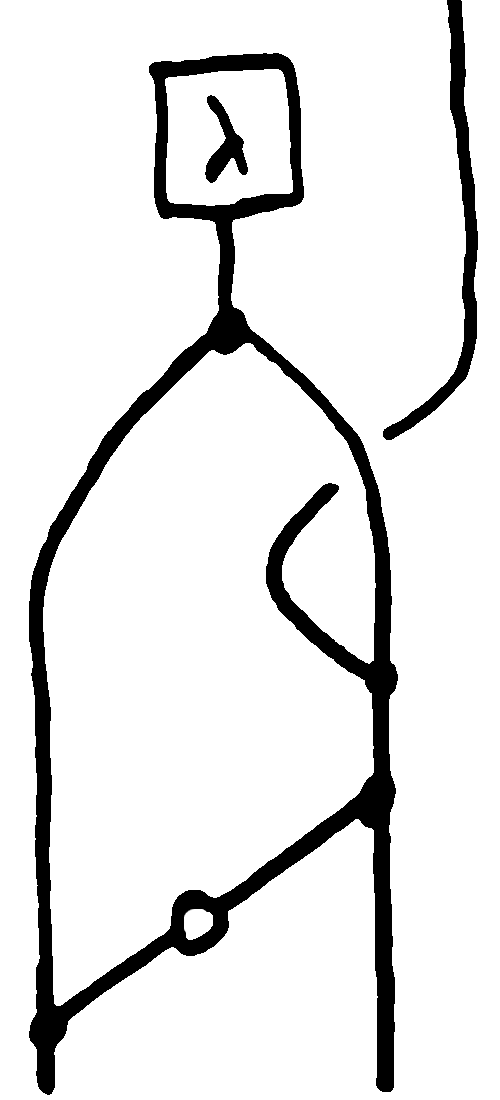}
~\overset{(1)}=~
\scanPIC{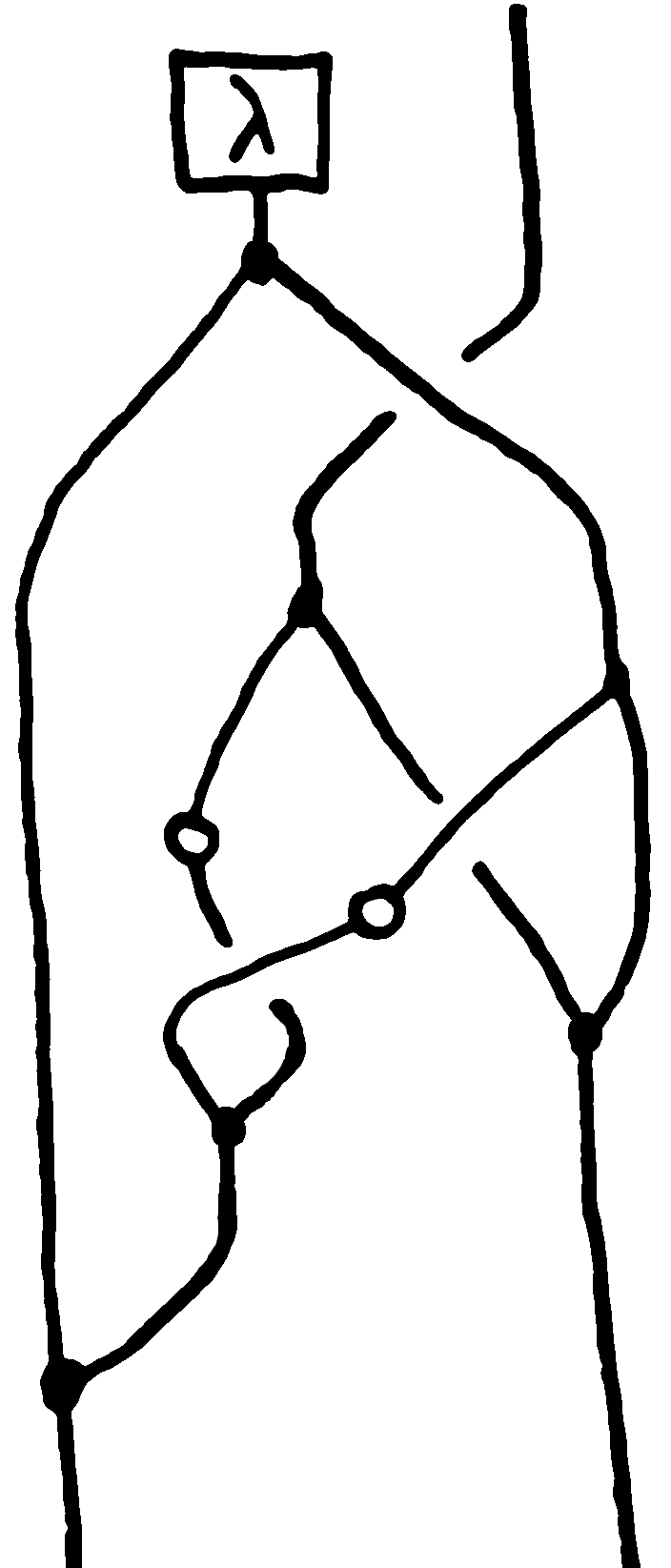}
~\overset{(2)}=~
\scanPIC{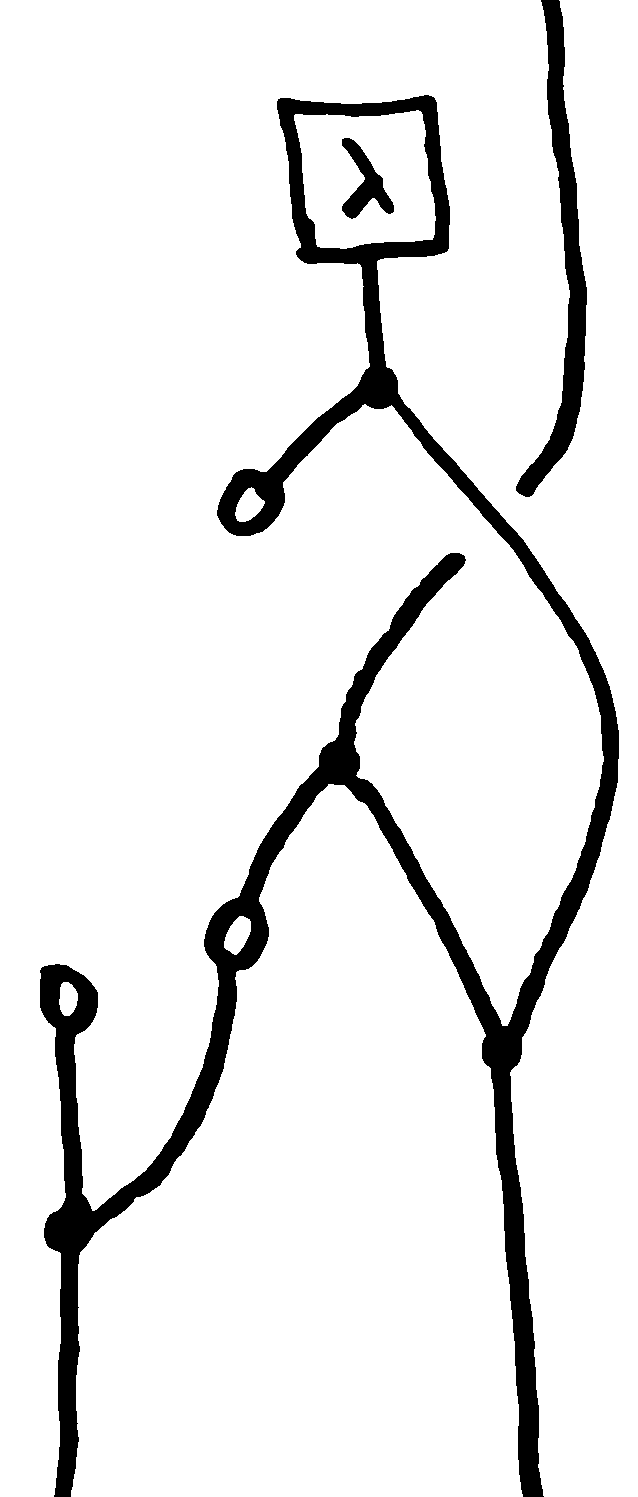} 
~\overset{(3)}=~
\scanPIC{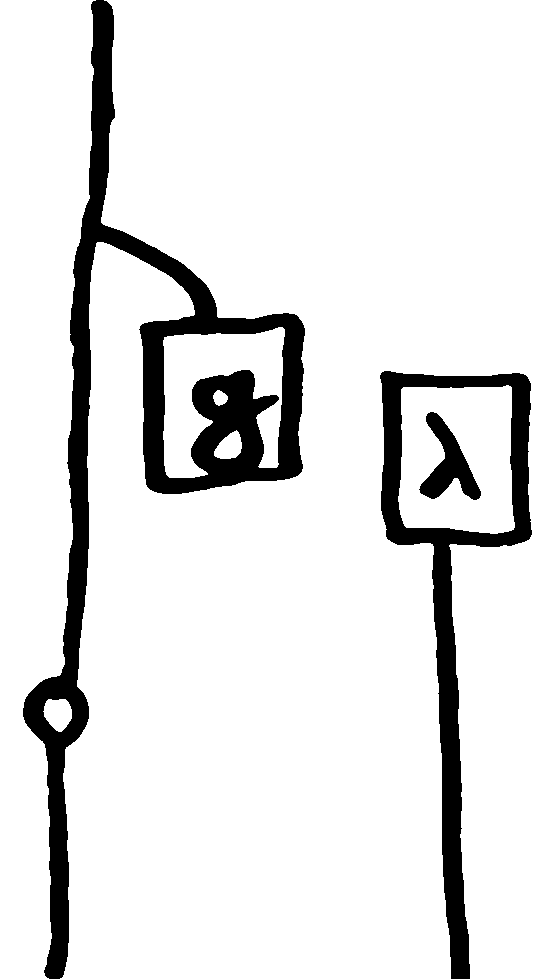} 
\quad ,
\ee
where in step 1 the algebra-map property of $\Delta$ was used and $S$ was moved past a coproduct, and step 2 follows from the bubble-property. Step 3 employs property b) in Theorem \ref{thm:main1}\,(i). Composing the above equation with $(\id \otimes \mu) \circ (\Delta \otimes \id)$ gives the first equality in the statement of the lemma (recall Lemma \ref{lem:H-intertwiner}\,a)). The analogous argument for b) is
\be
\scanPIC{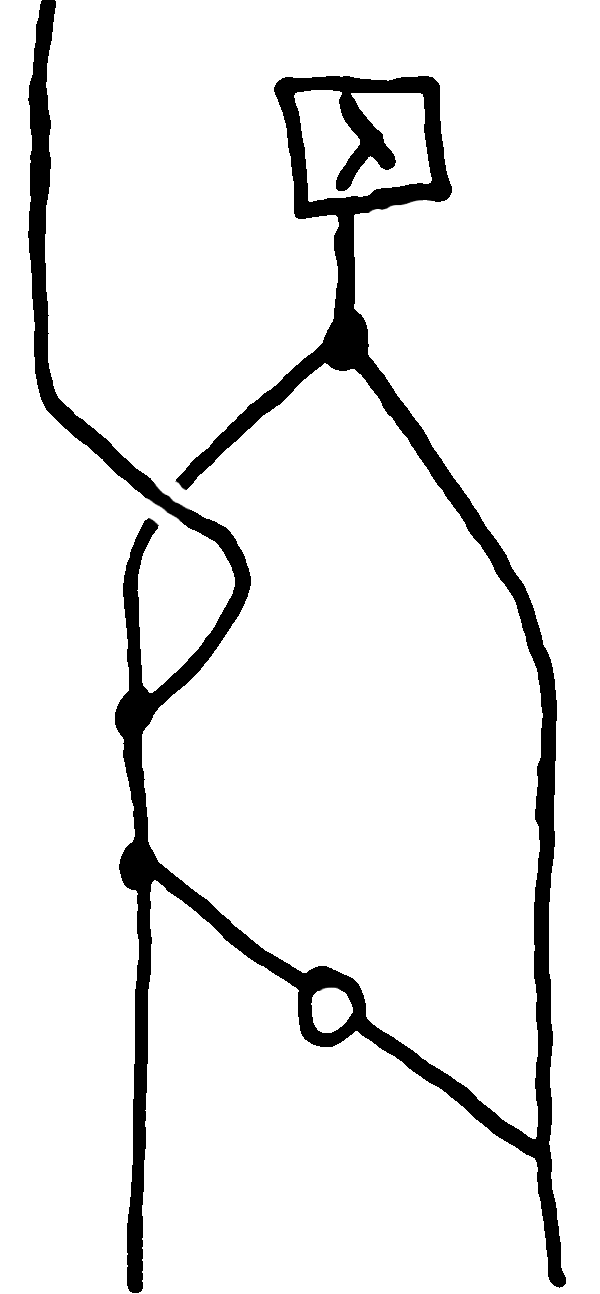}
~=~
\scanPIC{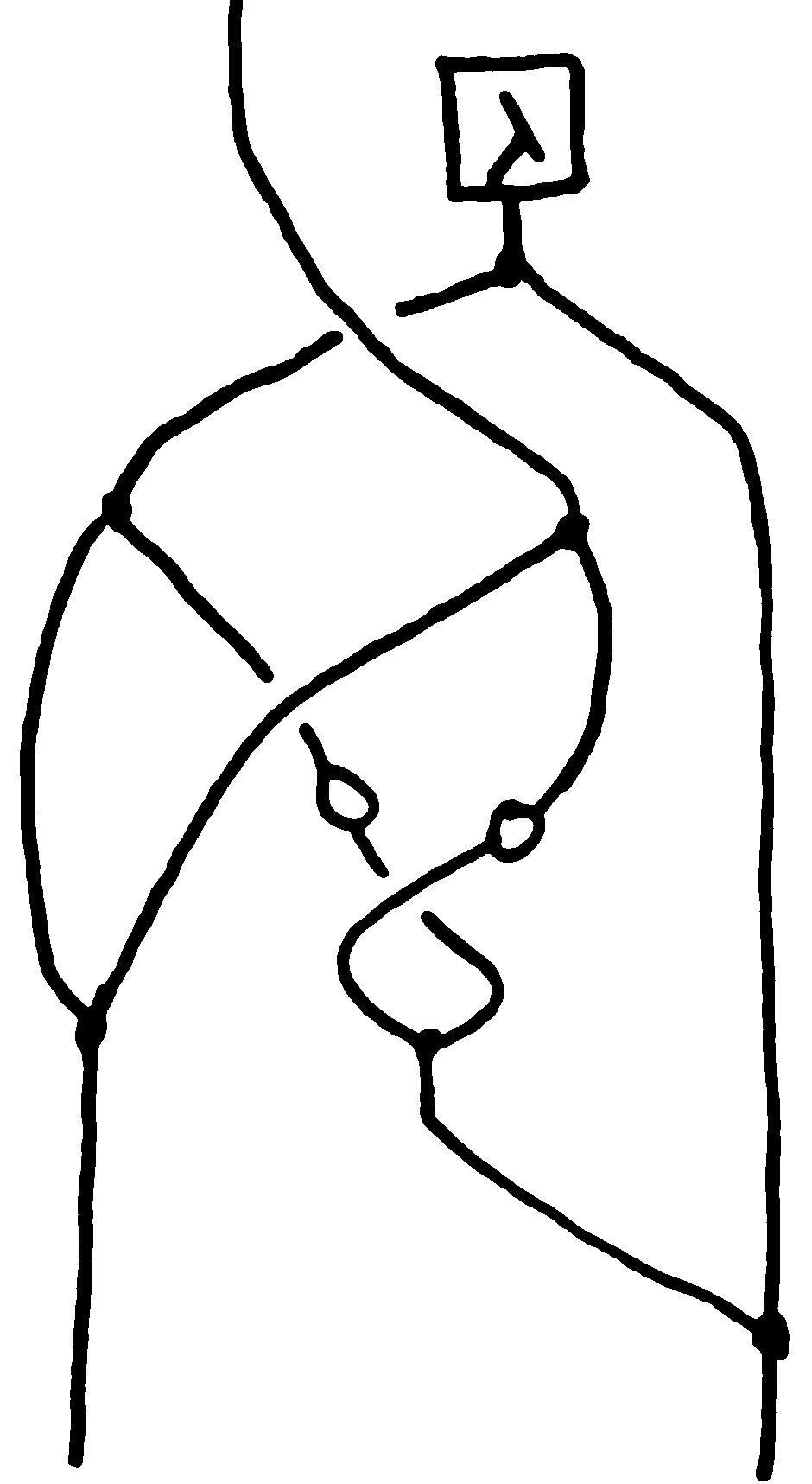}
~=~
\scanPIC{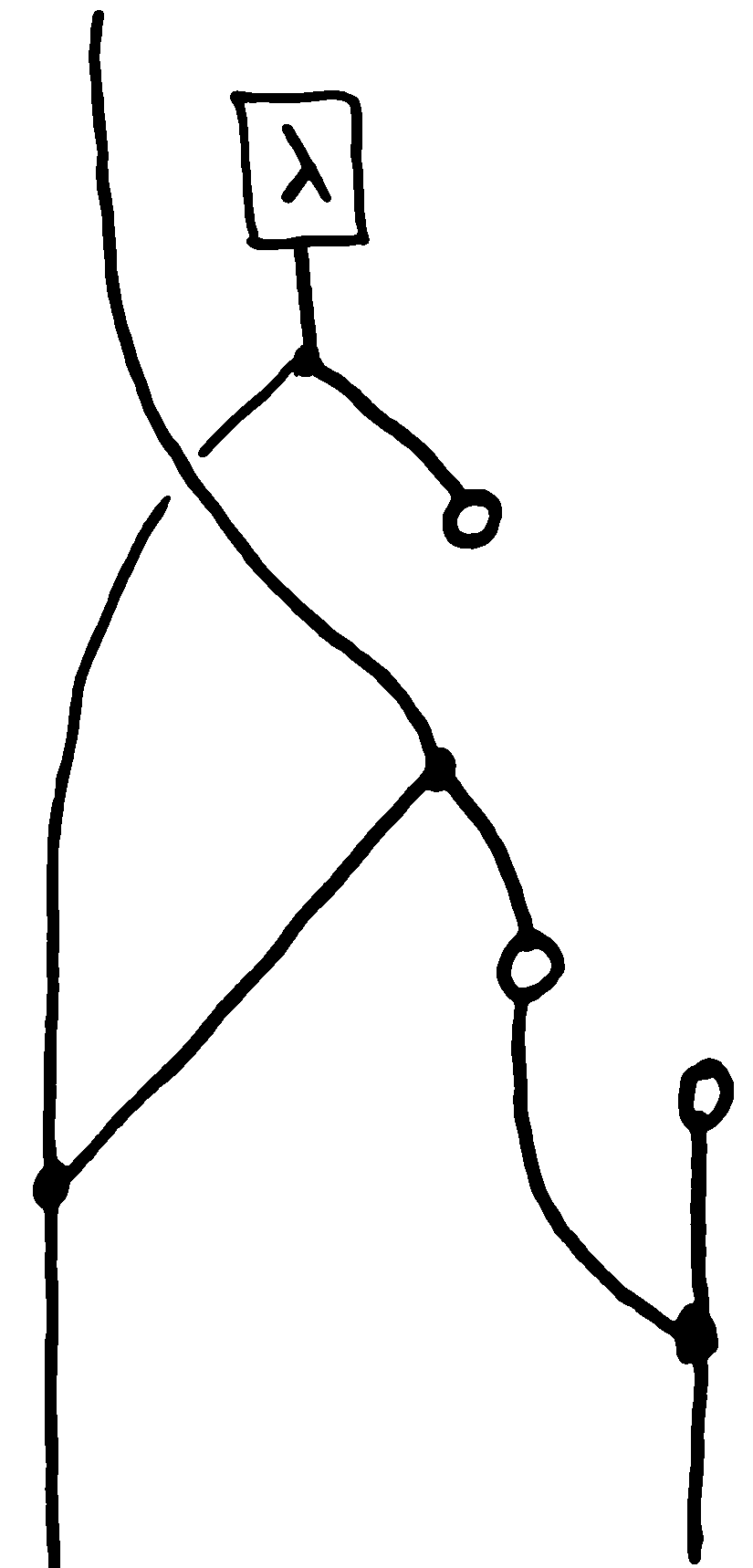}
~=~
\scanPIC{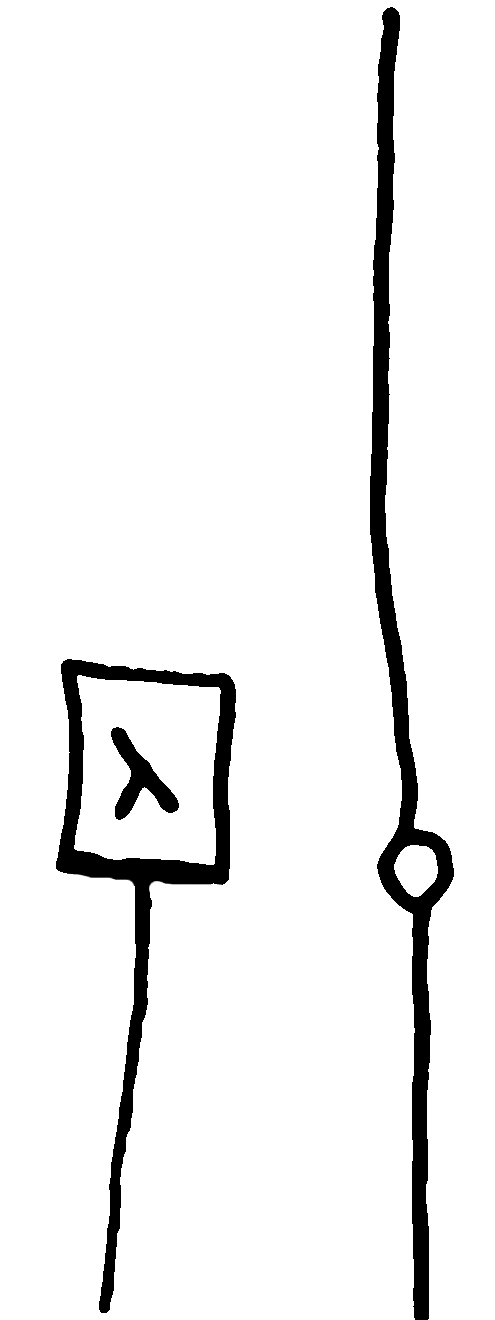}
\quad .
\ee
\end{proof}

Let $\Lambda$ be given by $(\lambda \otimes \Gamma) \circ \coev_H$ as in Lemma \ref{lem:integral-cointegral}. 

\begin{lemma}\label{lem:Lam-copair-nondeg}
The copairing $\Delta \circ \Lambda : \one \to H \otimes H$ is non-degenerate.
\end{lemma}

\begin{proof}
We will show that $b = \Delta \circ \Lambda$ is the coevaluation of a duality with evaluation given by $d = \lambda \circ \mu \circ (S \ot \id)$. One of the duality conditions follows from
\be
  \scanPIC{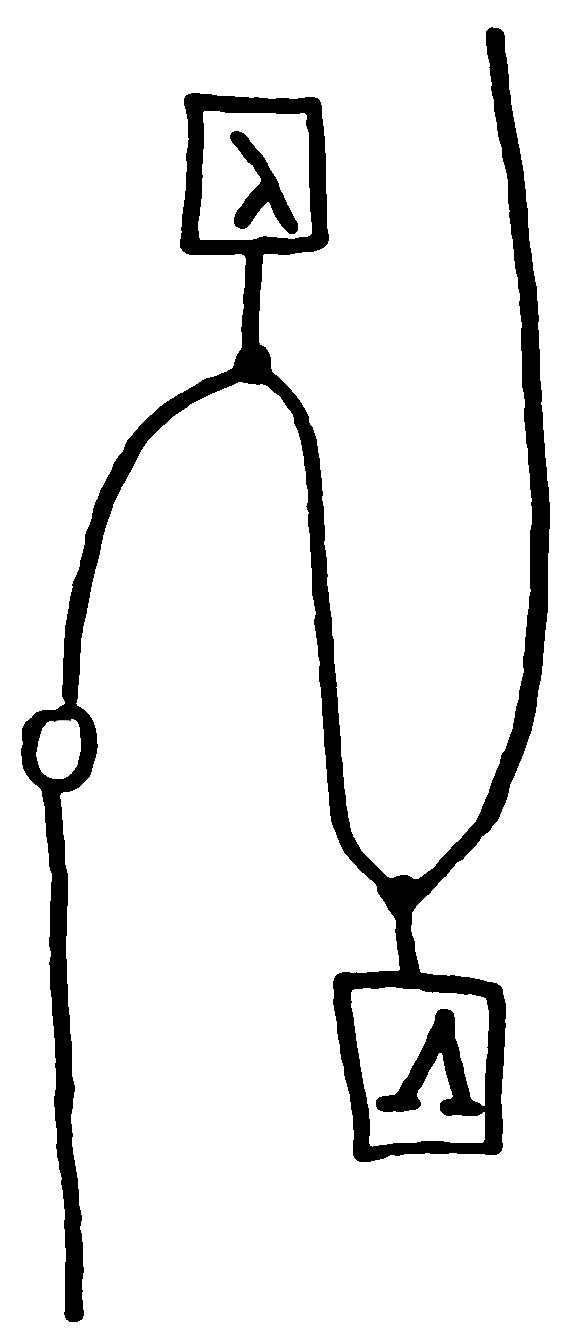}
  ~\overset{ \text{Lem.\,\ref{lem:move-prod-past-integral}}}{=} ~
  \scanPIC{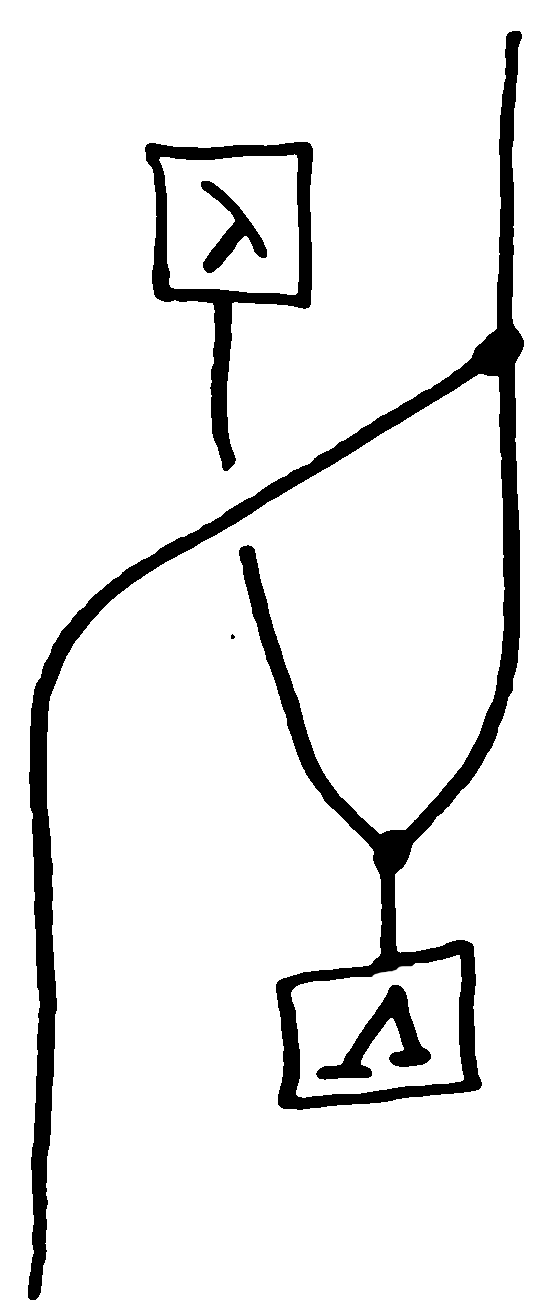}
  ~\overset{ \text{$\lambda$ coint.}}{=}~ 
  \scanPIC{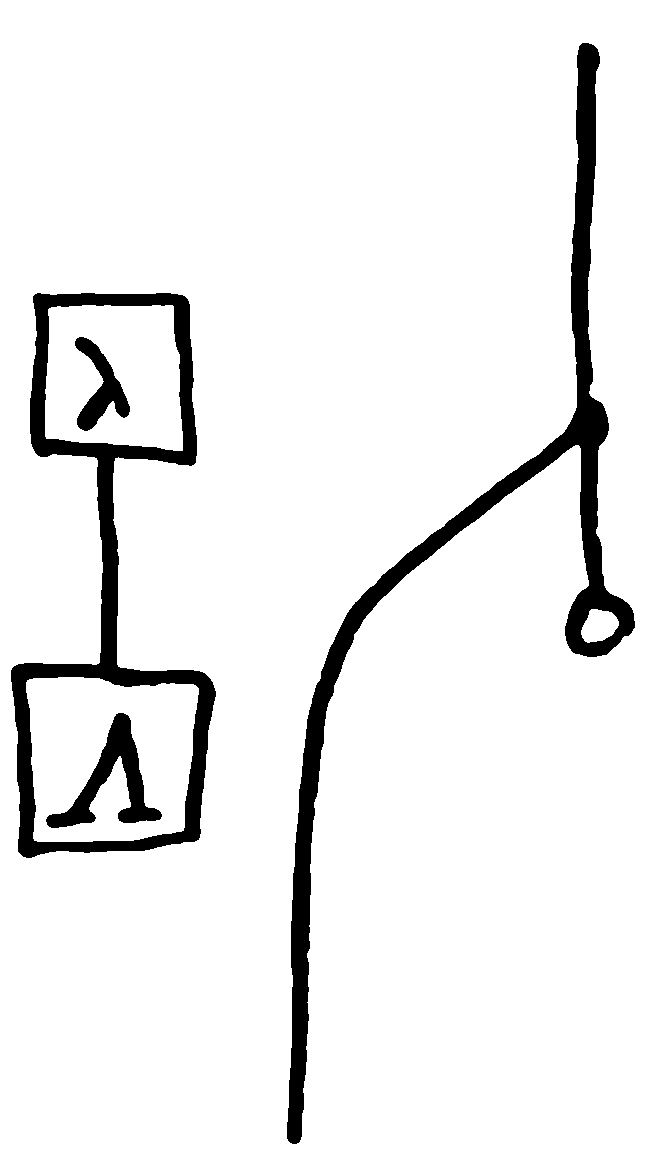}
  ~\overset{(*)}{=}~ \id_H \ ,
\ee
where (*) amounts to the normalisation condition in property b) of Theorem \ref{thm:main1}\,(i). Inserting this identity in $\lambda \circ (-) \circ \Lambda$ and using that $\Lambda$ is a left integral results in $(\lambda \circ S \circ \Lambda) \cdot (\lambda \circ \Lambda) = \lambda \circ \Lambda$. Together with the normalisation condition one learns that $\lambda \circ S \circ \Lambda = \id_\one$. This latter identity is needed when verifying the second duality condition. Namely, after a calculation analogous to the one above (using again Lemma \ref{lem:move-prod-past-integral}\,a) one arrives at $(\id \otimes d) \circ (b \otimes \id) = (\lambda \circ S \circ \Lambda) \cdot \id$.
\end{proof}

\begin{lemma}\label{lem:right-int-from-left-int}
Let $\alpha : H \to \one$ be given by $\alpha = \ev_H \circ (\Gamma^{-1} \otimes g)$. Then
\begin{itemize}
\item[{(i)}] $\alpha \circ \mu = \alpha \otimes \alpha$.
\item[{(ii)}] $\tilde\Lambda = (\id \otimes \alpha )\circ \Delta \circ \Lambda$ is a right integral.
\end{itemize}
\end{lemma}

\begin{proof}
We will need the identity 
\be \label{eq:right-int-from-left-int-aux1}
  \mu \circ (\Lambda\otimes\id) = \Lambda \circ \alpha \ ,
\ee
which follows from
\be
  \scanpic{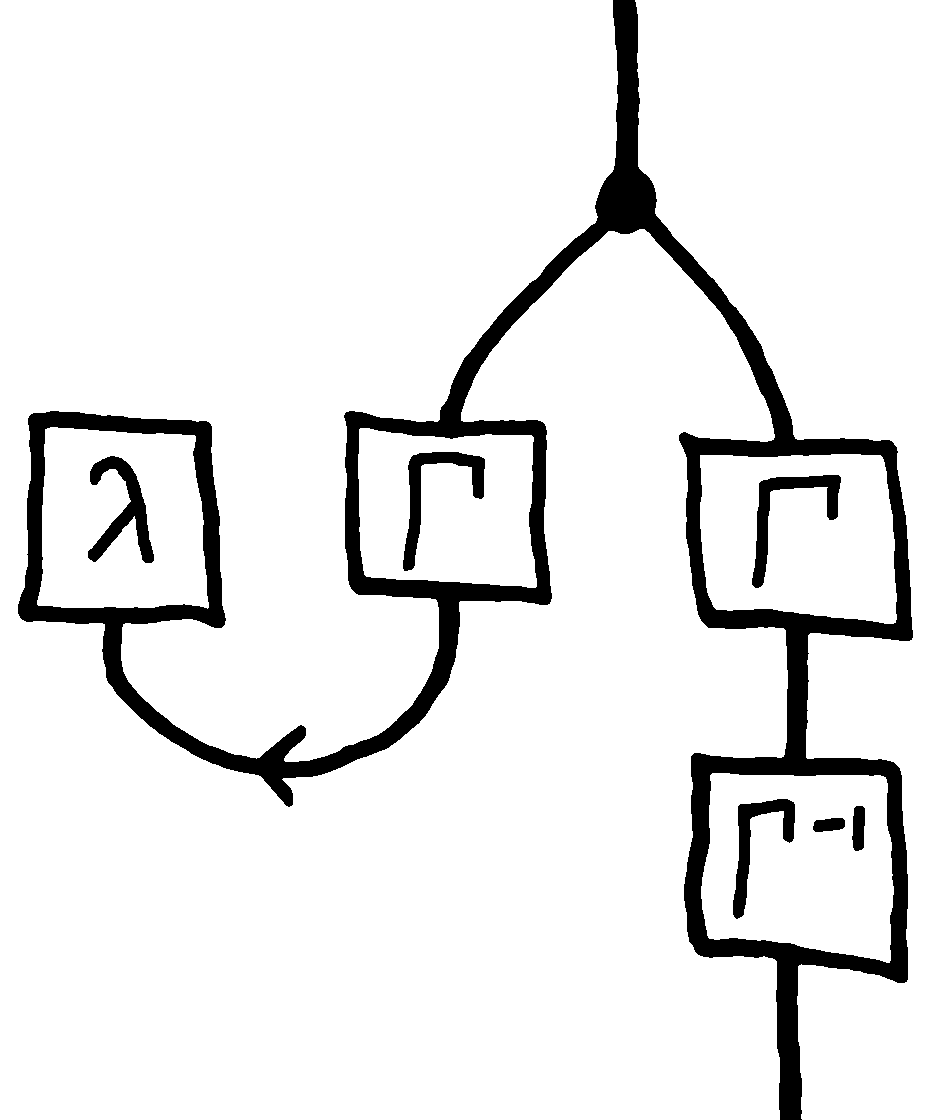}
  ~\overset{\text{Thm.\,\ref{thm:main1}\,(i)\,a)}}{=} ~
  \scanpic{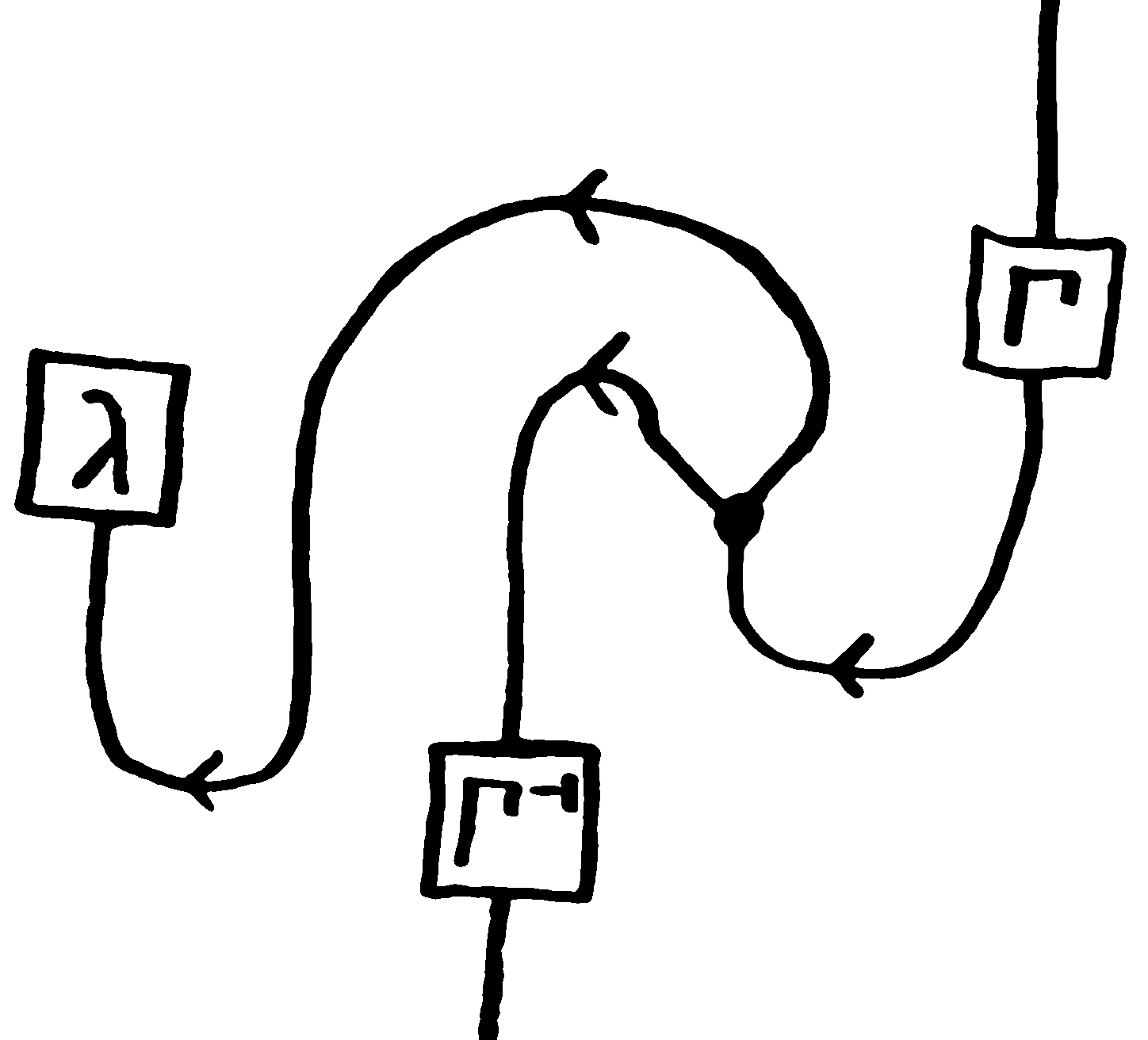}
  ~\overset{\text{Thm.\,\ref{thm:main1}\,(i)\,b)}}{=} ~
  \scanpic{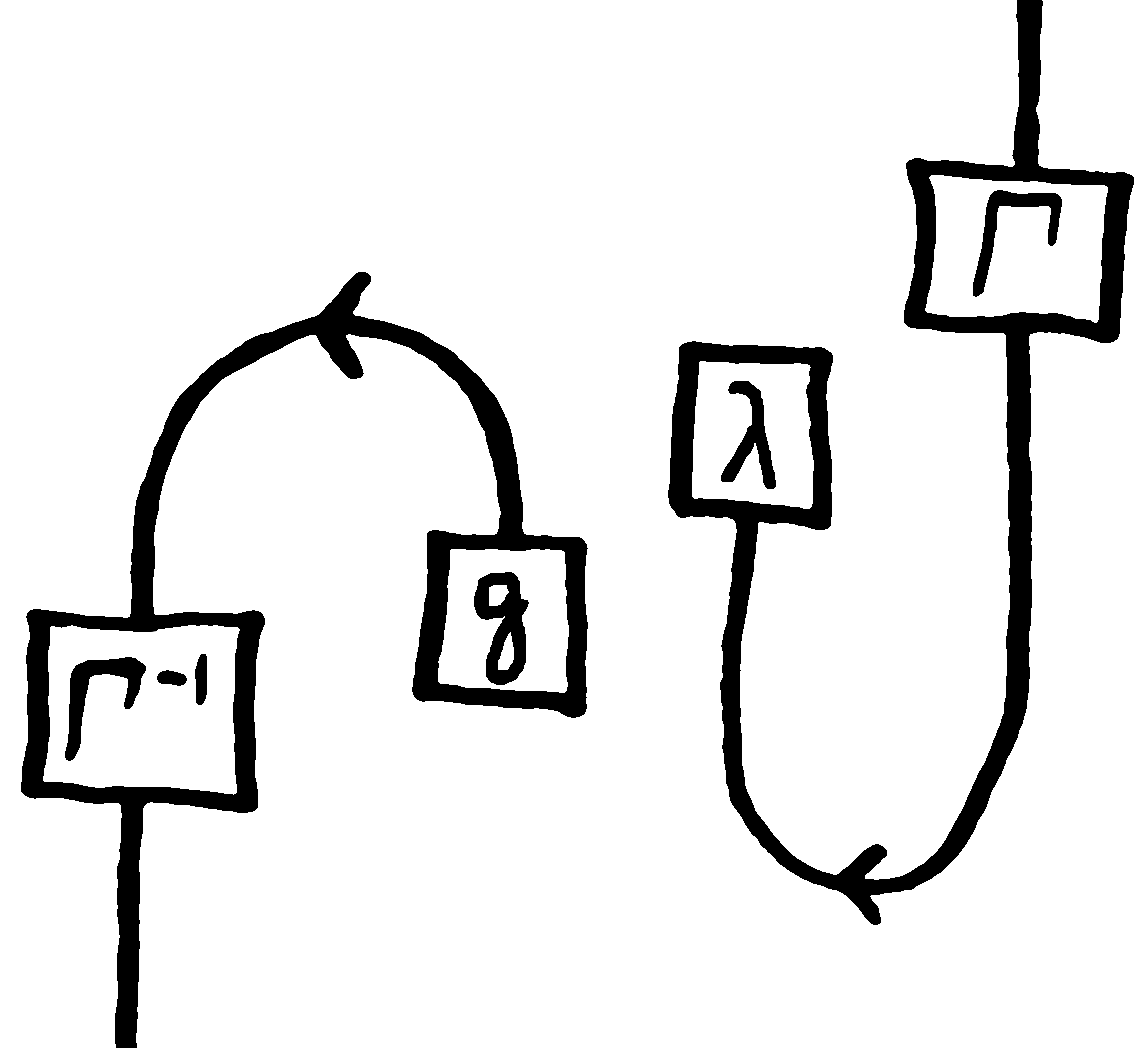}
  \quad .
\ee
For part (i) note that
\be
  \Lambda \otimes \alpha \otimes \alpha
  \overset{(1)}= 
  \mu \circ (\mu \otimes \id) \circ ( \Lambda \otimes \id \otimes \id)
  =
  \mu \circ ( \Lambda \otimes \mu )
  \overset{(3)}= 
  \Lambda \otimes (\alpha \circ \mu ) \ .
\ee
where step 1 follows from using \eqref{eq:right-int-from-left-int-aux1} twice, and step 3 from using \eqref{eq:right-int-from-left-int-aux1} once more. Composing this identity with $\lambda$ from the left and using the normalisation condition $\lambda \circ \Lambda = \id_\one$ from Theorem \ref{thm:main1}\,(i)\,b) proves part (i). 

For part (ii) we will first show that $\beta = \alpha \circ S : H \to \one$ is a comultiplicative inverse to $\alpha$, i.e.\ $(\alpha \otimes \beta) \circ \Delta = \eps$. Namely, insert \eqref{eq:right-int-from-left-int-aux1} into $\lambda \circ (-) \circ \eta$ to get $\lambda \circ \Lambda = (\lambda \circ \Lambda) \cdot (\alpha \circ \eta)$. Since $\lambda \circ \Lambda = \id_\one$, we get $\alpha \circ \eta = \id_\one$. Composing the equality in (i) with $(\id \otimes S) \circ \Delta$ and using the bubble-property of $S$ gives $(\alpha \otimes \beta) \circ \Delta = \eps$.
Part (ii) now follows from the equalities
\be
\scanpic{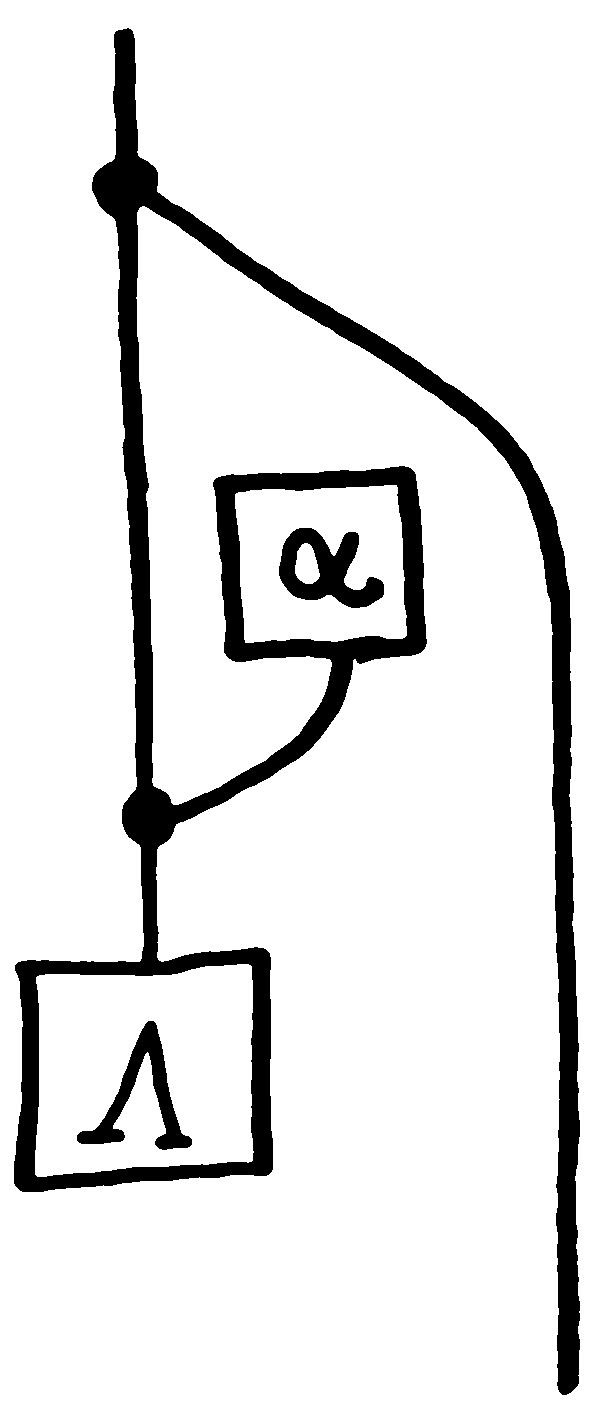}
~~=~
\scanpic{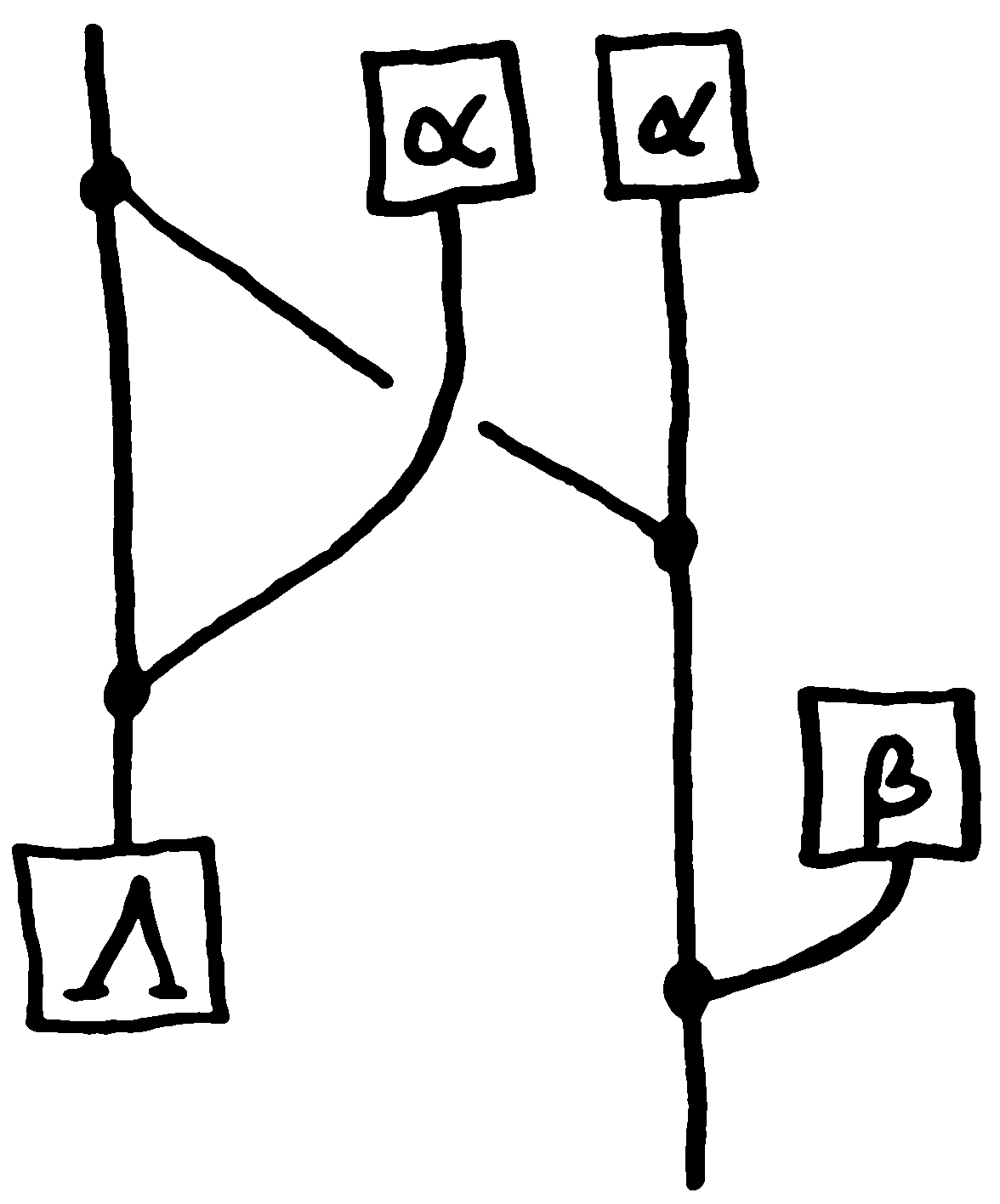}
~=~~
\scanpic{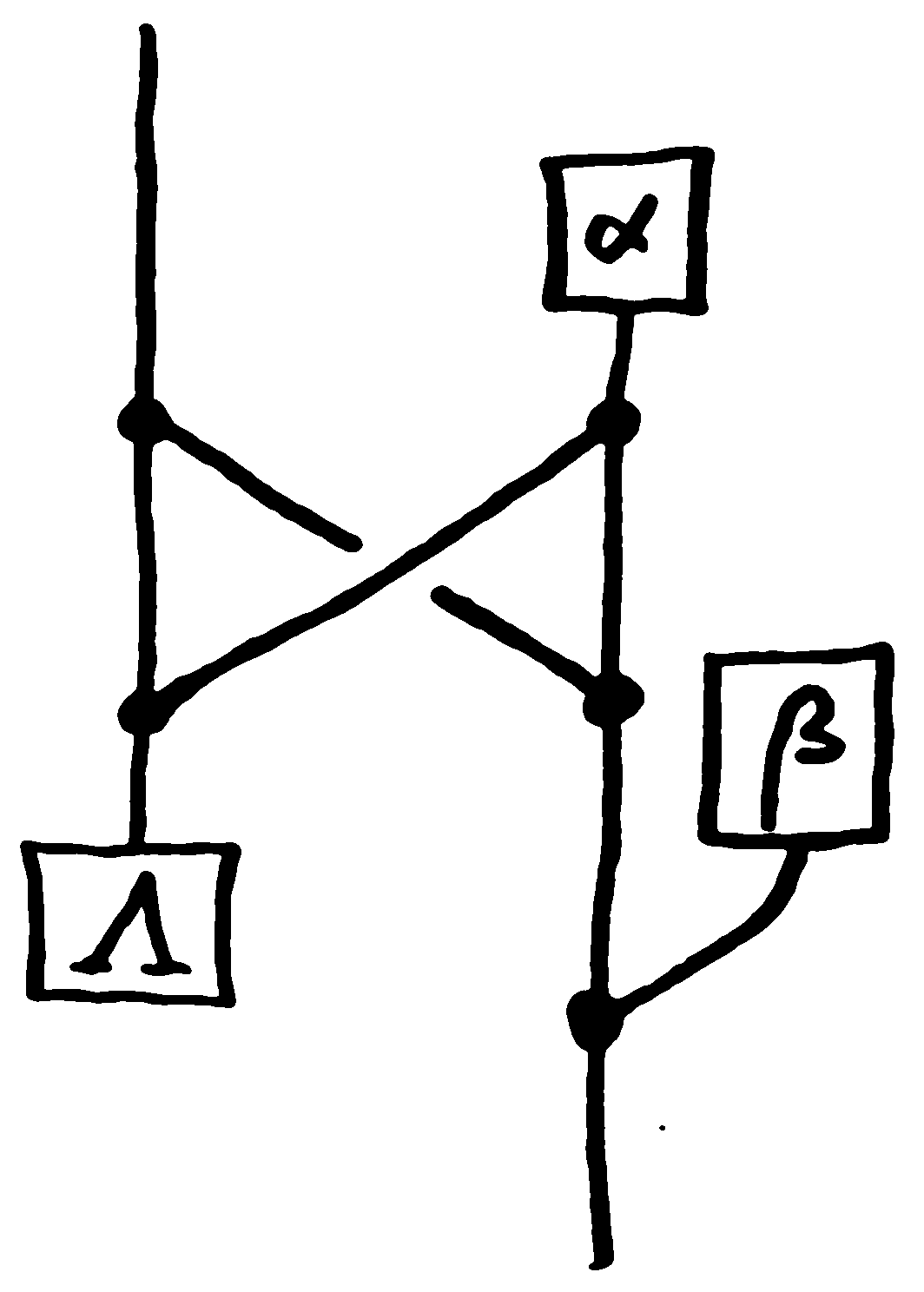}
~~=~
\scanpic{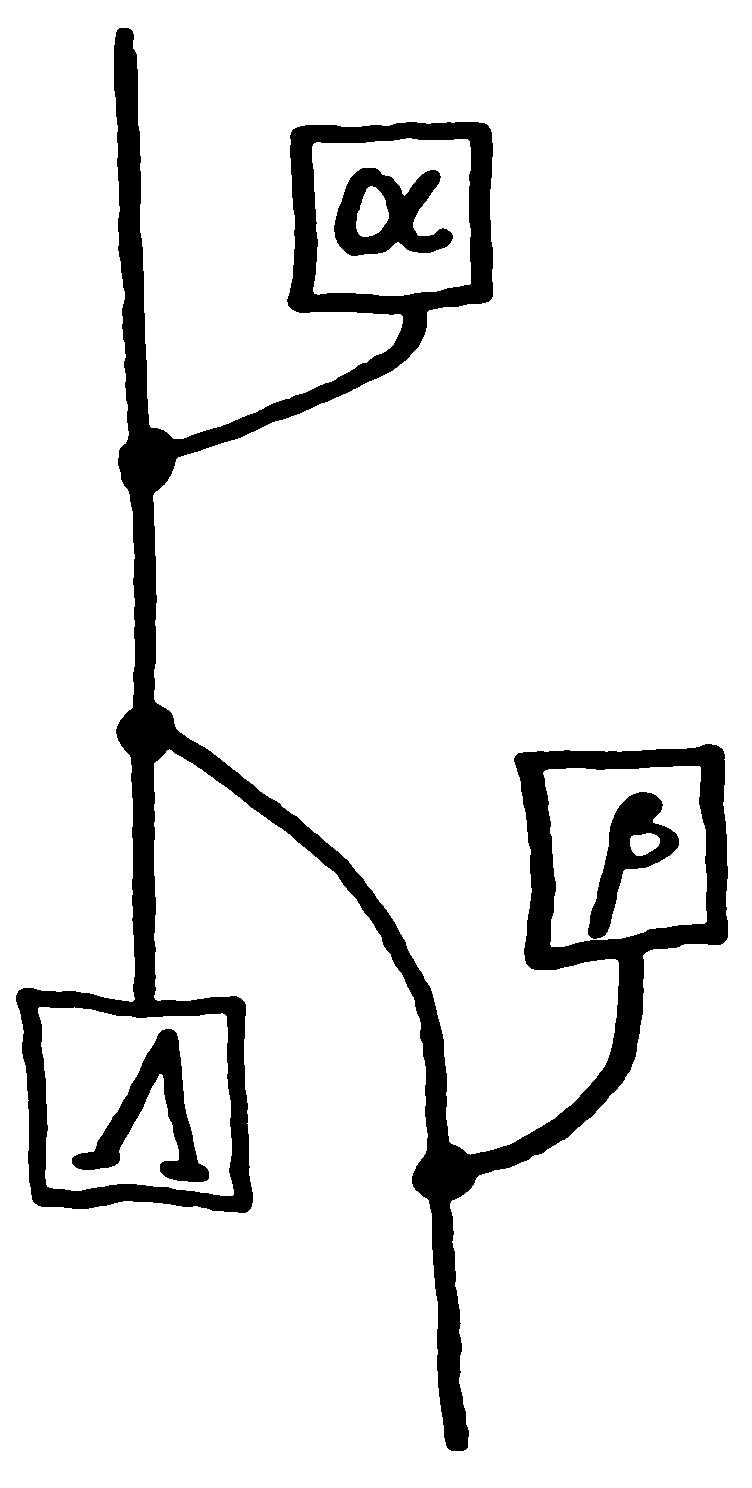}
\quad ,
\ee
together with \eqref{eq:right-int-from-left-int-aux1}.
\end{proof}

\begin{lemma}\label{lem:flip-lambda}
We have
$$
  \scanPIC{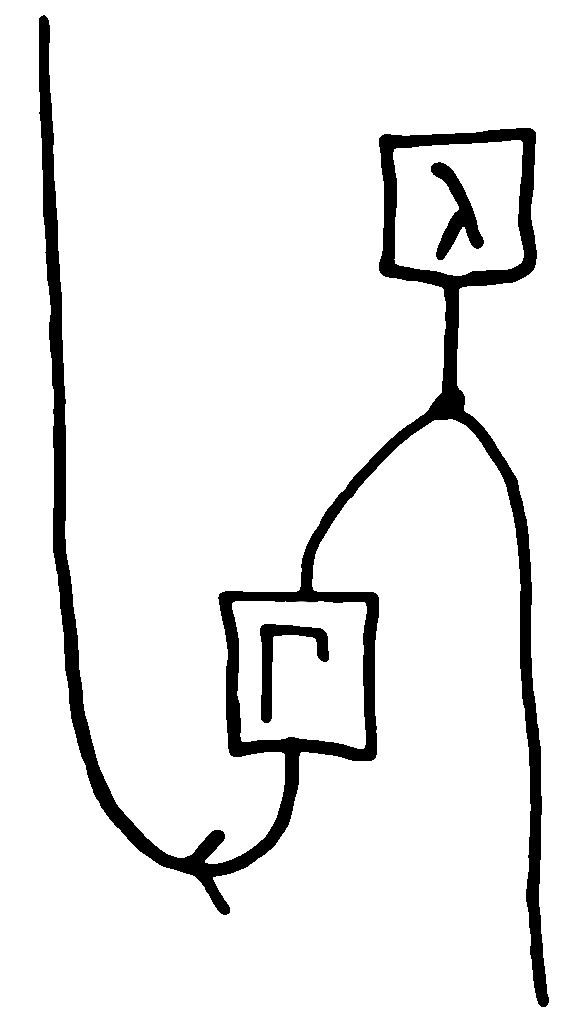}
  ~=~
  \scanPIC{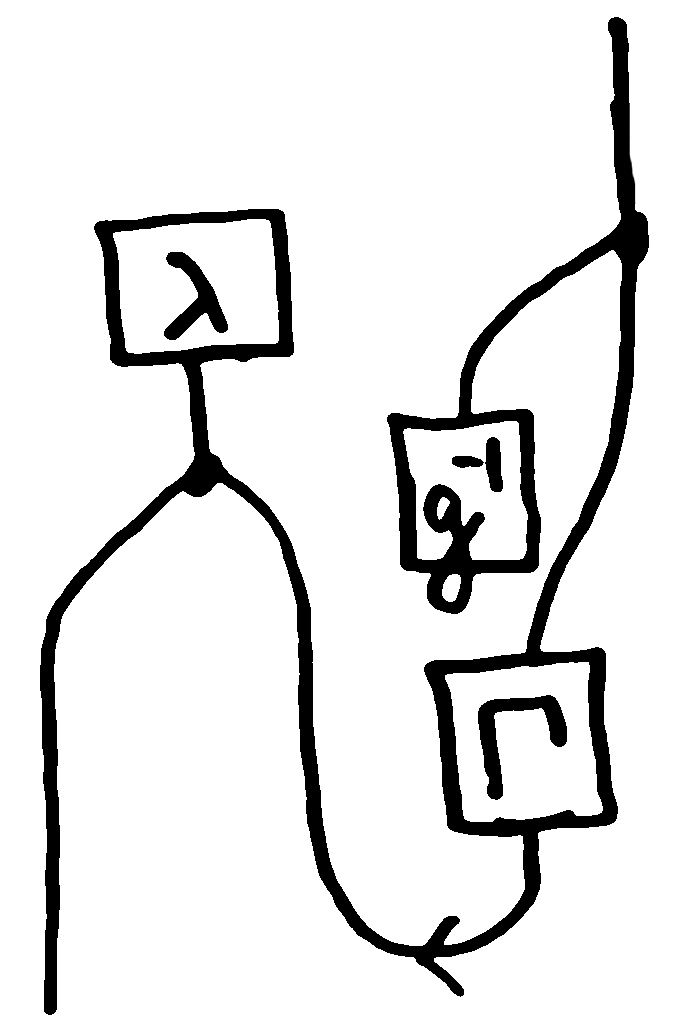}
  \quad .
$$
\end{lemma}

\begin{proof}
We will first establish equality of two auxiliary expressions and then conclude the statement of the lemma. The first expression is
\begin{align}
  &\scanPIC{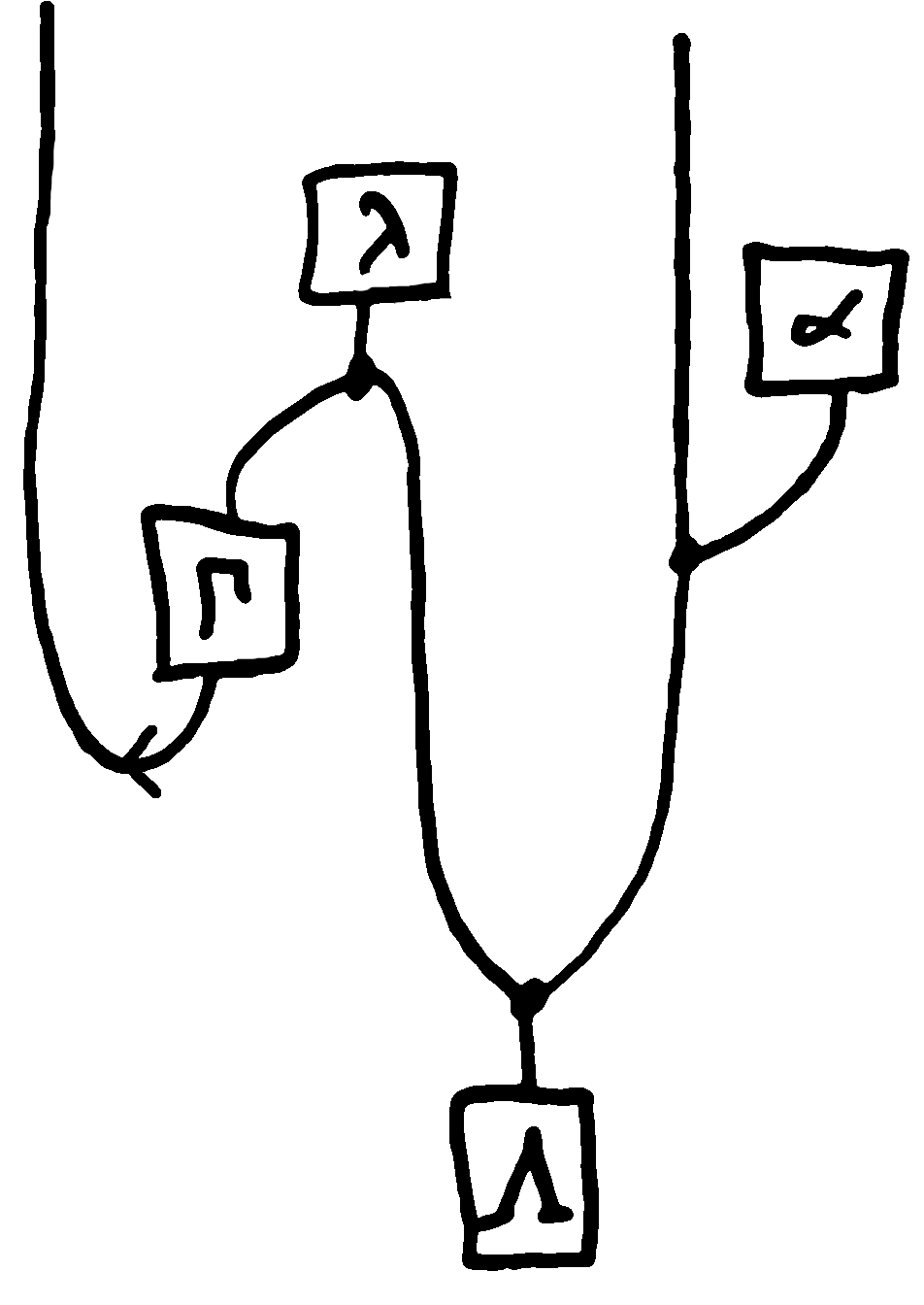}
  \overset{ \text{Lem.\,\ref{lem:move-prod-past-integral}}}{=} 
  \scanPIC{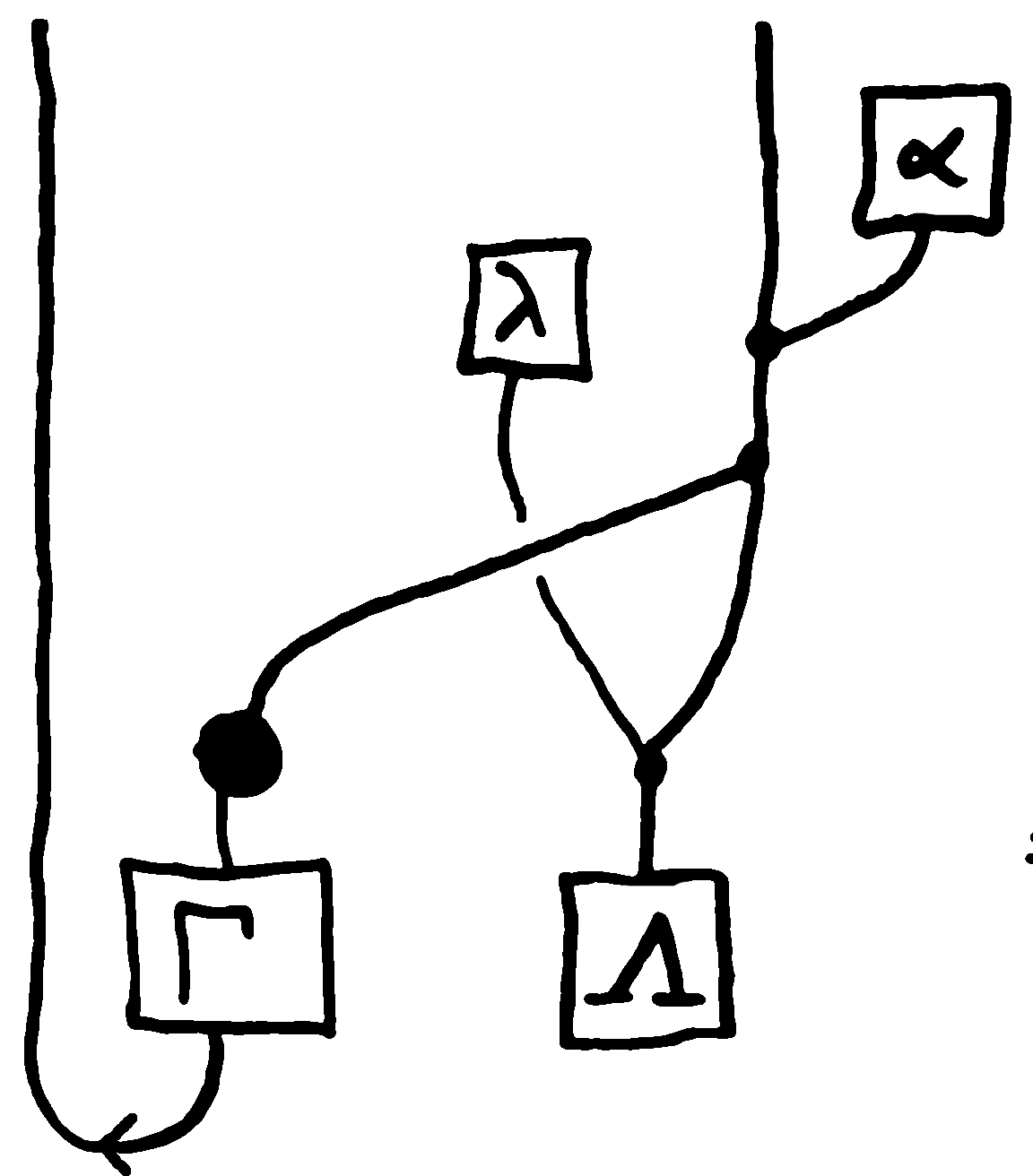}
  \overset{ \text{$\lambda$ coint.}}{=} 
  \scanPIC{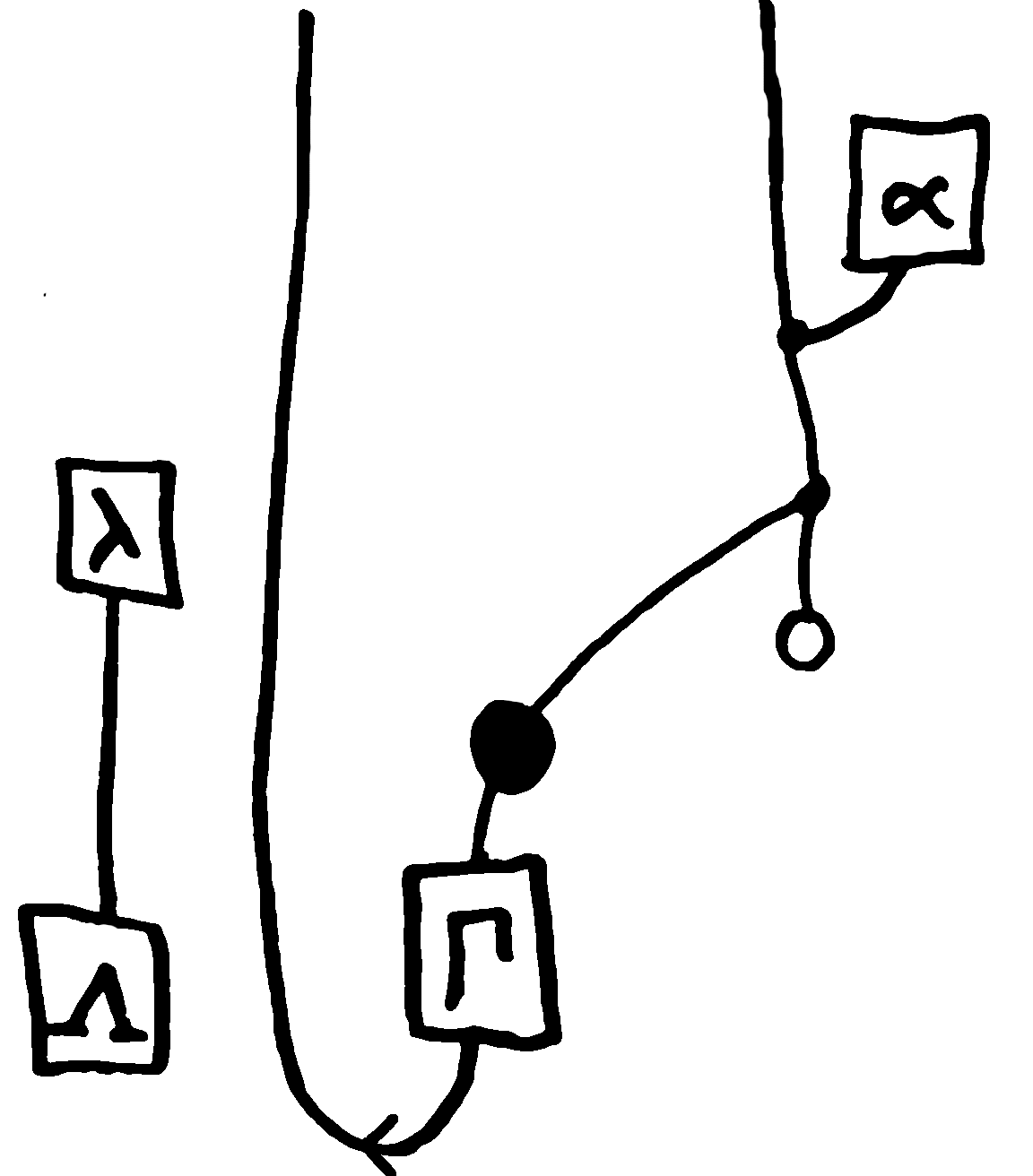}
\nonumber \\
  &\overset{(1)}{=} 
  \scanPIC{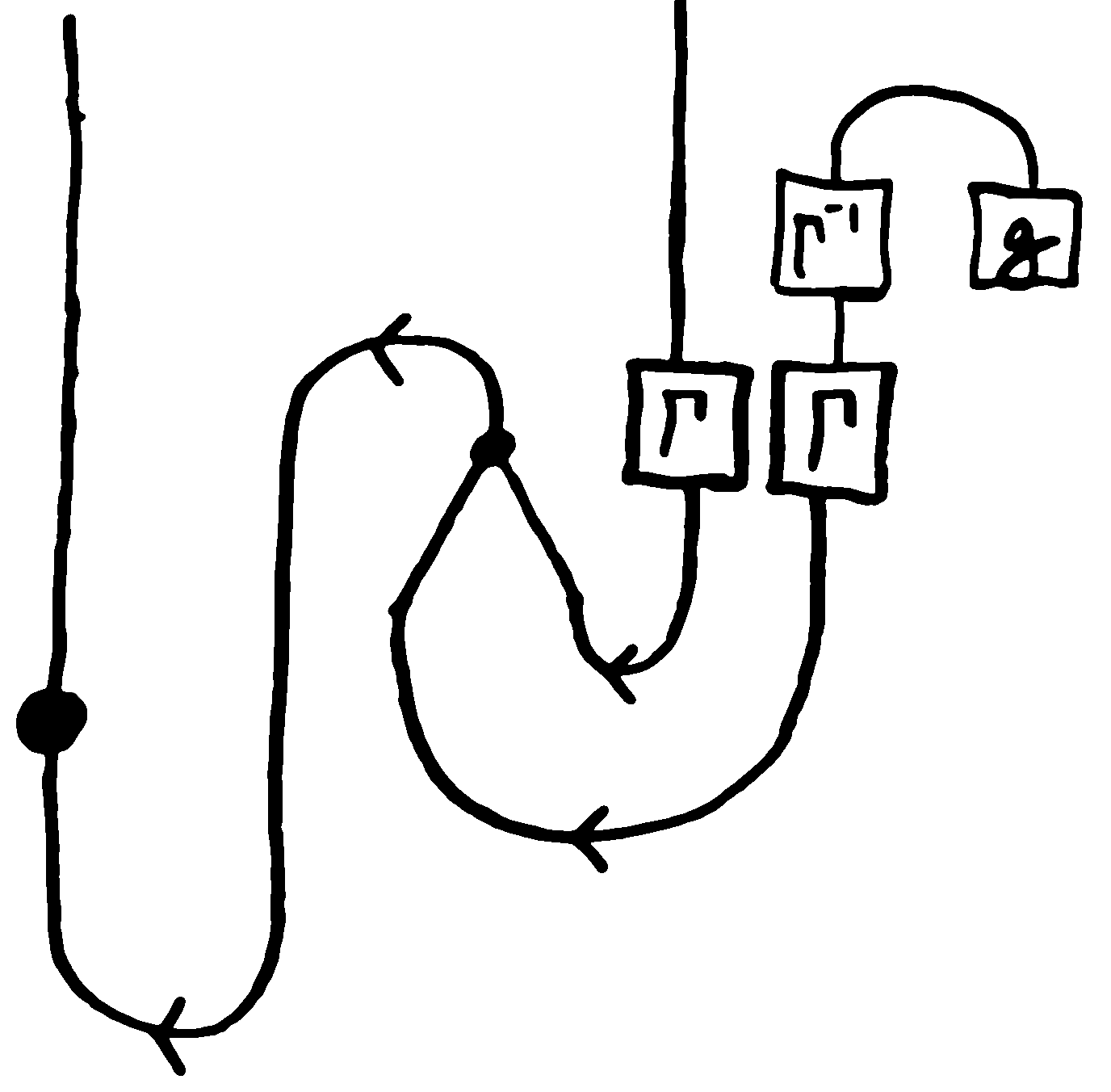}
  = 
  \scanPIC{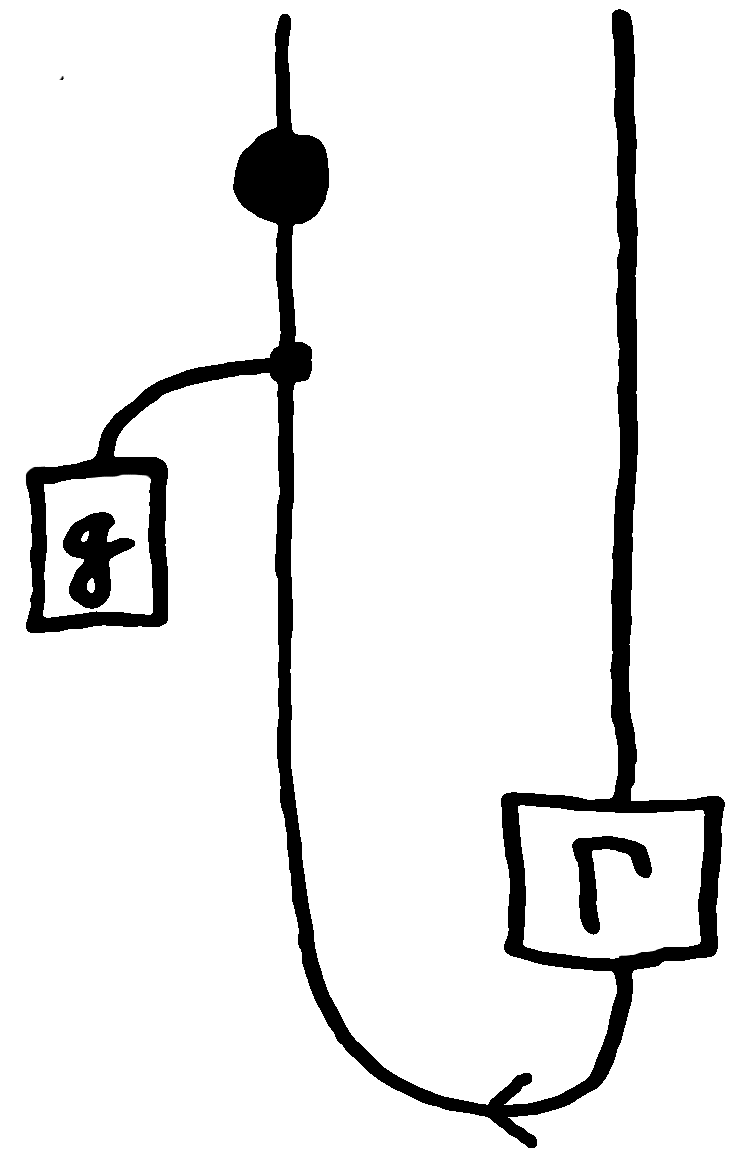}
  \overset{(2)}{=} 
  \scanPIC{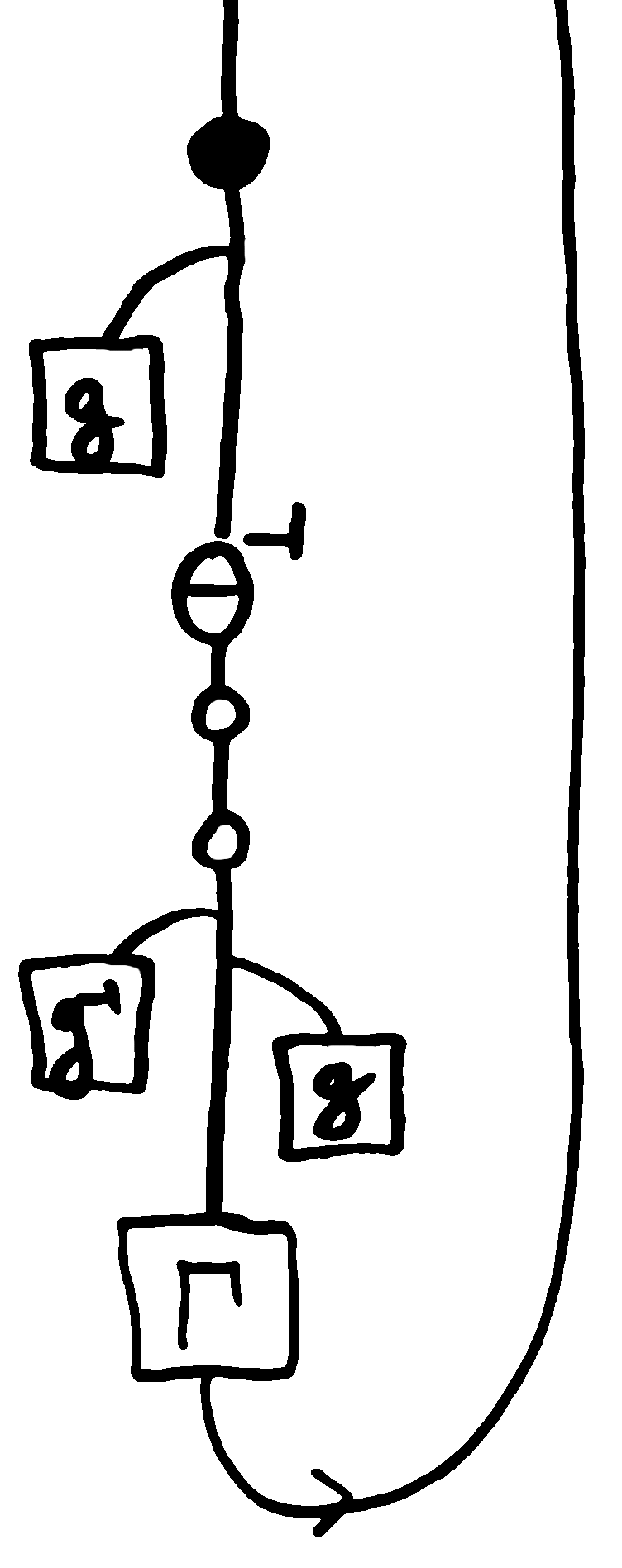}
  = 
  \scanPIC{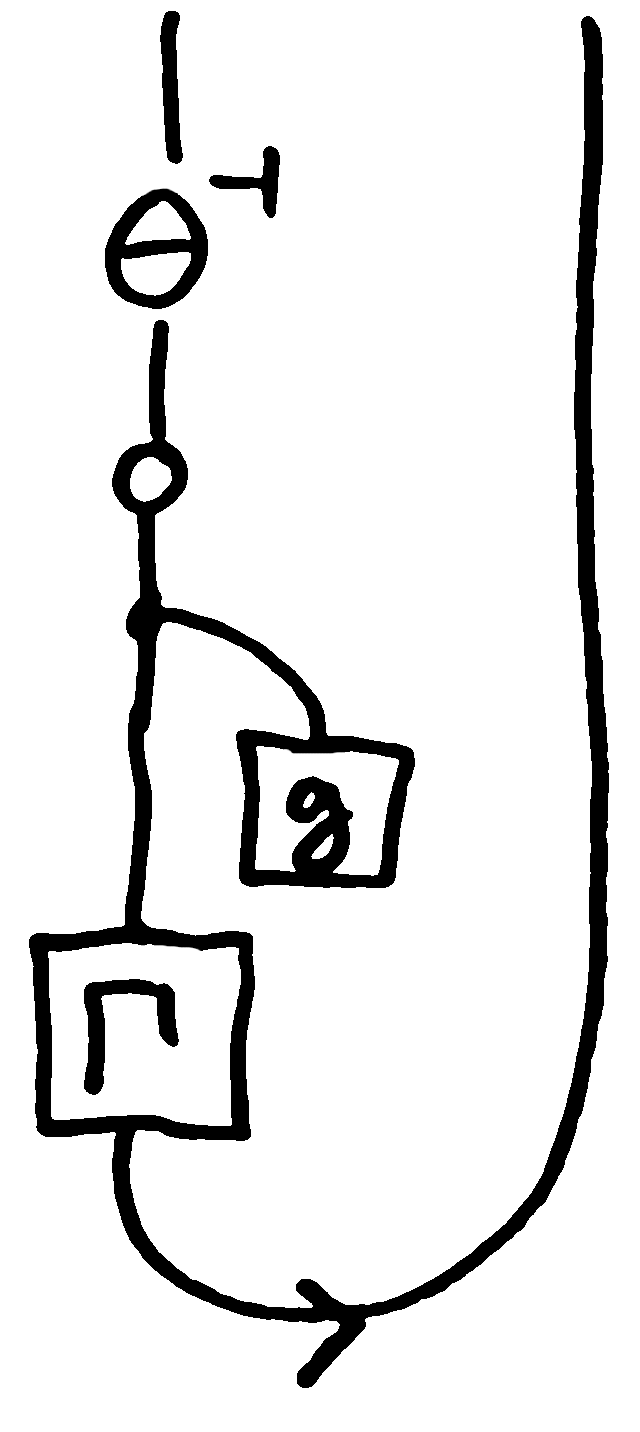}
  \quad ,
\label{eq:flip-lambda-aux1}
\end{align}
where (1) follows as $\Gamma$ is a Hopf-algebra isomorphism and $\lambda \circ \Lambda = \id_\one$ by property b) of Theorem \ref{thm:main1}\,(i) and (2) amounts to property c). The second expression is
\begin{align}
  &\scanPIC{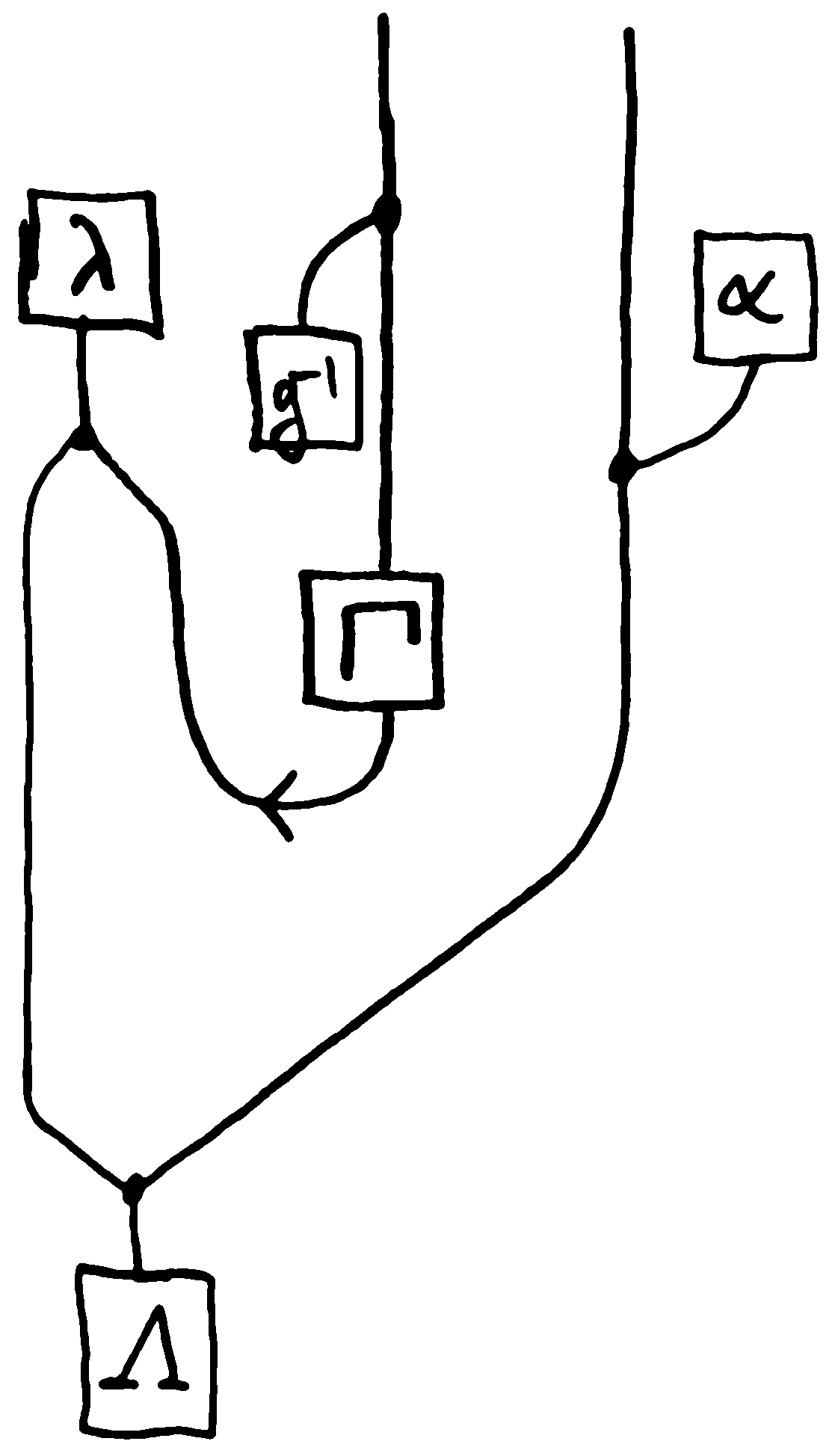} 
  = 
  \scanPIC{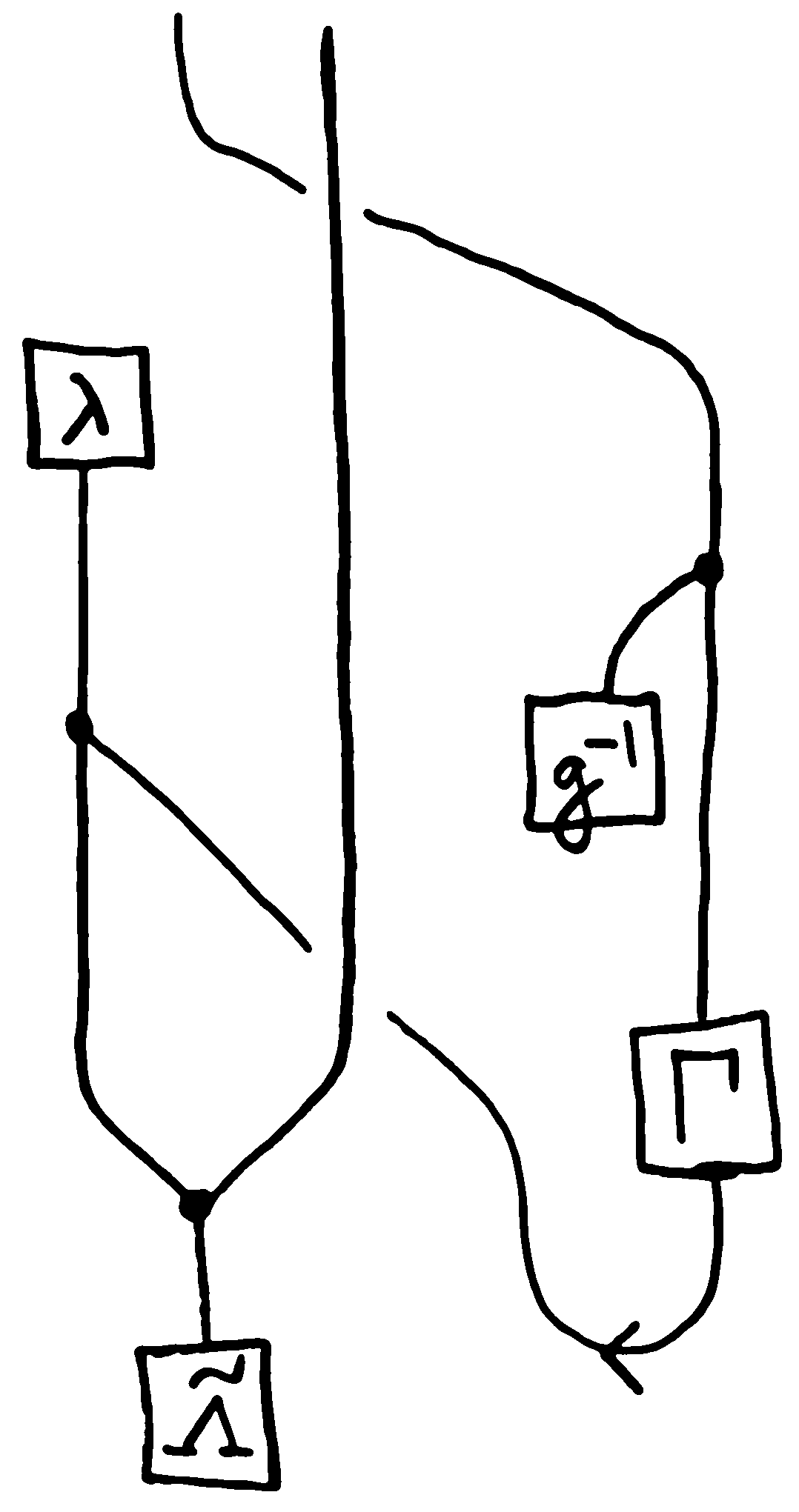}
  \overset{ \text{Lem.\,\ref{lem:move-prod-past-integral}}}{=} \scanPIC{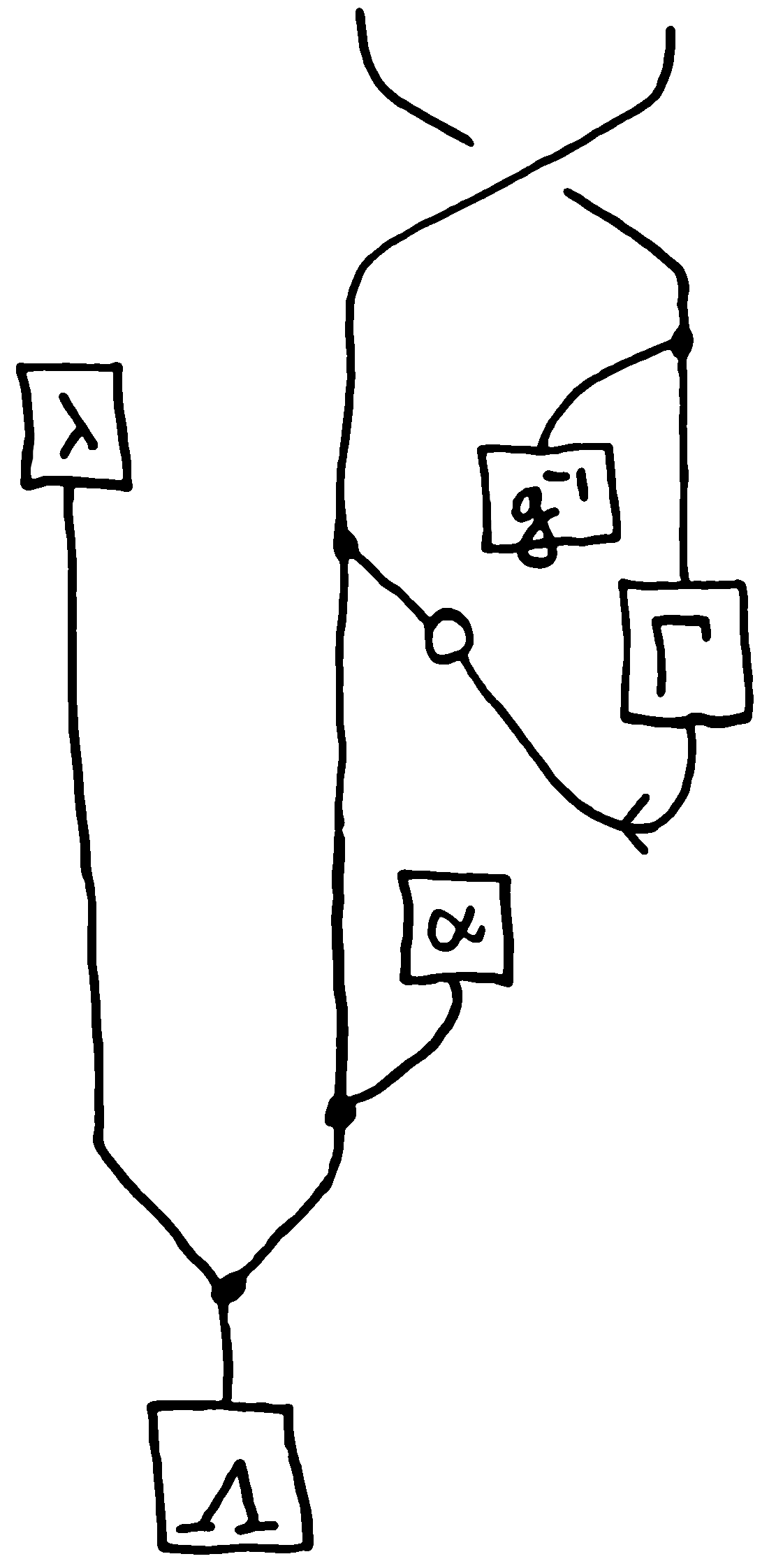}
  \nonumber\\
  &\overset{ \text{$\lambda$ coint.}}{=} \scanPIC{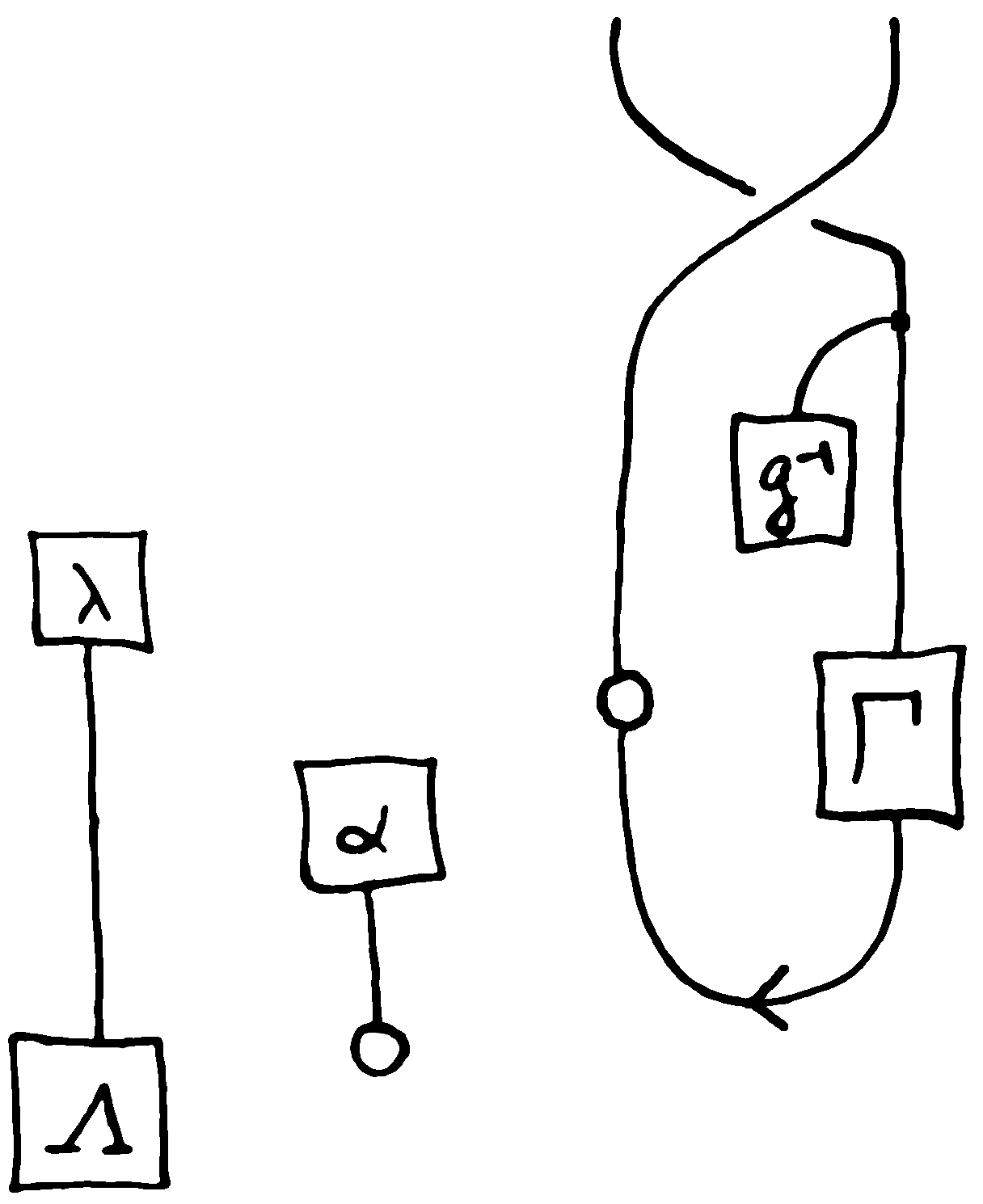}
  \overset{(*)}{=} \scanPIC{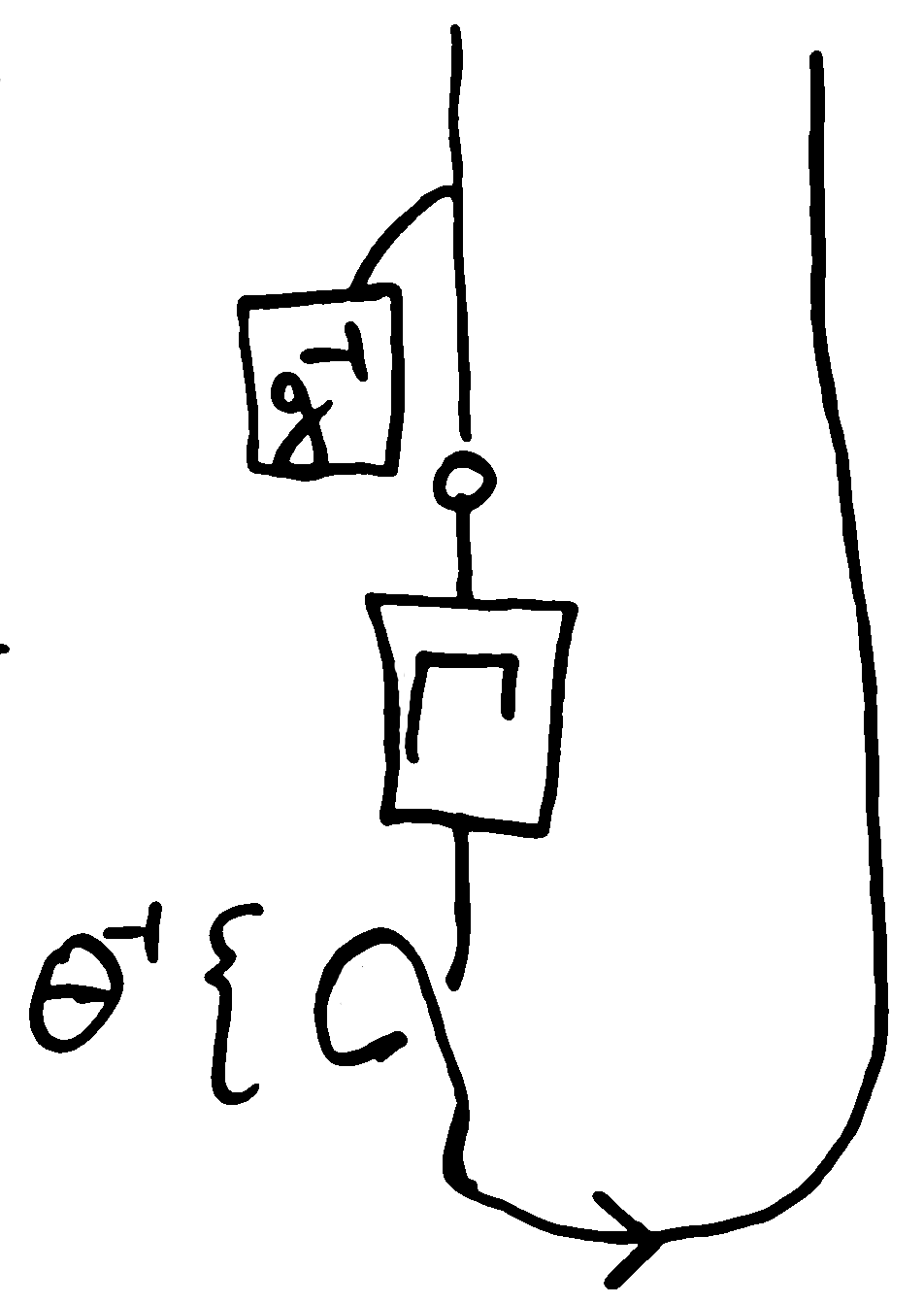} \ ,
\label{eq:flip-lambda-aux2}
\end{align}
where (*) follows as $\lambda \circ \Lambda = \id_\one$ and $\alpha\circ\eta = \id_\one$ (as one easily checks), and since $\Gamma \circ S^\vee = S \circ \Gamma$. Together with Lemma \ref{lem:g-grouplike-2}\,(ii), we see that the right hand side of \eqref{eq:flip-lambda-aux1} equals that of \eqref{eq:flip-lambda-aux2}. Since by Lemma \ref{lem:Lam-copair-nondeg} the copairing $\Delta \circ \Lambda$ is non-degenerate, and since $\alpha$ has a comultiplicative inverse (cf.\ the proof of Lemma \ref{lem:right-int-from-left-int}), the equality of \eqref{eq:flip-lambda-aux1} and \eqref{eq:flip-lambda-aux2} implies the lemma.
\end{proof}

\begin{proof}[Proof of Proposition \ref{prop:Hopf-to-pent}]
In Section \ref{sec:pentagon} we saw that the {\bf cases 1, 2, 5, 6, 9, 11} hold automatically. In the following we verify the remaining cases.

\medskip\noindent
{\bf Cases 3, 4:} In the proof of Lemma \ref{lem:Gamma-Hopf-iso} we saw that 
\be \label{eq:Hopf-iso-P3+4}
  \mu \circ (\Gamma \otimes \Gamma) = \Gamma \circ \mu_{H^\vee}
  ~ \Leftrightarrow ~
  \text{\eqref{eq:pentagon-3}} 
  \qquad \text{and} \qquad
  (\Gamma \otimes \Gamma) \circ \Delta_{H^\vee} = \Delta \circ \Gamma
  ~ \Leftrightarrow ~
  \text{\eqref{eq:pentagon-4}} \ .
\ee
Thus cases 3 and 4 follow since $\Gamma$ is a Hopf-algebra map. 

\medskip

It turns out to be more convenient to work with $\gamma$ from \eqref{eq:gdp-via-Hopf} instead of $\Gamma$ and use \eqref{eq:pentagon-3} and \eqref{eq:pentagon-4} (which we established above). In particular, $\gamma$ satisfies all the properties stated in Lemmas \ref{lem:gamma-inv-via-S} and \ref{lem:gamma-inv}. We will use these properties below without further mention. 

Take the equalities in \eqref{eq:Gamma-dual-via-gamma} (which hold because of property c) in Theorem \ref{thm:main1}\,(i)) and substitute the definition \eqref{eq:gdp-via-Hopf} of $\gamma$ in terms of $\Gamma$. Comparing the resulting identity to the definition of $\delta$ in \eqref{eq:gdp-via-Hopf}, we get the following two expressions for $\delta$ (the first is the definition and the second already appeared in Lemma \ref{lem:delta-braid} when proving the converse of Theorem \ref{thm:main1}\,(i)):
\be\label{eq:delta-expression-ab}
\raisebox{3em}{\text{a)}} \quad
\delta ~=~ \scanPIC{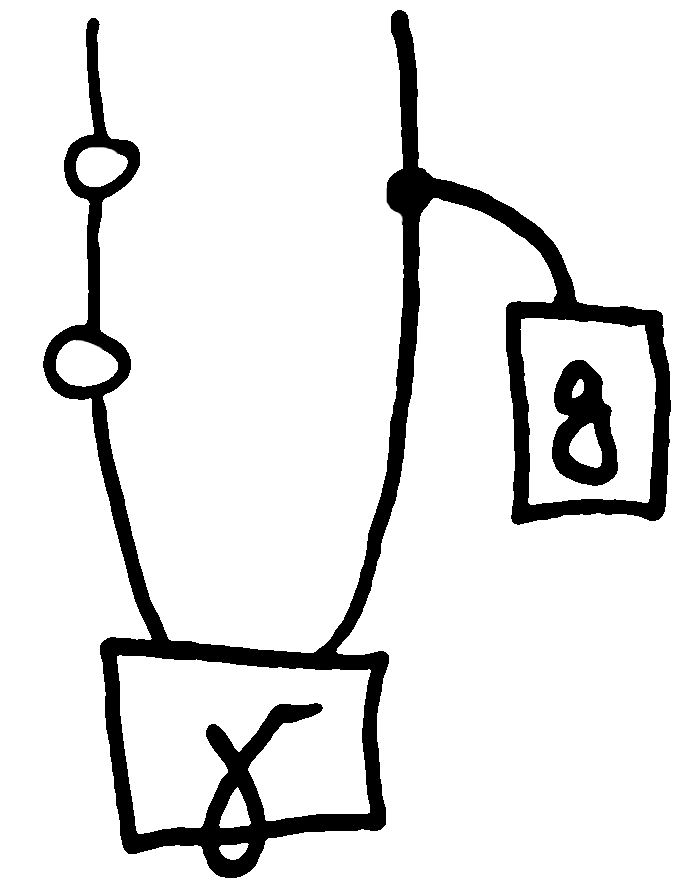}
\qquad , \qquad
\raisebox{3em}{\text{b)}} \quad
\delta ~=~ \scanPIC{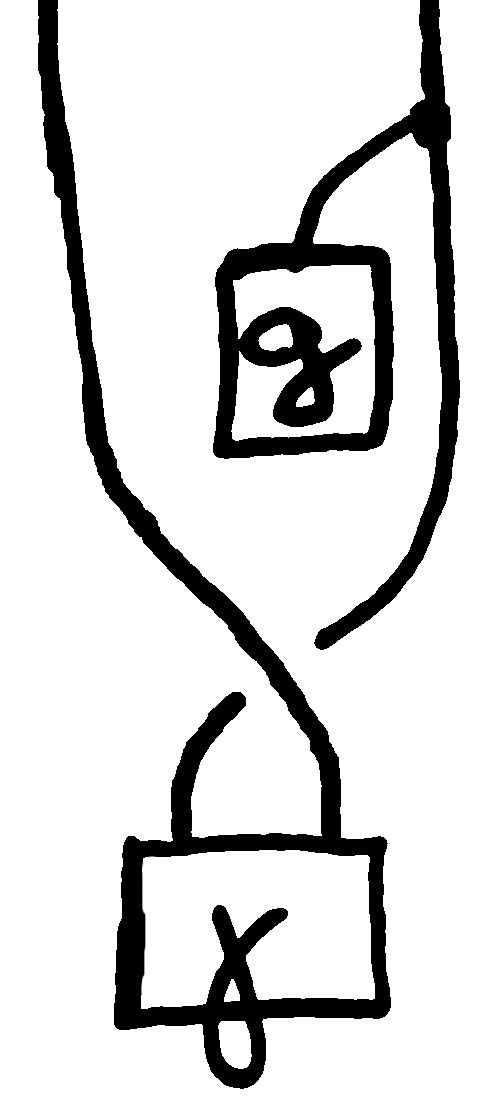}
\ee

\medskip\noindent
{\bf Cases 7, 8, 10, 12:} These pentagon identities are direct consequences of \eqref{eq:pentagon-3} and \eqref{eq:pentagon-4}. For case 7 consider the equalities
\begin{align}
  &\text{rhs of \eqref{eq:pentagon-7}} 
  \overset{(1)}{=} 
  \scanpic{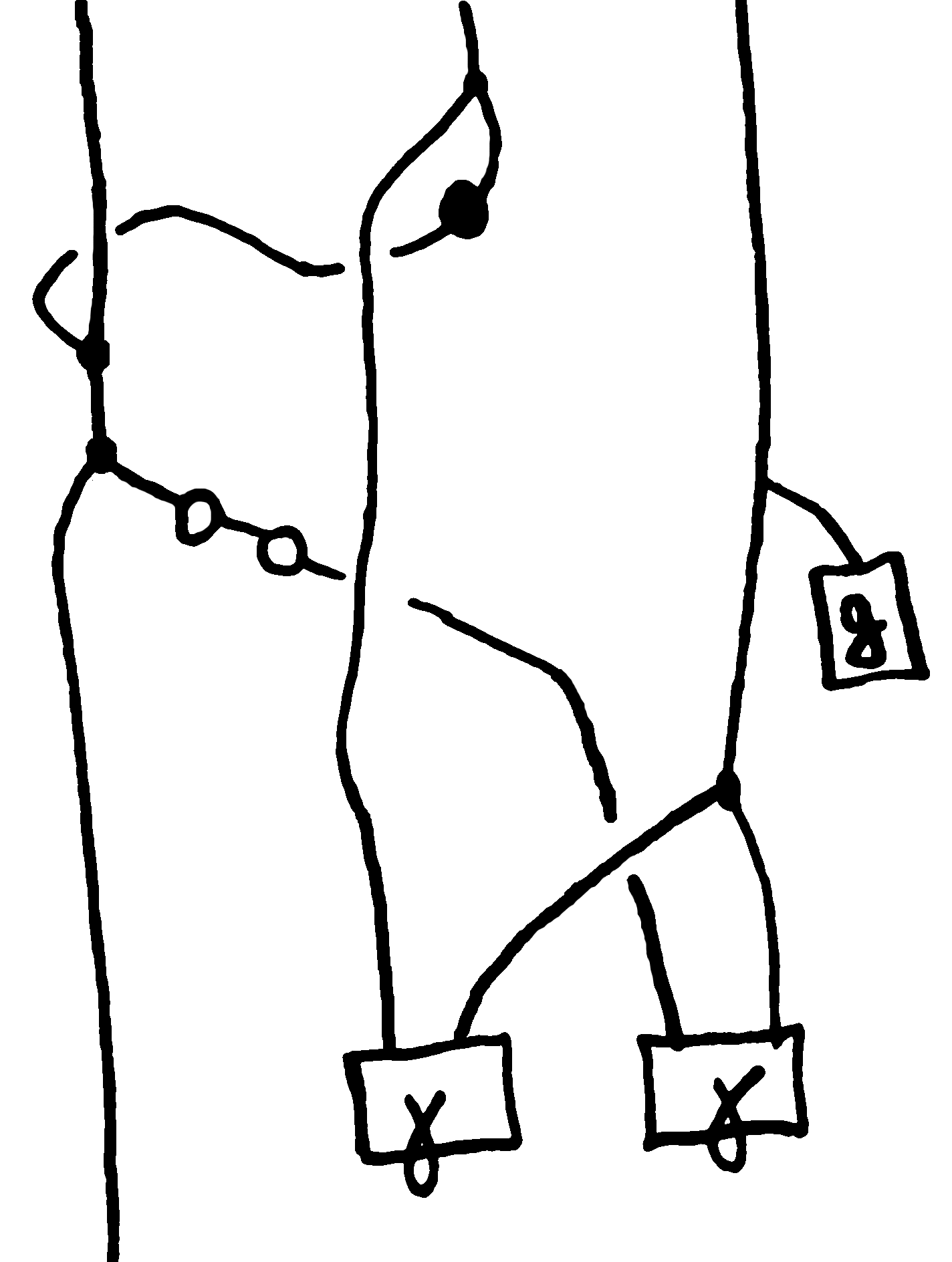}
  \overset{(2)}{=} 
  \scanpic{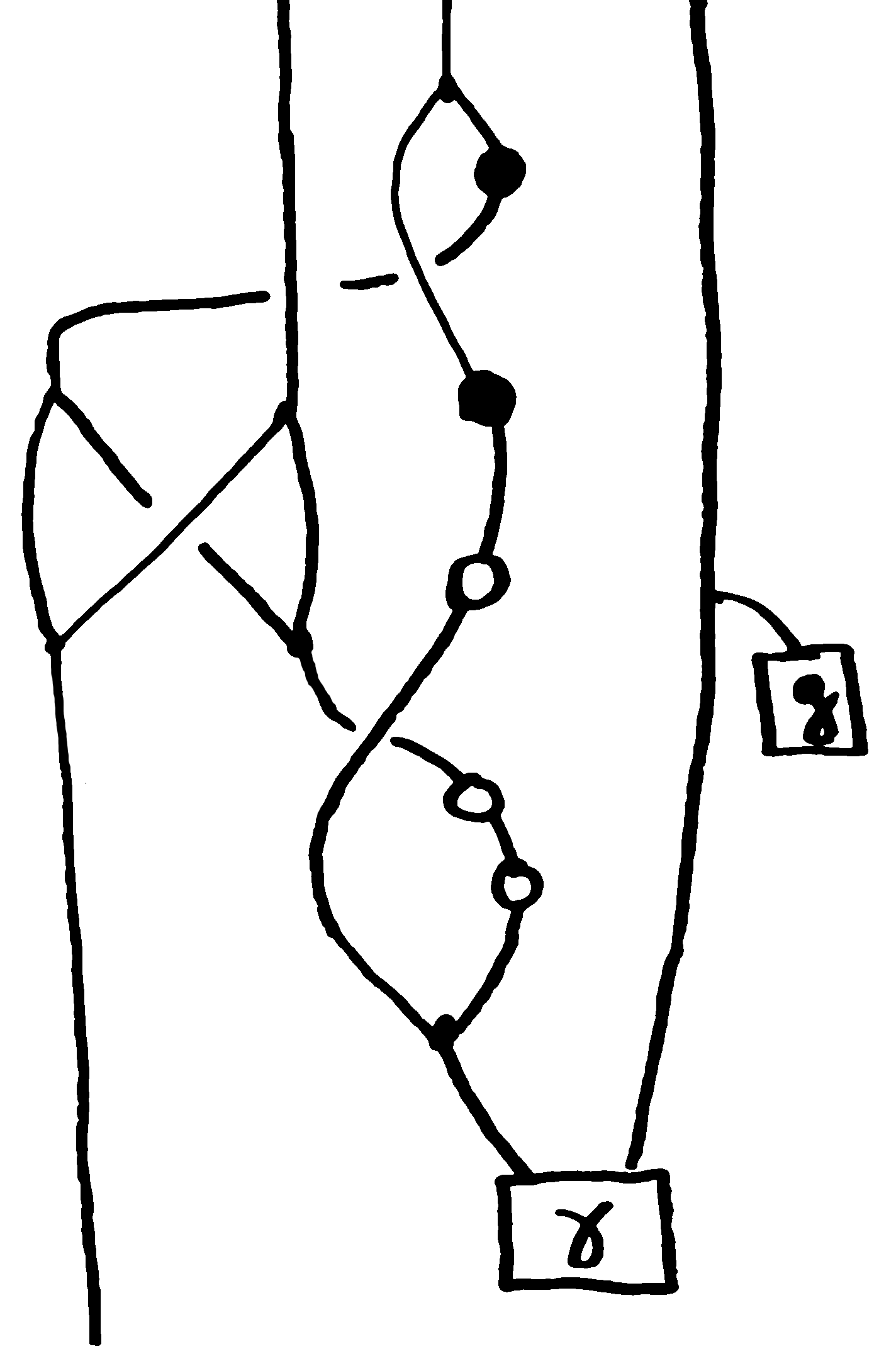}
\nonumber\\
  &\overset{(3)}{=} 
  \scanpic{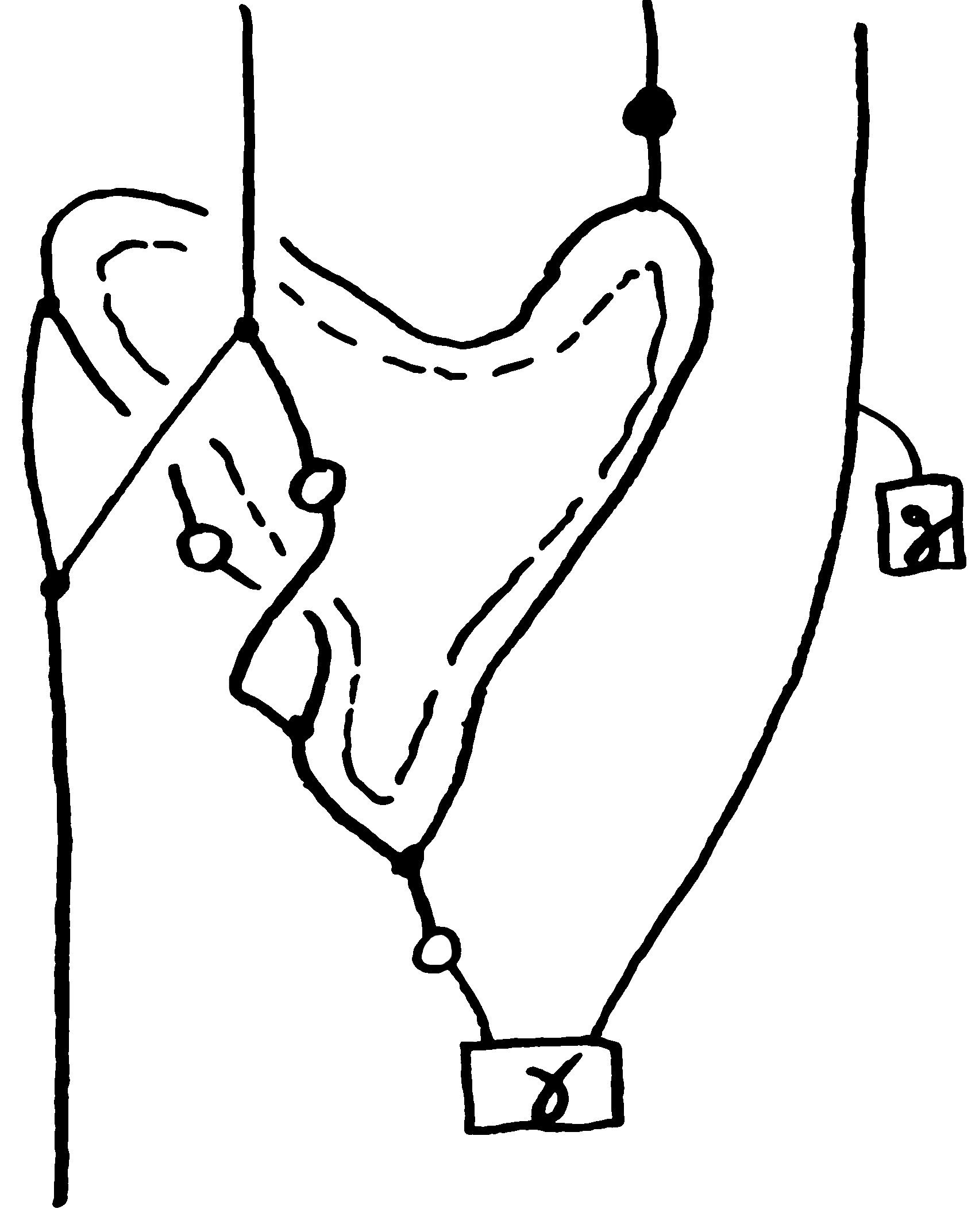}
  \overset{(4)}{=} 
  \scanpic{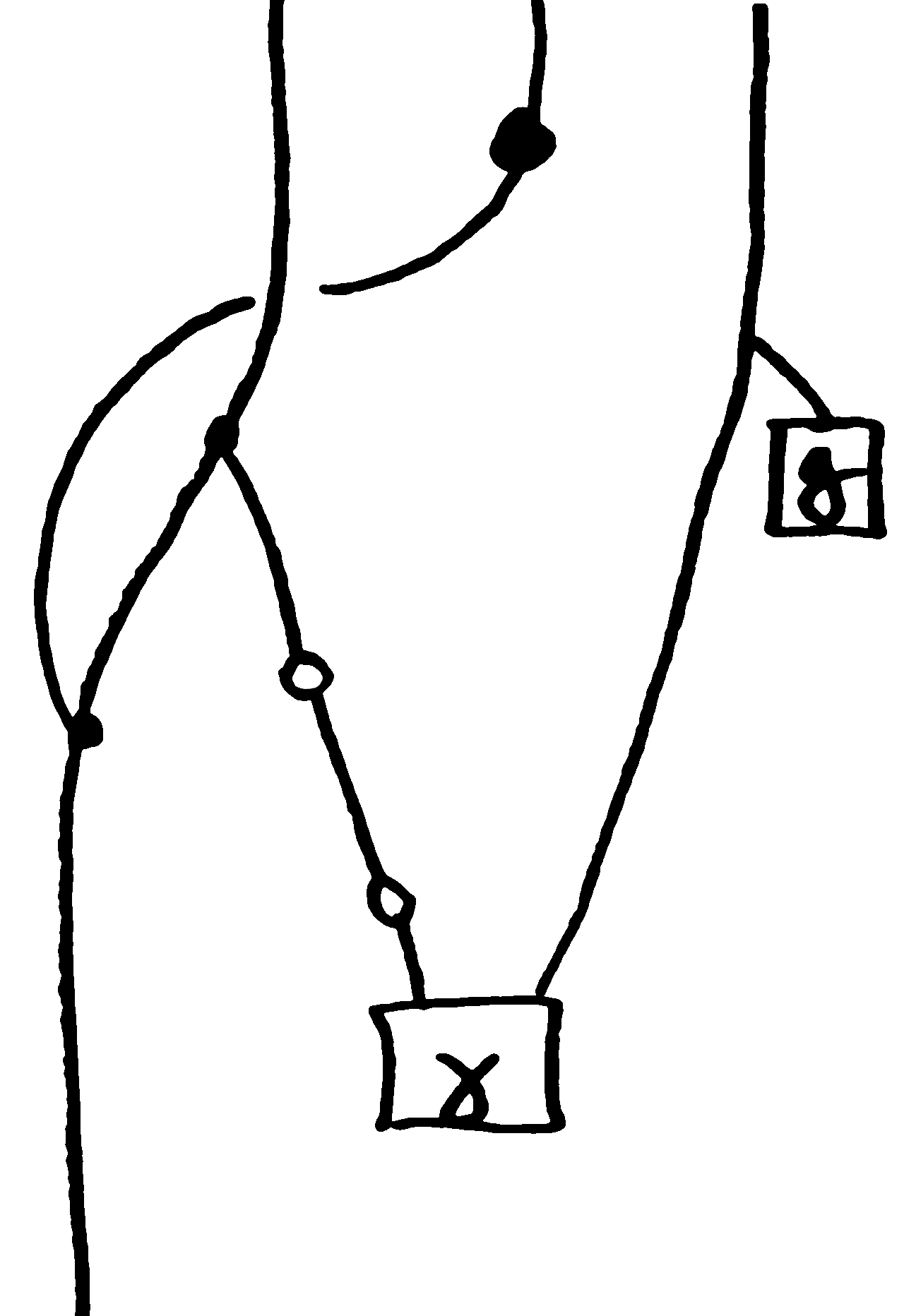}
  = 
  \text{lhs of \eqref{eq:pentagon-7}} \ .
\end{align}
In step (1), expression (a) for $\delta$ was substituted and associativity was used to move $g$ past the product. In step (2), property \eqref{eq:pentagon-3} of $\gamma$ and the algebra-map property of $\Delta$ was used, and a pair $S^{-1} \circ S$ has been inserted. In step (3), $S$ and $S^{-1}$ have been moved past products and coproducts. In step (4) one produces a loop to which one can apply the bubble-property. For case 8 we have, using expression (b) for $\delta$,
\be
  \text{lhs of \eqref{eq:pentagon-8}} 
  ~=~ 
  \scanpic{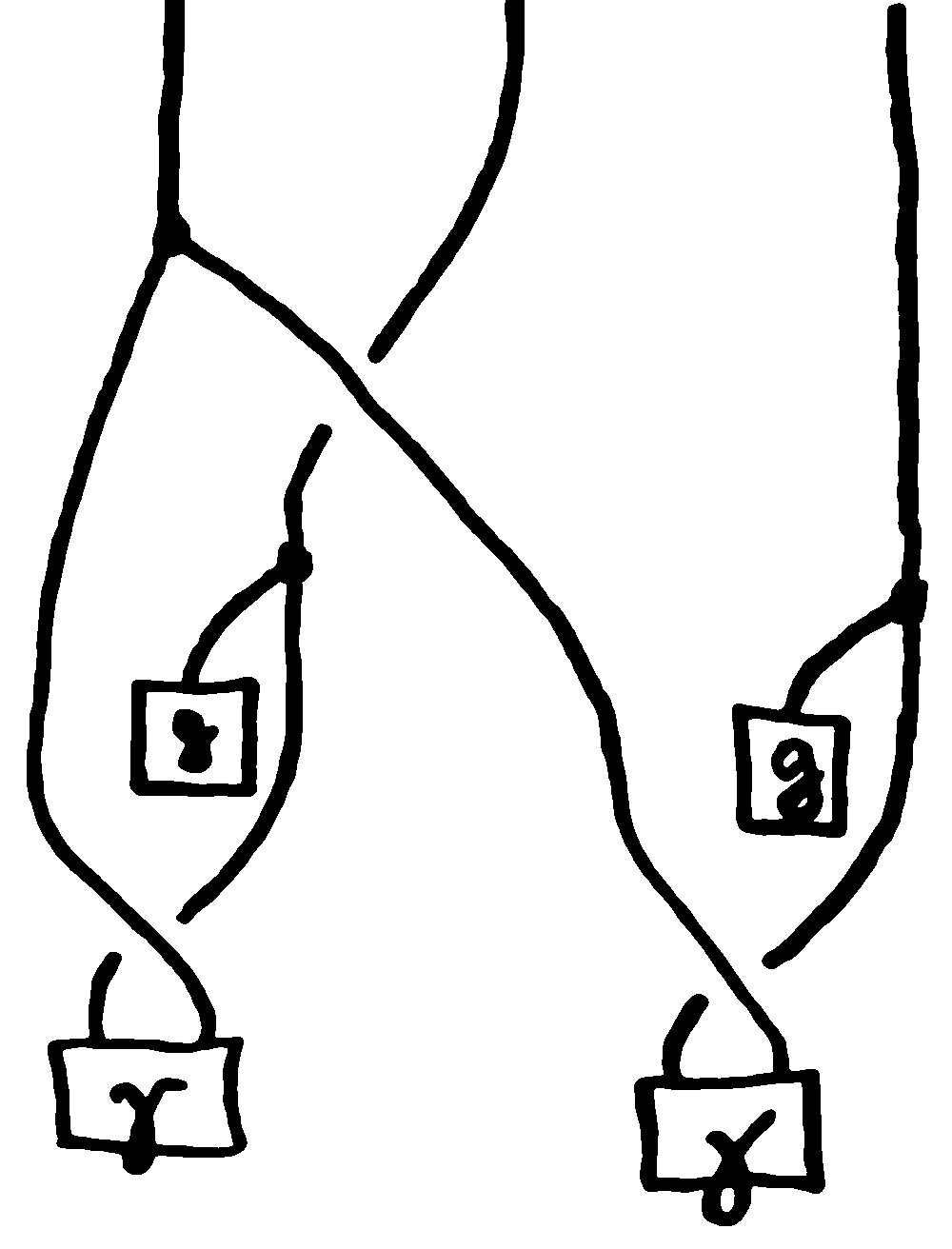} 
  ~\overset{\eqref{eq:pentagon-3}}{=}~
  \scanpic{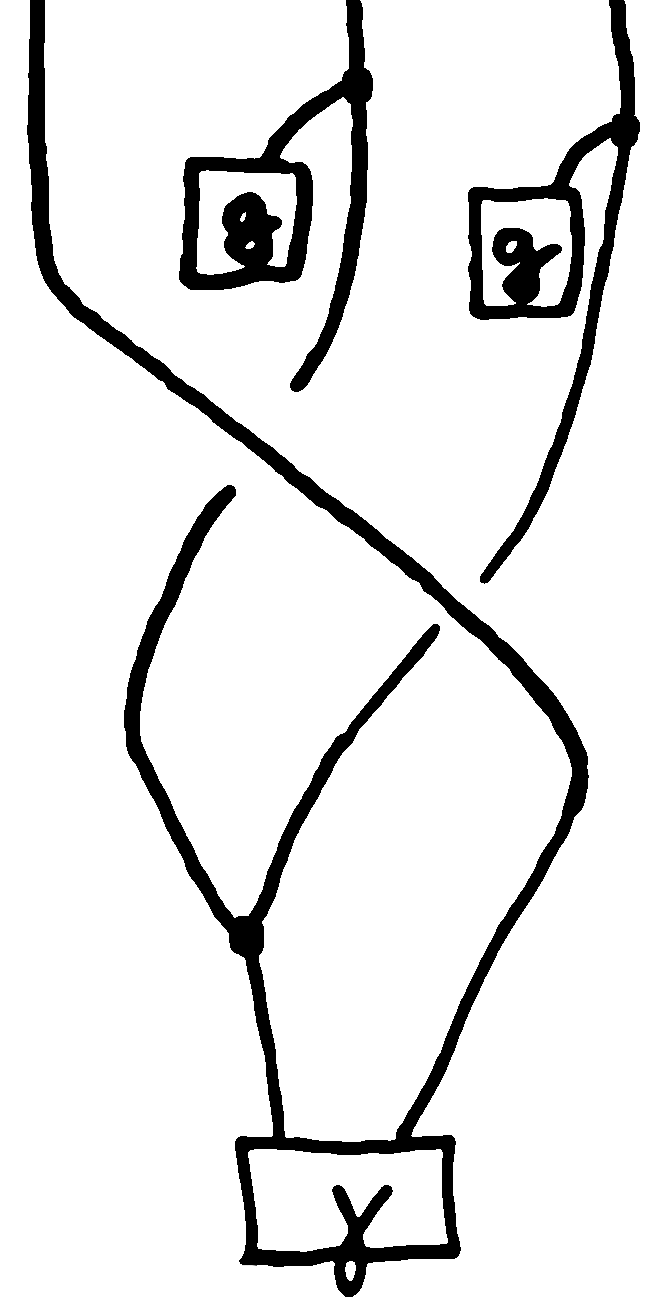} 
  ~\overset{\eqref{eq:leftmult-g-coalg}}{=}~ 
  \text{rhs of \eqref{eq:pentagon-8}} \ .
\ee
For case 10 use once more expression (b) for $\delta$, as well as \eqref{eq:pentagon-4} (step 1) and the algebra-map property of $\Delta$ (step 2):
\be
  \text{rhs of \eqref{eq:pentagon-10}} 
  ~=~ 
  \scanpic{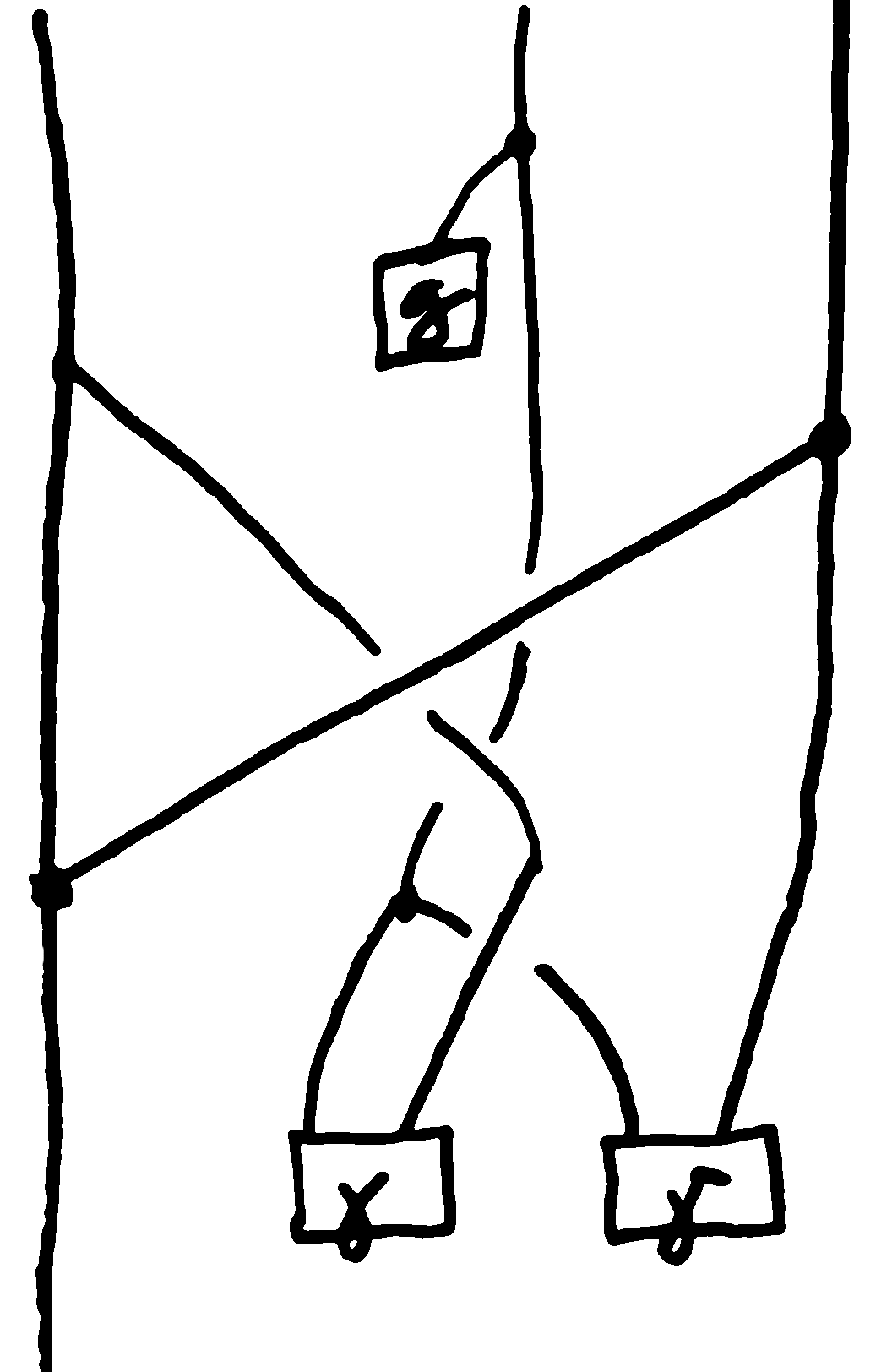} 
  ~\overset{(1)}{=}~
  \scanpic{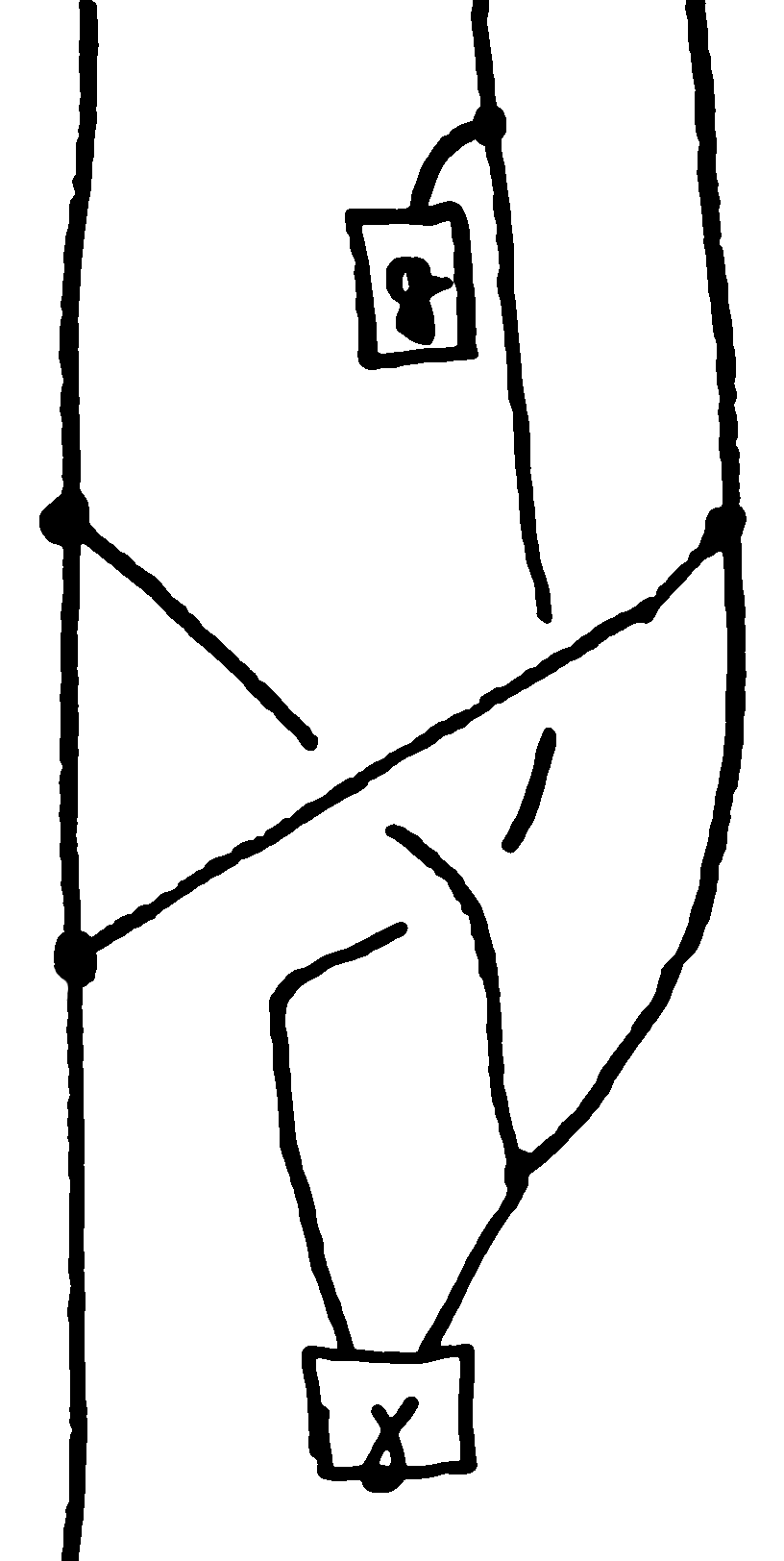} 
  ~\overset{(2)}{=}~
  \scanpic{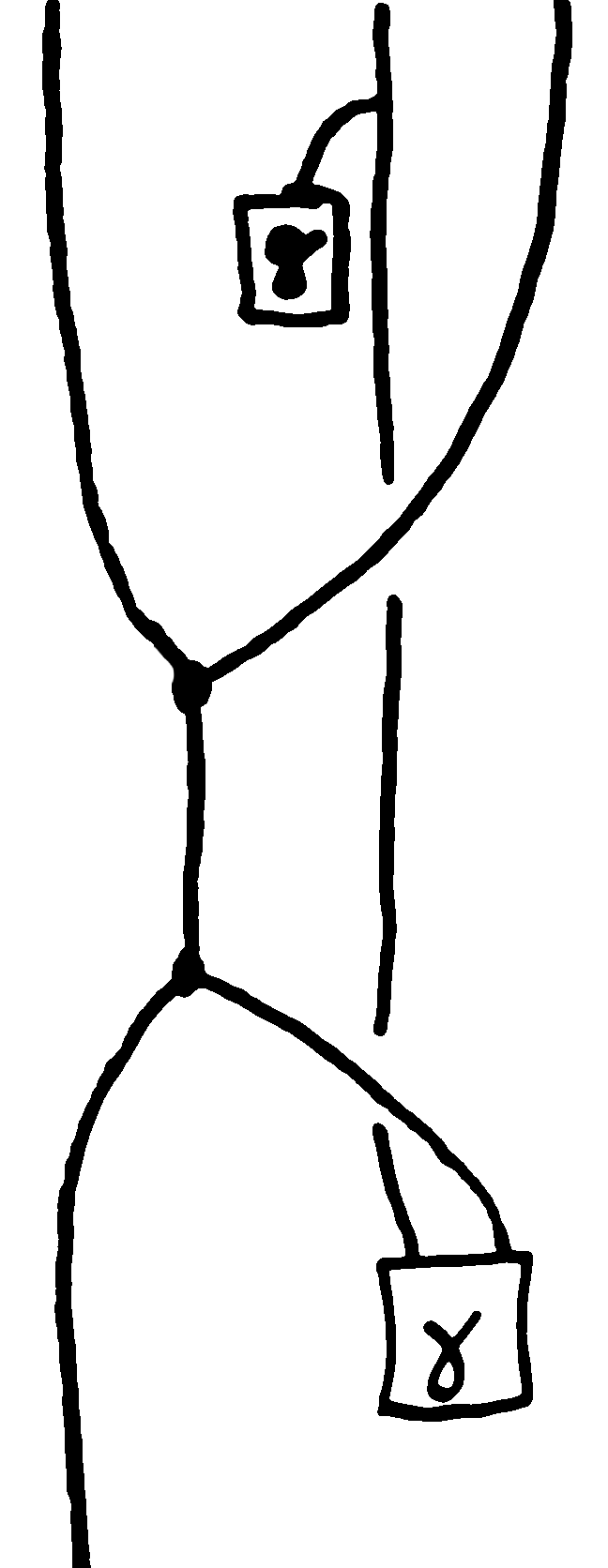} 
  ~=~ 
  \text{lhs of \eqref{eq:pentagon-10}} \ .
\ee
The computation for case 12 is
\be
  \text{lhs of \eqref{eq:pentagon-12}}
  ~\overset{\text{\eqref{eq:gdp-via-Hopf}}}{\underset{\text{\&\,assoc.}}{=}}~ 
  \scanPIC{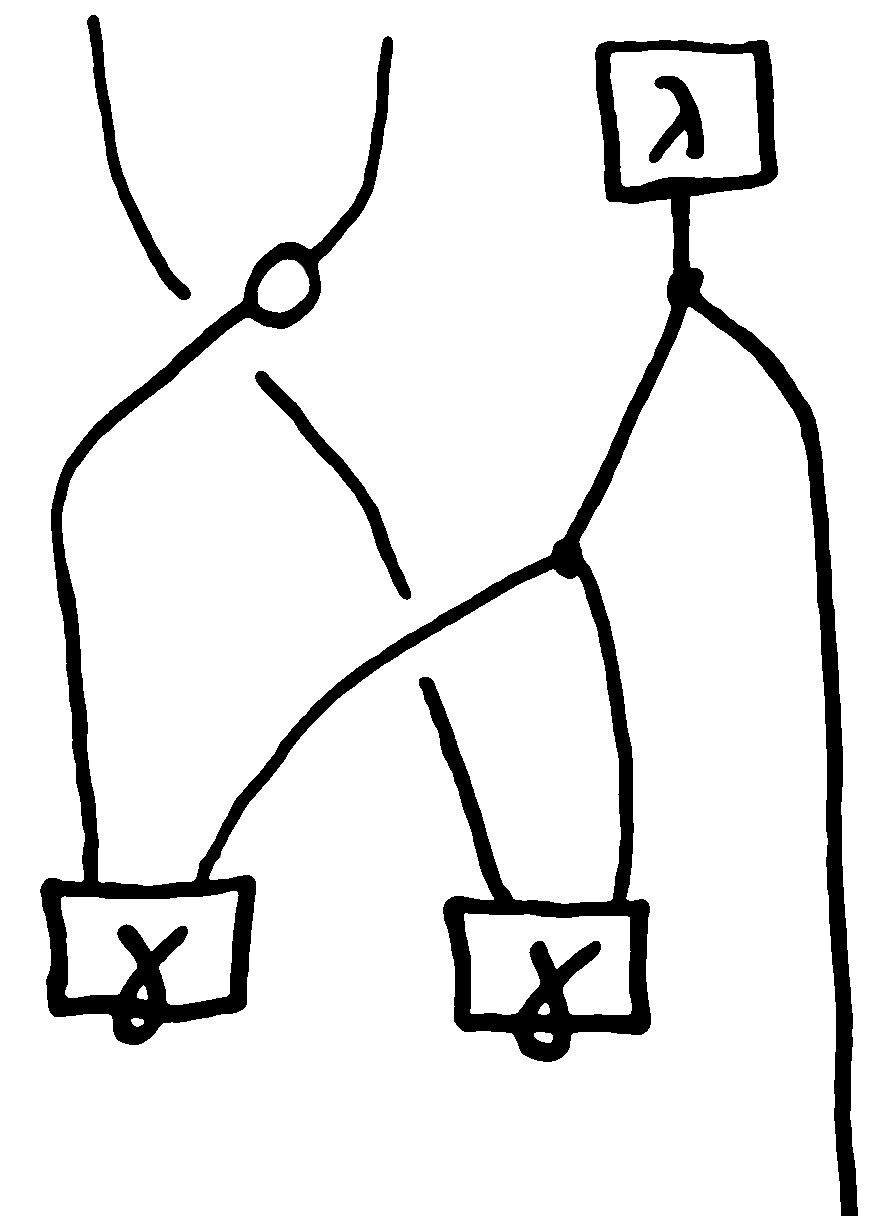} 
    ~\overset{\text{\eqref{eq:pentagon-3}}}=~ 
  \scanPIC{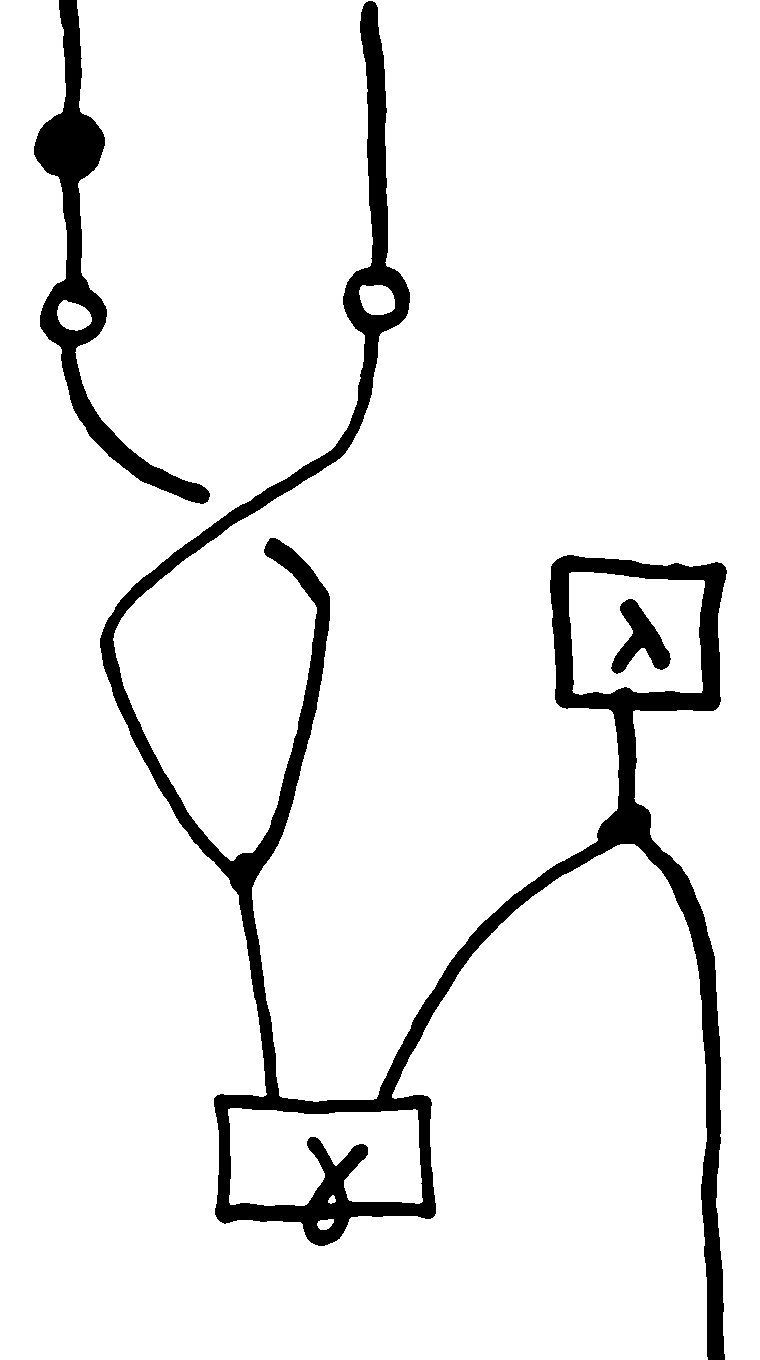} 
    ~\overset{\text{\eqref{eq:gdp-via-Hopf}}}=~ 
  \text{rhs of \eqref{eq:pentagon-12}} \ .
\ee

\medskip\noindent
{\bf Cases 13, 15:} In these cases we will make use of Lemma \ref{lem:move-coprod-past-lam}.
Firstly,
\begin{align}
  \text{rhs of \eqref{eq:pentagon-13}} 
   & 
  \overset{\text{(\ref{eq:delta-expression-ab}\,a)}}{\underset{\text{\eqref{eq:gdp-via-Hopf}}}{=}}~ 
  \scanpic{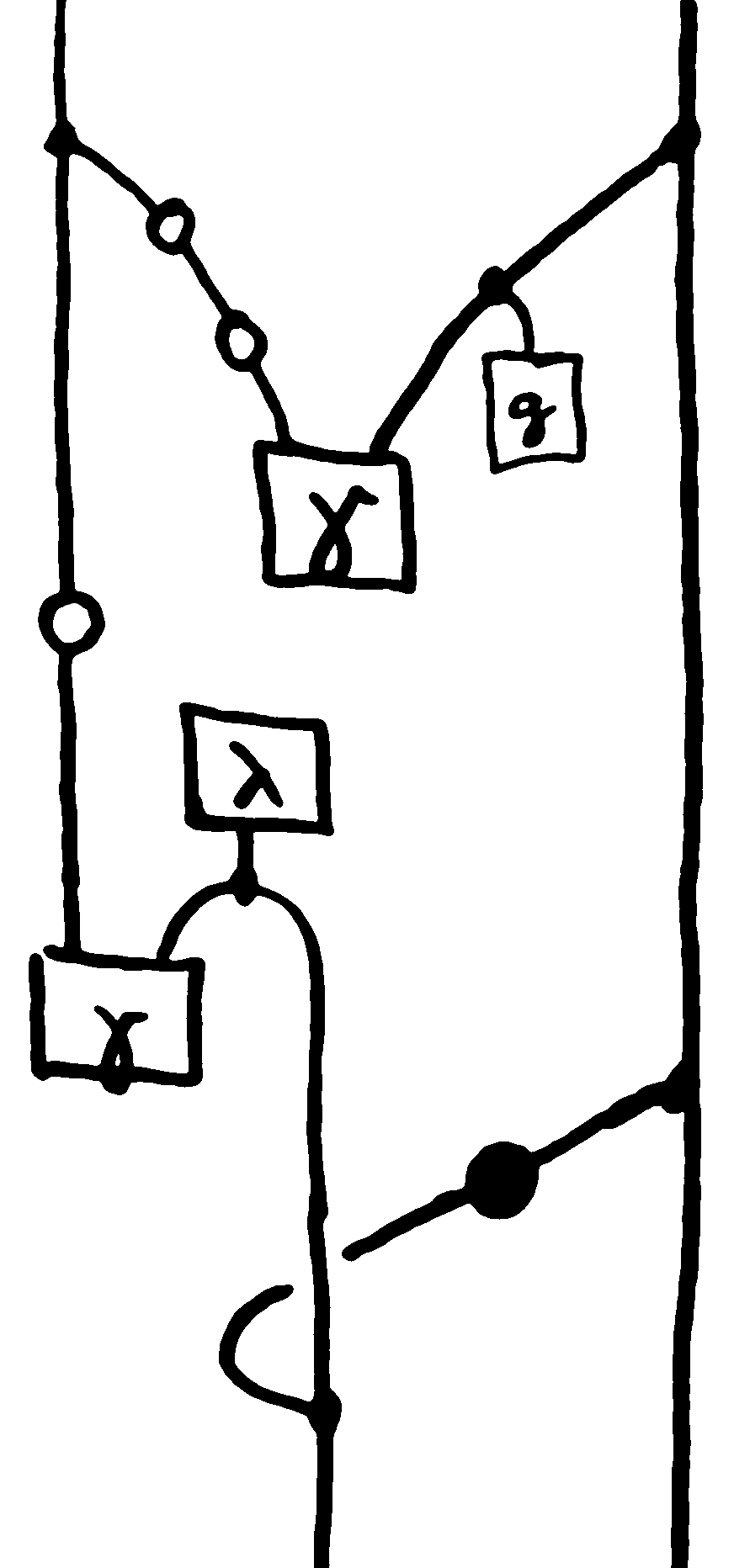} 
  ~\overset{(1)}{=}~ 
  \scanpic{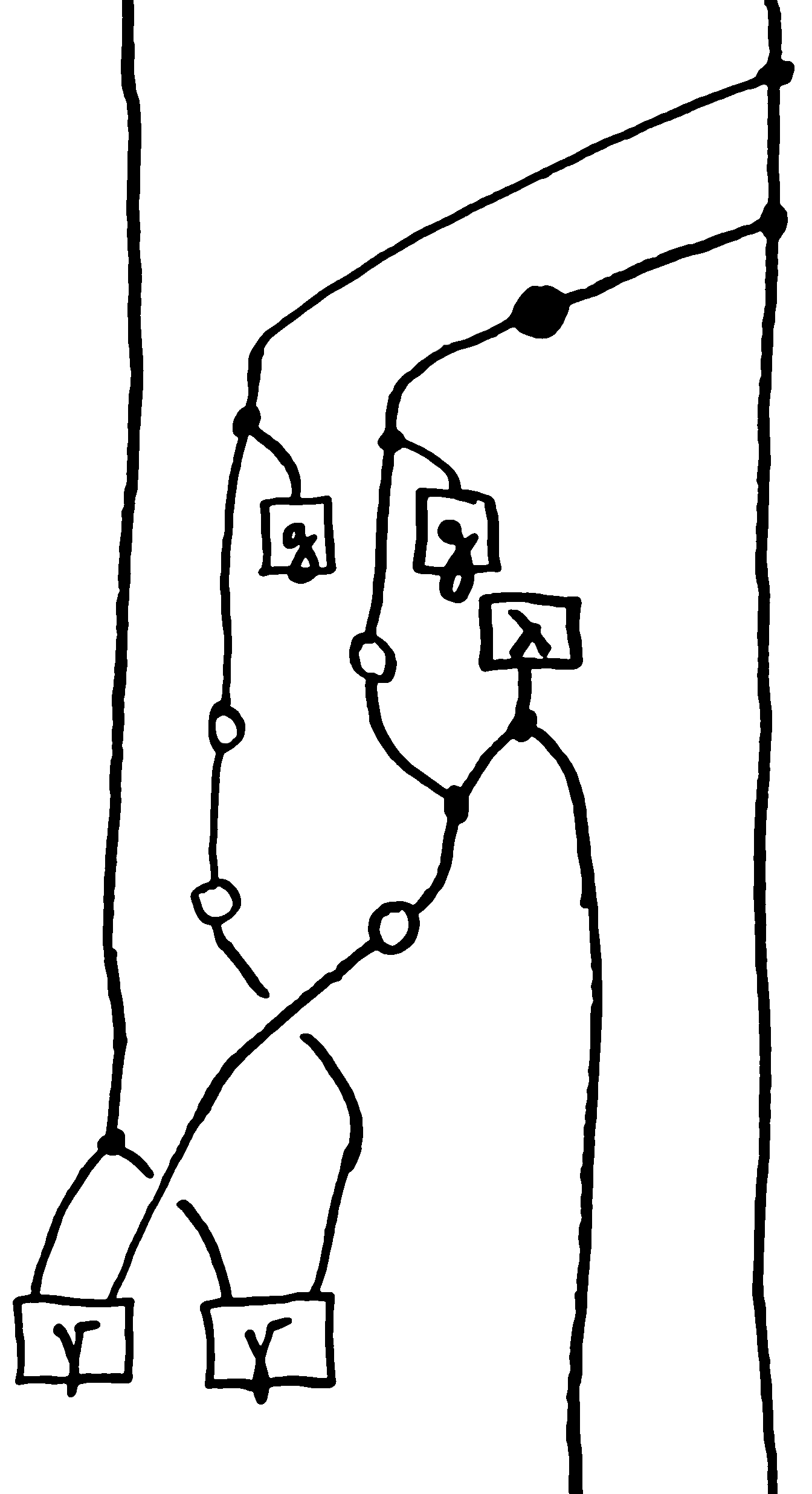} 
  ~\overset{(2)}{=}~ 
  \scanpic{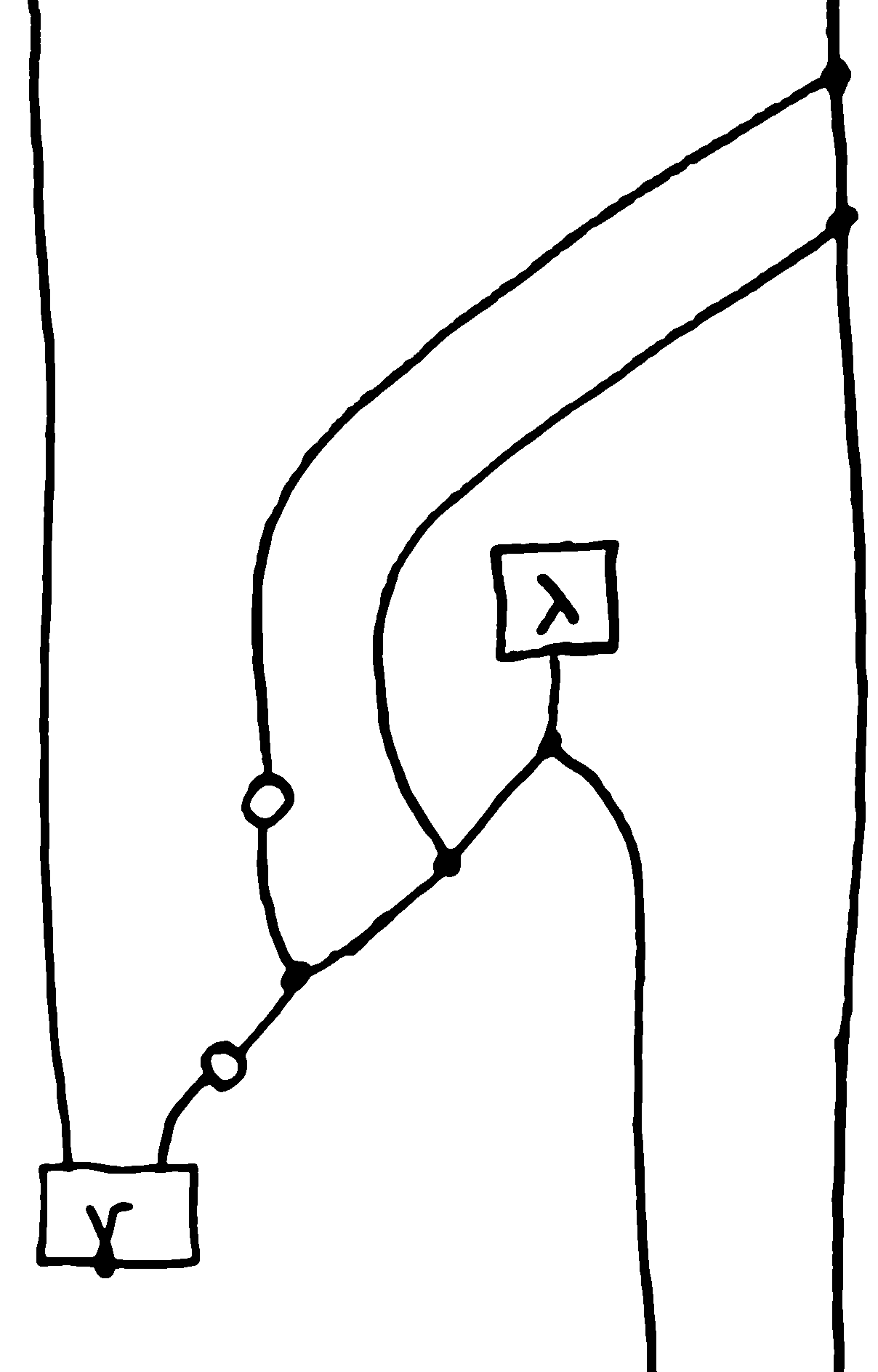} 
  \nonumber \\
  &\overset{(3)}{=}
  \text{lhs of \eqref{eq:pentagon-13}} \ ,
\end{align}
where step 1 is a deformation of the string diagram and an application of Lemma \ref{lem:move-coprod-past-lam}; in step 2, property \eqref{eq:pentagon-4} of $\gamma$ is applied, $S$ is moved past the resulting coproduct, the two $g$'s cancel (recall Lemma \ref{lem:g-grouplike-2}\,(ii)), and an $S^{-1} \circ S$-pair is cancelled. In step 3 the bubble-property is used to remove a loop.
Secondly,
\be
  \text{lhs of \eqref{eq:pentagon-15}} 
  ~=~ 
  \scanPIC{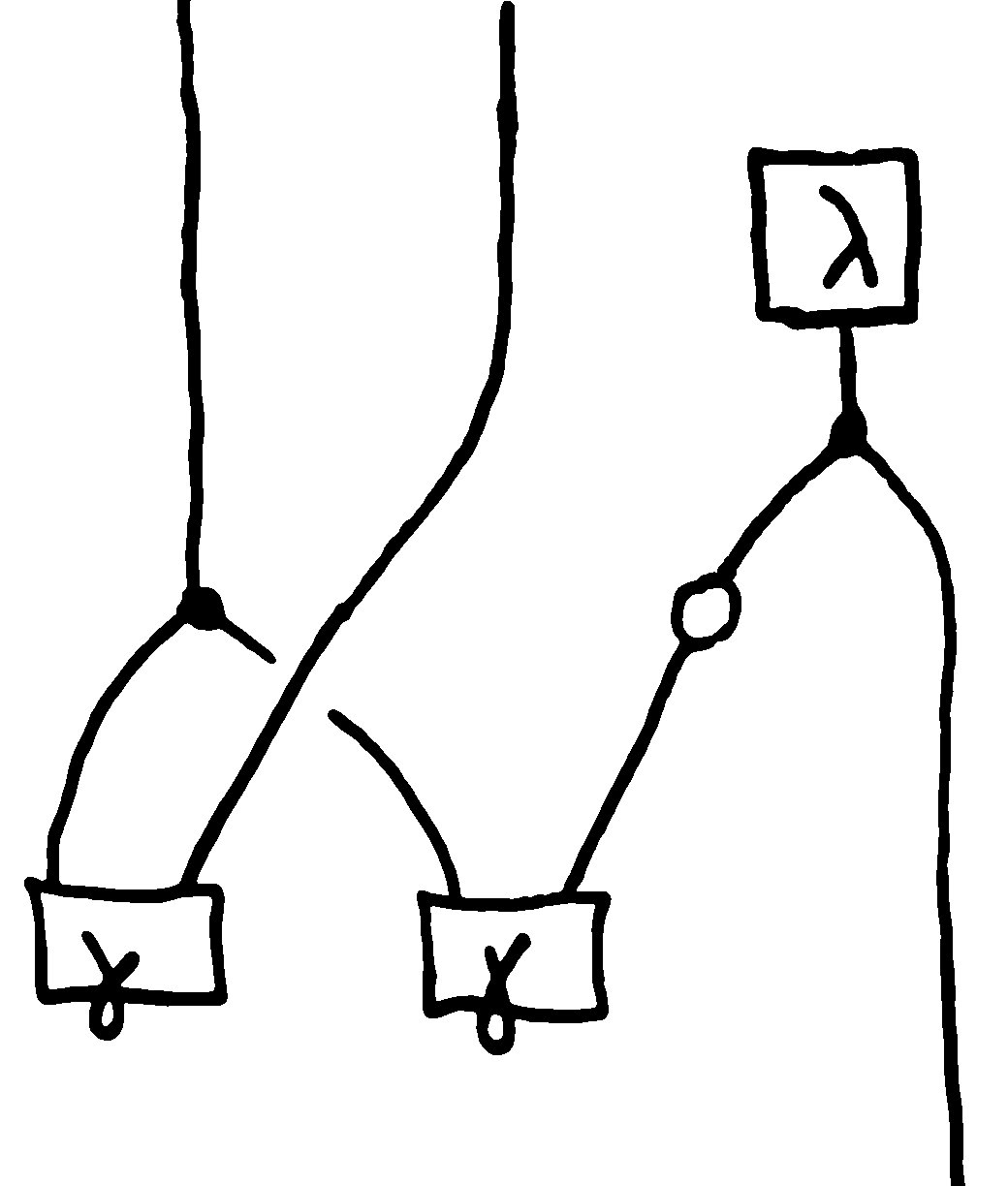} 
  ~\overset{\eqref{eq:pentagon-4}}{=} ~
  \scanPIC{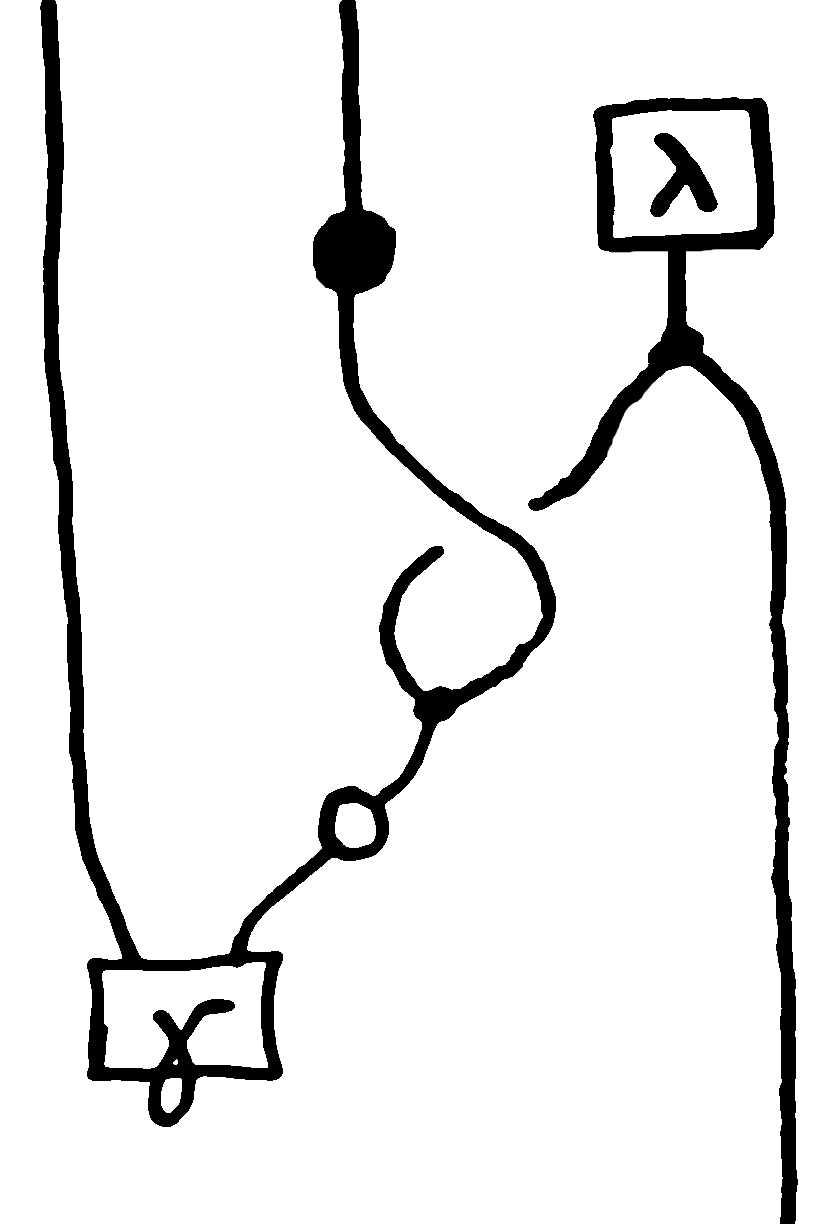} 
  ~\overset{\text{Lem.\,\ref{lem:move-coprod-past-lam}}}{=} ~
  \scanPIC{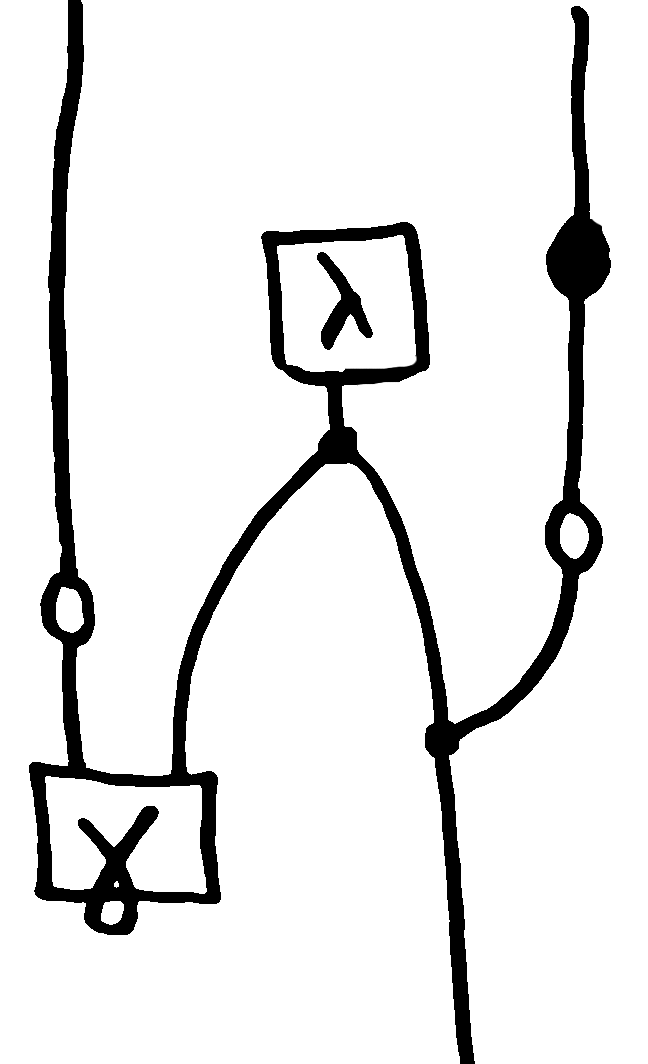} 
  ~=~ 
  \text{rhs of \eqref{eq:pentagon-15}} \ .
\ee

\medskip\noindent
{\bf Cases 14, 16:} First note that the statement of Lemma \ref{lem:flip-lambda} can be rewritten as
\be\label{eq:flip-lambda-alt}
  \scanPIC{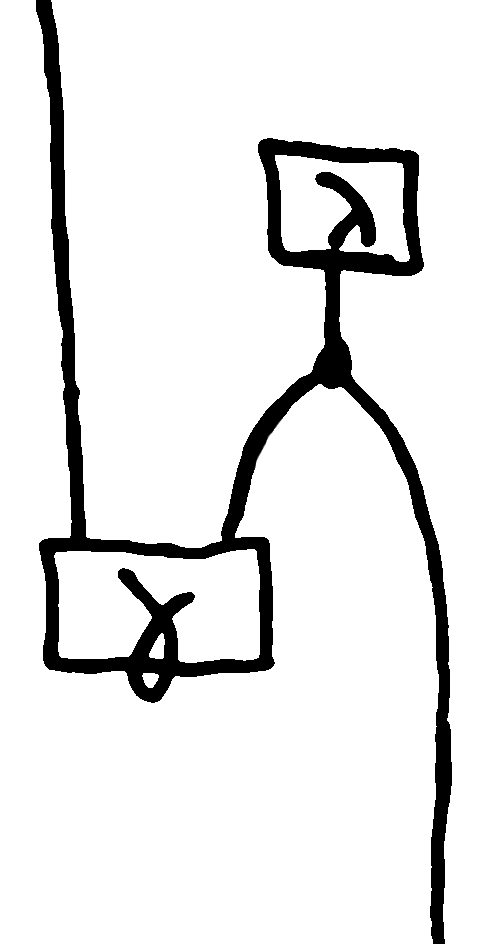} 
  ~=~ 
  \scanPIC{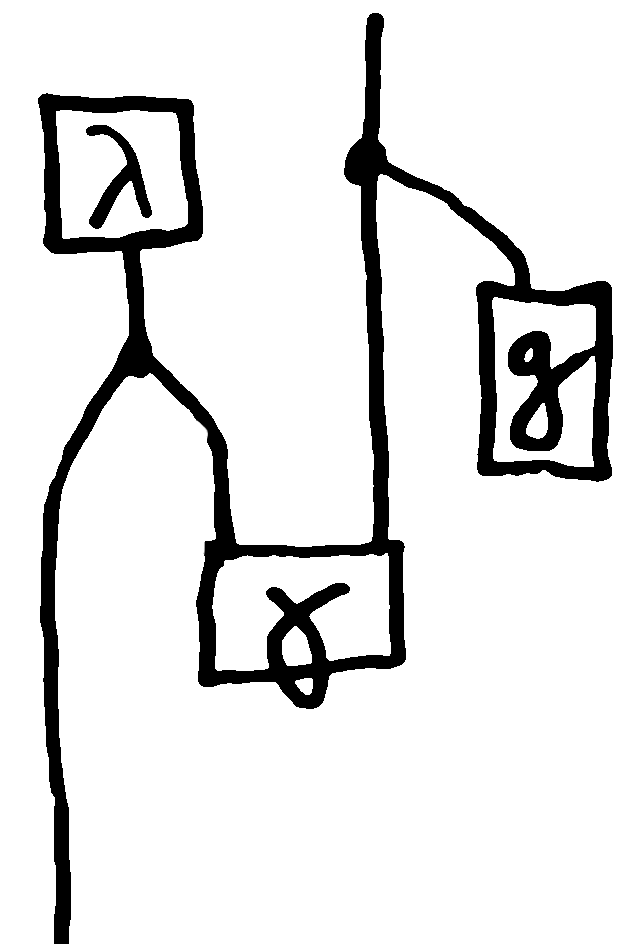} 
  \quad .
\ee
Next we turn to case 14. Because both sides of \eqref{eq:pentagon-14} commute with the right $H$-action on the right factor of $H \otimes H$, it is enough to show \eqref{eq:pentagon-14} after applying $\id \otimes \eta$ from the right. Using also Lemma \ref{lem:H-intertwiner} we see that
\be\label{eq:pentagon-14-rewrite}
  \text{\eqref{eq:pentagon-14}} \qquad \Leftrightarrow \qquad
  \scanPIC{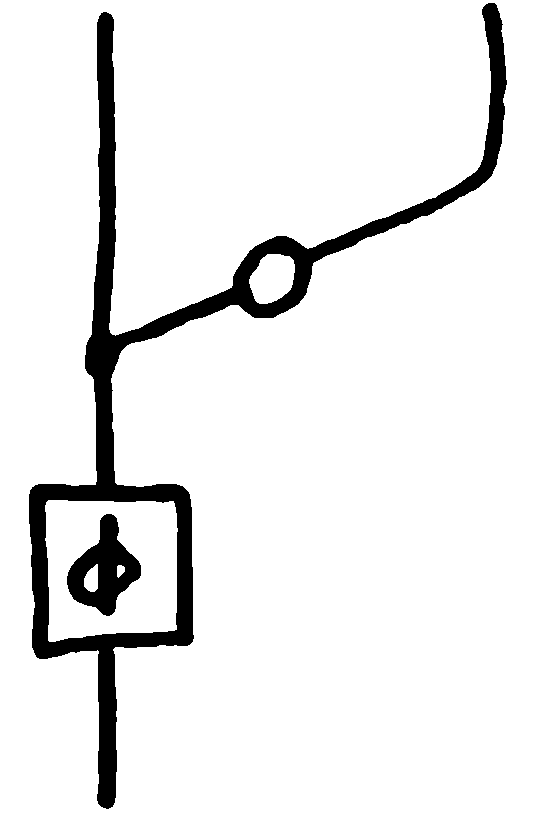}
  ~=~ 
  \scanPIC{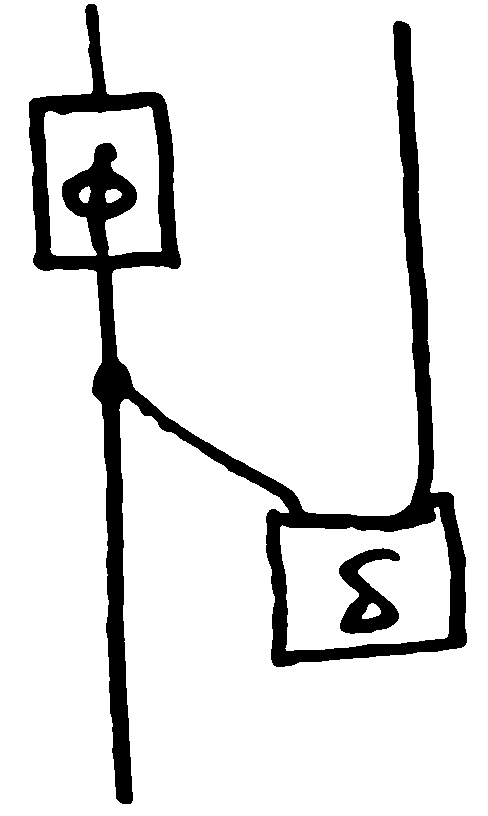} 
  \quad .
\ee
Starting from the right hand side, we compute
\begin{align}
  \text{rhs of \eqref{eq:pentagon-14-rewrite}} 
  &
  \overset{\text{(\ref{eq:delta-expression-ab}\,a)}}{\underset{\text{\eqref{eq:gdp-via-Hopf}}}{=}}~ 
  \scanPIC{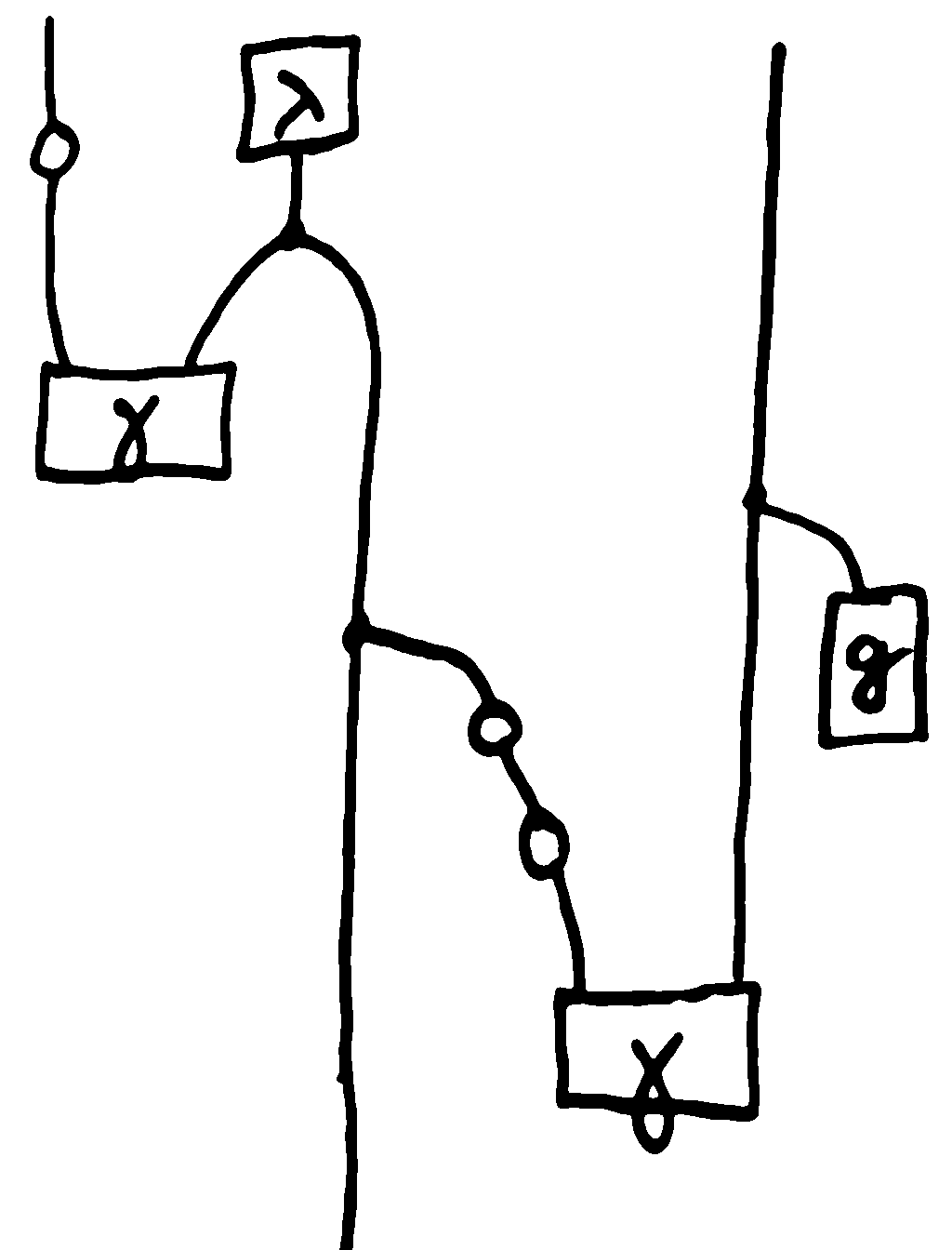}
  ~\overset{\eqref{eq:flip-lambda-alt}}{=} ~
  \scanPIC{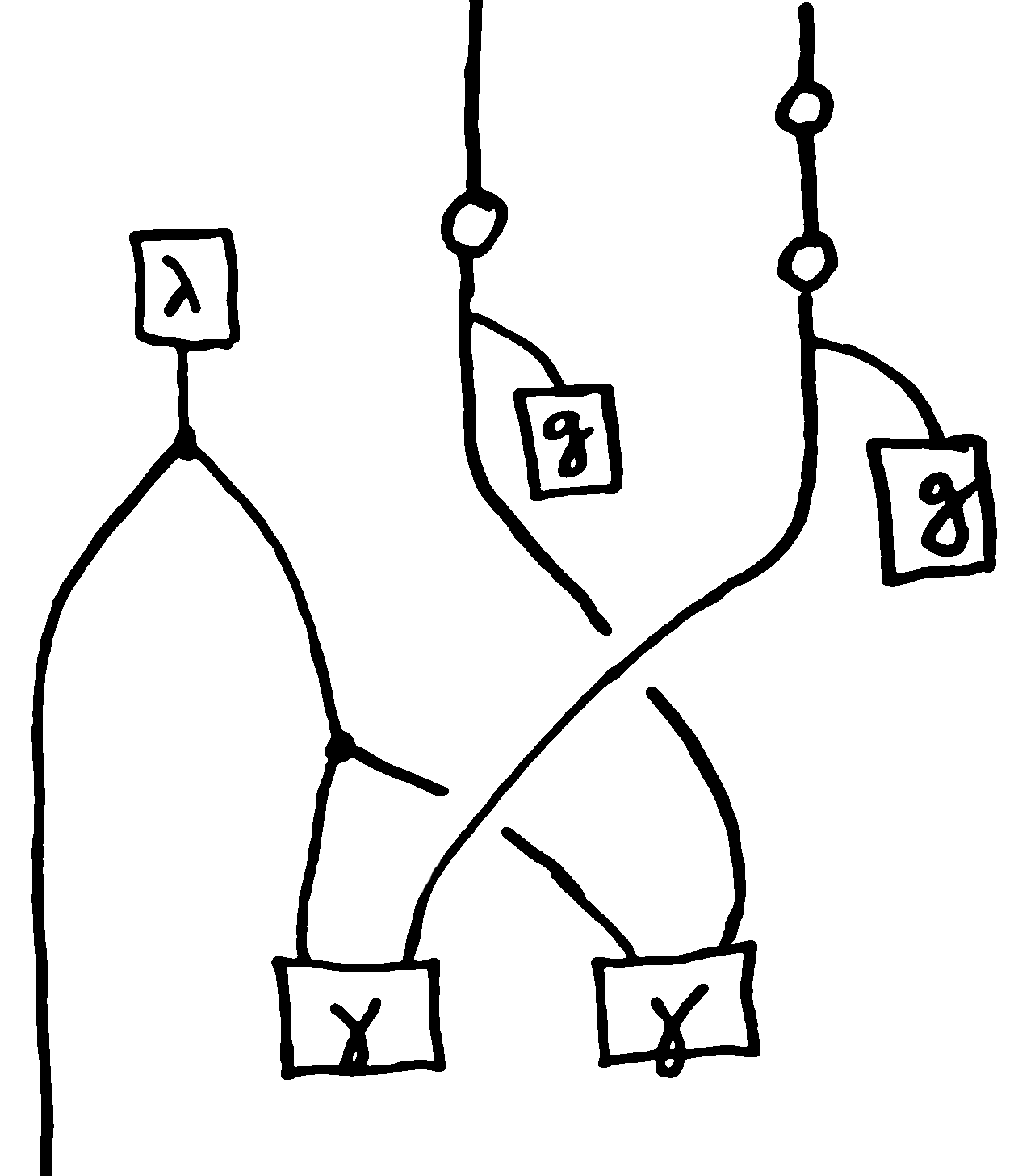}
  ~\overset{\eqref{eq:pentagon-4}}{=} ~
  \scanPIC{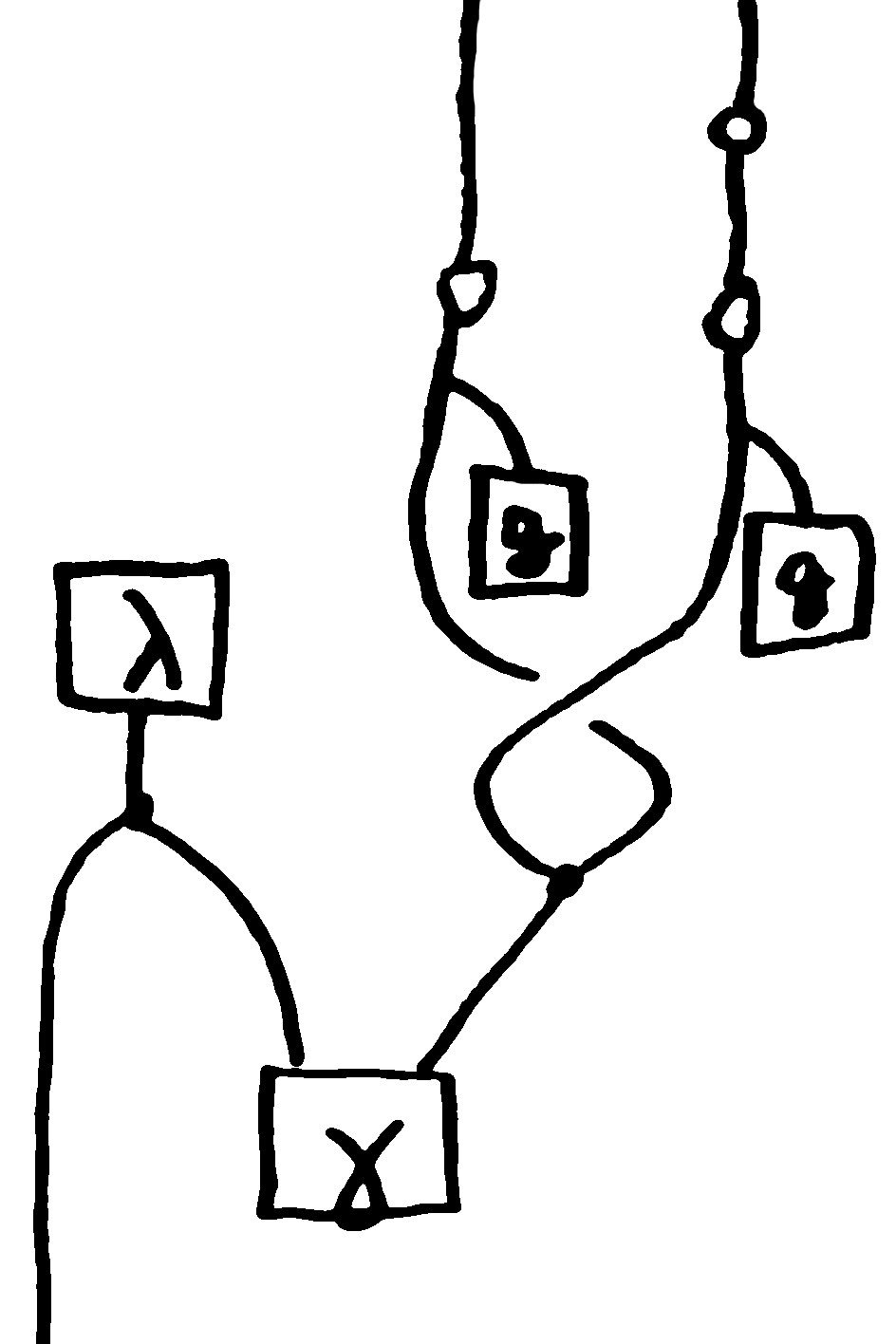}
 \nonumber \\
  &\overset{(*)}{=} ~
  \scanPIC{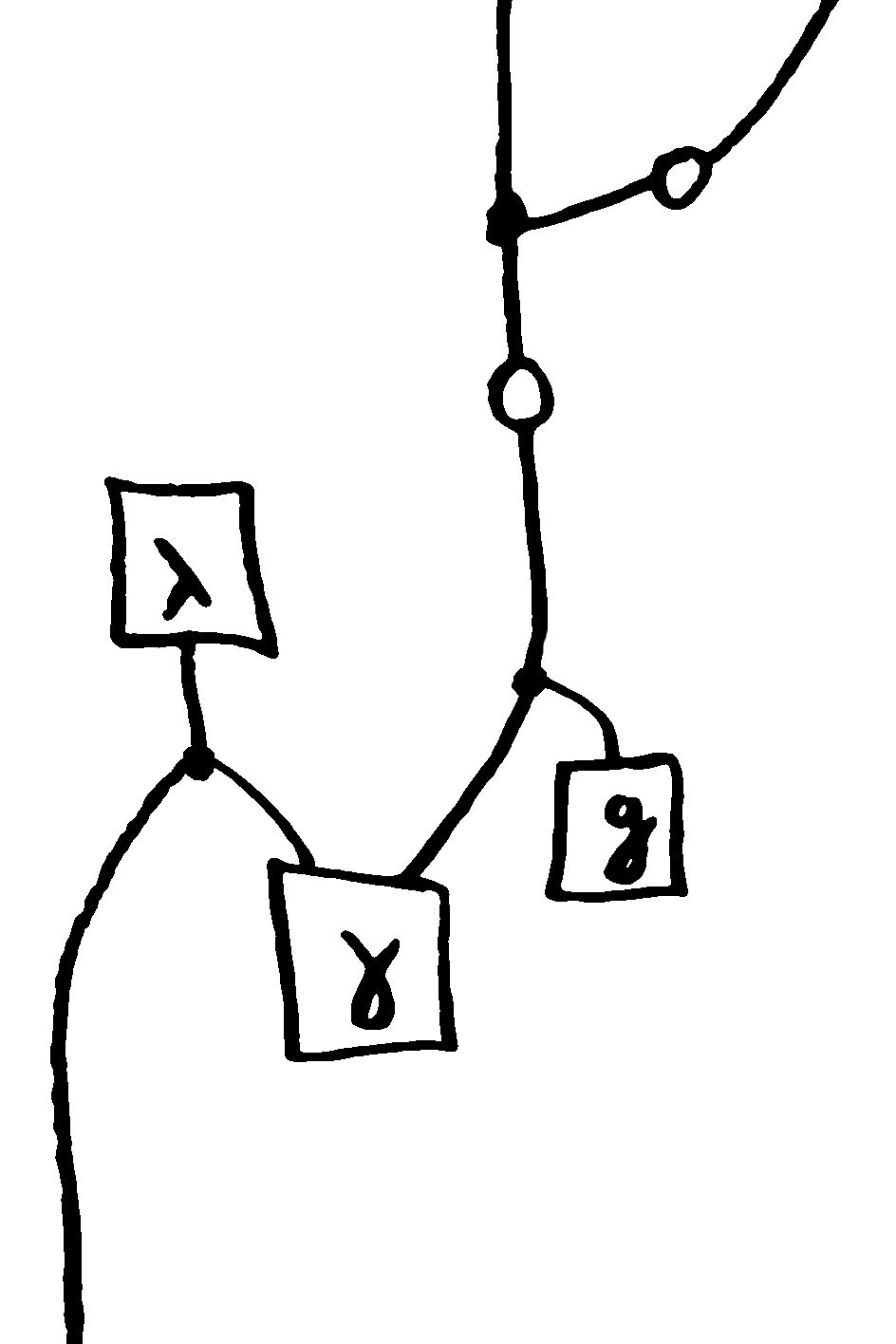}
  ~\overset{\eqref{eq:flip-lambda-alt}}{=} ~
  \text{lhs of \eqref{eq:pentagon-14-rewrite}} \ .
\end{align}
In (*) we used that $M_g$ is a coalgebra map (see below Lemma \ref{lem:g-grouplike-2}) and also moved $S$ past the coproduct. 

Finally, we show that case 16 of the pentagon holds. By Lemma \ref{lem:H-intertwiner}, both sides of \eqref{eq:pentagon-16} commute with the left $H$-action on the left factor of $H \otimes H$. It is thus enough to show \eqref{eq:pentagon-16} after applying $\eta \otimes \id$ from the right, that is,
\be\label{eq:pentagon-16-rewrite}
  \text{\eqref{eq:pentagon-16}} \qquad \Leftrightarrow \qquad
  \scanPIC{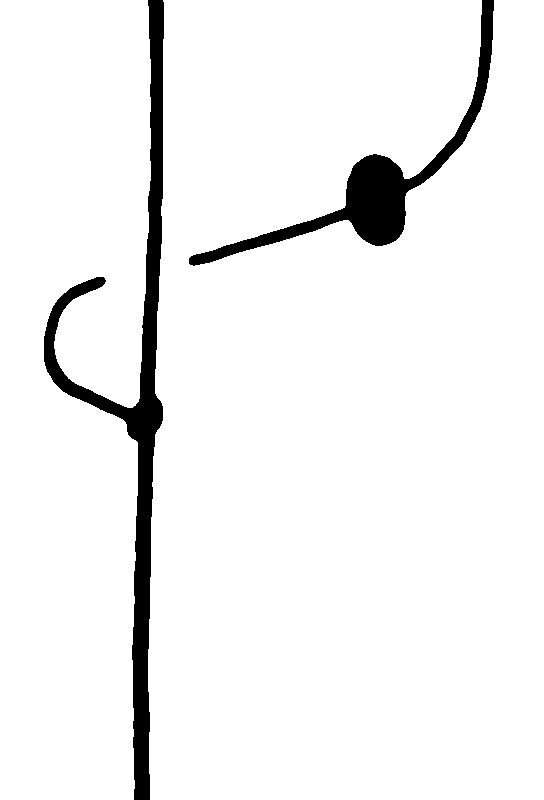}
  ~=~ 
  \scanPIC{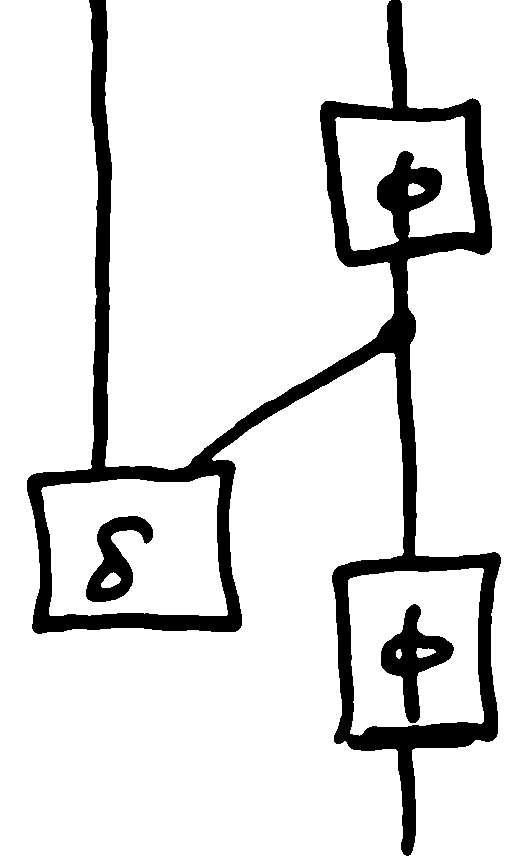} 
  \quad .
\ee
We rewrite the bottom part of the right hand side as
\be \label{eq:case16-aux1}
  \scanPIC{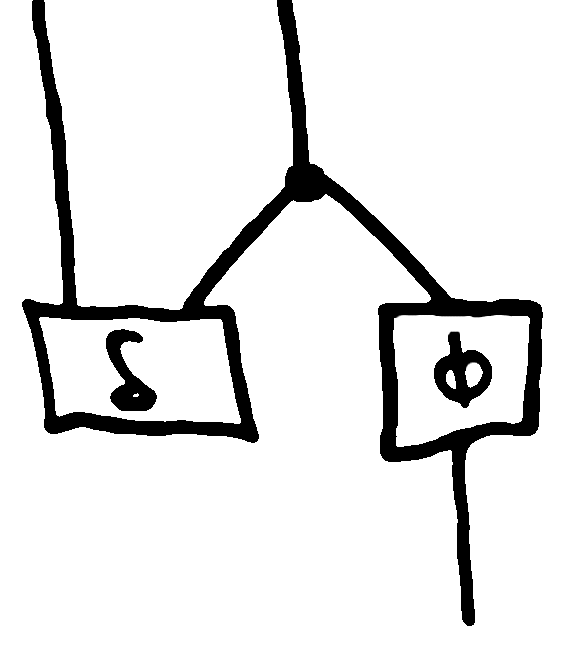}
  ~\overset{\text{(\ref{eq:delta-expression-ab}\,b)}}{=} ~
  \scanPIC{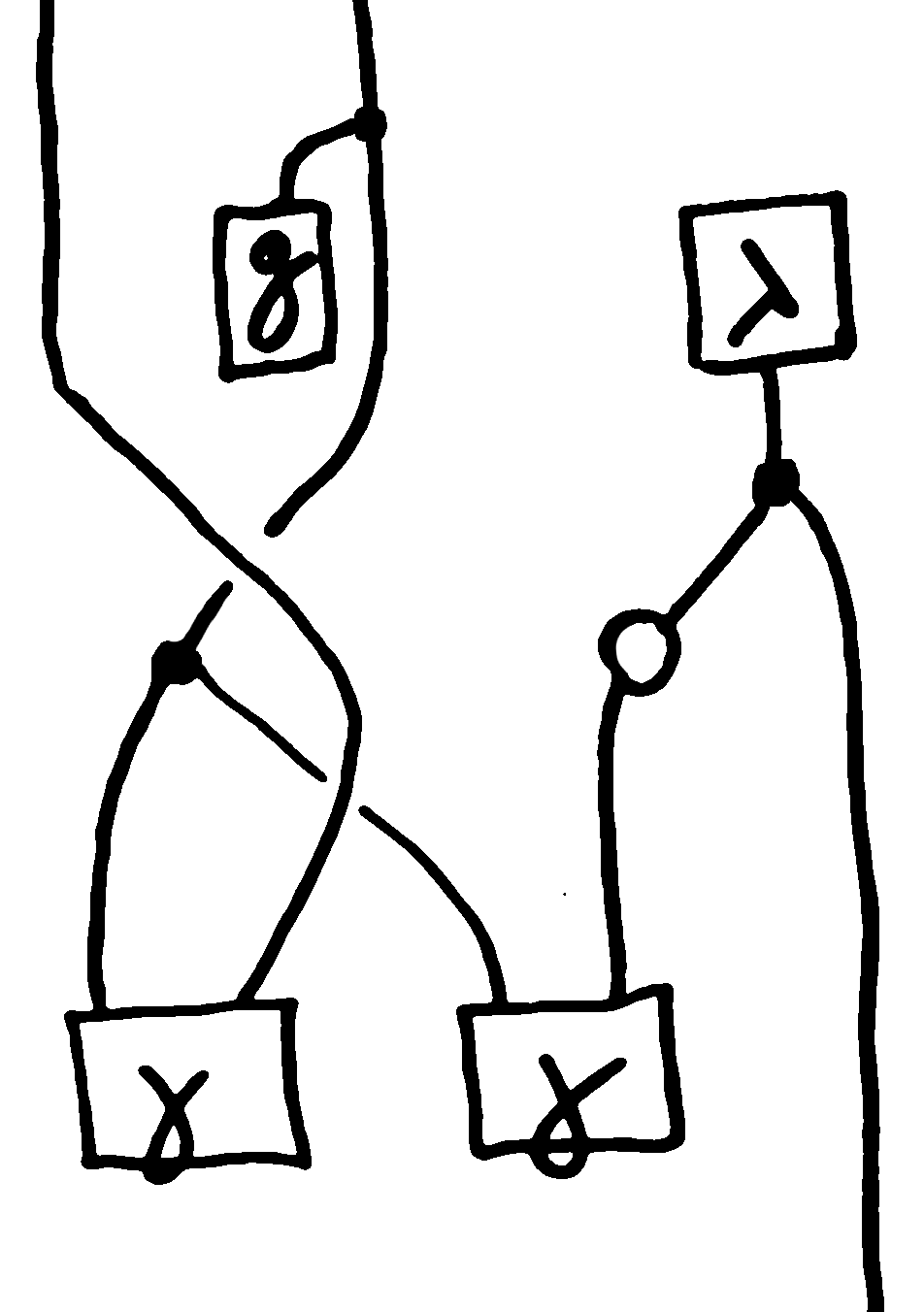}
  ~\overset{\eqref{eq:pentagon-4}}{=} ~
  \scanPIC{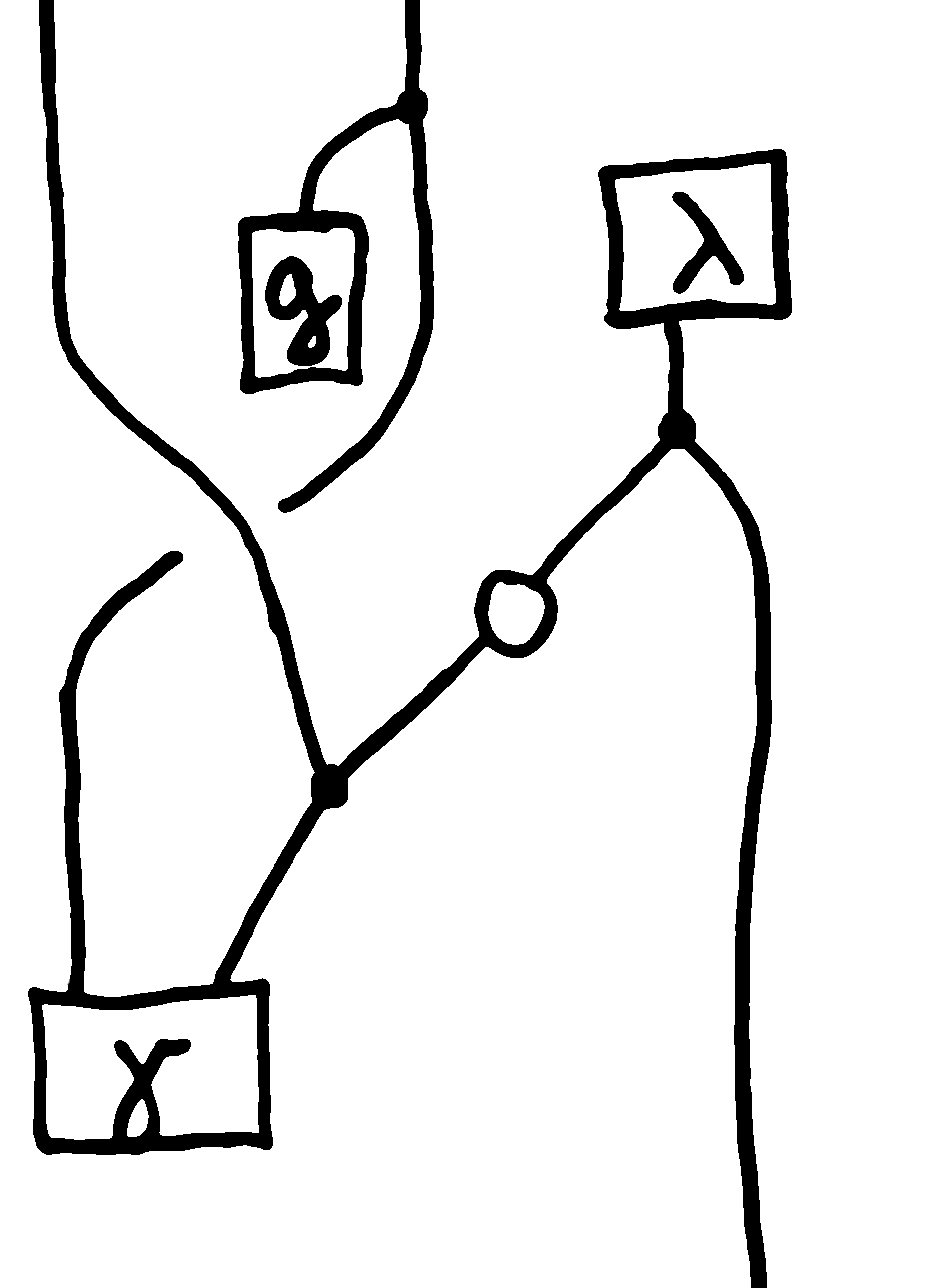}
  \quad .
\ee
For the second morphism $\phi$ in \eqref{eq:pentagon-16-rewrite} we will use the expression
\be\label{eq:case16-aux2}
  \phi ~=~ 
  \scanPIC{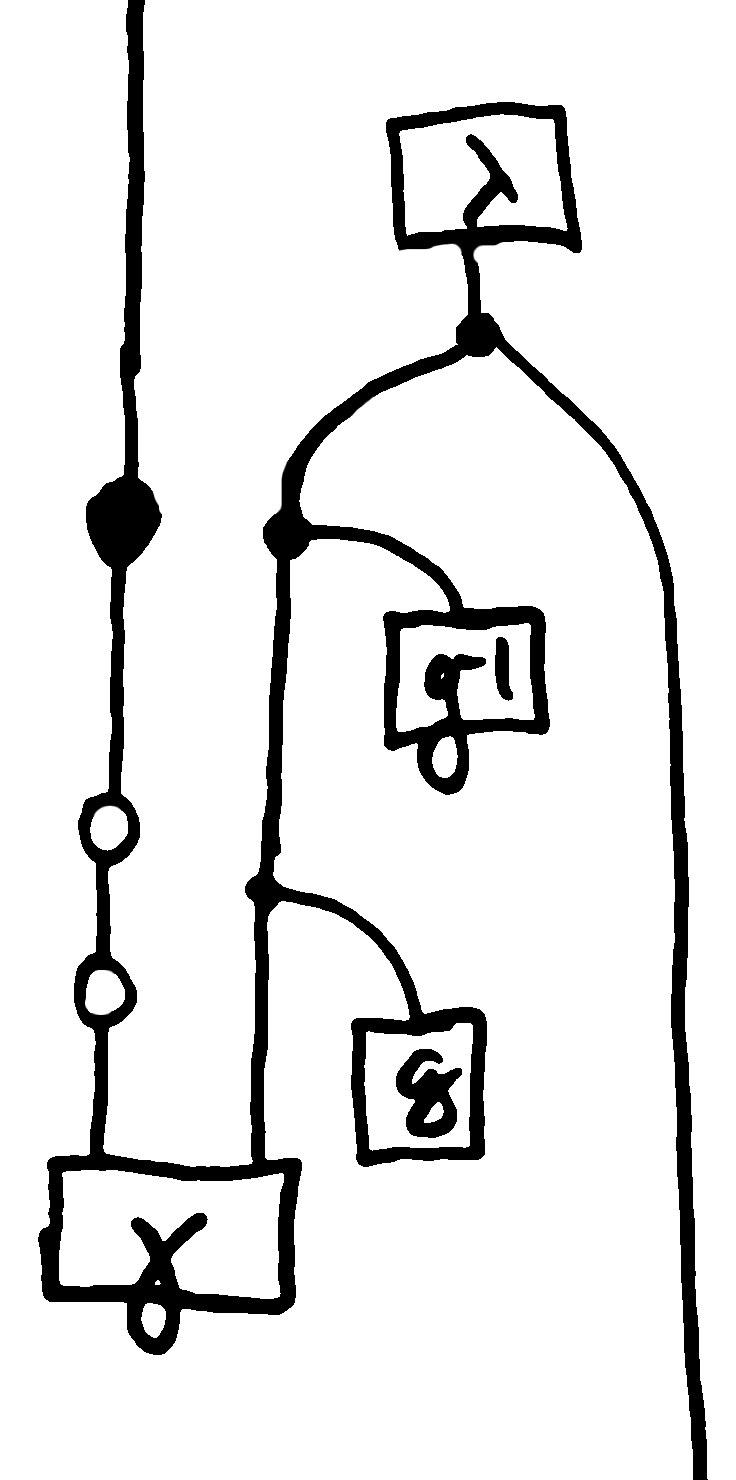}
  ~\overset{\eqref{eq:delta-expression-ab}}{=} ~
  \scanPIC{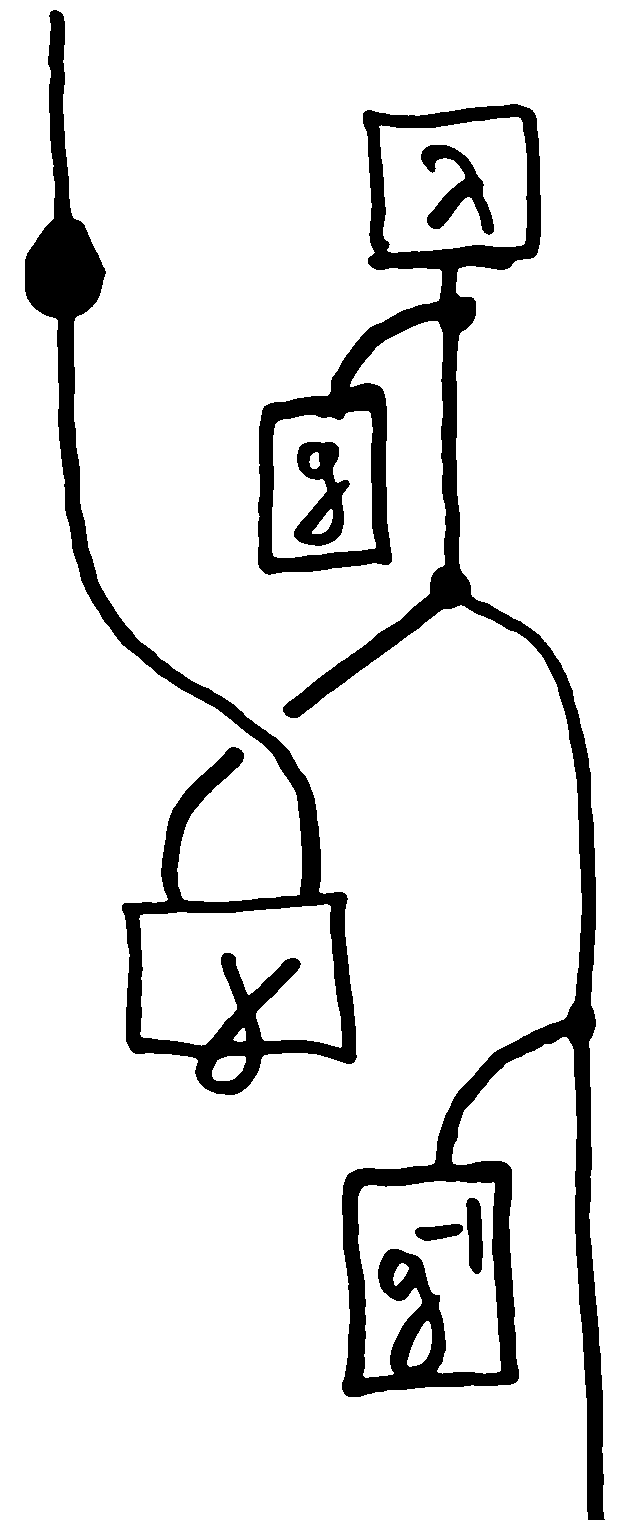}
  \quad .
\ee
Inserting \eqref{eq:case16-aux1} and \eqref{eq:case16-aux2} into the right hand side of \eqref{eq:pentagon-16-rewrite} gives
\begin{align}
&\text{rhs of \eqref{eq:pentagon-16-rewrite}} 
  ~=~ 
  \scanpic{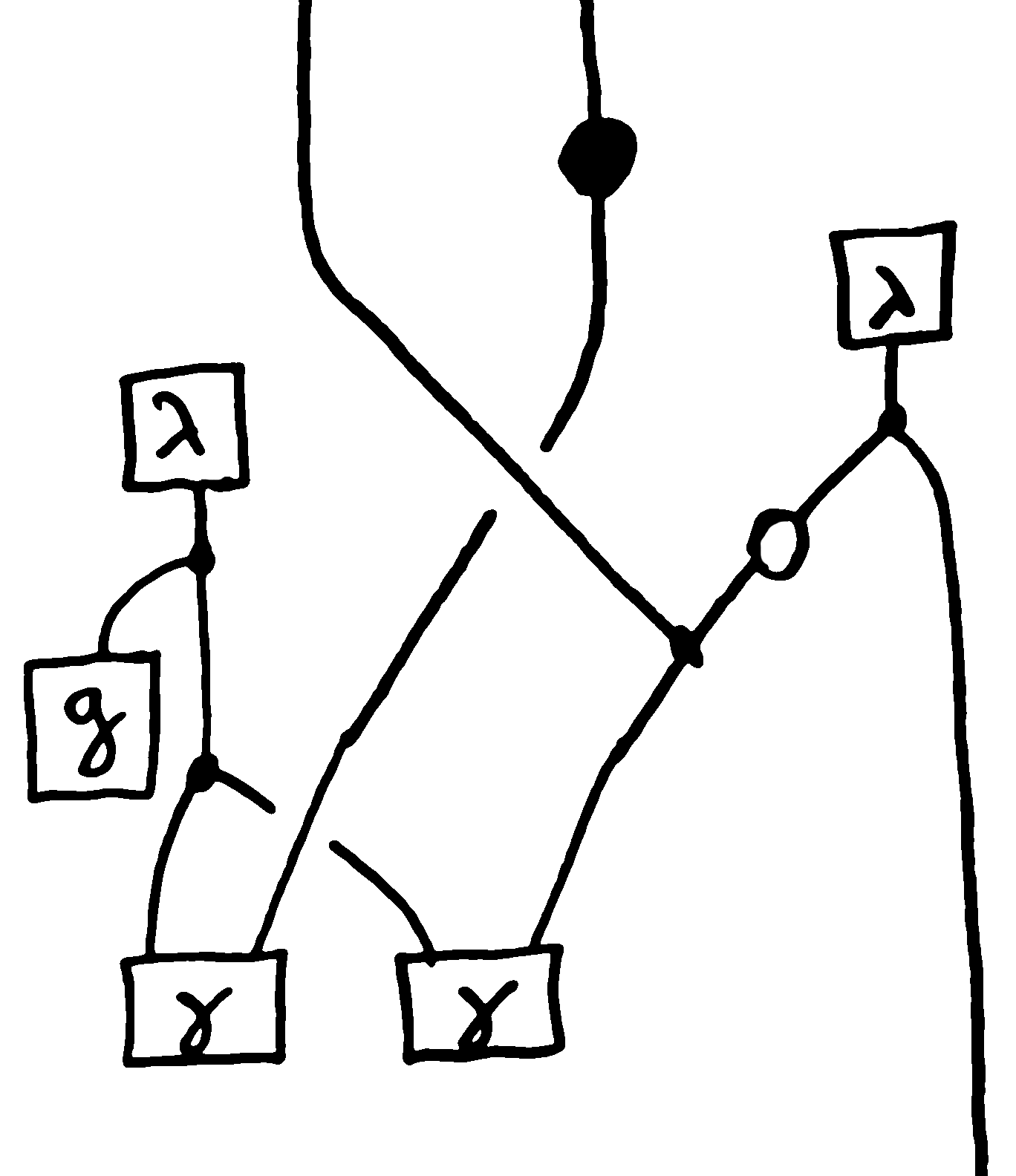}
  ~\overset{\eqref{eq:pentagon-4}}{=} ~
  \scanpic{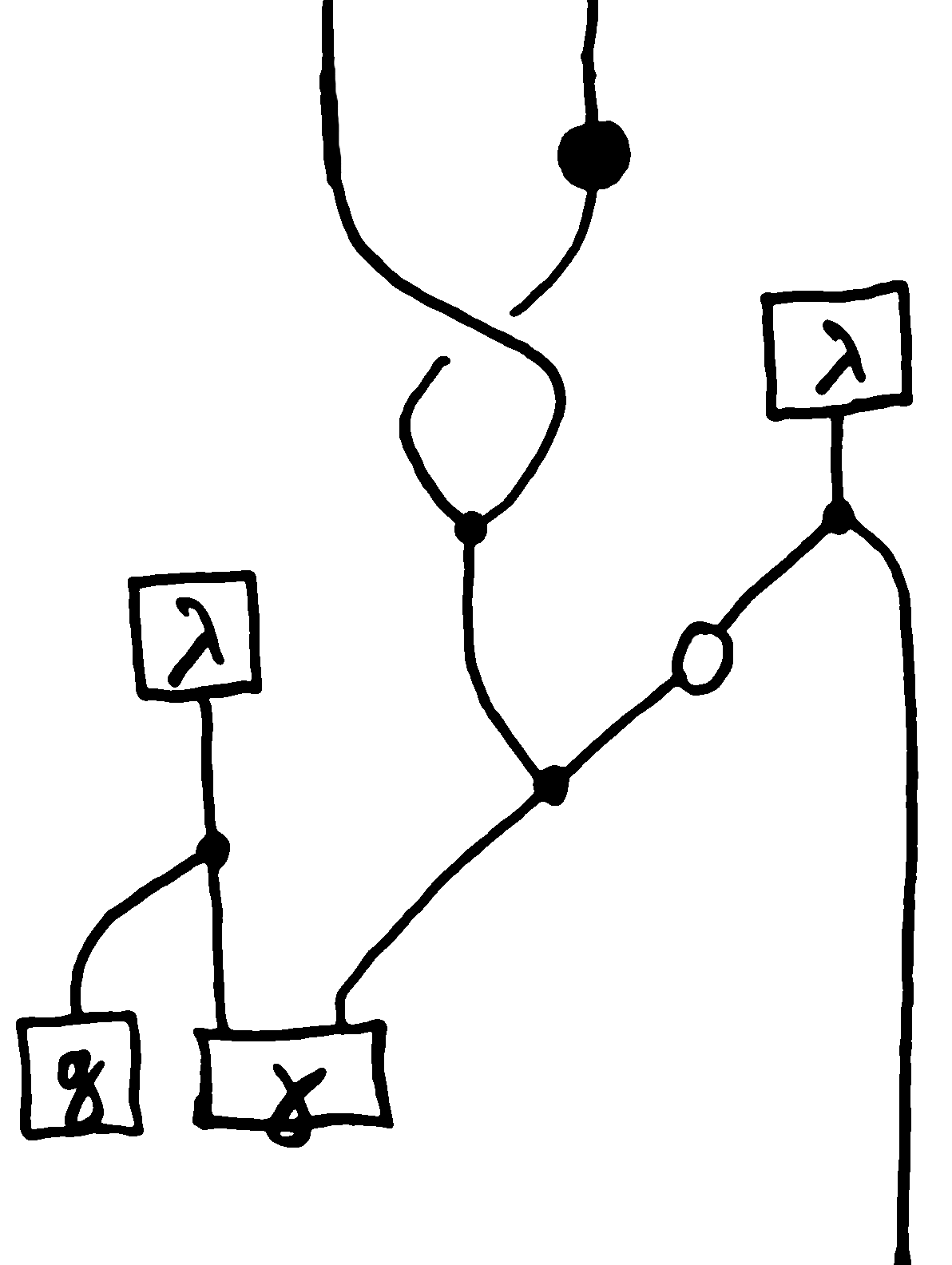}
  ~=~ 
  \scanpic{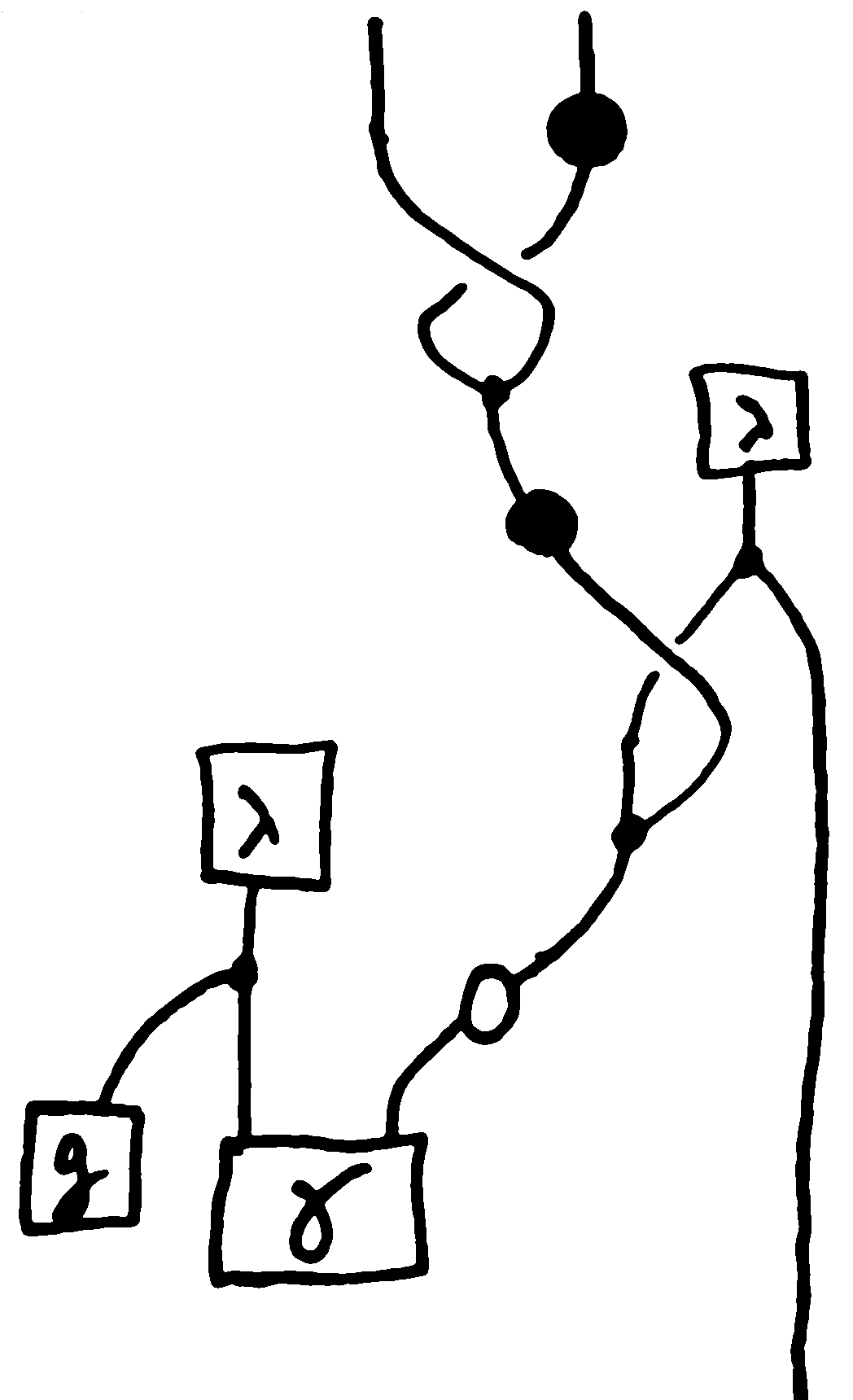}
 \nonumber \\
  &\overset{ \text{Lem.\,\ref{lem:move-coprod-past-lam}\,b)}}{=} ~
  \scanpic{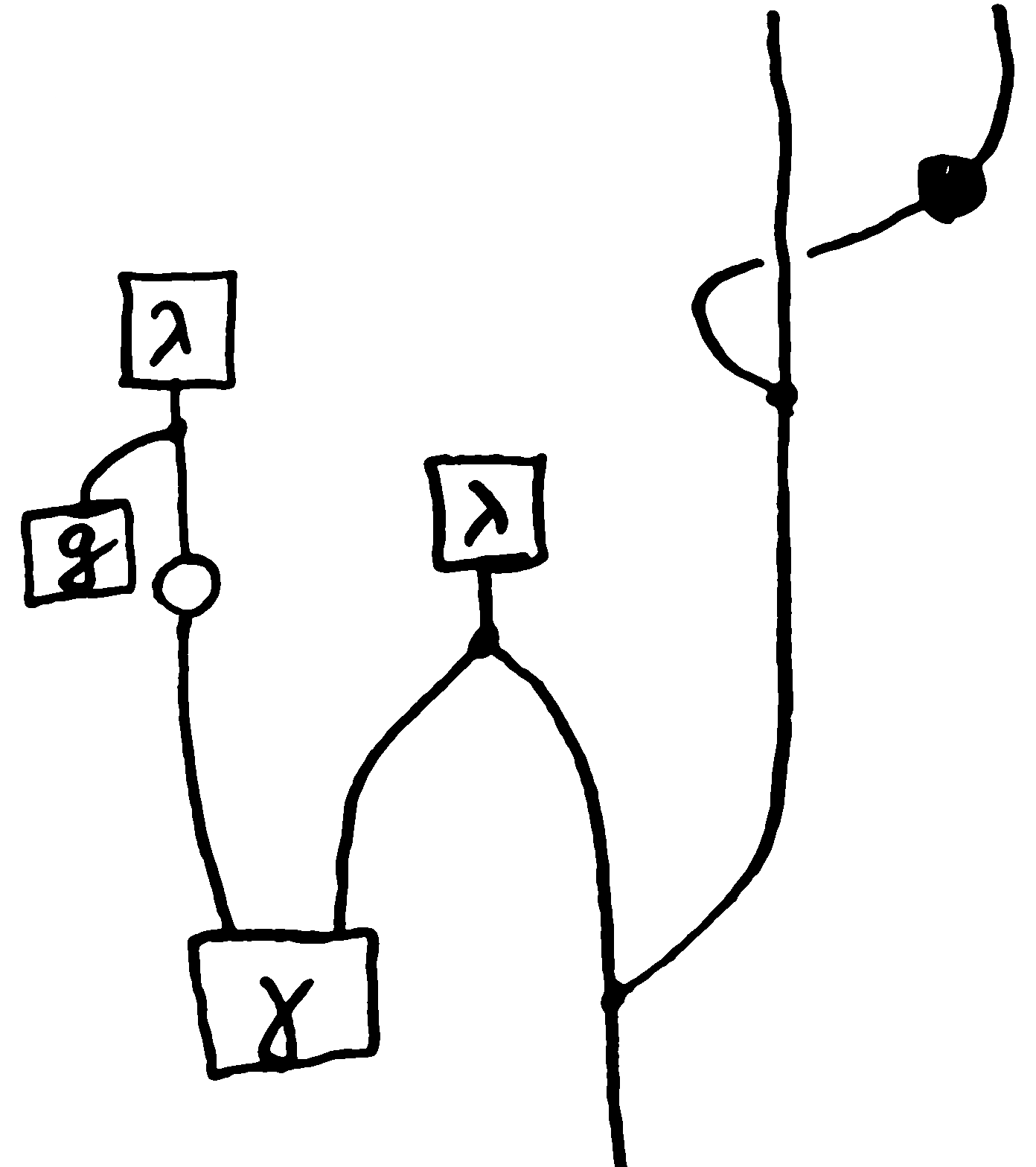}
  ~\overset{\eqref{eq:flip-lambda-alt}}{=} ~
  \scanpic{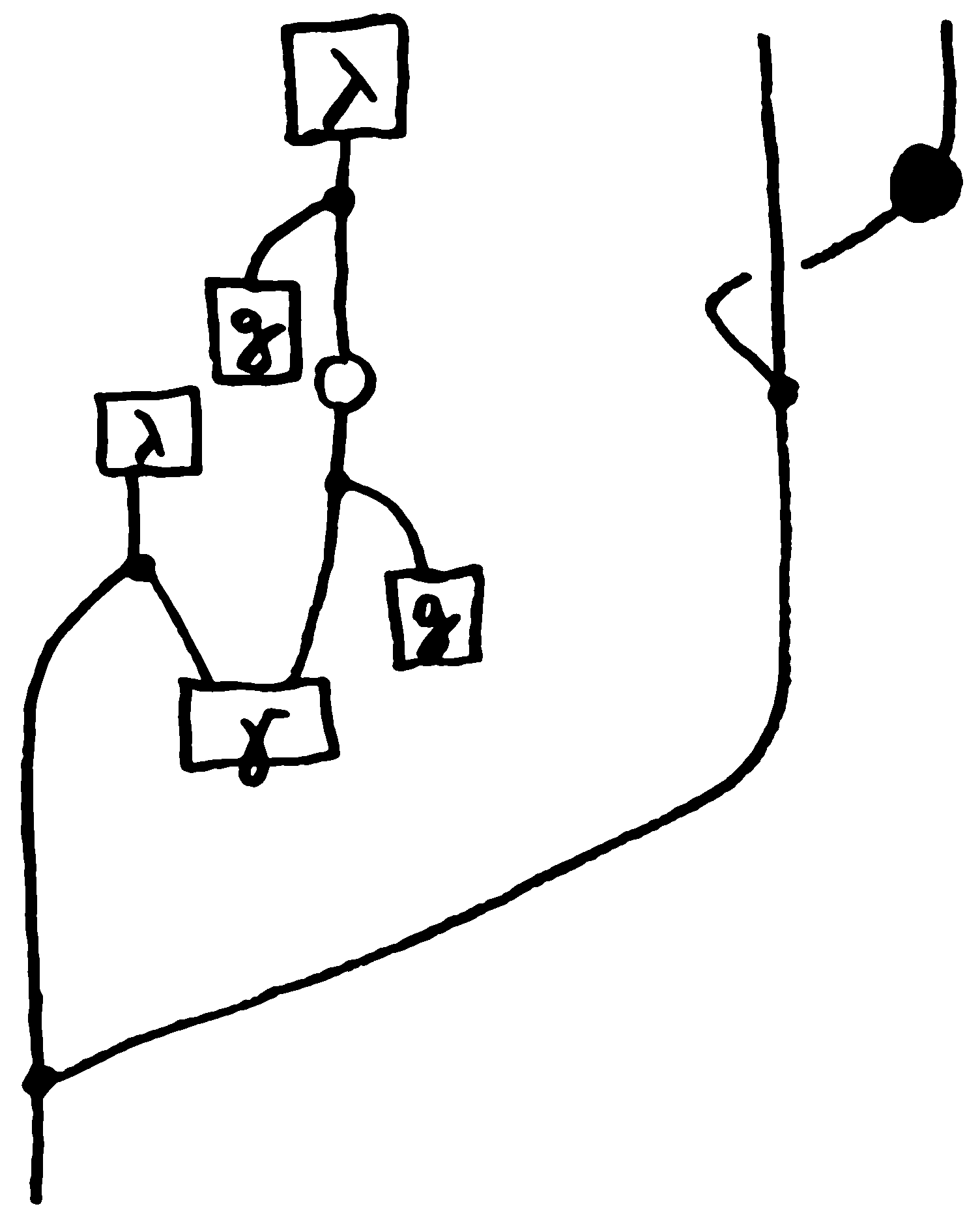}
  ~\overset{\eqref{eq:gdp-via-Hopf}}{=} ~
  \scanpic{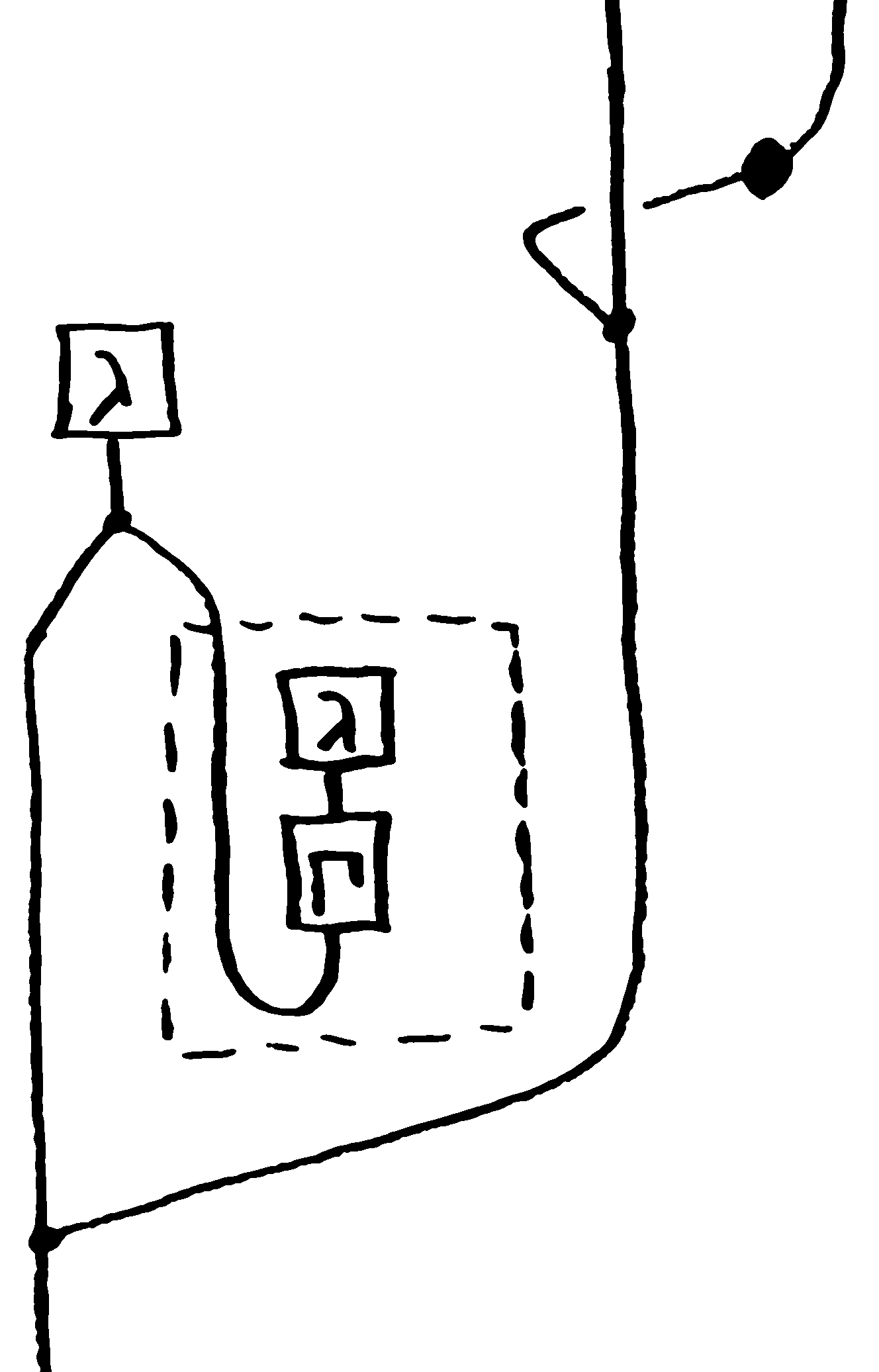}
 \nonumber \\
  &\overset{ \text{Lem.\,\ref{lem:integral-cointegral}}}{=} ~
  \scanpic{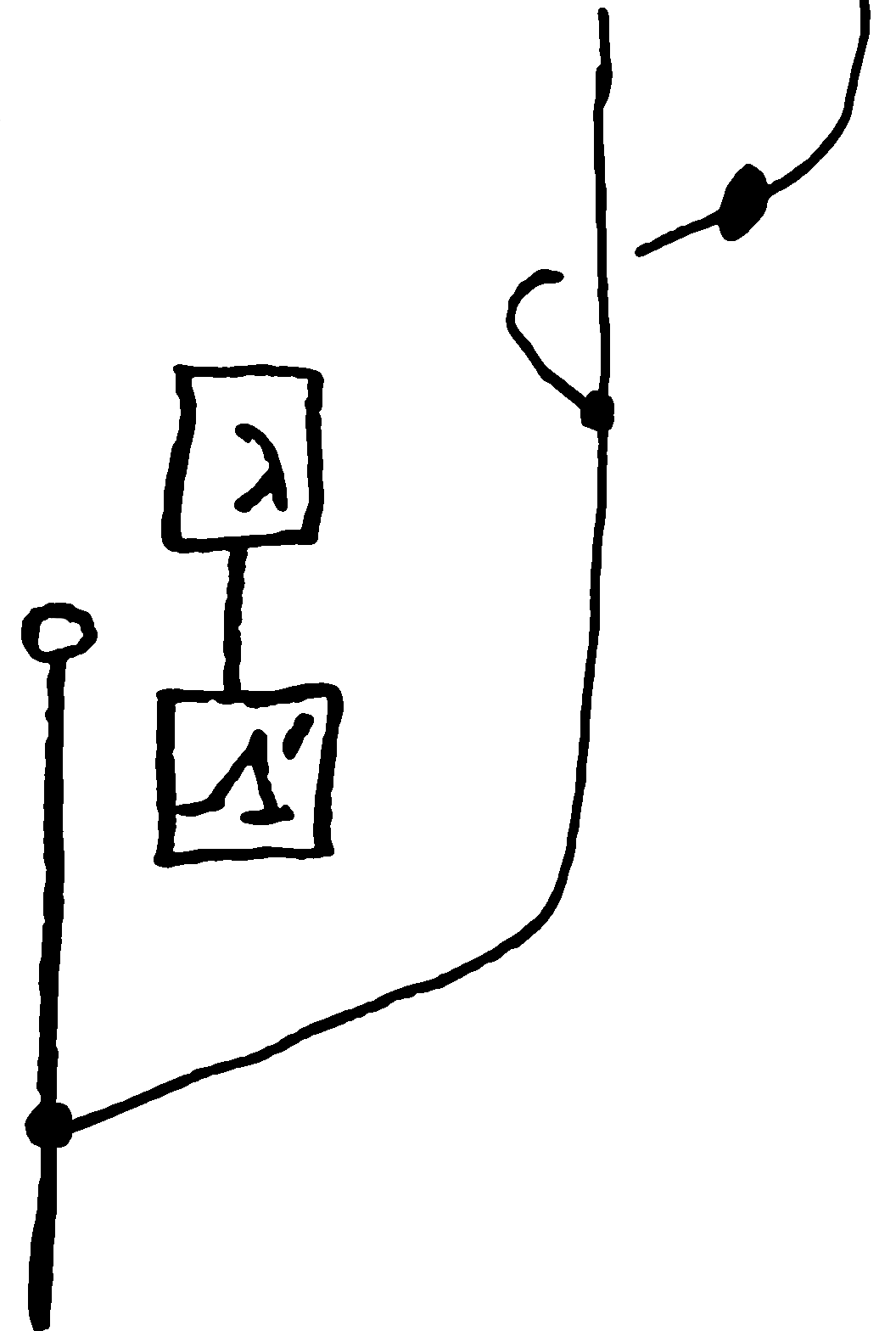}
  ~=~ 
  \text{lhs of \eqref{eq:pentagon-16-rewrite}} \ .
\end{align}
In the last equality the normalisation condition in property b) of Theorem \ref{thm:main1}\,(i) is used.

This completes the proof of the proposition.
\end{proof}

In practice it will be easier to avoid working with $\Gamma$ and the dual Hopf algebra $H^\vee$ and to use $\gamma : \one \to H \otimes H$ instead, cf.\,\eqref{eq:Hopf-via-gdp}. To be explicit, we reformulate Theorem \ref{thm:main1}\,(i) via $\gamma$ and denote the resulting monoidal category by $\Cc(H,\gamma,\lambda)$ instead of $\Cc(H,\Gamma,\lambda)$.

\begin{corollary} \label{cor:main1-i}
Let $\gamma : \one \to H \otimes H$ and $\lambda : H \to \one$ be two morphisms such that
\begin{itemize}
\item[a)] 
$\gamma$ is non-degenerate and a Hopf-copairing, that is, it satisfies
$(\id \otimes \eps) \circ \gamma = \eta = (\eps \otimes \id) \circ \gamma$, 
$(\id \otimes S) \circ \gamma = (S \otimes \id) \circ \gamma$ and 
\eqref{eq:pentagon-3}, \eqref{eq:pentagon-4},
\item[b)] 
$\lambda$ is a right cointegral for $H$, such that there exists $g : \one \to H$ with $(\id \otimes \lambda) \circ \Delta_H = g \circ \lambda$, and such that $(\lambda \otimes \lambda) \circ (\id \otimes S) \circ \gamma = \id_\one$,
\item[c)] 
$\gamma$ 
satisfies $c_{H,H}^{-1} \circ \gamma = (\id \otimes (S^2 \circ \Ad_g^{-1})) \circ \gamma$.
\end{itemize}
Then the family of natural isomorphism $\alpha_{A,B,C} : A \otimes_\Cc (B \otimes_\Cc C) \to (A \otimes_\Cc B) \otimes_\Cc C$ given in  Section \ref{results} with
$$
  \delta = (\id \otimes (M_g \circ S^2)) \circ \gamma
  \quad , \quad
  \phi = (\id \otimes \lambda) \circ (S \otimes \mu) \circ (\gamma \otimes \id)
$$
are associativity isomorphisms for $\otimes_\Cc$. 
\end{corollary}

\begin{proof}
Substituting the expression for $\gamma$ in terms of $\Gamma$ in \eqref{eq:gdp-via-Hopf} into a)--c) above, and conversely inserting the expression for $\Gamma$ in terms of $\gamma$ in \eqref{eq:Hopf-via-gdp} into Theorem \ref{thm:main1}\,(i) a)--c) immediately shows the equivalence of the conditions (for \eqref{eq:pentagon-3} and \eqref{eq:pentagon-4} see \eqref{eq:Hopf-iso-P3+4}, and to compare condition c) use \eqref{eq:Gamma-dual-via-gamma}).
\end{proof}

\subsection{Inverse associativity isomorphisms}

For later use in Section \ref{sec:braiding} we record the inverses of the five non-trivial associativity isomorphisms in  Section \ref{results}. They are
\begin{align}
&  \alpha^{-1}_{A^0,B^1,C^0} ~=~ \scanPIC{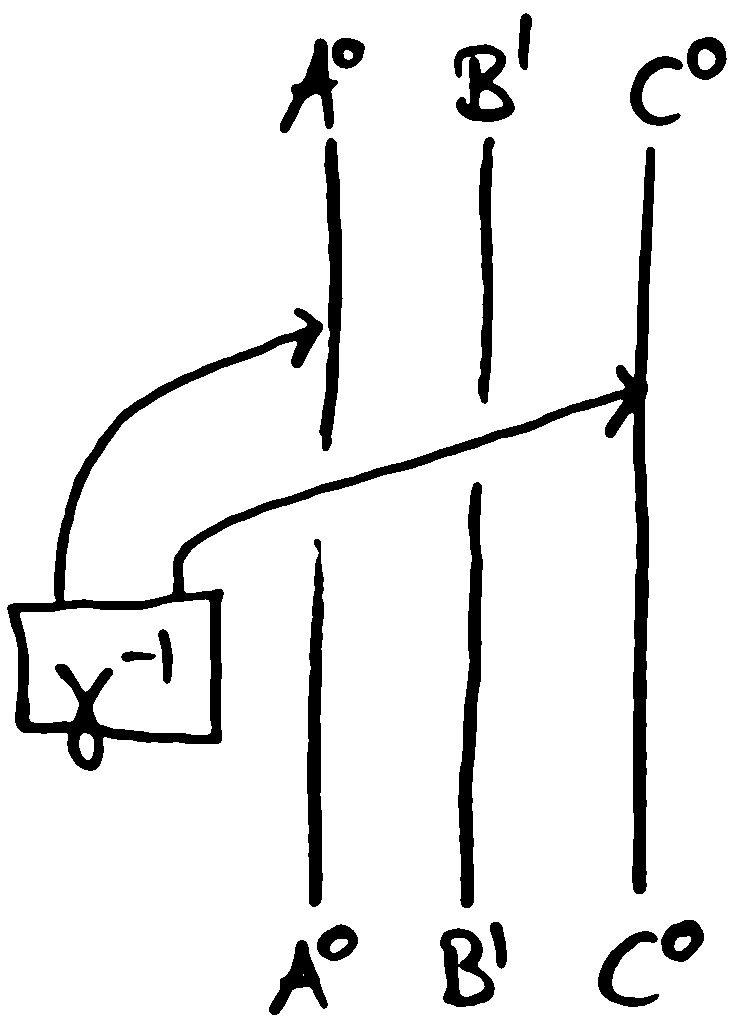}
\qquad , \qquad
  \alpha^{-1}_{A^1,B^1,C^1} ~=~   \scanPIC{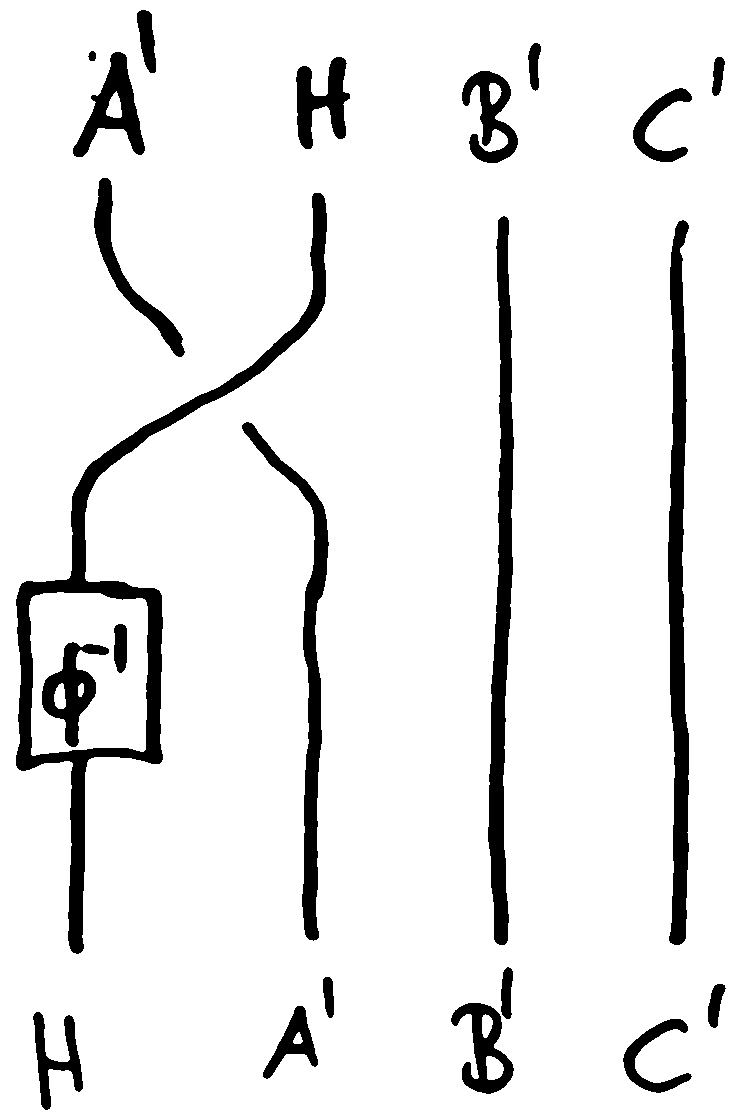}
\nonumber\\
&  \alpha^{-1}_{A^0,B^1,C^1} ~=~   \scanPIC{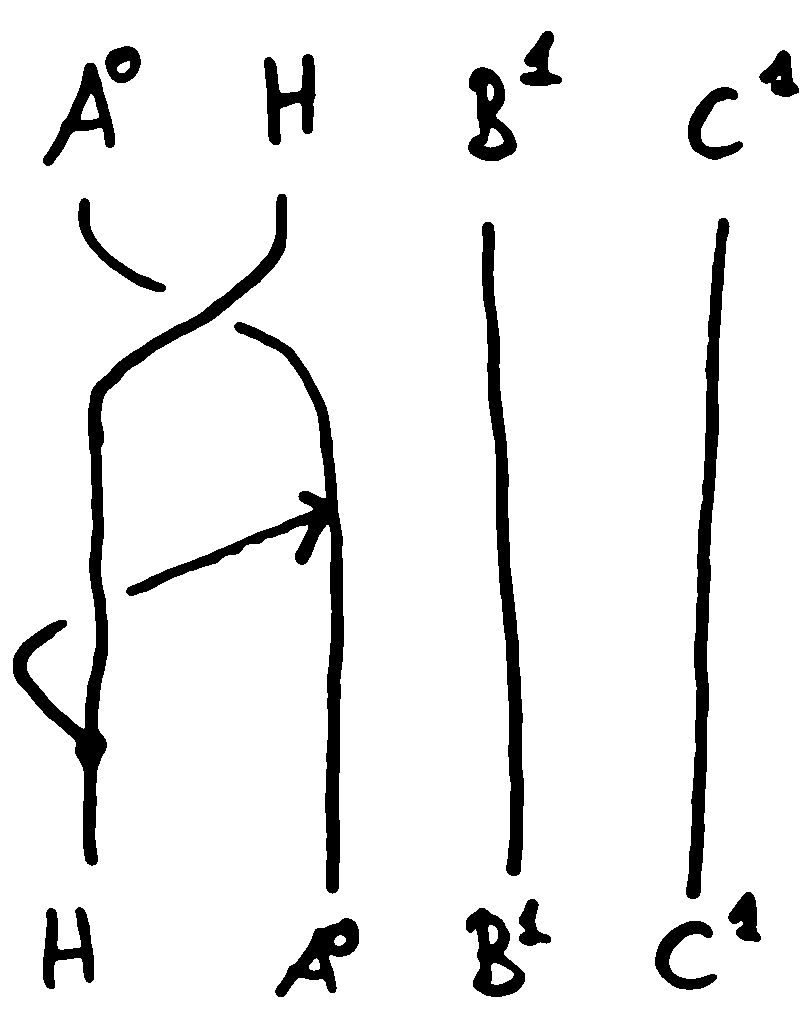}
  \quad , \quad
  \alpha^{-1}_{A^1,B^0,C^1} ~=~   \scanPIC{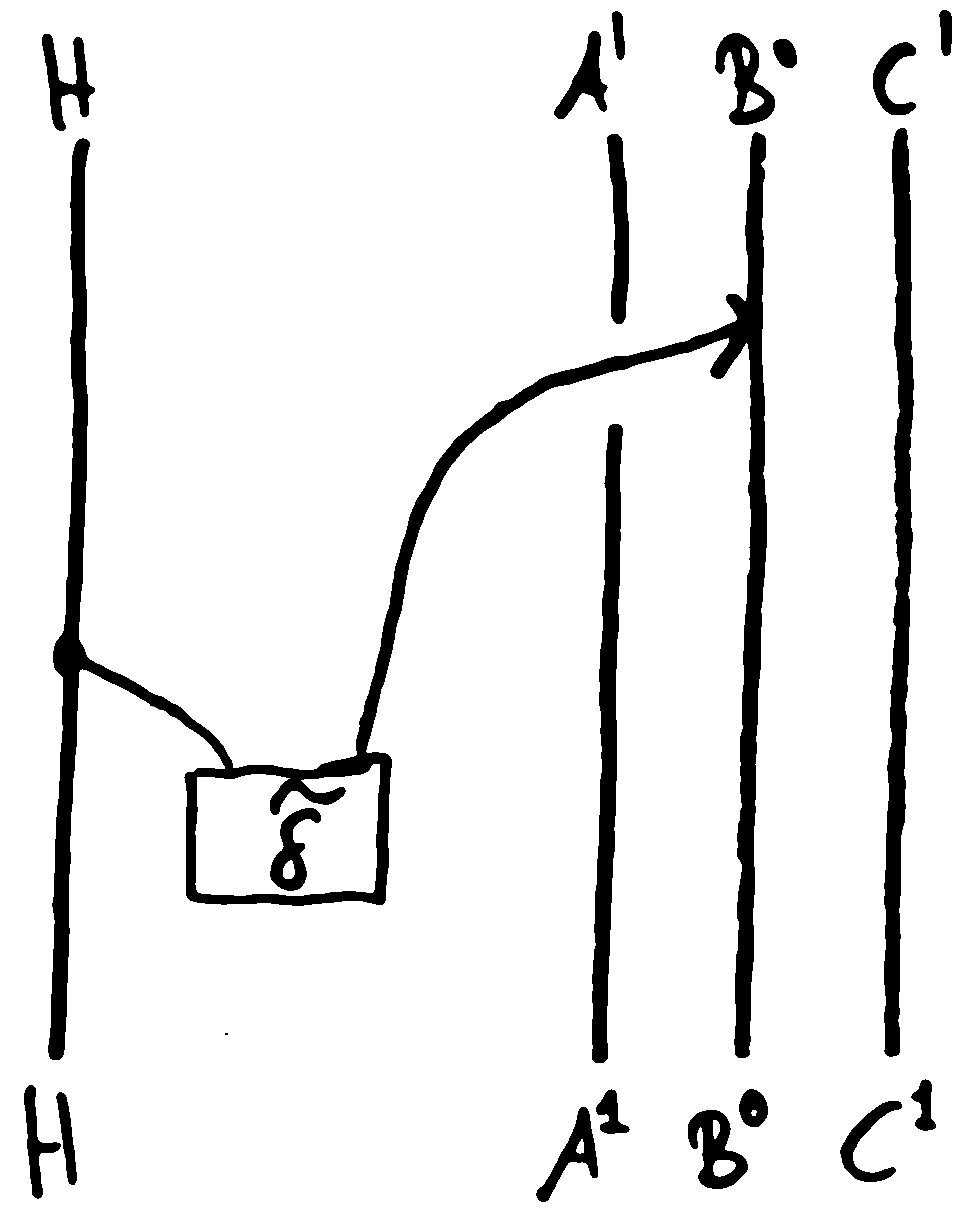}
  \quad , \quad
  \alpha^{-1}_{A^1,B^1,C^0} ~=~   \scanPIC{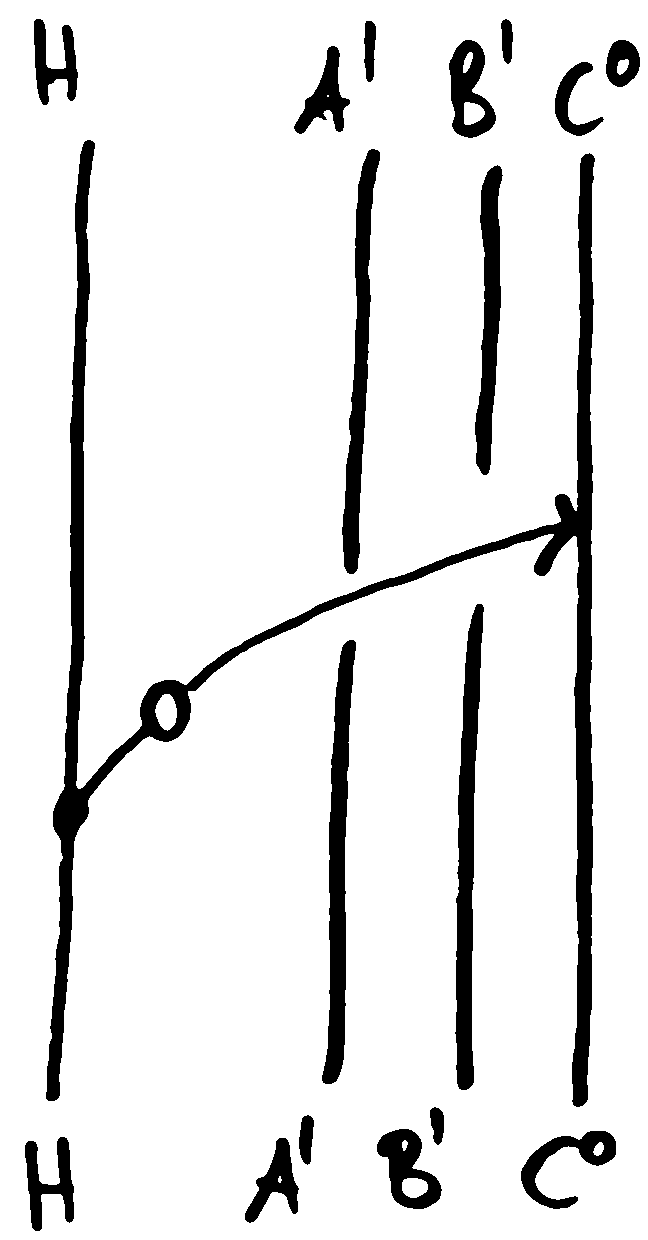}
  \quad,
\label{eq:inverse-assoc}  
\end{align}
where
\be \label{eq:tilde-delta-def}
  \tilde\delta = (\id \otimes {}_{g^{-1}}M) \circ \gamma^{-1} \ .
\ee
To see that the expression for $\tilde\delta$ indeed gives the inverse associator, substitute $\delta = (S \otimes M_g) \circ \gamma^{-1}$ (Lemma \ref{lem:delta+phi-via-gamma}) into $\alpha_{A^1,B^0,C^1}$ in  Section \ref{results} and rewrite $\tilde\delta = (S \otimes {}_{g^{-1}}M) \circ \gamma$ (use Lemma \ref{lem:gamma-inv-via-S}\,(ii)). Upon composing $\alpha_{A^1,B^0,C^1}$ and $\alpha^{-1}_{A^1,B^0,C^1}$, via Figure \ref{fig:Hopf-property}\,e) one arrives at the product of $\gamma$ and $\gamma^{-1}$ and at that of $g$ and $g^{-1}$, both of which cancel.

That $\alpha^{-1}_{A^0,B^1,C^1}$ and $\alpha^{-1}_{A^1,B^1,C^0}$ are inverses as claimed is immediate from Lemma \ref{lem:H-intertwiner}.

The inverse of $\phi$ has two useful expressions provided by \eqref{eq:pentagon-16}. Namely, inserting \eqref{eq:pentagon-16} into $(\eps \otimes \id) \circ (-) \circ (\eta \otimes \id)$ gives $S^{-1} = \phi \circ {}_gM \circ \phi$, or, equivalently,
\be\label{eq:phi-1_via_phi}
  \phi^{-1} = {}_gM \circ \phi \circ S = S \circ \phi \circ {}_gM \ .
\ee
Using these expressions, we can show the next lemma which will be useful in Section \ref{sec:braiding}.

\begin{lemma} \label{lem:lambda-S-g}
We have $\lambda \circ S = \lambda \circ {}_gM$ and $\lambda \circ S^{-1} = \lambda \circ M_g$.
\end{lemma}

\begin{proof}
For the first identity compose the two equivalent expressions for $\phi^{-1}$ in \eqref{eq:phi-1_via_phi} with $\eps$ from the left and use that $\lambda = \eps \circ \phi$. The second identity follows by composing the first one with $S^{-1} \circ M_g$ from the right.
\end{proof}

\subsection{Monoidal equivalences}\label{sec:gauge}

In this section we prove Theorem \ref{thm:main1}\,(ii).
Let $H,\gamma,\lambda$ be a data satisfying the conditions of Theorem \ref{thm:main1}\,(i) (or rather of Corollary \ref{cor:main1-i}) and let  $\Cc(H,\gamma,\lambda)$ be
 the corresponding $\Zb/2\Zb$-graded monoidal category $\Rep_\Sc(H) + \Sc$. 
Let $\varphi : H' \to H$ be a Hopf algebra isomorphism. It allows us to transport the data $\gamma,\lambda$ from $H$ to $H'$:
\be \label{eq:mon-gauge-xfer}
  \gamma' = (\varphi^{-1} \ot\varphi^{-1}) \circ\gamma \quad ,\qquad \lambda' = \lambda\circ \varphi \ .
\ee
Clearly the new data $\gamma',\lambda'$ also satisfies to the conditions of Corollary \ref{cor:main1-i}. 
Define the $\Zb/2\Zb$-graded functor
\be\lb{G}
  G=G(\varphi)~:~ \Cc=\Cc(H,\gamma,\lambda)\longrightarrow \Cc'= \Cc(H',\gamma',\lambda') 
\ee
to be the pullback functor $\varphi^*: \Rep_\Sc(H)\to \Rep_\Sc(H')$ on the component of degree zero and the identity functor on the component of degree one. Clearly $G(\varphi)$ is an equivalence. Now we will describe a monoidal structure $G(\varphi)_{A,B}:G(A \ot_\Cc B)\to G(A)\ot_{\Cc'} G(B)$ on $G(\varphi)$. We have four cases depending on whether $A,B$ belong to the degree zero or degree one components. Of the four cases the only non-trivial (non-identical) part of the monoidal structure isomorphism is when $A,B\in\Cc_1$. In this case 
\be\lb{Gts}
  \xymatrix{G(\varphi)_{A,B}:G(A\ot_\Cc B) = H\ot A\ot B\ar[rr]^(.52){\varphi^{-1}\ot id\ot id} && H'\ot A\ot B = G(A)\ot_{\Cc'} G(B)} \ .
\ee
The coherence condition for the monoidal structure is that for all $A,B,C\in\Cc$ the following diagram has to commute:
\be
  \xymatrix{G(A\otimes_\Cc(B\otimes_\Cc C)) \ar[d]_{G(\alpha_{A,B,C})}\ar[rr]^{G_{A,BC}} && G(A)\otimes_{\Cc'} G(B\otimes_\Cc C) \ar[rr]^(.45){id\otimes
G_{B,C}} & & G(A)\otimes_{\Cc'} (G(B)\otimes_{\Cc'} G(C)) \ar[d]_{\alpha'_{G(A),G(B),G(C)}}
\\
G((A\otimes_\Cc B)\otimes_\Cc C) \ar[rr]^{G_{A B,C}} && G(A\otimes_\Cc B)\otimes_{\Cc'} G(B) \ar[rr]^(.45){G_{A,B}\otimes id} & &
(G(A)\otimes_{\Cc'} G(B))\otimes_{\Cc'} G(C) }
\ee
For example, the coherence for $A,C\in\Cc_0$, $B\in\Cc_1$ is equivalent to the equation $\gamma' = (\varphi^{-1}\ot\varphi^{-1})\circ\gamma$. The remaining seven cases are equally straightforward.
Altogether we have the following proposition, which proves Theorem \ref{thm:main1}\,(ii).

\begin{proposition}\label{prop:mon-equiv}
Let $H,\gamma,\lambda$ and $H',\gamma', \lambda'$ be related via a Hopf algebra isomorphism as in \eqref{eq:mon-gauge-xfer}. Then the categories $\Cc(H,\gamma,\lambda)$ and $\Cc(H',\gamma',\lambda')$ are monoidally equivalent.
\end{proposition}
\begin{remark}
It is in general not true that {\em all} monoidal equivalences are of the form stated in this proposition. For example, invertible objects of order two in $\Sc$ give additional monoidal equivalences. Indeed, to $P\in\Sc$ such that $P\ot P\cong I$ one can associate an equivalence $\Cc\to \Cc$ of the form
\be
  Id:\Cc_0\to\Cc_0,\quad\quad P\ot -:\Cc_1\to\Cc_1 \ .
\ee  
The monoidal structure is build using the braiding of $\Sc$ and the isomorphism $P\ot P\cong I$.
\end{remark}

The morphism space $\Sc(\one,H)$ is an algebra with respect to the product $\mu_H\circ(x\ot y)$.
As before, we say that a morphism $x\in\Sc(\one,H)$ is group-like if $\Delta\circ x = x\ot x$. For a group-like $x$, the composition $S \circ x$ is the multiplicative inverse of $x$ (cf.\ Lemma \ref{lem:g-grouplike}\,(ii)). In particular, group-like elements of $\Sc(\one,H)$ are invertible.

\begin{lemma}\lb{auxlem}
Let $\phi$ and $\psi$ be Hopf algebra automorphisms of $H$. An $\Sc$-module monoidal natural transformation $\chi : \phi^*\to\psi^*$ between the pullback functors $\phi^*, \psi^* : \Rep_\Sc(H)\to\Rep_\Sc(H)$ has the form
$$
  \chi_M:\phi^*M\to \psi^*M \quad ,\quad \chi_M= \rho_M \circ (x\ot id_M),\ M\in \Rep_\Sc(H)
$$ 
for a group-like $x\in \Sc(\one,H)$ such that $\mu\circ (\psi\ot x) = \mu\circ (x\ot\phi)$. 
\end{lemma}

\begin{proof}
From naturality of $\chi_M$ with respect to the morphism of $H$-modules $\rho_M : H\ot F(M)\to M$ (here $F:\Rep_\Sc(H)\to\Sc$ is the forgetful functor) we have a commuting square in $\Sc$: 
\be \label{eq:S-nat-xfer-aux1}
\raisebox{2em}{
\xymatrix{H\ot F(M)\ar[rr]^{\chi_H\ot id} \ar[d]^{\rho_M} && H\ot F(M)\ar[d]^{\rho_M} \\ M\ar[rr]^{\chi_M} && M}
}
\ee
The identity $\chi_{H \otimes F(M)} = \chi_H\ot id_{F(M)}$ follows from the requirement that $\chi$ is an $\Sc$-module natural transformation.
Precomposing \eqref{eq:S-nat-xfer-aux1} with $\eta \ot id$ implies that $\chi_M=\rho_M \circ (x\ot id_M)$, where $x = \chi_H \circ \eta$.
That $x$ is group-like follows from monoidality of $\chi$.
The fact that $\chi_M$ is a morphism $\phi^*M\to \psi^*M$ implies the last equation of the lemma.
\end{proof}

There is a canonical embedding $\Sc \subset \Rep_\Sc(H) \subset \Cc(H,\gamma,\lambda)$ given by equipping $U \in \Sc$ with the trivial $H$-action. We denote by
\be
  \Aut_\ot(\Cc(H,\gamma,\lambda)/\Sc)
\ee
the group of isomorphism classes of {\em monoidal autoequivalences of} $\Cc(H,\gamma,\lambda)$ {\em over} $\Sc$. That is, autoequivalences which are monoidally equivalent to the identity functor when restricted to $\Sc$ (together with the trivialisation), up to natural monoidal isomorphisms which are the identity on $\Sc$. (Equivalently, consider monoidal autoequivalences which are also $\Sc$-module functors and divide by $\Sc$-module natural isomorphisms.)

Let $\Aut_{\mathrm{Hopf}}(H,\gamma,\lambda)$ be the group of Hopf algebra automorphisms of $H$ which leave $\gamma$ and $\lambda$ invariant (via \eqref{eq:mon-gauge-xfer}). 

\begin{proposition}\label{gma}
The assignment $\phi\mapsto G(\phi^{-1})$ defines an injective group homomorphism from $\Aut_{\mathrm{Hopf}}(H,\gamma,\lambda)$ to $\Aut_\ot(\Cc(H,\gamma,\lambda)/\Sc)$.
\end{proposition}

\begin{proof}
From the explicit expressions in \eqref{G} and \eqref{Gts} one checks that $G(\phi \circ\psi) \cong 
G(\psi) \circ G(\phi)$  as monoidal functors, establishing that we have a group homomorphism.
We will show that the only possibility for a natural isomorphism $G(\phi)\to G(\psi)$ is the identity.
A natural isomorphism $\chi:G(\phi)\to G(\psi)$ consists of two graded components $G(\phi)_0 = \phi^*\to\psi^* = G(\psi)_0$ and $G(\phi)_1 = Id_\Sc\to Id_\Sc = G(\psi)_1$. By monoidality of $\chi$ the following diagram commutes for all $M\in\Cc_0$ and $X\in\Cc_1$
\be
  \xymatrix{G(\phi)(M\otimes_\Cc X) \ar[rr]^{\chi_{F(M)\ot X}}  \ar[d]^{\id} && G(\psi)(M \otimes_\Cc X) \ar[d]^{\id} \\ G(\phi)(M)\otimes_\Cc G(\phi)(X) \ar[rr]^{F(\chi_{M})\otimes \chi_X} && G(\psi)(M)\otimes_\Cc G(\psi)(X)}
\ee   
Recall that $M \otimes_\Cc X = F(M) \otimes X$, etc., where $F : \Cc_0 \to \Sc$ is the forgetful functor.
By triviality of $\chi$ on $\Sc \subset \Cc_0$ we have $\chi_{F(M) \ot X} = id_{F(M)} \ot \chi_X$. Thus $F(\chi_{M}) \ot \chi_X = id_{F(M)}\ot\chi_X$, which implies $\chi_M=id_M$ (choose $X = \one \in \Sc$). By Lemma \ref{auxlem}, $\chi_M= \rho_M \circ (x\ot id_M)$ for a group-like element $x$ of $\Sc(\one,H)$, and $\phi$ and $\psi$ are related by conjugation with $x$. Since $\chi_M = \id_M$, we have $x = \eta$ and hence $\phi=\psi$.
\end{proof}

By restriction to $\Cc_0$, the map $\phi \mapsto G(\phi^{-1})$ in the above proposition gives a group homomorphism from $\Aut_{\mathrm{Hopf}}(H,\gamma,\lambda)$ to $\Aut_\ot(\Rep_\Sc(H)/\Sc)$. Lemma \ref{auxlem} shows that this restricted map is in general not injective since it has inner automorphisms as its kernel.

\subsection{Rigidity}\label{sec:rigid}

Fix $H$, $\Gamma$, $\lambda$ as in Theorem \ref{thm:main1} and abbreviate $\Cc = \Cc(H,\Gamma,\lambda)$. We will define left and right duals on $\Cc$ and give the corresponding duality morphisms, so that $\Cc$ becomes a rigid category. The category $\Sc$ is ribbon and so in particular rigid with coinciding left and right duals. 
We denote the duality maps of $\Sc$ by, for $U \in \Sc$,
\begin{align}
  \ev_U &: U^\vee \otimes U \to \one 
  \quad , \quad
  &
  \coev_U &: \one \to U \otimes U^\vee \quad ,
  \nonumber\\
  \widetilde\ev_U &: U \otimes U^\vee \to \one
  \quad , \quad
  &
  \widetilde\coev_U &: \one \to U^\vee \otimes U 
  \ .
\end{align}

\medskip\noindent
{\bf Right and left duals in $\mathbf{\Cc_0}$:} Given an $H$-module $M$ with action $\rho_M$, the right dual $M^*$ and left dual ${}^*\!M$ of $M$ is $M^\vee$ as an object together with action morphisms $\rho_{M^*}$ and $\rho_{{}^*\!M}$ given by
\be \label{eq:C0-dual-objects}
   \rho_{M^*} = \scanpic{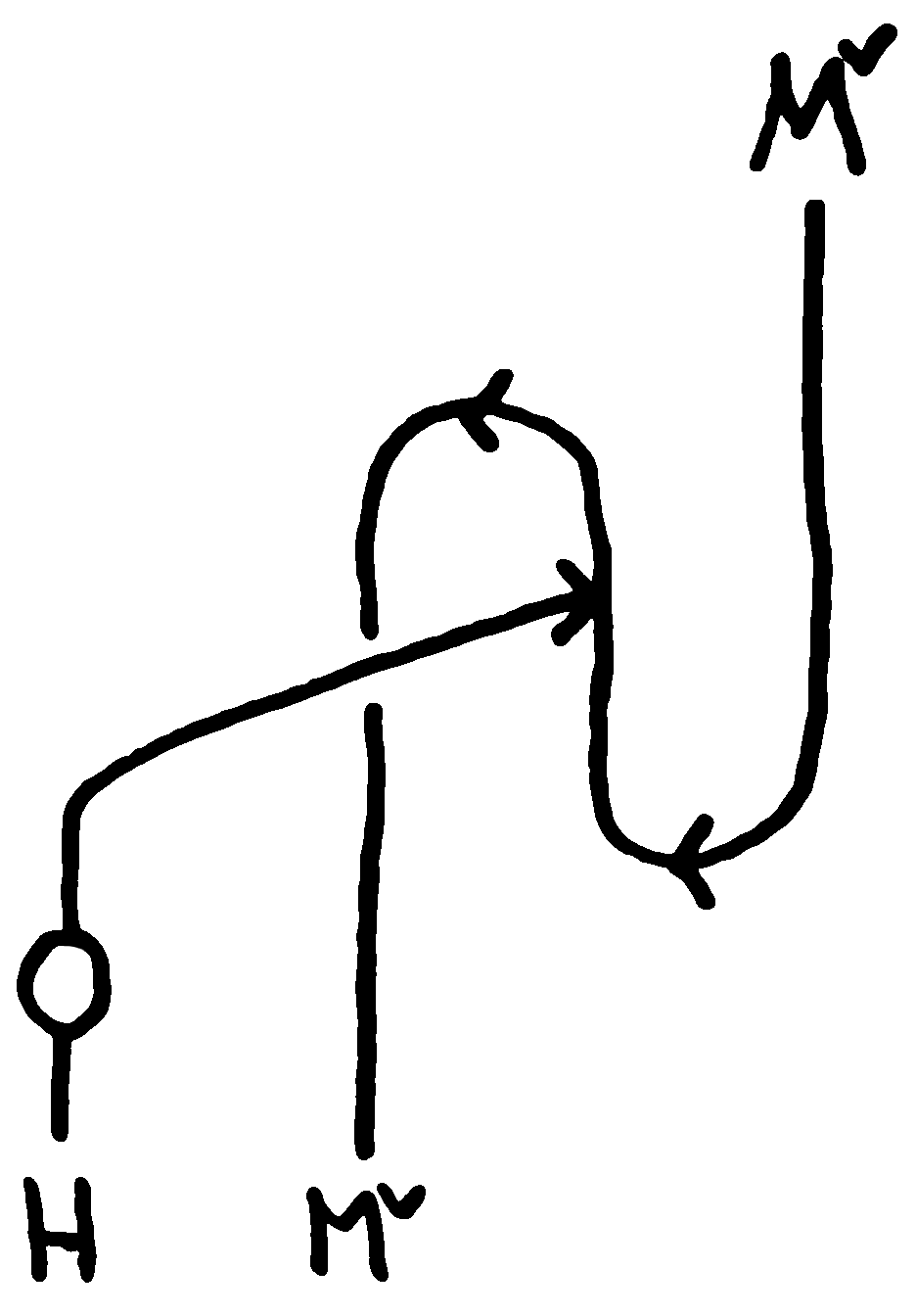} \qquad , \qquad
   \rho_{{}^*\!M} = \scanpic{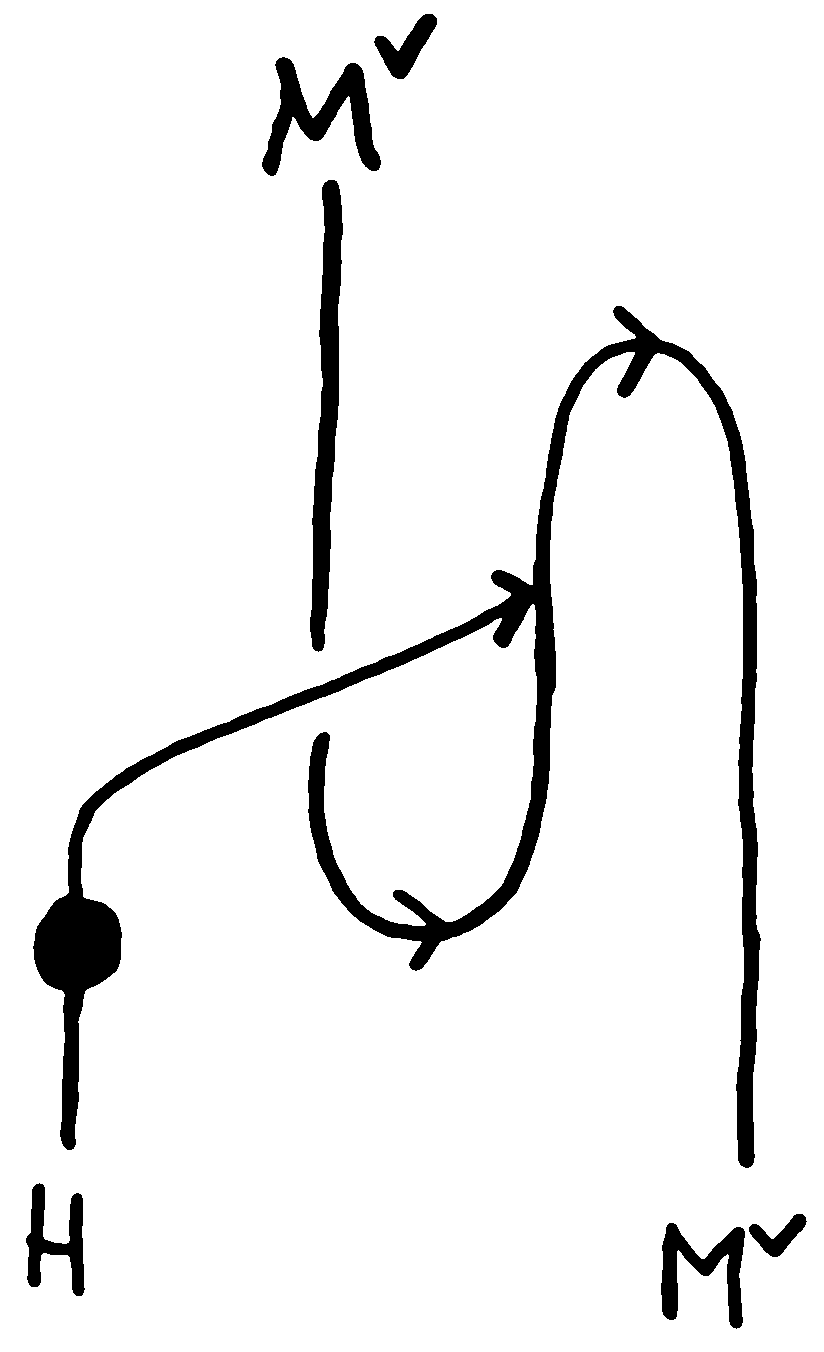} \quad .
\ee
One checks that $M^*$ and ${}^*\!M$ are indeed $H$-modules. The evaluation and coevaluation maps in $\Cc_0$ are just those of $\Sc$,
\begin{align}
  \ev^{\Cc}_M = \ev_M &~:~ M^* \otimes_{\Cc} M \to \one 
  \quad , \quad
  &
  \coev^{\Cc}_M = \coev_M &~:~ \one \to M \otimes_{\Cc} M^* \quad ,
  \nonumber\\
  \widetilde\ev^{\Cc}_M = \widetilde\ev_M &~:~ M \otimes_{\Cc} {}^*\!M \to \one
  \quad , \quad
  &
  \widetilde\coev^{\Cc}_M = \widetilde\coev_M &~:~ \one \to {}^*\!M \otimes_{\Cc} M 
  \ .
\label{eq:C0-duality maps}
\end{align}
It is clear that these satisfy the duality map conditions. It remains to prove that these are really morphisms in $\Cc_0$, i.e.\ that they are $H$-module maps. We give the calculation for $\ev^{\Cc}_M$ and $\widetilde\ev^{\Cc}_M$, the calculation for the coevaluation is similar:
\begin{align}
&\scanpic{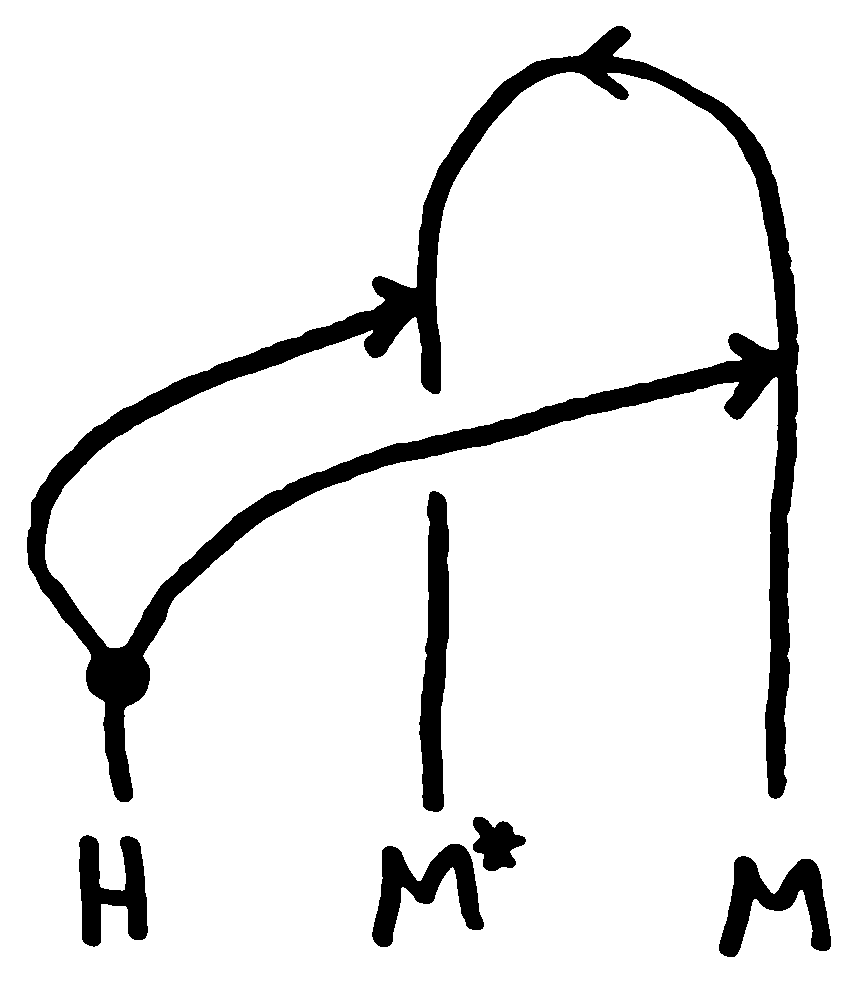} 
~=~ 
\scanpic{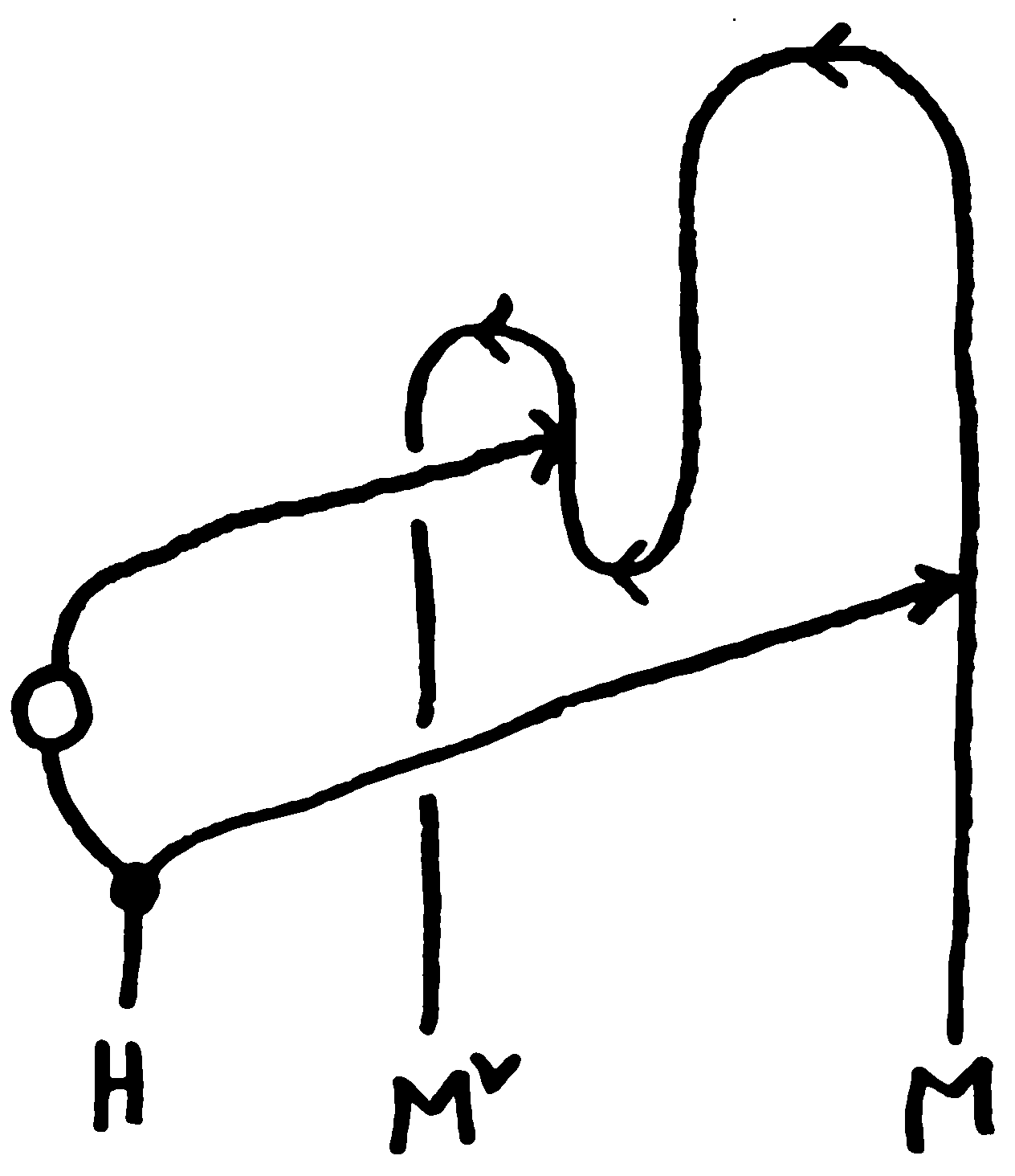} 
~=~ 
\scanpic{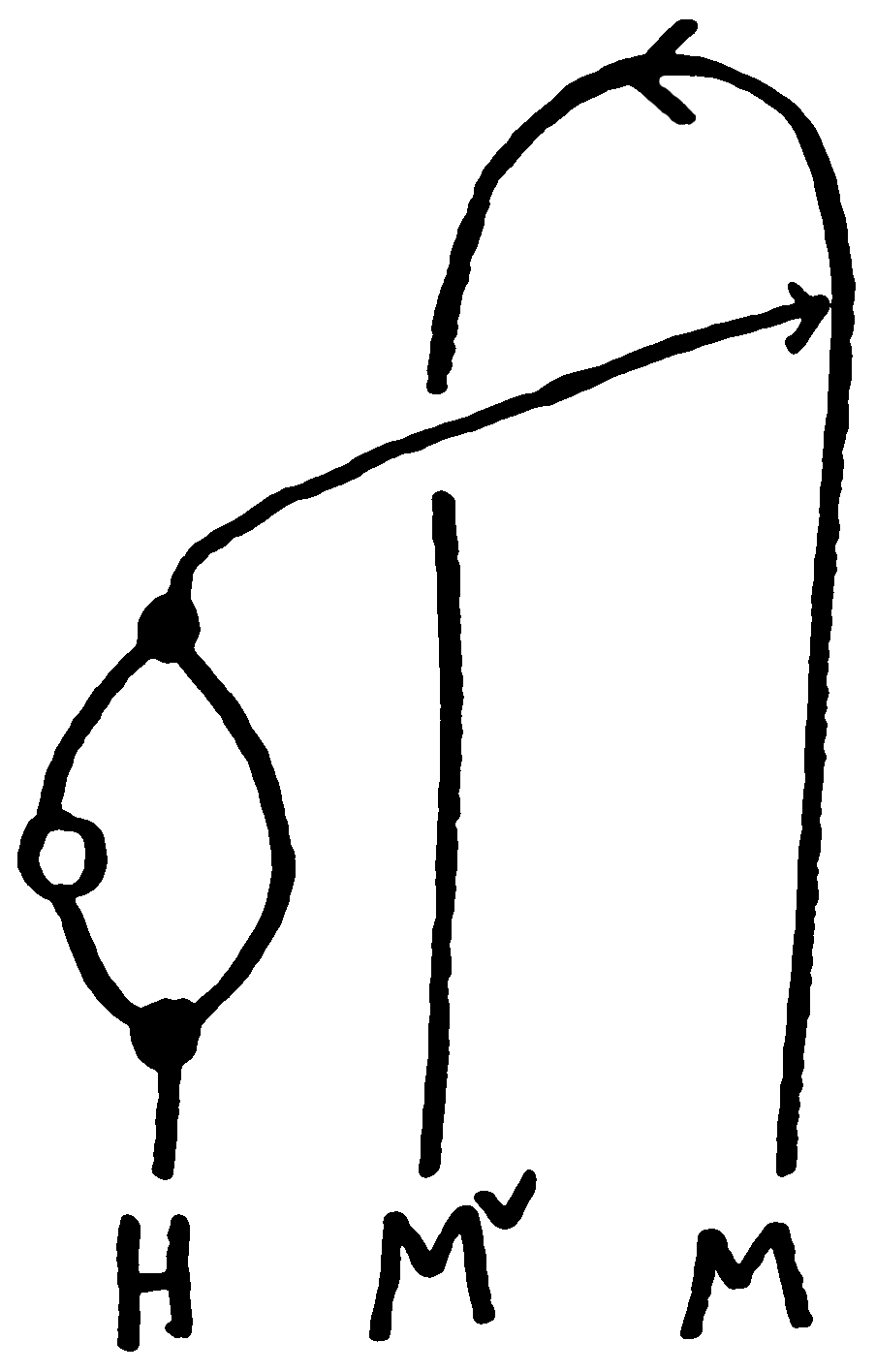} 
~=~ 
\scanpic{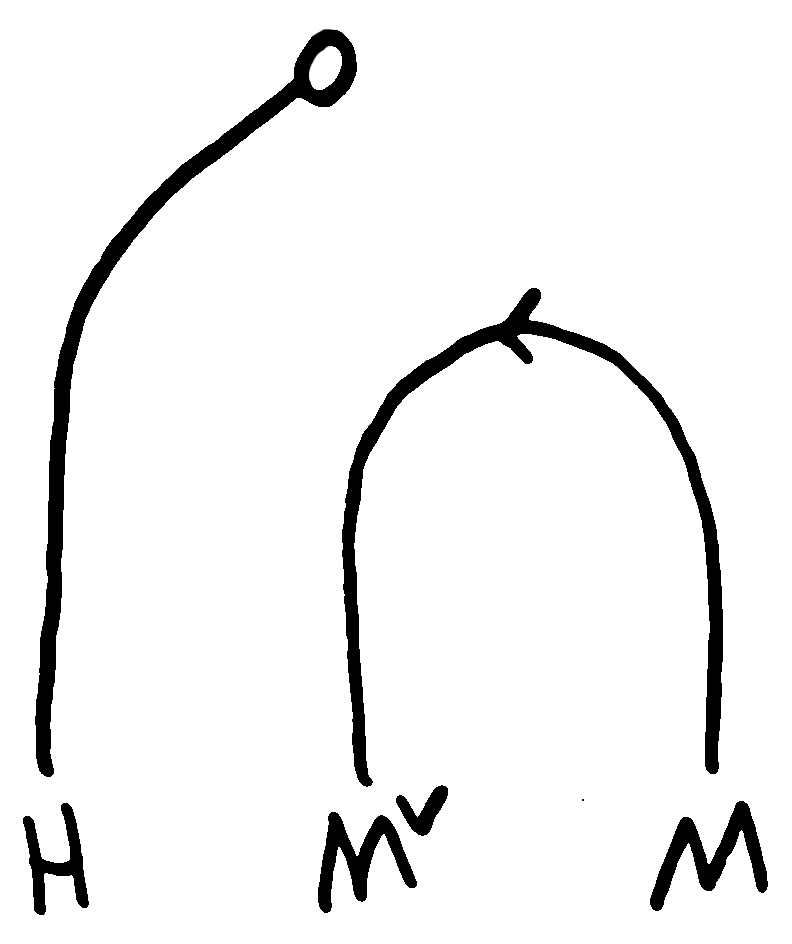} 
\quad ,
\nonumber \\
&\scanpic{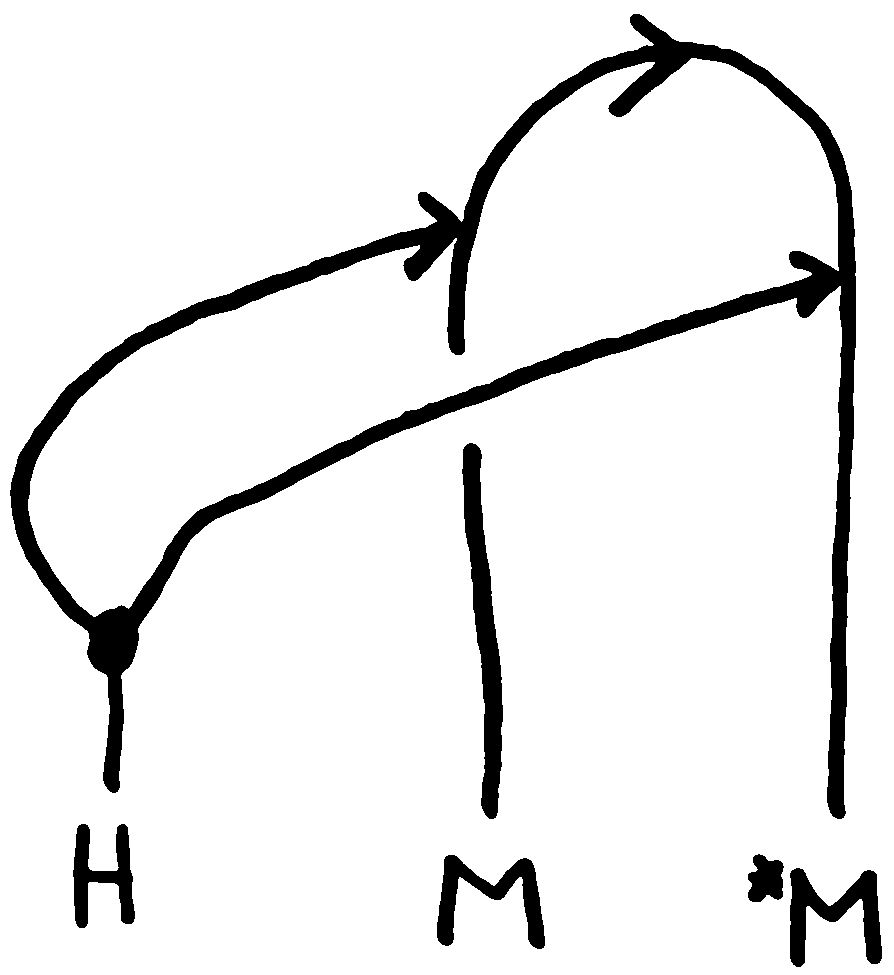} 
~=~ 
\scanpic{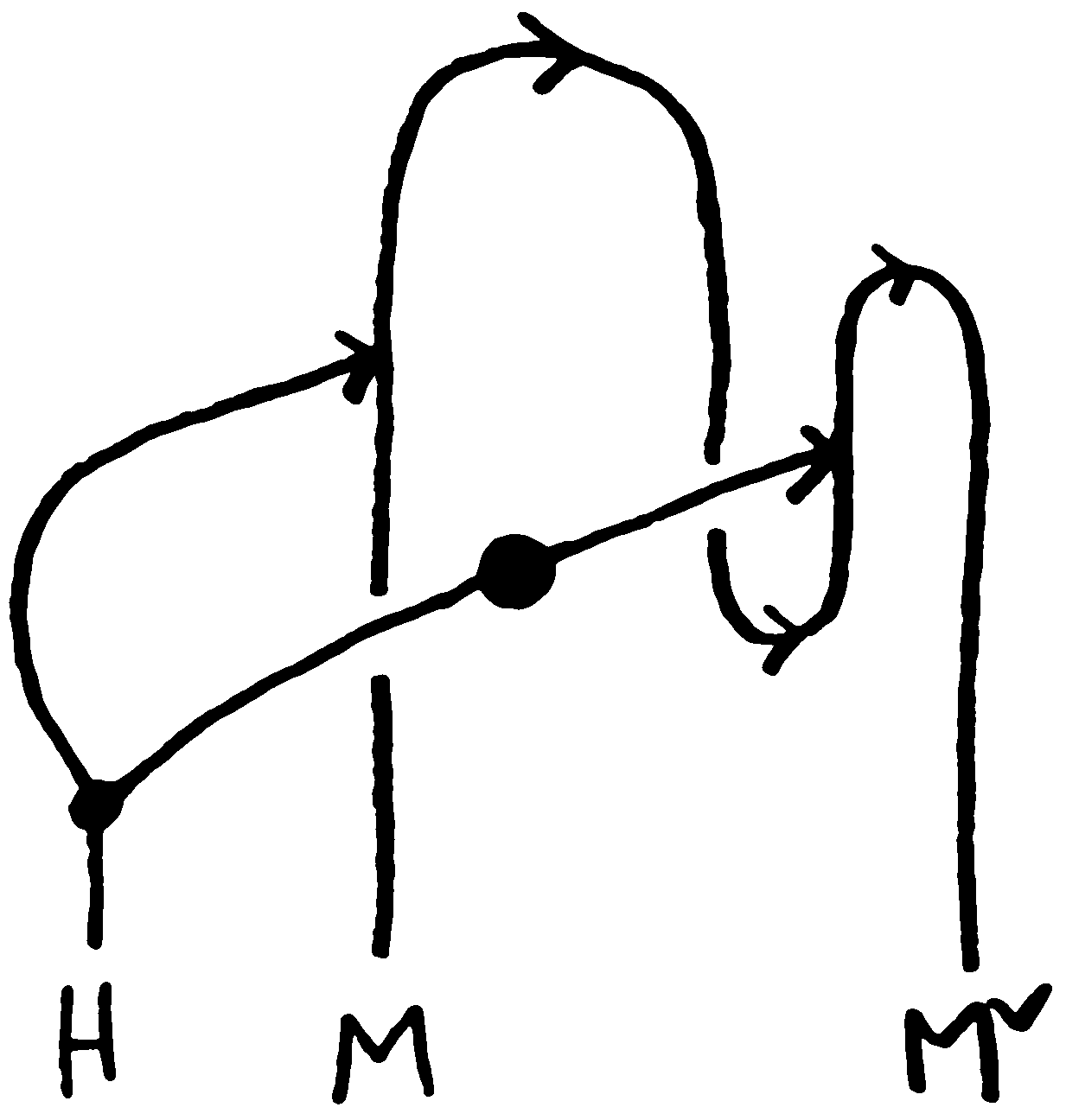} 
~=~ 
\scanpic{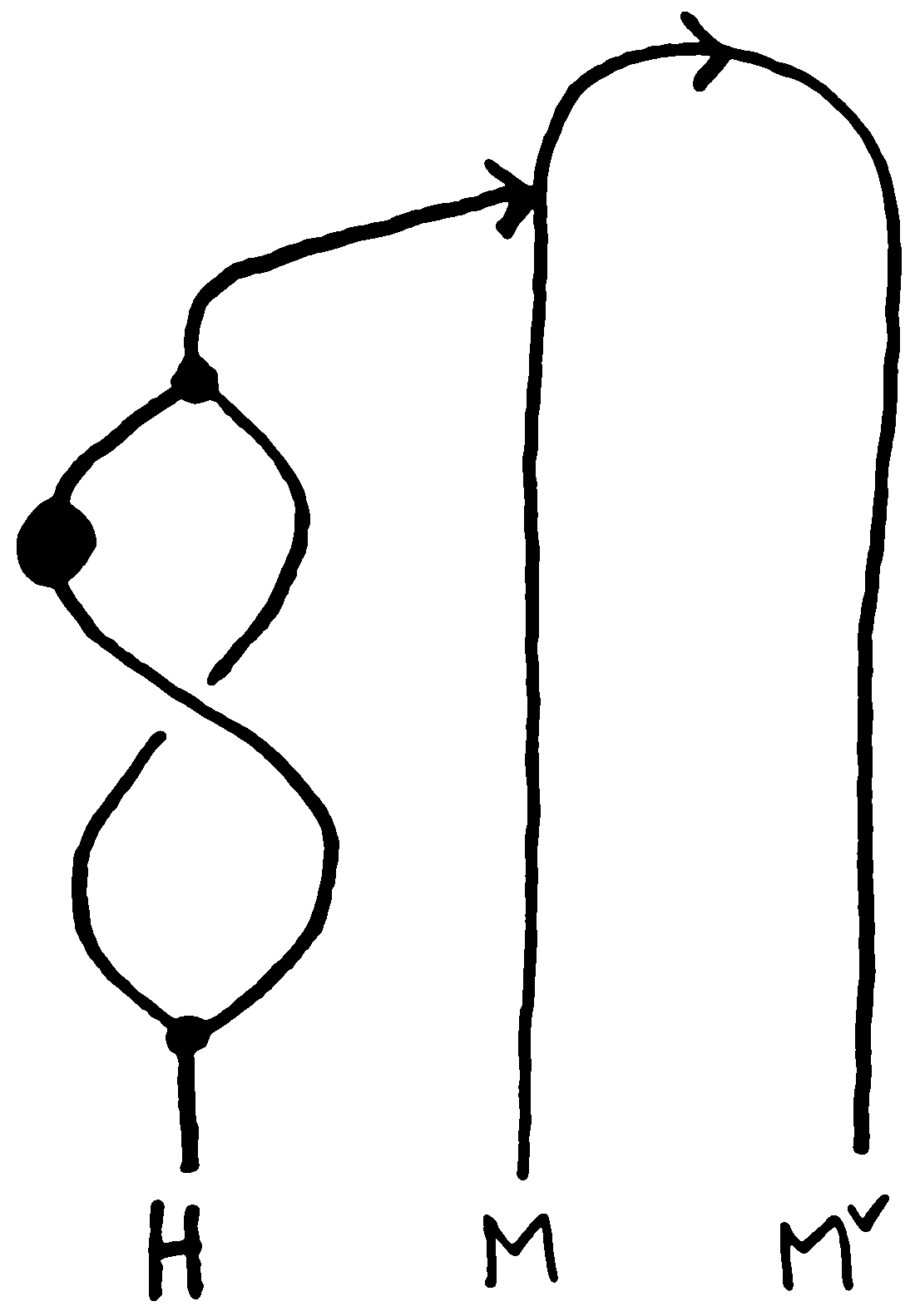} 
~=~ 
\scanpic{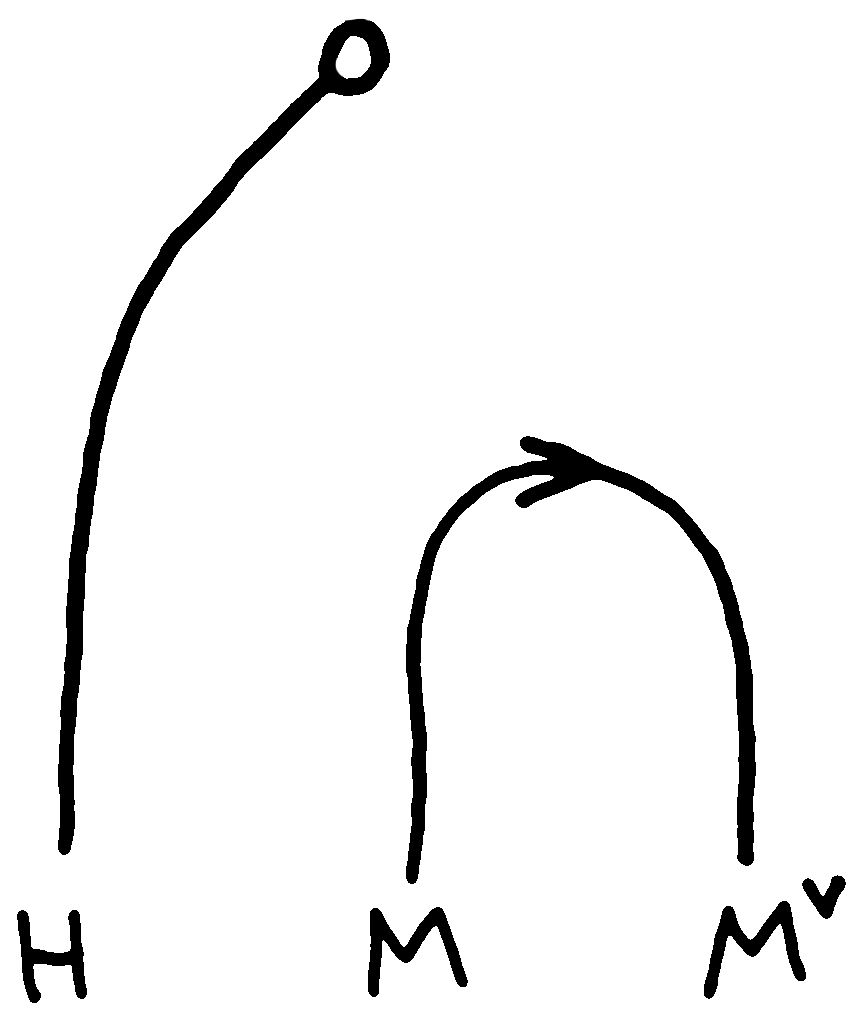} 
\quad .
\end{align}

\medskip\noindent
{\bf Right and left duals in $\mathbf{\Cc_1}$:} Let $X \in \Cc_1 = \Sc$. The left and right dual of $X$ is $X^\vee$, the dual in $\Sc$: 
\be \label{eq:C1-dual-objects}
  X^* = {}^*\!X = X^\vee \ .
\ee
Note that $X^* \otimes_{\Cc} X = H \otimes X^\vee \otimes X$, etc. For the duality maps we choose
\begin{align}
  \ev^{\Cc}_X = \eps \otimes \ev_X &~:~ X^* \otimes_{\Cc} X \to \one 
  \quad , \quad
  &
  \coev^{\Cc}_X = \Lambda \otimes \coev_X &~:~ \one \to X \otimes_{\Cc} X^* \quad ,
  \nonumber\\
  \widetilde\ev^{\Cc}_X = \eps \otimes \widetilde\ev_X &~:~ X \otimes_{\Cc} {}^*\!X \to \one
  \quad , \quad
  &
  \widetilde\coev^{\Cc}_X = \Lambda \otimes \widetilde\coev_X &~:~ \one \to {}^*\!X \otimes_{\Cc} X 
  \ ,
\label{eq:C1-duality maps}
\end{align}
where $\Lambda$ is the left integral determined by the right cointegral $\lambda$ as in Lemma \ref{lem:integral-cointegral}. The evaluation maps are morphisms in $\Cc_0$ because the counit is an algebra map; the coevaluation maps are morphisms in $\Cc_0$ by the defining property of a left integral.

Next we need to check that the duality map conditions are satisfied.
For the right duals, the following two morphisms have to be equal to $\id_X$ and $\id_{X^\vee}$, respectively (we omitted all `$\otimes_\Cc$'):
\begin{align}
&
X
\xrightarrow{~=~} 
\one X
\xrightarrow{~ \coev^{\Cc}_X \otimes \id ~}
(X X^*) X
\xrightarrow{~ \alpha^{-1}_{X,X^*,X} ~}
X(X^* X)
\xrightarrow{~ \id \otimes \ev^{\Cc}_X ~}
X \one
\xrightarrow{~=~} 
X \quad ,
\nonumber\\
&
X^* 
\xrightarrow{~=~} 
X^* \one
\xrightarrow{~ \id \otimes \coev^{\Cc}_X ~}
X^* (X X^*)
\xrightarrow{~ \alpha_{X^*,X,X^*} ~}
(X^* X) X^*
\xrightarrow{~ \ev^{\Cc}_X \otimes \id ~}
\one X^*
\xrightarrow{~=~} 
X^* \ .
\end{align}
Substituting the expressions for $\alpha$ and $\alpha^{-1}$ in terms of $\phi$ and $\phi^{-1}$ from  Section \ref{results} and \eqref{eq:inverse-assoc} leads to the two conditions
\be
  \eps \circ \phi^{-1} \circ \Lambda = \id_\one
  \qquad , \qquad
  \eps \circ \phi \circ \Lambda = \id_\one \ .
\ee
Since $\eps \circ \phi = \lambda$ and by the definition of $\Lambda$, the second condition is just the normalisation condition b) in Theorem \ref{thm:main1}\,(i). For the first condition substitute the expression $\phi^{-1} = S \circ \phi \circ {}_gM$ given in \eqref{eq:phi-1_via_phi} and use that $\Lambda$ is a left integral to remove the left multiplication by $g$.

The calculation for the left duality maps is analogous. Altogether we have

\begin{proposition}
The category $\Cc(H,\Gamma,\lambda)$ is rigid with left and right dual objects given by \eqref{eq:C0-dual-objects} and \eqref{eq:C1-dual-objects} and duality maps given by \eqref{eq:C0-duality maps} and \eqref{eq:C1-duality maps}.
\end{proposition}


\section{Braiding on $\Rep_\Sc(H)+\Sc$ for $\Sc$ symmetric} \label{sec:braiding}

In this section, we prove Theorem \ref{thm:main2} and its converse in the sense stated in Remark \ref{rem:main2}\,(i), as well as the existence of twists and the relation to the reverse category as asserted in Remark \ref{rem:twist+reverse}. We assume the conventions in Notations \ref{not:sec2}, \ref{not:sec3} and \ref{not:sec4}.

\subsection{Hexagon equations}\label{sec:hexagon}

The braiding isomorphisms given in \eqref{eq:braiding-ansatz} have to satisfy the two hexagon axioms in $\Cc$, that is, the following two diagrams have to commute:
\begin{align}
\text{L:}& 
\xymatrix{& A\otimes_\Cc(B\otimes_\Cc C) \ar[dl]_{\alpha_{A,B,C}} \ar[rr]^{c_{A,B C}} && (B\otimes_\Cc C)\otimes_\Cc A \\
(A\otimes_\Cc B)\otimes_\Cc C \ar[rd]_{c_{A,B}\otimes_\Cc id_C~~~~} &&&& B\otimes_\Cc(C\otimes_\Cc A) \ar[ul]_{\alpha_{B,C,A}} \\
& (B\otimes_\Cc A)\otimes_\Cc C && B\otimes_\Cc(A\otimes_\Cc C) \ar[ll]^{\alpha_{B,A,C}}  \ar[ur]_{~~~~id_B\otimes_\Cc c_{A,C}} }
\nonumber\\[.5em]
\text{R:}&\quad \xymatrix{& (A\otimes_\Cc B)\otimes_\Cc C\ar[rr]^{c_{A B,C}} && C\otimes_\Cc(A\otimes_\Cc B) \ar[dr]^{\alpha_{C,A,B}}  \\
A\otimes_\Cc(B\otimes_\Cc C) \ar[rd]_{id_A\otimes_\Cc c_{B,C}~~~} \ar[ru]^{\alpha_{A,B,C}} &&&& (C\otimes_\Cc A)\otimes_\Cc B \\
& A\otimes_\Cc(C\otimes_\Cc B) \ar[rr]^{\alpha_{A,C,B}} && (A\otimes_\Cc C)\otimes_\Cc B \ar[ur]_{~~c_{A,C}\otimes_\Cc id_B} }
\end{align}
(We have omitted the $\otimes_{\Cc}$ in the indices of $c$.) In formulas, the two hexagon conditions read:
\be
\begin{array}{rcrcl}
L &:& 
\alpha_{B,C,A}^{-1}
\circ 
c_{A,BC} 
\circ 
\alpha_{A,B,C}^{-1} 
&=& 
(\id_B \otimes_\Cc c_{A,C}) 
\circ 
\alpha_{B,A,C}^{-1} 
\circ 
(c_{A,B} \otimes_\Cc \id_C) 
\\[.5em]
R &:& 
\alpha_{C,A,B}
\circ 
c_{AB,C} 
\circ 
\alpha_{A,B,C} 
&=& 
(c_{A,C} \otimes_\Cc \id_B) 
\circ 
\alpha_{A,C,B} 
\circ 
(\id_A \otimes_\Cc c_{B,C}) 
\end{array}
\ee
Writing $A = A^a$, $B = B^b$, $C = C^c$ for $a,b,c \in \{0,1\}$ there are once more 16 equations to consider, which we label as follows:
\begin{center}
\begin{tabular}{cccc||cccc}
case & $a$ & $b$ & $c$ &
case & $a$ & $b$ & $c$ \\
\hline
1\,L,R) & 0 & 0 & 0  & 5\,L,R) & 0 & 1 & 1  \\
2\,L,R) & 0 & 0 & 1  & 6\,L,R) & 1 & 0 & 1   \\
3\,L,R) & 0 & 1 & 0  & 7\,L,R) & 1 & 1 & 0   \\
4\,L,R) & 1 & 0 & 0  & 8\,L,R) & 1 & 1 & 1 
\end{tabular}
\end{center}
Let us for the moment forget our assumption (cf.\ Notation \ref{not:sec4}) that $\omega^2_U  = \id_U$ for all $U \in \Sc$. This will allow us to single out the hexagon equation which imposes this condition, namely hexagon 8R below.









The string diagram for the hexagon in case 1L is
\be
  \scanPIC{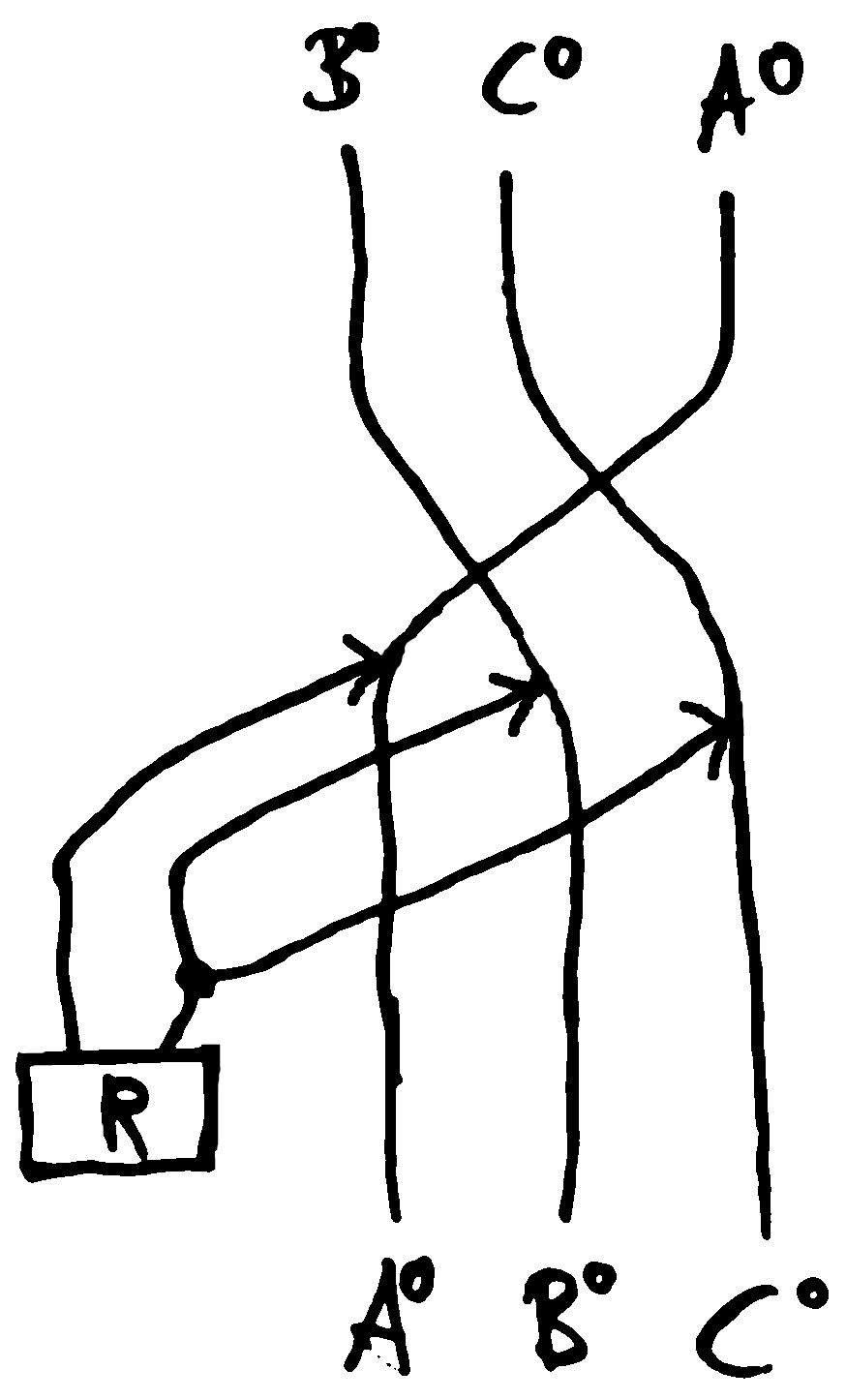}
  ~=~
  \scanPIC{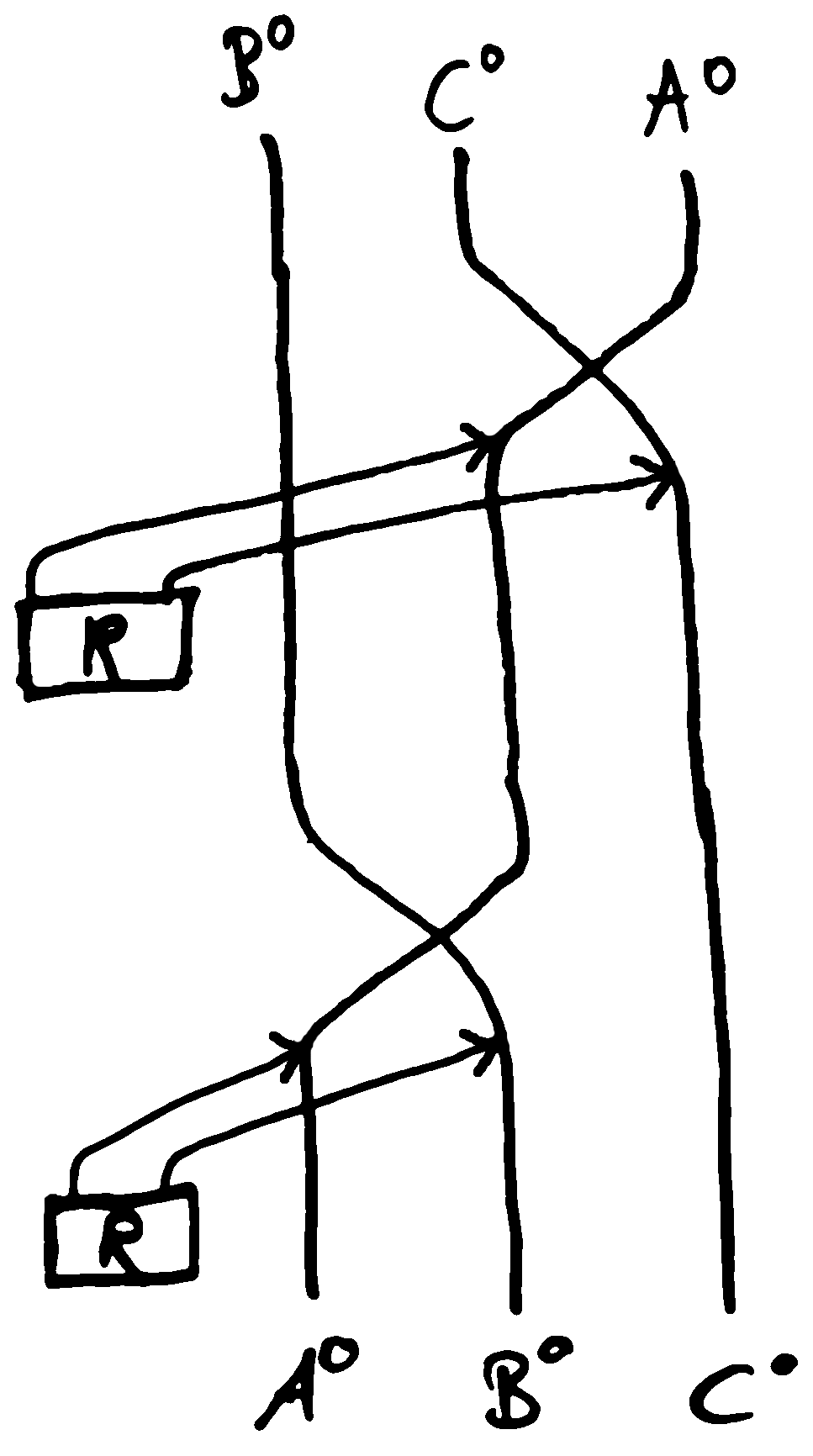}
  \quad .
\ee
As before, this can be equivalently formulated in terms of a condition involving only $H$ leading to \be\tag{H-1L}\label{eq:hexagon-1L}
  \scanPIC{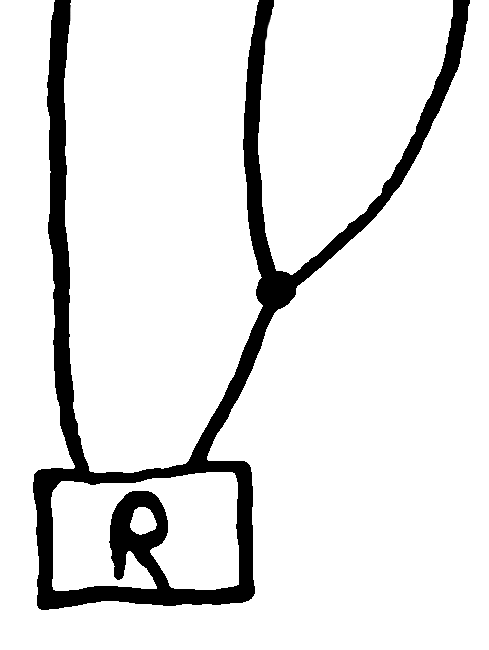}
  ~=~
  \scanPIC{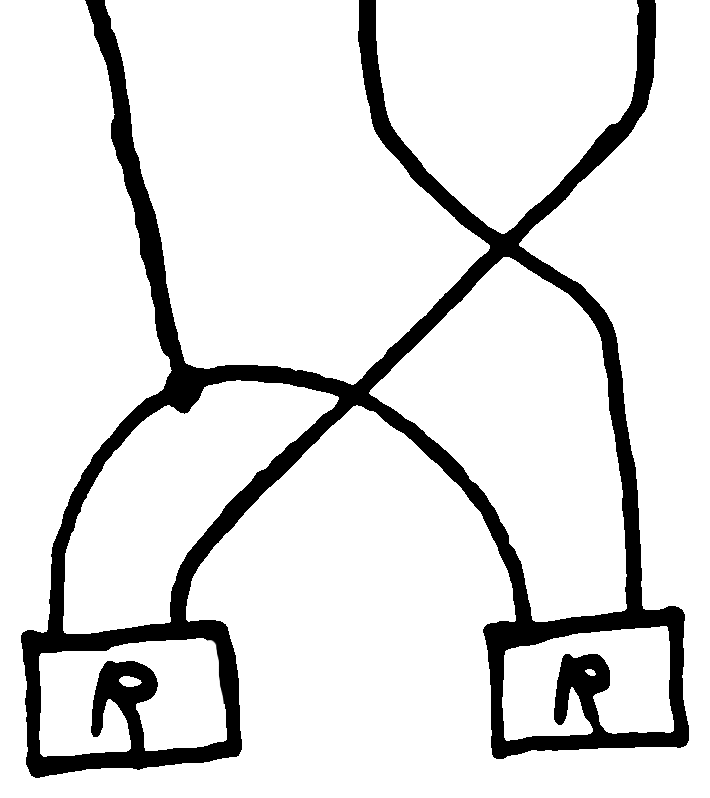}
  \quad .
\ee
For case 1R one finds analogously
the first equation in
\be\tag{H-1R}\label{eq:hexagon-1R}
  \scanPIC{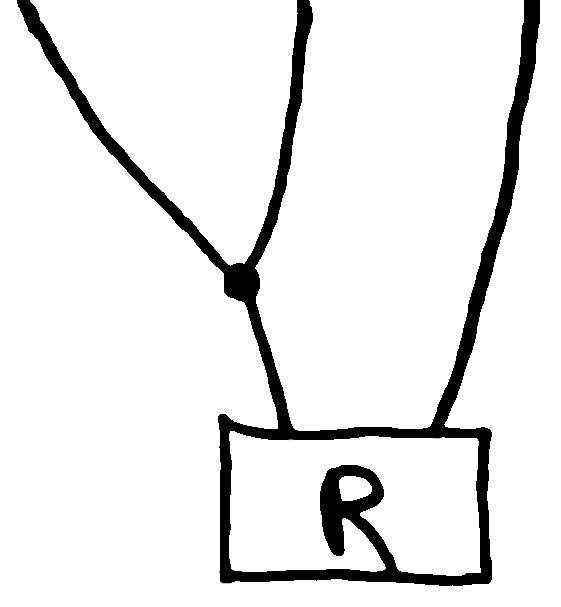}
  ~=~
  \scanPIC{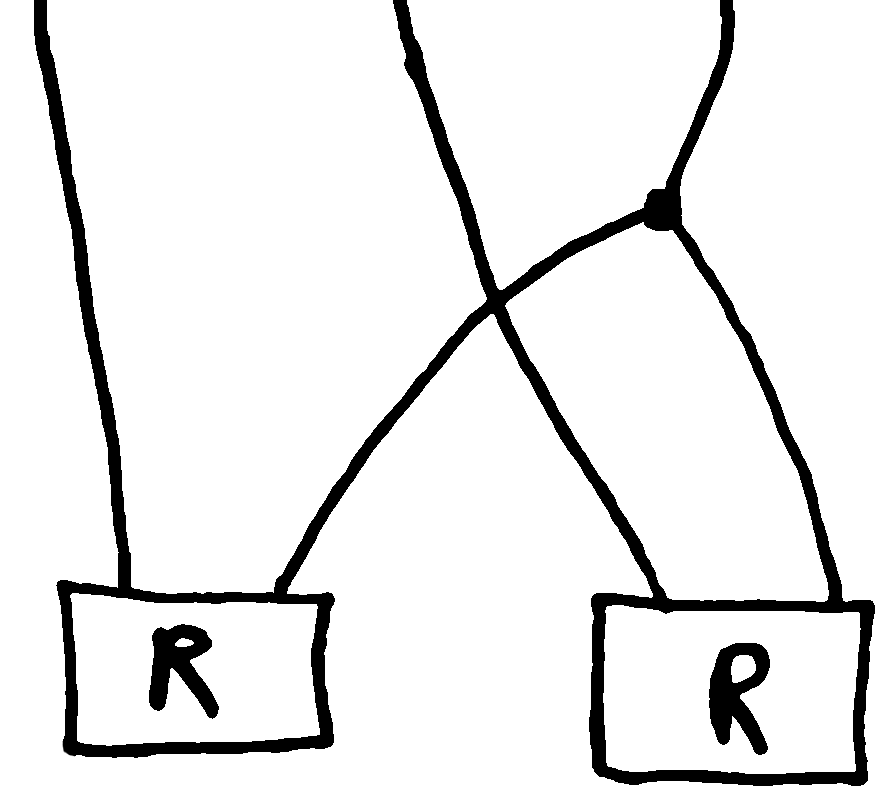}
  \quad .
\ee
This, of course, is nothing but the requirement that $H$ has to be quasi-triangular, which is equivalent to endowing $\Rep_{\Sc}(H)$ with a braiding, see e.g.\ \cite[Sect.\,2.5]{Majid:1995} or \cite[Sect.\,2.6]{Bespalov:1995}.

Let us look at two more cases in more detail as further examples. Consider the hexagon 5R. The resulting string diagram is
\be\label{eq:hexagon-5R-aux1}
  \scanpic{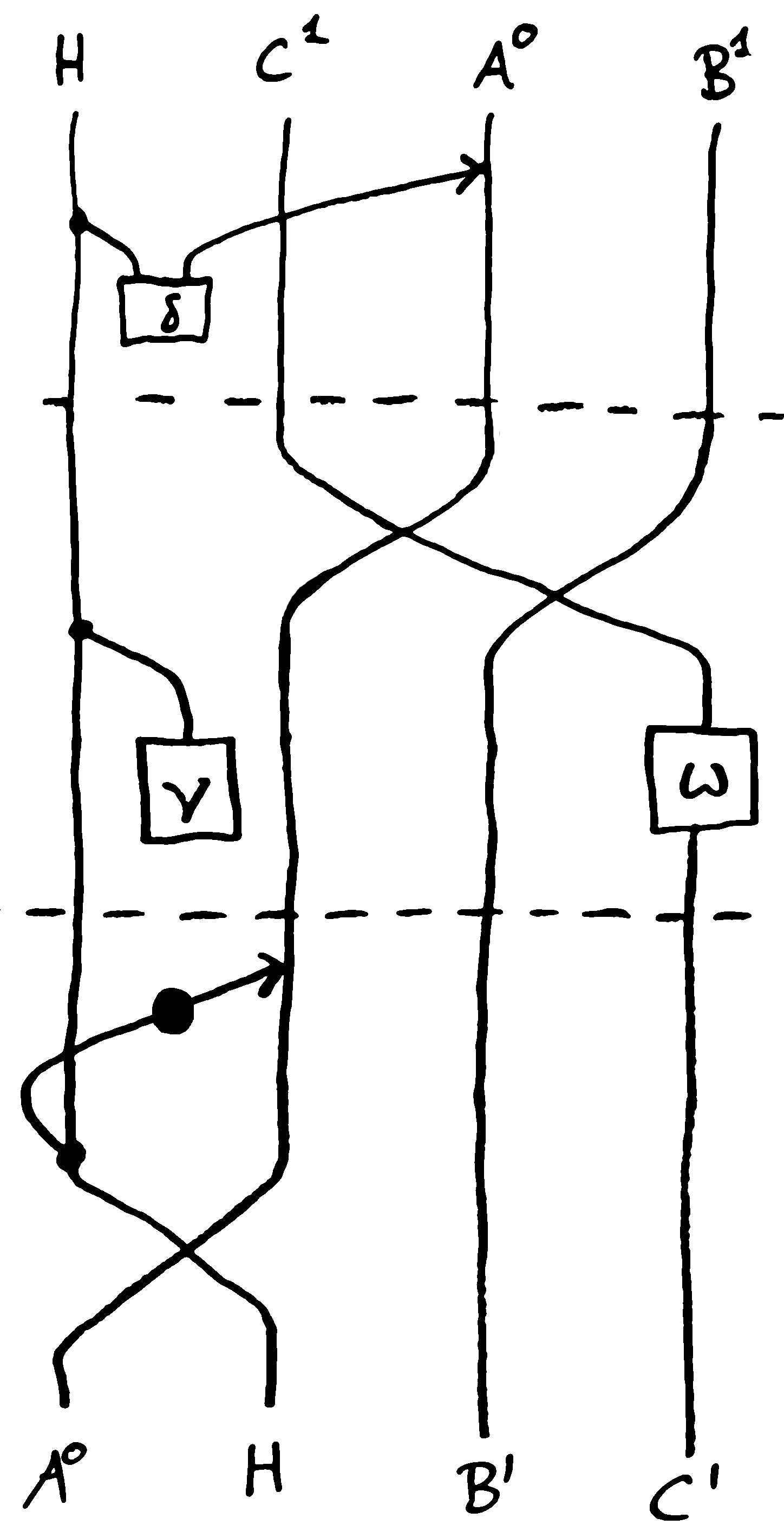}
  ~=~
  \scanpic{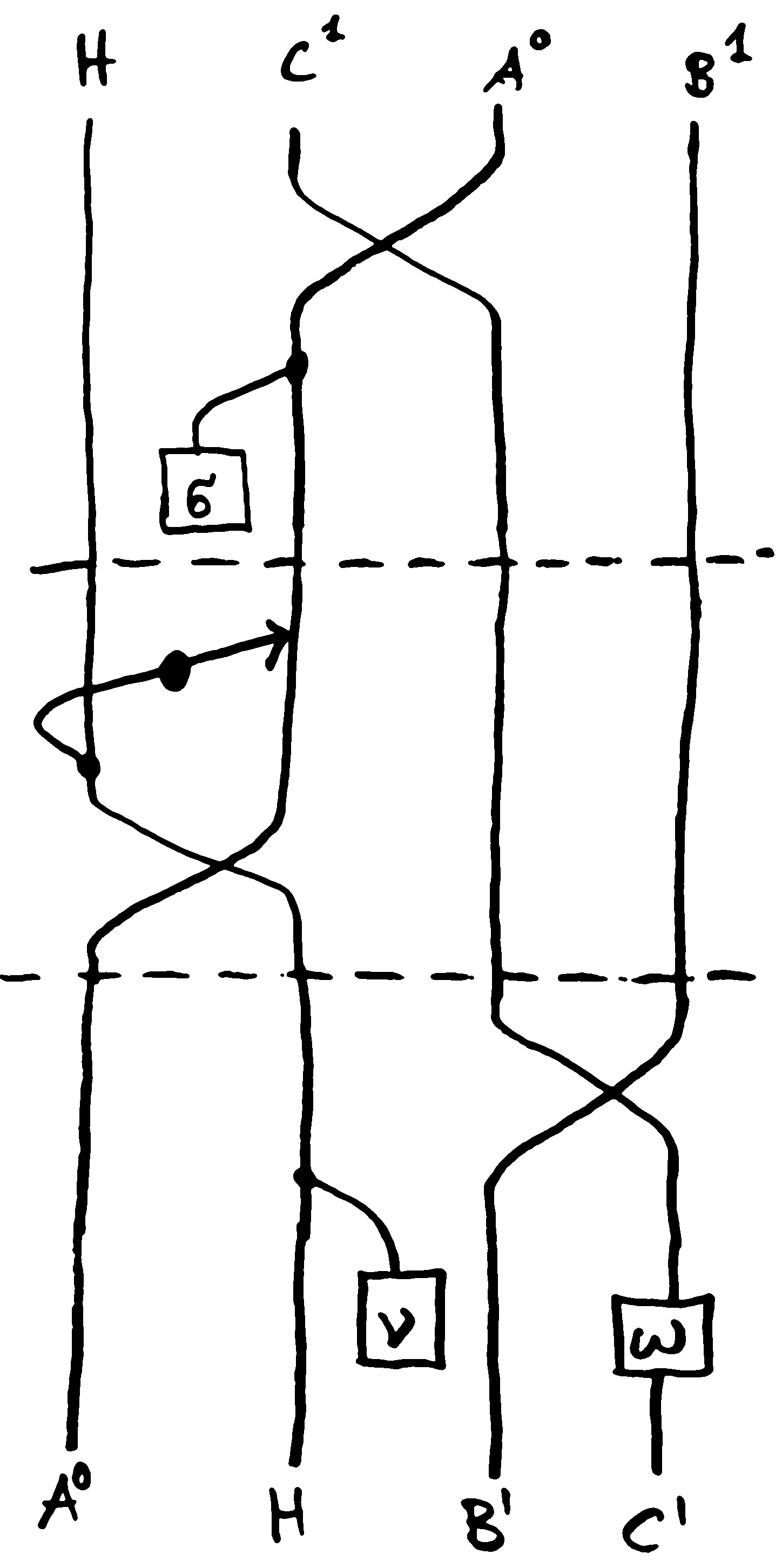}
  \quad .
\ee
The dashed lines indicate how the string diagram is build out of the individual associativity and braiding isomorphisms. As with the conditions obtained from the pentagon equations, we can specialise to $A^0=H$ and $B^1=C^1=\one$ and compose with $\eta \otimes \id$ from the right (since both sides are module maps for the $H$-right action on the $A^0 = H$ factor). This gives condition \eqref{eq:hexagon-5R} in the list below and one verifies that \eqref{eq:hexagon-5R} is indeed equivalent to \eqref{eq:hexagon-5R-aux1}. 

The final case we treat in detail is 8R because it is the reason that we will impose $\omega^2 = \Id$. The string diagram is
\be\label{eq:hexagon-8R-aux1}
  \scanpic{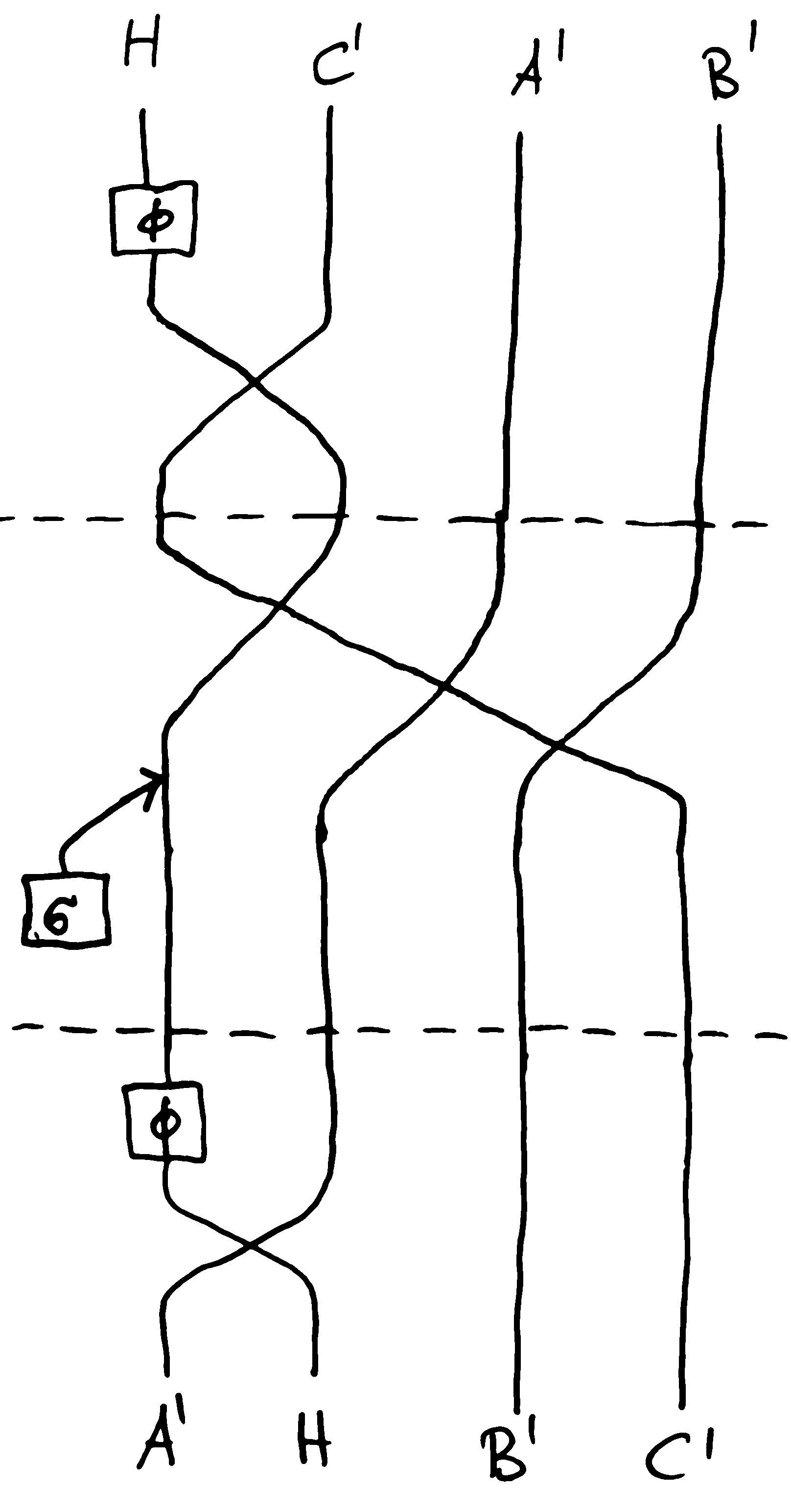}
  ~=~
  \scanpic{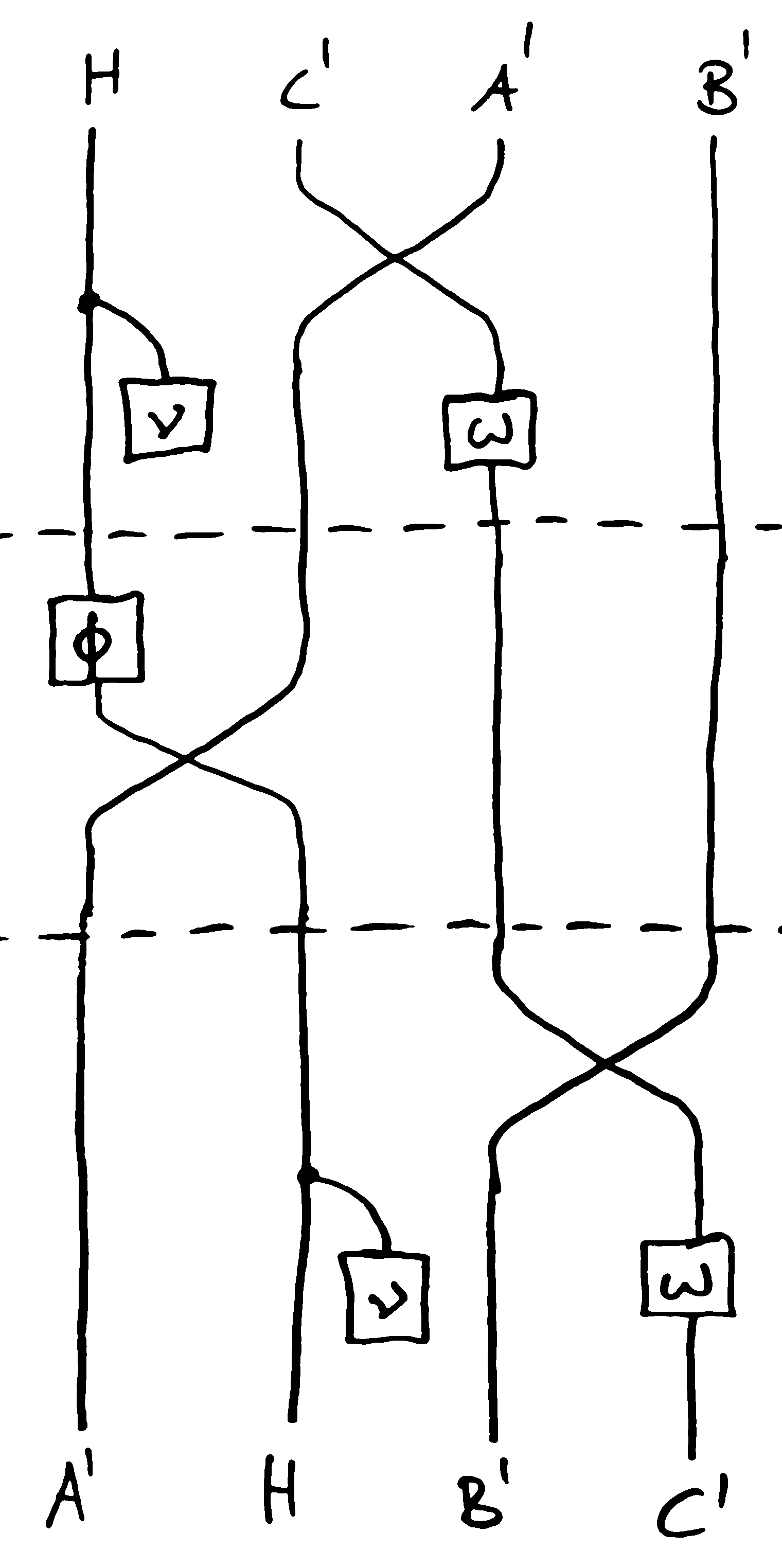}
  \quad .
\ee
Specialising to $A^1=B^1=C^1=\one$ results in \eqref{eq:hexagon-8R} in the list below. Both sides of \eqref{eq:hexagon-8R} are isomorphisms. Setting $A^1=B^1=\one$ in \eqref{eq:hexagon-8R-aux1} and composing the $H$-factor with the inverse of \eqref{eq:hexagon-8R} gives $\id_H \otimes \id_{C^1} = \id_H \otimes (\omega_{C^1})^2$. Finally, inserting this identity into $(\eps \otimes \id_{C^1}) \circ (-) \circ (\eta \otimes \id_{C^1})$ shows $(\omega_{C^1})^2 = \id_{C^1}$ for all $C^1 \in S$.
Conversely, we obtain \eqref{eq:hexagon-8R-aux1} from  \eqref{eq:hexagon-8R} and $(\omega_{C^1})^2 = \id_{C^1}$. Hence, from here on we will require that
\be
  \omega^2 = \Id_{\Sc} \ .
\ee

The 16 hexagon conditions, obtained as outlined above and given in the equivalent formulation involving only $H$, are \eqref{eq:hexagon-1L} and \eqref{eq:hexagon-1R} above and
\\
\noindent
\begin{tabular}{p{19em}p{19em}}
\be
  \scanPIC{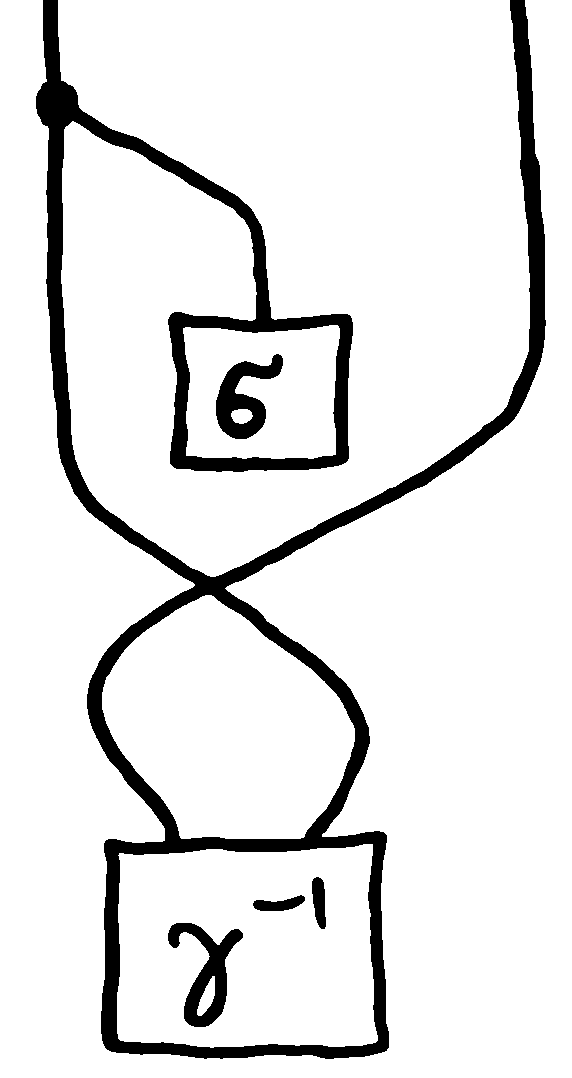} ~~=~~ \scanPIC{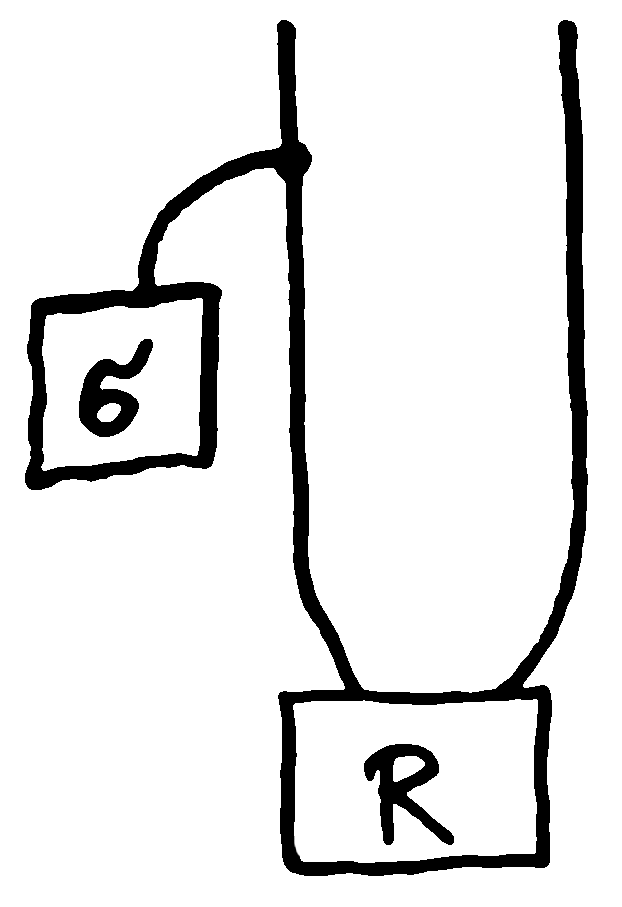} 
  \tag{H-2L} \label{eq:hexagon-2L}
\ee
&
\be
  \scanPIC{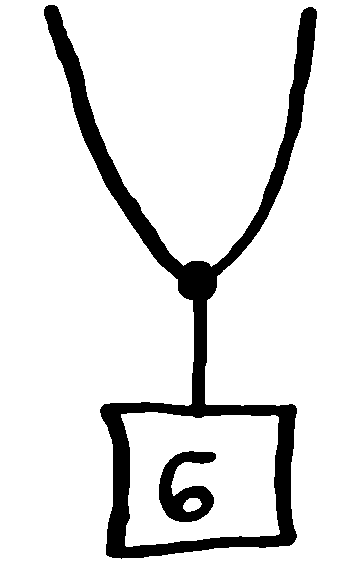} ~~=~~ \scanPIC{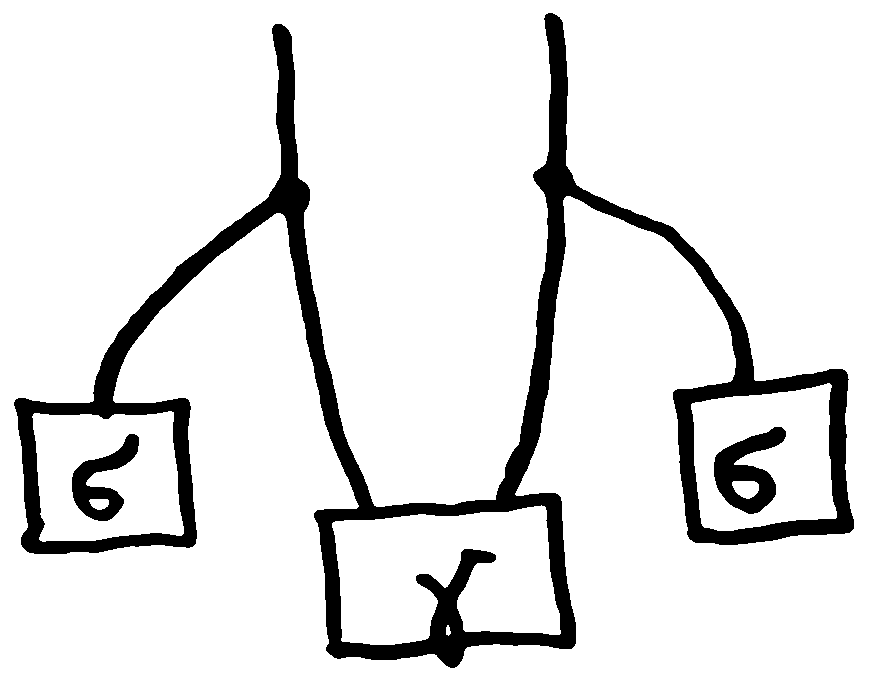} 
  \tag{H-2R} \label{eq:hexagon-2R} 
\ee
\end{tabular}
\\
\begin{tabular}{p{19em}p{19em}}
\be
  \scanPIC{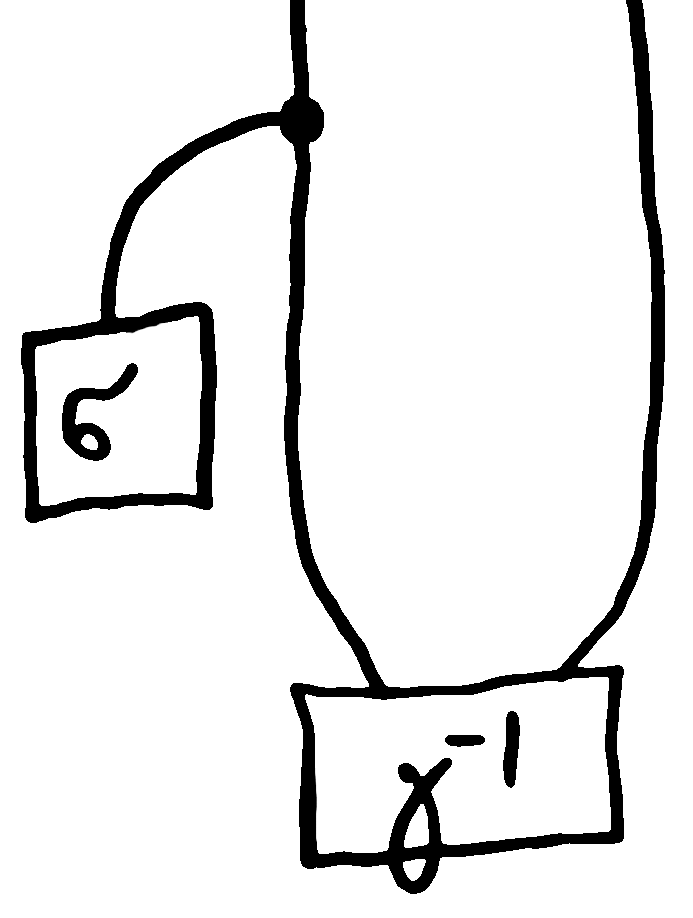} ~~=~~ \scanPIC{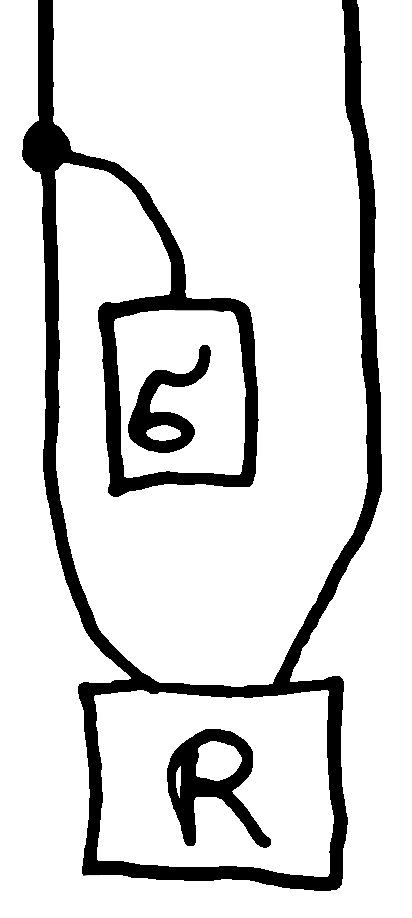} 
  \tag{H-3L} \label{eq:hexagon-3L} 
\ee
&
\be
  \scanPIC{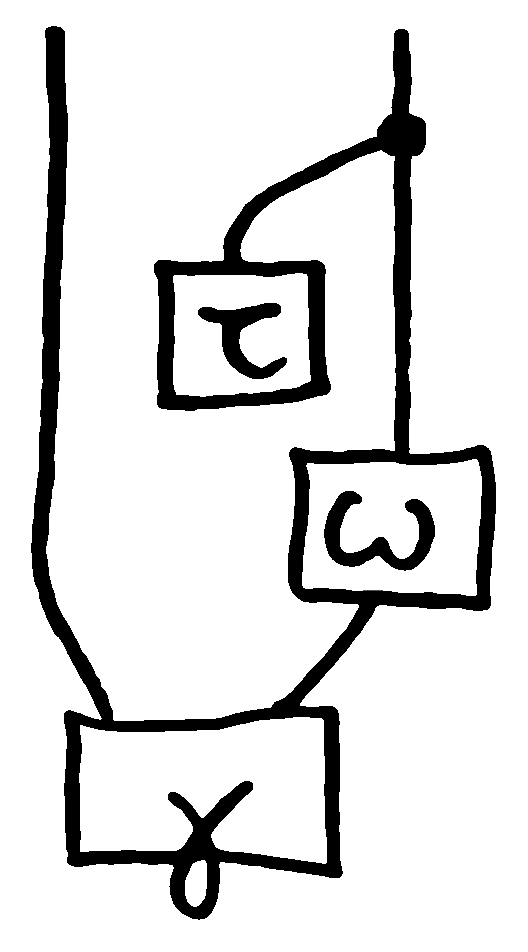} ~~=~~ \scanPIC{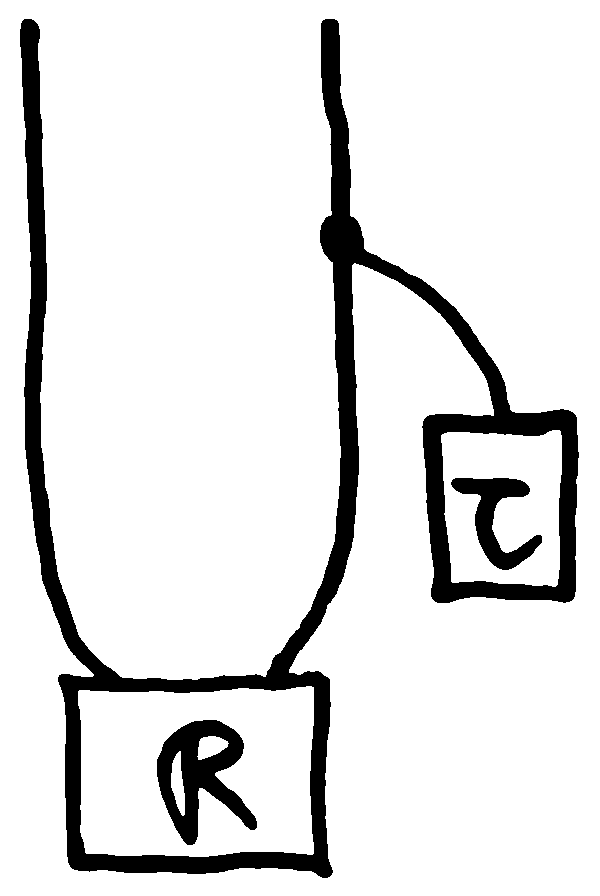} 
  \tag{H-3R} \label{eq:hexagon-3R}
\ee
\end{tabular}
\\
\begin{tabular}{p{19em}p{19em}}
\be
  \scanPIC{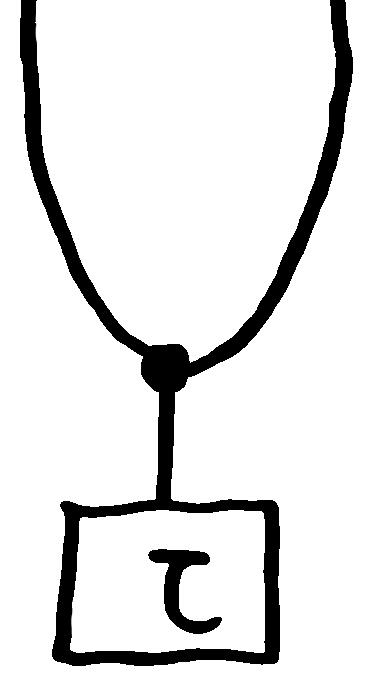} ~~=~~ \scanPIC{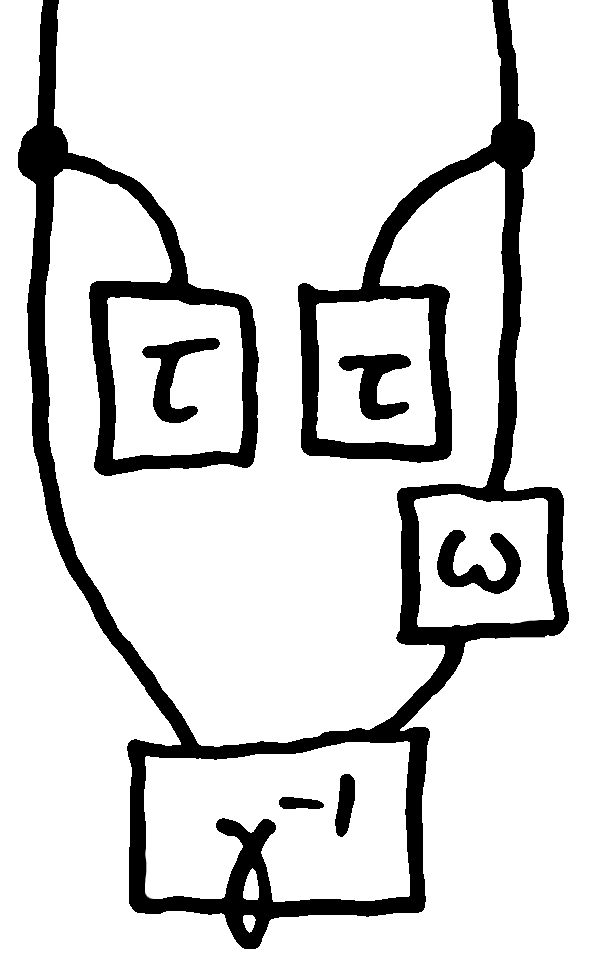} 
   \tag{H-4L} \label{eq:hexagon-4L}\ee
&
\be
  \scanPIC{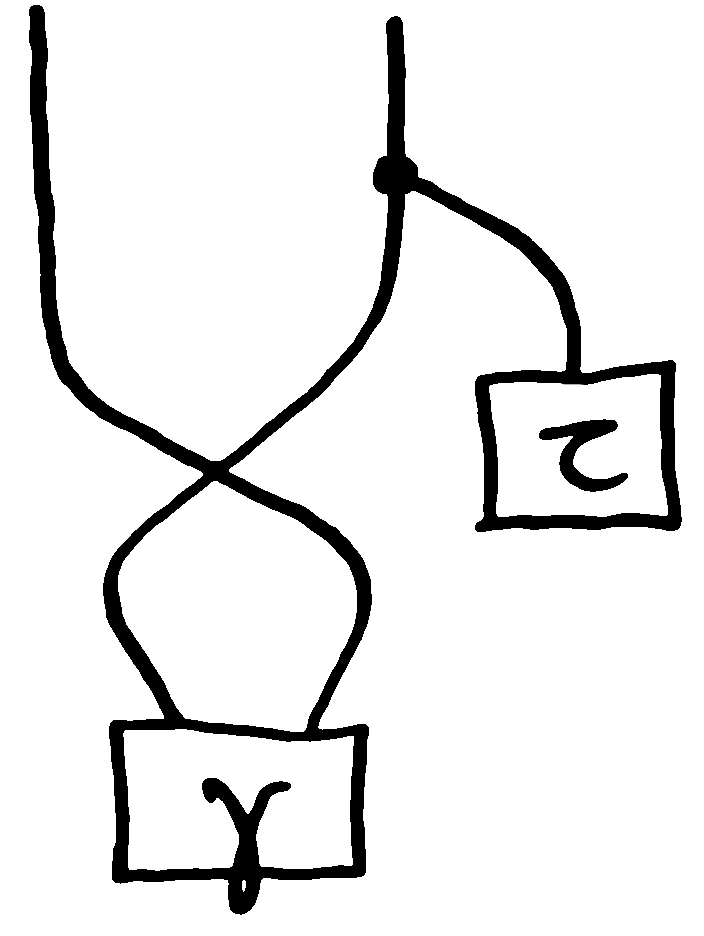} ~~=~~ \scanPIC{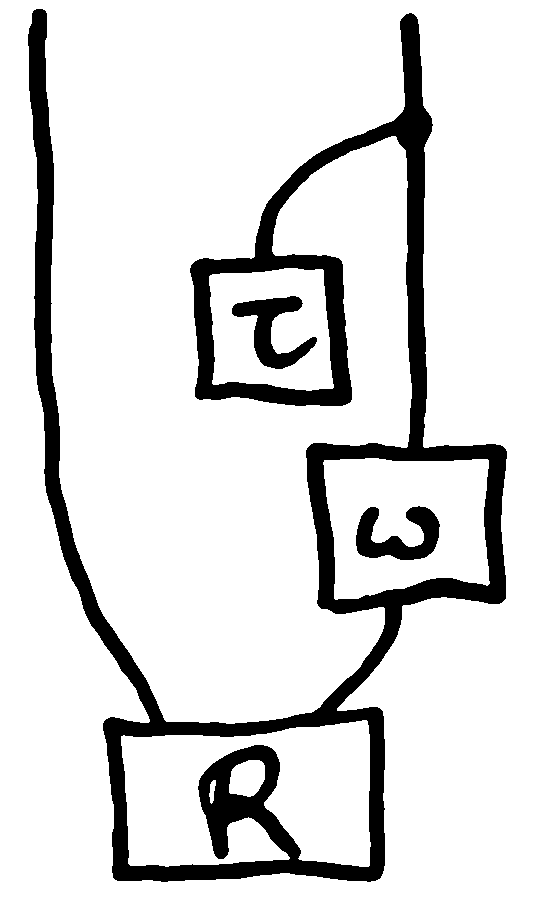} 
  \tag{H-4R} \label{eq:hexagon-4R}
\ee
\end{tabular}
\\
\begin{tabular}{p{19em}p{19em}}
\be
  \scanPIC{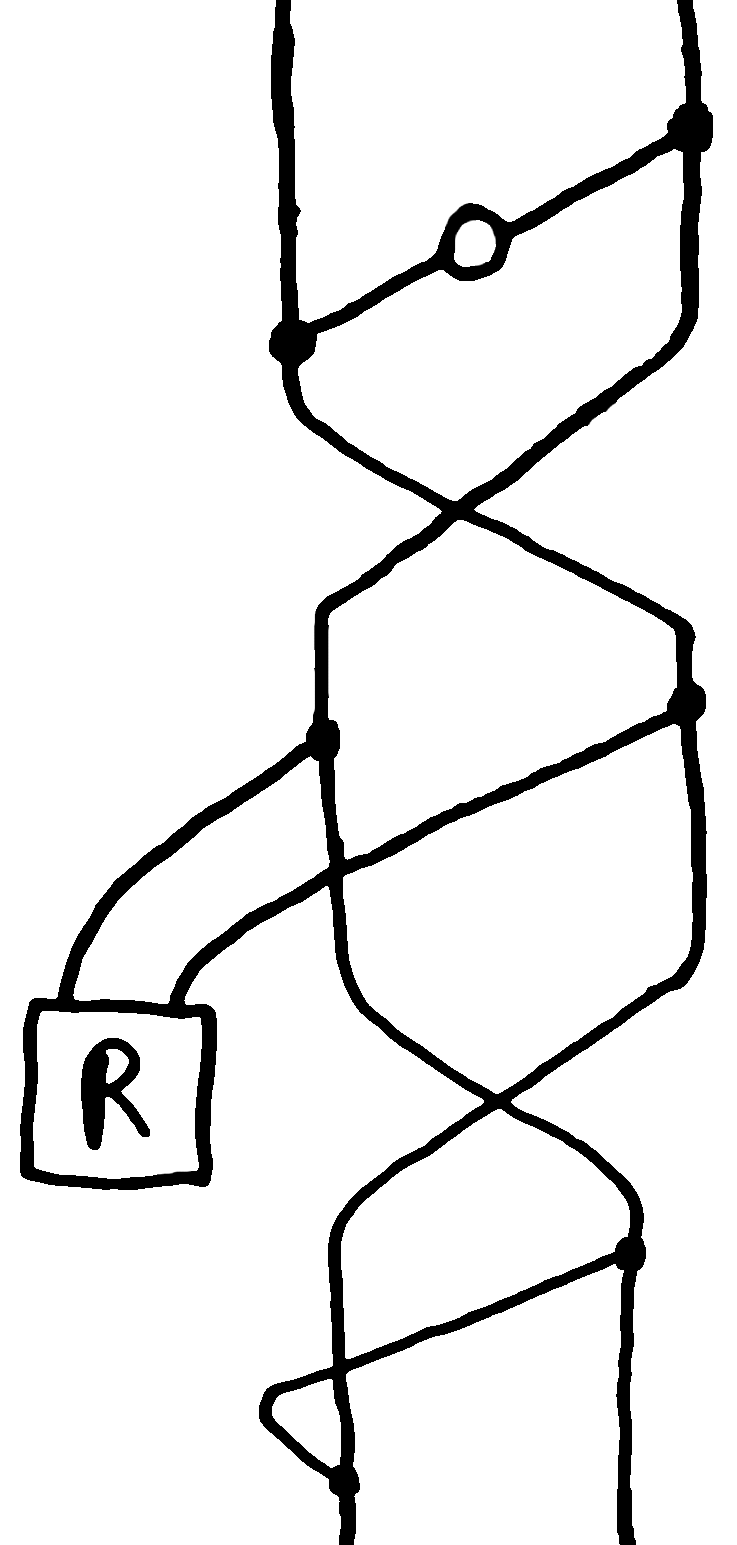} ~~=~~ \scanPIC{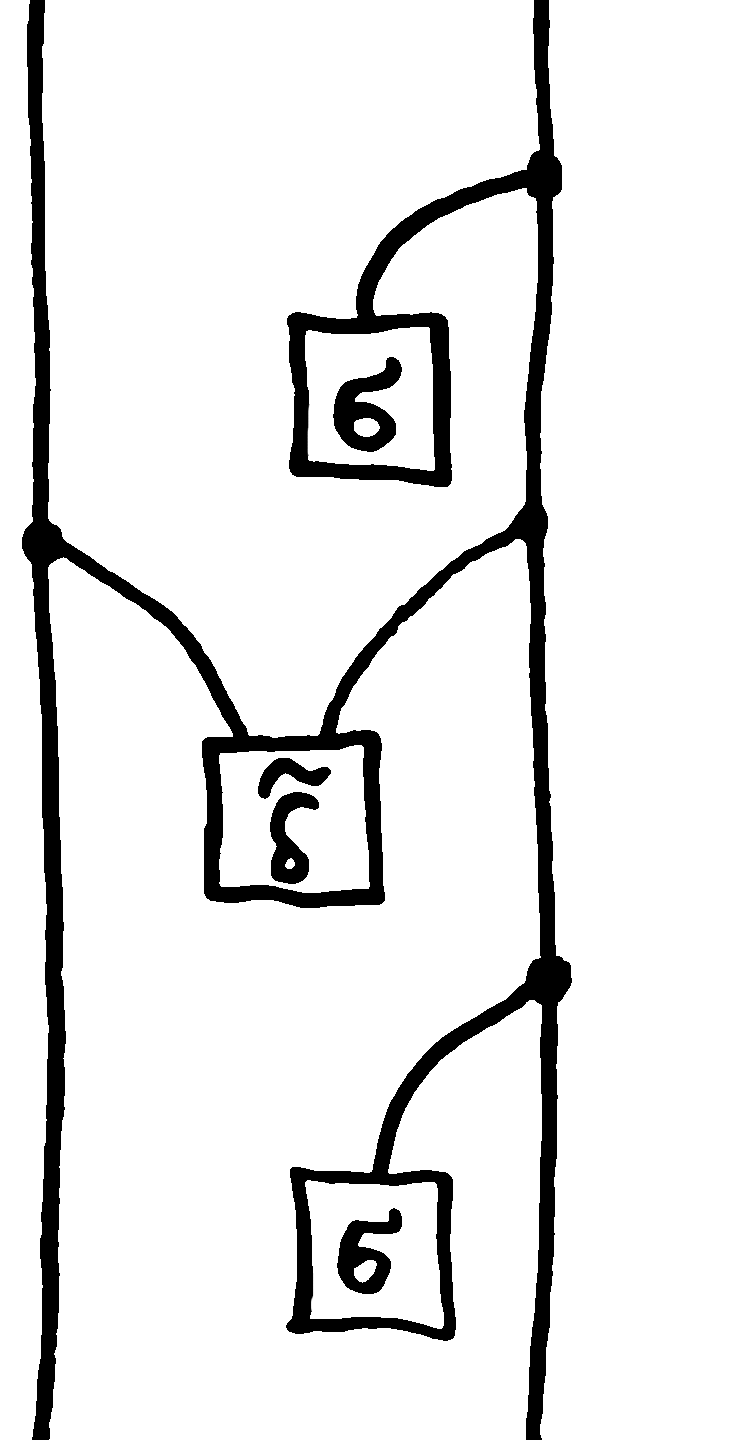} 
  \tag{H-5L} \label{eq:hexagon-5L}\ee
&
\be
  \scanPIC{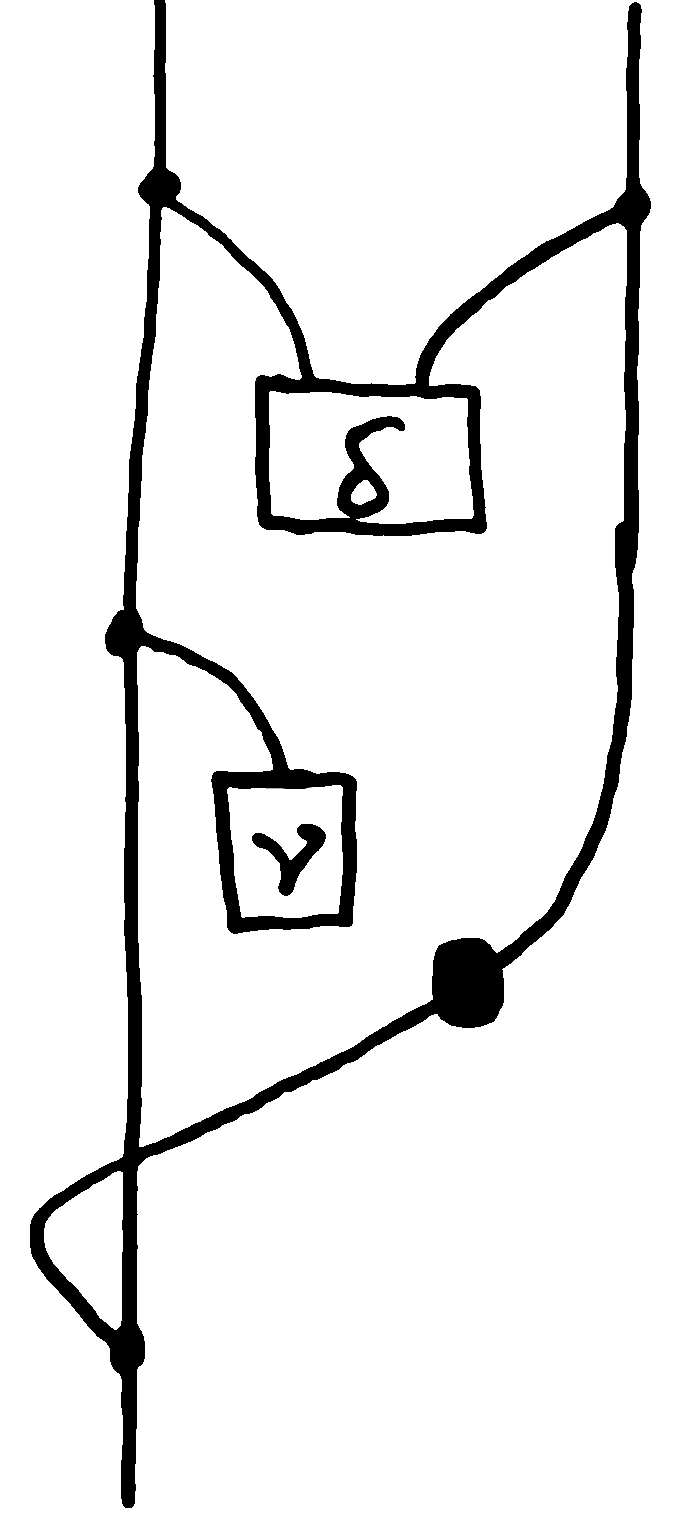} ~~=~~ \scanPIC{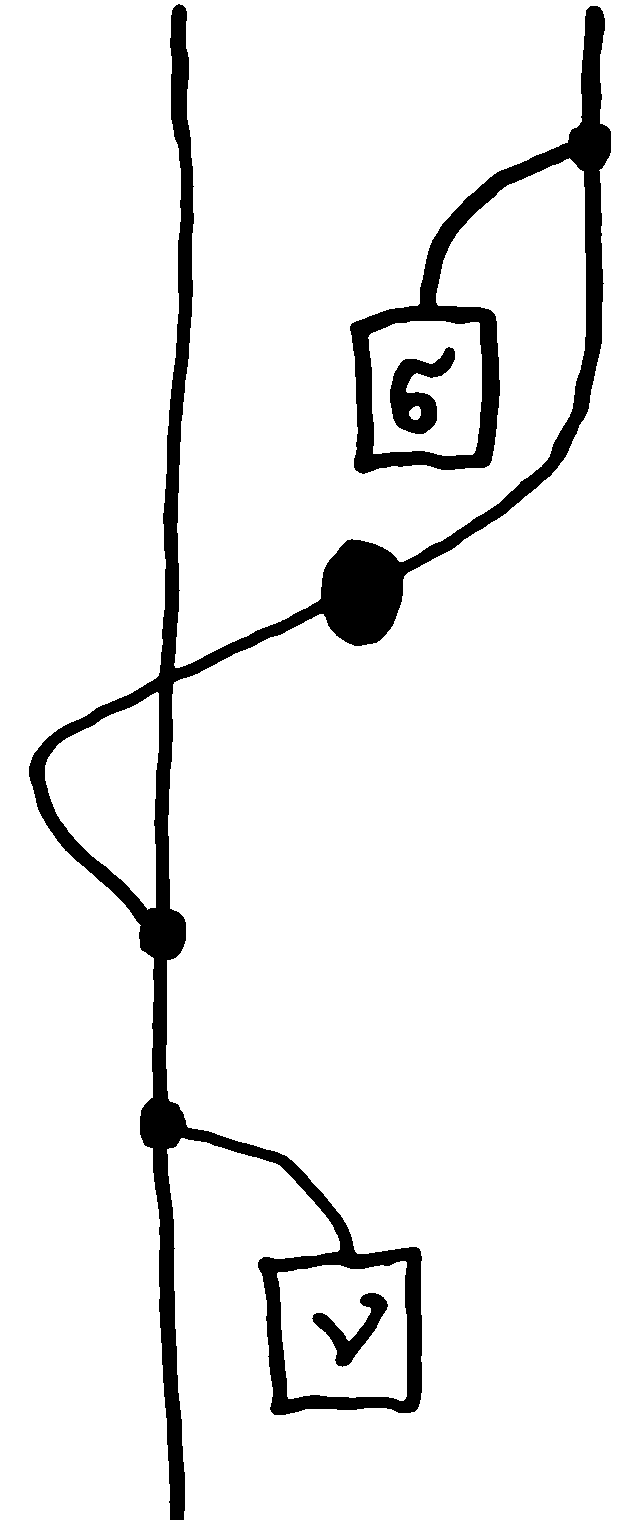} 
  \tag{H-5R} \label{eq:hexagon-5R}
\ee
\end{tabular}
\\
\begin{tabular}{p{19em}p{19em}}
\be
  \scanPIC{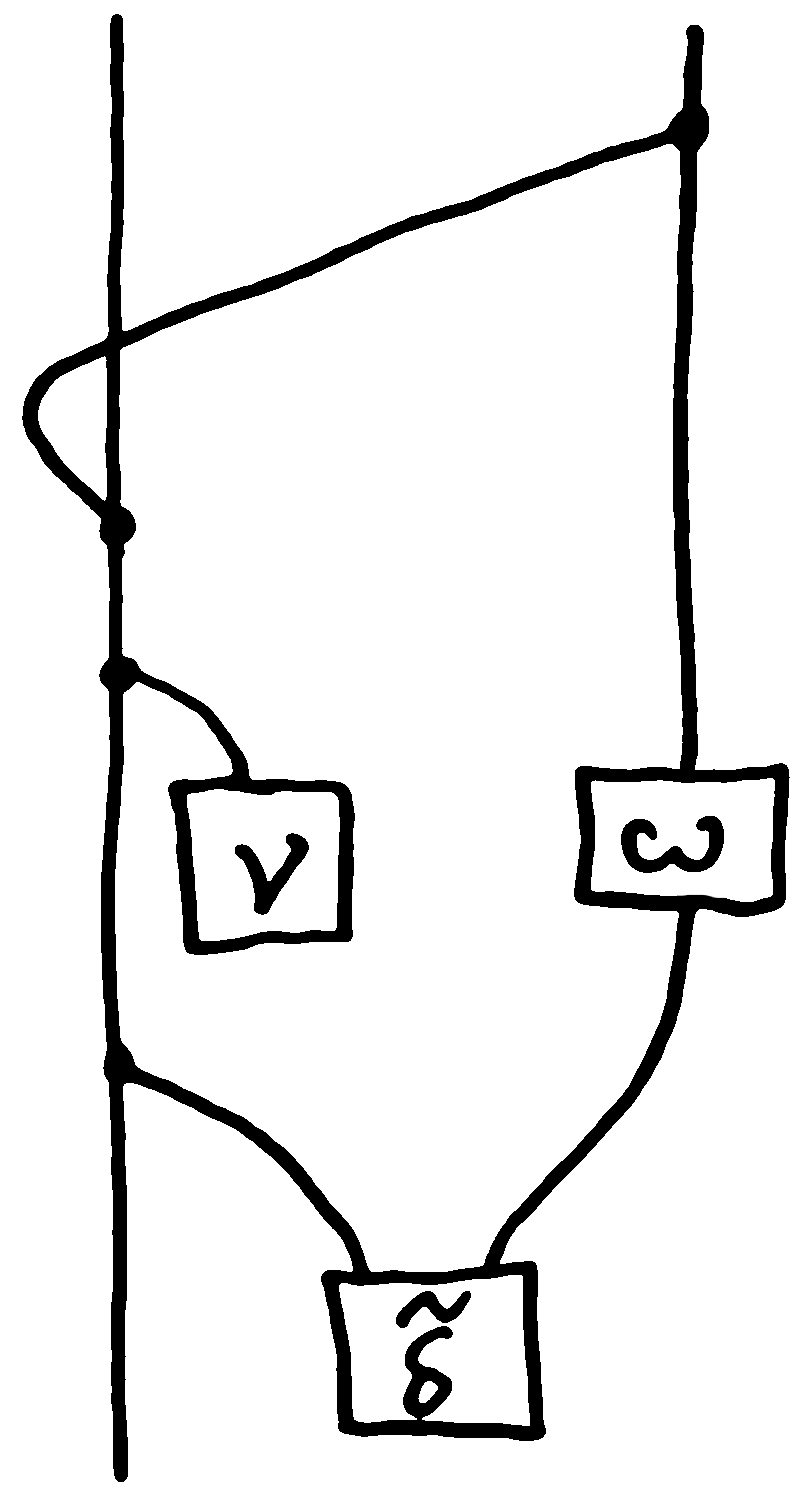} ~~=~~ \scanPIC{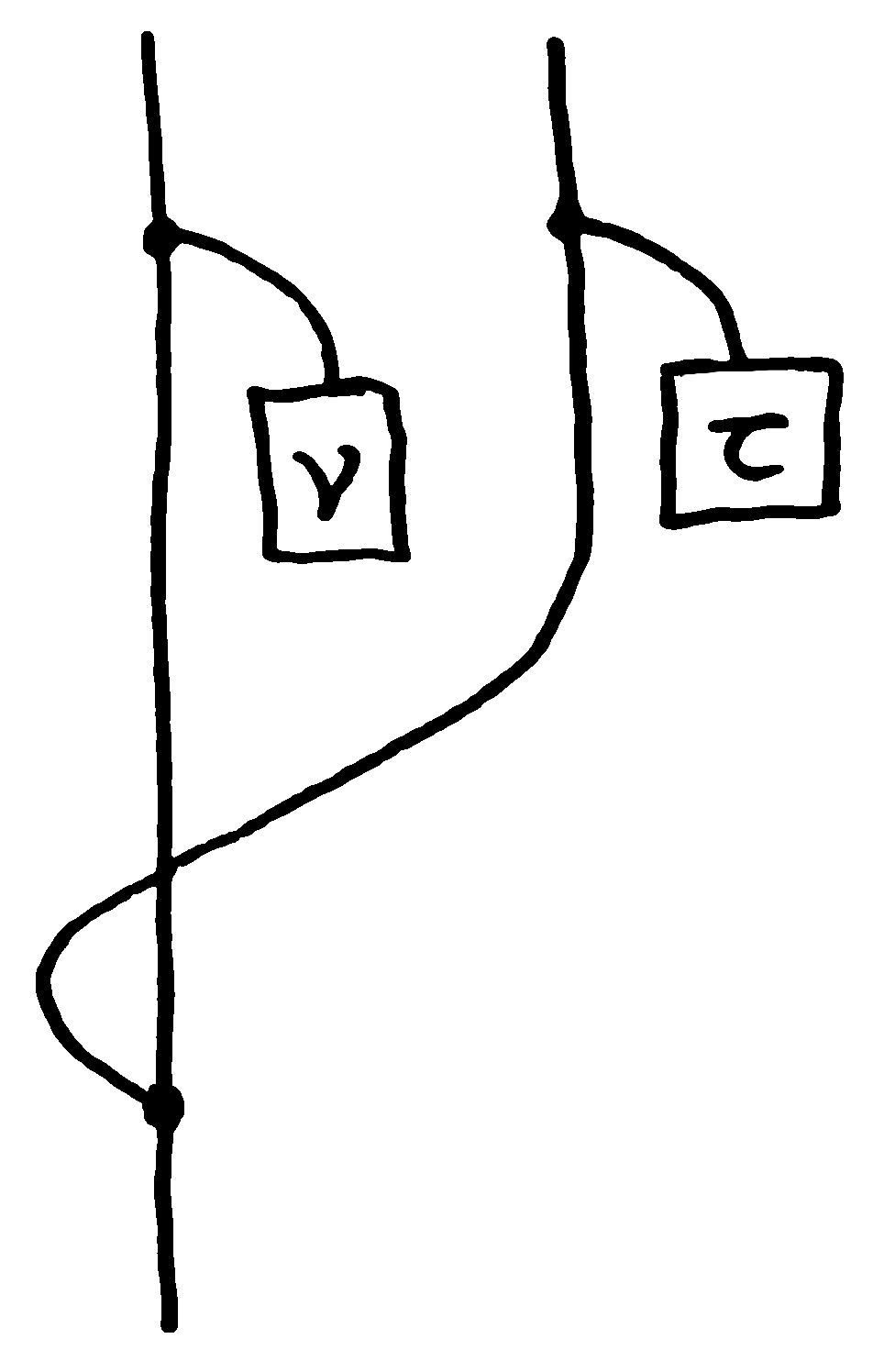} 
  \tag{H-6L} \label{eq:hexagon-6L} \ee
&
\be
  \scanPIC{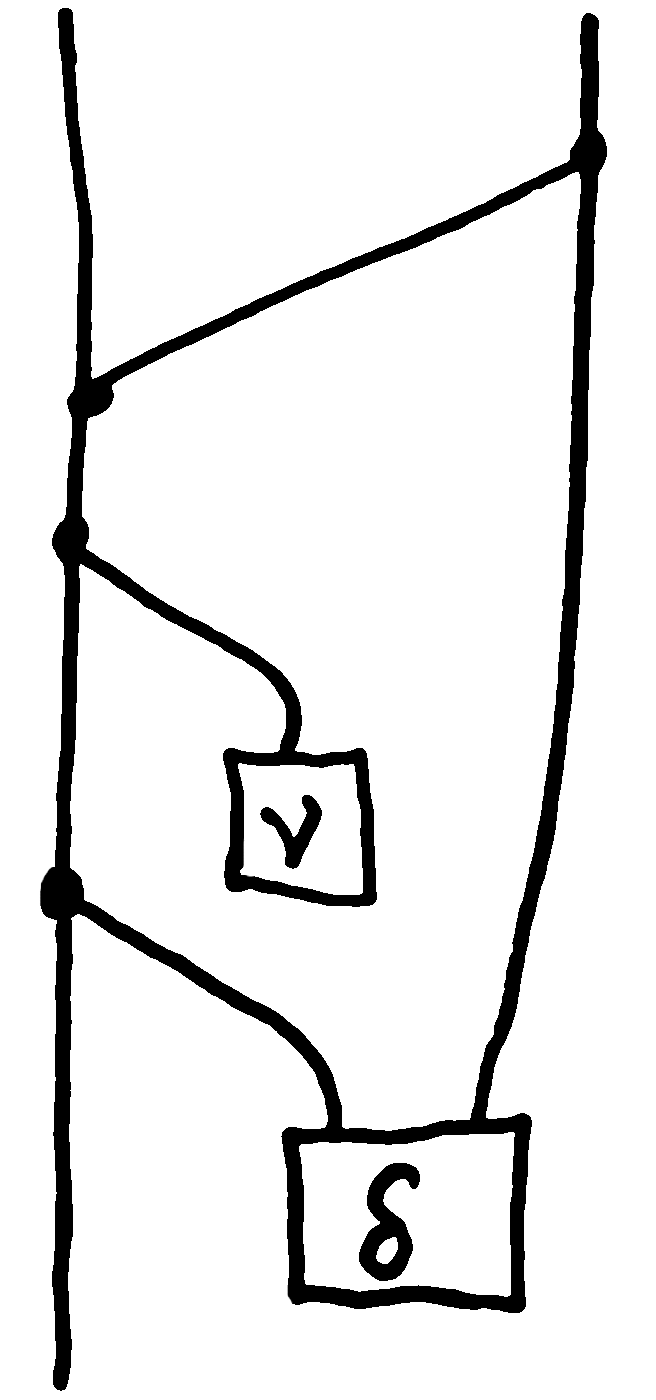} ~~=~~ \scanPIC{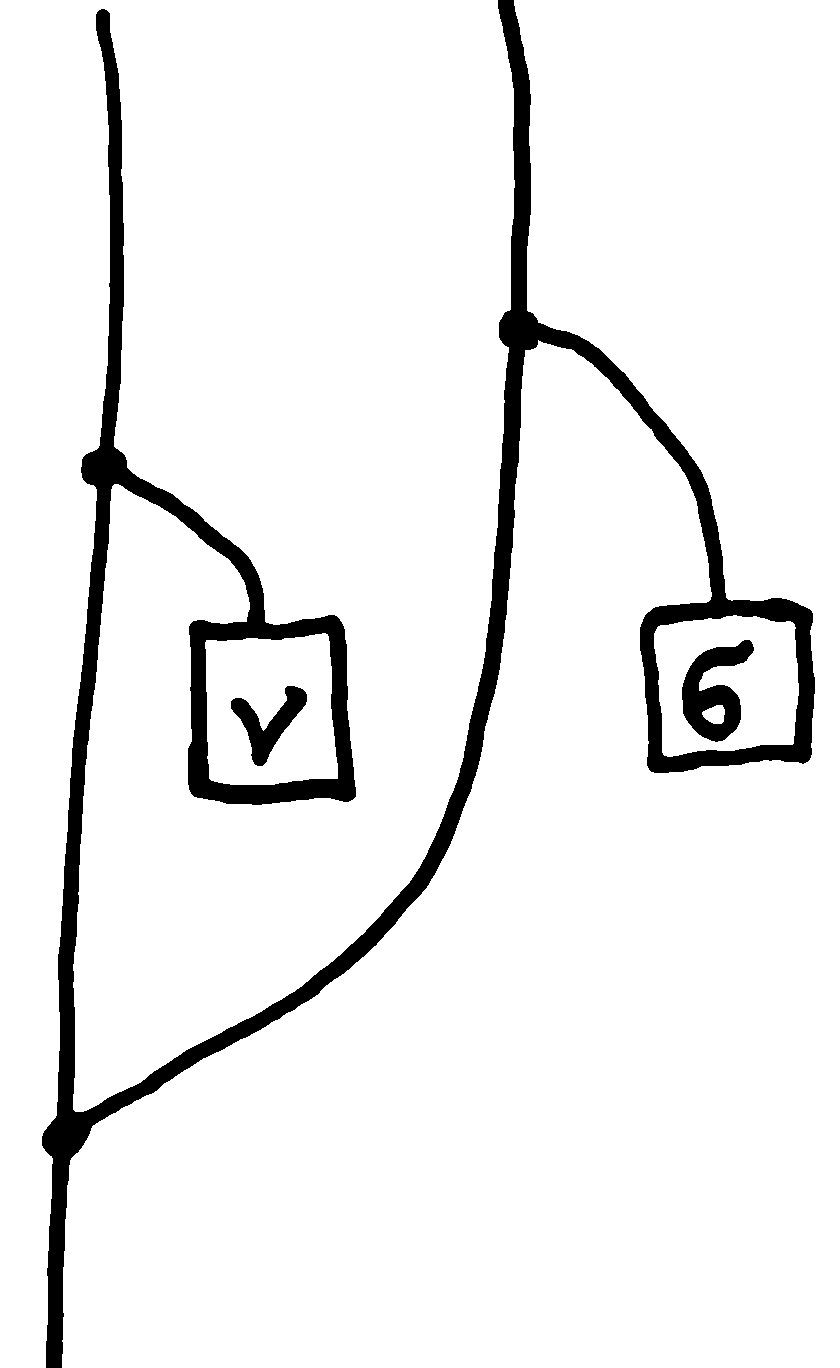} 
  \tag{H-6R} \label{eq:hexagon-6R}
\ee
\end{tabular}
\\
\begin{tabular}{p{19em}p{19em}}
\be
  \scanPIC{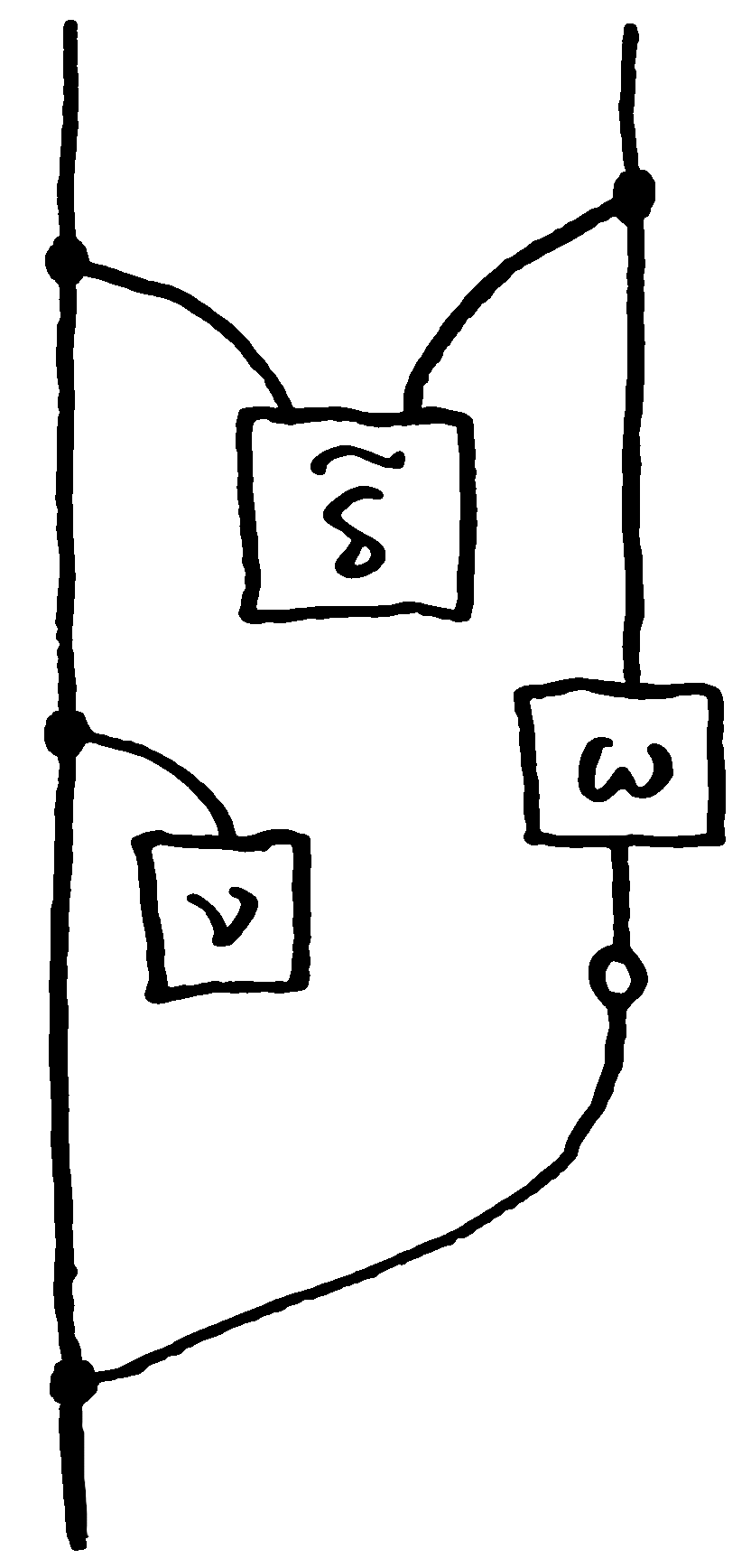} ~~=~~ \scanPIC{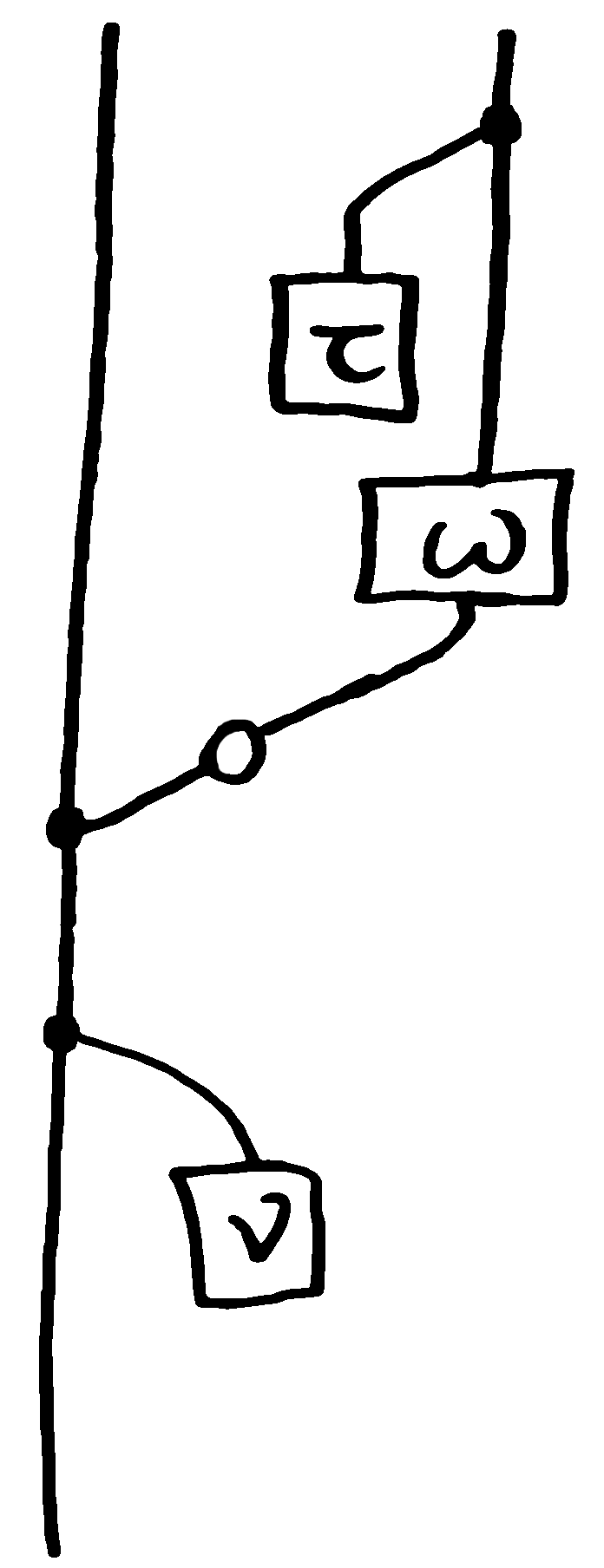} 
  \tag{H-7L} \label{eq:hexagon-7L}\ee
&
\be
  \scanPIC{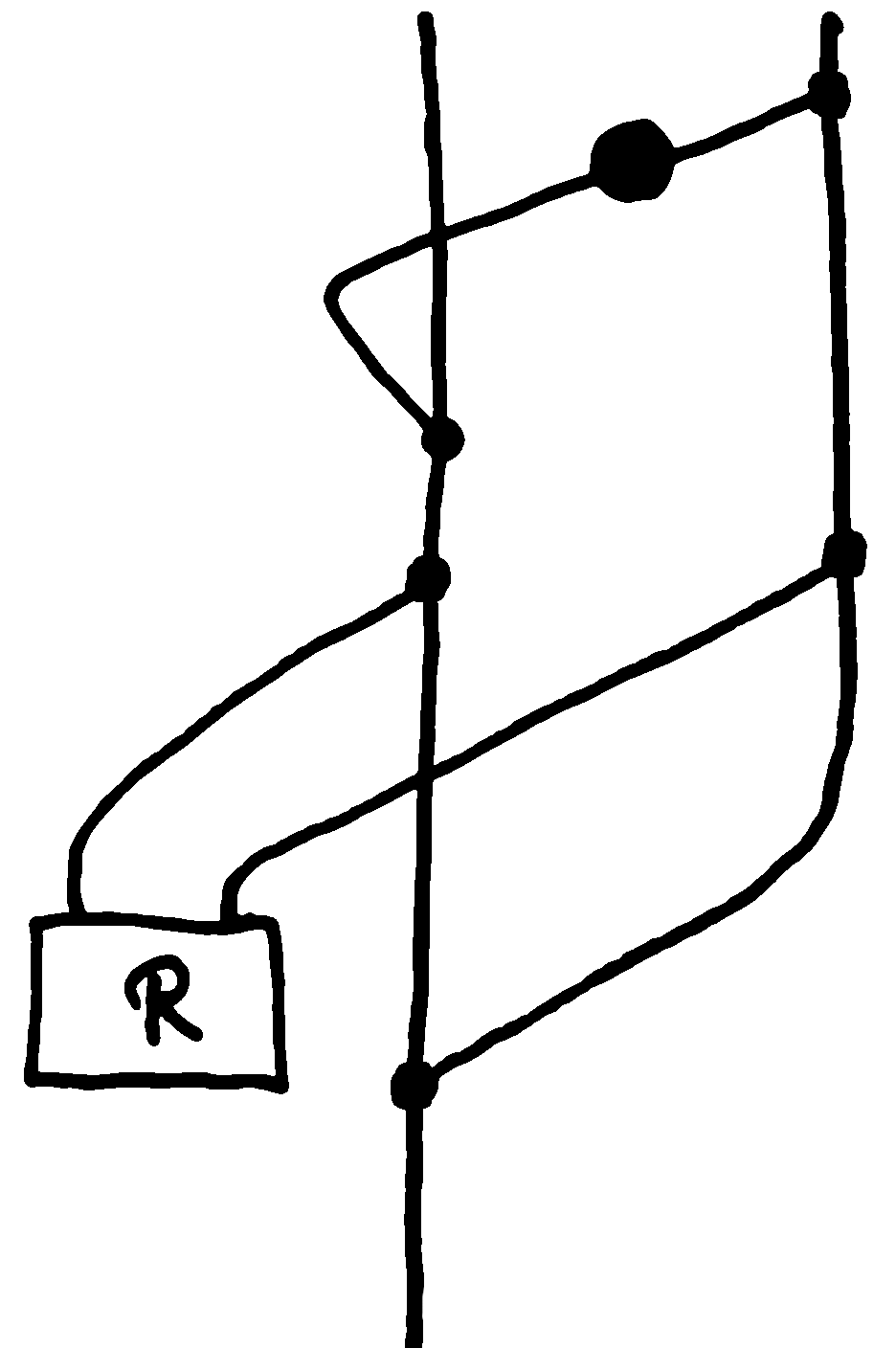} ~~=~~ \scanPIC{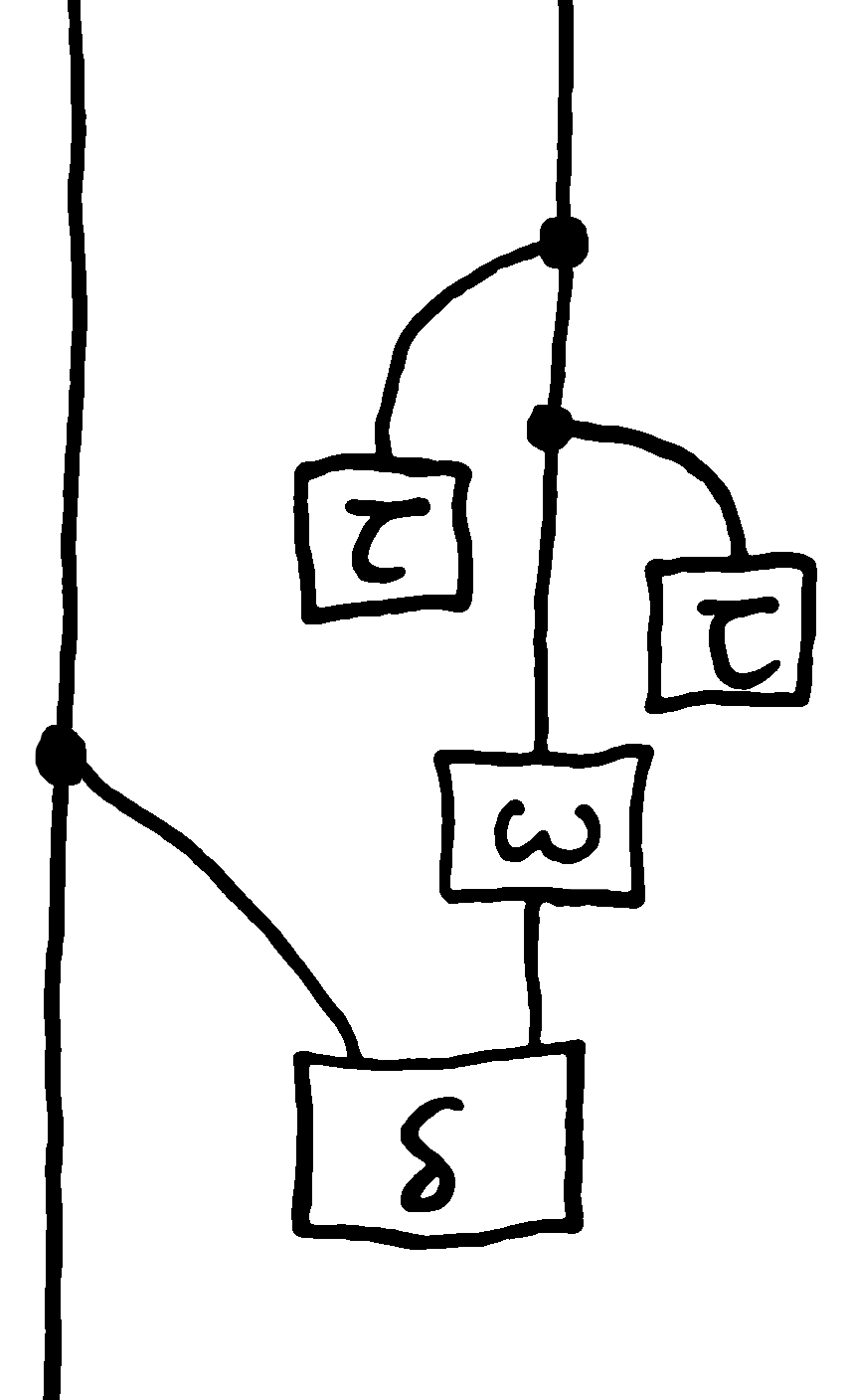} 
  \tag{H-7R} \label{eq:hexagon-7R} 
\ee
\end{tabular}
\\
\begin{tabular}{p{19em}p{19em}}
\be
  \scanPIC{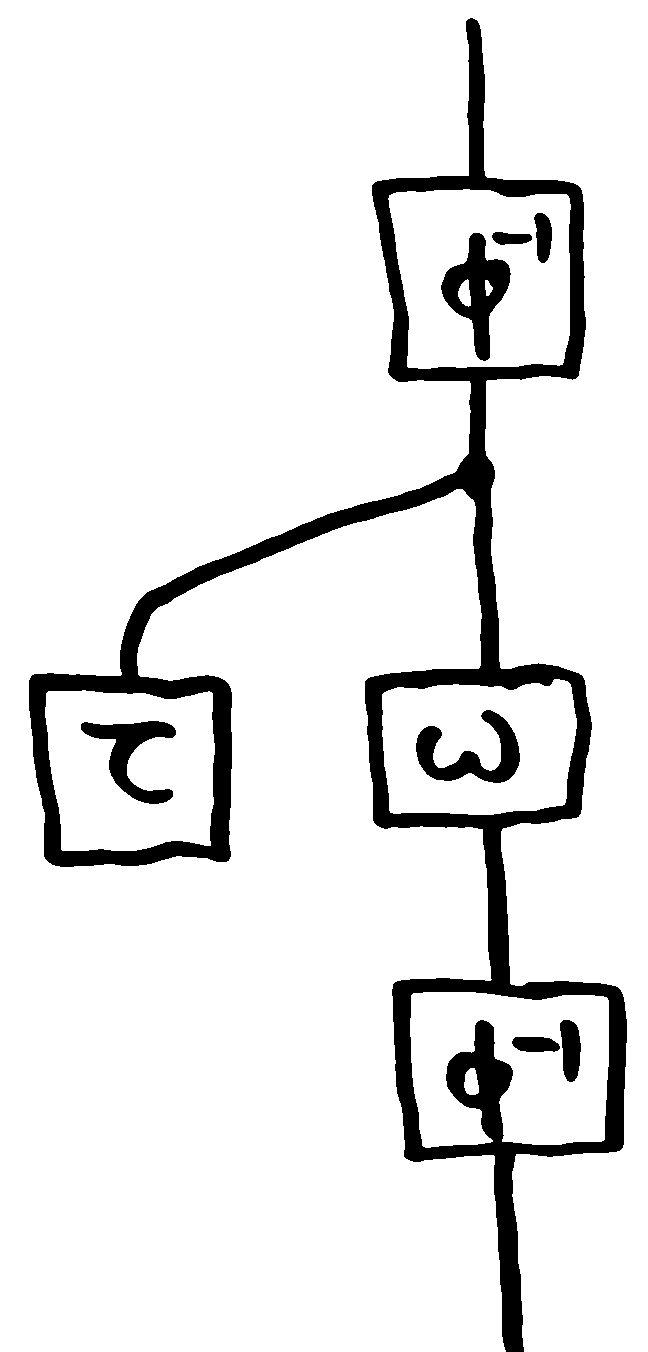} ~~=~~ \scanPIC{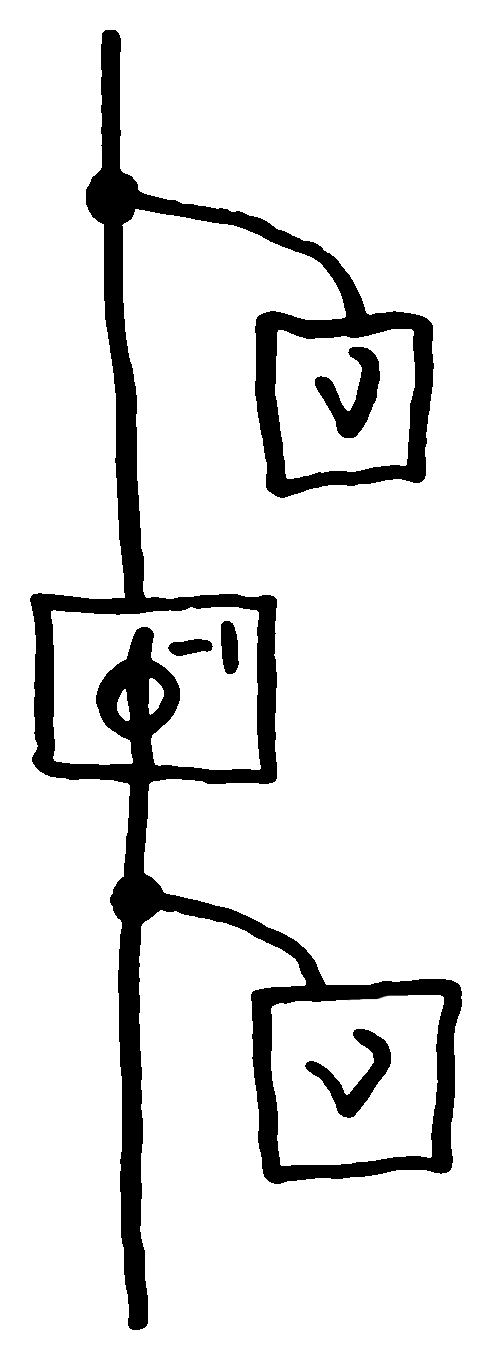} 
  \tag{H-8L} \label{eq:hexagon-8L}\ee
&
\be
  \scanPIC{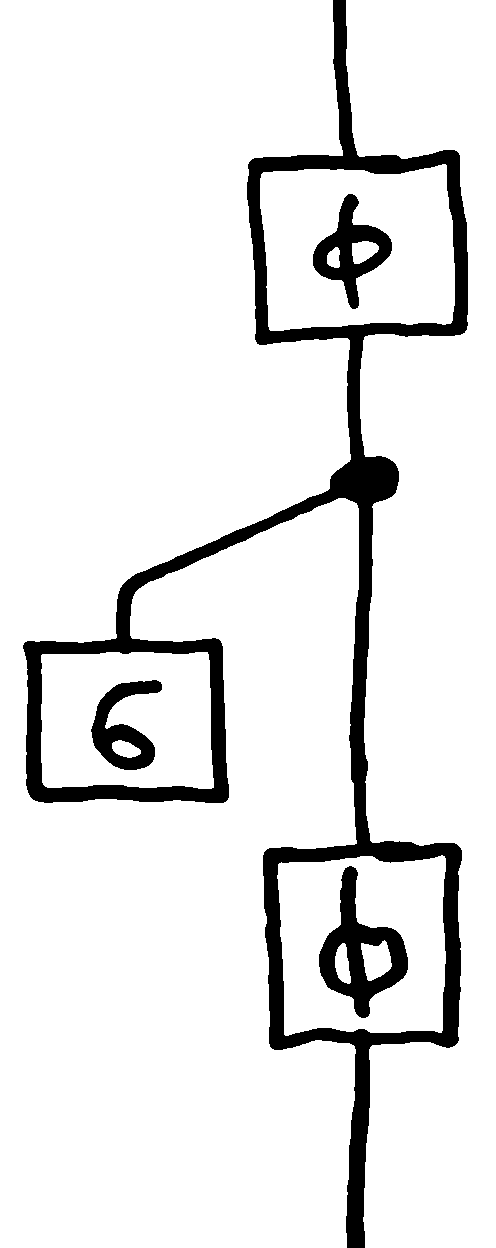} ~~=~~ \scanPIC{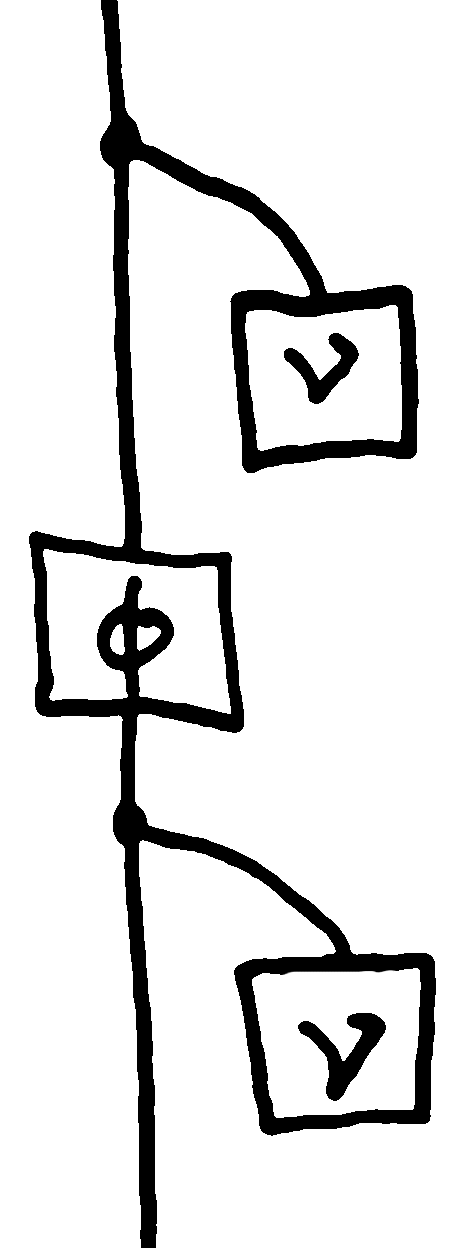} 
  \tag{H-8R} \label{eq:hexagon-8R}
\ee
\end{tabular}

\subsection{From hexagon to Hopf-algebraic data} \label{sec:hex-to-Hopf}

\begin{proposition} \label{prop:hex-to-Hopf}
Suppose that $R,\sigma,\tau,\nu$ from \eqref{eq:Rstn-def} satisfy \eqref{eq:R-intertwiner} and the 16 hexagon conditions \eqref{eq:hexagon-1L}--\eqref{eq:hexagon-8R}. Then $R,\tau,\nu$ are of the form \eqref{eq:R-etc-via-Hopf} with $\beta = \eps \circ \nu$ and conditions a)--e) in Theorem \ref{thm:main2}\,(i) are satisfied.
\end{proposition}

The proof will be given in the end of this section after a series of lemmas. For the discussion below, we will assume the conditions in Proposition \ref{prop:hex-to-Hopf} without further mention. 

\medskip

Using the fact that $\sigma$ has a multiplicative inverse, we have
\be \label{eq:gamma_via_sigma}
  \text{\eqref{eq:hexagon-2R}} ~\Leftrightarrow~
  \gamma 
  = ({}_{\sigma^{-1}}M \otimes M_{\sigma^{-1}}) \circ \Delta \circ \sigma \ .
\ee
If we form the multiplicative inverse in $H \otimes H$ of both sides, we see that
\be \label{eq:gamma-1_via_sigma}
  \gamma^{-1} = (M_\sigma \otimes {}_\sigma M) \circ \Delta \circ \sigma^{-1} \ .
\ee
We note at this point that since $\gamma$ and $\gamma^{-1}$ are non-degenerate copairings, and since $M_{\sigma^{\pm1}}$ and ${}_{\sigma^{\pm1}}M$ are isomorphisms, also the copairings $\Delta \circ \sigma^{\pm1}$ are non-degenerate.

Substituting the expression \eqref{eq:gamma-1_via_sigma} for $\gamma^{-1}$ into \eqref{eq:hexagon-2L} and \eqref{eq:hexagon-3L} leads to the following two expressions for $R$:
\begin{subequations}
\begin{align}
  R &= ({}_\sigma M \otimes {}_\sigma M) \circ \Delta \circ \sigma^{-1} && \text{from \eqref{eq:hexagon-3L}}
  \label{eq:R-via-sigma_1}
\\
  &= (M_\sigma \otimes M_\sigma) \circ \Delta^\mathrm{cop} \circ \sigma^{-1} && \text{from \eqref{eq:hexagon-2L}} \ ,
  \label{eq:R-via-sigma_2}
\end{align}
\end{subequations}
where $\Delta^\mathrm{cop} = c_{H,H} \circ \Delta$. For future use, we note that another way of writing \eqref{eq:hexagon-3L} is
\be\label{eq:R-via-gamma-1}
  R =  ( \Ad_\sigma \otimes \id ) \circ \gamma^{-1} \ .
\ee

\begin{lemma}\label{lem:A-prop-1}
We have
$$
\eps \circ \sigma = \eps \circ \sigma^{-1} = \id_\one 
\quad , \quad
\Delta^\mathrm{cop} \circ \Ad_\sigma = (\Ad_\sigma \otimes \Ad_\sigma) \circ \Delta \ .
$$
In particular, $\eps \circ \Ad_\sigma = \eps$, so that $\Ad_\sigma$ is a coalgebra map $H \to H_\mathrm{cop}$.
\end{lemma}

\begin{proof}
Inserting \eqref{eq:hexagon-5R} into $(\id \otimes \eps) \circ (-) \circ \eta$ and using $(\id \otimes \eps) \circ \delta = \eta$ results in $\nu = \eps(\sigma) \cdot \nu$. Since $\nu$ has a multiplicative inverse, this shows $\eps(\sigma) = \id_\one$. Composing $\mu \circ (\sigma \otimes \sigma^{-1}) = \eta$ with $\eps$ shows that also $\eps(\sigma^{-1}) = \id_\one$.

To understand the compatibility of $\Ad_\sigma$ with $\Delta$ consider the equalities
\begin{align}
\text{lhs of \eqref{eq:R-intertwiner}} & ~\overset{\eqref{eq:R-via-sigma_2}}{=}~ 
\scanpic{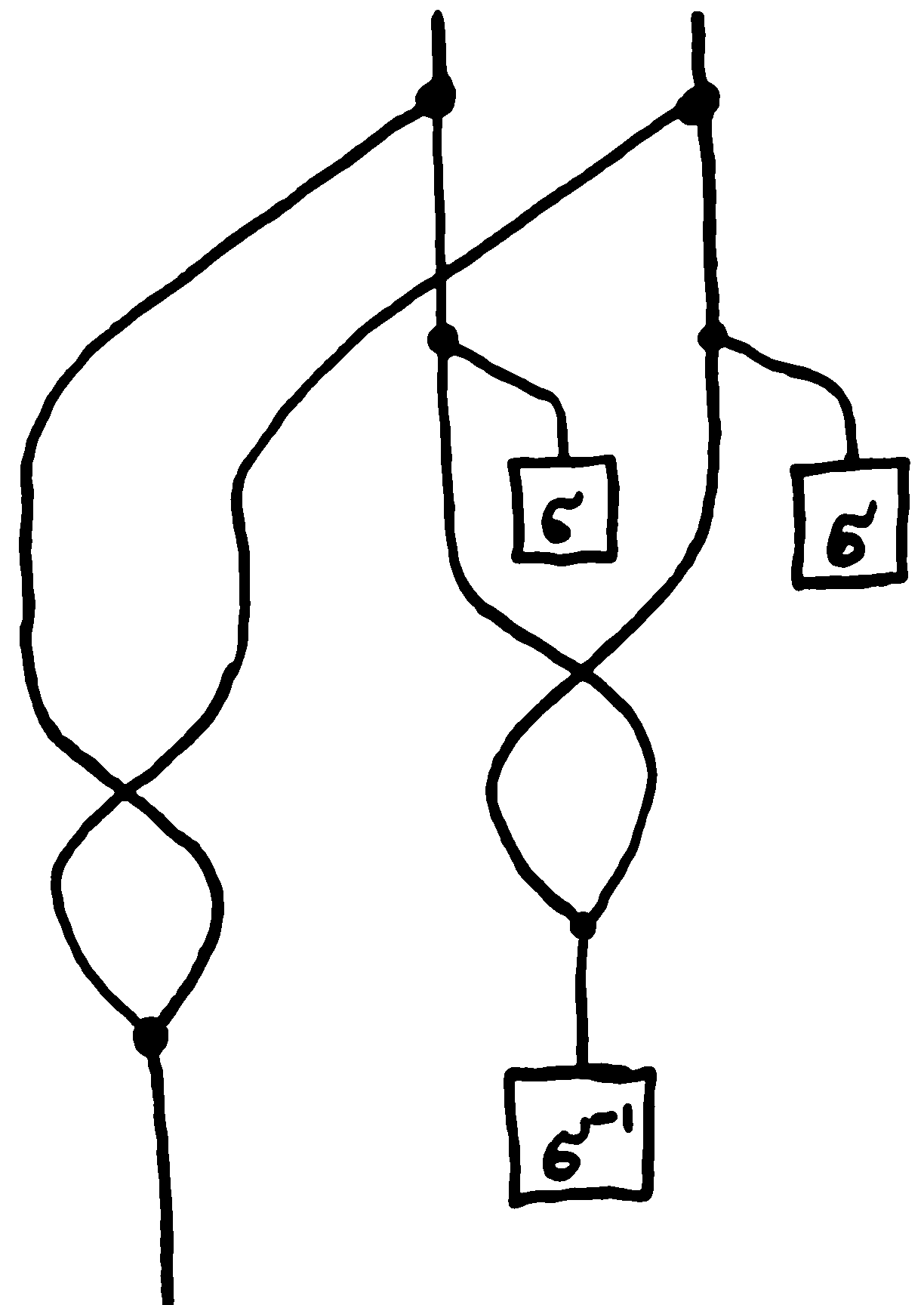}
~\overset{(*)}{=}~ 
\scanpic{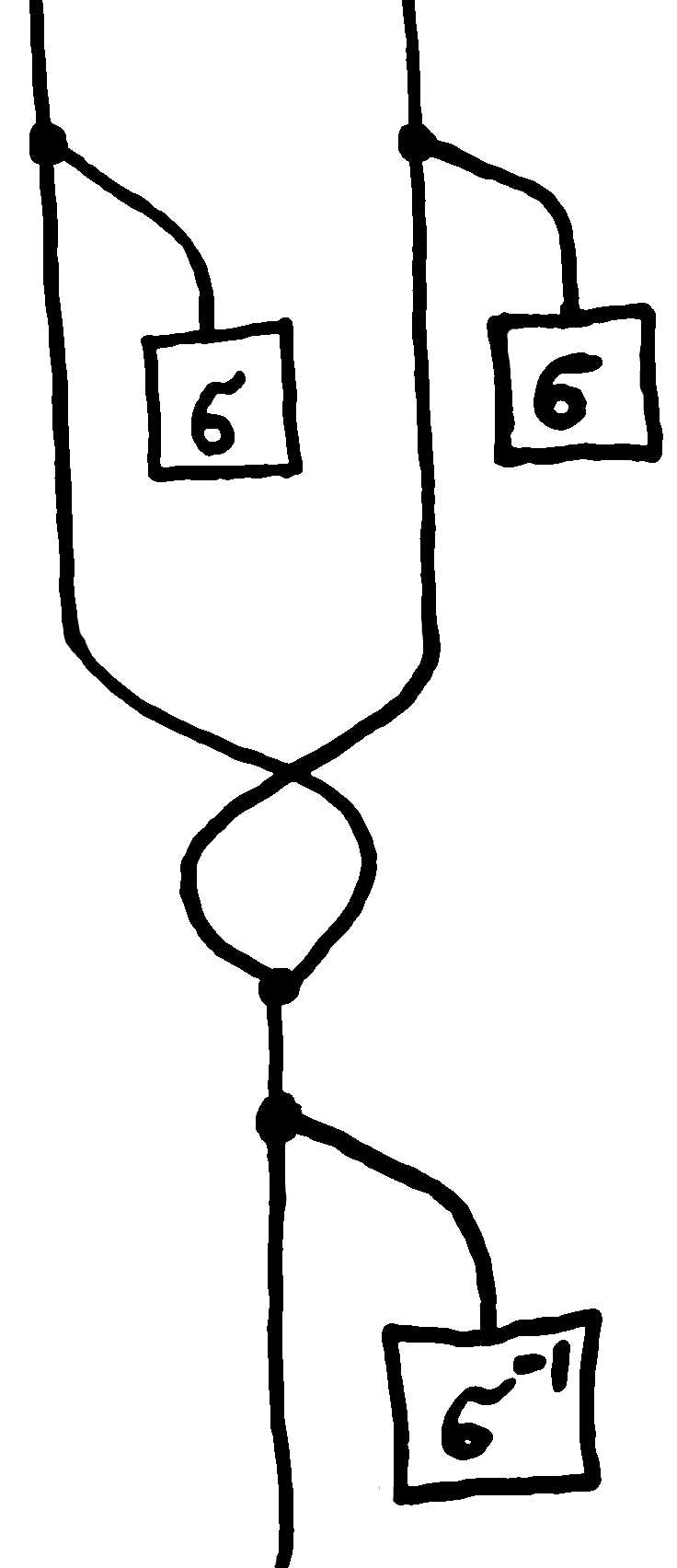} 
\quad ,
\nonumber \\
\text{rhs of \eqref{eq:R-intertwiner}} & ~\overset{\eqref{eq:R-via-sigma_1}}{=}~ 
\scanpic{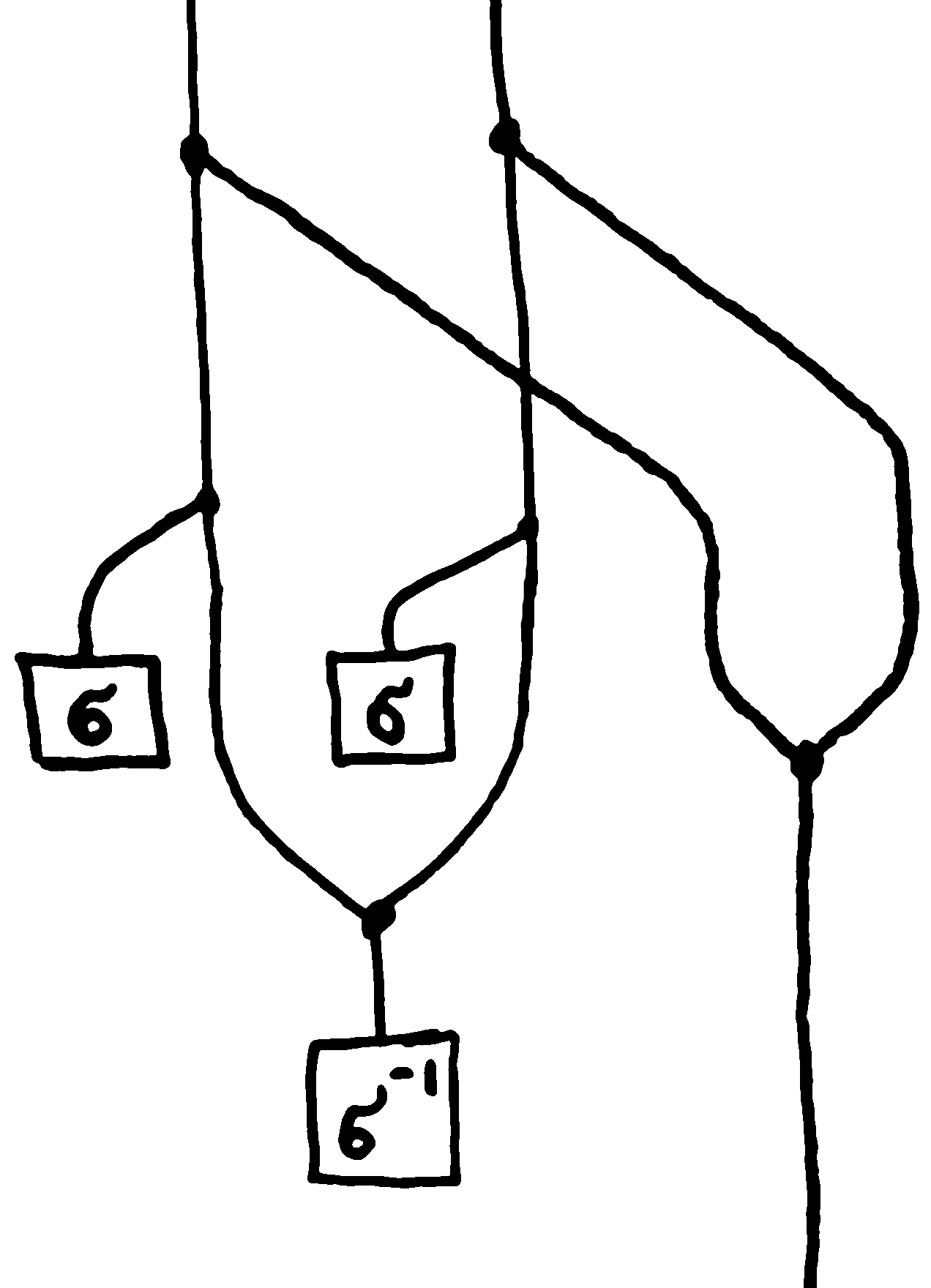}
~\overset{(*)}{=}~ 
\scanpic{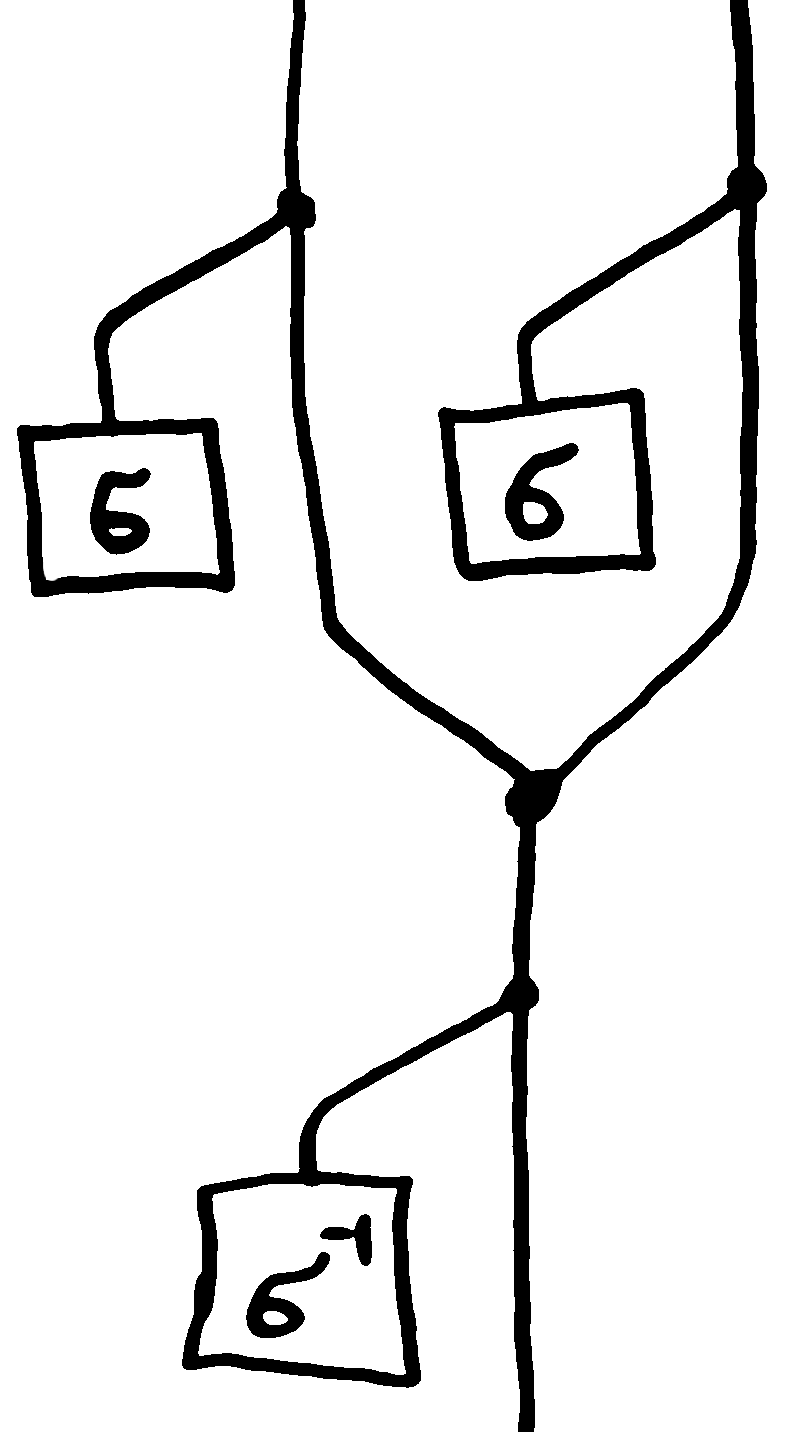} 
\quad ,
\label{eq:qcocomm-vs-Ad-anticoalg}
\end{align}
Step (*) is the algebra-map property of the coproduct. Equating the two expressions and inserting them into $(M_{\sigma^{-1}} \otimes M_{\sigma^{-1}}) \circ (-) \circ {}_\sigma M$ gives the second identity in the statement of the lemma.
\end{proof}

\begin{lemma}\label{lem:S2=GA-2}
(i) $(\Ad_\sigma^{\,2} \otimes \id)\circ \gamma = c_{H,H} \circ \gamma = (\id \otimes \Ad_\sigma^{-2})\circ \gamma$.
\\
(ii) $S^2 = \Ad_g \circ \Ad_\sigma^{-2}$.
\end{lemma}

\begin{proof}
To see (i) note that 
\be\begin{array}{l}
(\Ad_\sigma \otimes \Ad_\sigma) \circ c_{H,H} \circ \gamma
\overset{(1)}=
(\Ad_\sigma \otimes \Ad_\sigma) \circ (M_{\sigma^{-1}} \otimes {}_{\sigma^{-1}}M) \circ \Delta^\mathrm{cop} \circ \sigma
\\[.5em]
\qquad \overset{(2)}=
(M_{\sigma^{-1}} \otimes {}_{\sigma^{-1}}M) \circ \Delta \circ \sigma
\overset{(3)}=
(\Ad_\sigma \otimes \Ad_\sigma^{-1}) \circ \gamma \ ,
\end{array}
\ee
where in (1) the expression \eqref{eq:gamma_via_sigma} for $\gamma$ was inserted, in (2) the anti-coalgebra property of $\Ad_\sigma$ is used (Lemma \ref{lem:A-prop-1}), as well as $\Ad_\sigma(\sigma)=\sigma$. Finally, step (3) is once more \eqref{eq:gamma_via_sigma}. This shows the second identity claimed in (i). The first follows by composing the second one with $c_{H,H} \circ (\id \otimes \Ad_\sigma^{\,2})$.

For (ii) we start from the equality of the two expressions for $\delta$ given in \eqref{eq:delta-expression-ab}. This is equivalent to
\be
  \big\{ \id \otimes (\Ad_g^{-1} \circ S^2) \big\} \circ \gamma = c_{H,H} \circ \gamma \ .
\ee
Substituting the second identity in (i) and using non-degeneracy of $\gamma$ results in (ii).
\end{proof}

\begin{lemma}\label{lem:S-sigma}
(i) 
$S \circ \sigma = M_{g^{-1}} \circ \sigma$, as well as
$S^{-1} \circ \sigma = {}_{g^{-1}}M \circ \sigma$,
$S \circ \sigma^{-1} = {}_gM \circ \sigma^{-1}$, 
$S^{-1} \circ \sigma^{-1} = M_g \circ \sigma^{-1}$.
\\
(ii) $S \circ \Ad_\sigma^{-1} = \Ad_\sigma \circ \Ad_g^{-1} \circ S$ and $S \circ \Ad_\sigma = \Ad_\sigma \circ S^{-1}$.
\end{lemma}

\begin{proof}
For (i) we do a longer calculation than necessary which however will be useful again in Section \ref{sec:Hopf-to-hex}. Namely, the left hand side of \eqref{eq:hexagon-5L} can be rewritten as follows:
\begin{align}
&\text{lhs of \eqref{eq:hexagon-5L}}
~\overset{(1)}=~
\scanpic{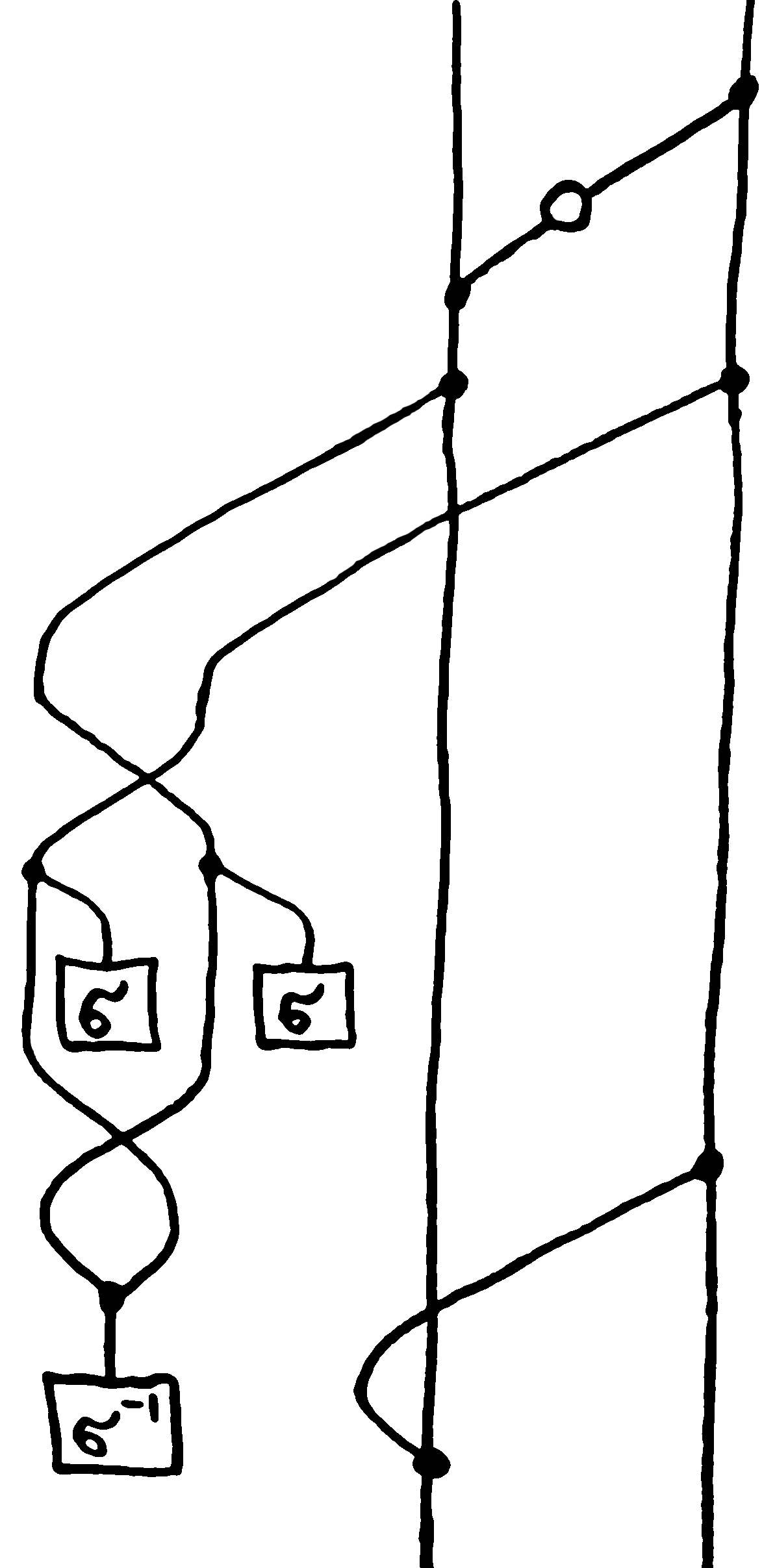}
~\overset{(2)}=~
\scanpic{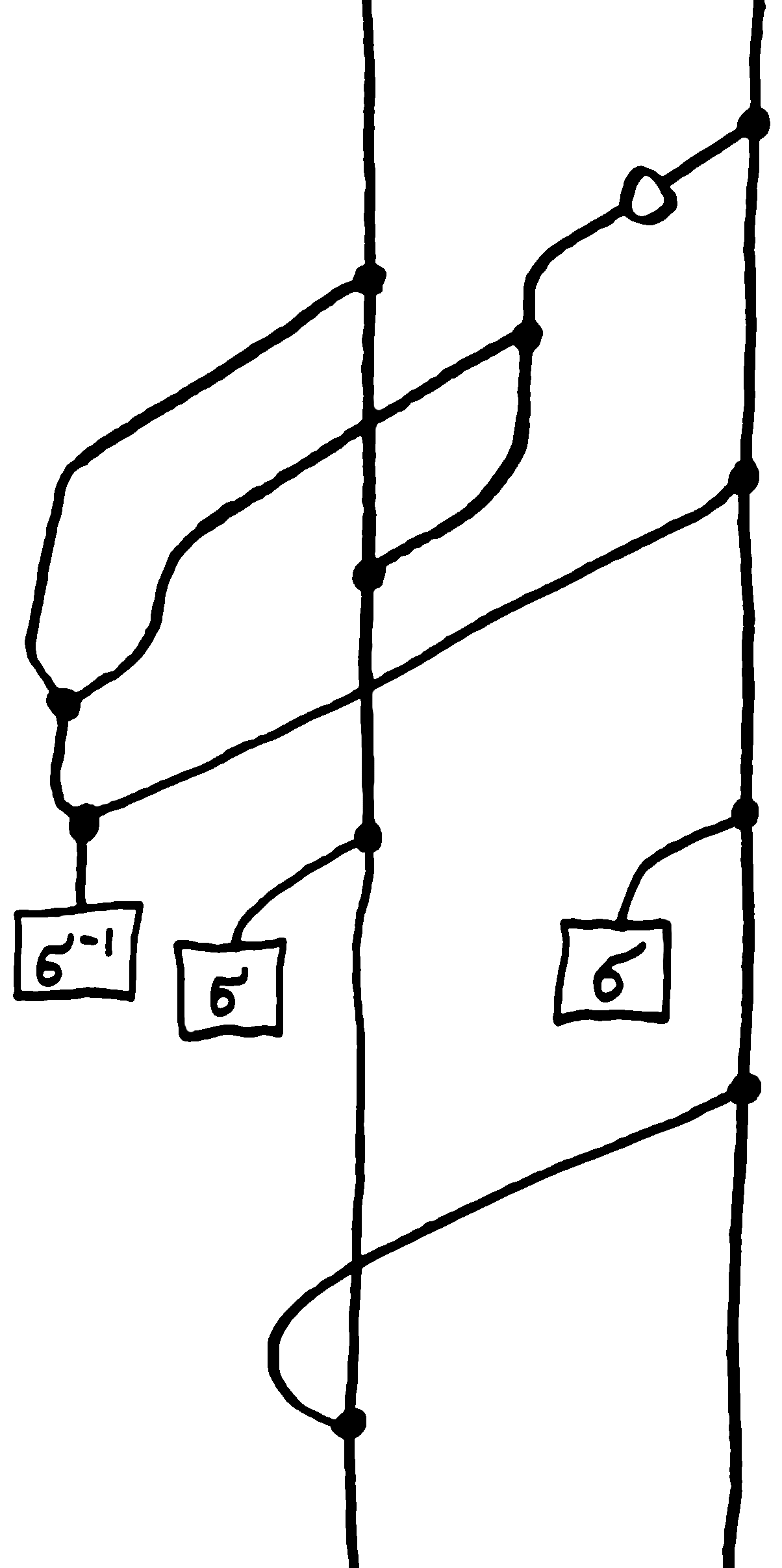}
~=~
\scanpic{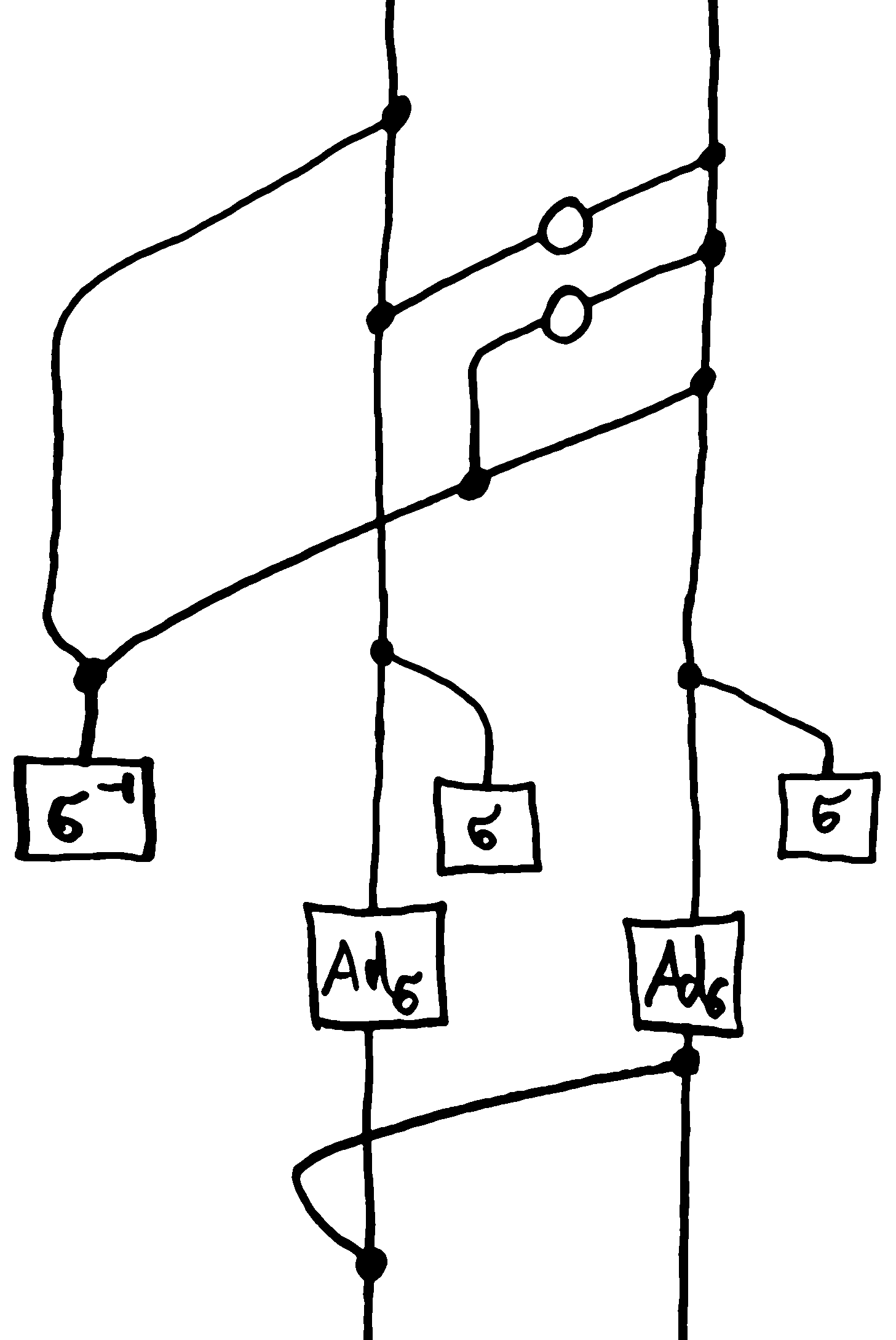}
\nonumber \\
&\overset{(4)}=~
\scanpic{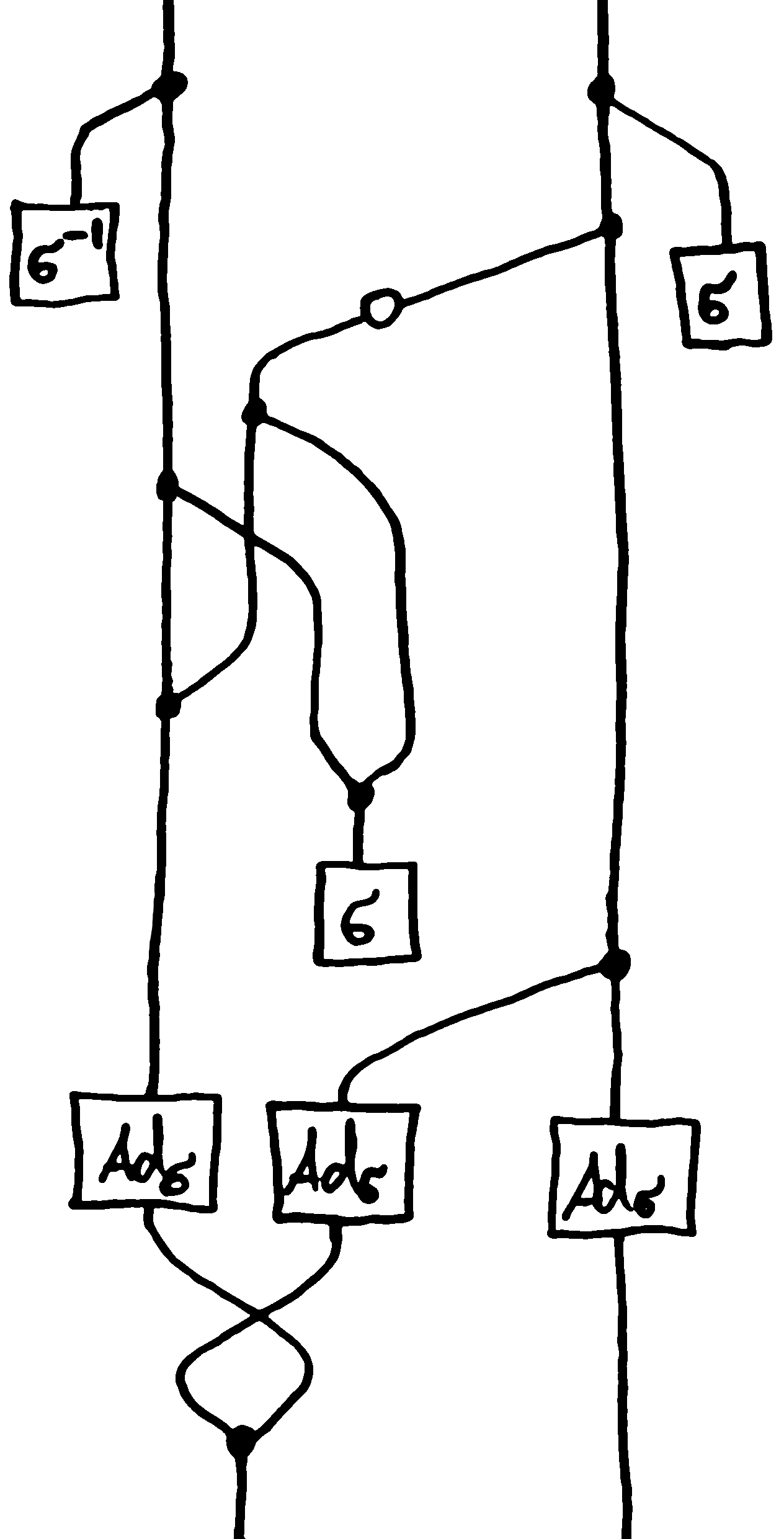}
~\overset{(5)}=~
\scanpic{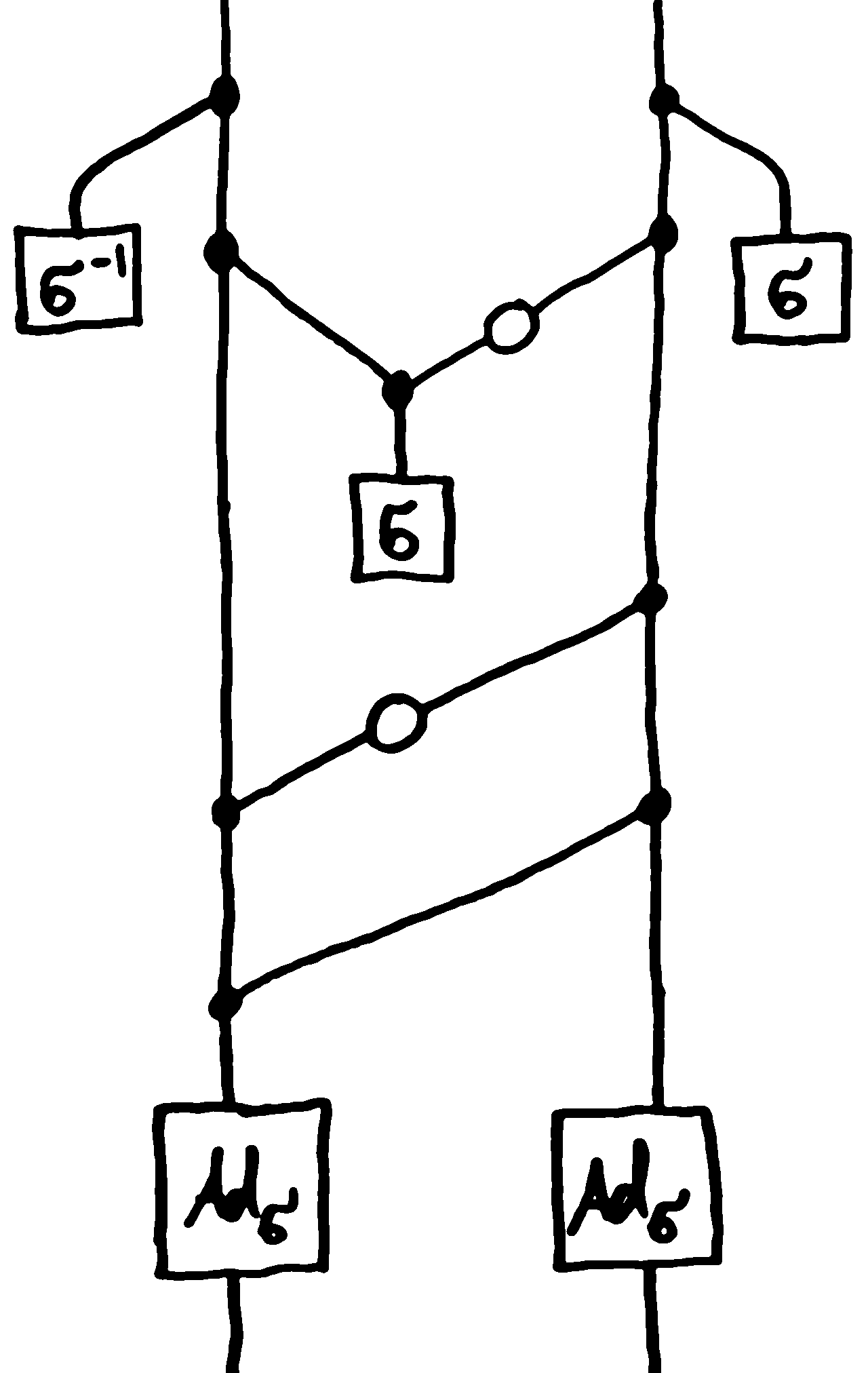}
~=~
\scanpic{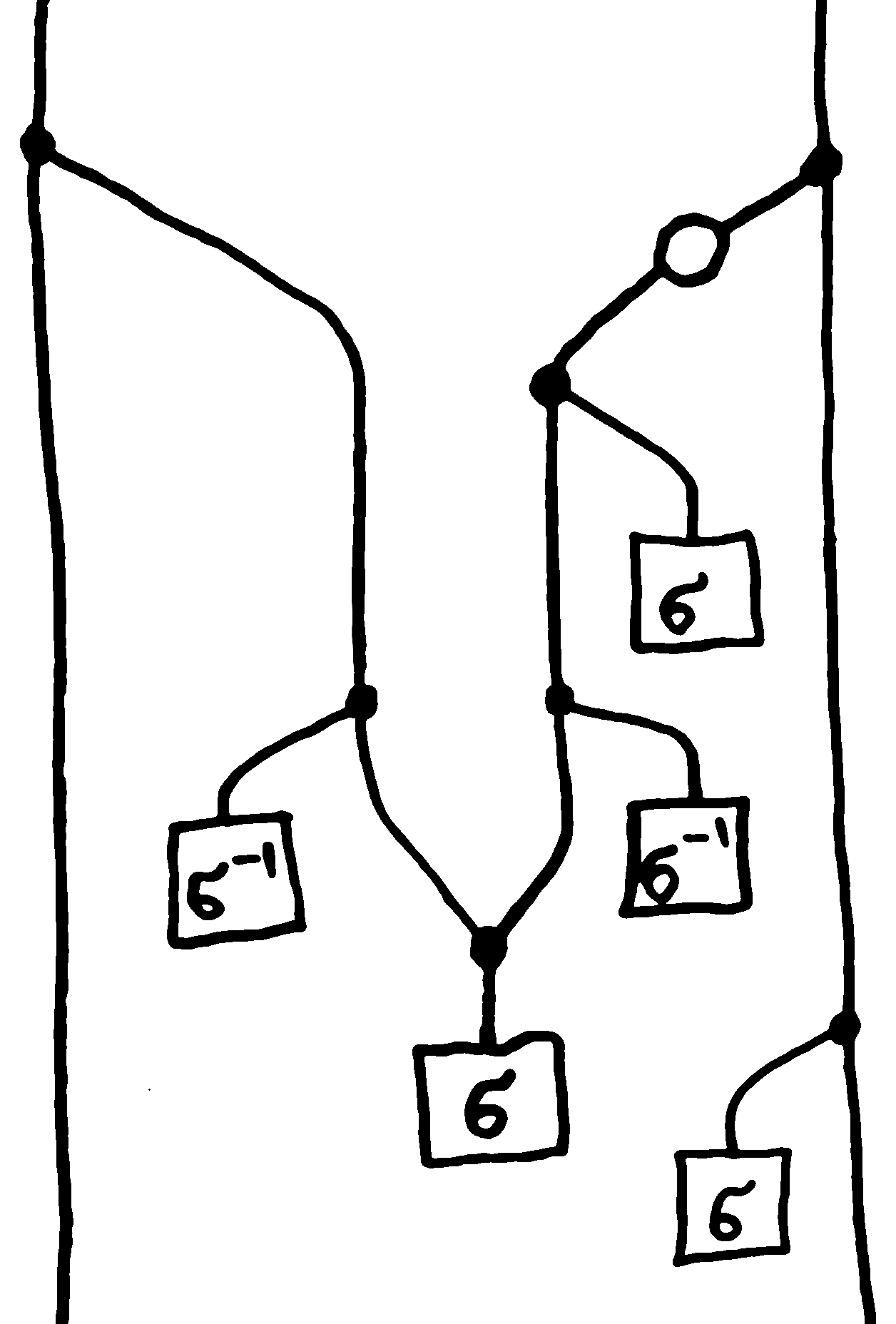}
~\overset{(7)}=~
\scanpic{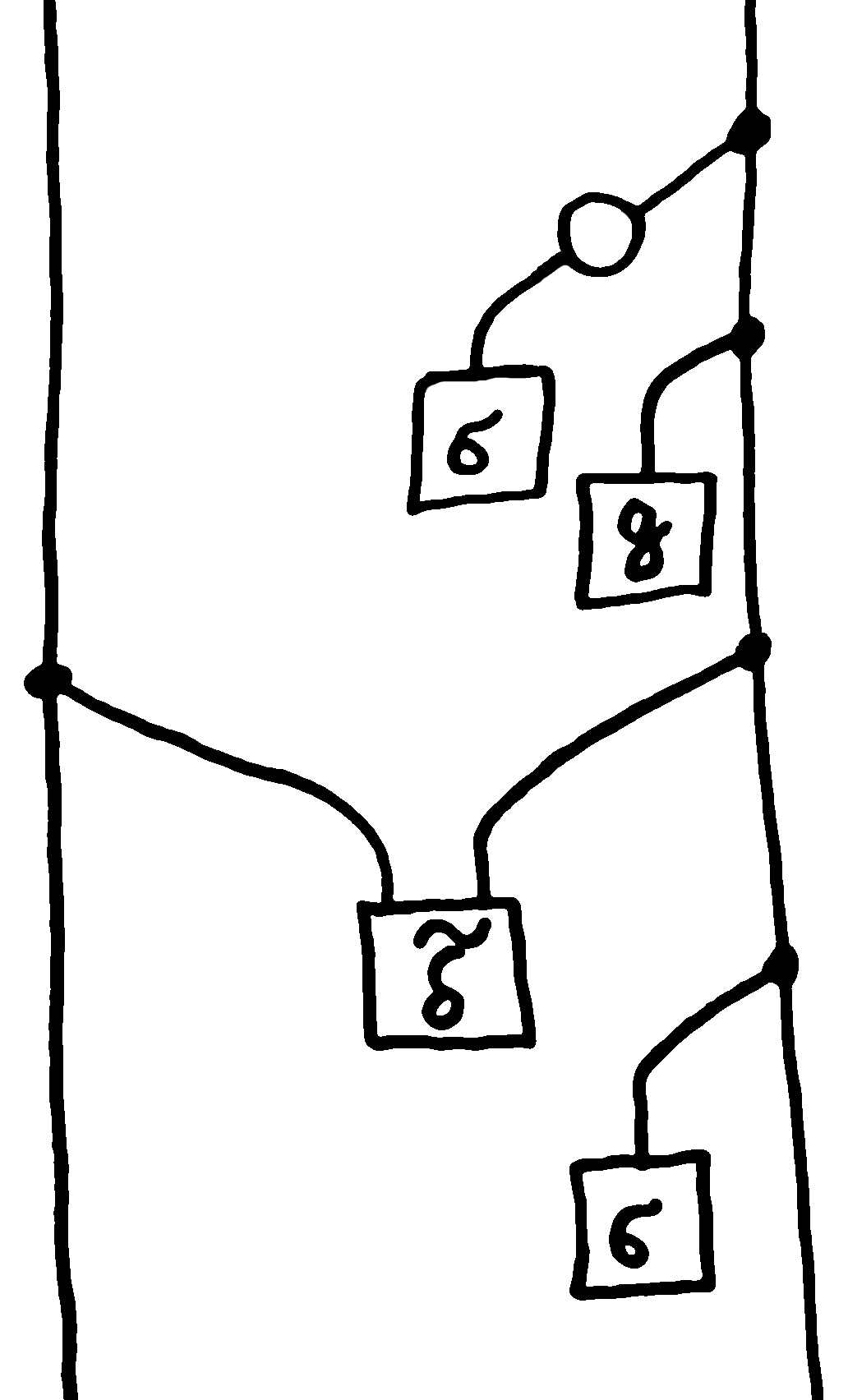}
\label{eq:hexagon-5L-alt}
\end{align}
In step (1) expression \eqref{eq:R-via-sigma_2} for $R$ has been substituted. In step (2), first the $\sigma$-morphisms are moved to their new position (associativity) and then the algebra-map property of the coproduct is used. In step (4), the bubble-property is employed, as well as the fact that $\Ad_\sigma$ is an algebra-map and once more the algebra-map property of the coproduct. Step (5) follows as $\Ad_\sigma$ is a coalgebra anti-automorphism (Lemma \ref{lem:A-prop-1}). For step (7) insert expression \eqref{eq:gamma_via_sigma} for $\gamma$, use Lemma \ref{lem:gamma-inv-via-S}\,(ii) to combine it with $S$ to $\gamma^{-1}$, then insert a pair ${}_gM \circ {}_{g^{-1}}M$ and substitute the definition of $\tilde\delta$ in \eqref{eq:tilde-delta-def}.  

If we compose the right hand sides of \eqref{eq:hexagon-5L-alt} and \eqref{eq:hexagon-5L} with $\eta \otimes \eta$, we arrive at the equality
\be
   \big\{ \id \otimes ({}_{S \circ \sigma}M \circ {}_gM \circ M_\sigma) \big\} \circ \tilde\delta 
   = 
   \big\{ \id \otimes ({}_{\sigma}M \circ M_\sigma) \big\} \circ \tilde\delta 
   \ .
\ee
Since $\tilde\delta$ is non-degenerate, this implies ${}_{S \circ \sigma}M \circ {}_gM \circ M_\sigma = {}_{\sigma}M \circ M_\sigma$. Composing this identity with $\mu \circ (g^{-1} \otimes \sigma^{-1})$ gives the first identity stipulated in (i). Composing the first identity with ${}_{g^{-1}}M \circ S^{-1} = S^{-1} \circ M_g$ gives the second identity. The remaining two identities are the multiplicative inverses of the first two.

The first equality in (ii) follows from
\be
  S \circ \Ad_{\sigma^{-1}} 
  = 
  \Ad_{S \circ \sigma} \circ S
  \overset{\text{(i)}}= 
  \Ad_{M_{g^{-1}} \circ \sigma} \circ S
  =
  \Ad_{\sigma} \circ \Ad_{g^{-1}} \circ S \ .
\ee
Taking the inverse on both sides gives $\Ad_\sigma \circ S^{-1} = S^{-1} \circ \Ad_g \circ \Ad_\sigma^{-1}$, or, equivalently, $S \circ \Ad_\sigma \circ S^{-1} = \Ad_g \circ \Ad_\sigma^{-1}$. Using this we compute
\be
S \circ S
\overset{\text{Lem.\,\ref{lem:S2=GA-2}}}=
\Ad_g \circ \Ad_\sigma^{-1} \circ \Ad_\sigma^{-1}
=
S \circ \Ad_\sigma \circ S^{-1}\circ \Ad_\sigma^{-1} \ ,
\ee
which implies the second equality in (ii).
\end{proof}

\begin{lemma}\label{lem:nu-tau}
We have $\tau = {}_{g^{-1}}M \circ M_{g^{-1}} \circ \sigma$ and $\nu = \beta \cdot \sigma^{-1}$ with $\beta = \eps(\nu)$. The element $\beta \in \End(\one)$ is invertible.
\end{lemma}

\begin{proof}
Inserting \eqref{eq:hexagon-5R} into $(\eps \otimes \id) \circ (-) \circ \eta$ gives $\eps(\nu) \cdot g = {}_\sigma M \circ S^{-1} \circ \nu$. This is equivalent to $\nu = \beta \cdot S \circ M_g \circ \sigma^{-1}$. Together with $M_g \circ \sigma^{-1} = S^{-1} \circ \sigma^{-1}$ (Lemma \ref{lem:S-sigma}\,(i)) one obtains the expression for $\nu$. Multiplying by $\nu^{-1}$ and composing with $\eps$ shows that $\beta$ is an invertible element of $\End(\one)$.
To get the expression for $\tau$ insert \eqref{eq:hexagon-7L} into $(\eps \otimes \id) \circ (-) \circ \eta$. This gives $\beta \cdot g^{-1} = {}_\tau M \circ S \circ \nu$, where we used naturality of $\omega$. Substituting the expression for $\nu$ and solving for $\tau$ gives $\tau = {}_{g^{-1}}M \circ S \circ \sigma$. Applying Lemma \ref{lem:S-sigma}\,(i) results in the formula for $\tau$.
\end{proof}

\begin{lemma}\label{lem:g-sig-g=sig}
We have ${}_gM \circ \sigma = M_{g^{-1}} \circ \sigma$.
\end{lemma}

\begin{proof}
We need the following reformulation of the left hand side of \eqref{eq:hexagon-7R}:
\be \label{eq:7R-aux1}
\text{lhs of \eqref{eq:hexagon-7R}} 
~\overset{\text{\eqref{eq:R-via-sigma_1}}}=~
\scanpic{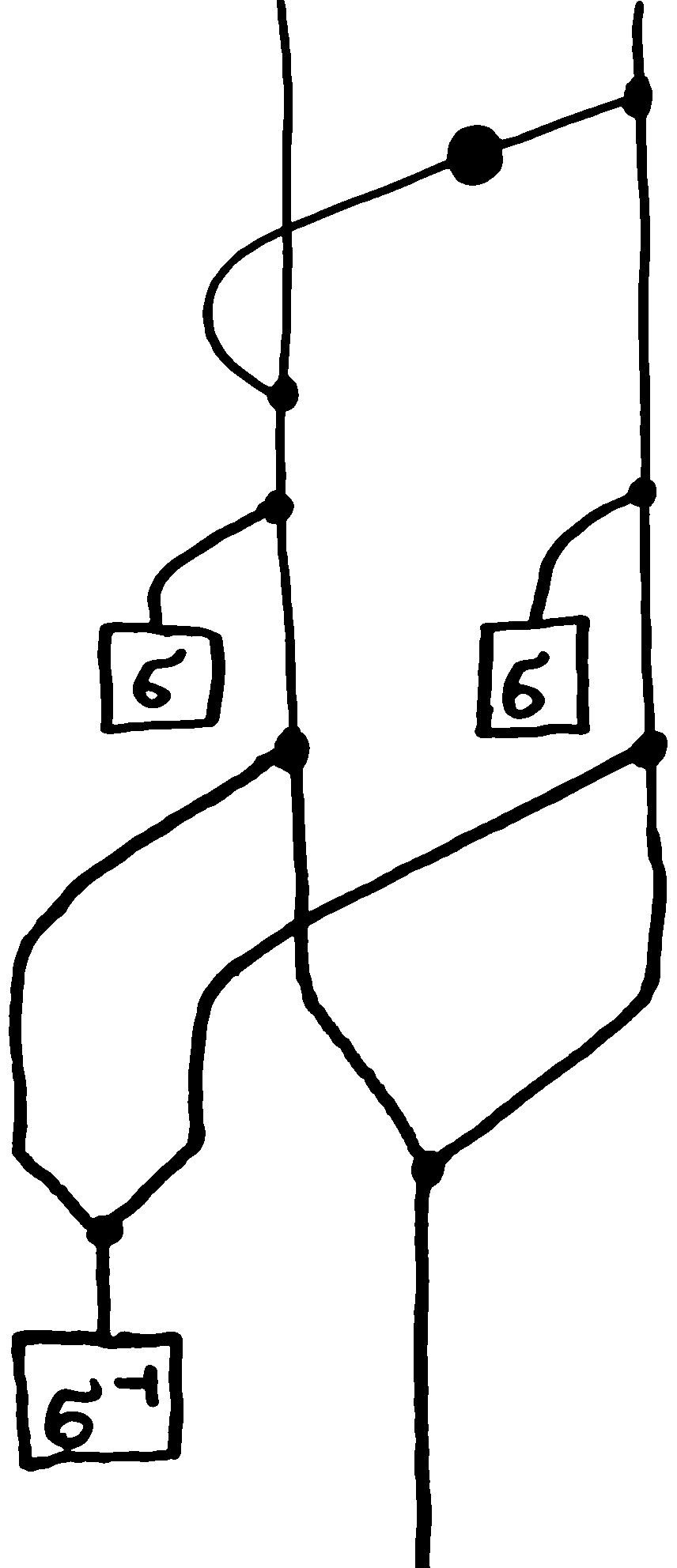}
~\overset{(2)}=~
\scanpic{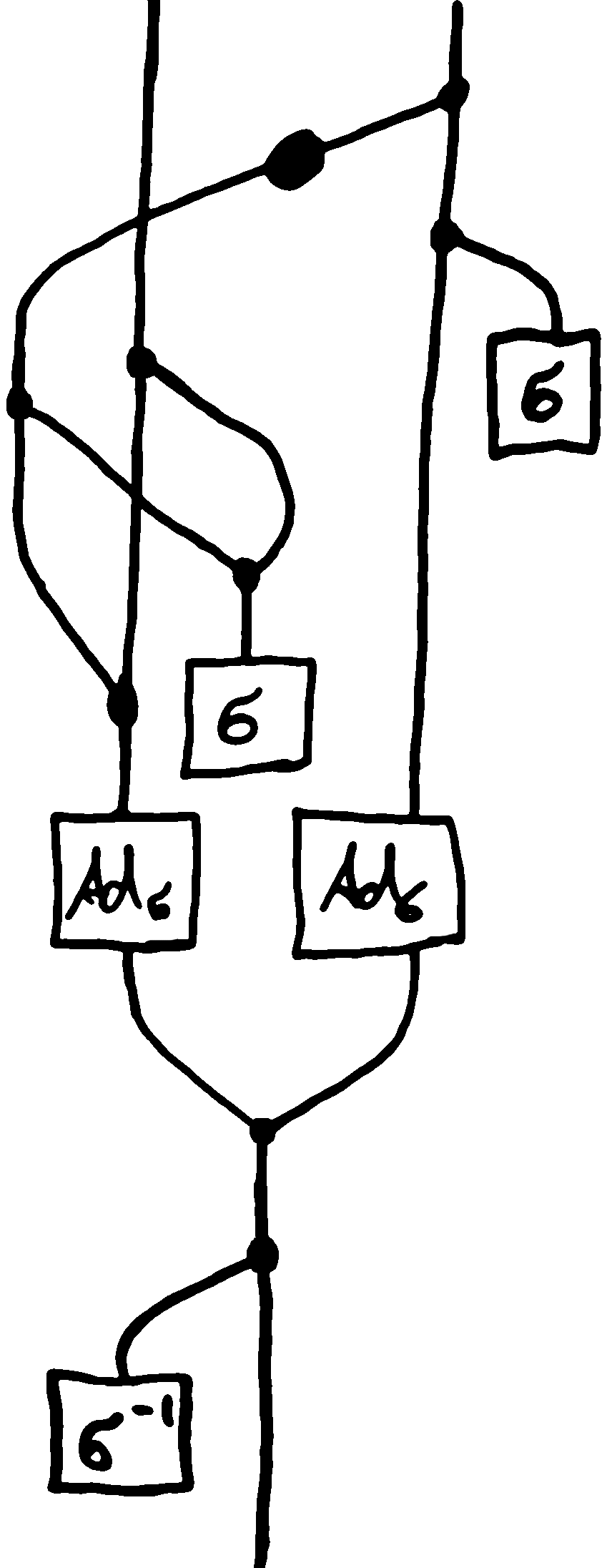}
~\overset{(3)}=~
\scanpic{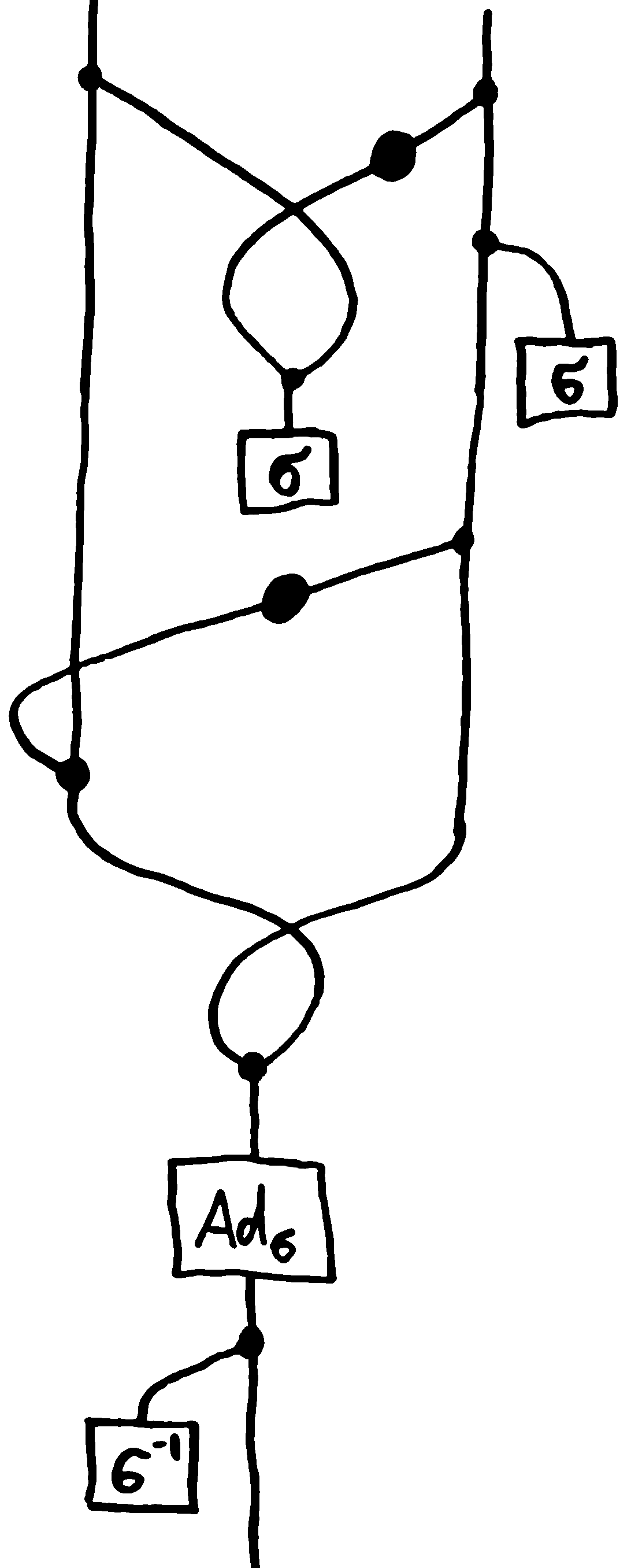}
~\overset{(4)}=~
\scanpic{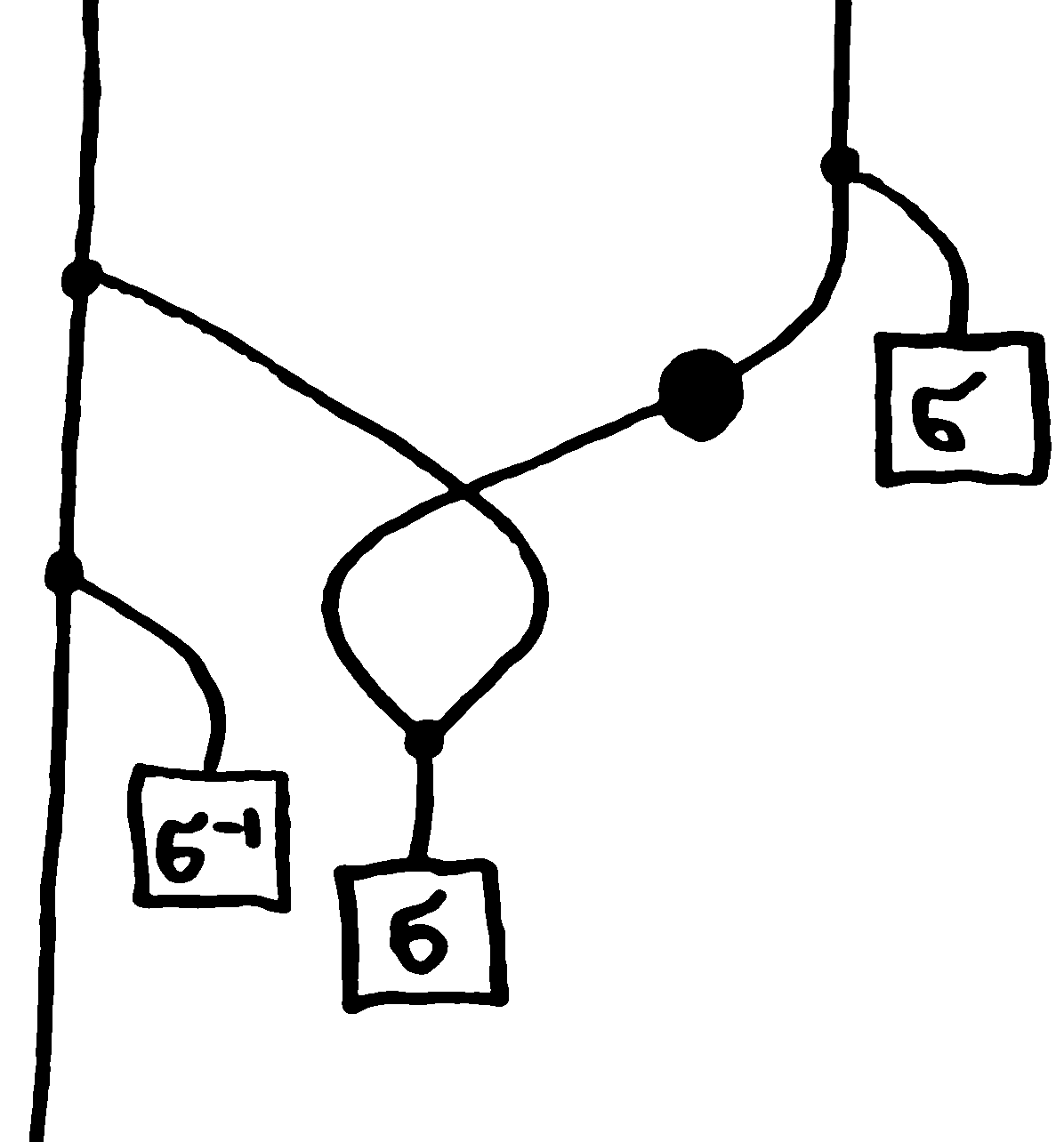}
\quad .
\ee
Step (2) is the algebra-map property of the coproduct (applied twice) and step (3) uses the anti-coalgebra map property of $\Ad_\sigma$ (Lemma \ref{lem:A-prop-1}). In step (4) the bubble-property for $S^{-1}$ has been applied. 

Now insert the right hand sides of \eqref{eq:7R-aux1} and \eqref{eq:hexagon-7R} into $(\eps \otimes \id) \circ (-) \circ \eta$. Together with $\eps(\sigma^{-1})=\id_\one$ (Lemma \ref{lem:A-prop-1}), this results in
\be
  M_\sigma \circ S^{-1} \circ \sigma = {}_\tau M \circ M_\tau \circ g \ .
\ee
Replacing $S^{-1} \circ \sigma = {}_{g^{-1}}M \circ \sigma$ (Lemma \ref{lem:S-sigma}\,(i)) and $\tau = {}_{g^{-1}}M \circ M_{g^{-1}} \circ \sigma$ (Lemma \ref{lem:nu-tau}) in the above identity gives an expression which simplifies to $\sigma = {}_{g^{-1}}M \circ M_{g^{-1}} \circ \sigma$.
\end{proof}

Two immediate consequences of Lemmas \ref{lem:nu-tau} and \ref{lem:g-sig-g=sig} are
\be \label{eq:tau=sig}
  \Ad_\sigma \circ \Ad_g = \Ad_{g^{-1}} \circ \Ad_\sigma 
  \quad , \quad
  \tau = \sigma 
  \ .
\ee

\begin{lemma}\label{lem:omega-expr}
We have $\omega = S \circ\Ad_\sigma^{-1} \circ \Gamma \circ (\Ad_\sigma)^\vee \circ \Gamma^{-1}$.
\end{lemma}

\begin{proof}
We start by rewriting hexagon \eqref{eq:hexagon-3R}:
\be \label{eq:hexagon-3R-alt}
\begin{array}{ccrcl}
  \text{\eqref{eq:hexagon-3R}}
  &\overset{\text{\eqref{eq:R-via-gamma-1}}}\Leftrightarrow&
  (\id \otimes \omega) \circ \gamma &=& \big\{ \Ad_\sigma \otimes (  {}_{\tau^{-1}}M \circ M_\tau ) \big\} \circ \gamma^{-1}
  \\
  &\overset{\text{\eqref{eq:tau=sig}}}\Leftrightarrow&
  (S \otimes \omega) \circ \gamma &=& \big\{ (S \circ \Ad_\sigma) \otimes \Ad_\sigma^{-1} \big\} \circ \gamma^{-1} 
  \\
  &\overset{\text{Lem.\,\ref{lem:S-sigma}\,(ii)}}\Leftrightarrow&
  (\id \otimes \omega) \circ \gamma^{-1} &=& \big\{  \Ad_\sigma \otimes ( S \circ\Ad_\sigma^{-1} ) \big\} \circ \gamma^{-1} \ .
\end{array}
\ee
Using duality maps to convert the above into an identity of morphisms $H^\vee \to H$ and inserting the definition of   $\Gamma$ in terms of $\gamma^{-1}$ from \eqref{eq:Hopf-via-gdp} results in
\be \label{eq:hexagon-3R-aux1}
  \omega \circ \Gamma = S \circ\Ad_\sigma^{-1} \circ \Gamma \circ (\Ad_\sigma)^\vee \ .
\ee
\end{proof}

\begin{lemma} \label{lem:lambda-S}
We have $\lambda(\sigma) = \beta^2$ and $\lambda \circ S = \lambda \circ \Ad_\sigma$.
\end{lemma}

\begin{proof}
We will need the following identity of morphisms $H \to \one$:
\begin{align}
\lambda \circ {}_\sigma M \circ \phi
&\overset{(1)}=
(\lambda \otimes \lambda)
\circ
({}_\sigma M \otimes \mu)
\circ
(\id \otimes S \otimes \id)
\circ
(\gamma \otimes \id)
\nonumber \\
&\overset{(2)}=
(\lambda \otimes \lambda)
\circ
(\id \otimes \mu)
\circ
(\id \otimes (S \circ M_{\sigma^{-1}}) \otimes \id)
\circ
(\Delta \otimes \id)
\circ
(\sigma \otimes \id)
\nonumber \\
&\overset{(3)}=
\lambda(\sigma) \cdot \lambda \circ {}_{S \circ \sigma^{-1}}M \ .
\label{eq:lambda-S-aux1}
\end{align}
In step (1) expression \eqref{eq:gdp-via-Hopf} is inserted for $\phi$ (cf.\ Lemma \ref{lem:delta+phi-via-gamma}), 
in step (2) expression \eqref{eq:gamma_via_sigma} for $\gamma$ is inserted, and
in step (3) the fact that $\lambda$ is a right cointegral is used.

Next, we will rewrite \eqref{eq:hexagon-8L} in a way which no longer involves $\phi^{-1}$. Recall from \eqref{eq:tau=sig} and Lemma \ref{lem:nu-tau} that $\tau = \sigma$ and $\nu = \beta \cdot \sigma^{-1}$. Combining ${}_g M \circ {}_\sigma M \circ {}_g M = {}_\sigma M$ from Lemma \ref{lem:g-sig-g=sig} with the inverse of $\phi$ given in \eqref{eq:phi-1_via_phi} results in
\be \label{eq:hex8L-alt1}
  \text{\eqref{eq:hexagon-8L}}
  \quad \Leftrightarrow \quad
  S \circ \phi \circ {}_\sigma M \circ \omega \circ \phi \circ S
  = \beta^2 \cdot M_{\sigma^{-1}} \circ {}_g M \circ \phi \circ S \circ  M_{\sigma^{-1}} \ .
\ee
If we compose this identity with $\eps$ from the left and use $\lambda = \eps \circ \phi$, together with naturality of $\omega$ we obtain $\lambda \circ {}_\sigma M \circ \phi \circ S = \beta^2 \cdot \lambda \circ S \circ  M_{\sigma^{-1}}$. Now substitute \eqref{eq:lambda-S-aux1} on the left hand side and use ${}_{S \circ \sigma^{-1}}M \circ S = S \circ M_{\sigma^{-1}}$ to conclude $\lambda(\sigma) \cdot \lambda = \beta^2 \cdot \lambda$. Since $\lambda \circ \Lambda' = 1$ (with $\Lambda'$ as in Lemma \ref{lem:integral-cointegral}), the first identity of the lemma follows.
  
For the second identity we compose the hexagon \eqref{eq:hexagon-8R} from the left with $\eps$. Together with $\eps(\sigma^{-1}) = \id_\one$ from Lemma \ref{lem:A-prop-1}, this results in $\lambda \circ   {}_\sigma M \circ \phi = \beta^2 \cdot \lambda \circ M_{\sigma^{-1}}$. Substituting \eqref{eq:lambda-S-aux1} and $\lambda(\sigma) = \beta^2$ gives
\be \label{eq:lambda-S-aux2}
  \lambda \circ {}_{S \circ \sigma^{-1}}M = \lambda \circ M_{\sigma^{-1}} \ .
\ee
Now use $S \circ \sigma^{-1} = {}_gM \circ \sigma^{-1}$ from Lemma \ref{lem:S-sigma}\,(i) as well as Lemma \ref{lem:lambda-S-g} to rewrite the left hand side of \eqref{eq:lambda-S-aux2} of as $\lambda \circ {}_{S \circ \sigma^{-1}}M = \lambda \circ {}_gM \circ {}_{\sigma^{-1}}M = \lambda \circ S  \circ {}_{\sigma^{-1}}M$. Comparing to the right hand side of \eqref{eq:lambda-S-aux2} gives the second identity claimed in the statement of the lemma.
\end{proof}

\bigskip

\begin{proof}[Proof of Proposition \ref{prop:hex-to-Hopf}]
That $R,\tau,\nu$ are of the form \eqref{eq:R-etc-via-Hopf} with $\beta = \eps(\nu)$ amounts to \eqref{eq:R-via-sigma_1}, \eqref{eq:tau=sig} and Lemma \ref{lem:nu-tau}. 
Condition a) in Theorem \ref{thm:main2}\,(i) is equivalent to \eqref{eq:gamma-1_via_sigma}. 
Condition b) was proved in Lemma \ref{lem:lambda-S}.
For condition c) note that $\Ad_\sigma$ is automatically an algebra-map. The anti-coalgebra map property was proved in Lemma \ref{lem:A-prop-1}, compatibility with the antipode in Lemma \ref{lem:S-sigma}\,(ii).
Condition d) holds by Lemmas \ref{lem:S-sigma}\,(i) and \ref{lem:g-sig-g=sig}.
Condition e) is Lemma \ref{lem:omega-expr}.
\end{proof}

\subsection{From Hopf-algebraic data to hexagon} \label{sec:Hopf-to-hex}

The following proposition proves Theorem \ref{thm:main2}\,(i).

\begin{proposition} \label{prop:Hopf-to-hex}
Suppose $\sigma$ has a multiplicative inverse, $\beta \in \End(\one)$ is invertible, and conditions a)--e) in Theorem \ref{thm:main2}\,(i) are satisfied. Then for $R,\tau,\nu$ as defined in \eqref{eq:R-etc-via-Hopf}, the 16 hexagon conditions  \eqref{eq:hexagon-1L}--\eqref{eq:hexagon-8R} hold, and $R$ satisfies \eqref{eq:R-intertwiner}.
\end{proposition}

\begin{proof} The proof consists of checking quasi-cocommutativity and all 16 hexagons one by one and will take up most of this section. We start with some preliminaries. Condition a) is equivalent to
\be\label{eq:gammainv-via-sigma_2nd}
  \gamma^{-1} = (M_\sigma \otimes {}_\sigma M) \circ \Delta \circ \sigma^{-1} \ ,
\ee
and together with \eqref{eq:R-etc-via-Hopf} we have
\be\label{eq:R-via-gamma_2nd}
  R 
  \overset{\text{\eqref{eq:R-etc-via-Hopf}}}= 
  ({}_\sigma M \otimes {}_\sigma M) \circ \Delta \circ \sigma^{-1}
  \overset{\text{\eqref{eq:gammainv-via-sigma_2nd}}}= 
  (\Ad_\sigma \otimes \id) \circ \gamma^{-1} \ .
\ee
Taking the multiplicative inverse of both sides of \eqref{eq:gammainv-via-sigma_2nd} gives
\be\label{eq:gamma-via-sigma_2nd}
  \gamma 
  = ({}_{\sigma^{-1}}M \otimes M_{\sigma^{-1}}) \circ \Delta \circ \sigma \ .
\ee
We also note that condition d) implies that
\be\label{eq:Adsig-Adg}
  \Ad_\sigma \circ \Ad_g = \Ad_g^{-1} \circ \Ad_\sigma \ .
\ee
Condition d) furthermore implies that all four identities in Lemma \ref{lem:S-sigma}\,(i) hold, which we reiterate here:
\be \label{eq:S-sigma-g_2nd}
S \circ \sigma = M_{g^{-1}} \circ \sigma
~,~~
S^{-1} \circ \sigma = {}_{g^{-1}}M \circ \sigma
~,~~
S \circ \sigma^{-1} = {}_gM \circ \sigma^{-1}
~,~~
S^{-1} \circ \sigma^{-1} = M_g \circ \sigma^{-1} \ .
\ee

\medskip\noindent
{\bf Quasi-cocommutativity:}
Inserting the anti-coalgebra map property of $\Ad_\sigma$ from condition c) 
into $(M_\sigma \otimes M_\sigma) \circ c_{H,H} \circ (-) \circ {}_{\sigma^{-1}}M$ results in
\be \label{eq:sig-Delta-sig-Deltaop}
  (M_\sigma \otimes M_\sigma) \circ \Delta^\mathrm{cop} \circ M_{\sigma^{-1}} 
  = 
  ({}_\sigma M \otimes{}_\sigma M ) \circ \Delta \circ {}_{\sigma^{-1}}M 
  \ .
\ee
Composing the right hand side with $\eta$ gives expression \eqref{eq:R-via-gamma_2nd} for $R$. The above equality then implies the alternative expression for $R$ also found in \eqref{eq:R-via-sigma_2} when proving the converse of Theorem \ref{thm:main2}\,(i):
\be
R = (M_\sigma \otimes M_\sigma) \circ \Delta^\mathrm{cop} \circ \sigma^{-1} 
  \label{eq:R-via-sigma_2-2nd}
\ee
Given the expressions \eqref{eq:R-via-gamma_2nd} and \eqref{eq:R-via-sigma_2-2nd} for $R$ and the identity \eqref{eq:sig-Delta-sig-Deltaop}, the computation in \eqref{eq:qcocomm-vs-Ad-anticoalg} shows that \eqref{eq:R-intertwiner} is satisfied.

\medskip\noindent
{\bf Cases 1L, 1R:} 
Substituting \eqref{eq:R-via-gamma_2nd} into \eqref{eq:hexagon-1L} shows that
\be
  \text{\eqref{eq:hexagon-1L}} 
  \quad \Leftrightarrow \quad
  \scanPIC{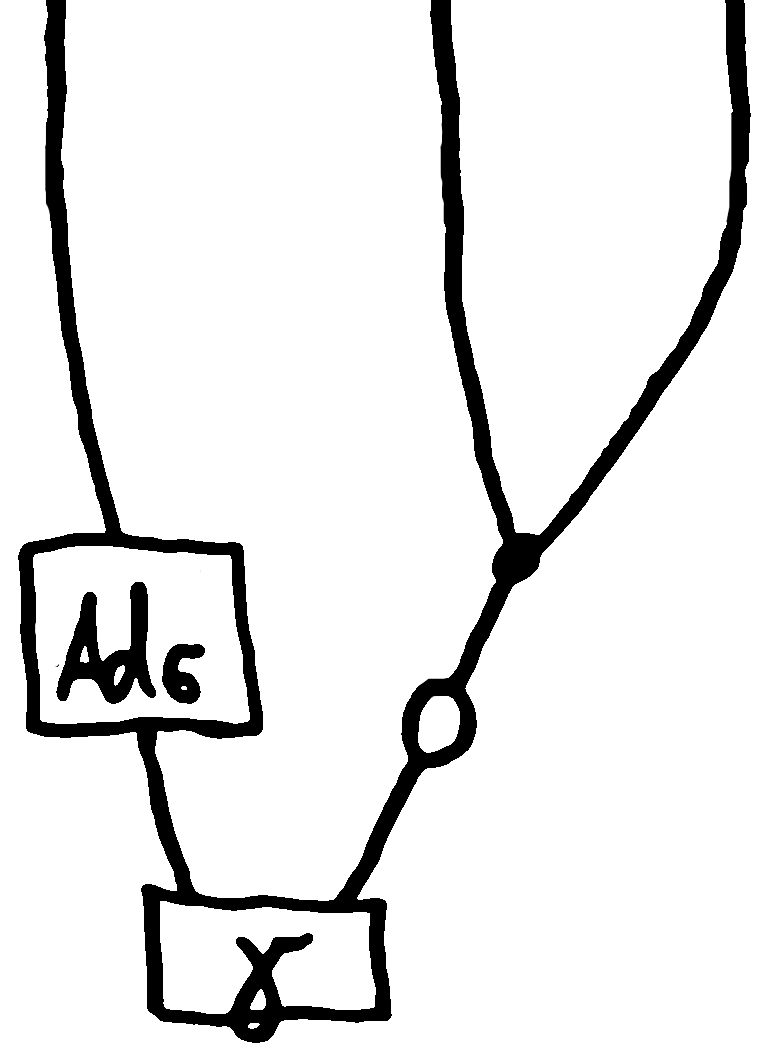} 
  ~=~
  \scanPIC{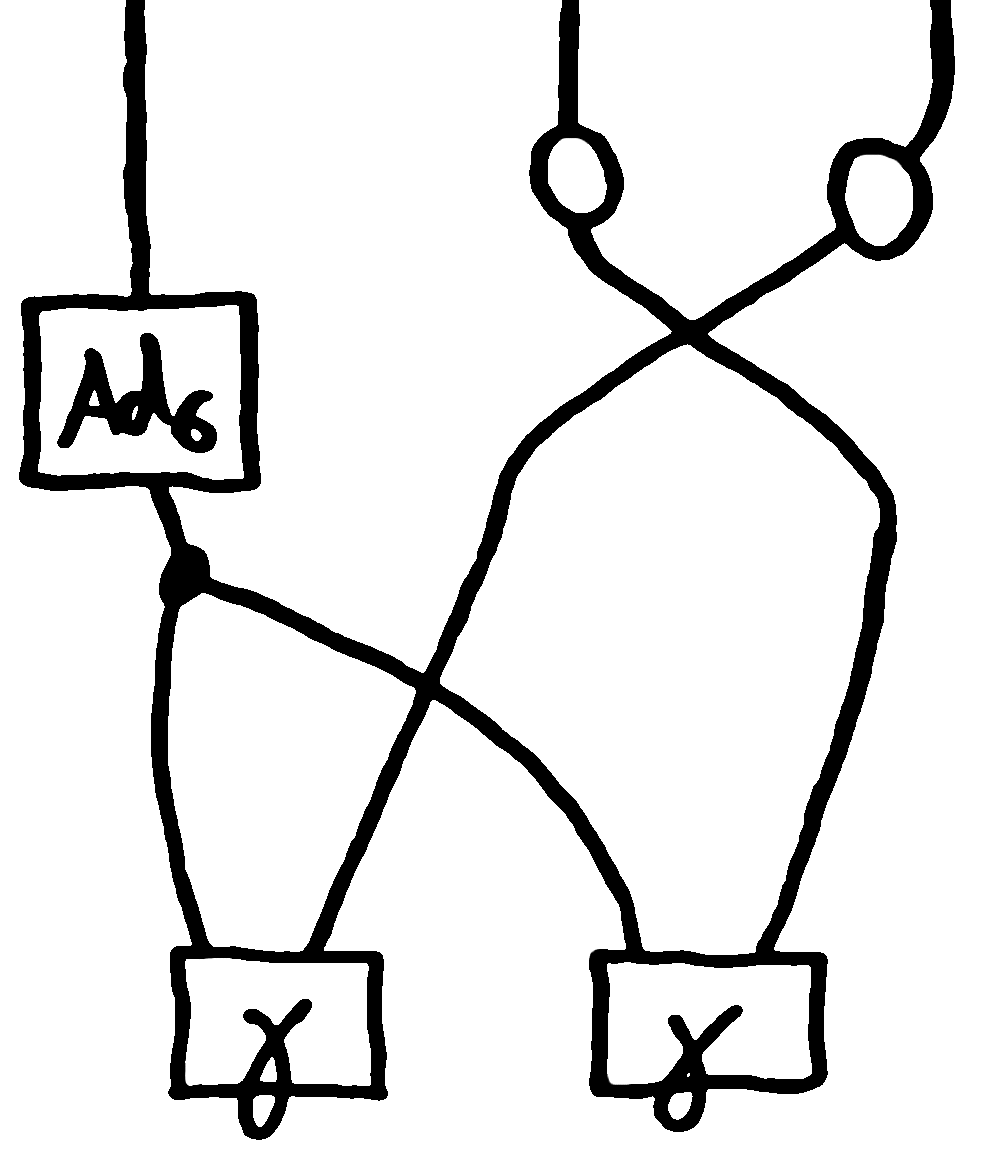} 
  \quad ,
\ee
which in turn is immediate from \eqref{eq:pentagon-4}. Along the same lines one verifies that \eqref{eq:hexagon-1R} is implied by \eqref{eq:pentagon-3} (use $R = ((\Ad_\sigma \circ S) \otimes \id) \circ \gamma$ and the fact that $\Ad_\sigma \circ S$ is a coalgebra map).

\medskip\noindent
{\bf Cases 2L, 2R:} 
The hexagon \eqref{eq:hexagon-2R} is just \eqref{eq:gamma-via-sigma_2nd}. Hexagon \eqref{eq:hexagon-2L} is immediate after substituting \eqref{eq:gammainv-via-sigma_2nd} for $\gamma^{-1}$ and \eqref{eq:R-via-sigma_2-2nd} for $R$. Incidentally, this also proves
\be \label{eq:R-via-gamma-op}
  R 
  = 
  (\Ad_\sigma^{-1} \otimes \id) \circ c_{H,H} \circ \gamma^{-1}  \ .
\ee
Comparing \eqref{eq:R-via-gamma_2nd} and \eqref{eq:R-via-gamma-op} we obtain the first of the two equalities in
\be \label{eq:gamma-and-c-gamma}
(\Ad_\sigma^{\,2} \otimes \id)\circ \gamma = c_{H,H} \circ \gamma = (\id \otimes \Ad_\sigma^{-2})\circ \gamma \ .
\ee
The second equality follows by composing the first one with $c_{H,H} \circ (\Ad_\sigma^{-2} \otimes \id)$. By the same argument as in the proof of Lemma \ref{lem:S2=GA-2}\,(ii), \eqref{eq:gamma-and-c-gamma} implies that $S^2 = \Ad_g \circ \Ad_\sigma^{-2}$. We will later use this identity in the form
\be \label{eq:Adsig-inv}
  \Ad_\sigma^{-1} = \Ad_\sigma \circ \Ad_g^{-1} \circ S^2 \ .
\ee

\medskip\noindent
{\bf Cases 3L, 3R:} 
The hexagon \eqref{eq:hexagon-3L} is just \eqref{eq:R-via-gamma_2nd}. Via the same steps as in \eqref{eq:hexagon-3R-alt} one checks that
\be \label{eq:hexagon-3R-alt2}
  \text{\eqref{eq:hexagon-3R}} \quad \Leftrightarrow \quad
  (\id \otimes \omega) \circ \gamma^{-1} = \big\{  \Ad_\sigma \otimes ( S \circ \Ad_\sigma^{-1} ) \big\} \circ \gamma^{-1} 
  \ .
\ee
The identity \eqref{eq:hexagon-3R-alt2} in turn is equivalent to condition e) (cf.\,\eqref{eq:hexagon-3R-aux1}).

\medskip\noindent
{\bf Cases 4L, 4R:} 
Substituting $\tau = \sigma$ from \eqref{eq:R-etc-via-Hopf} and using \eqref{eq:gamma-via-sigma_2nd}, we can rearrange \eqref{eq:hexagon-4L} to 
get
\be \label{eq:check-4L-aux1}
\text{\eqref{eq:hexagon-4L}} \quad \Leftrightarrow \quad
(\id \otimes \omega) \circ \gamma^{-1}  
= 
( \Ad_\sigma \otimes \Ad_\sigma^{-1}) \circ \gamma \ .
\ee
Inserting $\gamma = (\id \otimes S^{-1}) \circ \gamma^{-1}$ and using $\Ad_\sigma^{-1} \circ S^{-1} = S \circ \Ad_\sigma^{-1}$ (from condition c)) turns \eqref{eq:check-4L-aux1} into \eqref{eq:hexagon-3R-alt2}.

For hexagon \eqref{eq:hexagon-4R}, substitute $\tau = \sigma$ and the second expression for $R$ in \eqref{eq:R-via-gamma_2nd}, then bring all factors of $\Ad_\sigma$ to the left hand side. This gives
\be \label{eq:check-4R-aux1}
\text{\eqref{eq:hexagon-4R}} \quad \Leftrightarrow \quad
( \Ad_\sigma^{-1} \otimes \Ad_\sigma^{-1}) \circ c_{H,H} \circ \gamma 
= 
(\id \otimes \omega) \circ \gamma^{-1}  \ .
\ee
Starting from the right hand side, we compute
\begin{align}
(\id \otimes \omega) \circ \gamma^{-1} 
&\overset{\text{\eqref{eq:check-4L-aux1}}}= 
( \Ad_\sigma \otimes \Ad_\sigma^{-1}) \circ \gamma
\nonumber \\
&~\,=\,~
( \id  \otimes (\Ad_\sigma^{-1} \circ S^{-1}) ) \circ 
( \Ad_\sigma \otimes \id) \circ  \gamma^{-1}
\nonumber \\
&\overset{\text{\eqref{eq:R-via-gamma_2nd}}}{\underset{\text{\eqref{eq:R-via-gamma-op}}}{=}}
( \id  \otimes (\Ad_\sigma^{-1} \circ S^{-1}) ) \circ 
(\Ad_\sigma^{-1} \otimes \id) \circ c_{H,H} \circ \gamma^{-1}  \ ,
\label{eq:check-4R-aux2}
\end{align}
which is equal to the left hand side of \eqref{eq:check-4R-aux1}.

\medskip\noindent
{\bf Cases 5L, 5R:} 
For \eqref{eq:hexagon-5L} we already did all the work in \eqref{eq:hexagon-5L-alt}: a simple application of \eqref{eq:S-sigma-g_2nd} shows that the right hand side of \eqref{eq:hexagon-5L-alt} is equal to the right hand side of \eqref{eq:hexagon-5L}. For \eqref{eq:hexagon-5R} we have to do a small calculation. We start with the right hand side:
\be \label{eq:5R-aux1}
\beta^{-1} \cdot \text{rhs of \eqref{eq:hexagon-5R}} 
~=~
\scanpic{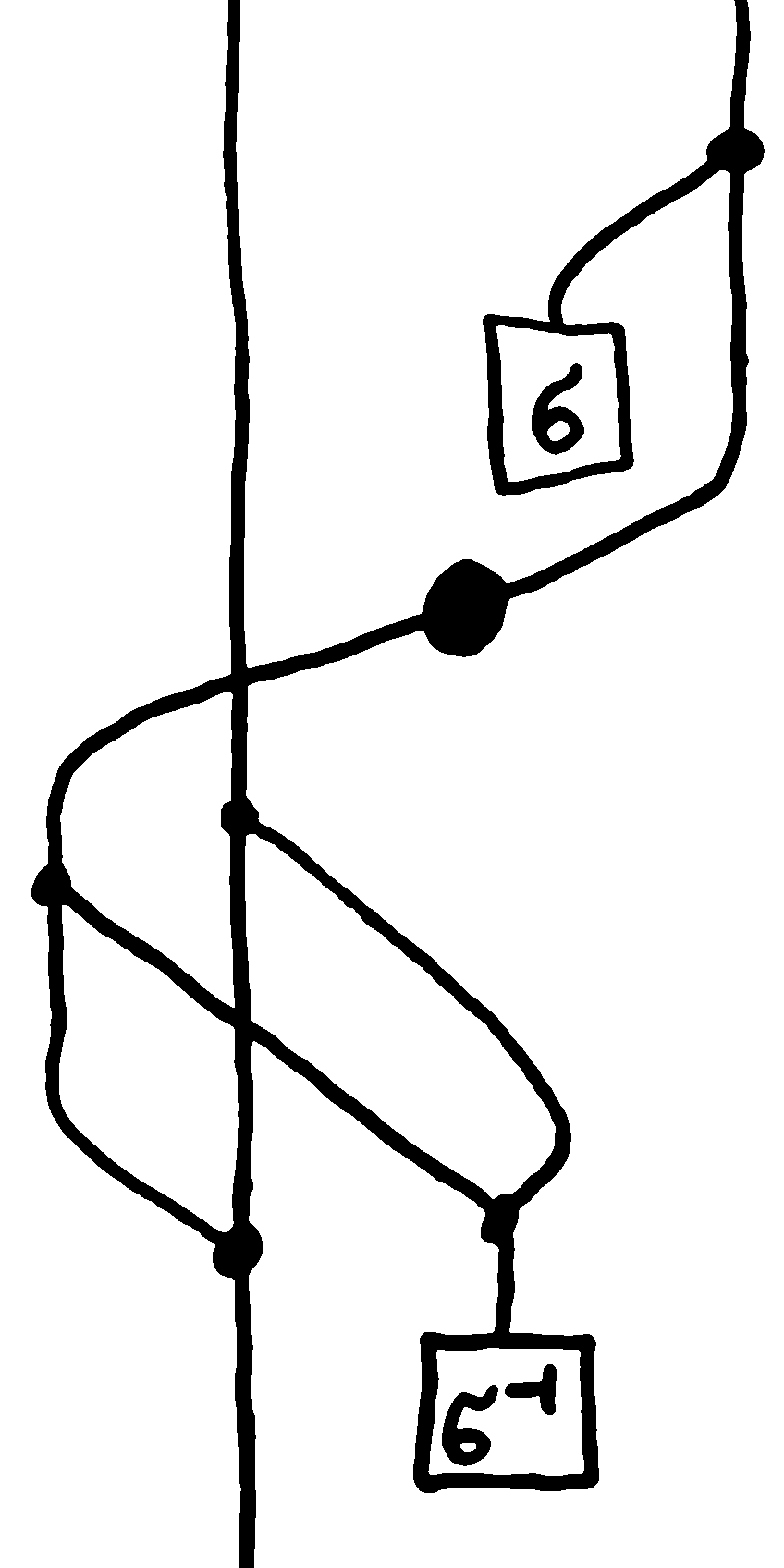} 
~=~
\scanpic{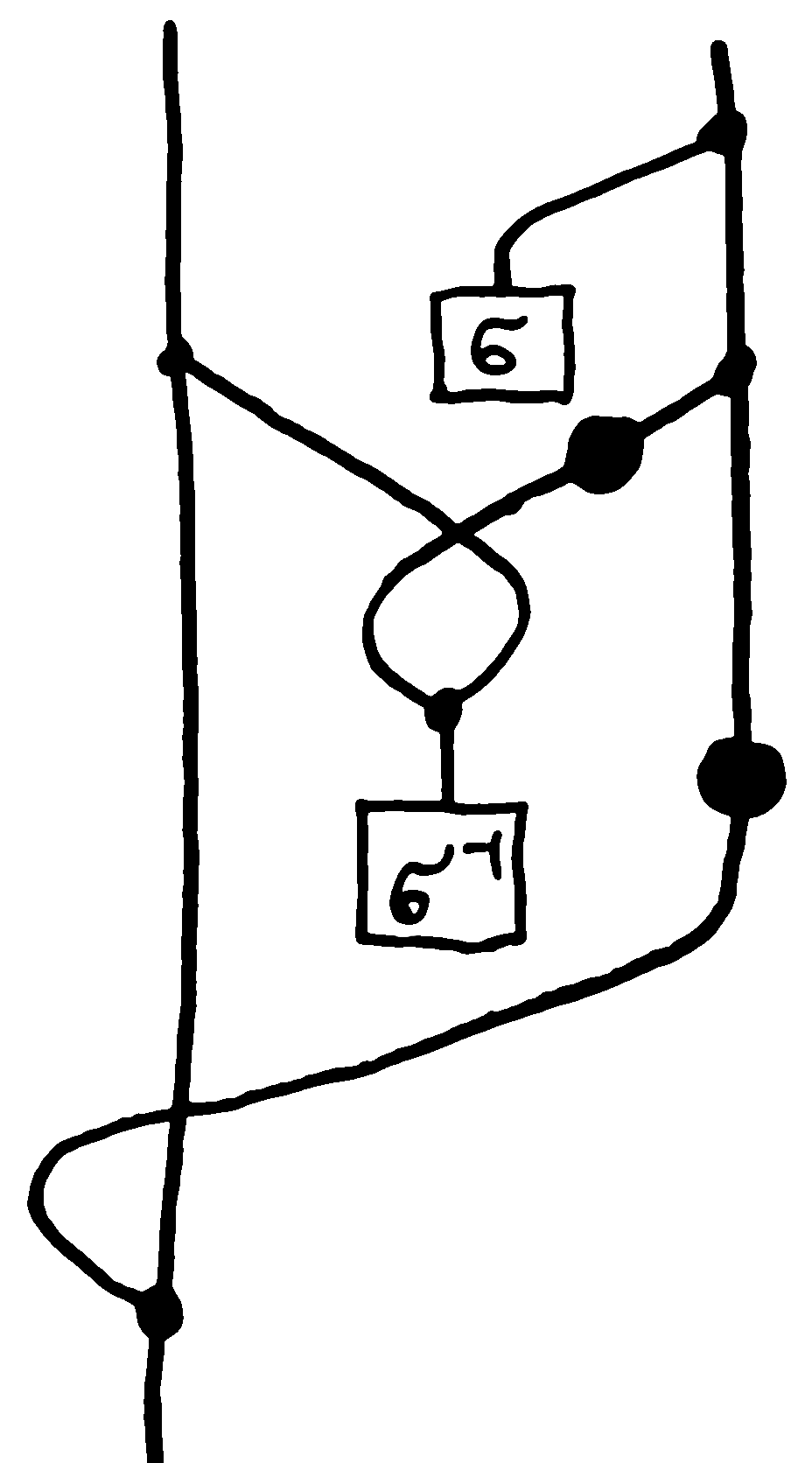} 
\quad .
\ee
For the left hand side we find
\be \label{eq:5R-aux2}
\beta^{-1} \cdot \text{lhs of \eqref{eq:hexagon-5R}} 
\overset{(*)}= 
\scanpic{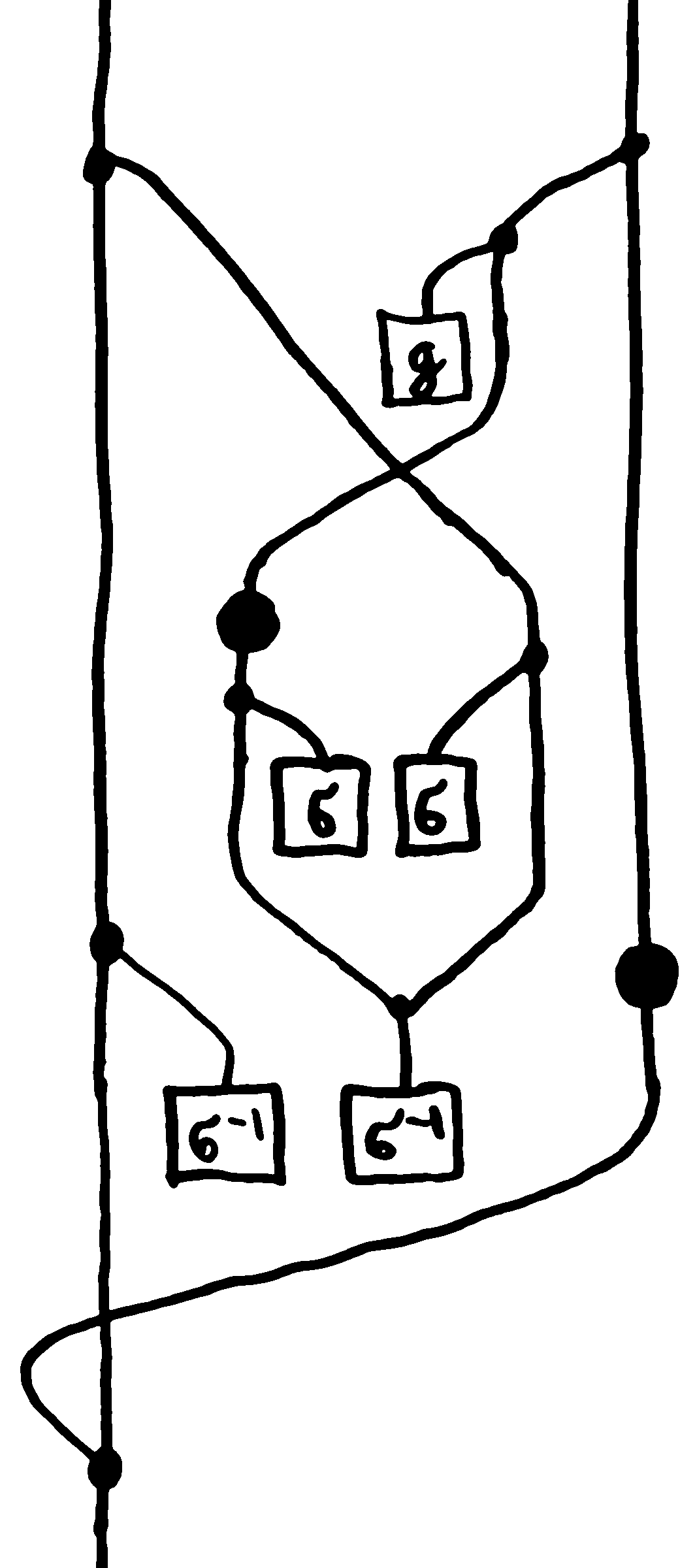} 
~=~
\scanpic{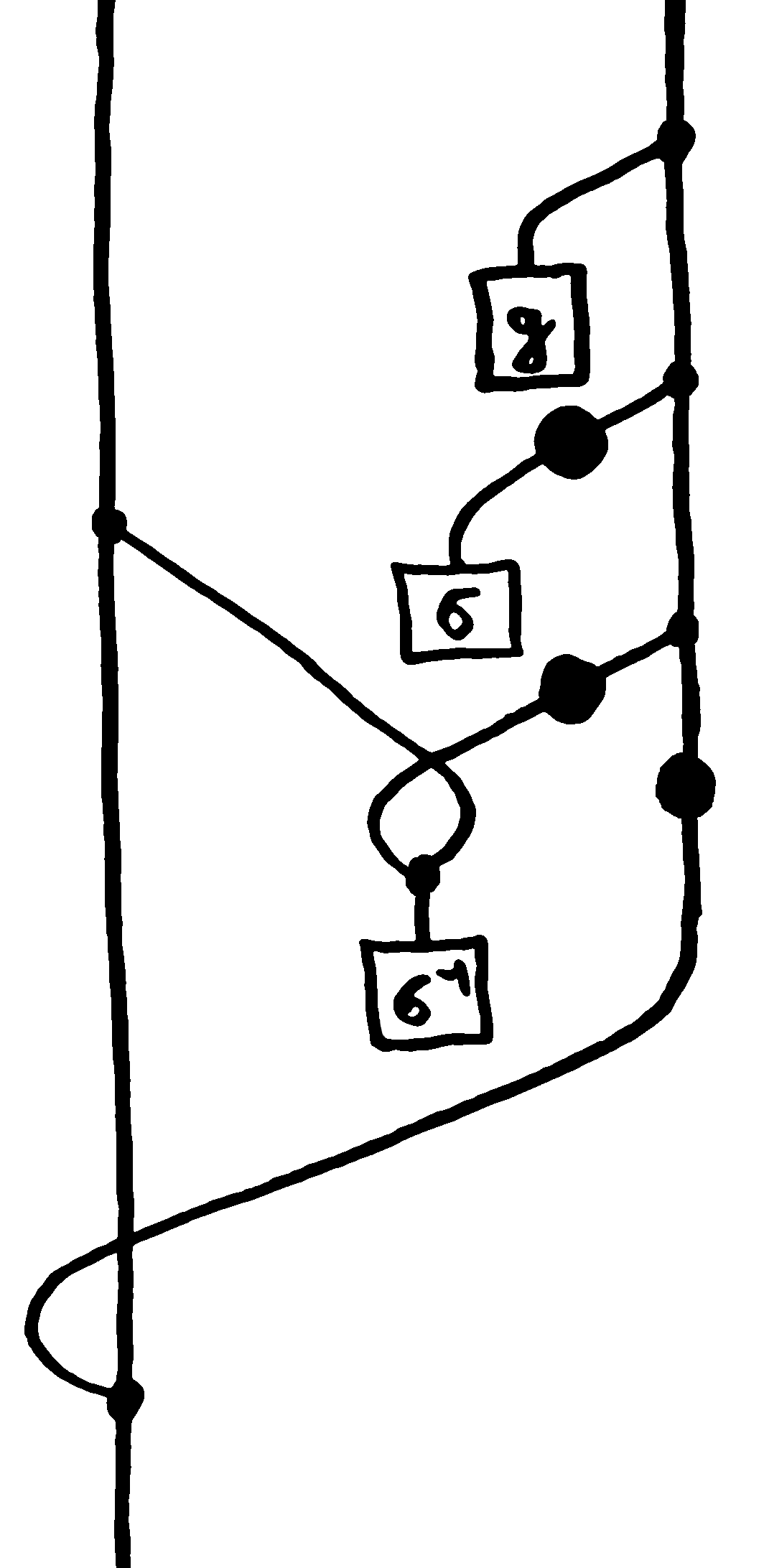} 
\quad .
\ee
In (*) we substituted the second expression for $\delta$ given in (\ref{eq:delta-expression-ab}\,b), replaced $\gamma$ by $(S^{-1} \otimes \id) \circ \gamma^{-1}$ and inserted \eqref{eq:gammainv-via-sigma_2nd} for $\gamma^{-1}$. The right hand side of \eqref{eq:5R-aux1} equals that of \eqref{eq:5R-aux2} upon using the identity $S^{-1} \circ \sigma = {}_{g^{-1}}M \circ \sigma$ from \eqref{eq:S-sigma-g_2nd}.

\medskip\noindent
{\bf Cases 6L, 6R:} We start by rewriting the left hand side of hexagon \eqref{eq:hexagon-6L}:
\begin{align}
\beta^{-1} \cdot \text{lhs of \eqref{eq:hexagon-6L}}
~&\overset{(1)}=~
\scanpic{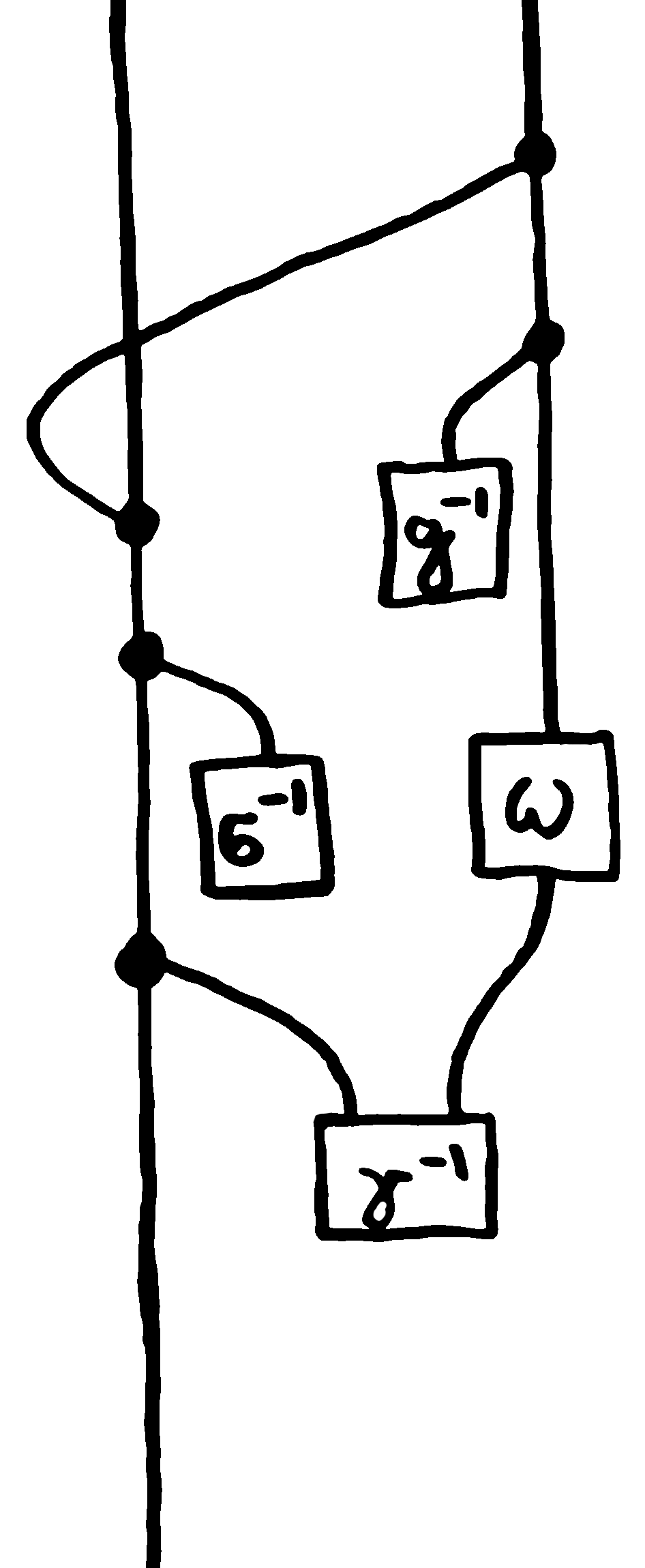}
~\overset{(2)}=~
\scanpic{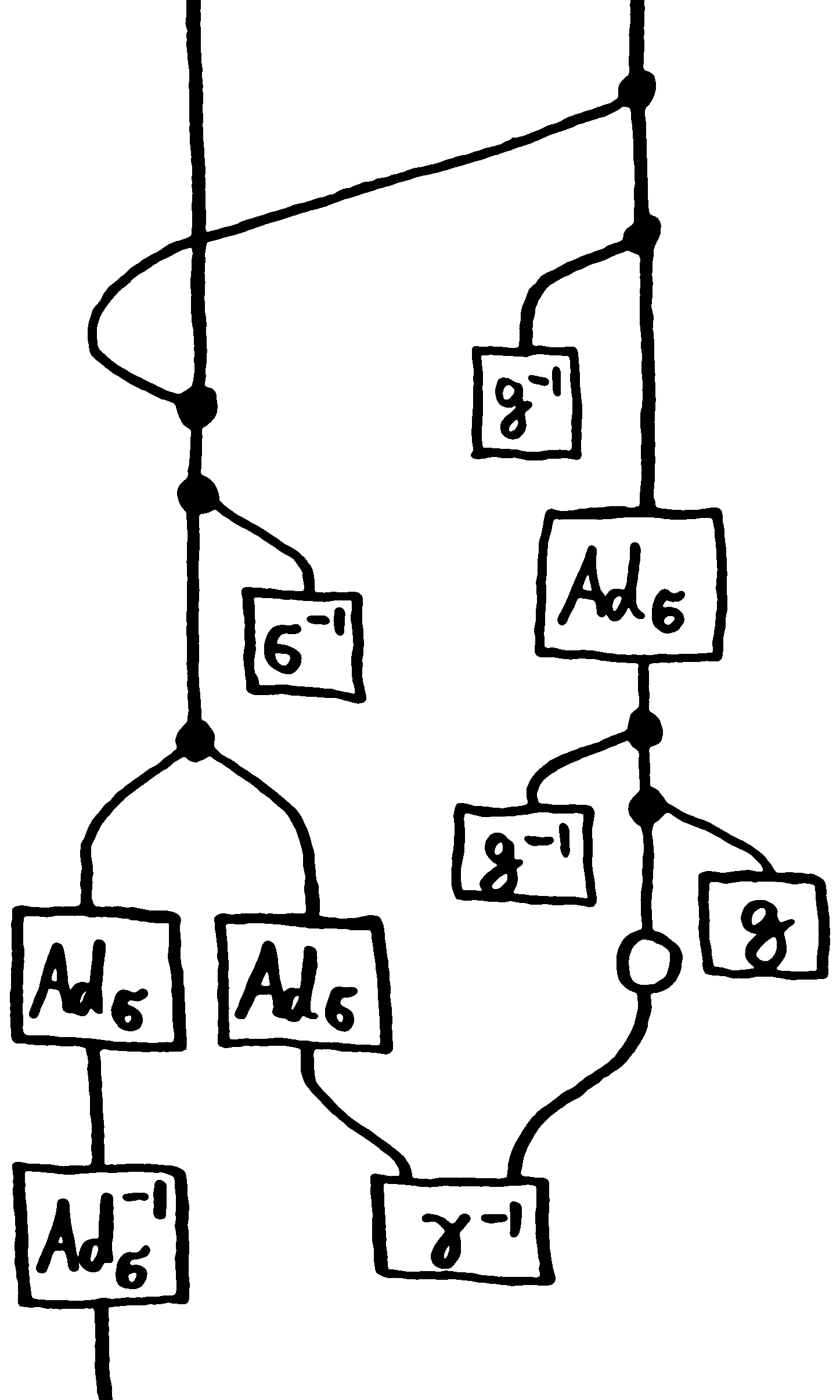}
~\overset{(3)}=~
\scanpic{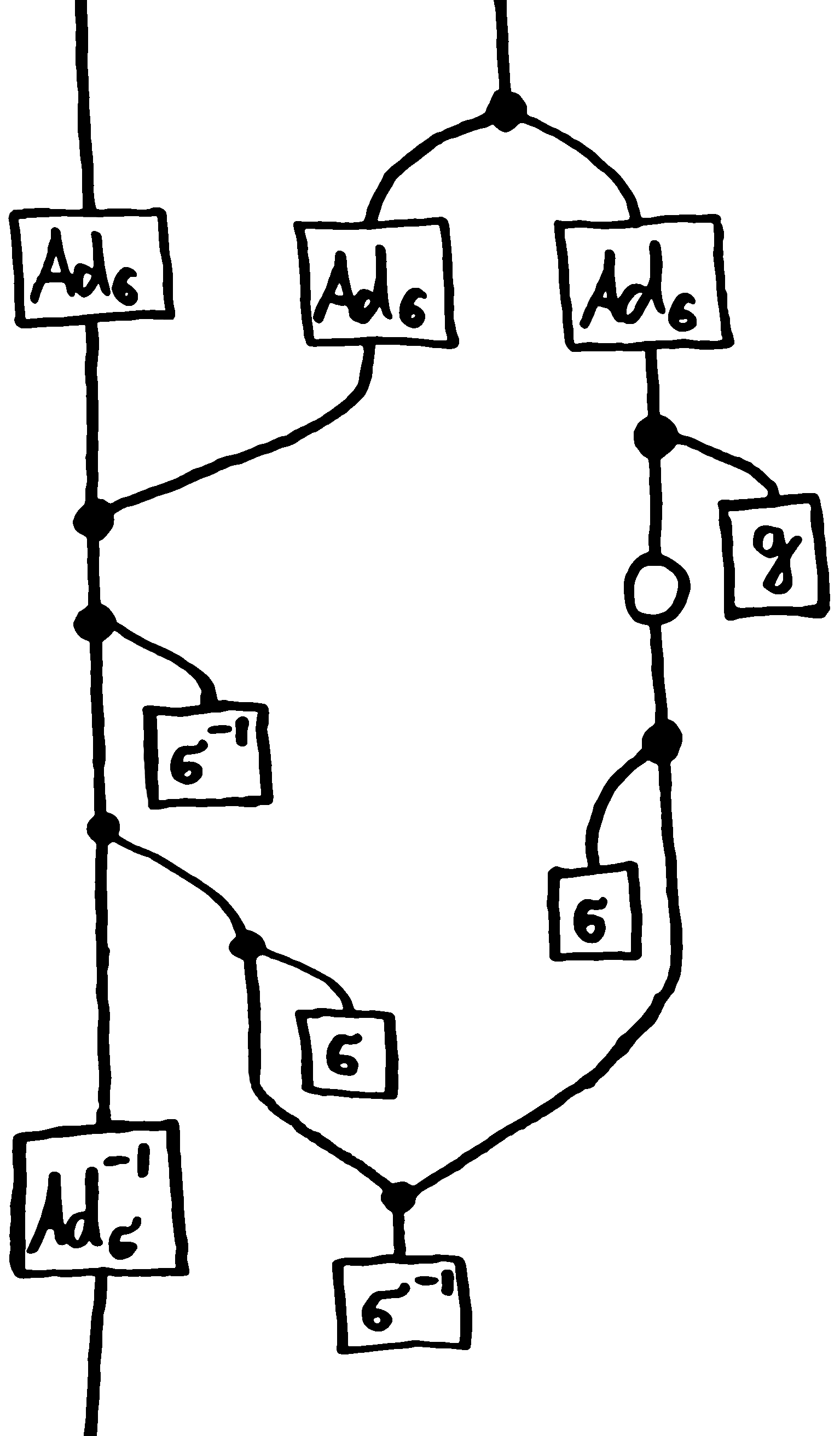}
\nonumber \\
&\overset{(4)}=~
\scanpic{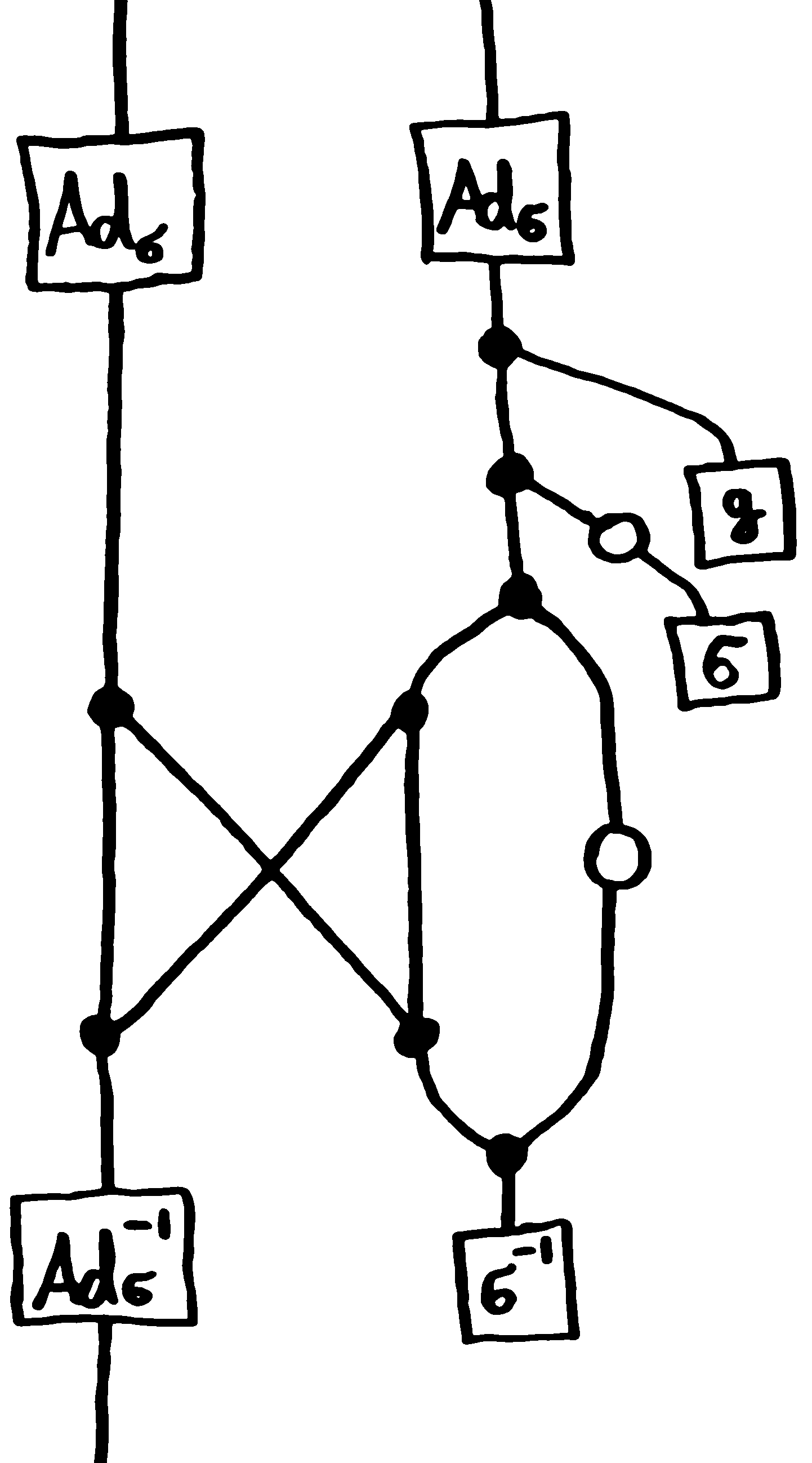}
~\overset{(5)}=~
\scanpic{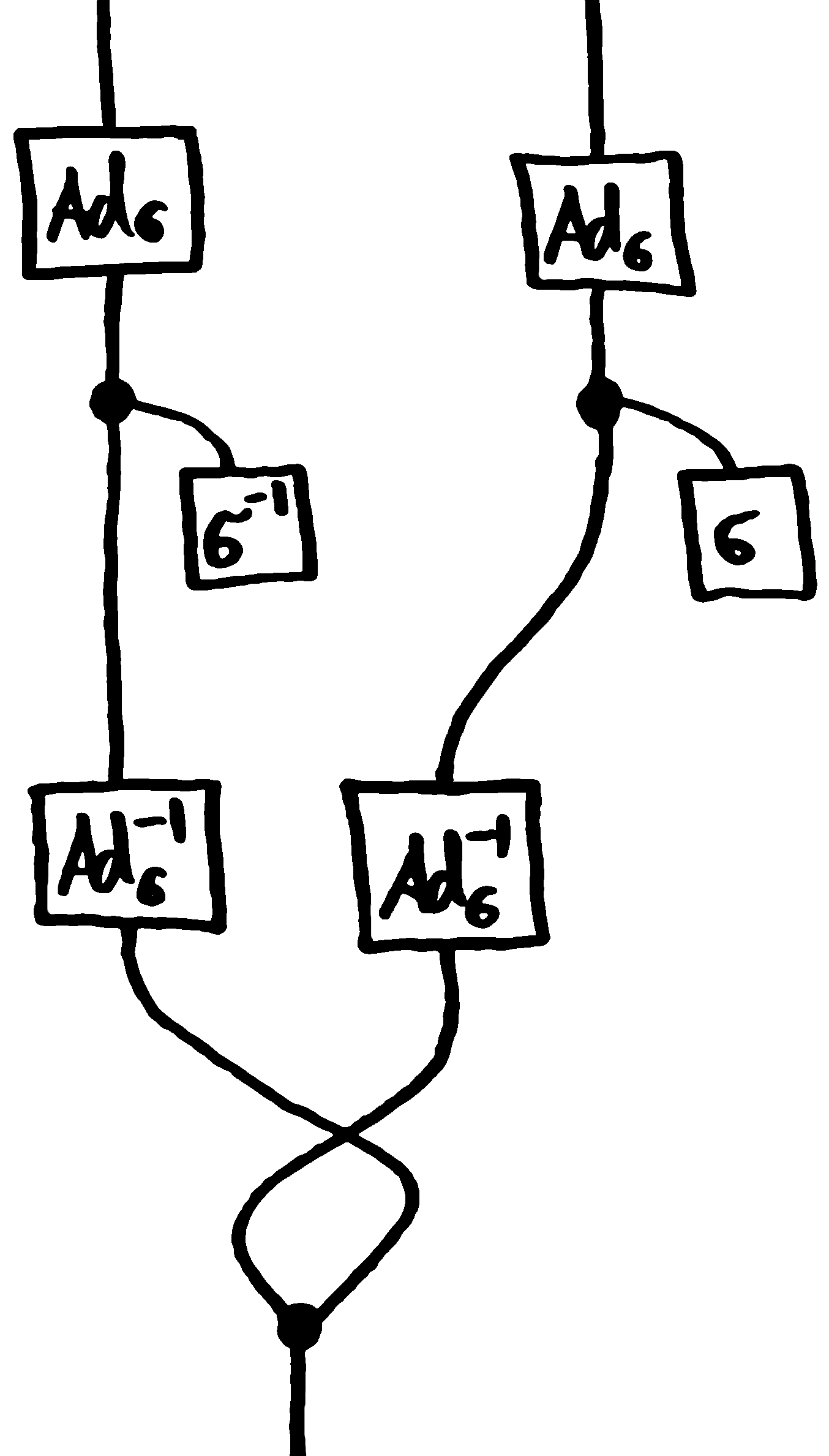}
 \label{eq:hexagon-6L-alt}
\end{align}
In step (1), $\nu = \beta \cdot \sigma^{-1}$ and \eqref{eq:tilde-delta-def} have been substituted.
For step (2), replace $(\id \otimes \omega) \circ \gamma^{-1}$ by the right hand side of \eqref{eq:hexagon-3R-alt2} and use condition c) and \eqref{eq:Adsig-inv} to replace $S \circ \Ad_\sigma^{-1} = \Ad_\sigma^{-1} \circ S^{-1} = \Ad_\sigma \circ \Ad_g^{-1} \circ S$. In step (3) we used that by condition d) we have ${}_{g^{-1}}M \circ \Ad_\sigma \circ {}_{g^{-1}}M = \Ad_\sigma$, as well as that by condition c), $\Ad_\sigma$ is a Hopf algebra isomorphism $H \to H_\mathrm{cop}$; furthermore, $\gamma^{-1}$ was replaced by \eqref{eq:gammainv-via-sigma_2nd}. For step (4), use the algebra-map property of $\Ad_\sigma$ once more, as well as the algebra map property of the coproduct. Step (5) amounts to the bubble-property, to the anti-coalgebra map property of $\Ad_\sigma$, and to $S \circ \sigma = M_{g^{-1}} \circ \sigma$ from \eqref{eq:S-sigma-g_2nd}. It is now immediate that the right hand side of \eqref{eq:hexagon-6L-alt} is equal to $\beta^{-1}$ times the right hand side of \eqref{eq:hexagon-6L}.

For \eqref{eq:hexagon-6R} we compute
\be \label{eq:6R-aux1}
\beta^{-1} \cdot \text{lhs of \eqref{eq:hexagon-6R}} 
\overset{(1)}= 
\scanpic{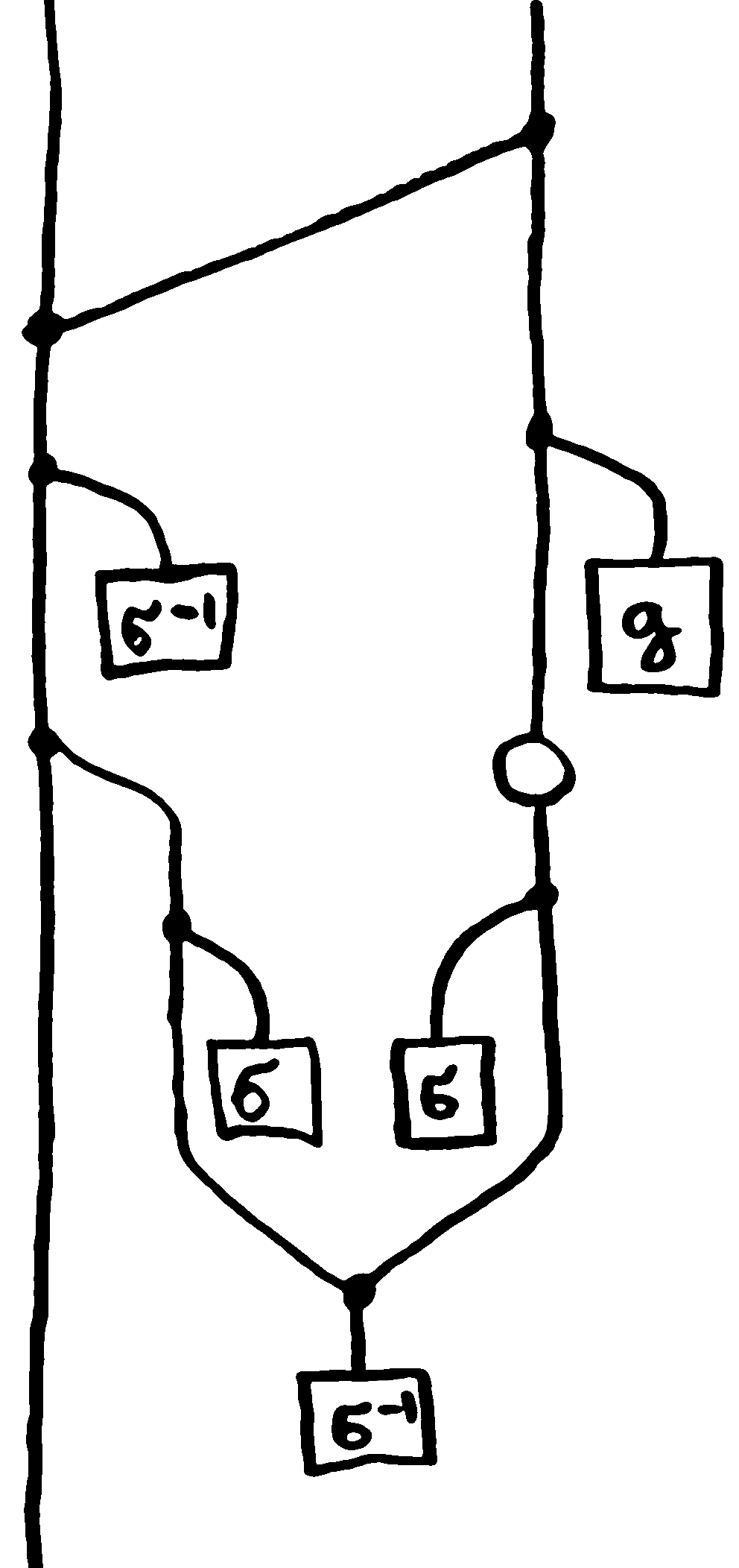} 
~=~
\scanpic{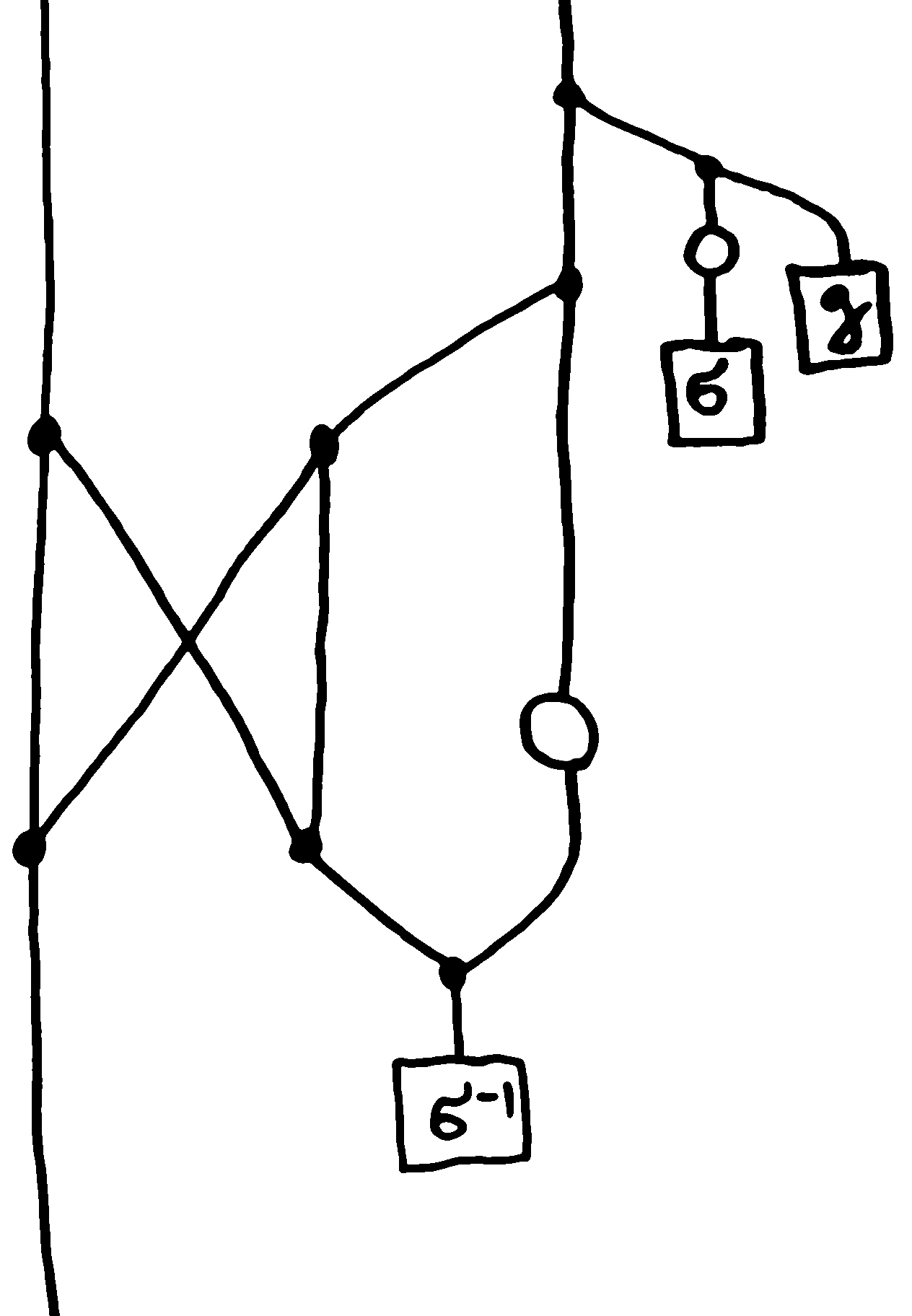} 
\overset{(3)}= 
\beta^{-1} \cdot \text{rhs of \eqref{eq:hexagon-6R}} 
\ee
In step (1) we used the expression $\delta = ( \id \otimes (M_g \circ S)) \circ \gamma^{-1}$ obtained from (\ref{eq:delta-expression-ab}\,a) as well as \eqref{eq:gammainv-via-sigma_2nd}. For step (3) use the bubble-property and $S \circ \sigma = M_{g^{-1}} \circ \sigma$ from \eqref{eq:S-sigma-g_2nd}.

\medskip\noindent
{\bf Cases 7L, 7R:} 
We will need that
\begin{align}
  \id_H &= \Ad_\sigma^{\,2} \circ \Ad_\sigma^{-1} \circ \Ad_\sigma^{-1} 
  \overset{\text{\eqref{eq:Adsig-inv}}}= 
 \Ad_\sigma^{\,2} \circ \Ad_\sigma \circ \Ad_g^{-1} \circ S^2 \circ \Ad_\sigma \circ \Ad_g^{-1} \circ S^2
  \nonumber \\
 &   \overset{\text{cond.\,c)}}{\underset{\text{\eqref{eq:Adsig-Adg}}}{=}}
 \Ad_\sigma^{\,4} \circ \Ad_g \circ S^{-2} \circ \Ad_g^{-1} \circ S^2
  \overset{\text{(*)}}=
  \Ad_\sigma^{\,4} \ ,
  \label{eq:Adsig4=id}
\end{align}
where (*) follows form $S \circ \Ad_g = \Ad_g \circ S$ (which in turn follows from Lemma \ref{lem:g-grouplike-2}\,(ii)). Turning to hexagon \eqref{eq:hexagon-7L}, we find
\be \label{eq:7L-aux1}
\beta^{-1} \cdot \text{rhs of \eqref{eq:hexagon-7L}} 
~=~
\scanpic{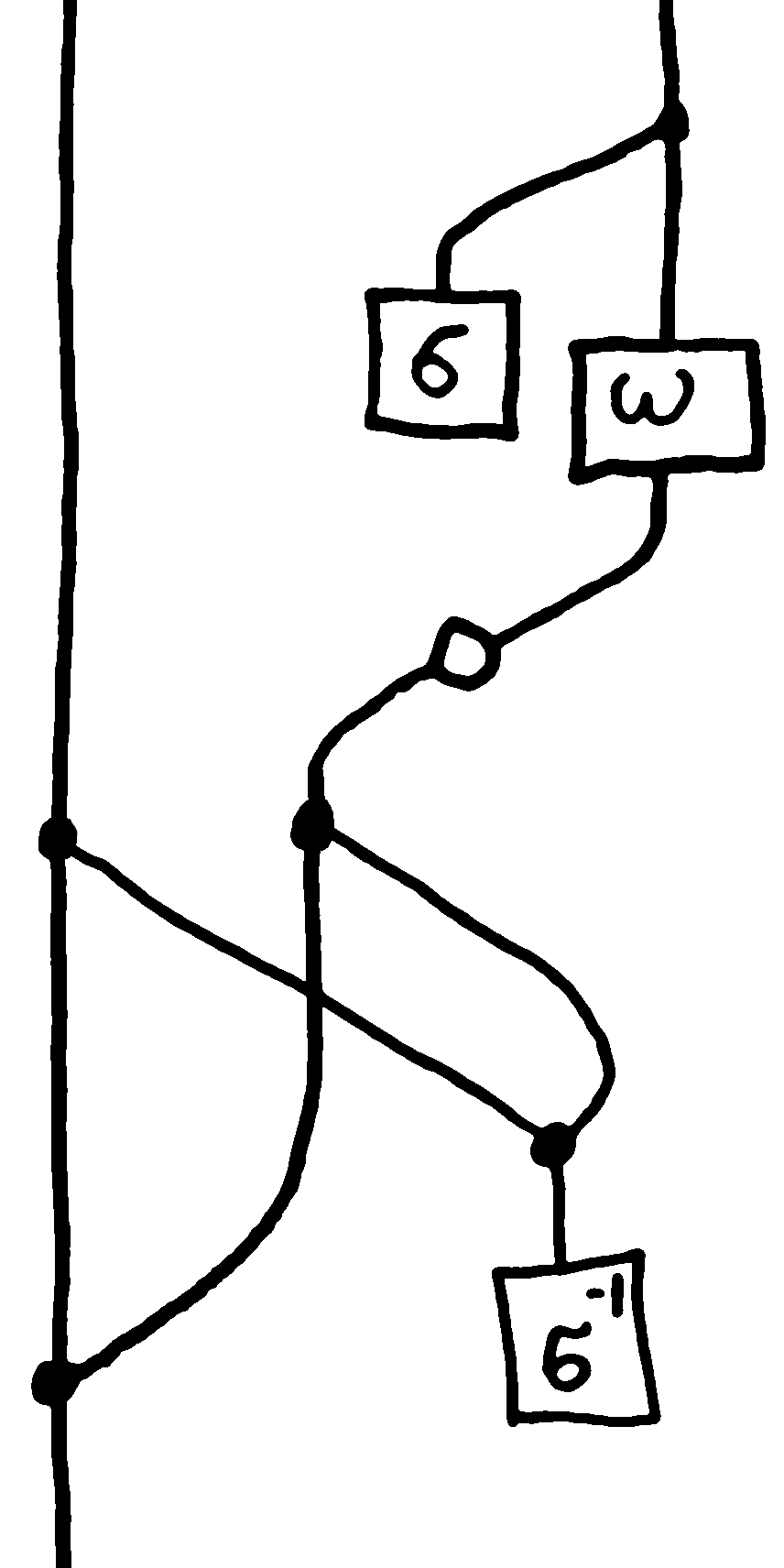}
~=~
\scanpic{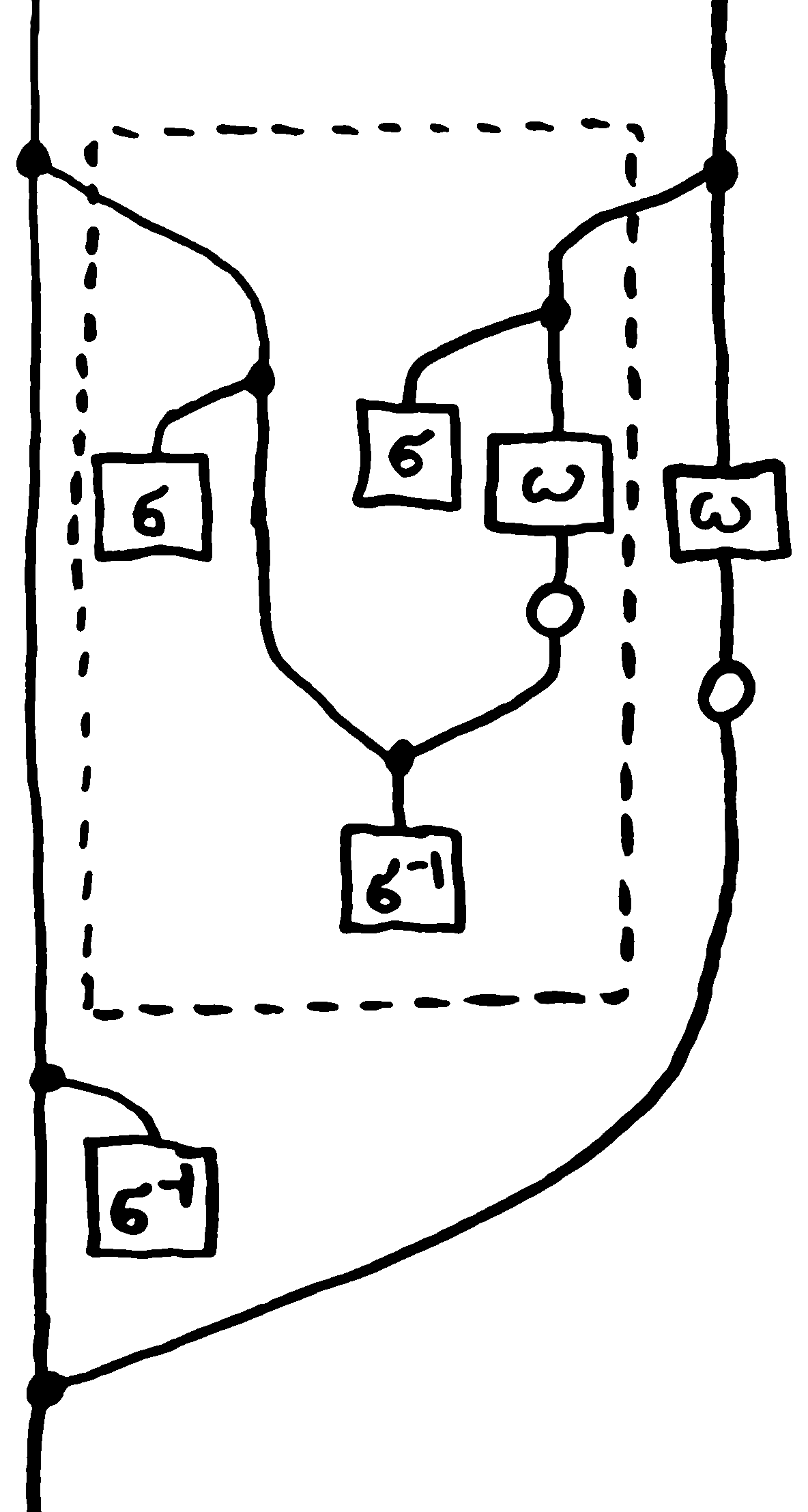}
\quad .
\ee
Comparing the right hand side of \eqref{eq:7L-aux1} to the left hand side of \eqref{eq:hexagon-7L}, we see that to establish \eqref{eq:hexagon-7L} it suffices to show that the morphism contained in the dashed box is equal to $\tilde\delta$. The morphism in the dashed box is equal to the first morphism in the following chain of equalities:
\be \label{eq:check-7L-aux1}
\begin{array}{l}
  ({}_\sigma M \otimes {}_\sigma M) \circ (\id \otimes (\omega \circ S)) \circ \Delta \circ \sigma^{-1}
\\[.5em] \qquad  
\overset{\text{\eqref{eq:gammainv-via-sigma_2nd}}}= ~
  \big(\Ad_\sigma \otimes \{ {}_\sigma M \circ S \circ {}_{\sigma^{-1}}M\} \big) \circ (\id \otimes \omega) \circ \gamma^{-1}
\\[.5em] \qquad  
\overset{\text{\eqref{eq:hexagon-3R-alt2}}}= ~
  \big(\Ad_\sigma^{\,2} \otimes \{ {}_\sigma M \circ S \circ {}_{\sigma^{-1}}M \circ S \circ \Ad_\sigma^{-1} \} \big) \circ \gamma^{-1}
\\[.5em] \qquad  
~\,\overset{\text{(*)}}=\,~ ~
  (\id \otimes {}_{g^{-1}}M) \circ (\Ad_\sigma^{\,2} \otimes \Ad_\sigma^{-2}) \circ \gamma^{-1}
\\[.5em] \qquad  
~\overset{\text{(**)}}=~ ~
  (\id \otimes {}_{g^{-1}}M) \circ \gamma^{-1}
~\overset{\text{\eqref{eq:tilde-delta-def}}}=~ \tilde\delta \ .
\end{array}
\ee
For step (*) first note that
\be
{}_\sigma M \circ S 
= 
S \circ M_{S^{-1} \circ \sigma} 
\overset{\text{\eqref{eq:S-sigma-g_2nd}}}= 
S \circ M_\sigma \circ M_{g^{-1}}
\overset{\text{cond.\,d)}}= 
S \circ M_{g} \circ M_\sigma 
=
{}_{g^{-1}}M \circ S \circ M_\sigma  \ .
\ee
Inserting this into the expression in curly brackets gives 
\be
 {}_\sigma M \circ S \circ {}_{\sigma^{-1}}M \circ S \circ \Ad_\sigma^{-1}
 =
 {}_{g^{-1}}M \circ S \circ \Ad_\sigma^{-1} \circ S \circ \Ad_\sigma^{-1}
 \overset{\text{cond.\,c)}}=
 {}_{g^{-1}}M \circ \Ad_\sigma^{-2}
\ee
Step (**) in \eqref{eq:check-7L-aux1} follows by first using \eqref{eq:gamma-and-c-gamma} to see $(\id \otimes \Ad_\sigma^{-2})\circ \gamma = (\Ad_\sigma^{\,2} \otimes \id)\circ \gamma$ and then using $\Ad_\sigma^{\,4} = \id_H$ from \eqref{eq:Adsig4=id}.
This establishes \eqref{eq:hexagon-7L}. 

To show \eqref{eq:hexagon-7R}, we start with the following series of equalities:
\begin{align}
(\id \otimes \omega) \circ \delta 
& \overset{\text{\eqref{eq:gdp-via-Hopf}}}= 
(\id \otimes \omega) \circ (\id \otimes (M_g \circ S)) \circ \gamma^{-1}
\nonumber \\ & \overset{\text{\eqref{eq:check-4R-aux2}}}= 
(\id \otimes (M_g \circ S)) \circ ( \Ad_\sigma^{-1}  \otimes (\Ad_\sigma^{-1} \circ S^{-1}) ) \circ c_{H,H} \circ \gamma^{-1}
\nonumber \\ & \overset{\text{\eqref{eq:Adsig-inv}}}= 
\big(\Ad_\sigma^{-1} \otimes \{ M_g \circ S \circ \Ad_g \circ S^{-2} \circ \Ad_\sigma \} \big) \circ c_{H,H} \circ \gamma
\nonumber \\ & ~\,=\,~ 
(\id \otimes ({}_gM \circ S^{-1})) \circ c_{H,H} \circ ( \Ad_\sigma \otimes \Ad_\sigma^{-1} ) \circ \gamma
\nonumber \\ & \overset{\text{\eqref{eq:gamma-via-sigma_2nd}}}= 
(\id \otimes ({}_gM \circ S^{-1})) \circ c_{H,H} \circ ( M_{\sigma^{-1}} \otimes {}_{\sigma^{-1}}M ) \circ \Delta \circ \sigma
\nonumber \\ & ~\,=\,~ 
\big( {}_{\sigma^{-1}}M \otimes \{ {}_gM \circ {}_{S^{-1}(\sigma^{-1})}M \circ S^{-1}\}\big) \circ c_{H,H} \circ \Delta \circ\sigma
\nonumber \\ & \overset{\text{\eqref{eq:S-sigma-g_2nd}}}{\underset{\text{cond.\,d)}}{=}} 
\big( {}_{\sigma^{-1}}M \otimes \{  {}_{\sigma^{-1}}M \circ S^{-1}\}\big) \circ c_{H,H} \circ \Delta \circ\sigma
\ .
\label{eq:7R-aux2}
\end{align}
Now recall the rewriting of the left hand side of \eqref{eq:hexagon-7R} given in \eqref{eq:7R-aux1}. Substituting \eqref{eq:7R-aux2} into the right hand side of \eqref{eq:hexagon-7R} we immediately see that this is equal to the right hand side of \eqref{eq:7R-aux1}. This proves \eqref{eq:hexagon-7R}.

\medskip\noindent
{\bf Cases 8L, 8R:} 
When composing \eqref{eq:pentagon-3} with $\id \otimes \id \otimes S$ one obtains
\be\label{eq:ga-1_doubling}
 \scanPIC{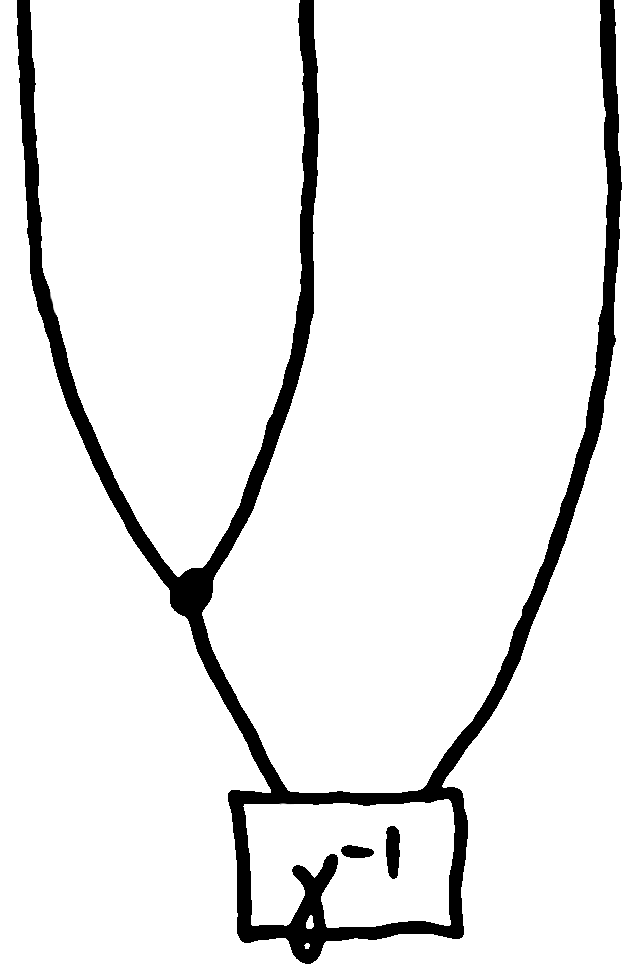} ~=~  \scanPIC{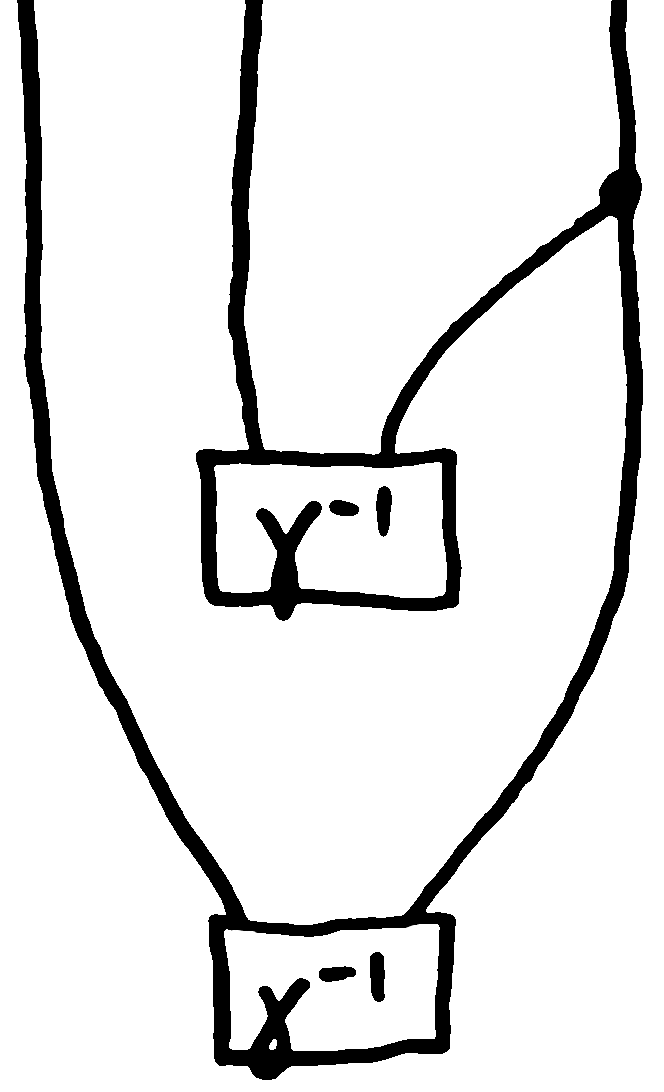}
 \quad .
\ee
Using $\phi = (S \otimes (\lambda \circ \mu)) \circ (\gamma \otimes \id)$ from \eqref{eq:gdp-via-Hopf} and inserting expression \eqref{eq:gamma-via-sigma_2nd} for $\gamma$ one obtains the first equality in
\be\label{eq:phi-alt1}
  \phi 
  ~=~ 
  \scanPIC{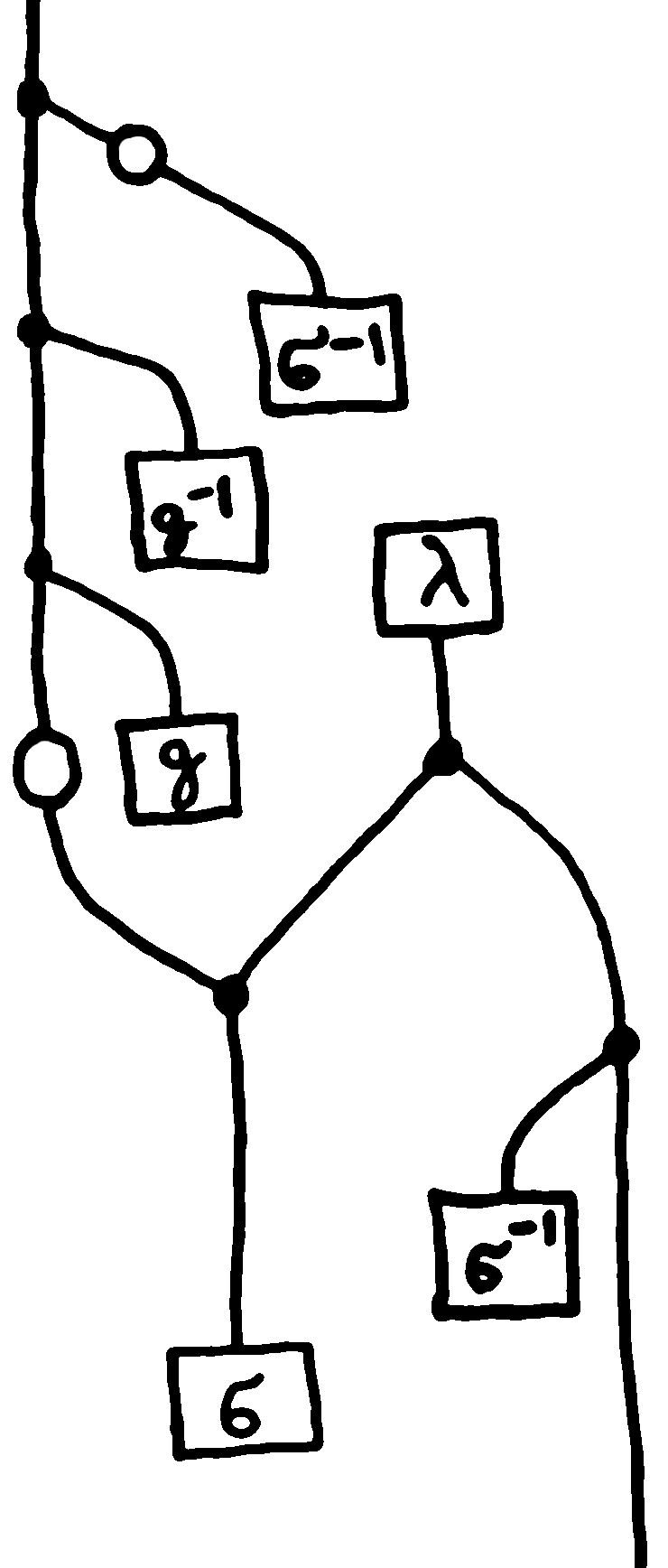}
  ~=~ 
  \scanPIC{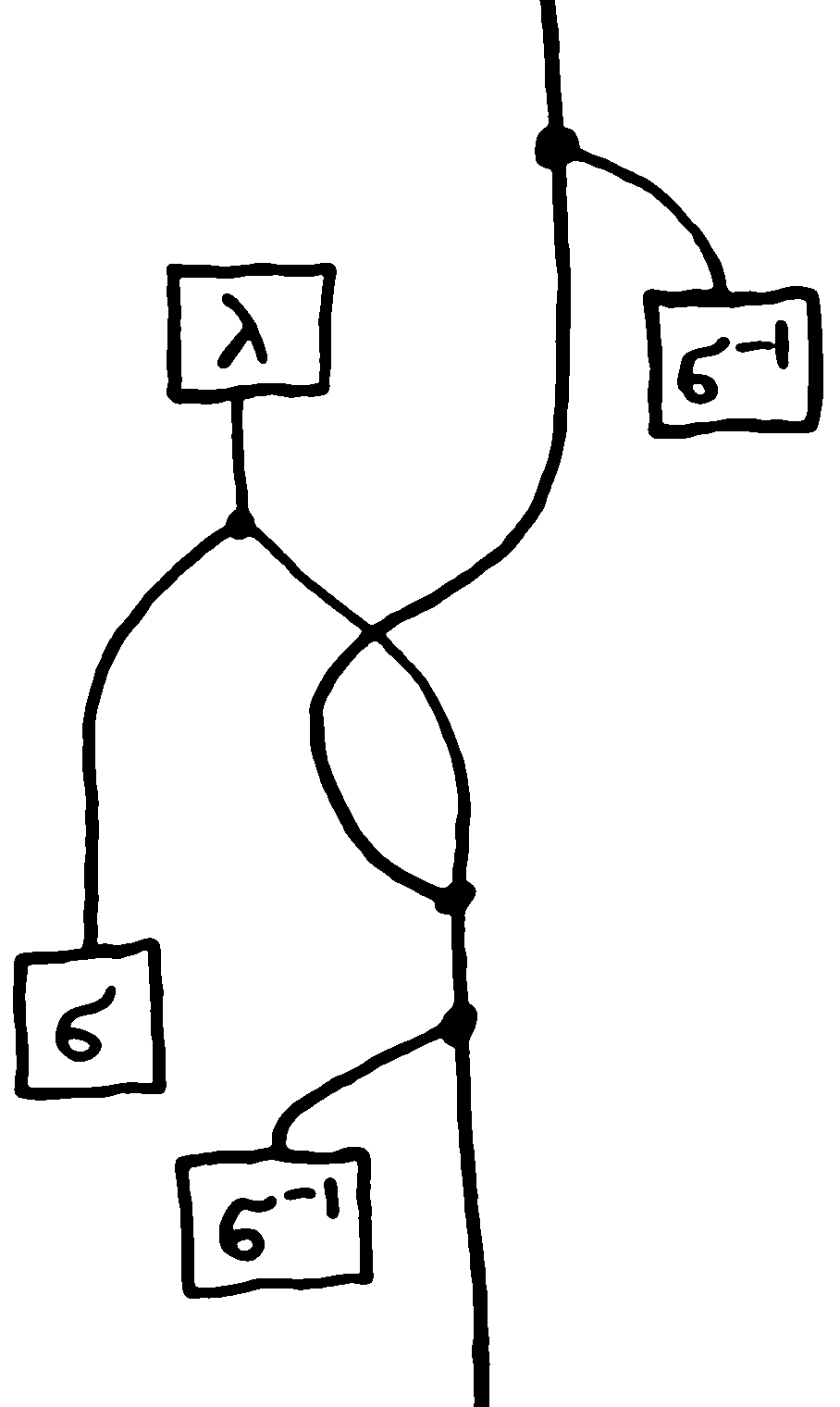}
  \quad .
\ee
The second equality is Lemma \ref{lem:move-coprod-past-lam}\,a) together with \eqref{eq:S-sigma-g_2nd}.

To show \eqref{eq:hexagon-8R}, consider the equalities
\be
  \text{lhs of \eqref{eq:hexagon-8R}} 
  \overset{\text{\eqref{eq:phi-alt1}}}{\underset{\text{\eqref{eq:gdp-via-Hopf}}}{=}}  
  \scanpic{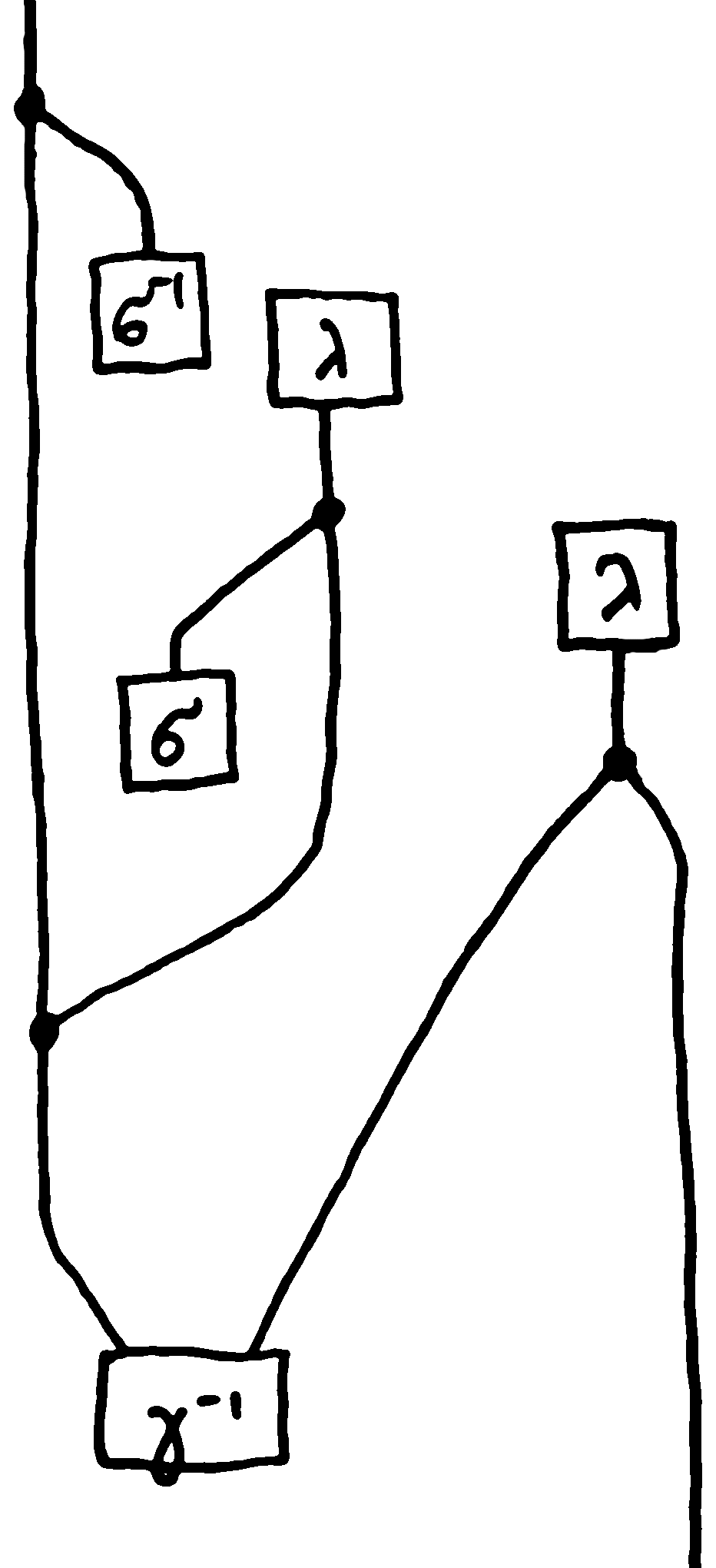}
  \overset{\text{\eqref{eq:ga-1_doubling}}}{\underset{\text{\eqref{eq:gdp-via-Hopf}}}{=}}  
  \scanpic{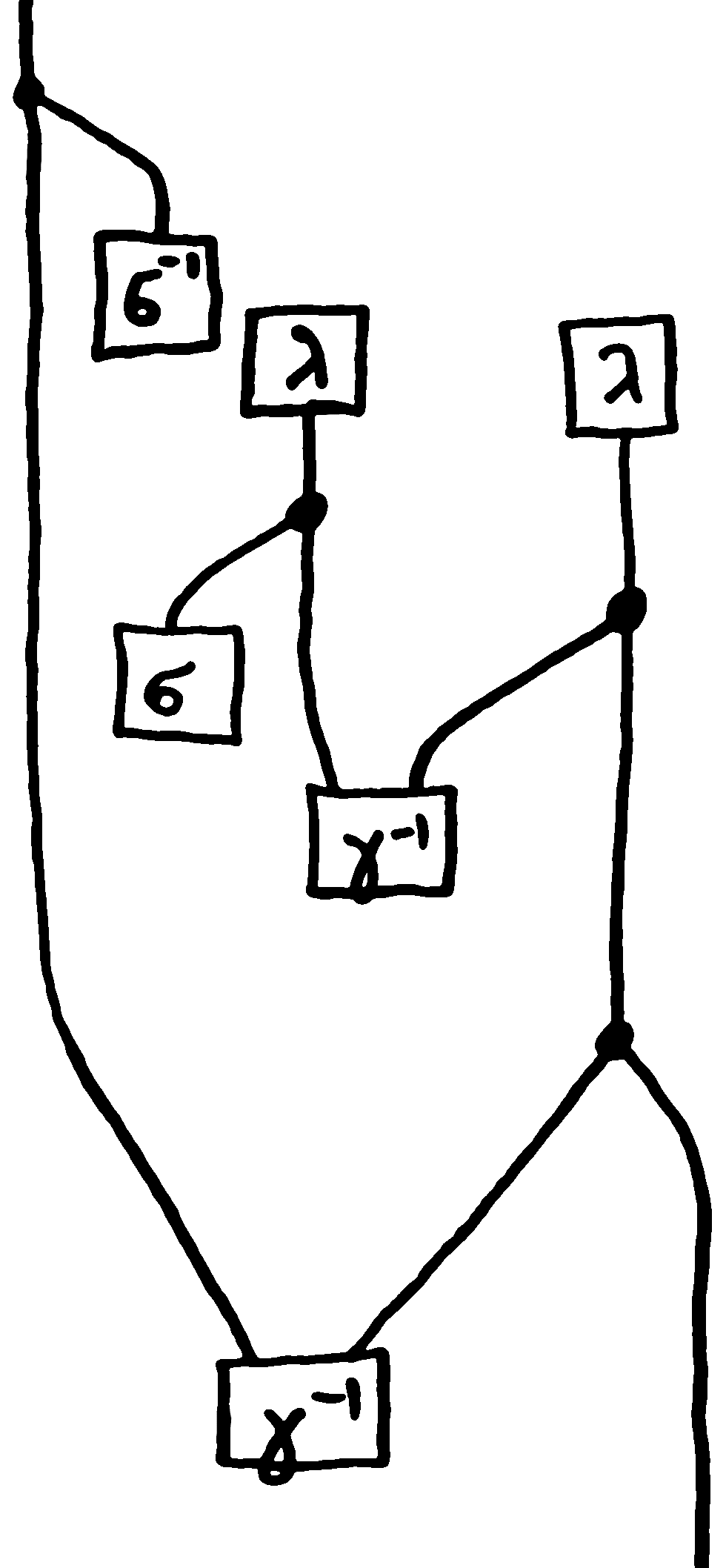}
  \overset{\text{\eqref{eq:lambda-S-aux1}}}{=}
  \lambda(\sigma) \cdot \scanpic{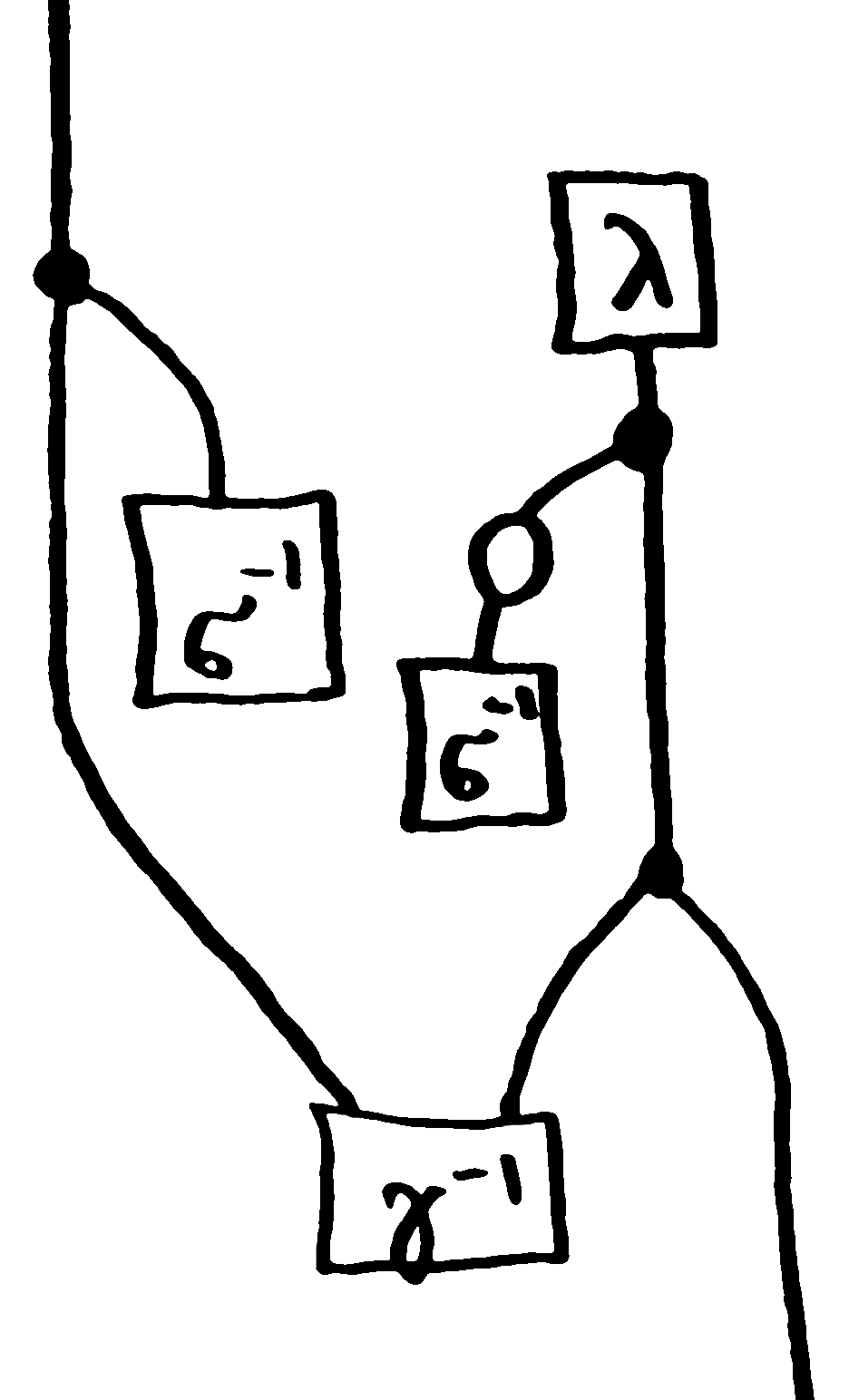}
  \overset{(*)}{=}
  \text{rhs of \eqref{eq:hexagon-8R}} 
\ee
Since $\lambda(\sigma) = \beta^2$ by condition b), to show (*) it is enough to show that $\lambda \circ {}_{S \circ \sigma^{-1}}M = \lambda \circ M_{\sigma^{-1}}$. This in turn follows from
\be \label{eq:check-8R-aux1}
  \lambda \circ {}_{S \circ \sigma^{-1}}M
  \overset{\text{\eqref{eq:S-sigma-g_2nd}}}{=}
  \lambda \circ {}_gM \circ {}_{\sigma^{-1}}M
  \overset{\text{Lem.\,\ref{lem:lambda-S-g}}}{=}
  \lambda \circ S \circ {}_{\sigma^{-1}}M
  \overset{\text{cond.\,b)}}{=}
  \lambda \circ \Ad_\sigma \circ {}_{\sigma^{-1}}M
  = 
  \lambda \circ M_{\sigma^{-1}} \ .
\ee

For hexagon \eqref{eq:hexagon-8L} we will start from the alternative formulation in \eqref{eq:hex8L-alt1}. Using \eqref{eq:hexagon-8R}, which we just proved, the left hand side of \eqref{eq:hex8L-alt1} becomes
\be
  \omega \circ S \circ (\phi \circ {}_\sigma M \circ \phi) \circ S
  = \beta^2 \cdot
  \omega \circ S \circ (M_{\sigma^{-1}} \circ \phi  \circ M_{\sigma^{-1}}) \circ S \ .
\ee
Thus, together with \eqref{eq:phi-1_via_phi} so far we have
\be \label{eq:hex8L-alt2}
  \text{\eqref{eq:hexagon-8L}}
  \quad \Leftrightarrow \quad
  S \circ M_{\sigma^{-1}} \circ \omega \circ \phi  \circ M_{\sigma^{-1}} \circ S 
  = M_{\sigma^{-1}} \circ {}_g M \circ \phi \circ S \circ  M_{\sigma^{-1}} \ .
\ee
On the left hand side of \eqref{eq:hex8L-alt2}, first use \eqref{eq:gdp-via-Hopf} to replace $\phi$ by $(\id \otimes (\lambda \circ \mu)) \circ (\gamma^{-1} \otimes \id)$ and then use monoidality of $\omega$ to move $\omega$ as in $(\omega \otimes \id) \circ \gamma^{-1} = (\id \otimes \omega) \circ \gamma^{-1}$. Next, substitute the right hand side of \eqref{eq:hexagon-3R-alt2} for $(\id \otimes \omega) \circ \gamma^{-1}$ and apply \eqref{eq:check-8R-aux1} to $\lambda \circ M_{\sigma^{-1}}$. Together with $S \circ \Ad_\sigma^{-1} = \Ad_\sigma^{-1} \circ S^{-1}$ from condition c), this gives the first equality in
\be \label{eq:check-8L-aux1}
  \text{lhs of \eqref{eq:hex8L-alt2}} ~=~ \scanpic{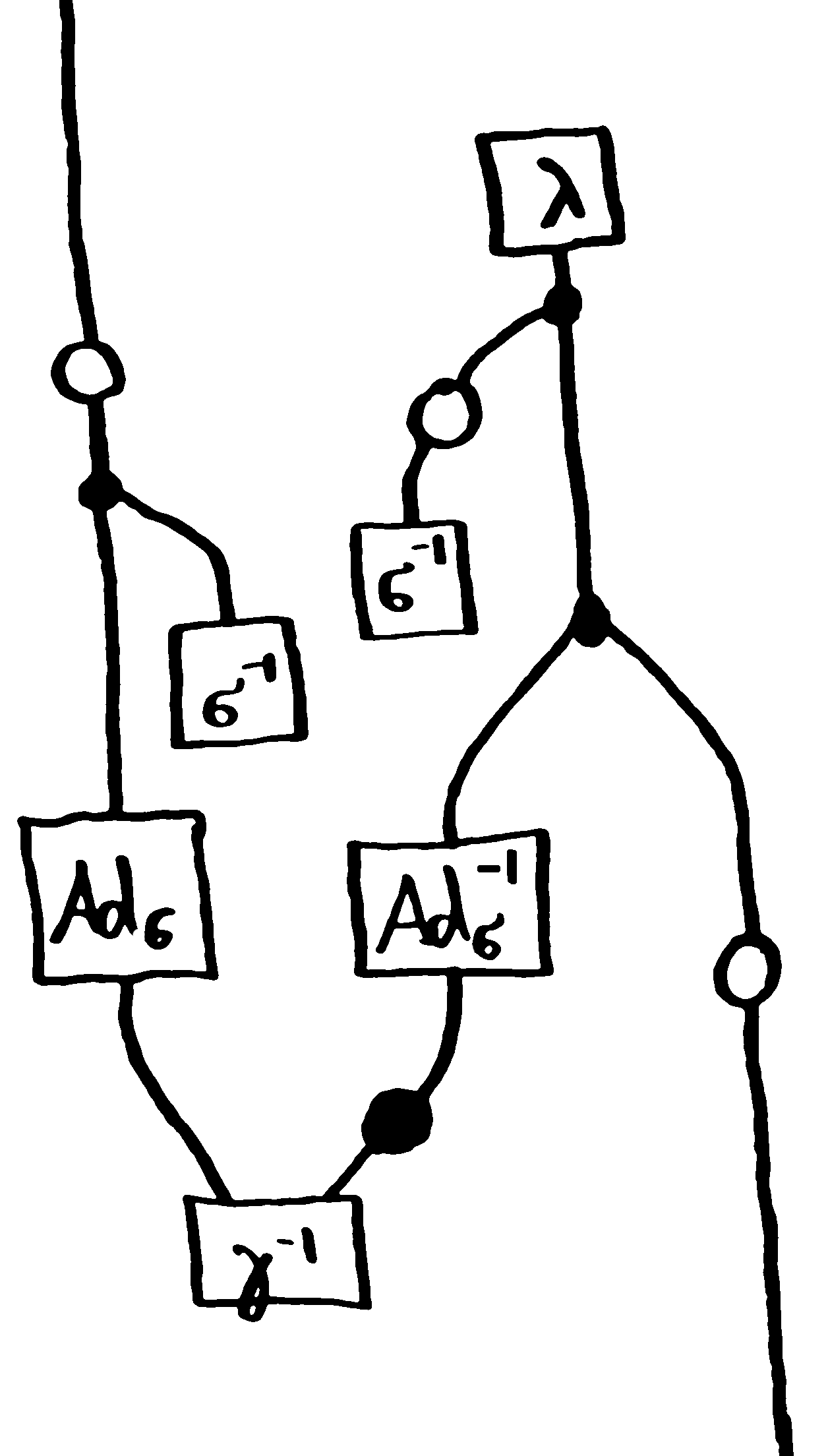} ~\overset{\text{(*)}}=~ \scanpic{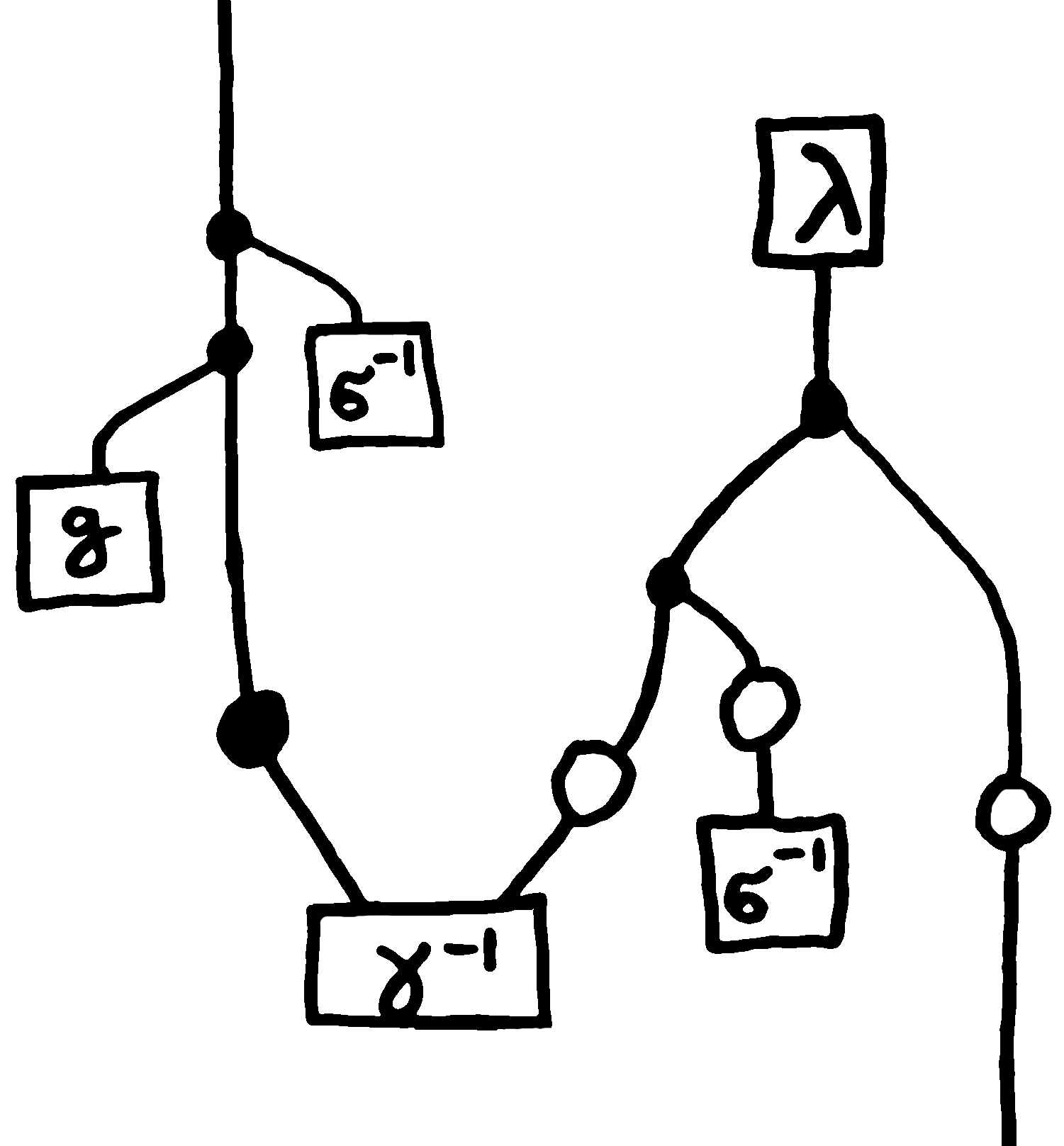} \quad .
\ee
Two separate simplifications make up step (*). Firstly, on the left-most $H$-leg we have
\be
  S \circ M_{\sigma^{-1}} \circ \Ad_\sigma 
  \overset{\text{cond.\,c)}}{=}
  {}_{S \circ \sigma^{-1}}M \circ \Ad_\sigma \circ S^{-1}
  \overset{\text{\eqref{eq:S-sigma-g_2nd}}}{=}
  {}_gM \circ M_{\sigma^{-1}} \circ S^{-1} \ .
\ee
Secondly, on the middle $H$-leg we compute
\be
  {}_{S \circ \sigma^{-1}}M \circ \Ad_\sigma^{-1} \circ S^{-1}
  \overset{\text{\eqref{eq:S-sigma-g_2nd}}}{\underset{\text{\eqref{eq:Adsig-inv}}}{=}}  
   {}_gM \circ {}_{\sigma^{-1}}M \circ \Ad_\sigma \circ \Ad_g^{-1} \circ S
   =
   M_{\sigma^{-1}} \circ M_g \circ S
  \overset{\text{\eqref{eq:S-sigma-g_2nd}}}{=}
   M_{S \circ \sigma^{-1}} \circ S \ .
\ee
On the right hand side of \eqref{eq:check-8L-aux1} we can now replace $(S^{-1} \otimes S) \circ \gamma^{-1} = \gamma^{-1}$ and move $S \circ \sigma^{-1}$ from the middle $H$-leg to the right-most $H$-leg. This results in the right hand side of \eqref{eq:hex8L-alt2}, establishing \eqref{eq:hexagon-8L} and completing the proof of Proposition \ref{prop:Hopf-to-hex}.
\end{proof}

\begin{remark}\label{rem:braided-via-gamma}
(i) The {\em Drinfeld element} of a quasi-triangular Hopf algebra is defined to be $u = \mu \circ (S \otimes \id) \circ c_{H,H} \circ R$. If the Hopf algebra satisfies the properties in Theorem \ref{thm:main2}\,(i), we can compute
\begin{align}
u 
&\overset{\text{\eqref{eq:R-via-gamma-op}}}{=}
\mu \circ (S \otimes \id) \circ c_{H,H} \circ (\Ad_\sigma^{-1} \otimes \id) \circ c_{H,H} \circ \gamma^{-1}
=
\mu \circ (S \otimes \Ad_\sigma^{-1} ) \circ \gamma^{-1}
\nonumber\\
&\overset{\text{\eqref{eq:gammainv-via-sigma_2nd}}}=
\mu \circ (S \otimes \Ad_\sigma^{-1} ) \circ (M_\sigma \otimes {}_\sigma M) \circ \Delta \circ \sigma^{-1}
= {}_{S \circ \sigma}M \circ M_\sigma \circ 
\mu \circ (S \otimes \id) \circ \Delta \circ \sigma^{-1}
\nonumber\\
&\overset{\text{cond.\,d)}}{=} 
g \sigma^2 \ .
\label{eq:Drinfeld-el}
\end{align}
Since by \eqref{eq:Adsig-inv}, $S^2 = \Ad_g \circ \Ad_\sigma^{-2}$, and by \eqref{eq:Adsig4=id}, $\Ad_\sigma^{\,4} = \id_H$, the element $u$ satisfies (see \cite[Prop.\,VIII.1]{Kassel-book} for the vector space case)
\be \label{eq:S2=Adu}
  S^2 = \Ad_u \ .
\ee
(ii)
As in Corollary \ref{cor:main1-i} it is helpful to formulate Theorem \ref{thm:main2}\,(i) without appealing to the dual Hopf algebra $H^\vee$. To this end we rewrite conditions a) and e) as
\begin{itemize}{\em
\item[a')]
$\gamma$ is determined through $\sigma$ by 
$$
  \gamma =  ({}_{\sigma^{-1}}M \otimes M_{\sigma^{-1}}) \circ \Delta \circ \sigma \ .
$$
\item[e')] The natural isomorphism $\omega$ evaluated on $H$ satisfies
$$
  (\id \otimes \omega_H) \circ \gamma^{-1} = \big\{  \Ad_\sigma \otimes ( S \circ \Ad_\sigma^{-1} ) \big\} \circ \gamma^{-1} 
  \ .
$$
}\end{itemize}
For a') we used \eqref{eq:gamma-via-sigma_2nd}, and e') is  \eqref{eq:hexagon-3R-alt2} together with the definition \eqref{eq:Hopf-via-gdp} of $\Gamma$ in terms of $\gamma^{-1}$ (cf.\ \eqref{eq:hexagon-3R-aux1}).
\end{remark}

\subsection{Braided monoidal equivalences}\label{sec:braided-equiv}

Let $H,\lambda,\sigma,\beta$ be data satisfying the conditions of Theorem \ref{thm:main2} (or rather those of Corollary \ref{cor:main1-i} and Remark \ref{rem:braided-via-gamma}\,(ii)) and let  $\Cc(H,\lambda,\sigma,\beta)$ be the corresponding $\Zb/2\Zb$-graded braided monoidal category $\Rep_\Sc(H) + \Sc$. Given a Hopf algebra isomorphism  $\varphi : H' \to H$, we can transport the data $\lambda,\sigma,\beta$ from $H$ to $H'$:
\be\label{eq:braid-mon-gauge-xfer}
  \lambda' = \lambda\circ \varphi \quad ,\qquad 
  \sigma' = \varphi^{-1}\circ\sigma
  \quad ,\qquad 
  \beta' = \beta \ .
\ee
Clearly, the new data $\lambda',\sigma',\beta'$ also satisfies the conditions of Theorem \ref{thm:main2}. 
The $\Zb/2\Zb$-graded functor from Section \ref{sec:gauge},
\be
  G=G(\varphi):\Cc(H,\lambda,\sigma,\beta)\to \Cc(H',\lambda',\sigma',\beta') \ ,
\ee  
is now a braided equivalence. 
Indeed, the monoidal structure $G(\varphi)_{A,B}:G(A\ot B)\to G(A)\ot G(B)$ on $G(\varphi)$ is compatible with the braidings in $\Cc(H,\lambda,\sigma,\beta)$ and $\Cc(H',\lambda',\sigma',\beta')$, i.e.\ the diagram 
\be
\xymatrix{G(A\otimes B) \ar[d]_{G(c_{A,B})}\ar[rr]^{G_{A,B}} && G(A)\otimes G(B)  \ar[d]^{c'_{G(A),G(B)}}
\\
G(B\otimes A) \ar[rr]^{G_{B,A}} && G(B)\otimes G(A)  }
\ee
commutes. For example the coherence for $A \in\Cc_0, B\in\Cc_1$ is equivalent to the equation $\sigma' = \varphi^{-1}\circ\sigma$. The remaining three cases are equally straightforward. Altogether this proves Theorem \ref{thm:main2}\,(ii):

\begin{proposition}\label{prop:braid-mon-equiv}
Let $H,\lambda,\sigma,\beta$ and $H',\lambda',\sigma',\beta'$ be related via a Hopf algebra isomorphism as in \eqref{eq:braid-mon-gauge-xfer}. Then the categories $\Cc(H,\lambda,\sigma,\beta)$ and $\Cc(H',\lambda',\sigma',\beta')$ are equivalent as braided monoidal categories.
\end{proposition}

Analogous to Proposition \ref{prop:mon-equiv}, it is in general not true that {\em all} braided equivalences are of the form stated above.

\medskip

Denote by $\Aut_{\mathrm{Hopf}}(H,\lambda,\sigma)$ the group of Hopf algebra automorphisms of $H$ which leave $\lambda$ and $\sigma$ invariant (via \eqref{eq:braid-mon-gauge-xfer}). Since $\sigma$ determines the copairing $\gamma$ we have an injective homomorphism $\Aut_{\mathrm{Hopf}}(H,\lambda,\sigma)\to \Aut_{\mathrm{Hopf}}(H,\gamma,\lambda)$. Together with Proposition \ref{gma} this gives 
an injective homomorphism 
\be\lb{bma}
  \Aut_{\mathrm{Hopf}}(H,\lambda,\sigma) \longrightarrow \Aut_{\mathrm{br}}(\Cc(H,\lambda,\sigma,\beta)/\Sc)
\ee
into the group of isomorphism classes of braided monoidal autoequivalences of $\Cc(H,\gamma,\lambda)$ over $\Sc$. 

\medskip

Recall from Section \ref{sec:Hopf} that the reverse category $\overline{\Cc}$ is equal to $\Cc$ as a monoidal category but has the braiding $\overline{c}_{X,Y} = c_{Y,X}^{-1}$.

\begin{proposition}
\label{prop:Cbar-equiv-C}
Let $\Cc(H,\lambda,\sigma,\beta)$ be a braided monoidal category as in Theorem \ref{thm:main2}. Then the triple $\lambda,\sigma^{-1},\beta^{-1}$ also satisfies the conditions of Theorem \ref{thm:main2} and there is a braided monoidal equivalence
$$
  \overline{\Cc(H,\lambda,\sigma,\beta)} \simeq \Cc(H,\lambda,\sigma^{-1},\beta^{-1}) \ .
$$
\end{proposition}

\begin{proof}
We will first check that the triple $\lambda,\sigma^{-1},\beta^{-1}$ satisfies the conditions of Theorem \ref{thm:main2}. Since $\Ad_\sigma$ is an algebra map, the inverse identity to \eqref{eq:check-4L-aux1} reads
\be \label{eq:gam-om-Ad}
  (\id \otimes \omega_H) \circ \gamma = ( \Ad_\sigma \otimes \Ad_\sigma^{-1}) \circ \gamma^{-1}
  \ .
\ee
The copairing $\gamma$ is defined via $\sigma$ as
$\gamma(\sigma) =  ({}_{\sigma^{-1}}M \otimes M_{\sigma^{-1}}) \circ \Delta \circ \sigma$, see Remark \ref{rem:braided-via-gamma}\,(ii\,a'). Write $\gamma = \gamma(\sigma)$ and $\gamma' = \gamma(\sigma^{-1})$. Then
\be\label{eq:gam-om-gam}
  (\id \otimes \omega_H) \circ \gamma  
  \overset{\text{\eqref{eq:gam-om-Ad}}}{\underset{\text{\eqref{eq:gammainv-via-sigma_2nd}}}{=}}
  ( \Ad_\sigma \otimes \Ad_\sigma^{-1}) \circ (M_\sigma \otimes {}_\sigma M) \circ \Delta \circ \sigma^{-1} 
  = ({}_{\sigma}M \otimes M_{\sigma}) \circ \Delta \circ \sigma^{-1} = \gamma'
  \ .
\ee
With the help of \eqref{eq:gam-om-gam} and naturality and monoidality of $\omega$ it is straightforward to verify that the pair $\lambda,\gamma'$ satisfies the conditions in Corollary \ref{cor:main1-i}. 

Conditions a)--e) of Theorem \ref{thm:main2} for $\lambda,\sigma^{-1},\beta^{-1}$ can be seen as follows:
\\[.3em]
\nxt For a) and e) use Remark \ref{rem:braided-via-gamma}\,(ii).
Condition a') holds by the definition of $\gamma'$ and for e') use \eqref{eq:gam-om-gam}. 
\\[.3em]
\nxt For b) we need to show $\lambda \circ S = \lambda \circ \Ad_\sigma^{-1}$ and $\lambda(\sigma^{-1}) = \beta^{-2}$. For the first identity, compose both sides of $\lambda \circ S = \lambda \circ \Ad_\sigma$ in b) with $S^{-1} \circ \Ad_\sigma^{-1}$ and use $\Ad_\sigma \circ S = S^{-1} \circ \Ad_\sigma$ from c). For the second identity insert \eqref{eq:hexagon-8R} into $\eps \circ (-) \circ \eta$ and use $\nu = \beta \cdot \sigma^{-1}$, $\eps \circ \phi = \lambda$ and $\phi \circ \eta = \Lambda'$ (the left integral from Lemma \ref{lem:integral-cointegral}) to get $\lambda \circ {}_\sigma M \circ \Lambda' = \beta^2 \cdot \eps(\sigma^{-1}) \cdot \lambda(\sigma^{-1})$. Now use the left integral property and the normalisation condition $\lambda \circ \Lambda' = \id_\one$, as well as $\eps(\sigma) = \id_\one = \eps(\sigma^{-1})$  from Lemma \ref{lem:A-prop-1}. 
\\[.3em]
\nxt For c) use that $\Ad_\sigma^{-1}$ is a Hopf algebra map from $H_\mathrm{cop} \to H$ which is the same as a Hopf algebra map $H \to H_\mathrm{cop}$. 
\\[.3em]
\nxt Finally, d) is immediate from \eqref{eq:S-sigma-g_2nd}

\smallskip

Next, we give the monoidal equivalence $I : \overline{\Cc} \to \Cc'$, where $\Cc = \Cc(H,\lambda,\sigma,\beta)$ and $\Cc' = \Cc(H,\lambda,\sigma^{-1},\beta^{-1})$ (which of course is the same as a monoidal equivalence $\Cc \to \Cc'$).
The underlying functor is the identity functor and the monoidal structure $I_{A,B}:A\ot_{\Cc} B\to A\ot_{\Cc} B$ is given by
\be
  I_{A^a,B^b} = id_{A^a}\ot_{\Cc} (\omega_{B^b})^a  \ ,
\ee
where $A^a \in \Cc_a$, $B^b \in \Cc_b$, and $(\omega_B)^0 = \id_B$, $(\omega_B)^1 = \omega_B$. We need to check that the diagram
\be
\raisebox{2em}{
  \xymatrix{A\otimes_\Cc(B\otimes_\Cc C)) \ar[d]_{\alpha_{A,B,C}}\ar[rr]^{I_{A,BC}} && A\otimes_{\Cc} (B\otimes_\Cc C) \ar[rr]^(.45){id\otimes_\Cc
I_{B,C}} & & A\otimes_{\Cc} (B\otimes_{\Cc} C) \ar[d]^{\alpha'_{A,B,C}}
\\
(A\otimes_\Cc B)\otimes_\Cc C \ar[rr]^{I_{A B,C}} && (A\otimes_\Cc B)\otimes_{\Cc} B \ar[rr]^(.45){I_{A,B}\otimes_\Cc id} & &
(A\otimes_{\Cc} B)\otimes_{\Cc} C }
}
\ee
commutes in all eight cases for the graded degrees of $A,B$ and $C$. With a little patience, one verifies that the five cases 000, 001, 100, 011, 110 are trivially true. Using monoidality and naturality of $\omega$,
as well as the definition of $\delta$ and $\phi$ in Corollary \ref{cor:main1-i}, it is easy to check that the remaining cases 010, 101, 111 all reduce to the identity \eqref{eq:gam-om-gam}.

\smallskip

Finally, we need to verify that $I : \overline{\Cc} \to \Cc'$ is indeed braided, that is, we need to verify commutativity of the diagram
\be
\raisebox{2em}{
\xymatrix{A\ot_\Cc B \ar[d]_{c_{B,A}^{-1}}\ar[rr]^{I_{A,B}} && A\ot_\Cc B  \ar[d]^{c'_{A,B}}
\\
B\ot_\Cc A \ar[rr]^{I_{B,A}} && B\ot_\Cc A  }
}
\qquad .
\ee
One easily checks that of the four cases of the graded degrees of $A$ and $B$, the only case which is not immediate is $A,B\in\Cc_0$ The latter amounts to $c_{H,H} \circ R^{-1} = R'$. To establish this identity, first note that
\be
  R^{-1} 
  \overset{\text{\eqref{eq:R-etc-via-Hopf}}}{=}
  ( M_{\sigma^{-1}} \otimes M_{\sigma^{-1}} ) \circ \Delta \circ \sigma \ .
\ee
Using this, we have
$c_{H,H} \circ R^{-1}
=
  ( M_{\sigma^{-1}} \otimes M_{\sigma^{-1}} ) \circ \Delta^\mathrm{cop} \circ \sigma 
\overset{\text{\eqref{eq:R-via-sigma_2-2nd}}}{=}
 R'$.
\end{proof}

\subsection{Ribbon structure}\label{sec:ribbon}

Let $\Cc= \Cc(H,\lambda,\sigma,\beta)$ be a braided monoidal category as in Theorem \ref{thm:main2} and suppose that $\sigma^2$ is central:
\be\label{eq:Adsig2=id-assumption}
\text{Assumption:} \qquad \Ad_\sigma^{\,2} = \id_H \ .
\ee
Under the above assumption, in this section we give a pivotal structure on $\Cc$. By definition, this turns $\Cc$ into a ribbon category whose twist isomorphisms are determined by the pivotal structure.

\medskip

Recall the identity $S^2 = \Ad_u = \Ad_g \circ \Ad_\sigma^{\,2}$ from \eqref{eq:S2=Adu}. From this we see that one immediate implication of \eqref{eq:Adsig2=id-assumption} is that
\be \label{eq:S2=Adg}
  S^2 = \Ad_g \ .
\ee

Let $\delta_U : U \to  U^{\vee\vee}$ be the pivotal structure of the ribbon category $\Sc$ (which is assumed to be symmetric with trivial twist as stated in Notations \ref{not:sec4}). For $M \in \Cc_0$ and $X \in \Cc_1$ define the isomorphisms
\be\label{eq:deltaC-def}
  \delta^{\Cc}_M = \delta_M \circ \rho^H_M \circ (g \otimes \id_M) \quad , \quad
  \delta^{\Cc}_X = \delta_X  \ .
\ee
Recall the right duality $(-)^*$ on $\Cc$ defined in Section \ref{sec:rigid}.

\begin{lemma}\label{lem:delta-nat-iso}
The family $\delta^{\Cc}_A : A \to A^{**}$ is a natural isomorphism $\delta^{\Cc} : \Id \to (-)^{**}$ in $\Cc$.
\end{lemma}

\begin{proof}
Since for $X \in \Cc_1$ we have $X^* = X^\vee$, $\delta^{\Cc}_X$ is an isomorphism $X \to X^{**}$. For $M \in \Cc_0$ iterate the definition of $\rho_{M^*}$ to obtain the $H$-action on $M^{**}$. Since $\Sc$ has trivial twist, and since by \eqref{eq:S2=Adg} we have $S^2 = \Ad_g$, one finds that
\be
  \rho_{M^{**}} = \delta_M \circ \rho_M \circ (\Ad_g \otimes \delta^{-1}_M) \ .
\ee  
It is then clear that $\rho_{M^{**}} \circ (\id_H \otimes \delta^{\Cc}_M) = \delta^{\Cc}_M \circ \rho_M$. Finally, naturality of $\delta^{\Cc}$ is immediate.
\end{proof}

\begin{remark}
A different family of isomorphisms is given by $\delta^{\Cc}_M = \delta_M \circ \rho^H_M \circ (u \otimes \id_M)$ where $u$ is the Drinfeld element from \eqref{eq:Drinfeld-el}. Since $S^2 = \Ad_u$ (see \eqref{eq:S2=Adu}), for this choice Lemma \ref{lem:delta-nat-iso} holds without the extra assumption of $\Ad_\sigma$ being an involution. However, the proof of monoidality of $\delta^{\Cc}$ in Proposition \ref{prop:delta-pivotal} below requires the choice \eqref{eq:deltaC-def}, see Remark \ref{rem:delta-choice}.
\end{remark}

We can now use $\delta^{\Cc}$ to define a left duality on $\Cc$ such that both the left and right dual of $A \in \Cc$ are given by $A^*$. To distinguish the left duality maps form those in Section \ref{sec:rigid}, we use hats on the left duality maps instead of tildes. Namely, for all $A \in \Cc$ we set
\begin{align}
\widehat\ev_A^\Cc
~&:=~
\big[\, 
 A \otimes_{\Cc} A^* 
 \xrightarrow{\delta_{A}^\Cc \otimes \id_{A^*}} 
 A^{**} \otimes_{\Cc} A^*
 \xrightarrow{\ev_{A^*}^\Cc} 
\one
 \, \big] \ ,
  \nonumber\\
\widehat\coev_A^\Cc
~&:=~
\big[\, 
 \one
 \xrightarrow{\coev_{A^*}^\Cc} 
 A^{*} \otimes_{\Cc} A^{**}
 \xrightarrow{\id_{A^*} \otimes (\delta_{A}^\Cc)^{-1} } 
A^* \otimes_{\Cc} A
\, \big] \ .
\label{eq:duality-hat-def}
\end{align}
Combining the maps $\widehat \ev_A^\Cc$ and $\widetilde \coev_A^\Cc$ one obtains an isomorphism $A^* \to {}^*\!A$, where ${}^*\!A$ was defined in Section \ref{sec:rigid}. We will henceforth only work with $A^*$ as the right and left dual and will no longer refer to ${}^*\!A$. 

Using the definition of $\delta^\Cc$ it is easy to give the new left duality maps explicitly. For $M \in \Cc_0$ we have
\begin{align}
\widehat\ev_M^\Cc
&= \widetilde\ev_M \circ (\rho_M \otimes \id_{M^*}) \circ (g \otimes \id_M \otimes \id_{M^*}) \quad ,
  \nonumber\\
\widehat\coev_M^\Cc
&=
(\id_{M^*} \otimes \rho_M) \circ (\id_{M^*} \otimes g^{-1} \otimes \id_M) \circ \widetilde\coev_M \ ,
\end{align}
and for $X \in \Cc_1$ 
\be
\widehat\ev_X^\Cc = \eps \otimes \widetilde\ev_X \quad , \quad
\widehat\coev_X^\Cc = \Lambda \otimes \widetilde\coev_X \ .
\ee
We have now gathered the ingredients to prove that $\delta^\Cc$ is monoidal.

\begin{proposition}\label{prop:delta-pivotal}
$\delta^{\Cc}$ is a pivotal structure on $\Cc$, i.e.\ it is a natural monoidal isomorphism $\Id \to (-)^{**}$.
\end{proposition}

\begin{proof}
By Lemma \ref{lem:delta-nat-iso}, $\delta^\Cc$ is a natural isomorphism $\Id \to (-)^{**}$. It remains to show that $\delta^{\Cc}$ is monoidal. This is equivalent to the requirement that the two morphisms $R,L : B^* \otimes_{\Cc} A^* \to (A \otimes_{\Cc} B)^*$ given below are equal.\footnote{
  That $R=L$ is equivalent to $\delta^\Cc$ being monoidal is standard; one place where a detailed proof is spelled out is \cite[Lem.\,2.12\,(ii)]{Carqueville:2010hu}.}
In writing $L$ and $R$ we have omitted the `$\otimes_{\Cc}$' between objects and morphisms, and we have written `$1$' for the identity morphism.
\begin{align}
R ~=~ \Big[ \, &B^* A^* 
\xrightarrow{=}
(B^* A^*) \one
\xrightarrow{1\,1\,\coev^{\Cc}_{AB}}
(B^* A^*) \{ (AB)(AB)^* \}
\nonumber\\ &
\xrightarrow{\alpha^{\Cc}_{B^*A^*,AB,(AB)^*}}
\{(B^* A^*) (AB)\}(AB)^* 
\xrightarrow{\alpha^{\Cc}_{B^*A^*,A,B} \, 1}
\{((B^* A^*) A)B\}(AB)^* 
\nonumber\\ &
\xrightarrow{(\alpha^{\Cc}_{B^*,A^*,A})^{-1}\,1\, 1}
\{(B^* (A^* A))B\}(AB)^* 
\xrightarrow{1 \, \ev^{\Cc}_A \, 1 \, 1}
\{(B^* 1)B\}(AB)^* 
\xrightarrow{=}
\{B^* B\}(AB)^* 
\nonumber\\ &
\xrightarrow{\ev^{\Cc}_B \, 1}
1(AB)^* 
\xrightarrow{=}
(AB)^* ~ \Big]
\nonumber \\
L ~=~ \Big[ \, &B^* A^* 
\xrightarrow{=}
\one (B^* A^*)
\xrightarrow{\widehat\coev^{\Cc}_{AB} \, 1\,1}
\{ (AB)^*(AB) \}(B^* A^*) 
\nonumber\\ &
\xrightarrow{(\alpha^{\Cc}_{(AB)^*,AB,B^*A^*})^{-1}}
 (AB)^*\{(AB) (B^* A^*) \}
\xrightarrow{1\,(\alpha^{\Cc}_{A,B,B^*A^*})^{-1}}
 (AB)^*\{A(B (B^* A^*)) \}
\nonumber\\ &
\xrightarrow{1\,1\,\alpha^{\Cc}_{B,B^*,A^*}}
 (AB)^*\{A((B B^*) A^*) \}
\xrightarrow{1 \, 1 \, \widehat\ev^{\Cc}_B \, 1}
 (AB)^*\{(A1) A^* \}
\xrightarrow{=}
 (AB)^*\{AA^*\}
\nonumber\\ &
\xrightarrow{1\,\widehat\ev^{\Cc}_A}
(AB)^* 1
\xrightarrow{=}
(AB)^* ~ \Big]
\end{align}
Each of $A,B$ can be chosen from $\Cc_0$ or $\Cc_1$ so that we have to verify four identities. We will look at these one by one. For $A^0, B^0 \in \Cc_0$ we must check that
\be \label{eq:delta-monoidal-aux1}
\scanpic{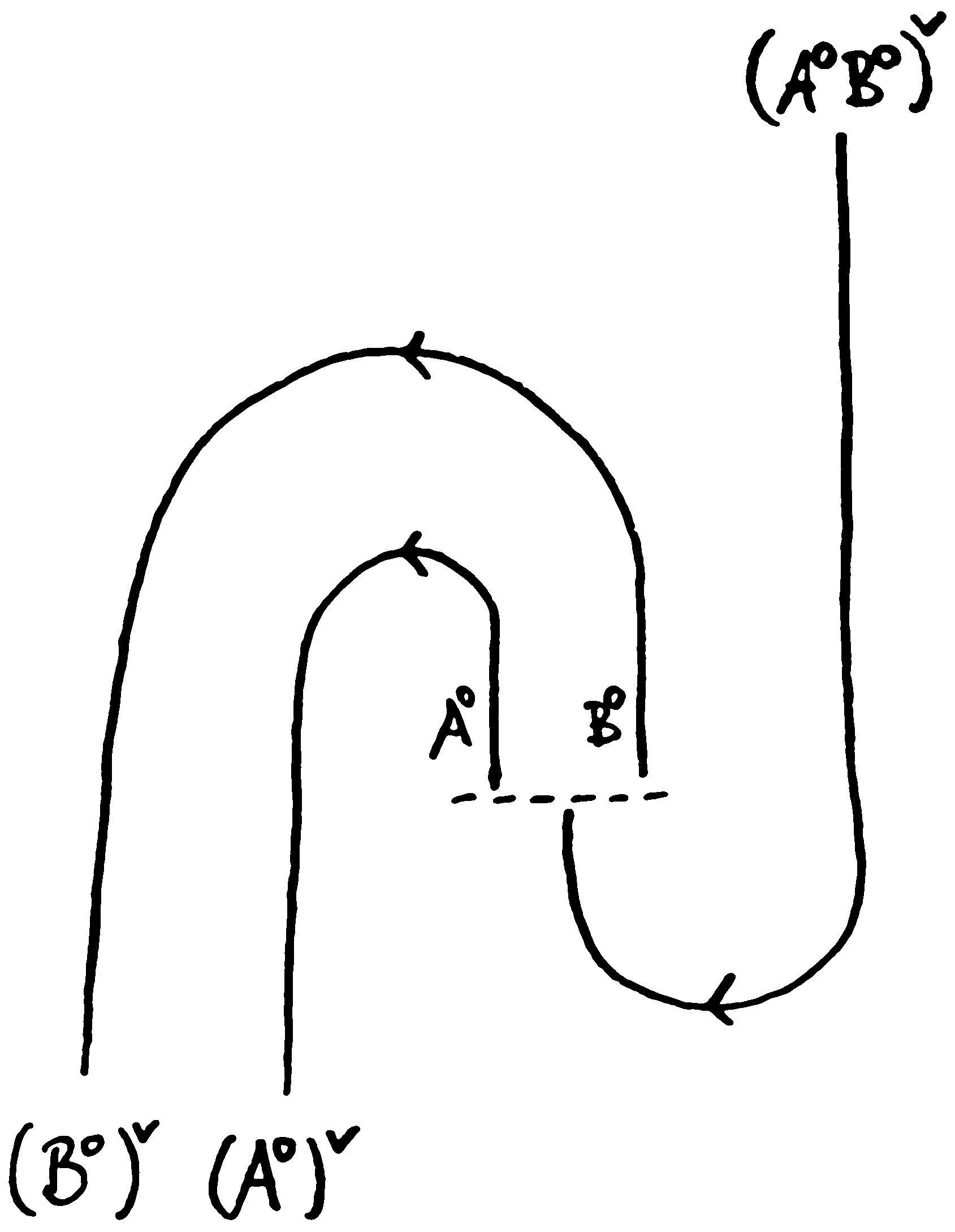} 
~=~ 
\scanpic{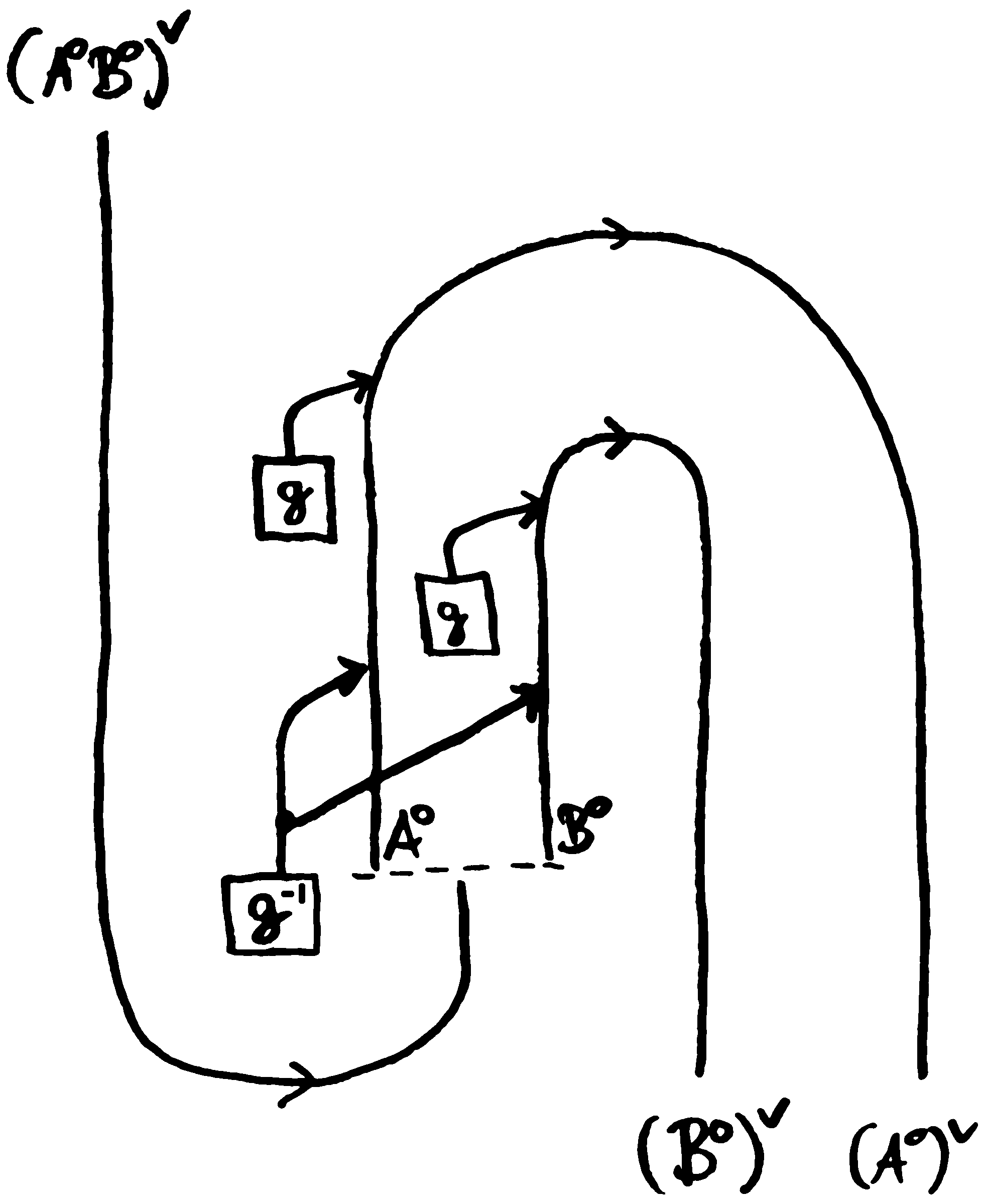} 
\quad .
\ee
The $g$-actions cancel each other because $g$ is group-like so that $\Delta \circ g^{-1} = g^{-1} \otimes g^{-1}$. After this, the identity follows since $\Sc$ is pivotal. For $A^1 \in \Cc_1$ and $B^0 \in \Cc_0$ we get
\be\label{eq:pivot-aux2}
\scanPICC{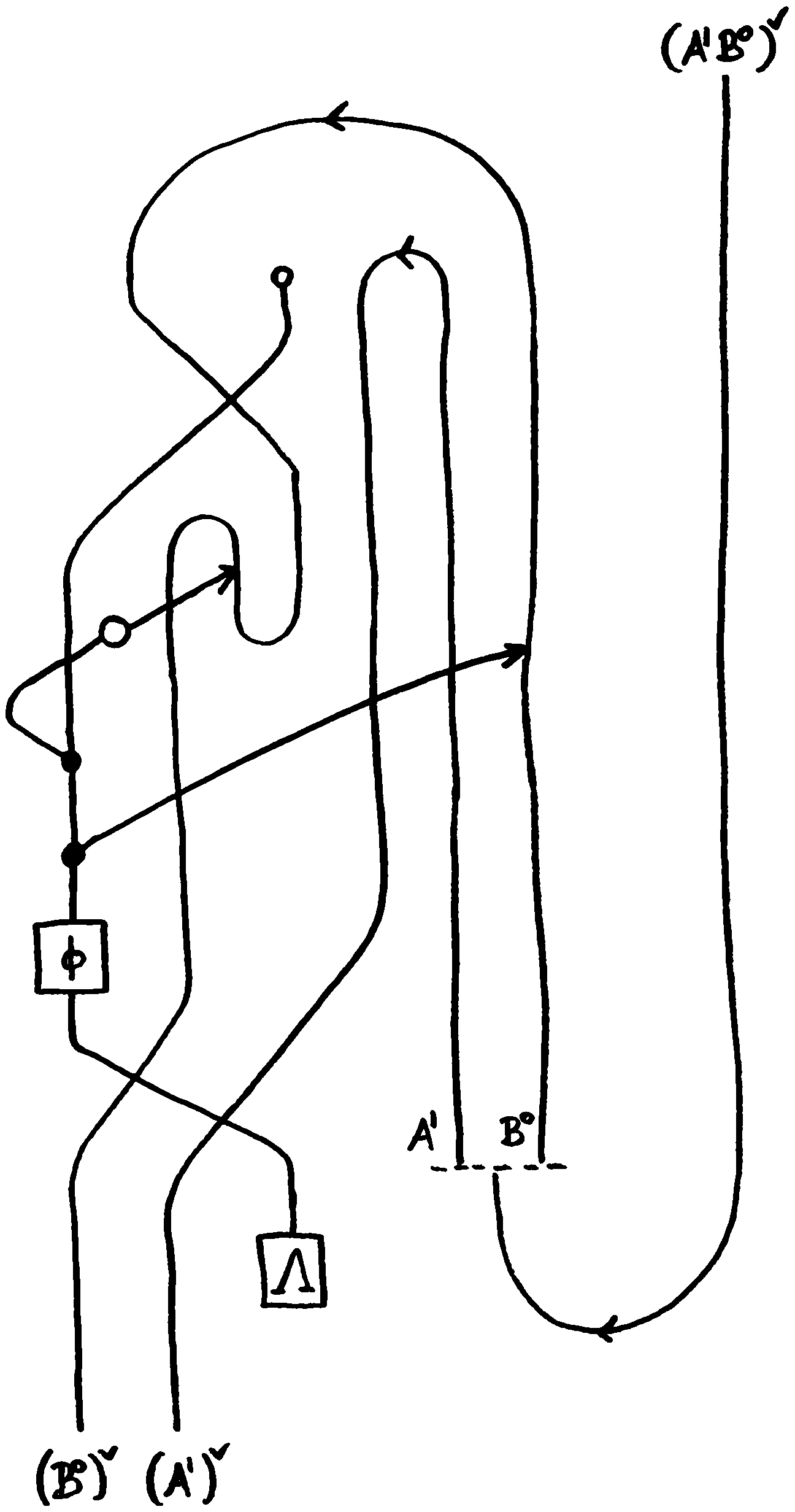} 
~=~ 
\scanPICC{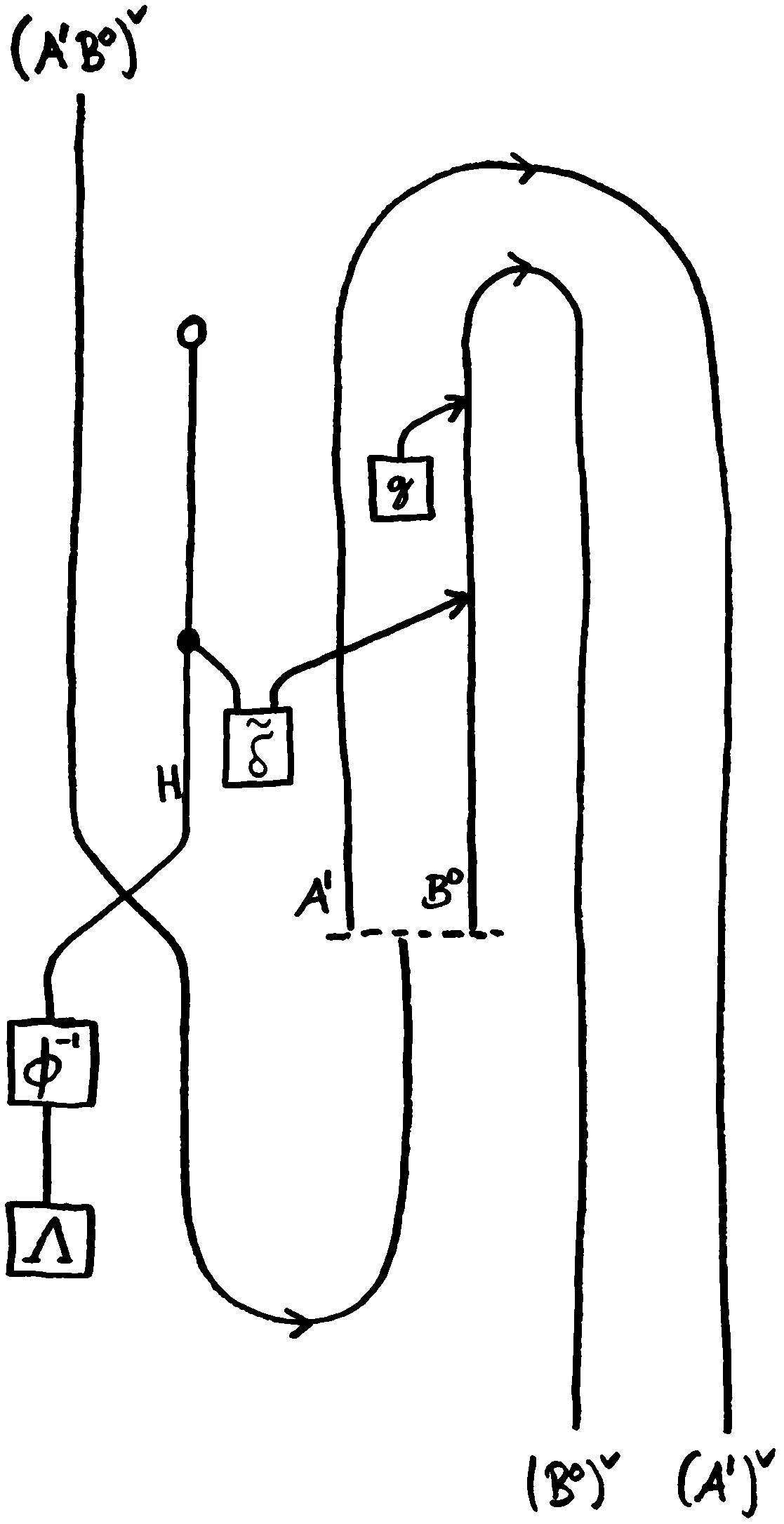} 
\quad .
\ee
On the left hand side, note that the two $H$-actions on $B^0$ combine to a loop, and one can use the bubble-property to obtain $\eps \circ \phi \circ \Lambda$, which is $\id_\one$. On the right hand side, the counit produces a term $(\eps \otimes \id_H) \circ \tilde \delta = g^{-1}$ by \eqref{eq:tilde-delta-def} which cancels with the other $g$ acting on $B^0$, and a term
\be
   \eps \circ \phi^{-1} \circ \Lambda 
   \overset{\eqref{eq:phi-1_via_phi}}= 
   \eps \circ S \circ \phi \circ {}_gM \circ \Lambda
   \overset{\text{$\Lambda$ left int.}}= 
   \lambda \circ \Lambda = \id_\one \ .
\ee
The identity \eqref{eq:pivot-aux2} now follows since $\Sc$ is pivotal. 

\begin{remark}\label{rem:delta-choice}
Let us interrupt the proof of Proposition \ref{prop:delta-pivotal} for a quick comment. 
Given the definition of the left duality maps in \eqref{eq:duality-hat-def} and our ansatz that $\delta^\Cc_X = \delta_X$ for all $X \in \Cc_1$, the equality  \eqref{eq:pivot-aux2} reduces to the requirement that $\delta_{B^0}^\Cc \circ {}_{g^{-1}}M = \id_{B^0}$. In this sense, monoidality of $\delta^\Cc$ forces the choice \eqref{eq:deltaC-def} and the assumption \eqref{eq:Adsig2=id-assumption}. 
\end{remark}

For $A^0 \in \Cc_0$ and $B^1 \in \Cc_1$ one finds
\be\label{eq:pivot-aux3}
\scanPICC{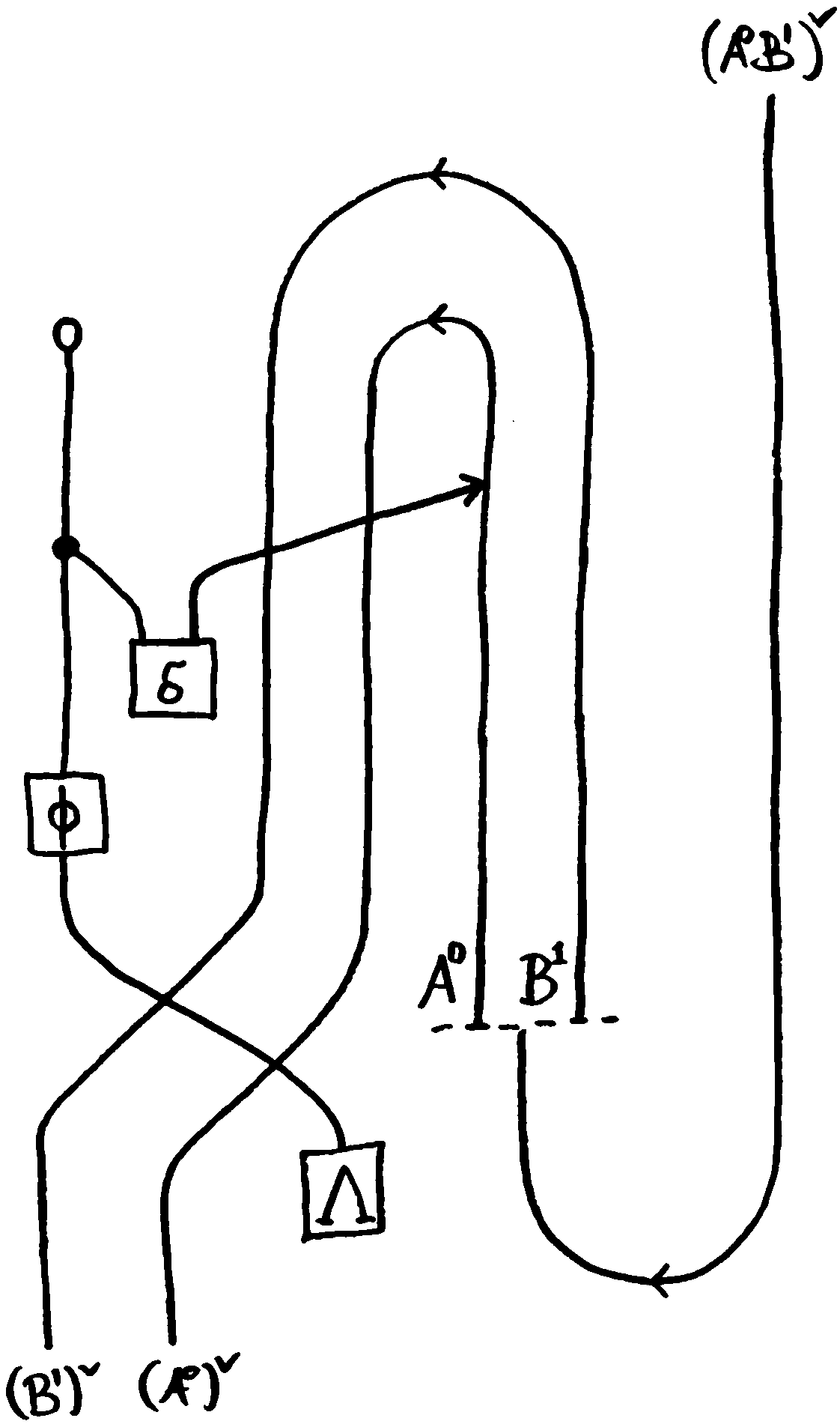} 
~=~ 
\scanPICC{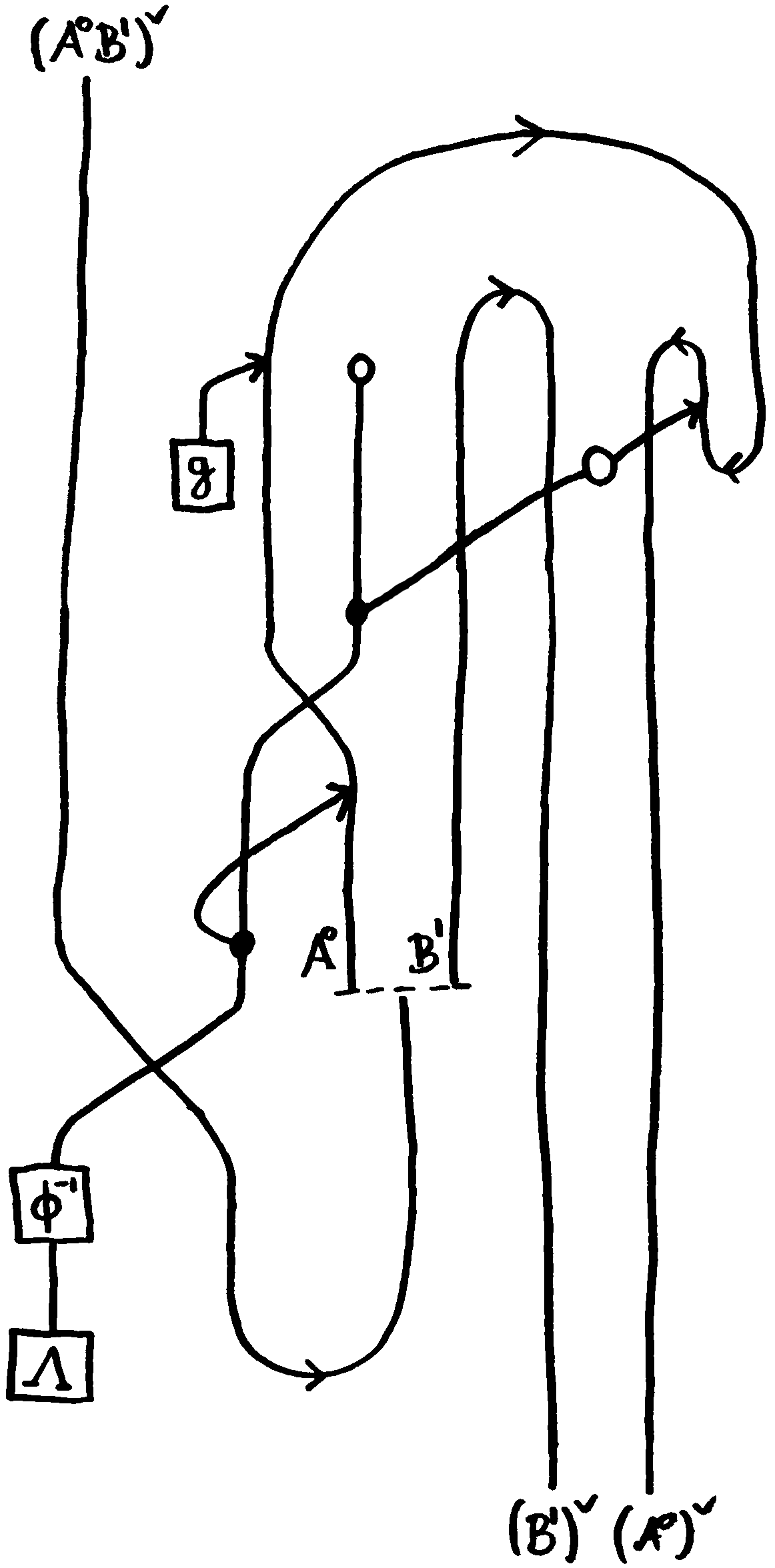} 
\quad .
\ee
On the left hand side, after using $\eps \circ \mu = \eps \otimes \eps$ one has a contribution $(\eps \otimes \id_H) \circ \delta = g$ acting on $A^0$ and a contribution $\eps \circ \phi \circ \Lambda = \id_\one$. For the right hand side first note that
\begin{align}
  \phi^{-1} \circ \Lambda 
  &\overset{\eqref{eq:phi-1_via_phi}}=
  S \circ \phi \circ {}_gM \circ \Lambda 
  \overset{\text{(*)}}=
  S^2 \circ \big( \id_H \otimes (\lambda \circ \mu) \big) \circ (\gamma \otimes \Lambda)
  \nonumber \\
   &\overset{\text{$\Lambda$ left int.}}= 
   (\lambda \circ \Lambda) \cdot (S^2 \otimes \eps)\circ \gamma
  = \eta \ ,
\end{align}
where in (*) we used that $\Lambda$ is a left integral to remove ${}_gM$, and Lemma \ref{lem:delta+phi-via-gamma} to write $\phi$ in terms of $\lambda$ and $\gamma$.
Substituting this removes all the $H$-actions from the right hand side of \eqref{eq:pivot-aux3} except for the singe action of $g$ on $A^0$, in agreement with the left hand side. Finally, for $A^1,B^1 \in \Cc_1$ we have to show
\be
\scanPICC{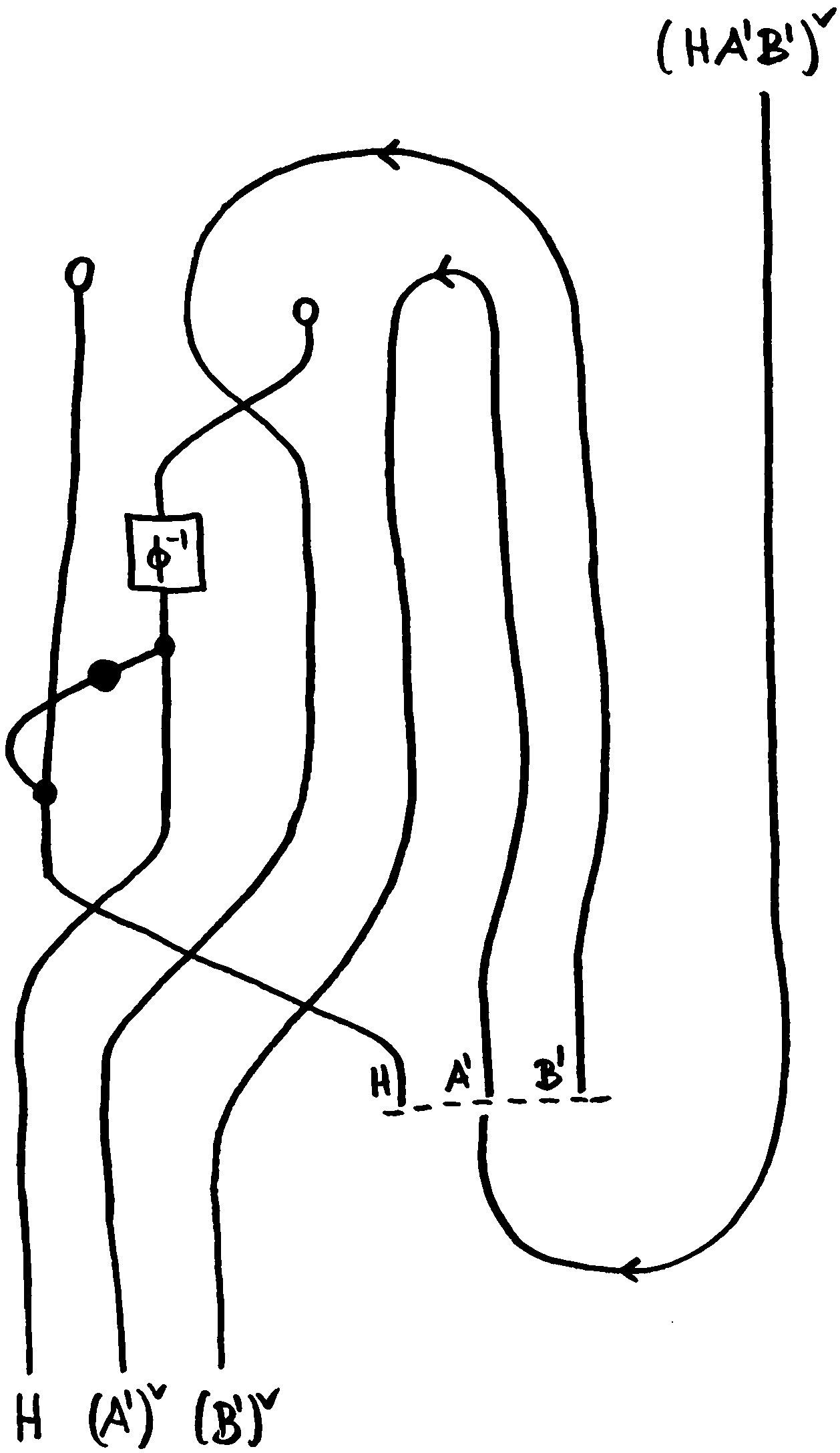} 
~=~ 
\scanPICC{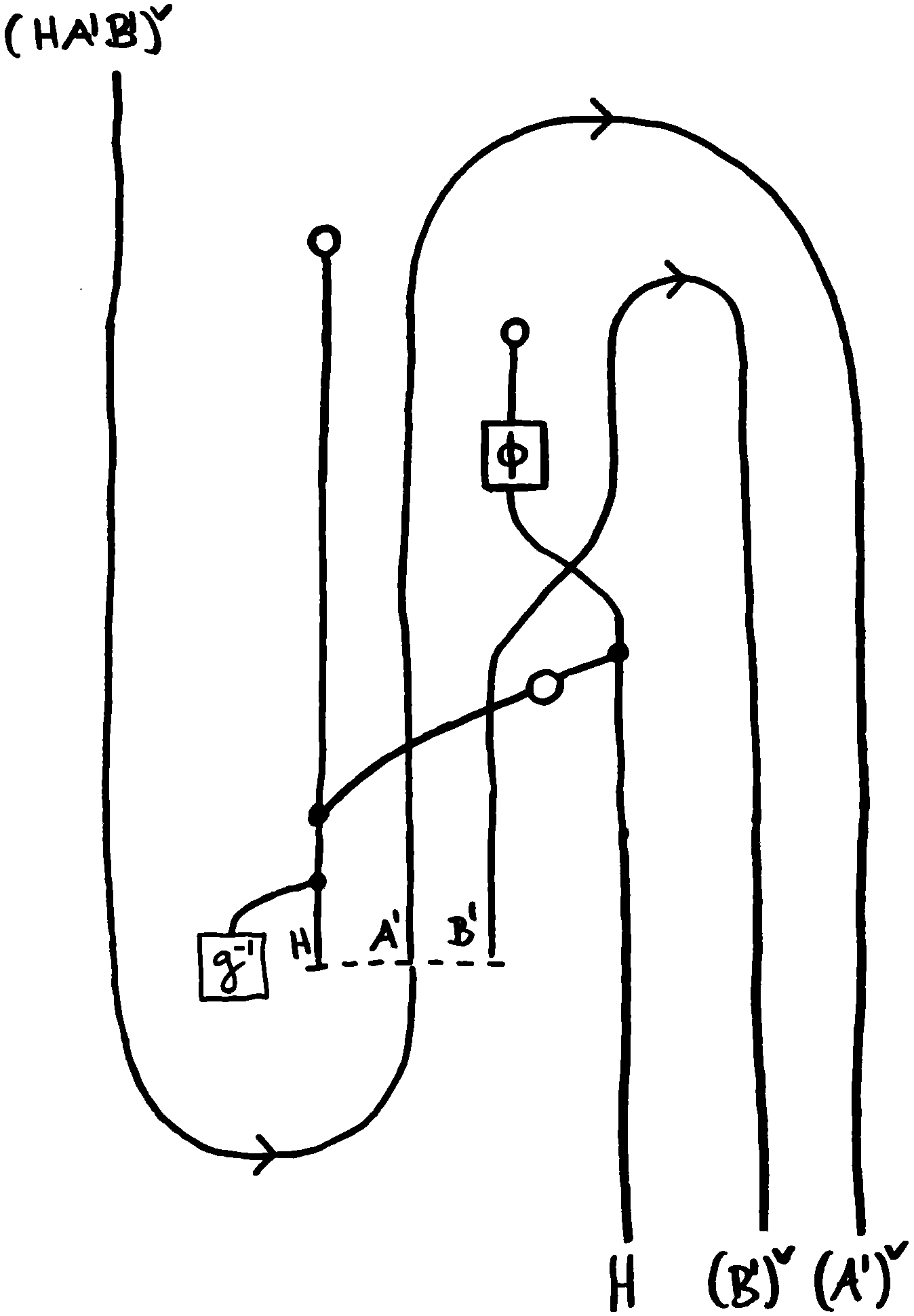} 
\quad .
\ee
Using pivotality of $\Sc$, one sees that this is equivalent to 
\be\label{eq:pivot-aux5}
\scanpic{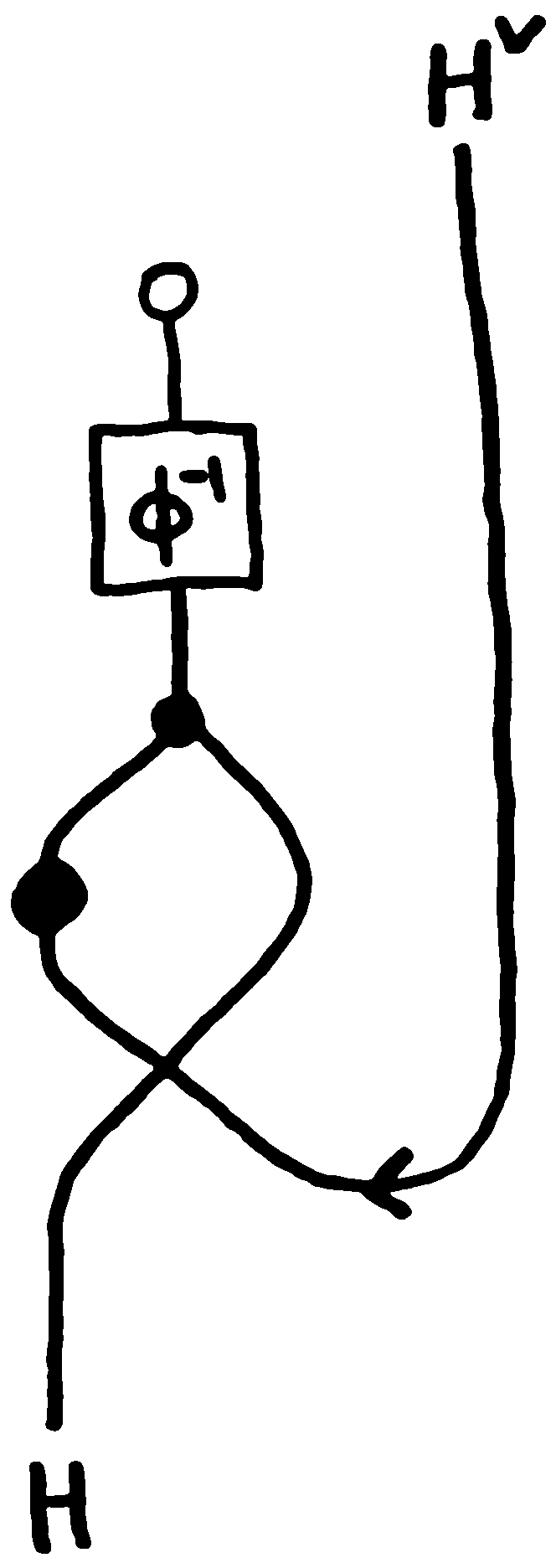} 
~=~ 
\scanpic{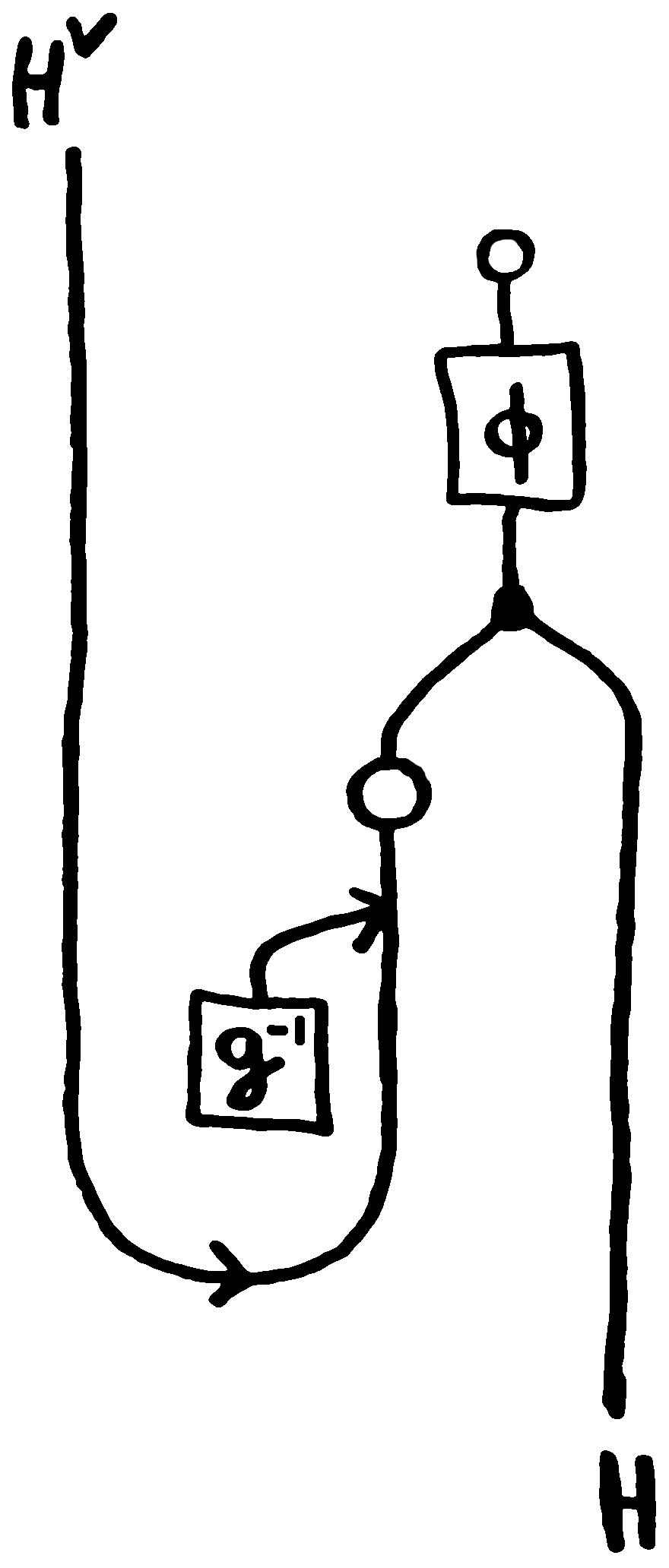} 
\quad .
\ee
Inserting both sides into $\widetilde\ev_H \circ (\id_H \otimes (-))$, i.e.\ `bending the upper $H$ leg to the left', and substituting $\phi^{-1} = S \circ \phi \circ {}_gM$ from \eqref{eq:phi-1_via_phi} shows that \eqref{eq:pivot-aux5} is equivalent to (recall that $\Sc$ has trivial twist)
\be\label{eq:pivot-aux6}
  \lambda \circ \mu \circ \big[ ({}_gM \circ S^{-1}) \otimes \id_H \big]
  =
  \lambda \circ \mu \circ \big[ ( S \circ {}_{g^{-1}}M ) \otimes \id_H \big] \ .
\ee
The above equality holds since
\be
  {}_gM \circ S^{-1} 
  \overset{\text{\eqref{eq:S2=Adg}}}= 
  {}_gM \circ \Ad_g^{-1} \circ S^2 \circ S^{-1}
  =
  M_{g} \circ S
  = 
  S \circ {}_{g^{-1}}M \ .
\ee
\end{proof}

The above proposition shows that $\Cc$ is ribbon. The twist isomorphisms $\theta_A : A \to A$ are computed from the pivotal structure (hidden in the left duality maps) as follows:
\begin{align}
\theta_A ~=~ \Big[ \, &
A 
\xrightarrow{=}
A \one
\xrightarrow{1\,\coev^{\Cc}_{A}}
A(AA^*)
\xrightarrow{\alpha^{\Cc}_{A,A,A^*}}
(AA)A^*
\xrightarrow{c^{\Cc}_{A,A}\,1}
(AA)A^*
\nonumber \\ & \hspace{10em}
\xrightarrow{(\alpha^{\Cc}_{A,A,A^*})^{-1}}
A(AA^*)
\xrightarrow{1\,\widehat\ev^{\Cc}_{A}}
A \one
\xrightarrow{=}
A 
\,\Big] \ .
\label{eq:theta-via-piv}
\end{align}

\begin{proposition}\label{prop:twist}
For $M \in \Cc_0$ and $X \in \Cc_1$ we have
$$
  \theta_M = \rho_M \circ (\sigma^{-2} \otimes \id_M) \quad , \quad
  \theta_X = \beta^{-1} \cdot \omega_X \ .
$$  
\end{proposition}

\begin{proof}
Inserting the definition of evaluation, coevaluation, associativity isomorphisms and braiding into \eqref{eq:theta-via-piv} for $A=M$ results in
\be
  \theta_M = \rho_M \circ (t \otimes \id_M) \quad \text{with} \quad
  t = \mu \circ c_{H,H} \circ ({}_gM \otimes \id_H) \circ R \ .
\ee
In $t$, replace $R$ by
\begin{align}
R 
&\overset{\text{\eqref{eq:R-via-gamma-op}}}= 
(\Ad_\sigma^{-1} \otimes \id) \circ c_{H,H} \circ \gamma^{-1}  
=
c_{H,H} \circ (\id \otimes (\Ad_\sigma^{-1} \circ S)) \circ \gamma
\nonumber\\
&\overset{\text{\eqref{eq:gamma-via-sigma_2nd}}}= 
c_{H,H} \circ (\id \otimes (\Ad_\sigma^{-1} \circ S)) \circ ({}_{\sigma^{-1}}M \otimes M_{\sigma^{-1}}) \circ \Delta \circ \sigma
\ .
\end{align}
This gives
\begin{align}
t
&=
\mu \circ \big( {}_{\sigma^{-1}}M \otimes \{ {}_gM \circ \Ad_\sigma^{-1} \circ S \circ M_{\sigma^{-1}} \} \big) \circ \Delta \circ \sigma
\nonumber\\
&=
\Ad_\sigma^{-1} \circ \mu \circ \big( \id \otimes \{ {}_gM \circ {}_{\sigma^{-1}}M \circ {}_{S \circ \sigma^{-1}}M \} \big) \circ (\id \otimes S) \circ \Delta \circ \sigma
\nonumber\\
&\overset{\text{(*)}}=
\Ad_\sigma^{-1} \circ \mu \circ ( \id \otimes M_{\sigma^{-2}} ) \circ (\id \otimes S) \circ \Delta \circ \sigma
\nonumber\\
&=
\Ad_\sigma^{-1} \circ M_{\sigma^{-2}} \circ \mu \circ (\id \otimes S) \circ \Delta \circ \sigma = \sigma^{-2} 
\ .
\end{align}
In (*) we used \eqref{eq:S-sigma-g_2nd} to write ${}_{S \circ \sigma^{-1}}M = {}_gM \circ {}_{\sigma^{-1}}M$, as well as condition d) in Theorem \ref{thm:main2} (cf.\ Lemma \ref{lem:g-sig-g=sig}) to cancel the ${}_gM$'s. The morphism in curly brackets then becomes ${}_{\sigma^{-2}}M$, and by \eqref{eq:Adsig2=id-assumption} this is equal to $M_{\sigma^{-2}}$.

Computing $\theta_X$ from \eqref{eq:theta-via-piv} produces $\theta_X = a \cdot \omega_X$, where the constant $a$ is given by
\be
  a = \eps \circ \phi^{-1} \circ M_{\nu} \circ \phi \circ \Lambda \ .
\ee
To simplify this, use the hexagon identity \eqref{eq:hexagon-8R} in the form $\phi^{-1} \circ M_{\nu} \circ \phi = {}_\sigma M \circ \phi \circ M_{\nu^{-1}}$, as well as $\eps \circ \sigma = \id_\one$ (Lemma \ref{lem:A-prop-1}) to get $a = \lambda \circ M_{\nu^{-1}} \circ \Lambda = \beta^{-1} \cdot \lambda \circ M_{\sigma} \circ \Lambda$. 
In \eqref{eq:check-8R-aux1} we saw that $\lambda \circ {}_{S \circ \sigma^{-1}}M = \lambda \circ M_{\sigma^{-1}}$. Composing this with ${}_{S \circ \sigma}M \circ M_\sigma$ gives $\lambda \circ M_{\sigma} = \lambda \circ {}_{S \circ \sigma}M$. If we apply both sides of this identity to $\Lambda$ and use the left integral property we find
\be
   \lambda \circ M_{\sigma} \circ \Lambda 
   = 
   \lambda \circ {}_{S \circ \sigma}M
   \overset{\text{$\Lambda$ left int.}}{\underset{\eps(\sigma)=\id_\one}{=}}
   \lambda \circ \Lambda 
   \overset{\text{Thm.\,\ref{thm:main1}\,(i)\,b)}}= 
   \id_\one \ . 
\ee   
Altogether, we obtain $a = \beta^{-1}$.
\end{proof}

In a ribbon category one has $c_{B,A} \circ c_{A,B} = \theta_{A \otimes B} \circ (\theta_A^{-1} \otimes \theta_B^{-1})$. Via Proposition \ref{prop:twist} we obtain the following expressions for the double braiding ($A^i, B^i \in \Cc_i$):
\begin{align}
c^\Cc_{B^0,A^0} \circ c^\Cc_{A^0,B^0} &~=~ \rho^{H \otimes H}_{A^0 \otimes B^0} \circ (Q \otimes \id_{A^0} \otimes \id_{B^0}) \ ,
\nonumber\\
c^\Cc_{B^1,A^0} \circ c^\Cc_{A^0,B^1} &~=~ \{\rho_{A^0} \circ (\sigma^2 \otimes \omega_{A^0})\} \otimes \id_{B^1} \ ,
\nonumber\\
c^\Cc_{B^0,A^1} \circ c^\Cc_{A^1,B^0} &~=~ \id_{A^1} \otimes \{\rho_{B^0} \circ (\sigma^2 \otimes \omega_{B^0})\}  \ ,
\nonumber\\
c^\Cc_{B^1,A^1} \circ c^\Cc_{A^1,B^1} &~=~ \lambda(\sigma) \cdot {}_{\sigma^2}M \otimes \omega_{A^1} \otimes  \omega_{B^1} \ ,
\label{eq:double-braiding}
\end{align}
where
\be\label{eq:Q-matrix-1}
  Q = (M_{\sigma^2} \otimes M_{\sigma^2}) \circ \Delta \circ \sigma^{-2} \ ,
\ee
and where we used $\beta^2 = \lambda(\sigma)$ (condition b) in Theorem \ref{thm:main2}). We can also compute $Q$ (which is sometimes called the {\em monodromy matrix}) by composing the braiding in \eqref{eq:braiding-ansatz} twice, which gives
\be\label{eq:Q-matrix-2}
  Q = \mu_{H \otimes H} \circ \{ (c_{H,H} \circ R) \otimes R \} \ .
\ee
Altogether, the element $t = \sigma^{-2}$ is central and satisfies
\be \label{eq:t-ribbon-el}
            \eps \circ t \overset{\text{(*)}}=  \id_\one
  \quad , \quad
  S \circ t \overset{\text{(**)}}= t
  \quad , \quad
  \Delta \circ t \overset{\text{\eqref{eq:Q-matrix-1}}}= \mu_{H \otimes H} \circ \{Q \otimes (t \otimes t) \}  \ .
\ee
Indeed, (*) follows since by Lemma \ref{lem:A-prop-1} we have $\eps(\sigma^{-1}) = \id_\one$, for (**) use $S \circ \sigma^{-1} = {}_gM \circ \sigma^{-1}$ from \eqref{eq:S-sigma-g_2nd} and that the two $g$'s cancel via condition d) of Theorem \ref{thm:main2}. A central element $t$ of $H$ which satisfies \eqref{eq:t-ribbon-el} is called {\em ribbon element}, see \cite[Def.\,XIV.6.1]{Kassel-book} (with $\theta = t^{-1}$ to match the definition of the twist isomorphisms).


\section{Examples}\label{sec:mon-ex}

\subsection{Tambara-Yamagami categories}\label{sec:TY}

Let $\Sc=\vect(k)$ be the category of finite-dimensional vector spaces over a field $k$. Let $H$ be the algebra of functions $k(A)$ on a finite group $A$ such that the characteristic of $k$ does not divide the order $|A|$. The Hopf algebra $k(A)$ is self dual if and only if the group $A$ is abelian ($k(A)$ is always cocommutative; if it is self-dual, it is also commutative). Non-degenerate Hopf copairings on $k(A)$ correspond to non-degenerate bi-characters on $A$: for a non-degenerate bi-character $\chi:A\times A\to k^\times$ the element $\gamma = \sum_{a,b\in A}\chi(a,b)p_a\otimes p_b\in k(A)\ot k(A)$ is a non-degenerate Hopf copairing. Here $p_a\in k(A)$ is the $\delta$-function supported on $a\in A$: $p_a(x) = \delta_{a,x}$ for $x\in A$.

As $A$ is abelian, any right cointegral $\lambda:k(A)\to k$ is automatically a left cointegral, i.e.\ the element $g\in k(A)$ is the identity. Any integral has the form $\lambda(p_a) = c$ for some $c\in k$. 
For a non-degenerate bi-character we have $\sum_{b\in A}\chi(a,b) = 0$ for $a \neq 0$, so that  $\sum_{a,b\in A}\chi(a,b) = |A|$.
The normalisation condition $(\lambda\ot(\lambda\circ S))(\gamma) = 1$ amounts to $c^2 = |A|^{-1}$, because
\be
(\lambda\ot (\lambda\circ S))(\gamma) = (\lambda\ot \lambda)\Big(\sum_{a,b\in A}\chi(a,b)\cdot p_a\otimes p_{b^{-1}}\Big) = c^2\sum_{a,b\in A}\chi(a,b) = c^2|A| \ .
\ee
Since the antipode for $k(A)$ is involutive, by condition c) in Corollary \ref{cor:main1-i} the copairing must be symmetric, i.e.\ it corresponds to a symmetric bi-character $\chi(a,b)=\chi(b,a)$. 
Altogether, we recover the data from \cite{Tambara:1998}.

In more detail, it was proved in \cite[Thm.\,3.2]{Tambara:1998} that associativity constraints on a $\Zb/2\Zb$-graded category $\Rep(k(A)) + \vect(k)$ with the tensor product of Table \ref{tab:tensorproducts} are in 1-1 correspondence with pairs $(\chi,c)$. Moreover the actual shape of the associativity isomorphisms in \cite[Def.\,3.1]{Tambara:1998} fits our template in  Section \ref{results}. 
The categories of the form $\Cc(k(A),\gamma,\lambda) =: \Cc(A,\chi,c)$ are known as {\em Tambara-Yamagami} categories. 
It was also proved in \cite[Thm.\,3.2]{Tambara:1998} that $\Cc(A,\chi,c)$ and $\Cc(A,\chi',c')$ are tensor equivalent iff $c=c'$ and there is an automorphism $\phi$ of $A$ such that $\chi(\phi(a),\phi(b)) = \chi'(a,b)$ for all $a,b\in A$. The `if' part of that statement agrees with Proposition \ref{prop:mon-equiv} (the `only if' part is not covered by that proposition).

\medskip

Braided structures on Tambara-Yamagami categories were described by Siehler \cite{Siehler:2001}. It was proved in \cite{Siehler:2001} that Tambara-Yamagami category $\Cc(A,\chi,c)$ affords a braiding if and only if $A$ is an elementary abelian 2-group.

Theorem \ref{thm:main2} reproduces the condition and data from \cite[Thm.\,1.2]{Siehler:2001}. 
Namely, since $k(A)$ is commutative, condition e) implies that  the antipode $S$ is the identity (Remark \ref{rem:main2}\,(iii) together with $\omega=\id$), which means that $A$ is an elementary abelian 2-group, i.e.\ $A = (\Zb/2\Zb)^n$.
The non-trivial conditions on $\sigma\in k(A)$ and $\beta\in k$ amount to (see Remark \ref{rem:main2}\,(iii) and condition a') in Remark \ref{rem:braided-via-gamma}\,(ii))
\be
  (\sigma\ot\sigma) \cdot \gamma = \Delta(\sigma) \quad,\qquad 
  \lambda(\sigma) = \beta^2 \quad,\quad 
  \varepsilon(\sigma) = 1 \quad,\quad 
  S(\sigma) = \sigma \ .
\ee
Writing $\sigma = \sum_{a\in A}\sigma(a)p_a$ we get the following conditions on $\sigma:A\to k$:
$$\sigma(a)\sigma(b)\chi(a,b) = \sigma(ab),\quad c\sum_{a\in A}\sigma(a) = \beta^2,\quad \sigma(e) = 1,\quad \sigma(a^{-1}) = \sigma(a).$$
The first and the last two conditions say that $\sigma$ is a quadratic function on $A$ associated with the symmetric bi-character $\chi$. It always exists (for an algebraically closed $k$ of characteristic zero) and two quadratic functions associated with a symmetric bi-character differ by a character on $A$. Thus the number of different pairs $(\sigma,\beta)$ associated with $(\chi,c)$ is equal to $2|\hat A|=2^{n+1}$ (here $n$ is the rank of $A$ and $\hat A$ is the group of characters).

Since the quadratic function $\sigma$ determines the bi-character we denote by $\Cc(A,\sigma,c,\beta)$ the braided tensor category structure on the Tambara-Yamagami category corresponding to $\sigma$ and $c,\beta\in k$.
According to Section \ref{sec:braided-equiv}, two such categories $\Cc(A,\sigma,c,\beta)$ and $\Cc(A,\sigma',c,\beta)$ are equivalent\footnote{
  It was also claimed in \cite[Thm.\,1.2\,(2)]{Siehler:2001} that non-equivalent braided structures on $\Cc(A,\chi,c)$ correspond to $(n{+}1)$-tuples of $\pm1$, which actually are just elements of $\hat A\times \Zb/2\Zb$. 
  However, this statement only takes into account braided auto-equivalences which act as the identity functor on the category. This does not exhaust all braided equivalences because the group of automorphisms of $A$ stabilising $\chi$ can still permute quadratic functions associated with $\chi$. 
} 
if there is an automorphism $\phi$ of $A$ such that $\sigma' = \sigma\circ\phi$. 

In particular for $A=\Zb/2\Zb$ we recover eight different structures of a braided tensor category on Ising category \cite{Etingof:2009}. Indeed there is a unique symmetric non-degenerate bicharacter $\chi$ on $\Zb/2\Zb$ and there are two choices for each of $c,\sigma,\beta$.

\subsection{Symplectic fermions}\label{sec:symp-ferm-mon}

The reason to call this class of examples `symplectic fermions' is the relation of the resulting braided monoidal category $\Cc$ to conformal field theory \cite{Runkel:2012cf}, as briefly mentioned in the introduction. In this section we will just describe the categorical structure.

\medskip

Let $\Sc$ be the category $\svect(k)$ of finite-dimensional super-vector spaces over a field $k$ of characteristic zero. Let $\mathfrak{h}$ be a non-zero finite-dimensional purely odd super-vector space, together with a non-degenerate super-symmetric pairing $(-|-)$. As $\mathfrak{h}$ is purely odd, this means that $(x|y) = - (y|x)$ for all $x,y \in \mathfrak{h}$. Non-degeneracy of $(-|-)$ then requires $\mathfrak{h}$ to be even dimensional. We set
\be
  d = \dim(\mathfrak{h}) ~\in~ 2 \Zb_{>0} \ .
\ee
If $d/2$ is odd, we in addition require that $\sqrt{-1} \in k$.

The Hopf algebra in $\Sc$ is the symmetric algebra $S(\mathfrak{h})$ of $\mathfrak{h}$ in $\svect(k)$. In other words, if we map $S(\mathfrak{h})$ to vector spaces (non-super) via the forgetful functor, it becomes the exterior algebra of $\mathfrak{h}$. Or, if we think of $\mathfrak{h}$ as an abelian Lie super-algebra then $S(\mathfrak{h})$ is the universal enveloping algebra $U(\mathfrak{h})$ of $\mathfrak{h}$. In any case, both the even and the odd component of $S(\mathfrak{h})$ have dimension $2^{d-1}$. The counit, coproduct and antipode are as for $U(\mathfrak{h})$, that is, for all $x\in \mathfrak{h}$,
\be
  \eps(x) = 0
  \quad , \quad
  \Delta(x) = x \otimes 1 + 1 \otimes x 
  \quad , \quad
  S(x) = -x \ .
\ee
Note that $S(\mathfrak{h})$ is commutative and co-commutative in $\svect(k)$, e.g.\ $x \cdot y = - y \cdot x$ for all $x,y \in \mathfrak{h}$. 
The super-vector space $S(\mathfrak{h})$ is graded by the number of tensor factors $\mathfrak{h}$ in $S(\mathfrak{h})$, and it is non-zero in grades $0,1,\dots,d$. We denote the component of grade $m$ by $S^m(\mathfrak{h})$.

Choose a basis $\{a_i\}_{i=1,\dots,d}$ of $\mathfrak{h}$ and denote by $\{b_i\}_{i=1,\dots,d}$ the basis dual to $\{a_i\}$ via $(a_i|b_j) = \delta_{i,j}$. Let
\be
  C = \sum_{i=1}^{d} b_i \otimes a_i
  ~\in~ S(\mathfrak{h}) \otimes S(\mathfrak{h})
\ee
be the copairing associated to $(-|-)$. One verifies that $C$ is independent of the choice of basis $\{a_i\}$. For later use we remark that since the basis $\{b_i\}$ is dual to $\{ - a_i \}$ in the sense that $(b_i|-a_j) = \delta_{i,j}$, basis independence implies $C = - \sum_{i=1}^{d} a_i \otimes b_i$. In other words, 
\be\label{eq:sf-C-sym}
  c_{S(\mathfrak{h}),S(\mathfrak{h})} \circ C = C \ .
\ee
Next, define
\be \label{eq:gamma-from-exp}
  \gamma ~=~ \exp(C) ~=~ 1 + C + \tfrac1{2!}  \, C \cdot C + \tfrac1{3!}  \, C \cdot C \cdot C + \cdots ~~ \in S(\mathfrak{h}) \otimes S(\mathfrak{h}) \ ,
\ee
where $1$ is the unit in $S(\mathfrak{h}) \otimes S(\mathfrak{h})$ and `$\cdot$' is the multiplication in $S(\mathfrak{h}) \otimes S(\mathfrak{h})$. The sum is finite because $C^{d+1}=0$. The $m$'th order term in \eqref{eq:gamma-from-exp} consists of summands of the form $(\text{const}) \cdot  (-1)^{m(m-1)/2} /m! \cdot (b_{i_1} \cdots b_{i_m}) \otimes (a_{i_1} \cdots a_{i_m})$, for $i_1 < \cdots <i_m$, where the sign arises from the braiding in the definition of $\mu_{S(\mathfrak{h}) \otimes S(\mathfrak{h})}$. Since any rearrangement of factors to bring the indices in the correct order leads to the same sign in the left and right tensor factor, we simply have $(\text{const})=m!$. Thus,
\be \label{eq:sy-gamma_s-in-basis}
  \gamma = 1 + \sum_{m=1}^d ~ (-1)^{m(m-1)/2} \!\!\! \sum_{i_1<\cdots<i_m} (b_{i_1} \cdots b_{i_m}) \otimes (a_{i_1} \cdots a_{i_m}) \ .
\ee
This proves in particular that $\gamma$ is non-degenerate as a copairing (since it is of the form $\gamma = \sum_j u_j \otimes v_j$ where $\{ u_j \}$ and $\{v_j\}$ are some bases of $S(\mathfrak{h})$). 
Let us also verify the remaining points in condition a) of Corollary \ref{cor:main1-i}.
That $(\id \otimes \eps) \circ \gamma = 1 = (\eps \otimes \id) \circ \gamma$ and 
$(\id \otimes S) \circ \gamma = (S \otimes \id) \circ \gamma$ is evident from \eqref{eq:sy-gamma_s-in-basis}. To check \eqref{eq:pentagon-3}, first note that
\be
(\Delta \otimes \id)(C) = \sum_{i=1}^d \big( b_i \otimes 1 \otimes a_i + 1 \otimes b_i \otimes a_i \big) = C_{13} + C_{23} \ ,
\ee
where $C_{kl}$ refers to the term in $S(\mathfrak{h}) \otimes S(\mathfrak{h}) \otimes S(\mathfrak{h})$ which has the $b_i$ in the $k$'th tensor factor and the $a_i$ in the $l$'th factor. Since $\Delta \otimes \id$ is an algebra map and since $S(\mathfrak{h}) \otimes S(\mathfrak{h}) \otimes S(\mathfrak{h})$ is commutative, we have
\begin{align}
(\Delta \otimes \id)(\exp(C)) &=  \exp( (\Delta \otimes \id)(C))
\nonumber\\
 &= \exp(C_{13} + C_{23}) = \exp( C_{13}) \cdot \exp( C_{23}) \ ,
\end{align}
where `$\cdot$' here is the product in $S(\mathfrak{h}) \otimes S(\mathfrak{h}) \otimes S(\mathfrak{h})$. This is precisely \eqref{eq:pentagon-3}. The identity  \eqref{eq:pentagon-4} is verified in the same way.

Condition c) in turn follows since $S^2=\id$, and since $\Ad_{g}^{-1} = \id$ for commutative Hopf algebras, together with \eqref{eq:sf-C-sym}.

It remains to analyse condition b) in Corollary \ref{cor:main1-i}. For this we need to understand the cointegrals of $S(\mathfrak{h})$ (since $S(\mathfrak{h})$ is cocommutative, there is no distinction between left and right cointegrals and one has $g=1$). Let 
\be
  \hat C = \mu(C) = \sum_{i=1}^d b_i a_i ~ \in ~ S^2(\mathfrak{h})  \ .
\ee
Since $(-|-)$ is non-degenerate, $\hat C^{\frac d2} \neq 0$. We define the {\em symplectic volume} $vol$ to be the element of $S(\mathfrak{h})^*$ which vanishes on $S^m(\mathfrak{h})$ for $m<d$, and which in the one-dimensional top-degree $d$ satisfies
\be \label{eq:vol_h-def}
  vol(\hat C^{\frac d2}) = 1 \ .
\ee
Since $C$ is independent of the choice of basis $\{ a_i \}$, so is $vol$. It is easy to see that $vol$ is a cointegral. Indeed, since $\Delta(a_{i_1} \cdots a_{i_m}) = (a_{i_1} \cdots a_{i_m}) \otimes 1 + \cdots$, where the omitted terms have less than $m$ factors of $a$ on the left, the defining identity $(vol \otimes \id)(\Delta(a_{i_1} \cdots a_{i_m})) = vol(a_{i_1} \cdots a_{i_m}) \, 1$ is satisfied (and is non-vanishing only for $m=d$). Conversely, it is easy to convince oneself that all cointegrals are of the form $t \cdot vol$.

For the next computation, we assume that the basis $\{ a_1,a_2,a_3,a_4,\dots,a_d \}$ we have chosen is {\em symplectic}, that is, it satisfies $(a^{2m-1}|a^{2m}) = 1$ and $(a^{2m-1}|a^{n}) = 0$ for $n \neq 2m$. The dual basis in this case is $\{ b_1,b_2,b_3,b_4,\dots,b_{d-1},b_d \}=\{ a_2,-a_1,a_4,-a_3,\dots,a_d,-a_{d-1} \}$. In terms of this basis, we have
\begin{align}
\hat C &= b_1 a_1 + b_2 a_2 + \cdots + b_{d-1} a_{d-1} + b_d a_d
\nonumber \\
&= a_2 a_1 - a_1 a_2 + \cdots + a_d a_{d-1} - a_{d-1} a_d
= \textstyle - 2 \sum_{m=1}^{d/2} a_{2m-1} a_{2m} \ .
\end{align}
From this expression it is straightforward to verify that
\be
  \hat C^{\frac d2} = (-2)^{\frac d2} \, \big(\tfrac d2\big)! ~ a_1 a_2 \cdots a_{d-1} a_d \ .
\ee
Thus, $vol(a_1\cdots a_d) = \big( (-2)^{\frac d2} \, \big(\tfrac d2\big)! \big)^{-1}$, and for $\lambda = t \cdot vol$ we obtain
\begin{align}
(\lambda \otimes \lambda)((\id \otimes S)(\gamma))
&= (-1)^{\frac d2}  \cdot
 \lambda \big(\,a_2(-a_1)\cdots a_d(-a_{d-1})\,\big) \cdot \lambda\big(\,a_1 \cdots a_d\,\big)
\nonumber \\
&=  (-1)^{\frac d2}   t^2 / \big( 2^{\frac d2} \big(\tfrac d2\big)! \big)^2 \ .
 \end{align}
We used that $(\id \otimes S)(\gamma)=\gamma^{-1} = \exp(-C)$ (see Lemma \ref{lem:gamma-inv-via-S}\,(ii)), and that since $d$ is even, the coefficient of the top degree contribution in $\exp(-C)$ is the same as in \eqref{eq:sy-gamma_s-in-basis}, namely $(-1)^{d(d-1)/2} = (-1)^{\frac d2}$.

Altogether we have now shown that for $\zeta \in \{ \pm 1\}$, the pair
\be \label{eq:sf-mon-gamma-lam}
  \gamma = \exp(C) \quad,\quad \lambda_\zeta
  = \zeta \cdot \big(\tfrac d2\big)! \, \big(2\,\sqrt{-1}\big)^{\frac d2} \cdot  vol
\ee
satisfies conditions a)--c) in Corollary \ref{cor:main1-i}. The data $\delta$ and $\phi$ in Corollary \ref{cor:main1-i} is given explicitly by
\be
  \delta = \gamma \quad , \quad
  \phi(u) = \sum_{m=0}^d ~ (-1)^{m(m+1)/2} \!\!\! \sum_{i_1<\cdots<i_m} 
   \lambda_\zeta(
  a_{i_1} \cdots a_{i_m} \cdot u) ~ b_{i_1} \cdots b_{i_m} \ .
\ee

\begin{remark}\lb{gsf}
Since all non-degenerate symplectic forms on $\h$ are isomorphic, Proposition \ref{prop:mon-equiv} implies that up to equivalence the monoidal categories $\Cc(S(\h), \gamma,\lambda)$ do not depend on the choice of the symplectic form.
\\
The group $\Aut_\mathrm{Hopf}(H,\gamma,\lambda)$ is isomorphic to $Sp(\h)$, where $Sp(\h)$ acts on $\h \subset S(\h)$. This action uniquely extends to an algebra automorphism of $S(\h)$ which is automatically a Hopf algebra automorphism and keeps $\gamma$ and $\lambda$ fixed. By Proposition \ref{gma}  we have an injective group homomorphism 
\be\lb{asf}Sp(\h)\longrightarrow \Aut_\ot\big(\Cc(S(\h), \gamma,\lambda) / \svect \big)\ee from the symplectic group into the group of isomorphism classes of tensor autoequivalences over $\svect$ of symplectic fermions.
\end{remark}

The most fundamental case is $d=2$. In terms of a basis $\{a_1,a_2\}$ of $\mathfrak{h}$ with $(a_1|a_2) = 1$ and $\zeta \in \{ \pm 1\}$ we have
\be
\begin{array}{l}
C = a_2 \otimes a_1 - a_1 \otimes a_2
\quad , \quad 
\hat C = - 2 a_1 a_2
\quad , \quad 
vol_\h(a_1 a_2) = - \tfrac12 \ ,
\\[.5em]
\gamma = \delta = 1 + a_2 \otimes a_1 - a_1 \otimes a_2 -  (a_1 a_2) \otimes (a_1 a_2) \ ,
\\[.5em]
\lambda(1)=\lambda(a_1)=\lambda(a_2)=0
\quad , \quad 
\lambda_\zeta(a_1a_2)= - \zeta\,\sqrt{-1} \ ,
\\[.5em]
\phi(1) = \zeta \, \sqrt{-1} \cdot a_1 a_2
~~,~~~
\phi(a_j) = \zeta \, \sqrt{-1} \cdot  a_j
~~,~~~
\phi(a_1 a_2) = -\zeta  \, \sqrt{-1} \cdot  \,1 \ ,
\end{array}
\ee
where $j \in \{1,2\}$. 

\medskip

Now we look at the braiding in the category of symplectic fermions.
Recall that $S(\mathfrak{h})$ is commutative and cocommutative, so that the conditions a)--e) in Theorem \ref{thm:main2}\,(i) simplify to conditions a)--e) in Remark \ref{rem:main2}\,(iii).
We need to fix a natural monoidal isomorphism of the identity functor on $\Sc = \svect(k)$. The simplest choice would be $\omega_V = \id_V$ for all $V \in \svect(k)$, but for this choice condition e), namely $\omega_H = S$, would fail in our example. Instead, we choose $\omega_V : V \to V$ to be the parity involution on the super-vector space $V$, which acts as the identity on even elements and multiplies odd elements by $-1$. Since the antipode is given by the parity involution on $S(\mathfrak{h})$, this choice satisfies $\omega_H = S$.

\medskip

For condition a) -- or rather its incarnation a') in Remark \ref{rem:braided-via-gamma}\,(ii) -- we need a suitable $\sigma \in S(\mathfrak{h})$. We claim that
\be
  \sigma = \exp\!\big( \tfrac{1}{2} \, \hat C \big)
\ee
satisfies $\gamma =  (\sigma^{-1} \otimes 1) \cdot \Delta(\sigma) \cdot (1 \otimes \sigma^{-1})$. To see this, note that (recall that $C =  \sum_{i=1}^{d} b_i \otimes a_i = - \sum_{i=1}^{d} a_i \otimes b_i$)
\begin{align}
  \Delta(\hat C) &= \sum_{i=1}^d \Delta(b_i)\Delta(a_i)
  = \sum_{i=1}^d \Big( (b_ia_i) \otimes 1 + b_i \otimes a_i - a_i \otimes b_i + 1 \otimes (b_ia_i) \Big)
\nonumber \\
  &= \hat C \otimes 1 + 2C  + 1 \otimes \hat C \ ,
\end{align}
and so
\be
  \Delta(\sigma) = \exp\!\big( \tfrac{1}{2} \Delta(\hat C) \big)
  = \exp\!\big( \tfrac{1}{2}\, \hat C \otimes 1 + C  +  \tfrac{1}{2} \, 1 \otimes \hat C \big)
  = (\sigma \otimes 1) \cdot \gamma \cdot (1 \otimes \sigma) \ .
\ee
For condition b) we first observe that indeed $\lambda_\zeta \circ S = \lambda_\zeta$ as $\lambda_\zeta$ is non-vanishing only in degree $d$, and $S$ acts as the identity on $S^d(\mathfrak{h})$. Next we need to fix an element $\beta \in k$ such that
\be \label{eq:beta2-symp-ferm}
  \beta^2 = \lambda_\zeta(\sigma) = \big(\tfrac 12\big)^{\frac d2} / \big(\tfrac d2\big)! \cdot \lambda_\zeta(\hat C^{\frac d2})
  \overset{\text{\eqref{eq:vol_h-def}}}{\underset{\text{\eqref{eq:sf-mon-gamma-lam}}}{=}} \zeta \cdot (\sqrt{-1})^{\frac d2} \ .
\ee
We will assume that $k$ contains the square root of the element on the right hand side. Condition c) is trivially true in the (co)commutative case. Condition d) holds since $\sigma = 1 + \cdots$, so that $\eps(\sigma)=1$, and since $\eps$ is parity-even, so that $S(\sigma) = \sigma$. Condition e) was already checked above.

Thus by Theorem \ref{thm:main2}\,(i), $\Cc$ is a braided category with braiding isomorphisms determined by the data
\be\lb{brsf}
  R = \gamma^{-1} = \exp(-C)
  \quad , \quad
  \tau = \sigma = \exp(\tfrac 12 \hat C)
  \quad , \quad
  \nu  = \beta\,\sigma^{-1} = \beta\,\exp(-\tfrac 12 \hat C)  \ .
\ee

\begin{remark}
In the application to conformal field theory, a piece of data easily accessible are the twist eigenvalues on irreducible representations. They are given by $e^{-2 \pi i h}$ where $h$ is the $L_0$-weight of the highest weight vector. The irreducible representations in $\Cc_0$ are the even and odd one-dimensional trivial $H$-modules, which have twist eigenvalue $1$, and the even and odd simple objects in $\Cc_1$, which have twist eigenvalue $\pm\beta^{-1}$ (Proposition \ref{prop:twist}). In the case of a single pair of symplectic fermions ($d=2$), by \eqref{eq:beta2-symp-ferm} we have 
   $\beta^2 = e^{-\pi i / 2}$ (for $\sqrt{-1}=+i$ and $\zeta = -1$), 
and the relevant square root is $\beta = e^{-2 \pi i / 8}$,
    see \cite{Runkel:2012cf} for details.
This then matches the conformal weights of the highest weight vectors of the four irreducible representations in the symplectic fermion model, which are $0,1,-\frac18,\frac38$, see \cite[Sect.\,2]{Gaberdiel:1996np} or \cite{Abe:2005}.
\end{remark}

As an application, let us find the transparent objects in $\Cc$. 

\begin{proposition}
For any choice of $\zeta,\beta$, the transparent objects in $\Cc$ all lie in $\Cc_0$ and are precisely the trivial $S(\mathfrak{h})$-representations on purely parity-even super-vector spaces. 
\end{proposition}

\begin{proof}
Let $\Pi k = k^{0|1}$ be the standard one-dimensional odd super-vector space. We turn $\Pi k$ into an object of $\Cc_0 = \Rep_\Sc(H)$ by endowing it with the trivial $H$-action (given by evaluating with $\eps$). By \eqref{eq:double-braiding} we have, for any $V \in \Cc_1$,
\be
  c^\Cc_{V,\Pi k} \circ c^\Cc_{\Pi k,V} = \eps(\sigma^2) \cdot \omega_{\Pi k} \otimes \id_V = - \id_{\Pi k} \otimes \id_V ~\neq~ \id_{\Pi k} \otimes \id_V \ .
\ee
This shows at once that $\Pi k$ is not transparent, and neither is any object in $\Cc_1$. Finally, since $Q = R_{21} \cdot R = \exp(-2C)$ is non-degenerate, the transparent objects in $\Cc_0$ alone are the trivial representations of $S(\mathfrak{h})$, that is, $\svect(k)$ itself. Since we already saw that $\Pi k$ is not transparent as an object of $\Cc$, this shows that precisely the trivial parity-even $H$-representations are transparent.
\end{proof}

In other words, the symmetric centre of $\Cc$ is $\vect$, i.e.\ the category $\Cc$ is {\em non-degenerate} (as defined, e.g., in \cite{DGNO}).
Note that the symmetric centre of $\Cc_0$ alone is equivalent to $\svect(k)$; such categories are called {\em slightly-degenerate} in \cite{Etingof:2008a,Davydov:2011a}.

\begin{remark}
In Remark \ref{gsf} we saw that the automorphisms of $S(\h)$ which fix $\gamma$ act as $Sp(\h)$ on $\h \subset S(\h)$; these automorphisms automatically fix $\lambda_\zeta$. In fact, since these automorphisms also fix $C$, they equally fix $\hat C$ and thus $\sigma$. Conversely, an automorphism that fixes $\sigma$ also fixes $\gamma$ by Remark \ref{rem:braided-via-gamma}\,(ii\,a'). Thus $\Aut_\mathrm{Hopf}(H,\lambda,\sigma) \cong Sp(\h)$ and by \eqref{bma}, the injection \eqref{asf} gets promoted to an injective homomorphism
\be
  Sp(\h) ~ \lhook\joinrel\relbar\joinrel\rightarrow ~ \Aut_{\mathrm{br}}\big( \Cc(S(\h),\lambda,\sigma,\beta)/\svect \big)
\ee
into the group of isomorphism classes of braided tensor autoequivalences over $\svect$ of symplectic fermions.
\end{remark}

\subsection{Sweedler's 4-dimensional Hopf algebra}\label{sec:Sweedler-mon}

In this example, $k$ is a field of characteristic $\neq 2$ and $\Sc$ is the category of finite-dimensional $k$-vector spaces. We use the conventions in \cite[Ch.\,VIII.2]{Kassel-book}. Sweedler's four-dimensional Hopf algebra $H_4$ has basis $\{ 1,g,x,gx \}$ with multiplication determined by $g^2=1$, $x^2=0$, $xg = - gx$. The coalgebra structure is determined by $\eps(g)=1$, $\eps(x)=0$ and
\be
  \Delta(g) = g \otimes g 
  \quad , \qquad 
  \Delta(x) = 1 \otimes x + x \otimes g
  \quad , \qquad 
  \Delta(gx) = g \otimes gx + gx \otimes 1 \ .
\ee
The antipode is $S(g)=g$, $S(x)=gx$ and $S(gx) = -x$ and has order 4. We also remark that $S^2 = \Ad_g$.

From the form of the coproduct it is immediate that all solutions to $(\lambda \otimes \id)(\Delta(a)) = \lambda(a) \cdot 1$, i.e.\ all right cointegrals, are of the form
\be
  \lambda_t(1) = \lambda_t(g) = \lambda_t(x) = 0
  \quad , \quad
  \lambda_t(gx) = t \ ,
\ee
for some $t \in k$. The $\lambda_t$ satisfy $(\id \otimes \lambda_t)(\Delta(a)) = \lambda_t(a) \cdot g$. A family of solutions to \eqref{eq:pentagon-3} and \eqref{eq:pentagon-4} is
\be
  \gamma_s = 
  \tfrac12\big( 
  1 \otimes 1 + 1 \otimes g + g \otimes 1 - g \otimes g 
  \big)
  +
  \tfrac{s}2\big( 
  x \otimes x + x \otimes gx + gx \otimes x - gx \otimes gx 
  \big) 
\ee  
for $s \in k$. Non-degeneracy of $\gamma_s$ requires $s \neq 0$, and in this case $\gamma_s$ satisfies condition a) in Corollary \ref{cor:main1-i}. For condition b) it remains to check the normalisation, which results in
\be
  (\lambda_t \otimes \lambda_t) \circ (\id \otimes S) \circ \gamma_s 
  = \tfrac12 t^2 s 
  \quad \Rightarrow \quad
  t^2 s = 2 \ .
\ee
Finally, since $S^2 = \Ad_g$, condition c) demands that $\gamma$ be symmetric, which it is. Thus for all $t \in k^\times$, the pair
\be \label{eq:ex-H4-gamma-lambda}
  \gamma_{2/t^2} ~,~ \lambda_{t}
\ee
satisfies the conditions in Corollary \ref{cor:main1-i}. The data $\delta$ and $\phi$ are given by
\be
  \delta = 
  \tfrac12\big( 
  1 \otimes g + 1 \otimes 1 + g \otimes g - g \otimes 1 
  \big)
  +
  t^{-2} \big( 
  x \otimes gx + x \otimes x + gx \otimes gx - gx \otimes x 
  \big) 
\ee
and
\be
  \phi(1) = t^{-1} (1+g)x
  ~,~~
  \phi(g) = t^{-1} (1-g)x
  ~,~~
  \phi(x) = \tfrac12 t (1-g)
  ~,~~
  \phi(gx) = \tfrac12 t (1+g) \ .
\ee

Note that the Sweedler Hopf algebra has a one parameter family of Hopf algebra automorphisms 
\be
  f_{c}:H_4\to H_4 ~~,\quad f_{c}(g) = g ~~,\quad f_c(x) = cx,\quad c\in k^\times \ ,
\ee
which relates pairs $( \gamma_{2/t^2} \,,\, \lambda_{t})$ for different $t$. Thus by Proposition \ref{prop:mon-equiv}, the categories $\Cc(H_4, \gamma_{2/t^2} , \lambda_{t})$ for different $t \in k^\times$ are all monoidally equivalent.

\medskip

We now discuss the braiding on $\Cc = \Rep_\Sc(H) + \Sc$, or, rather, its non-existence.
It is known that Sweedler's 4-dimensional Hopf algebra $H_4$  is quasi-triangular with a 1-parameter familiy of possible $R$-matrices (see, e.g., \cite[Ch.\,4.2.F]{CP-book}),
\be\label{eq:sweedler-R}
  R_s = \tfrac12\big( 
  1 \otimes 1 + 1 \otimes g + g \otimes 1 - g \otimes g 
  \big)
  +
  \tfrac{s}2\big( 
  x \otimes x + x \otimes gx + gx \otimes gx - gx \otimes x 
  \big) \ ,
\ee
where $s \in k$. Since the existence of an $R$-matrix is a necessary condition to have a braiding on all of $\Cc$, let us try to find the remaining data. For simplicity we will work over the field $k = \Cb$.

Recall the solution $\gamma_{2/t^2},\lambda_t$ from \eqref{eq:ex-H4-gamma-lambda}. First we need a suitable invertible element $\sigma \in H_4$, which satisfies condition a'). After a short calculation one finds that there are precisely four such elements. Namely, choose a $4^\text{th}$ root of unity $\zeta$, set $a = \zeta^2$, $b = \sqrt{2} \, e^{-\pi i /4} \zeta / t$, and define
\be
  \sigma = \tfrac12(1+ai)\big\{1-ai\cdot g+b(x-ai\cdot gx)\big\}  \ .
\ee
Then $\sigma^{-1} = \tfrac12(1-ai)\big\{1+ai\cdot g-b(x-ai\cdot gx)\big\}$ and one verifies that $\gamma_{2/t^2} = ({}_{\sigma^{-1}}M \otimes M_{\sigma^{-1}}) \circ \Delta \circ \sigma$ as required. 

We now found all solutions to a'). 
From Lemma \ref{lem:S2=GA-2}\,(ii) we know that $S^2 = \Ad_g \circ \Ad_\sigma^{-2}$. For $H_4$ we have in addition that $S^2 = \Ad_g$, so that any $\sigma$ that satisfies the conditions a)--e) in Theorem \ref{thm:main2} must satisfy $\Ad_\sigma^{\,2} = \id$. Unfortunately, $\Ad_\sigma^{\,2}(x) = -x$, independent of $\zeta$.

One may ask if there is a different choice of $\gamma$ for which the above procedure works. This turns out to be not so. One way to see this is to use the fact that any $R$-matrix must be of the form \eqref{eq:sweedler-R}. It is then possible to verify that for no choice of $s \in \Cb^\times$ there is a $\sigma$ which solves \eqref{eq:R-etc-via-Hopf}.

\subsection{A 16-dimensional non-(co)commutative triangular Hopf algebra}\label{sec:H16-non-comm}

This example is included solely to make the point that the conditions in Theorem \ref{thm:main2}\,(i) do allow for a non-(co)commutative solution. It suffers from being a bit on the technical side, but we were happy to find an example in the first place. The starting point is the triangular Hopf algebra given in \cite[Ex.\,4.6.6]{Gelaki:2002} (the Hopf algebra itself appeared already in \cite{Kashina:2000}, the triangular structure we use is that from \cite{Gelaki:2002}, the expression of $R$ in terms of $\sigma$ given below is new).

We will construct this example `backwards', namely, we will start with the known $R$-matrix, guess a $\sigma$ which satisfies \eqref{eq:R-etc-via-Hopf} and then recover $\gamma$.

\medskip

In this example the symmetric category is $\Sc = \vect(\Cb)$ and so $\omega = \id$. Following \cite[Ex.\,4.6.6]{Gelaki:2002}, let $G = \Zb/2 \times \Zb/2$ and $A = \Zb/4$. The group $G$ acts on $u\in A$ via $(m,n).u = (-1)^n u$. Set $\tilde G = G \ltimes A$, so that the product and inverse in $\tilde G$ are (we will write the group operations in $\tilde G$ multiplicatively)
\be\begin{array}{l}
  (m,n;u) \cdot (m',n';u') = (m+m',n+n';u+(-1)^n u') \ ,
\\[.5em]
  (m,n;u)^{-1} = (-m,-n;(-1)^{n+1} u)\ .
\end{array}
\ee
Note that $\tilde G$ is not commutative (in fact, $\tilde G \cong C_2 \times D_8$).
As an algebra, the Hopf algebra $H_{16}$ is just the corresponding group algebra:
\be
  H_{16} = \Cb[\tilde G]
  \qquad \text{(as an algebra)} \ .
\ee
The coalgebra structure will be obtained via a twist.
Define $\pi : G \to A$ via $\pi(0,0)=0$, $\pi(1,0) = 2$, $\pi(0,1)=1$, $\pi(1,1)=3$ (this is a bijective 1-cocycle in the sense of \cite[Def.\,4.6.1]{Gelaki:2002}). Pick a group isomorphism from $A$ to its character group, $\psi : A \to A^*$, for example the one defined on the generator $1 \in A$ by
\be
  \psi_1(x) = \delta_{x,0} + i \, \delta_{x,1} - \delta_{x,2} - i \, \delta_{x,3} \ ,
\ee
so that $\psi_y(x) = (\psi_1(x))^y$ for $x,y \in A$. Define
\be
  \bar J = \frac{1}{|A|} \sum_{x,y \in A}
    \psi_y(x) \cdot \pi^{-1}(x) \otimes y
    \quad \in H_{16} \otimes H_{16} \ .
\ee
The first factor in the tensor product lies in $G \subset H_{16}$ and the second one in $A \subset H_{16}$. We will also need the multiplicative inverse of $\bar J$. Let $T : A \to A$ be defined by $\pi^{-1}(-x) + \pi^{-1}(T(x)) = 0$; since $G$ has order 2, this works out to be $T(u)=-u$. Then \cite[Prop.\,4.6.3]{Gelaki:2002}
\be
  \bar J^{-1} = \frac{1}{|A|} \sum_{x,y \in A}
    (\psi_y(x))^{-1} \cdot \pi^{-1}(T(x)) \otimes y \ .
\ee
By \cite[Prop.\,4.6.3]{Gelaki:2002}, $\bar J$ is a twist, and via the conventions in \cite[Sect.\,2.3]{Gelaki:2002} we obtain the coproduct and antipode on the basis $x \in \tilde G$ of $H_{16}$
\be
\Delta(x) = \bar J^{-1} \cdot (x \otimes x) \cdot \bar J
\quad , \quad
S(x) = \bar Q^{-1} \cdot x^{-1} \cdot \bar Q \ ,
\ee
where $\bar Q = \mu \circ (S_\text{untwisted} \otimes \id) \circ \bar J$, and $x^{-1}$ is the inverse group element in $\tilde G$ (so that $S_\text{untwisted}(x) = x^{-1}$). The counit is $\eps(x) = 1$ for all $x \in \tilde G$. Most importantly, we obtain the $R$-matrix
\be
  R = \bar J_{21}^{-1} \cdot \bar J 
  = \frac{1}{|A|^2} \sum_{x,y,x',y' \in A} \frac{\psi_{y'}(x')}{\psi_{y}(x)} \cdot (y \,\pi^{-1}(x')) \otimes (\pi^{-1}(T(x))y') \ .
\ee
The above $R$-matrix is triangular and non-degenerate \cite{Gelaki:2002}. This completes the definition of $H_{16}$ as a triangular Hopf algebra. It is semisimple and co-semisimple (since $S^2=\id$, see \cite[Thm.\,4.4]{Larson:1988}), and neither commutative nor cocommutative.

\medskip

Next we need to understand the right cointegrals for $H_{16}$. This space is one-dimensional (see \cite[Prop.\,1.1]{Larson:1988}) and is spanned by $\lambda_1 \in H_{16}^*$ with
\be
  \lambda_1(x) = \delta_{x,1}
  \qquad (x \in \tilde G) \ .
\ee
From this one checks that $H_{16}$ is unimodular (this is also clear as $H_{16}$ is semisimple), i.e.\ $\lambda_1$ is also a left-integral. In other words, the element $g \in H_{16}$ which enters Theorems \ref{thm:main1} and \ref{thm:main2} is the unit: $g = 1 \in H_{16}$.

\medskip

Now consider the element
\be
\raisebox{3.5em}{$\displaystyle
  \sigma 
  = \frac14 \big(
  \raisebox{-2em}{$\begin{array}{r@{}lr@{}lr@{}lr@{}l}
  1 \,\cdot &(000)& +~ i \,\cdot &(001)& -~1 \,\cdot &(002)& -~i \,\cdot &(003)\\ 
    +~ 1 \,\cdot &(010)& +~ 1 \,\cdot &(011)& + ~1\,\cdot &(012)& +~ 1 \,\cdot &(013)\\
    +~ 1 \,\cdot &(100)& -~i \,\cdot &(101)& -~1 \,\cdot &(102)& +~ i \,\cdot &(103)\\ 
    +~ 1 \,\cdot &(110)& -~1 \,\cdot &(111)& +~ 1\,\cdot &(112)& -~1 \,\cdot &(113) 
    \end{array}$}
    \raisebox{-3.8em}{\big)\ ,}  $}
\ee
where $(mnu)$ denotes an element of $\tilde G = (\Zb/2 \times \Zb/2) \ltimes \Zb/4$.
The element $\sigma$ satisfies
\be
  \sigma^2 = 1
  \qquad \text{and} \qquad
  R = (\sigma \otimes \sigma) \cdot \Delta(\sigma^{-1}) \ ,
\ee
so that \eqref{eq:R-etc-via-Hopf} holds.\footnote{
  We found this $\sigma$ by trail and error with the help of Mathematica.
  The $\sigma$ we give is not the only possible choice, but for the purpose here one such $\sigma$
  is enough. It would be nice to have a more systematic approach to
  finding examples.}
It is now a matter of patience to verify that for $\zeta \in \{ \pm 1 \}$ the pair
\be
  \gamma = (\sigma^{-1} \otimes 1) \cdot \Delta(\sigma) \cdot (1 \otimes \sigma^{-1})
  \quad , \quad
  \lambda = 4\zeta \cdot \lambda_1
\ee
satisfies the conditions in Corollary \ref{cor:main1-i}. One also checks that the conditions in Theorem \ref{thm:main2}\,(i) are satisfied (choose any $\beta$ with $\beta^2 = \lambda(\sigma) = \zeta$), so that $\Cc$ becomes a braided monoidal category. In fact, since $\sigma^2=1$, by Section \ref{sec:ribbon} $\Cc$ is even ribbon.

\medskip

Finally, let us find the transparent objects in $\Cc$. Firstly, $\Cc_0$ itself is symmetric (by construction, since $R$ is a triangular $R$-matrix). Next, since $\sigma^2 = 1$ and $\omega = \id$, by \eqref{eq:double-braiding} the mixed double braidings are all equal to the identity. The double braiding of two objects from $\Cc_1$ is equal to $\lambda(\sigma) = \zeta$ times the identity. Thus for $\zeta = 1$, $\Cc$ is again symmetric, while for $\zeta = -1$, the transparent objects in $\Cc$ are given by $\Cc_0$.

\newcommand\arxiv[2]      {\href{http://arXiv.org/abs/#1}{#2}}
\newcommand\doi[2]        {\href{http://dx.doi.org/#1}{#2}}
\newcommand\httpurl[2]    {\href{http://#1}{#2}}


\begin{thebibliography}{YYY}
\setlength{\itemsep}{-0em}
\small

\bibitem[Ab]{Abe:2005}
T.~Abe,
{\it A $\Zb_2$-orbifold model of the symplectic fermionic vertex operator superalgebra},
\doi{10.1007/s00209-006-0048-5}{Mathematische Zeitschrift {\bf 255}  (2007) 755--792}
\arxiv{math/0503472}{[math.QA/0503472]}.

\bibitem[Be]{Bespalov:1995}
Yu.N.~Bespalov,
{\it Crossed modules and quantum groups in braided categories},
\doi{10.1023/A:1008674524341}{Appl.\ Cat.\ Struct.\ {\bf 5} (1997) 155--204}
\arxiv{q-alg/9510013}{[q-alg/9510013]}.

\bibitem[Br]{Bruguieres:2000}
A.~Brugui\`eres,
  {\it Cat\'egories pr\'emodulaires, modularisations et invariants des vari\'et\'es de dimension 3},
  \doi{10.1007/s002080050011}{Math.\ Annal.\ {\bf 316} (2000) 215--236}. 

\bibitem[CP]{CP-book}
V.~Chari and A.~Pressley, {\em A guide to quantum groups}, Cambridge University Press, 1994.

\bibitem[CR]{Carqueville:2010hu}
  N.~Carqueville and I.~Runkel,
  {\it Rigidity and defect actions in Landau-Ginzburg models},
  \doi{10.1007/s00220-011-1403-x}{Comm.\ Math.\ Phys.\ {\bf 310} (2012) 135--179}
  \arxiv{1006.5609}{[1006.5609 [hep-th]]}.

\bibitem[DR]{DR-prep}
A.~Davydov and I.~Runkel, in preparation.

\bibitem[DNO]{Davydov:2011a}  
  A.~Davydov, D.~Nikshych and V.~Ostrik,
  {\it On the structure of the Witt group of braided fusion categories},
  \doi{10.1007/s00029-012-0093-3}{Selecta Mathematica, New Series, online (2012)}
  \arxiv{1109.5558}{[1109.5558 [math.QA]]}.

\bibitem [DGNO]{DGNO}  
  V.~Drinfeld, S.~Gelaki, D.~Nikshych and V.~Ostrik.
  {\em On braided fusion categories I}, 
  \doi{10.1007/s00029-010-0017-z}{Selecta Mathematica, New Series \textbf{16} (2010) 1--119} 
  \arxiv{0906.0620}{[0906.0620 [math.QA]]}.

\bibitem [EM]{Eilenberg:1954} S.~Eilenberg and S.~Mac Lane,
\textit{On the groups $H(\Pi,n)$, I, II},
\httpurl{www.jstor.org/stable/1969820}{Ann.\ Math.\ \textbf{58} (1953) 55--106}; 
\httpurl{www.jstor.org/stable/1969702}{\textbf{60} (1954) 49--137}.

\bibitem[ENO1]{Etingof:2008a}  
  P.I.~Etingof, D.~Nikshych and V.~Ostrik.
  {\it Weakly group-theoretical and solvable fusion categories},
  \doi{10.1016/j.aim.2010.06.009}{Adv.\ Math.\ \textbf{226} (2011) 176--205}
  \arxiv{0809.3031}{[0809.3031 [math.QA]]}.

\bibitem[ENO2]{Etingof:2009}
P.I.~Etingof, D.~Nikshych and V.~Ostrik, 
{\it Fusion categories and homotopy theory}
\doi{10.4171/QT/6}{Quantum Topology \textbf{1} (2010) 209--273}
\arxiv{0909.3140}{[0909.3140 [math.QA]]}.

\bibitem[Ge]{Gelaki:2002}
S.~Gelaki, 
{\em On the Classification of Finite-Dimensional Triangular Hopf Algebras}, in 
\doi{10.2277/0521815126}{{\em New Directions in Hopf Algebras}, 
MSRI Publ.\ {\bf 43}, 69--116, S.~Montgomery and H.-J.~Schneider (eds.), Cambridge University Press, 2002}.

\bibitem[GK]{Gaberdiel:1996np}
  M.R.~Gaberdiel and H.G.~Kausch,
  {\it A rational logarithmic conformal field theory},
  \doi{10.1016/0370-2693(96)00949-5}{Phys.\ Lett.\  B {\bf 386} (1996) 131--137}
  \arxiv{hep-th/9606050}{[hep-th/9606050]}.

\bibitem[Iz1]{Izumi:1993}
M.~Izumi,
{\it Subalgebras of infinite $C^*$-algebras with finite Watatani indices, I: Cuntz algebras},
\doi{10.1007/BF02100056}{Commun.\ Math.\ Phys.\ {\bf 155} (1993) 157--182}. 

\bibitem[Iz2]{Izumi:1998}
M.~Izumi,
{\it Subalgebras of infinite $C^*$-algebras with finite Watatani indices, II: Cuntz-Krieger algebras},
\doi{10.1215/S0012-7094-98-09118-9}{Duke Math.\ J.\ {\bf 91} (1998) 409--461}. 

\bibitem[JS]{Joyal:1993}
A.~Joyal and R.~Street, 
{\it Braided tensor categories}, 
\doi{doi:10.1006/aima.1993.1055}{Adv.\ Math.\ {\bf 102} (1993) 20--78}.

\bibitem[Kac]{Kac:1998}
V.~Kac,
{\it Vertex algebras for beginners}, 2nd ed., AMS, 1998. 

\bibitem[Kau]{Kausch:1995py}
  H.G.~Kausch,
  {\it Curiosities at c = -2},
  \arxiv{hep-th/9510149}{hep-th/9510149}.

\bibitem[KR]{Kirillov:1990}
A.N.~Kirillov and N.~Reshetikhin,
{\it q-Weyl group and a multiplicative formula for universal R-matrices},
\doi{10.1007/BF02097710}{Commun.\ Math.\ Phys.\ {\bf 134} (1990) 421--431}.

\bibitem[Ksh]{Kashina:2000}
Y.~Kashina, {\it Classification of semisimple Hopf algebras of dimension 16},
\doi{10.1006/jabr.2000.8409}{J.\ Alg.\ {\bf 232} (2000) 617--663}
\arxiv{math/0004114}{[math.QA/0004114]}.

\bibitem[Ksl]{Kassel-book}
C.~Kassel, {\em Quantum groups}, Springer, 1995.

\bibitem[LR]{Larson:1988}
R.G.~Larson and D.E.~Radford, 
{\it Finite dimensional cosemisimple Hopf algebras in characteristic 0 are semisimple},
\doi{10.1016/0021-8693(88)90107-X}{J.\ Alg.\ {\bf 117} (1988) 267--289}.

\bibitem[LS]{Levendorskii:1991}
S.Z.~Levendorskii and Ya.S.~Soibel'man,
{\it Quantum Weyl group and multiplicative formula for the R-matrix of a simple Lie algebra},
\doi{10.1007/BF01079599}{Funct.\ Ana.\ and Appl.\ {\bf 25} (1991) 143--145}.

\bibitem[Ma]{Majid:1995}
S.~Majid, 
{\it Algebras and Hopf algebras in braided categories},
Lecture Notes in Pure and Applied Math.\ {\bf 158} (1994) 55--105
\arxiv{q-alg/9509023}{[q-alg/9509023]}.

\bibitem[Mu]{Muger:1998a}
M.~M\"uger, 
{\it Galois theory for braided tensor categories and the modular closure}, 
\doi{10.1006/aima.1999.1860}{Adv.\ Math.\ {\bf 150} (2000) 151--201}
\arxiv{math/9812040}{[math.CT/9812040]}.

\bibitem[Os]{Ostrik:2002}
V. Ostrik, {\it Module categories over the Drinfeld double of a finite group}, 
\doi{doi:10.1155/S1073792803205079}{Int.\ Math.\ Res.\ Not.\ {\bf 27} (2003) 1507--1520}
\arxiv{math/0202130}{[math.QA/0202130]}.
 
\bibitem[Ru]{Runkel:2012cf}
I.~Runkel, {\em A braided monoidal category for free superbosons}, 
\arxiv{1209.5554}{1209.5554 [math.QA]}.

\bibitem[Si]{Siehler:2001}
J.~Siehler, {\it Braided near-group categories}, 
\arxiv{math.QA/0011037}{math.QA/0011037}.

\bibitem[ST]{Snyder:2008}
N.~Snyder and P.~Tingley, {\em The half-twist for $U_q(g)$ representations},
\doi{10.2140/ant.2009.3.809}{Algebra and Number Theory {\bf 3} (2009) 809--834}
\arxiv{0810.0084}{[0810.0084 [math.QA]]}.

\bibitem[Sw]{Sweedler:1967}
M.E.~Sweedler, 
{\it Hopf algebras},  Benjamin, 1969.

\bibitem[Ta]{Takeuchi:1999}
M.~Takeuchi,
{\it Finite Hopf algebras in braided tensor categories},
\doi{10.1016/S0022-4049(97)00207-7}{J.\ Pure Appl.\ Alg.\ {\bf 138} (1999) 59--82}.

\bibitem[TY]{Tambara:1998}
    D.~Tambara and S.~Yamagami, 
{\it Tensor categories with fusion rules of self-duality for finite abelian groups}, 
\doi{10.1006/jabr.1998.7558}{J.\ Alg.\ {\bf  209} (1998) 692--707}.

\end{thebibliography}
\end{document}